\documentclass[12pt,a4paper]{amsart}
\usepackage{latexsym,amsxtra,eucal}
\usepackage[cmtip,arrow]{xy}
\usepackage{pb-diagram,pb-xy}

\calclayout

\newcommand{\+}{\protect\nobreakdash-}
\renewcommand{\:}{\colon}
\renewcommand{\;}{,\>}
\renewcommand{\.}{\mskip0.5\thinmuskip}

\let\le\undefined
\DeclareMathSymbol{\le}{\mathrel}{AMSa}{"36}         
\let\ge\undefined
\DeclareMathSymbol{\ge}{\mathrel}{AMSa}{"3E}         
\DeclareMathSymbol{\birarrow}{\mathrel}{AMSa}{"13}   
\DeclareMathSymbol{\varkappa}{\mathalpha}{AMSb}{"7B} 
\DeclareMathSymbol{\varnothing}{\mathord}{AMSb}{"3F} 

\DeclareMathOperator{\Hom}{Hom}
\DeclareMathOperator{\Cohom}{Cohom}
\DeclareMathOperator{\Ctrhom}{Ctrhom}

\DeclareMathOperator{\Tor}{Tor}
\DeclareMathOperator{\Ext}{Ext}
\DeclareMathOperator{\Cotor}{Cotor}
\DeclareMathOperator{\Coext}{Coext}
\DeclareMathOperator{\Ctrtor}{Ctrtor}

\DeclareMathOperator{\coker}{coker}
\DeclareMathOperator{\cone}{cone}
\DeclareMathOperator{\Cb}{Cob}
\DeclareMathOperator{\Br}{Bar}
\DeclareMathOperator{\id}{id}

\newcommand{\gr}{\mathrm{gr}}

\newcommand{\rarrow}{\longrightarrow}
\newcommand{\larrow}{\longleftarrow}
\newcommand{\mpsto}{\longmapsto}
\newcommand{\ot}{\otimes}
\renewcommand{\d}{\partial}

\newcommand{\limpr}{\varprojlim}
\newcommand{\limin}{\varinjlim}

\newcommand{\om}[1]{\!\.\overline{\.{\m}^#1\!\.}\.}

\newcommand{\DG}{\mathsf{DG}}
\newcommand{\Pro}{\mathsf{Pro}}

\renewcommand{\mod}{{\operatorname{\mathsf{--mod}}}}
\newcommand{\comod}{{\operatorname{\mathsf{--comod}}}}
\newcommand{\contra}{{\operatorname{\mathsf{--contra}}}}
\newcommand{\discr}{{\operatorname{\mathsf{--discr}}}}
\newcommand{\topol}{{\operatorname{\mathsf{--topol}}}}
\newcommand{\separ}{{\operatorname{\mathsf{--separ}}}}
\newcommand{\secm}{{\operatorname{\mathsf{--secm}}}}

\newcommand{\Rfr}{^{\operatorname{\R\mathsf{--fr}}}}
\newcommand{\Rcof}{^{\operatorname{\R\mathsf{--cof}}}}
\newcommand{\Rctr}{^{\operatorname{\R\mathsf{--ctr}}}}
\newcommand{\Rco}{^{\operatorname{\R\mathsf{--co}}}}

\newcommand{\Sfr}{^{\operatorname{\S\mathsf{--fr}}}}
\newcommand{\Scof}{^{\operatorname{\S\mathsf{--cof}}}}
\newcommand{\Sctr}{^{\operatorname{\S\mathsf{--ctr}}}}
\newcommand{\Sco}{^{\operatorname{\S\mathsf{--co}}}}

\newcommand{\modrRfr}{{\operatorname{\mathsf{mod}^
  {\R\mathsf{--fr}}\mathsf{--}}}}
\newcommand{\modrRfrproj}{{\operatorname{\mathsf{mod}^
  {\R\mathsf{--fr}}_{\mathsf{proj}}\mathsf{--}}}}
\newcommand{\modrRcof}{{\operatorname{\mathsf{mod}^
  {\R\mathsf{--cof}}\mathsf{--}}}}
\newcommand{\modrRcofproj}{{\operatorname{\mathsf{mod}^
  {\R\mathsf{--cof}}_{\mathsf{proj}}\mathsf{--}}}}
\newcommand{\modrRctr}{{\operatorname{\mathsf{mod}^
  {\R\mathsf{--ctr}}\mathsf{--}}}}
\newcommand{\modrRco}{{\operatorname{\mathsf{mod}^
  {\R\mathsf{--co}}\mathsf{--}}}}
\newcommand{\comodrRfr}{{\operatorname{\mathsf{comod}^
  {\R\mathsf{--fr}}\mathsf{--}}}}
\newcommand{\comodrRfrinj}{{\operatorname{\mathsf{comod}^
  {\R\mathsf{--fr}}_{\mathsf{inj}}\mathsf{--}}}}
\newcommand{\comodrRcof}{{\operatorname{\mathsf{comod}^
  {\R\mathsf{--cof}}\mathsf{--}}}}
\newcommand{\comodrRco}{{\operatorname{\mathsf{comod}^
  {\R\mathsf{--co}}\mathsf{--}}}}
\newcommand{\comodrRcofinj}{{\operatorname{\mathsf{comod}^
  {\R\mathsf{--cof}}_{\mathsf{inj}}\mathsf{--}}}}

\newcommand{\alg}{{\operatorname{\mathsf{--alg}}}}
\newcommand{\coalg}{{\operatorname{\mathsf{--coalg}}}}

\newcommand{\op}{{\mathrm{op}}}
\newcommand{\sop}{{\mathsf{op}}}

\newcommand{\abs}{{\mathsf{abs}}}
\newcommand{\co}{{\mathsf{co}}}
\newcommand{\ctr}{{\mathsf{ctr}}}
\newcommand{\si}{{\mathsf{si}}}
\newcommand{\proj}{{\mathsf{proj}}}
\newcommand{\inj}{{\mathsf{inj}}}
\newcommand{\free}{{\mathsf{free}}}
\newcommand{\cofr}{{\mathsf{cofr}}}
\newcommand{\fin}{{\mathsf{fin}}}

\newcommand{\wcdg}{{\mathsf{wcdg}}}
\newcommand{\cdg}{{\mathsf{cdg}}}
\newcommand{\cnlp}{{\operatorname{\mathit{k}\mathsf{--cnlp}}}}

\newcommand{\Seis}{{\mathsf{Seis}}}
\newcommand{\FSeis}{{\mathsf{FSeis}}}

\newcommand{\R}{{\mathfrak R}}
\newcommand{\m}{{\mathfrak m}}
\newcommand{\I}{{\mathfrak I}}
\newcommand{\J}{{\mathfrak J}}
\newcommand{\F}{{\mathfrak F}}
\renewcommand{\P}{{\mathfrak P}}
\newcommand{\Q}{{\mathfrak Q}}
\newcommand{\A}{{\mathfrak A}}
\newcommand{\B}{{\mathfrak B}}
\newcommand{\M}{{\mathfrak M}}
\newcommand{\N}{{\mathfrak N}}
\newcommand{\K}{{\mathfrak K}}
\renewcommand{\L}{{\mathfrak L}}
\renewcommand{\S}{{\mathfrak S}}
\newcommand{\T}{{\mathfrak T}}
\newcommand{\U}{{\mathfrak U}}
\newcommand{\V}{{\mathfrak V}}
\newcommand{\W}{{\mathfrak W}}
\newcommand{\gC}{{\mathfrak C}}
\newcommand{\C}{{\mathfrak C}}
\newcommand{\D}{{\mathfrak D}}

\newcommand{\cM}{{\mathcal M}}
\newcommand{\cN}{{\mathcal N}}
\newcommand{\cK}{{\mathcal K}}
\newcommand{\cL}{{\mathcal L}}
\newcommand{\cU}{{\mathcal U}}
\newcommand{\cV}{{\mathcal V}}
\newcommand{\cP}{{\mathcal P}}
\newcommand{\cQ}{{\mathcal Q}}
\newcommand{\cJ}{{\mathcal J}}
\newcommand{\cC}{{\mathcal C}}
\newcommand{\cD}{{\mathcal D}}

\newcommand{\sD}{{\mathsf D}}

\newcommand{\boI}{{\mathbb I}}
\newcommand{\boR}{{\mathbb R}}
\newcommand{\boL}{{\mathbb L}}
\newcommand{\boZ}{{\mathbb Z}}

\newcommand{\bu}{{\text{\smaller\smaller$\scriptstyle\bullet$}}}

\newcommand{\oc}{\mathbin{\text{\smaller$\square$}}}

\newcommand{\ocn}{\odot}

\newcommand{\lrarrow}{\.\relbar\joinrel\relbar\joinrel\rightarrow\.}

\newcommand{\comp}{\sphat\.}

\newcommand{\Ainfty}{$\mathrm{A}_\infty$}

\theoremstyle{plain}
\newtheorem{thm}{Theorem}[subsection]

\newtheorem{lem}[thm]{Lemma}
\newtheorem{prop}[thm]{Proposition}
\newtheorem{cor}[thm]{Corollary}
\theoremstyle{definition}
\newtheorem{rem}[thm]{Remark}
\newtheorem{ex}[thm]{Example}

\newcommand{\Section}[1]{\bigskip\section{#1}\medskip}

\begin{document}

\title{Weakly curved \Ainfty-algebras over \\ a topological local ring}
\author{Leonid Positselski}

\address{Institute of Mathematics, Czech Academy of Sciences,
\v Zitn\'a~25, 115~67 Prague~1, Czech Republic; and
\newline\indent  Laboratory of Algebraic Geometry, National Research
University Higher School of Economics, Moscow 119048, Russia; and
\newline\indent Sector of Algebra and Number Theory, Institute
for Information Transmission Problems, Moscow 127051, Russia; and
\newline\indent Department of Mathematics, Faculty of Natural Sciences,
University of Haifa, Mount Carmel, Haifa 31905, Israel
\newline\indent}

\email{posic@mccme.ru}

\begin{abstract}
 We define and study the derived categories of the first kind
for curved DG- and \Ainfty\+algebras complete over
a pro-Artinian local ring with the curvature elements divisible
by the maximal ideal of the local ring.
 We develop the Koszul duality theory in this setting and deduce
the generalizations of the conventional results about
\Ainfty\+modules to the weakly curved case.
 The formalism of contramodules and comodules over pro-Artinian
topological rings is used throughout the paper.
 Our motivation comes from the Floer--Fukaya theory.
\end{abstract}

\maketitle

\tableofcontents

\setcounter{section}{-1}
\section{Introduction}
\medskip

\subsection{{}}
 The conventional definition of the derived category involves
localizing the homotopy category (or alternatively, the category
of complexes and closed morphisms between them) by the class of
quasi-isomorphisms.
 In fact, one needs more than just such a definition in order
to do homological algebra in a derived category.
 Constructing derived functors and computing the Ext groups
requires having appropriate classes of resulutions.

 The \emph{classical homological algebra} can be roughly described
as the study of derived categories that can be equivalently
defined as the localizations of the homotopy categories by
quasi-isomorphisms and the full subcategories in the homotopy
categories formed by complexes of projective or injective objects,
or DG\+modules that are projective/injective as graded modules.

 This is true, e.~g., for appropriately bounded complexes over
an abelian or exact category with enough projectives or injectives,
or for bounded DG\+modules over a DG\+algebra with nonpositive
cohomological grading, or over a connected simply connected
DG\+algebra with nonnegative cohomological grading.

\subsection{{}}
 It is known since the pioneering work of Spaltenstein~\cite{Spal}
that one can work with the unbounded derived categories of modules
and sheaves using resolutions satisfying stronger conditions than
termwise projectivity, flatness or injectivity.
 These classes of resolutions, now known as \emph{homotopy projective},
\emph{homotopy flat}, etc., complexes, are defined by conditions
imposed on the complex as a whole and depending on the differential
in the complex, rather than only on its terms.
 This approach was extended to unbounded DG\+modules over
unbounded DG\+rings by Keller~\cite{Kel} and
Bernstein--Lunts~\cite{BL}.

 Another point of view, first introduced in Hinich's paper~\cite{Hin}
in the case of cocommutative DG\+coalgebras, involves strengthening
the conditions imposed on quasi-isomophisms rather than
the conditions on resolutions.
 Extended to DG\+comodules by Lef\`evre-Hasegawa~\cite{Lef} and
others, this theory took its fully developed form in the present
author's monograph~\cite{Psemi} and memoir~\cite{Pkoszul}, where
the general definitions of the \emph{derived categories of
the second kind} were given.

 The terminology of \emph{two kinds of derived categories} goes back
to the classical paper of Husemoller, Moore, and Stasheff~\cite{HMS},
where the distinction between \emph{differential derived functors of
the first} and \emph{the second kind} was introduced.
 The conventional unbounded derived category, studied by Spaltenstein,
Keller, et al., is called \emph{the derived category of the first kind}.
 The definition of the derived category of the second kind has
several versions, called the \emph{absolute derived}, \emph{complete
derived}, \emph{coderived}, and
\emph{contraderived category}~\cite{Pkoszul,PP2}.
 Here the ``coderived category'' terminology comes from Keller's
brief exposition~\cite{Kel2} of the related results from
Lef\`evre-Hasegawa's thesis.

\subsection{{}}
 As the latter terminological system suggests, the coderived categories
are more suitable for comodules, and similarly the contraderived
categories more suitable for contramodules~\cite{EM,Pcon}, than
for modules.
 The philosophy of ``taking the derived categories of the first kind
for modules, and the derived categories of the second kind for
comodules or contramodules'' was used in the author's monograph
on semi-infinite homological algebra~\cite{Psemi}.
 It works well in Koszul duality, too~\cite{Pkoszul}, although
the derived categories of the second kind for DG\+modules also have
their uses in the case of DG\+algebras whose underlying graded
algebras have finite homological dimension.

 Unlike the conventional quasi-isomorphism, the equivalence relations
used to define the derived categories of the second kind are not
reflected by forgetful functors; so one cannot tell whether, say,
a DG\+comodule is trivial in the coderived category (\emph{coacyclic})
or not just by looking on its underlying complex of vector spaces.
 Thus the forgetful functors between derived categories of
the second kind are generally \emph{not} conservative (which stands in
the way of possible application of the presently popular techniques
of the $\infty$\+categorical Barr--Beck theorem~\cite{Lur}).

 On the other hand, derived categories of the second kind make
perfect sense for curved DG\+modules or DG\+comodules~\cite{Pkoszul},
to which the conventional definition of the derived category
(of the first kind) is not applicable, as curved structures have no
cohomology groups.
 This sometimes forces one to consider derived categories of
the second kind for modules, including modules over rings of
infinite homological dimension, inspite of all the arising
technical complications~\cite{PP2,Psing}.

 The aim of this paper is to show how the derived category of
the first kind can be defined for curved DG-modules and
curved \Ainfty\+modules, if only in a rather special situation of
algebras over a complete local ring with the curvature element
divisible by the maximal ideal of the local ring.

\subsection{{}}
 Let us explain the distinction between the derived categories
of the first and the second kind in some more detail
(see also~\cite[Sections~0.1\+-0.3]{Pkoszul}
and~\cite[Preface and Section~0.2.9]{Psemi}).
 The classical situation, when there is no difference between
the two kinds of derived categories, is special in that no
infinite summation occurs when one totalizes a resolution of
a complex or a DG\+module.
 When the need to use the infinite summation arises, however, one is
forced to choose between taking direct sums or direct products.

 Informally, this means specifying the direction along the diagonals
of a bicomplex (or a similar double-indexed family of groups with
several differentials) in which the terms ``increase'' or
``decrease'' in the order of magnitude.
 Using the appropriate kind of completion is presumed, in which
one takes infinite direct sums in the ``increasing'' direction
and infinite products in the ``decreasing'' one.
 The components of the total differential of the bicomplex-like
structure are accordingly ordered; and the spectral sequence in
which one first passes to the cohomology with respect to
the dominating component of the differential converges (at
least in the weak sense of~\cite{EM0}) to the cohomology of
the related totalization.
 
 In the \Ainfty\+algebra situation, the conventional theory of
the first kind presumes that the operations $m_i$, \ $i\ge1$,
are ordered so that $m_1\gg m_2\gg m_3\gg\dotsb$ in the order of
magnitude, i.~e., the component $m_1$ dominates.
 So, in particular, if the differential $d=m_1$ is acyclic on
a given \Ainfty\+algebra or \Ainfty\+module, then such an algebra
or module vanishes in the homotopy category and the higher
operations $m_2$, $m_3$, etc.\ on it do not matter from the point
of view of the theory of the first kind.

 On the other hand, considering the derived category of
the second kind, e.~g., for CDG\+modules over a CDG\+algebra,
means choosing the multiplication as the dominating component,
i.~e., setting $m_2\gg m_1\gg m_0$.
 Hence the importance of the underlying graded algebras or modules
of CDG\+algebras or CDG\+modules, with the differentials and
the curvature elements in them forgotten, in the study of
the coderived, contraderived, and absolute derived categories.

 Of course, the above vague wording should be taken with a grain
of salt, and the notation is symbolic: it is not any particular
maps~$m_i$ (some of which may well happen to vanish for some
particular algebras) but the whole vector spaces of such maps
that are ordered in the ``order of magnitude''.

\subsection{{}}
 A theory of the second kind for curved \Ainfty\ structures
can also be developed.
 Essentially, it would mean that no ``divergent'' infinite sequences
of the higher operations should be allowed to occur.
 As usually, it is technically easier to do that for coalgebras,
where the ``convergence condition'' on the higher
comultilpications~$\mu_i$ appears naturally.
 This theory is reasonably
well-behaved~\cite[Sections~7.4\+-7.6]{Pkoszul}.

 It may be possible to have a theory of the second kind for
curved \Ainfty\+algebras, too, e.~g., by restricting oneself to those
curved \Ainfty\+algebras and \Ainfty\+modules in each of which there
is a finite number of nonvanishing higher operations $m_i$
only~\cite[final sentences of Remark~7.3]{Pkoszul}.
 However, proving theorems about such \Ainfty\+modules would involve
working with topological coalgebras, which seems to be technically
quite unpleasant (infinite operations on modules would be
problematic, etc.)

 A more delicate approach might involve imposing the convergence
condition according to which the operations~$m_i$ eventually vanish
in the restriction to the tensor powers of every fixed
finite-dimensional subspace of the curved \Ainfty\+algebra, and
similarly for the higher components~$f_i$ of the curved
\Ainfty\+morphisms, modules, etc.
 The bar-construction of such an \Ainfty\+algebra would be defined
as a DG\+coalgebra object in the tensor category of
ind-pro-finite-dimensional vector spaces.

 Perhaps the most reasonable way to deal with the forementioned problem
would require replacing the ground category of vector spaces with that
of pro-vector spaces, with the approximate effect of interchanging
the roles of algebras and coalgebras~\cite[Remark~2.7]{Psemi}.
 This is also what one may wish to do should the need to develop
a theory of the first kind for \Ainfty\+coalgebras
arise (cf.~\cite[Remark~7.6]{Pkoszul}).

\subsection{{}}
 On the other hand, having a theory of the first kind for curved
\Ainfty\+algebras would essentially mean setting $m_0\gg m_1\gg
m_2\gg\dotsb$, i.~e., designating the curvature element~$m_0$ as
the dominant term.
 The problem is that $m_0$, being just an element of a curved
\Ainfty\+algebra, is too silly a structure to be allowed
to dominate unrestrictedly.
 When nondegenerate enough (and it is all too easy for an element of
a vector space to be nondegenerate enough) and made dominating,
it would kill all the other structure of such a curved
\Ainfty\+algebra or \Ainfty\+module.

 That is why every curved \Ainfty\+algebra over a field which is
either considered as nonunital and has a nonzero curvature element,
or has a curvature element not proportional to the unit, is
\Ainfty\+isomorphic to a curved \Ainfty\+algebra with $m_i=0$
for all $i\ge1$ \cite[Remark~7.3]{Pkoszul}.
 Moreover, every \Ainfty\+module over a (unital or not) curved
\Ainfty\+algebra with a nonzero curvature element over a field
is contractible.

\subsection{{}}
 Hence the alternative of developing a theory of the first kind for
curved \Ainfty\+algebras over a local ring $\R$, with the curvature
element being required to be divisible by the maximal ideal
$\m\subset\R$.
 Let us first discuss this idea in the simplest case of the ring of
formal power series $\R=k[[\epsilon]]$, where $k$~is a field.

 In terms of the above ordering metaphor, this means having
two scales of orders of magnitude at the same time.
 On the one hand, the $\epsilon$\+adic topology is presumed, i.~e.,
$1\gg \epsilon\gg \epsilon^2\gg\dotsb$
 On the other hand, it is a theory of the first kind, so
$m_0\gg m_1\gg m_2\gg\dotsb$
 Given that $m_0$~is assumed to be divisible by~$\epsilon$ and
$m_1$~isn't, the question which of the two scales has the higher
priority arises immediately.

 If we want to make our theory as far from trivial as possible,
the natural answer is to designate the $\epsilon$\+adic scale
as the more important one.
 In the theory developed in this paper, this is achieved by having
the topology of a complete local ring~$\R$ built into the tensor
categories of $\R$\+modules in which our \Ainfty\+algebras
and \Ainfty\+modules live.
 That is where $\R$\+contramodules (and also $\R$\+comodules) come
into play.

\subsection{{}}
 In the conventional setting of (uncurved) DG- and
\Ainfty\+algebras over a field, the notion of \Ainfty\+morphisms
can be used to define the derived category of DG\+modules.
 Indeed, the homotopy category of \Ainfty\+modules over
an \Ainfty\+algebra coincides with their derived category, and
the derived category of \Ainfty\+modules over a DG\+algebra
is equivalent to the derived category of DG\+modules.
 So the complex of \Ainfty\+morphisms between two DG\+modules
over a DG\+algebra computes the $\Hom$ between them in
the derived category of DG\+modules.

 A similar definition of the derived category of curved
DG\+modules over a curved DG\+algebra was suggested in~\cite{Nic}.
 Then it was shown in the subsequent paper~\cite{KLN} that
the ``derived category of CDG\+modules'' defined in this way
vanishes entirely whenever the curvature element of
the CDG\+algebra is nonzero and one is working over a field
(as it follows from the above discussion).

 One of the results of this paper is the demonstration of a setting
in which this kind of definition of the derived category of
curved DG\+modules is nontrivial and reasonably well-behaved.

\subsection{{}}
 Before we start explaining what $\R$\+contramodules and
$\R$\+comodules are, let us have a look on the situation
from another angle.
 
 The passage from uncurved to curved algebras is supposed not
only to expand the class of algebras being considered, but also
enlarge the sets of morphisms between them.
 In fact, the natural functor from DG\+algebras to CDG\+algebras
is faithful, but not fully faithful~\cite{Pcurv}.
 Together with the curvature elements in algebras,
\emph{change-of-connection} elements in morphisms between algebras
are naturally supposed to come.

 One of the consequences of the existence of the change-of-connection
morphisms in the category of CDG\+algebras is the impossibility of
extending to CDG\+modules the conventional definition of the
derived category (of the first kind) of DG\+modules over
a DG\+algebra.
 Quite simply, CDG\+isomorphic DG\+algebras may have entirely
different derived categories of DG\+modules.
 Moreover, the derived category of DG\+modules over a DG\+algebra
is invariant under quasi-isomorphisms of DG\+algebras; and this
already is incompatible with the functoriality with respect to
change-of-connection morphisms.
 Indeed, \emph{any} two DG\+algebras over a field can be connected by
a chain of transformations, some of which are quasi-isomorphisms,
while the other ones are CDG\+isomorphisms of
DG\+algebras~\cite[Examples~9.4]{Pkoszul}.

\subsection{{}}
 So another problem with the na\"\i ve attempt to develop a theory
of the first kind for curved \Ainfty\+algebras over a field is that
one cannot have change-of-connection morphisms in it.
 One can say that the \Ainfty\+morphisms between such
\Ainfty\+algebras are too numerous, in that all the operations
$m_1$, $m_2$,~\dots\ can be killed by \Ainfty\+isomorphisms if only
the curvature element $m_0$ is not proportional to the unit, and
still they are too few, in that morphisms with nonvanishing
change-of-connection components~$f_0$ cannot be considered.

 To be more precise, recall that an \emph{\Ainfty\+morphism}
$f\:A\rarrow B$ is defined a sequence of maps $f_i\:A^{\ot i}
\rarrow B$, where $i\ge1$ \cite{Lef}.
 For curved \Ainfty\+algebras, one would like to define curved
\Ainfty\+morphisms as similar sequences of maps~$f_i$ starting
with $i=0$, the component $f_0\in B^1$ being
the change-of-connection element.
 The problem is that the compatibility equations on
the maps~$f_i$ in terms of the \Ainfty\+operations
$m_i^A\:A^{\ot i}\rarrow A$ and $m_i^B\:B^{\ot i}\rarrow B$ contain
a meaningless infinite summation when $f_0\ne0$  and one is working,
e.~g., over a field of coefficients (unless $m_i^B=0$ for $i\gg0$).

 The explanation is that while the maps $m_i^A$ are interpreted as
the components of a coderivation of the tensor coalgebra cogenerated
by the graded vector space $A$, the maps $f_i$ are the components
of a morphism between such tensor coalgebras.
 And while coderivations may not preserve coaugmentations of
conilpotent coalgebras over fields, coalgebra morphisms always do.

\subsection{{}}
 The latter problem can be solved by having $f_0$ divisible by
the maximal ideal~$\m$ of a complete local ring~$\R$ and
the components of the graded $\R$\+modules $A$ and $B$ complete
in the $\m$\+adic topology, to make the relevant infinite sums
convergent.
 One also wants the components of one's \Ainfty\+algebras over $\R$
to be free (complete) $\R$\+modules, so that their completed tensor
product over $\R$ is an exact functor.

 This is a good definition of the category of curved \Ainfty\+algebras
to work with; but when dealing with \Ainfty\+modules, it is useful to
have an abelian category to which their components may belong.
 And the category of (infinitely generated) $\m$\+adically complete
$\R$\+modules is \emph{not} an abelian already for $\R=k[[\epsilon]]$.
 The natural abelian category into which complete $\R$\+modules
are embedded is that of \emph{$\R$\+contramodules}.

 In particular, when $\R=\mathbb Z_l$ is the ring of $l$\+adic
integers, the abelian category of $\R$\+contramodules is that
of the \emph{weakly $l$\+complete abelian groups} of
Jannsen~\cite{Jan}, known also as the \emph{Ext-$p$-complete
abelian groups} of Bousfield--Kan~\cite{BK} (where $p=l$).
 Contramodules over $k[[\epsilon]]$ are very
similar~\cite[Remarks~A.1.1 and~A.3]{Psemi}.
 Another name for contramodules over the adic completion of a Noetherian
ring by an ideal is the \emph{cohomologically complete modules}
of Yekutieli et al.~\cite{PSY,PSY2,Yek}.

 Contramodules over a topological ring $\R$, particularly,
a topological ring which does not contain any field (such as
$\R=\mathbb Z_l$), are defined as modules/algebras over a certain
\emph{monad on the category of sets} associated with~$\R$.
 Notice that a systematic study of a class of such monads, called
the ``algebraic'' monads, was undertaken by Durov in~\cite{Dur};
however, the monads that appear in connection with contramodules
are not algebraic, as they do not preserve filtered inductive
limits of sets.
 We refer to the introduction to~\cite{PR} for a detailed discussion
and further references.

\subsection{{}}
 Generally, contramodules are modules with infinite summation
operations~\cite{Pcon}.
 Among other things, they provide a way of having an abelian
category of nontopological modules with some completeness
properties over a coring or a topological ring.
 Defined originally by Eilenberg and Moore~\cite{EM} as natural
counterparts of comodules over coalgebras over commutative rings,
contramodules were studied and used in the present author's
monograph~\cite{Psemi} for the purposes of the semi-infinite
cohomology theory and the comodule-contramodule correspondence.

 In particular, $k[[\epsilon]]$\+contramodules form a full subcategory
of the category of $k[[\epsilon]]$\+modules (and even a full
subcategory of the category of $k[\epsilon]$\+modules).
 This subcategory contains all the $k[[\epsilon]]$\+modules $M$
such that $M\simeq\limpr_n M/\epsilon^n M$, and also some other
$k[[\epsilon]]$\+modules (hence the ``weakly complete'' terminology).
 The natural map $\M\rarrow\limpr_n\M/\epsilon^n\M$ is surjective
for every $k[[\epsilon]]$\+contramodule $\M$, but it may not be
injective~\cite[Section~A.1.1 and Lemma~A.2.3]{Psemi}.

 The more familiar \emph{comodules}, on the other hand, are basically
discrete or torsion modules.
 For a pro-Artinian topological ring, they are defined as
the opposite category to that of Gabriel's \emph{pseudo-compact
modules}~\cite{Gab}.
 Notice that our $\R$\+comodules are \emph{not} literally discrete
modules over a topological ring~$\R$, although any choice of
an injective hull of the irreducible discrete module over
a pro-Artinian commutative local ring $\R$ provides an equivalence
between these two abelian categories (and there is even
a \emph{natural} equivalence when $\R$ is a profinite-dimensional
algebra over a field or a profinite ring).
 So the $k[[\epsilon]]$\+comodules are just $k[[\epsilon]]$\+modules
with a locally nilpotent action of~$\epsilon$.

 More specifically, comodules over a pro-Artinian topological ring $\R$
are defined in this paper as ind-objects in the abelian category
opposite to the category of discrete $\R$\+modules of finite length.
 We refer to the book~\cite{KS} for a background discussion of
ind-objects.
 Let us mention the similarity between this definition of ours
and the notion of ind-coherent sheaves on ind-schemes and
ind-inf-schemes, studied by Gaitsgory and Rozenblyum~\cite{Gai,GR}.
 One difference between the two approaches is that the ind-coherent
sheaves in the sense of~\cite{Gai,GR} are \emph{complexes} of
sheaves; so they form a DG\+category or a stable $(\infty,1)$\+category,
i.~e., a refined version of a triangulated category.
 Our $\R$\+comodules, on the other hand, form an abelian category.

 The conventional formalism of tensor operations (i.~e.,
the tensor product and Hom) on modules or bimodules
over rings can be extended to comodules and contramodules over
noncocommutative corings, where the natural operations are in much
greater abundance and variety (there are five of them to be found
in~\cite{Psemi,Pkoszul}).
 For contramodules and comodules over a pro-Artinian commutative ring,
we define the total of \emph{seven} such operations in this paper.

 This allows to consider, in particular, curved \Ainfty\+modules
over $\R$\+free $\R$\+complete curved \Ainfty\+algebras with,
alternatively, either $\R$\+contramodule (``weakly complete'')
or $\R$\+comodule (``torsion'', ``discrete'') coefficients.
 By another instance of the derived comodule-contramodule
correspondence, the corresponding two homotopy categories are
naturally equivalent.

\subsection{{}}
 Yet another reason to work with complete modules or contramodules
rather than just conventional modules over a local ring is the need
to use Nakayama's lemma as the basic technical tool.
 The point is, the conventional version of Nakayama's lemma for
modules over local rings only holds for finitely generated modules.
 Of course, one does not want to restrict oneself to
finite-dimensional vector spaces or finitely generated modules
over the coefficient ring when doing the homological algebra of
\Ainfty\+algebras and \Ainfty\+modules.

 So we want to have an abelian category of modules with infinite
direct sums and products where Nakayama's lemma holds.
 Notice that Nakayama's lemma holds for infinitely generated
modules over an Artinian local ring.
 As a natural generalization of this obvious observation,
the appropriate version of Nakayama's lemma for infinitely
generated contramodules over a topological ring with
a topologically nilpotent maximal ideal was obtained
in~\cite[Section~A.2 and Remark~A.3]{Psemi}.

 For comodules over a pro-Artinian local ring, we use the dual
version of Nakayama's lemma, which is clearly true.

\subsection{{}} 
 We call curved DG\+algebras in the tensor category of free
$\R$\+contramodules with the curvature element divisible
by the maximal ideal $\m\subset\R$ \emph{weakly curved
DG\+algebras}, or \emph{wcDG\+algebras} over~$\R$.
 Morphisms of wcDG\+algebras are CDG\+al\-gebra morphisms 
with the change-of-connection elements divisible by~$\m$.

 CDG\+modules over a wcDG\+algebra are referred to as
wcDG\+modules.
 The similar terminology is used for \Ainfty\+algebras:
a \emph{weakly curved \Ainfty\+algebra}, or
a \emph{wc~\Ainfty\+al\-gebra}, is a curved \Ainfty\+algebra
in the tensor category of free $\R$\+contramodules with
the curvature element divisible by~$\m$.

 So we can summarize much of the preceding discussion by saying
that \emph{theories} (i.~e., derived categories and derived functors)
\emph{of the first kind make sense in the weakly curved,
but not in the strongly curved case}.

\subsection{{}}
 Let us explain how we define the equivalence relation on
wcDG\+modules and wc~\Ainfty\+modules.
 First assume that the underlying graded $\R$\+module of our
weakly curved module $\M$ over $\A$ is a free graded
$\R$\+contramodule.

 In this case we simply apply the functor of reduction modulo~$\m$
to obtain an uncurved DG- or \Ainfty\+module $\M/\m\M$ over
an uncurved algebra $\A/\m\A$.
 The weakly curved module $\M$ is viewed as a trivial object of our
triangulated category of modules if the complex $\M/\m\M$
is acyclic.
 In particular, when the curvature element of $\A$ in fact vanishes,
this condition means that the complex of free $\R$\+contramodules
$\M$ should be contractible (rather than just acyclic).

 So the triangulated category of wcDG- or wc~\Ainfty\+modules
that we construct is not actually their derived category of
the first kind, but rather a mixed, or
\emph{semiderived category}~\cite{Psemi}.
 It behaves as the derived category of the first kind
in the direction of $\A$ relative to $\R$, and the derived
category of the second kind, or more precisely the contraderived
category, along the variables from~$\R$.

 For this reason we call the weakly curved modules $\M$ such that
$\M/\m\M$ is an acyclic complex of vector spaces \emph{semiacyclic}.
 The prefix ``semi'' here means roughly ``halfway between acyclic
and contraacyclic'' or even ``halfway between acyclic and
contractible''; so it should be thought of as a condition
\emph{stronger} than the acyclicity.

 As to wcDG- or wc~\Ainfty\+modules $\N$ over $\A$ whose underlying
graded $\R$\+modules are $\R$\+contramodules that are not necessarily
free, we replace them with $\R$\+free wc $\A$\+modules $\M$ isomorphic
to $\N$ in the contraderived category of weakly curved $\A$\+modules
before reducing modulo~$\m$ to check for semiacyclicity.
 So our semiderived category of arbitrary $\R$\+contramodule
wcDG- or wc~\Ainfty\+modules over $\A$ is the quotient category
of their contraderived category by the kernel of the derived
reduction functor.
 To construct such an $\R$\+free resolution $\M$ of a given
weakly curved $\A$\+module $\N$, it suffices to find an $\R$\+free
left resolution for $\N$ in the abelian category of
$\R$\+contramodule weakly curved $\A$\+modules and totalize it
by taking infinite products along the diagonals.

 The definition of the equivalence relation on wcDG\+modules
or wc~\Ainfty\+modules over $\A$ whose underlying graded
$\R$\+modules are cofree $\R$\+comodules is similar, except
that functor $\cM\mpsto{}_\m\cM$ of passage to the maximal
submodule annihilated by~$\m$ is used to obtain a complex of
vector spaces from a curved $\A$\+module in this case.
 And for arbitrary $\R$\+comodule weakly curved $\A$\+modules,
the semiderived category is constructed as the quotient
category of the coderived category of such modules by
the kernel of the derived $\m$\+annihilated submodule functor.

\subsection{{}}
 In particular, when the categories of $\R$\+contramodules and
$\R$\+comodules have finite homological dimensions (e.~g.,
$\R$ is a regular complete Noetherian local ring), any acyclic
complex of free $\R$\+contramodules or cofree $\R$\+modules
is contractible.
 In this case, our semiderived category of wcDG- or
wc~\Ainfty\+modules can be viewed as a true derived category of
the first kind and called simply the \emph{derived category}.

 On the other hand, it is instructive to consider the case of
the ring of dual numbers $R=k[\epsilon]/\epsilon^2$.
 In this case, there is no difference between $R$\+contramodules
and $R$\+comodules, which are both just $R$\+modules; and
accordingly no difference between $R$\+contramodule and
$R$\+comodule curved $A$\+modules.

 Still, their semiderived categories are different, in the sense
that the two equivalence relations on weakly curved $A$\+modules
(in other words, the two classes of semiacyclic curved modules)
are different in the case of weakly curved modules that are not
(co)free over~$R$.
 Indeed, they are different already for $A=R$, as the classes
of coacyclic and contraacyclic complexes of $R$\+modules are
different~\cite[Examples~3.3]{Pkoszul}.

 The two semiderived categories of weakly curved $A$\+modules are
equivalent (as they are for any weakly curved algebra $\A$ over any
pro-Artinian local ring~$\R$), but the equivalence is a nontrivial
construction when applied to modules that are not $R$\+(co)free.

\subsection{{}}
 The most striking aspect of the theory of (semi)derived categories
of weakly curved modules developed in this paper is just how
nontrivial these are.
 On the one hand, there is a general tendency of the curvature
to trivialize the categories of modules, well-known to
the specialists now (see, e.~g., \cite{KLN}).

 One reason for this is that apparently no curved modules over
a given curved algebra can be pointed out \emph{a priori} that
would not be known to vanish in the homotopy category already.
 In particular, a curved algebra has \emph{no} natural structure
of a curved module over itself.
 Some explicit constructions of CDG\+modules are used in
the proofs of the general theorems about them in~\cite{Pkoszul,
PP2,Psing}, but these always produce contractible CDG\+modules.
 There are lots of examples of nontrivial CDG\+modules, but these
are CDG\+modules over CDG\+algebras of some special types
(CDG\+bimodules~\cite{PP2}, Koszul CDG\+algebras~\cite{Pcurv},
change-of-connection transformations of DG\+algebras, etc.)

 In particular, an example from~\cite{KLN} shows that our
(semi)derived category of weakly curved modules over a wcDG- or
wc~\Ainfty\+algebra $\A$ may vanish entirely even when
the DG\+algebra $\A/\m\A$ has a nonzero cohomology algebra.
 This can already happen over $\R=k[[\epsilon]]$, or indeed,
over $\R=k[\epsilon]/\epsilon^2$.

 Specifically, let $\A=\R[x,x^{-1}]$ be the graded algebra over~$\R$
generated by an element~$x$ of degree~$2$ and its inverse element
$x^{-1}$, with the only relation saying that these two elements
should be inverse to each other, endowed with the zero differential,
vanishing higher operations $m_i=0$ for $i\ge3$, and the curvature
element $h=\epsilon x$.
 Then the homotopy categories of $\R$\+free and $\R$\+cofree
wcDG\+modules over $\A$ already vanish, as consequently do
the (semi)derived categories of wcDG- and
wc~\Ainfty\+modules over~$\A$ (see Example~\ref{kln-counterex}).

 For a similar wcDG\+algebra $\A'=\R[x]$ with $\deg x=2$, \ $d=0$,
and $h=\epsilon x$, one obtains a nonvanishing (semi)derived category
of wcDG- or wc~\Ainfty\+modules in which all
the $\R$\+contramodules of morphisms are annihilated by~$\epsilon$
(see Example~\ref{kln-counterex2}).

\subsection{{}}
 On the other hand, the obvious expectation that the
$\R$\+(contra)modules of morphisms in the triangulated category of
weakly curved $\A$\+modules are always torsion modules is most
emphatically \emph{not true}.
 The reason for the obvious expectation is, of course, that
the homotopy category of curved \Ainfty\+modules is trivial
over a field.

 The explanation is that one cannot quite localize contramodules.
 The functor assigning to an $\R$\+contramodule the tensor product
of its underlying $\R$\+module with the field of quotients of~$\R$
does \emph{not} preserve either the tensor product or the internal
$\Hom$ of contramodules; nor, indeed, does it preserve even infinite
direct sums.

\subsection{{}}
 Our computations of the $\R$\+contramodules $\Hom$ in certain
(semi)derived categories of wcDG- and wc~\Ainfty\+modules are
based on the Koszul duality theorems generalizing those
in~\cite{Pkoszul}.
 The semiderived category of wcDG\+modules over
a wcDG\+algebra $\A$ is equivalent to the coderived category of
CDG\+comodules and the contraderived category of CDG\+contramodules
over the CDG\+coalgebra $\gC=\Br(\A)$ obtained by applying
the bar construction to~$\A$.
 Similarly, the co/contraderived category of CDG\+co/contramodules
over an $\R$\+free CDG\+coalgebra $\gC$ that is conilpotent
modulo~$\m$ is equivalent to the semiderived category of
wcDG\+modules over the cobar construction $\Cb(\gC)$.

 In particular, it follows that the semiderived category of
wcDG\+modules over a wcDG\+algebra $\A$ is equivalent to
the semiderived category of wc~\Ainfty\+modules over $\A$
considered as a wc~\Ainfty\+algebra, so our lumping
together of the wcDG- and wc~\Ainfty\+modules in the preceding
discussion is justified.
 On the other hand, the semiderived category of
wc~\Ainfty\+modules over a wc~\Ainfty\+algebra $\A$ is equivalent
to the semiderived category of CDG\+modules over the enveloping
wcDG\+algebra of~$\A$.

\subsection{{}}
 To obtain a specific example of a nontrivial $\Hom$ computation
in a (semi)de\-rived category of wcDG- or wc~\Ainfty\+modules,
one can start with an ungraded $\R$\+free coalgebra $\gC$ considered
as a CDG\+coalgebra concentrated in degree~$0$ with a zero
differential and a zero curvature function.
 Then the co/contraderived category of, say, $\R$\+free
CDG\+co/contramodules over $\gC$ is just the co/contraderived
category of the exact category of $\R$\+free $\gC$\+co/contramodules.
 In particular, the exact category of $\R$\+free comodules
over $\gC$ embeds into its coderived category (and similarly
for contramodules).

 So considering $\gC$ as a comodule over itself we obtain
an example of an object in the coderived category whose
endomorphism ring is a nonvanishing \emph{free} $\R$\+contramodule.
 On the other hand, whenever the $k$\+coalgebra $\gC/\m\gC$ is
conilpotent, by the Koszul duality theorems mentioned
above the coderived category of $\R$\+free comodules
over $\gC$ is equivalent to the semiderived category of
wcDG- or wc~\Ainfty\+modules over $\Cb(\gC)$.
 It remains to pick an $\R$\+free coalgebra $\gC$ that is
conilpotent modulo~$\m$ but has no coaugmentation over~$\R$
in order produce an example of a semiderived category of
wcDG- or wc~\Ainfty\+modules with a nonzero $\R$\+free
$\Hom$ contramodule.

 In the simplest case of the coalgebra dual to the $\R$\+free
algebra of finite rank $\R[y]/(y^2=\epsilon)$
(with $\R=k[[\epsilon]]$ or $k[\epsilon]/\epsilon^2$, as above)
we obtain the wcDG\+algebra $\A=\Cb(\gC)$ that is freely generated
over $\R$ by an element $x$ of degree~$1$, with the zero
differential and the curvature element $h=\epsilon x^2$.
 The cobar construction assigns to the cofree comodule $\gC$
over $\gC$ a wcDG\+module $\M$ over $\A$ with an underlying
graded $\A$\+module freely generated by two elements.
 The algebra of endomorphisms of $\M$ in the (semi)derived
category of wcDG- or wc~\Ainfty\+modules over $\A$ is isomorphic
to $\R[y]/(y^2=\epsilon)$, so it is a free $\R$\+contramodule
of rank two (see Example~\ref{clifford-ex}).

\subsection{{}}
 One of our most important results in this paper is that
the semiderived category wcDG- or wc~\Ainfty\+modules over
a wcDG- or wc~\Ainfty\+algebra $\A$ over~$\R$ is compactly generated.
 So are the coderived category of CDG\+comodules and
the contraderived category of CDG\+contramodules over an $\R$\+free
CDG\+coalgebra~$\gC$.

 In the case case of wcDG- or wc~\Ainfty\+modules we present
a \emph{single}, if not quite explicit, compact generator.
 In order to construct this generator, one has to consider
$\R$\+comodule coefficients.
 In the semiderived category of $\R$\+comodule wcDG- or
wc~\Ainfty\+modules over $\A$, the weakly curved module $\A/\m\A$
is the desired compactly generating object.
 As we have already mentioned, the semiderived categories of
$\R$\+contramodule and $\R$\+comodule weakly curved modules
over $\A$ are equivalent, but the equivalence is a somewhat
complicated construction.

 In the case of the coderived category of CDG\+comodules over $\gC$,
we also consider $\R$\+comodule coefficients, and have CDG\+comodules
whose underlying graded $\R$\+comodules have finite length form
a triangulated subcategory of compact generators.
 For the contraderived category of CDG\+contramodules over $\gC$,
there is, once again, no explicit construction: one just has to
identify this category with the coderived category of
CDG\+comodules.

 While mainly interested in the curved modules and co/contramodules
in the tensor category of $\R$\+contramodules, it is chiefly for
the purposes of these constructions of compact generators that
we pay so much attention to the $\R$\+comodule coefficients and
the $\R$\+comodule-contramodule correspondence in our exposition.

\subsection{{}}
 Mixing contramodules with comodules is a tricky business, though.
 We have already discussed $\R$\+contramodule weakly curved
$\A$\+modules and $\R$\+comodule weakly curved $\A$\+modules in
this introduction; and now we have just mentioned $\R$\+comodule
curved $\C$\+comodules.
 So let us use the occasion to \emph{warn} the reader that, apparently,
it makes little sense to consider arbitrary $\R$\+contramodule
$\C$\+comodules or $\R$\+comodule $\C$\+contramodules, as such
categories of graded modules are not even abelian, the relevant
functors on the categories of $\R$\+contramodules and $\R$\+comodules
not having the required exactness properties.

 The (exotic derived) categories of $\R$\+free curved $\C$\+comodules
and $\R$\+cofree curved $\C$\+contramodules make perfect sense and
are well-behaved; and so are the categories of arbitrary
$\R$\+contramodule (or just $\R$\+free) curved $\C$\+contramodules
and arbitrary $\R$\+comodule (or just $\R$\+cofree) curved
$\C$\+comodules.
 The (exotic derived) categories of $\R$\+contramodule or
$\R$\+comodule (weakly) curved $\A$\+modules are also well-behaved,
as are the categories of $\R$\+free or $\R$\+cofree (weakly)
curved $\A$\+modules.
 But arbitrary (other than just $\R$\+free or $\R$\+cofree)
$\R$\+contramodule $\C$\+comodules or $\R$\+comodule
$\C$\+contramodules are generally problematic.

\subsection{{}}
 To end, let us briefly discuss the motivation and possible
applications.
 We are not in the position to suggest here any specific ways in
which the techniques we are developing could be applied in Fukaya's
Lagrangian Floer theory.
 In fact, the Novikov ring, which is the coefficient ring of
the Floer--Fukaya theory, is \emph{not} pro-Artinian
(\emph{nor} is it a topological local ring in our definition); so our
results do not seem to be at present directly applicable.

 Thus we restrict ourselves to stating that curved \Ainfty\+algebras
\emph{do} seem to appear in the Floer--Fukaya business, and that
their curvature (and change-of-connection) elements \emph{do} seem
to be, by the definition, divisible by appropriate maximal ideal(s)
of the coefficient ring(s)~\cite{Fu,FOOO,Cho}.
 Our study \emph{does} imply that, generally speaking, quite
nontrivial derived categories of modules can be associated with
curved algebras of this kind.
 Working over a complete local ring, rather than over a field, is
the price one has to pay for being able to obtain these
derived categories of modules.

 Furthermore, the semiderived categories of weakly curved DG-
and \Ainfty\+modules have all the usual properties of the derived
categories of DG- and \Ainfty\+modules over algebras over fields.
 The only caveat is that the nontriviality is \emph{not guaranteed}:
the triangulated categories of weakly curved modules may sometimes
vanish when, on the basis of the experience with uncurved modules
over algebras over fields, one would not expect them to
(as it was noticed in~\cite{KLN}).

\subsection{{}}
 Another possible application has to do with the deformation
theory of DG\+algebras.
 As pointed out in~\cite{KLN}, if one presumes that the deformations
of DG\+algebras should be controlled by their Hochschild 
cohomology complexes, one discovers that deformations in the class
of CDG\+algebras are to be considered on par with the conventional
DG\+algebra deformations.

 A curved infinitesimal or formal deformation of a DG\+algebra $A$
over a field $k$ is a wcDG\+algebra $\A$ over the ring
$R=k[\epsilon]/\epsilon^2$ or $\R=k[[\epsilon]]$, respectively.
 The problem of constructing deformations of the derived categories
of DG\+modules corresponding to curved deformations of DG\+algebras
was discussed in~\cite{KLN} (cf.\ the recent paper~\cite{DL}).

 Without delving into the implications of the deformation theory
viewpoint, let us simply state that what seems to be a reasonable
definition of the conventional derived category of wcDG\+modules
in the case of a pro-Artinian topological local ring $\R$ of
finite homological dimension is developed in this paper.
 So, at least, the case of a formal deformation may be (in some way)
covered by our theory.

\subsection{{}}
 I am grateful to all the people who have been telling me about
the curved \Ainfty\+algebras appearing in the Floer--Fukaya
theory throughout the recent years, and particularly Tony Pantev,
Maxim Kontsevich, Alexander Kuznetsov, Dmitri Orlov,
Anton Kapustin, and Andrei Losev.
 It would have never occured to me to consider curved algebras
over local rings without their influence.
 I~also wish to thank Bernhard Keller, Pedro Nicol\'as, and
particularly Wendy Lowen for an interesting discussion of
infinitesimal curved deformations of DG\+algebras.
 I would like to thank Sergey Arkhipov for suggesting that
tensor products of contramodules should be defined, B.~Keller for
directing my attention to Gabriel's pseudo-compact modules,
Alexander Efimov for answering my numerous questions about
the Fukaya theory, Ed Segal for a conversation about curved
\Ainfty\+algebras over a field, and Pierre Deligne for suggesting
the most natural definition of the tensor product of contramodules.
 Last but not least, I~am grateful to all the participants of
the informal student ``contraseminar'', which worked for several
weeks in Moscow in January--March 2014, for their participation,
interest, questions, and doubts.
 I~also wish to thank the anonymous referee for his comments,
which helped to improve the exposition.
 The author was partially supported by the Simons Foundation grant
and RFBR grants in Moscow, by a fellowship from the Lady Davis
Foundation at the Technion, by ISF grant~\#\,446/15 at
the University of Haifa, and by research plan RVO:~67985840 at
the Institute of Mathematics of the Czech Academy of Sciences
in Prague.

\Section{$\R$-Contramodules and $\R$-Comodules}

\subsection{Topological rings}  \label{topological-rings}
 Unless specified otherwise, all \emph{rings} in this paper are
assumed to be associative, commutative and unital.
 Some of the most basic of our results will be equally
applicable to noncommutative rings.

 Let $\R$ be a topological ring in which open ideals form
a base of neighborhoods of zero.
 We will always assume that $\R$ is separated and complete;
in other words, the natural map $\R\rarrow\limpr_\I\R/\I$,
where the projective limit is taken over all open ideals
$\I\subset\R$, is an isomorphism.

 A \emph{topological local ring} $\R$ is a topological ring
with a topologically nilpotent open ideal $\m$ such that
the quotient ring $\R/\m$ is a field.
 Here ``topologically nilpotent'' means that for any open
ideal $\I\subset\m$ there exists an integer $n\ge1$ such that
$\m^n\subset\I$.
 Clearly, a topological local ring is also local as
an abstract ring.
 For example, the completion of any (discrete) ring by the powers
of any of its maximal ideals is a topological local ring in
the projective limit topology.

 A topological ring is called \emph{pro-Artinian} if its
discrete quotient rings are Artinian rings.
 The projective limit of any filtered diagram of Artinian rings
and surjective morphisms between them is a pro-Artinian ring
(see~Corollary~\ref{appx-pro-rings-cor1}).
 For example, a complete Noetherian local ring $\R$ with the maximal
ideal $\m$ is a pro-Artinian topological ring in the $\m$\+adic
topology (cf.\ Appendix~\ref{noetherian-local-appx}).
 The dual vector space to any coassociative, cocommutative and
counital coalgebra over a field is a pro-Artinian topological ring
in its linearly compact topology.

\subsection{$\R$-contramodules}  \label{contramodules-sect}
 In this section we do not need to assume our rings to be
commutative yet.
 Given an (abstract) ring $R$, one can define (left) $R$\+modules in
the following fancy way.
 For any set $X$, let $R[X]$ denote the set of all finite formal
linear combinations of the elements of $X$ with coefficients in~$R$.
 The embedding $X\rarrow R[X]$ defined in terms of the zero and
unit elements of $R$ and the ``opening of parentheses'' map
$R[R[X]]\rarrow R[X]$ make the functor $X\mpsto R[X]$
a monad on the category of sets.
 The $R$\+modules are the algebras/modules over this monad
(cf.~\cite{Dur}).

 Now let $\R$ be a topological ring.
 For any set $X$, let $\R[[X]]$ be the set of all (infinite)
formal linear combinations $\sum_{x\in X} r_x x$ of the elements
of~$X$ with coefficients in $\R$ such that the family of
coefficients $r_x\in\R$ converges to zero as ``$x$~goes to infinity''.
 This means that for any open ideal $\I\subset\R$ the set of all
$x\in X$ such that $r_x\notin\I$ is finite.
 There is the obvious embedding $\varepsilon_X\:X\rarrow \R[[X]]$.
 The ``opening of parentheses'' map $\rho_X\:\R[[\R[[X]]]]\rarrow
\R[[X]]$ is defined by the rule
$$\textstyle
 \sum_{y\in\R[[X]]} r_y y \mpsto \sum_{x\in X}(\sum_y r_y r_{yx}) x,
 \quad\text{where $y=\sum_{x\in\R}r_{yx}x$}.
$$
 Here the infinite sum $\sum_{y\in\R[[X]]} r_y r_{yx}$ converges
in $\R$, since the family $r_y$ converges to zero and $\R$
is complete.
 In fact, this construction is applicable to any topological
associative ring $\R$ where open right ideals form a base of
neighborhoods of zero~\cite[Remark~A.3]{Psemi}.

 The above natural transformations $\varepsilon_X\:X\rarrow \R[[X]]$
and $\rho_X\:\R[[\R[[X]]]]\rarrow\R[[X]]$ define the structure of
an (associative and unital) monad on the functor $X\mpsto\R[[X]]$.
 An \emph{$\R$\+contramodule} is a module over this monad.
 In other words, an $\R$\+contramodule $\P$ is a set endowed
with a map of sets $\pi_\P\:\R[[\P]]\rarrow\P$ satisfying
the conventional associativity and unitality axioms
$\pi_\P\circ\R[[\pi_\P]] = \pi_\P\circ\rho_\P$ and
$\pi_\P\circ\varepsilon_\P=\id_\P$ for an algebra/module over
a monad $X\mpsto \R[[X]]$.
 We call the map $\pi_\P$ the \emph{contraaction map} and the above
associativity equation the \emph{contraassociativity equation}.
 The category of $\R$\+contramodules is denoted by $\R\contra$.

 The morphism of monads $\R[X]\rarrow\R[[X]]$ defines
the (nontopological) $\R$\+module structure on the underlying
set of every $\R$\+contramodule, so we have the forgetful functor
$\R\contra\rarrow\R\mod$ from the category of $\R$\+contramodules
to the abelian category of $\R$\+modules.

 Equivalently, an $\R$\+contramodule can be defined as a set endowed
with the following ``infinite summation'' operations.
 For any family of elements $r_\alpha$ converging to zero in $\R$
and any family of elements $p_\alpha\in\P$ (where $\alpha$~runs
over some index set) there is a well-defined element
$\sum_\alpha r_\alpha p_\alpha\in\P$.
 These operations must satisfy the idenitities of unitality:
$$\textstyle
 \sum_\alpha r_\alpha p_\alpha = p_{\alpha_0}
 \quad\text{ if the set $\{\alpha\}$ consists of
 one element $\alpha_0$ and $r_{\alpha_0}=1$},
$$
associativity:
$$\textstyle
 \sum_\alpha r_\alpha\sum_\beta r_{\alpha\beta}p_{\alpha\beta}
 = \sum_{\alpha,\beta}(r_\alpha r_{\alpha\beta})p_{\alpha\beta}
 \quad \text{ if $r_\alpha\to0$ and $\forall\alpha$
 $r_{\alpha\beta}\to0$ in $\R$,}
$$
and distributivity:
$$ \textstyle
 \sum_{\alpha,\beta} r_{\alpha\beta} p_\alpha =
 \sum_\alpha (\sum_\beta r_{\alpha\beta}) p_\alpha
 \quad \text{ if $r_{\alpha\beta}\to 0$ in $\R$}.
$$
 Here the summation over $\alpha,\beta$ presumes a set of
pairs $\{(\alpha,\beta)\}$ mapping into another set by
$(\alpha,\beta)\mpsto\alpha$ (i.~e., the range of possible~$\beta$'s
may depend on a chosen~$\alpha$).

 The finite and infinite operations are compatible in the sense
of the equations
$$\textstyle
 \sum_\alpha r_\alpha(p_\alpha+q_\alpha)
 = \sum_\alpha r_\alpha p_\alpha + \sum_\alpha r_\alpha q_\alpha,
 \quad
 \sum_\alpha (r'_\alpha + r''_\alpha) p_\alpha
 = \sum_\alpha r'_\alpha p_\alpha + \sum_\alpha r''_\alpha p_\alpha,
$$
and 
$$\textstyle
 \sum_\alpha r_\alpha (s_\alpha p_\alpha) =
 \sum_\alpha (r_\alpha s_\alpha)p_\alpha, \quad
 \sum_\alpha (s r_\alpha) p_\alpha = s\sum_\alpha r_\alpha p_\alpha
$$
for $p_\alpha$, $q_\alpha\in\P$ and $r_\alpha$, $r'_\alpha$,
$r''_\alpha$, $s$, $s_\alpha\in\R$ with
$r_\alpha$, $r'_\alpha$, $r''_\alpha$ converging to zero.
 Using these identities, one can define the $\R$\+contramodule
structures on the kernel and cokernel of an $\R$\+contramodule
morphism $\P\rarrow\Q$ taken in the category of $\R$\+modules.

 Hence $\R\contra$ is an abelian category and $\R\contra\rarrow
\R\mod$ is an exact functor.
 One also easily checks that infinite products exist in
$\R\contra$ and the forgetful functor to $\R\mod$ preserves them.

 For the reasons common to all monads, for any set $X$ the set
$\R[[X]]$ has a natural $\R$\+contramodule structure.
 The functor taking a set $X$ to the $\R$\+contramodule $\R[[X]]$
is left adjoint to the forgetful functor from $\R\contra$ to
the category of sets.
 We call contramodules of the form $\R[[X]]$ \emph{free\/
$\R$\+contramodules}.
 Free contramodules are projective objects in the abelian
category of $\R$\+contramodules, there are enough of them,
and hence every projective $\R$\+contramodule is a direct
summand of a free $\R$\+contramodule.

 For any collection of sets $X_\alpha$, the free contramodule
$\R[[\coprod_\alpha X_\alpha]]$ generated by the disjoint
union of $X_\alpha$ is the direct sum of the free contramodules
$\R[[X_\alpha]]$ in the category $\R\contra$.
 This allows to compute, at least in principle, the direct
sum of every collection of $\R$\+contramodules, by presenting
them as cokernels of morphisms of free contramodules and
using the fact that infinite direct sums commute with cokernels.
 So infinite direct sums exist in $\R\contra$.

\begin{rem} \label{dim-1-contramodules}
 The functors of infinite direct sum are not exact in $\R\contra$
in general; in fact, one can check that they are not exact already
for the ring $\R=k[[z,t]]$ of formal power series in two variables
over a field~$k$.
 However, when the abelian category of $\R$\+contramodules has
homological dimension not exceeding~$1$, the infinite direct sums in
$\R\contra$ \emph{are} exact.

 Indeed, the derived functor of infinite direct sum in $\R\contra$
can be defined and computed using left projective resolutions.
 Now if every subcontramodule of a projective contramodule is
projective, it remains to check that the direct sum of a family of
injective morphisms of projective contramodules is an injective
morphism.
 This follows from the fact that the natural map from the direct
sum of a family of projective contramodules to their direct product
is injective.

 The latter assertion holds for contramodules over any topological
ring~$\R$.
 It suffices to prove it for free contramodules, for which it can be
checked in terms of the explicit constructions of all the objects
involved. 
\end{rem}

\begin{rem} \label{adically-separated-complete}
 The following three remarks contain a collection of results connecting
the concept of an $\R$\+contramodule, as defined in this section, with
the more familiar notion of a (separated and) complete
topological $\R$\+module.

 Let $R$ be a Noetherian commutative ring, $m\subset R$ be an ideal,
$\R=\varprojlim_{n\ge1}R/m^nR$ be the $m$\+adic completion of
the ring $R$, viewed as a topological ring in the projective limit
topology (which coincides with the $m$\+adic topology in these
assumptions), and $\m=\varprojlim_{n\ge1}m/m^n=\R m$ be the extension
of the ideal~$m$ in the ring~$\R$.
 Then, according to Theorem~\ref{noetherian-r-contramodules-thm} from
Appendix~\ref{noetherian-local-appx} to this paper, the forgetful
functor $\R\contra\rarrow R\mod$ is fully faithful, i.~e.,
the infinite summation operations of an $\R$\+contramodule are
uniquely determined by its underlying $R$\+module structure.
 The $R$\+modules in the essential image of this forgetful functor are
called \emph{cohomologically $m$\+adically complete} in the terminology
of Porta--Shaul--Yekutieli~\cite{PSY,PSY2,Yek} or~\emph{$m$\+contramodule
$R$\+modules} in the terminology of our paper~\cite{Pcta}.
 The full subcategory of $m$\+contramodule $R$\+modules
$\R\contra\subset R\mod$ is closed under kernels, cokernels,
extensions, and infinite products.

 Let us say that an $R$\+module $P$ is \emph{$m$\+adically separated}
if the natural map $P\rarrow\varprojlim_{n\ge1}P/m^nP$ is injective,
and \emph{$m$\+adically complete} if this map is surjective.
 Then any $\R$\+contramodule (\,$=$~cohomologically $m$\+adically
complete $R$\+module) is $m$\+adically complete; and any $m$\+adically
separated and complete $R$\+module is an $\R$\+contramodule; but
the converse implications are not true~\cite[Section~2, Theorem~5.6,
and Lemma~5.7]{Pcta}.
 So the category $R\mod_{m\secm}$ of $m$\+adically separated and complete
$R$\+modules is a full subcategory in $\R\contra$.

 The $m$\+adic completion functor $M\longmapsto\Lambda_m(M)=
\varprojlim_{n\ge1}M/m^nM$ takes an arbitrary $R$\+module $M$ to
an $m$\+adically separated and complete
$R$\+module~\cite[Corollary~3.6]{Yek0}.
 The completion functor $\Lambda_m\:R\mod\rarrow R\mod_{m\secm}$ is left
adjoint to the embedding functor $R\mod_{m\secm}\rarrow R\mod$
\cite[Theorem~5.8]{Pcta}.
 In particular, the restriction $\Lambda_m\:\R\contra\rarrow
R\mod_{m\secm}$ of the functor $\Lambda_m$ to the full subcategory
$\R\contra\subset R\mod$ is left adjoint to the embedding functor
$R\mod_{m\secm}\rarrow\R\contra$.
 The category of $m$\+adically separated and complete $R$\+modules
is \emph{not} as well-behaved as the category of $\R$\+contramodules;
for instance, the category $R\mod_{m\secm}$ is not
abelian~\cite[Example~1]{Yek}, \cite[Example~2.7(1)]{Pcta}.
 The full subcategory $R\mod_{m\secm}$ is not closed under cokernels
in $R\mod$ or $\R\contra$, \emph{nor} it is closed under
extensions~\cite[Example~2.5]{Sim}.
\end{rem}

\begin{rem} \label{countable-base-separated}
 More generally, let $\R$ be a topological associative ring with
a countable base of neighborhoods of zero consisting of
open right ideals.
 For any left $\R$\+contramodule $\P$, set $\Lambda_\R(\P)=
\varprojlim_{\J\subset\R}\P/\J\P$, where $\J$~ranges over all the open
right ideals in $\R$ and $\J\P$ denotes the image of the contraaction
map $\J[[\P]]\rarrow\P$ (cf.\ the next Section~\ref{nakayama-sect}).
 Let us emphasize that the abelian group $\Lambda_\R(\P)$ depends on
the $\R$\+contramodule structure on $\P$, and not only on the underlying
$\R$\+module structure.
 For any left $\R$\+contramodule $\P$, the abelian group
$\Lambda_\R(\P)$ has a natural structure of left
$\R$\+contramodule~\cite[Lemma~6.2(a)]{PR}.
 The underlying $\R$\+module of $\Lambda_\R(\P)$ is endowed with
a natural (projective limit) topology, making it a complete, separated
topological left $\R$\+module.

 The natural $\R$\+contramodule morphism $\lambda_{\R,\P}\:\P\rarrow
\Lambda_\R(\P)$ is always surjective (see~\cite[Lemma~A.2.3]{Psemi},
\cite[Lemma~D.1.1]{Pcosh}, or~\cite[Lemma~6.3(b)]{PR}).
 So one has $\Lambda_\R(\P)\simeq\P/\bigcap_{\J\subset\R}\J\P$.
 A left $\R$\+contramodule $\P$ is said to be \emph{separated} if
the map~$\lambda_{\R,\P}$ is an isomorphism.
 Thus one can say that \emph{every left\/ $\R$\+contramodule is
complete, but it does not have to be separated}.
 For any $\R$\+contramodule $\P$, the $\R$\+contramodule
$\Lambda_\R(\P)$ is separated~\cite[Lemma~6.2(b)]{PR}.
 All projective left $\R$\+contramodules are
separated~\cite[Lemma~6.9]{PR}.

 Denote the full subcategory of separated left $\R$\+contramodules by
$\R\separ\subset\R\contra$.
 Then the completion functor $\Lambda_\R\:\R\contra\rarrow\R\separ$
is left adjoint to the embedding functor $\R\separ\rarrow\R\contra$.
 The same counterexamples mentioned in the previous
Remark~\ref{adically-separated-complete} show that the category
$\R\separ$ is \emph{not} abelian in general; \emph{nor} it is closed
under cokernels or extensions in $\R\contra$ \cite[Remark~6.5]{PR}.
\end{rem}

\begin{rem}
 Generally speaking, a complete separated topological $\R$\+module
does \emph{not} have an underlying $\R$\+contramodule structure.
 In suffices to consider the topological module of rational
$p$\+adic numbers $\mathbb Q_p$ over the topological rings
of $p$\+adic integers $\R=\boZ_p$.
 The $\boZ_p$\+module $\mathbb Q_p$ does \emph{not} have
a $\boZ_p$\+contramodule structure (as it is clear from
Lemma~\ref{nakayama-lemma} below).

 However, given a topological ring $\R$ with a base of neighborhoods
of zero consisting of open right ideals, one can consider
the category $\R\topol$ of all complete separated topological left
$\R$\+modules $\T$ satisfying the following condition.
 For any neighborhood of zero $\U\subset\T$ there should exist an open
right ideal $\J\subset\R$ such that $rt\in\U$ for all $r\in\J$ and
$t\in\T$.
 Then every topological left $\R$\+module from the category
$\R\topol$ has a natural underlying left $\R$\+contramodule structure,
with the sum $\sum_\alpha r_\alpha t_\alpha\in\T$, for any family of
elements $r_\alpha$ converging to zero in $\R$ and any family of
elements $t_\alpha\in\T$, defined as the limit of finite partial sums
in the topology of~$\T$.

 Thus we have a forgetful functor $\R\topol\rarrow\R\contra$.
 This functor is far from being fully faithful: for example, when
$\R=k$ is a field with the discrete topology, $\R\topol$ is
the category of all complete separated topological $k$\+vector spaces,
$\R\contra$ is the category of (abstract or discrete) $k$\+vector
spaces, and the functor $\R\topol\rarrow\R\contra$ forgets all
the information about the topology.
 This example also shows that the category $\R\topol$ is not abelian.

 Conversely, given a left $\R$\+contramodule $\P$, consider
the projective limit $\Lambda_\R(\P)=\varprojlim_{\J\subset\P}\P/\J\P$,
where, once again, $\J$ ranges over all the open right ideals in $\R$
and $\J\P\subset\P$ is the image of the contraaction map
$\J[[\P]]\rarrow\P$.
 Then $\Lambda_\R(\P)$ is a complete separated topological abelian
group in the projective limit topology.
 The topological group $\Lambda_\R(\P)$ has a natural structure of
topological left $\R$\+module, making it an object of the category
$\R\topol$.
 The functor $\Lambda_\R\:\R\contra\rarrow\R\topol$ is left adjoint
to the forgetful functor $\R\topol\rarrow\R\contra$.

 In the particular case when $\R$ has a countable base of neighborhoods
of zero, the functor $\Lambda_\R$ was discussed in the previous
Remark~\ref{countable-base-separated} (in particular, the discussion
there shows that this functor is not fully faithful, either).
\end{rem}

\subsection{Nakayama's lemma}  \label{nakayama-sect}
 Recall that a closed ideal $\m$ in a topological ring $\R$ is said
to be \emph{topologically nilpotent} if any neighborhood of zero in $\R$
contains a large enough power $\m^n$, \ $n\ge1$, of the ideal~$\m$.
 The following result is a generalization of~\cite[Lemma~4.11]{Jan}
and~\cite[Corollary~0.3]{PSY2}
(cf.~Appendix~\ref{noetherian-local-appx}).

\begin{lem} \label{nakayama-lemma} 
 Let\/ $\m$ be a topologically nilpotent closed ideal in a topological
ring\/~$\R$ and\/ $\P$ be a nonzero\/ $\R$\+contramodule.
 Then the image of the contraaction map\/ $\pi\:\m[[\P]]\rarrow\P$
differs from\/~$\P$.
\end{lem}

\begin{proof}
 The following proof does not depend on the commutativity assumption
on~$\R$ (cf.~\cite[Lemma~A.2.1 and Remark~A.3]{Psemi}).
 Assume that the map $\pi\:\m[[\P]]\rarrow\P$ is surjective; let
$p\in\P$ be an element.
 Notice that for any surjective map of sets $f\:X\rarrow Y$, the induced
map $\m[[f]]\:\m[[X]]\rarrow\m[[Y]]$ is surjective.

 Define inductively $\m^{(0)}[[X]]=X$ and $\m^{(n)}[[X]]=
\m[[\m^{(n-1)}[[X]]]]$ for $n\ge1$.
 Let $p_1\in\m[[\P]]$ be a preimage of~$p$ under the map~$\pi$.
 Furthermore, let $p_n\in\m^{(n)}[[\P]]$ be a preimage of~$p_{n-1}$
under the map $\m^{(n-1)}[[\pi]]$.

 For any set $X$, let $\rho_X^{(n)}\:\m^{(n)}[[X]]\rarrow\m[[X]]$
denote the iterated monad multiplication/contraaction map.
 The abelian group $\m[[X]]$ is complete in its natural topology
with the base of neighborhoods of zero formed by the subgroups
$I[[X]]$, where $I\subset\m$ are open ideals.
 Besides, the map $\rho_X\:\m[[\m[[X]]]]\rarrow\m[[X]]$ is
continuous.

 Set $q_n=\rho^{(n-1)}_{\m[[\P]]}(p_n)\in\m^{(2)}[[\P]]$ for all
$n\ge2$.
 Since $\m$~is topologically nilpotent, the sum $\sum_n q_n$
converges in the topology of $\m[[\m[[\P]]]]$.
 Now we have $\m[[\pi]](q_n)=\rho_\P(q_{n-1})$ for all $n\ge3$
and $\m[[\pi]](q_2)=p_1$.
 Hence
$$\textstyle
 \m[[\pi]](\sum_{n=2}^\infty q_n) - \rho_\P(\sum_{n=2}^\infty q_n)
 = p_1
$$
and $p=\pi(p_1)=0$ by the contraassociativity equation.
\end{proof}

 Let $\R$ be a topological local ring with the maximal ideal~$\m$.
 Denote by $k=\R/\m$ the residue field.
 Assign to any $\R$\+contramodule $\P$ its quotient group
$\P/\m\P$ by the image $\m\P$ of the contraaction map $\m[[\P]]
\rarrow\P$.
 Then $\P/\m\P$ is a vector space over~$k$.
 In particular, for a free $\R$\+contramodule $\R[[X]]$ we obtain
the vector space $\R[[X]]/\m(\R[[X]])=k[X]$ with the basis $X$
over~$k$.

\begin{lem} \label{nakayama-proj-free}
 Over a topological local ring\/~$\R$, the classes of free and
projective contramodules coincide.
\end{lem}

\begin{proof}
 Let $\P$ be a projective $\R$\+contramodule.
 Picking a basis $X$ in the vector space $\P/\m\P$, we obtain
an $\R$\+contramodule morphism $\R[[X]]\rarrow\P/\m\P$, where
$\P/\m\P$ is endowed with an $\R$\+contramodule structure
induced by its $k$\+vector space structure.
 Since the $\R$\+contramodule $\R[[X]]$ is projective,
this morphism can be lifted to an $\R$\+contramodule morphism
$f\:\R[[X]]\rarrow\P$ (it suffices to choose preimages in $\P$
for all elements of~$X$).
 By Lemma~\ref{nakayama-lemma}, the cokernel of the morphism~$f$
is a zero contramodule, so $f$~is surjective.
 Since $\P$ is projective, it follows that $f$~is a projection onto
a direct summand in the abelian category of $\R$\+contramodules,
so $\R[[X]]=\P\oplus\Q$.
 Then $\Q/\m\Q=0$ and it remains to apply Lemma~\ref{nakayama-lemma}
again in order to conclude that $\Q=0$.
\end{proof}

\begin{lem} \label{nakayama-acycl-contract}
 Let\/ $\K^\bu$ be a (possibly unbounded) complex of free contramodules
over a topological local ring\/ $\R$ with the maximal ideal\/~$\m$.
 Then the complex\/ $\K^\bu$ is contractible if and only if
the complex of vector spaces\/ $\K^\bu/\m\K^\bu$ is contractible
(i.~e., acyclic).
\end{lem}

\begin{proof}
 Choose a contracting homotopy for $\K^\bu/\m\K^\bu$ and lift it
to a homotopy map $h$ on $\K^\bu$ using projectivity of
the contramodules $\K^i$.
 Then the endomorphism $dh+hd$ of the component $\K^i$ is
an identity map modulo $\m\K^i$.
 It remains to show that any morphism~$f$ of free $\R$\+contramodules
such that the induced map $f/\m f$ is an isomorphism of vector
spaces is an isomorphism of $\R$\+contramodules.
 The proof of this assertion repeats the above proof of
Lemma~\ref{nakayama-proj-free}.
\end{proof}

\begin{rem}  \label{assume-local}
 The results below in this section are applicable to any noncommutative
right pro-Artinian topological ring (i.~e., a filtered projective limit
of noncommutative rings with surjective morphisms between them).
 However, we will only prove them for a pro-Artinian topological local
ring here, using the above version of Nakayama's lemma as the main
technical tool.
 The general case of a pro-Artinian commutative ring can be deduced by
decomposing such a ring $\R$ as an infinite product of topological
local rings $\R_\alpha$ and any contramodule over $\R$ as the infinite
product of contramodules over~$\R_\alpha$
(see~\cite[Lemma~A.1.2]{Psemi}).
 In the noncommutative case, one proceeds by lifting the primitive
idempotents of the maximal prosemisimple quotient ring of~$\R$ to
a converging family of idempotents in~$\R$, etc.
 We do not go into that here, because only local rings are important
for our purposes.
\end{rem}

 Let $\R$ be a pro-Artinian (topological local) ring and
$\J\subset\R$ be a closed ideal.
 By Corollary~\ref{appx-pro-rings-cor2}, the topological ring $\R/\J$
is complete.
 Moreover, for any set $X$ the natural map $\R[[X]]\rarrow
\R/\J[[X]]$ is surjective, since it is a morphism of
$\R$\+contramodules that is surjective modulo~$\m$.
 Hence we have the exact sequence $0\rarrow\J[[X]]\rarrow\R[[X]]
\rarrow\R/\J[[X]]\rarrow0$.

 To any $\R$\+contramodule $\P$, one can assign its quotient group
$\P/\J\P$ by the image $\J\P$ of the contraaction map
$\J[[\P]]\rarrow\P$.
 Then $\P/\J\P$ is a contramodule over the quotient ring $\R/\J$;
it is the maximal quotient contramodule of $\P$ that is
a contramodule over $\R/\J$.
 In other words, the functor $\P\mpsto\P/\J\P$ is left adjoint to
the functor of contrarestriction of scalars $\R/\J\contra\rarrow
\R\contra$ (sending an $\R/\J$\+contramodule $\Q$ to the same set
$\Q$ considered as an $\R$\+contramodule).

 Hence the functor $\P\mpsto\P/\J\P$ is right exact and commutes
with infinite direct sums.
 One easily checks that it takes the free contramodule $\R[[X]]$
to the free contramodule $\R/\J[[X]]$.

 In particular, for an open ideal $\I\subset\R$ the functor
$\P\mpsto\P/\I\P$ takes values in the category of
$\R/\I$\+modules.
 This functor is well-defined for any topological ring~$\R$.

\begin{lem} \label{nakayama-reduct-proj}
 Let\/ $\R$ be a topological local ring.
 Then an\/ $\R$\+contramodule $\P$ is projective if and only if
the\/ $\R/\I$\+module\/ $\P/\I\P$ is projective for every
open ideal\/ $\I\subset\R$.
\end{lem} 

\begin{proof}
 The ``only if'' assertion is obvious; let us prove the ``if''.
 Pick a basis $X$ in the $k$\+vector space $\P/\m\P$ and consider
the free $\R$\+contramodule $\F=\R[[X]]$.
 Lifting the elements of $X$ into $\P$, we construct a morphism
of $\R$\+contramodules $\F\rarrow\P$ such that the induced morphism
of $k$\+vector spaces $\F/\m\F\rarrow\P/\m\P$ is an isomorphism.

 By Lemma~\ref{nakayama-lemma}, the morphism $\F\rarrow\P$ is
surjective.
 For any open ideal $\I\subset\R$, consider the morphism of
projective/free $\R/\I$\+modules $\F/\I\F\rarrow\P/\I\P$.
 This morphism is split, so for its kernel $K$ we have
$K/(\m/\I)K=0$, hence $K=0$.

 Now consider the kernel $\K$ of the morphism $\F\rarrow\P$.
 We have shown that $\K$ is contained in $\I\F$ as
a subcontramodule of $\F$, for any open ideal $\I\subset\R$.
 The $\R$\+contramodule $\F$ being free and the intersection of
all open ideals $\I$ in $\R$ being zero, the intersection of all
subcontramodules $\I\F$ in $\F$ is also zero; thus $\K=0$.
 (Cf.~\cite[Lemma~A.3 and Remark~A.3]{Psemi}.)
\end{proof}

\begin{lem} \label{nakayama-reduct-products}
 For any closed ideal\/ $\J$ in a pro-Artinian ring\/ $\R$,
the functor\/ $\P\mpsto\P/\J\P$ from the category of\/
$\R$\+contramodules to the category of\/ $\R/\J$\+contramodules
preserves infinite products.
\end{lem}

\begin{proof}
 As mentioned above (see Remark~\ref{assume-local}), we will assume
that $\R$ is a topological local ring.
 Presenting arbitrary contramodules $\P_\alpha$ as the cokernels
of morphisms of free $\R$\+contramodules, we reduce the question
to the case of the product of free $\R$\+contramodules
$\P_\alpha=\R[[X_\alpha]]$.
 Set $\Q=\prod_\alpha\R[[X_\alpha]]$.

 Pick a basis $X$ in the product of vector spaces $\prod_\alpha
k[X_\alpha]$.
 The morphism of $\R$\+contramodules $\R[[X]]\rarrow k[X]
\rarrow \prod k[[X_\alpha]]$ can be lifted to a morphism
of $\R$\+contramodules $\R[[X]]\rarrow\prod_\alpha\R[[X_\alpha]]$.
 For any open ideal $\I\subset\R$, the composition
$\R[[X]]\rarrow\prod_\alpha\R[[X_\alpha]]\rarrow\prod_\alpha
\R/\I[X_\alpha]$ factorizes through $\R/\I[X]$.
 The map $\R[[X]]\rarrow\prod_\alpha\R[[X_\alpha]]$ is
the projective limit of the maps $\R/\I[X]\rarrow
\prod_\alpha\R/\I[X_\alpha]$ taken over the open ideals~$\I$.

 For any Artinian local ring $R$ with the maximal ideal~$m$
and any collection of $R$\+modules $P_\alpha$ one has
$(\prod_\alpha P_\alpha)/m(\prod_\alpha P_\alpha)\simeq
\prod_\alpha P_\alpha/mP_\alpha$, since the ideal $m$ is
finitely generated.
 Moreover, the product of projective modules over a right
Artinian ring is projective~\cite{Bas,Ch}.
 Hence the morphism of free $\R/\I$\+modules $\R/\I[X]\rarrow
\prod_\alpha\R/\I[X_\alpha]$, being an isomorphism modulo~$m$,
is itself an isomorphism.

 Passing to the projective limit over all open ideals $\I\subset\R$,
we conclude that the map $\R[[X]]\rarrow\prod_\alpha\R[[X_\alpha]]$
is an isomorphism.
 Similarly, passing to the projective limit over the open ideals $\I$
containing $\J$, we have the isomorphism $\R/\J[[X]]\simeq
\prod_\alpha\R/\J[[X_\alpha]]$.
 The two isomorphisms form a commutative square with the natural
surjections.
 It remains to recall that $\R[[X]]/\J(\R[[X]])=\R/\J[[X]]$.
\end{proof}

 Along the way, we have also proven the next lemma.

\begin{lem} \label{nakayama-proj-products}
 The class of projective contramodules over a pro-Artinian ring\/
$\R$ is closed under infinite products. \qed
\end{lem}

 Notice that any projective $\R$\+contramodule is a direct summand
of an infinite product of copies of the $\R$\+contramodule $\R$,
as one can immediately see from the above proof.

\begin{rem}  \label{contrarestriction-direct-sums}
 The functor of contrarestriction of scalars $\R/\J\contra\rarrow
\R\contra$ does \emph{not} preserve infinite direct sums in general.
 It suffices to consider the case when $\m^2=0$ and $\m=\J$ is
an infinite-dimensional (linearly compact) $k$\+vector space.
 However, when $\R$ is a Noetherian local ring with
the $\m$\+adic topology, the functor of contrarestriction of scalars
from $\R/\J$ to $\R$ does preserve infinite direct sums, and
moreover, has a right adjoint functor, constructed as follows.
 Given an $\R$\+contramodule $\P$, one can consider the subset
${}_\J\P=\Hom^\R(\R/\J,\P)\subset\P$ of all elements annihilated
by $\J$ acting in $\P$ viewed as an $\R$\+module.
 This is clearly an $\R$\+subcontramodule of $\P$; in the above
Noetherian assumption, one can check that the contraaction
map $\J[[{}_\J\P]]\rarrow {}_\J\P$ vanishes.
\end{rem}

\subsection{$\R$-comodules}  \label{r-comodules-sect}
 Let $\R$ be a pro-Artinian topological ring.
 A \emph{comodule} over $\R$ is an ind-object of the abelian
category opposite to the category of discrete $\R$\+modules
of finite length.
 We refer to~\cite[Chapter~6 and Section~8.6]{KS} for background
material on ind-objects.
 By the definition, $\R$\+comodules form a locally Noetherian
(and even locally finite) Grothendieck abelian category.
 We denote the category of $\R$\+comodules by $\R\comod$.

 Comodules over pro-Artinian rings have no underlying groups or sets.
 However, by the definition, there is an anti-equivalence of categories
$$
 ({-})^\op\:\R\comod\lrarrow\Pro(\R\discr_\fin)
$$
connecting the category of $\R$\+comodules with the category of
pro-objects in the abelian category of discrete $\R$\+modules of
finite length $\R\discr_\fin$.
 The functor $\cM\mpsto\cM^\op$ formally inverts the arrows, assigning
a downwards directed diagram in $\R\discr_\fin$ to an upwards directed
diagram in the category opposite to $\R\discr_\fin$.

 In the terminology of Gabriel's
dissertation~\cite[Sections~IV.3--IV.4]{Gab}, the category
$\Pro(\R\discr_\fin)$ is interpreted as the category of pseudo-compact
topological modules over a pseudo-compact topological ring~$\R$.
 An exact, conservative functor of projective limit acts
from the category of pro-objects $\Pro(\R\discr_\fin)$ to the category of
abelian groups (see Corollary~\ref{appx-pro-modules-cor}).
 Moreover, for any $\R$\+comodule $\cM$ the abelian group $\cM^\op$
has a natural $\R$\+contramodule structure, because discrete
$\R$\+modules of finite length, being modules over discrete quotient
rings of $\R$, can be viewed as $\R$\+contramodules, and the projective
limits taken in $\R\contra$ are preserved by the forgetful functor
from $\R$\+contramodules to abelian groups.

 The ring $\R$ itself is the projective limit of its discrete quotient
rings, which are discrete $\R$\+modules of finite length.
 So $\R$ can be viewed as an object of $\Pro(\R\discr_\fin)$.
 Hence there is a distinguished object $\cC=\cC(\R)$ in $\R\comod$
such that $\cC^\op=\R$.
 For any $\R$\+comodule $\cM$, the $\R$\+comodule morphisms
$\cM\rarrow\cC$ correspond bijectively to morphisms
$\R\rarrow\cM^\op$ in $\Pro(\R\discr_\fin)$, that is, to elements
of the underlying abelian group (\,$=$~projective limit) of $\cM^\op$.
 The functor $\cM\longmapsto\cM^\op$ is exact and faithful; thus
$\cC(\R)$ is an injective cogenerator of the abelian category
$\R\comod$.
 Since this category is locally Noetherian, direct sums of copies
of the object $\cC$ are also injective.
 We will call these \emph{cofree} $\R$\+comodules.
 One can easily see that there are enough of them, so any
injective $\R$\+comodule is a direct summand of a cofree one.

 As any Grothendieck abelian category, the category $\R\comod$
has arbitrary infinite products.
 One can describe them in the following way.

 Given a closed ideal $\J\subset\R$, we assign to any $\R$\+comodule
$\cM$ its maximal subcomodule ${}_\J\cM$ that is a comodule over
$\R/\J$.
 In other words, the functor $\cM\mpsto{}_\J\cM$ is right adjoint
to the functor of corestriction of scalars $\R/\J\comod\rarrow
\R\comod$ (induced on the ind-objects by the functor sending
a discrete $\R/\J$\+module of finite length to the same abelian 
group considered as an $\R$\+module).

 Hence the functor $\cM\mpsto{}_\J\cM$ is left exact and preserves
infinite products.
 It follows that $\prod_\alpha \cM_\alpha = \limin_\I\prod_\alpha
{}_\I\cM_\alpha$ for any family of objects $\cM_\alpha\in\R\comod$,
where the inductive limit is taken over all open ideals
$\I\subset\R$.
 As to the functors of infinite product in the abelian
category of comodules over an Artinian ring $R=\R/\I$, these are
exact, since the category $R\comod$ has enough projectives
(because the category of $R$\+modules of finite length has
enough injectives).
 
 Clearly, the functor $\cM\mpsto{}_\J\cM$ also commutes with all
filtered inductive limits.
 For any closed ideal $\J\subset\R$, one has ${}_\J\cC(\R)=\cC(\R/\J)$.
 Conversely, an $\R$\+comodule $\cJ$ is injective if and only if
the $\R/\I$\+comodule ${}_\I\cJ$ is injective for every open ideal
$\I\subset\R$.
 We use the conventional duality on the category of finite-dimensional
$k$\+vector spaces in order to identify comodules over the field~$k$
(considered as a discrete topological ring) with
(possibly infinite-dimensional) $k$\+vector spaces. 

\begin{lem} \label{comodule-nakayama}
 Let\/ $\R$ be a pro-Artinian topological local ring with the maximal
ideal\/~$\m$ and the residue field $k=\R/\m$, and let\/
$\cM$ be a nonzero\/ $\R$\+comodule.
 Then the $k$\+vector space\/ ${}_\m\cM$ is nonzero.
\end{lem}

\begin{proof}
 One can present the ind-object $\cM$ as the inductive limit of
a filtered diagram of $\R$\+comodules $\cM_\alpha$ of finite length
and injective morphisms $\cM_\alpha\rarrow\cM_\beta$ between them.
 Then the maps of $k$\+vector spaces ${}_\m(\cM_\alpha)\rarrow
{}_\m(\cM_\beta)$ are also injective, hence so are the maps
${}_\m(\cM_\alpha)\rarrow{}_\m\cM$.
 It remains to use the fact that for any nonzero module (of finite
length) $M$ over an Artinian local ring $R$ with the maximal ideal~$m$,
the quotient module/vector space $M/mM$ is nonzero.
\end{proof}

\begin{lem} \label{comodule-inj-cofree}
 Over a pro-Artinian topological local ring\/ $\R$, the classes of
cofree and injective comodules coincide.
\end{lem}

\begin{lem} \label{comodule-acycl-contract}
 Let\/ $\cK^\bu$ be a (possibly unbounded) complex of cofree comodules
over a pro-Artinian local ring\/ $\R$ with the maximal ideal\/~$\m$
and the residue field $k=\R/\m$.
 Then the complex\/ $\cK^\bu$ is contractible if and only if
the complex of $k$\+vector spaces\/ ${}_\m\cK^\bu$ is contractible
(i.~e., acyclic).
\end{lem}

\begin{proof}[Proofs of Lemmas~\textup{\ref{comodule-inj-cofree}}
and~\textup{\ref{comodule-acycl-contract}}]
 Completely analogous to the proofs of
Lemmas~\ref{nakayama-proj-free} and~\ref{nakayama-acycl-contract},
up to the duality and with the use of Lemma~\ref{nakayama-lemma}
replaced by that of Lemma~\ref{comodule-nakayama}.
\end{proof}

\begin{rem}
 The functor of corestriction of scalars $\R/\J\comod\rarrow
\R\comod$ has also a left adjoint functor $\cM\mpsto\cM/\J\cM$.
 The latter is defined on comodules of finite length by
the rule $(\cM/\J\cM)^\op={}_\J(\cM^\op)$, where ${}_\J M$ denotes
the maximal submodule annihilated by $\J$ in a discrete
$\R$\+module $M$, and extended to arbitrary $\R$\+comodules
as a functor preserving inductive limits.
 It follows that the functor of corestriction of scalars from
$\R/\J$ to $\R$ preserves infinite products.
\end{rem}

\subsection{Hom and contratensor product}  \label{hom-operations}
 For any topological ring $\R$ and any two $\R$\+contramodules
$\P$ and $\Q$, the set $\Hom^\R(\P,\Q)$ of all $\R$\+contramodule
homomorphisms $\P\rarrow\Q$ has a natural $\R$\+contramodule
structure with the infinite summation operation
$(\sum_\alpha r_\alpha f_\alpha)(x) = \sum_\alpha r_\alpha
f_\alpha(x)$, where $f_\alpha\in\Hom^\R(\P,\Q)$ are any
contramodule homomorphisms, $r_\alpha\in\R$ is a family of
elements converging to zero, the infinite sum in the left hand
side belongs to $\Hom^\R(\P,\Q)$, and the infinite sum in
the right hand side is taken in~$\Q$.
 We have constructed the functor
$$
 \Hom^\R({-},{-})\:\R\contra^\sop\times\R\contra\lrarrow\R\contra.
$$

 In particular, for any set $X$ and any $\R$\+contramodule $\Q$
there is a natural isomorphism of $\R$\+contramodules
$\Hom^\R(\R[[X]],\Q)\simeq\prod_{x\in X}\Q$.
 Hence it follows from Lemma~\ref{nakayama-proj-products} that
the functor $\Hom^\R({-},{-})$ preserves the class of projective
contramodules over a pro-Artinian topological ring~$\R$.
 Clearly, the functor $\Hom^\R$ is left exact and takes infinite
direct sums of contramodules in the first argument and infinite
products in the second argument to infinite products; it also
becomes exact when a projective contramodule is substituted as
the first argument.

 For any closed ideal $\J\subset\R$ and any two $\R$\+contramodules
$\P$ and $\Q$, there is a natural morphism of $\R/\J$\+contramodules
$\Hom^\R(\P,\Q)/\J\Hom^\R(\P,\Q)\rarrow\Hom^{\R/\J}(\P/\J\P,\Q/\J\Q)$.
 By Lemma~\ref{nakayama-reduct-products}, this morphism is
an isomorphism whenever the $\R$\+contramodule $\P$ is projective.
 Besides, for any $\R$\+contramodule $\P$ and $\R/\J$\+contramodule
$\Q$ there is a natural isomorphism of $\R$\+contramodules
$\Hom^\R(\P,\Q)\simeq\Hom^{\R/\J}(\P/\J\P,\Q)$, where
the $\R$\+contramodule structures on $\Q$ and the $\Hom$ contramodule
in the right hand side are obtained from the $\R/\J$\+contramodule
structures using the functor of contrarestriction of scalars
$\R/\J\contra\rarrow\R\contra$.

 For any pro-Artinian topological ring $\R$ and any $\R$\+comodules
$\cM$ and $\cN$, the set $\Hom_\R(\cM,\cN)$ of all $\R$\+comodule
morphisms $\cM\rarrow\cN$ has a natural $\R$\+contramodule
structure, which is constructed as follows.
 If $\cM=\limin_\alpha\cM_\alpha$, where $\cM_\alpha$ are
$\R$\+comodules of finite length, then $\Hom_\R(\cM,\cN) \simeq
\limpr_\alpha\Hom_\R(\cM_\alpha,\cN)$ and the $\R$\+contramodule
structure on the left hand side of this isomorphism is defined
as the projective limit of the contramodule structures on
the $\Hom$ groups in the right hand side (recall that the forgetful
functor $\R\contra\rarrow\R\mod$ preserves projective limits).
 Now $\cM_\alpha$ is a comodule of finite length over a discrete
Artinian quotient ring $R_\alpha=\R/\I_\alpha$ of the topological
ring $\R$, and $\Hom_\R(\cM_\alpha,\cN)\simeq\Hom_{R_\alpha}
(\cM_\alpha\;{}_{\I_\alpha}\cN)$ is an $R_\alpha$\+module, hence also
an $\R$\+contramodule.
 Thus we have the functor
$$
 \Hom_\R({-},{-})\:\R\comod^\sop\times\R\comod\lrarrow\R\contra.
$$ 

 Clearly, the functor $\Hom_\R({-},{-})$ is left exact and assigns
infinite products of contramodules to infinite direct sums of comodules
in the first argument and infinite products of comodules in the second
argument.
 Substituting $\cN=\cC(\R)$, we obtain the contravariant functor
$\R\comod\rarrow\R\contra$ taking $\cM$ to $\cM^\op$ that was
introduced in Section~\ref{r-comodules-sect}.

 For any closed ideal $\J\subset\R$ and two $\R$\+comodules $\cM$
and $\cN$, there is a natural morphism of $\R/\J$\+contramodules
$\Hom_\R(\cM,\cN)/\J\Hom_\R(\cM,\cN)\rarrow\Hom_\R({}_\J\cM,
{}_\J\cN)$.
 This morphism is an isomorphism whenever the $\R$\+comodule $\cN$
is injective.
 Indeed, the case $\cM=\cN=\cC(\R)$ is obvious; when both $\cM$
and $\cN$ are cofree $\R$\+comodules, the assertion follows
from Lemma~\ref{nakayama-reduct-products}. 
 Finally, since the functor $\cN\mpsto{}_\J\cN$ preserves injectivity,
the assertion for an injective comodule $\cN$ and an arbitrary
comodule $\cM$ can be deduced by presenting $\cM$ as the kernel of
a morphism of injective comodules.
 Besides, for any $\R/\J$\+comodule $\cM$ and $\R$\+comodule $\cN$
there is a natural isomorphism of $\R$\+contramodules 
$\Hom_\R(\cM,\cN)\simeq\Hom_{\R/\J}(\cM,{}_\J\cN)$, where
the $\R$\+comodule structure on $\cM$ and the $\R$\+contramodule
structure on the $\Hom$ contramodule in the right hand side are
defined using the functors of contra- and corestriction of scalars.

 Now we define the functor of \emph{contratensor product}
$$
 {-}\ocn_\R{-}\:\R\contra\times\R\comod\lrarrow\R\comod
$$
of contramodules and comodules over a pro-Artinian topological
ring~$\R$.
 For any $\R$\+contramodule $\P$ and any $\R$\+comodule $\cM$,
the contratensor product $\P\ocn_\R\cM$ is an $\R$\+comodule
constructed following the rules
\begin{enumerate}
\renewcommand{\theenumi}{\roman{enumi}}
\item $\P\ocn_\R\limin_\alpha\cM_\alpha = \limin_\alpha \P/\I_\alpha\P
      \ocn_{R_\alpha}\cM_\alpha$, where $\cM_\alpha$ is a comodule
      over a discrete Artinian quotient ring $R_\alpha=\R/\I_\alpha$
      of the topological ring~$\R$;
\item for any Artinian ring $R$, the functor $\ocn_R$ commutes with
      inductive limits in both arguments; and
\item $(Q\ocn_R\cN)^\op = \Hom_R(Q,\cN^\op)$ for any module $Q$ and
      any module of finite length $\cN^\op$ over an Artinian ring~$R$.
\end{enumerate}

 Clearly, the functor $\ocn_\R$ is right exact and commutes with
infinite direct sums in both arguments.
 It is also obvious that there is a natural isomorphism of
$\R$\+comodules $\R\ocn_\R\cM\simeq\cM$ for any $\R$\+comodule $\cM$.
 A natural isomorphism of $\R$\+contramodules
$\Hom_\R(\P\ocn_\R\cM\;\cN)\simeq\Hom^\R(\P,\Hom_\R(\cM,\cN))$
can easily be constructed for any $\R$\+contramodule $\P$ and
$\R$\+comodules $\cM$ and $\cN$.
 In particular, there is a natural isomorphism of $\R$\+contramodules
$(\P\ocn_\R\cM)^\op\simeq\Hom^\R(\P,\cM^\op)$ for any
$\R$\+contramodule $\P$ and $\R$\+comodule~$\cM$.

 Given a closed ideal $\J\subset\R$, for any $\R$\+contramodule $\P$
and $\R$\+comodule $\cM$ there is a natural morphism of
$\R/\J$\+comodules $\P/\J\P\ocn_{\R/\J}\.{}_\J\cM \simeq
\P\ocn_\R{}_\J\cM\rarrow {}_\J(\P\ocn_\R\cM)$, which is an isomorphism
when $\P$ is a projective $\R$\+contramodule.
 Here the contratensor product over $\R$ in the middle term is
identified with the $\R/\J$\+comodule from which its $\R$\+comodule
structure is obtained by the corestriction of scalars.

 Let $\cN$ be an $\R$\+comodule.
 It follows from the above assertion that the functor $\R\contra\rarrow
\R\comod$ defined by the rule $\P\mpsto\cN\ocn_\R\P$ is left adjoint
to the functor $\R\comod\rarrow\R\contra$ defined by the rule
$\cM\mpsto\Hom_\R(\cN,\cM)$.
 For any $\R$\+contramodule $\P$ and $\R$\+comodule $\cM$, set
$\Phi_\R(\P)=\cC(\R)\ocn_\R\P$ and $\Psi_\R(\cM)=\Hom^\R(\cC(\R),\cM)$.

 The following result is a generalization of~\cite[Proposition~2.1]{Har}
and~\cite[VI.4.5]{BK}.
 In the triangulated setting, it implies a natural equivalence between
the contraderived category of $\R$\+contramodules and the coderived
category of $\R$\+comodules, as we will see below in
Corollary~\ref{non-adj-fin-dim-reduct-r-co-contra},
Proposition~\ref{non-adj-r-co-contra}, or
Corollary~\ref{non-adj-derived-co-contra}.
 (Cf.\ the \emph{Matlis--Greenlees--May duality}~\cite{Mat,DG,PSY,Pmgm},
which is an equivalence of the conventional derived categories of 
discrete modules and contramodules over the adic completions of certain
rings by finitely generated ideals.)

\begin{prop} \label{r-co-contra}
\emergencystretch=0em\hfuzz=4pt
 The functors\/ $\Phi_\R$ and\/ $\Psi_\R$ restrict to mutually
inverse equivalences between the additive categories of free\/
$\R$\+contramodules and cofree\/ $\R$\+comodules.
\end{prop}

\begin{proof}
 Clearly, the functors $\Phi_\R$ and $\Psi_\R$ take the contramodule
$\R$ to the comodule $\cC=\cC(\R)$ and back; the functor $\Phi_\R$
preserves infinite direct sums and the functor $\Psi_\R$ preserves
infinite products.
 It follows that the functors $\Phi_\R$ and $\Psi_\R$ take projective
contramodules to injective comodules and back.
 It remains to check, e.~g., that the functor $\Psi_\R$ preserves
infinite direct sums of injective comodules.

 Indeed, a morphism of projective contramodules $\P\rarrow\Q$ is
an isomorphism whenever so are all the morphisms $\P/\I\P\rarrow
\Q/\I\Q$ for open ideals $\I\subset\R$.
 Now $\Hom_\R(\cC,\cM)/\I\Hom_\R(\cC,\cM)\simeq\Hom_{\R/\I}(\cC(\R/\I),
{}_\I\cM)$ for an injective $\R$\+comodule $\cM$ and the functor
$\Hom_R(\cN,{-})$ preserves infinite direct sums whenever $\cN$
is a comodule of finite length over an Artinian ring~$R$.
\end{proof}

\subsection{Contramodule tensor product and Ctrhom}
\label{contra-operations}
 Let $\R$ be a topological ring, and let $\P$, $\Q$, and $\K$ be
$\R$\+contramodules.
 A map $b\:\P\times\Q\rarrow\K$ is said to be \emph{contrabilinear}
if for any two families of elements $r_\alpha$ and $s_\beta\in\R$ 
converging to zero in the topology of $\R$ and any two families
of elements $p_\alpha\in\P$ and $q_\beta\in\Q$ the equation
$$\textstyle
 b(\sum_\alpha r_\alpha p_\alpha,\.\sum_\beta s_\beta q_\beta)=
 \sum_{\alpha,\beta}r_\alpha s_\beta b(p_\alpha,q_\beta)
$$
holds in~$\K$.
 An $\R$\+contramodule $\L$ is said to be the \emph{contramodule
tensor product} of the $\R$\+contramodules $\P$ and $\Q$ 
(denoted $\L=\P\ot^\R\Q$) if for any $\R$\+contramodule $\K$
the contrabilinear maps $\P\times\Q\rarrow\K$ correspond bijectively
to $\R$\+contramodule morphisms $\L\rarrow\K$ is a way functorial
in~$\K$.
 The following more explicit construction proves existence of
contramodule tensor products.

 For free $\R$\+contramodules $\R[[X]]$ and $\R[[Y]]$, set
$\R[[X]]\ot^\R \R[[Y]] = \R[[X\times Y]]$.
 There is a natural map $\R[[X]]\ot_\R \R[[Y]]\rarrow
\R[[X]]\ot^\R \R[[Y]]$, where $\ot_\R$ denotes the tensor
product in the category of $\R$\+modules $\R\mod$, taking
the tensor $\sum_{x\in X} r_x x\ot_\R\sum_{y\in Y}s_y y$
to the formal sum $\sum_{(x,y)\in X\times Y} r_xs_y(x,y)$,
where $r_x$, $s_y\in\R$ are two families of elements
converging to zero.

 Let $f\:\R[[X']]\rarrow\R[[X'']]$ and $g\:\R[[Y']]\rarrow
\R[[Y'']]$ be two homomorphisms of free $\R$\+contramodules.
 The data of the morphisms $f$ and~$g$ is equivalent to
the data of two families of elements $f(x')\in\R[[X'']]$ and
$g(y')\in\R[[Y'']]$, where $x'\in X'$ and $y'\in Y'$.
 Define the homomorphism $(f\ot g)\:\R[[X'\times Y']]\rarrow
\R[[X''\times Y'']]$ by the rule $(f\ot g)(x'\ot y') = 
f(x')\ot_\R g(y')\in\R[[X''\times Y'']]$.

 One readily checks that we have constructed a biadditive
tensor product functor~$\ot^\R$ on the category of free
$\R$\+contramodules.
 Since there are enough free contramodules in $\R\contra$,
this functor extends in a unique way to a right exact functor
$$
 {-}\ot^\R{-}\:\R\contra\times\R\contra\lrarrow\R\contra,
$$
which we wanted to construct.

 The functor of tensor product of free contramodules is naturally
associative, commutative, and unital with the unit object~$\R$;
hence so is the functor of tensor product of arbitrary
$\R$\+contramodules.
 The functor of tensor product of contramodules also preserves
infinite direct sums, since the functor of tensor product of
free contramodules does.
 In particular, the functor $\R[[X]]\ot^\R{-}$ assigns to any
$\R$\+contramodule $\P$ the direct sum of $X$ copies of~$\P$.

 There is a natural isomorphism of $\R$\+contramodules
$\Hom^\R(\P,\Hom^\R(\Q,\N))\simeq\Hom^\R(\P\ot^\R\Q\;\N)$ for any
three $\R$\+contramodules $\P$, $\Q$ and~$\N$; so $\Hom^\R$ is
the internal $\Hom$ functor for the tensor product functor~$\ot^\R$.
 As above, it suffices to construct this isomorphism for free
$\R$\+contramodules $\P$ and $\Q$, which is easy.

 Assuming that $\R$ is a pro-Artinian topological ring, one easily
checks that the reduction functor $\P\mpsto\P/\J\P$ takes tensor
products of $\R$\+contramodules to tensor products of
$\R/\J$\+contramodules for any closed ideal $\J\subset\R$.
 Besides, whenever the functor of contrarestriction of scalars
$\R/\J\contra\allowbreak\rarrow\R\contra$ preserves infinite
direct sums (see Remark~\ref{contrarestriction-direct-sums}),
there is a natural isomorphism of $\R$\+contramodules
$\R/\J\ot^\R\P\simeq\P/\J\P$ for any $\R$\+contramodule $\P$,
where the $\R$\+contramodule structures on $\R/\J$ and
$\P/\J\P$ are defined in terms of the functor of
contrarestriction of scalars.
 Indeed, both sides are right exact functors of the argument
$\P$ commuting with infinite direct sums, and the isomorphism
holds for $\P=\R$.

\begin{prop}
 For any $\R$\+contramodules\/ $\P$ and\/ $\Q$ and\/
$\R$\+comodule\/ $\cM$, there is a natural isomorphism of\/
$\R$\+comodules
$$
 (\P\ot^\R\Q)\ocn_\R\cM\simeq\P\ocn_\R(\Q\ocn_\R\cM).
$$
 In other words, the functor\/ $\ocn_\R$ makes\/ $\R\comod$ a module
category over the tensor category\/ $\R\contra$.
\end{prop}

\begin{proof}
 Follows from the definition of the functor of contratensor
product~$\ocn_\R$ and the above discussion of compatibility of
the reduction functors $\P\mpsto\P/\J\P$ with the tensor products
of $\R$\+contramodules.
\end{proof}

 Next we define the functor of \emph{contrahomomorphisms}
$$
 \Ctrhom_\R({-},{-})\:\R\contra^\sop\times\R\comod\lrarrow\R\comod
$$
from contramodules to comodules over a pro-Artinian topological
ring~$\R$.
 For any $\R$\+contramodule $\P$ and any $\R$\+comodule $\cM$,
the $\R$\+comodule $\Ctrhom_\R(\P,\cM)$ is constructed following
the rules
\begin{enumerate}
\renewcommand{\theenumi}{\roman{enumi}}
\item $\Ctrhom_\R(\P,\cM) =
      \limin_\I\Ctrhom_{\R/\I}(\P/\I\P\;{}_\I\cM)$, where
      the inductive limit is taken over all open ideals $\I\subset\R$;
\item for any Artinian ring $R$, the functor $\Ctrhom_R$ takes
      inductive limits in the first argument to projective limits;
\item for any module of finite length $Q$ over an Artinian ring $R$,
      the functor $\Ctrhom_R(Q,{-})$ preserves filtered inductive
      limits; and
\item $\Ctrhom_R(Q,\cN)^\op = Q\ot_R\cN^\op$ for any modules of
      finite length $Q$ and $\cN^\op$ over an Artinian ring~$R$.
\end{enumerate}

 Clearly, the functor $\Ctrhom_\R$ is left exact in both arguments.
 The functor $\Ctrhom_\R(\R[[X]],{-})$ assigns to any $\R$\+comodule
$\cM$ the direct product of $X$ copies of~$\cM$.
 Presenting an arbitrary $\R$\+contramodule as the cokernel of 
a morphism of free $\R$\+contramodules, one can conclude that
the functor $\Ctrhom_\R$ transforms infinite direct sums of
$\R$\+contramodules in the first argument and infinite products of
$\R$\+comodules in the second argument into infinite products
of $\R$\+comodules.
  In particular, there is a natural isomorphism $\Ctrhom_\R(\R,\cM)
\simeq\cM$.

\begin{prop}
 For any\/ $\R$\+contramodules\/ $\P$ and\/ $\Q$ and\/
$\R$\+comodule\/~$\cM$, there is a natural isomorphism of\/
$\R$\+comodules
$$
 \Ctrhom_\R(\P\ot^\R\Q\;\cM)\simeq
 \Ctrhom_\R(\P,\Ctrhom_\R(\Q,\cM)).
$$
 In other words, the functor\/ $\Ctrhom_\R$ makes the category\/
$\R\comod^\sop$ opposite to\/ $\R\comod$ a module category over\/
$\R\contra$.
\end{prop}

\begin{proof}
 First one can construct the desired isomorphism in the case of
free $\R$\+contra\-modules $\P$ and $\Q$, using the above discussion.
 Then the general case is dealt with using the exactness properties
of the functors $\ot^\R$ and $\Ctrhom_\R$.
\end{proof}

 Finally, for any $\R$\+contramodule $\P$ and $\R$\+comodules $\cM$
and $\cN$ there is a natural isomorphism of $\R$\+contramodules
$\Hom_\R(\P\ocn_\R\cM\;\cN)\simeq\Hom_\R(\cM,\Ctrhom_\R(\P,\cN))$.
 Given a closed ideal $\J\subset\R$, for any $\R$\+contramodule
$\P$ and $\R$\+comodule $\cM$ there is a natural isomorphism
of $\R/\J$\+comodules ${}_\J\!\.\Ctrhom_\R(\P,\cM)\simeq
\Ctrhom_{\R/\J}(\P/\J\P\;{}_\J\cM)$.
 Besides, there is a natural isomorphism of $\R$\+comodules
$\Ctrhom_\R(\R/\J,\cM)\simeq{}_\J\cM$, where the $\R$\+comodule
structure on the right hand side is defined by means of
the functor of corestriction of scalars $\R/\J\comod\rarrow\R\comod$.

\begin{lem}  \label{idempotent-lifting}
 Let\/ $\A$ be an (associative, noncommutative, unital) algebra
in the tensor category of free contramodules over a pro-Artinian
topological ring\/~$\R$, and let\/ $\J\subset\R$ be a closed ideal.
 Then any idempotent element in\/ $\A/\J\A$ can be lifted to
an idempotent element in\/~$\A$.
\end{lem}

\begin{proof}
 Notice first of all that $\A/\J\A$ is an algebra in the tensor
category of free contramodules over $\R/\J$, the underlying abelian
groups to $\A$ and $\A/\J\A$ have (noncommutative) ring structures,
and the natural projection $\A\rarrow\A/\J\A$ is a morphism of rings,
so the question about lifting idempotents makes sense.
 As usually, we will assume for simplicity that $\R$ is local.

 To prove the assertion, we will apply Zorn's lemma.
 Let $e_\J\in\A/\J\A$ be our idempotent element.
 Consider the set of all pairs $(\K,e_\K)$, where $\K\subset\R$ is
a closed ideal contained in $\J$ and $e_\K\in\A/\K\A$ is
an idempotent element lifting~$e_\J$.
 Endow this set with the obvious partial order relation.

 Given a linearly ordered subset of elements
$(\K_\alpha,e_{\K_\alpha})$, consider the intersection $\K$ of
all the ideals~$\K_\alpha$.
 By Corollary~\ref{appx-ideals-cor}, any open ideal $\I\subset\R$
containing $\K$ contains also one of the ideals~$\K_\alpha$.
 Hence the ring $\A/\K\A=\limpr_{\I\supset\K}\A/\I\A$ (where
$\I\subset\R$ are open ideals) is the projective limit of
the rings $\A/\K_\alpha\A$.
 So we can construct an idempotent element $e_\K\in\A/\K\A$
such that the pair $(\K,e_\K)$ majorates all the pairs
$(\K_\alpha,e_{\K_\alpha})$.

 Let $(\K,e_\K)$ be an element of our poset such that $\K\ne0$.
 Pick an open ideal $\I\subset\R$ not containing~$\K$.
 Then the ring $\A/\K\A$ is the quotient ring of
$\A/(\I\cap\K)\A$ by the ideal $\K\A/(\I\cap\K)\A$.
 Since $\K/(\I\cap\K)\simeq(\K+\I)/\I$ is a discrete and nilpotent
closed ideal in $\R/(\I\cap\K)$, the ideal $\K\A/(\I\cap\K)\A
\subset\A/(\I\cap\K)\A$ is also nilpotent.
 Thus it remains to use the well-known lemma on lifting idempotents
modulo a nil ideal (see, e.~g., \cite[Lemma~I.12.1]{Fe}) in order
to lift $e_\K$ to an idempotent element $e_{\I\cap\K}\in
\A/(\I\cap\K)\A$, completing the Zorn lemma argument and
the whole proof.

 Alternatively, one can use~\cite[Lemmas~I.12.1\+-2]{Fe} to
the effect that for any fixed, not necessarily idempotent
lifting $a\in\A$ of the element~$e_\J$ and any open ideal
$\I\subset\R$ there exists a unique idempotent 
$e_\I\in\A/\I\A$ expressible as a polynomial in $a \bmod \I\A$
with integer coefficients and lifting the idempotent element
$e_\J \bmod (\I+\J)\A/\J\A\in\A/(\I+\J)\A$.
 The compatible system of idempotents~$e_\I$ defines
the desired idempotent element of $\A=\limpr_\I\A/\I\A$
lifting~$e_\J$.
\end{proof}

\subsection{Cotensor product and Cohom} \label{co-operations}
 The following two operations on $\R$\+comod\-ules and
$\R$\+contramodules are not as important for our purposes in this paper
as the previous five.
 We consider them largely for the sake of completeness (however,
cf.\ the proof of Corollary~\ref{r-homol-dim} below, and
the constructions of functors $\ot_\B$ and $\Hom_\B$ in
Section~\ref{non-adj-graded} and functors $\oc_\C$ and $\Cohom_\C$ in
Section~\ref{non-adj-graded-co}).

 Let $\R$ be a pro-Artinian topological ring.
 The \emph{cotensor product} operation $(\cM,\.\cN)\mpsto\cM\oc_\R\cN$,
$$
 {-}\oc_\R{-}\:\R\comod\times\R\comod\lrarrow\R\comod
$$
on the category of $\R$\+comodules is obtained by passing to
the category of ind-objects and the opposite category from
the operation of the tensor product over~$\R$ of discrete
$\R$\+modules of finite length.
 Clearly, $\oc_\R$ is an associative and commutative tensor
category structure on $\R\comod$ with the unit object $\cC(\R)$.

 The functor ${-}\oc_\R{-}$ is left exact and commutes with infinite
direct sums.
 For any closed ideal $\J\subset\R$ and any $\R$\+comodules $\cM$
and $\cN$ there are natural isomorphisms of $\R/\J$\+comodules
$\cC(\R/\J)\oc_\R\cM\simeq{}_\J\cM$ and ${}_\J(\cM\oc_\R\cN)
\simeq{}_\J\cM\oc_\R\cN\simeq{}_\J\cM\oc_{\R/\J}\.{}_\J\cN$,
where the two cotensor products over $\R$ are identified with
the $\R/\J$\+comodules from which their $\R$\+comodule
structures are obtained by the corestriction of scalars.

 Let us define the functor of \emph{cohomomorphisms}
$$
 \Cohom_\R({-},{-})\:\R\comod^\sop\times\R\contra\rarrow\R\contra.
$$
from comodules to contramodules over~$\R$.
 For any $\R$\+comodule $\cM$ and $\R$\+contramod\-ule $\P$,
the $\R$\+contramodule $\Cohom_\R(\cM,\P)$ from $\cM$ to $\P$
is constructed as follows.

 For a cofree $\R$\+comodule $\cM=\bigoplus_{x\in X}\cC(\R)$ and
a projective $\R$\+contramodule $\P$, one sets
$\Cohom_\R(\bigoplus_X\cC(\R)\;\P)=\prod_X\P$.
 It follows that $\Cohom_\R(\bigoplus_X\cC(\R)\;\P) \simeq
\limpr_\I\Cohom_{\R/\I}({}_\I\bigoplus_X\cC(\R)\;\P/\I\P)$,
where the projective limit is taken over all open ideals
$\I\subset\R$.
 To define the (contravariant) action of $\Cohom_\R({-},\P)$ on
morphisms of cofree $\R$\+comodules $\cM$, one can first do
so for an Artinian ring $R$, which is easy; then pass to
the projective limit over~$\I$.
 Alternatively, set $\Cohom_\R(\cM,\P)=\Hom^\R(\Psi_\R(\cM),\P)$
(see Section~\ref{hom-operations}) for any cofree $\R$\+comodule
$\cM$ and any $\R$\+contramodule~$\P$.

 Since there are enough cofree comodules and projective
contramodules, one can extend the above-defined functor in
a unique way to a right exact functor $\Cohom^\R\:
\R\comod^\sop\times\R\contra\rarrow\R\contra$, as desired.
 The functor $\Cohom_\R({-},{-})$ transforms infinite direct sums of
$\R$\+comodules in the first argument and infinite products of
$\R$\+contramodules in the second argument into infinite
products of $\R$\+contramodules.

\begin{prop}
 For any\/ $\R$\+comodules\/ $\cM$ and\/ $\cN$ and\/
$\R$\+contramodule\/ $\P$, there are natural isomorphisms of\/
$\R$\+contramodules\/ $\Cohom_\R(\cC(\R),\P)\simeq\P$ and
$$
 \Cohom_\R(\cM\oc_\R\cN\;\P)\simeq
 \Cohom_\R(\cM,\Cohom_\R(\cN,\P)).
$$
 So the functor\/ $\Cohom_\R$ makes the category\/ $\R\contra^\sop$
opposite to\/ $\R\contra$ a module category over the tensor category\/
$\R\comod$.
\end{prop}

\begin{proof}
 Reduce first to the case of cofree $\R$\+comodules and a projective
$\R$\+contra\-module, and then to the case of a discrete Artinian
ring~$R$.
\end{proof}

 Given a closed ideal $\J\subset\R$, for any $\R$\+comodule $\cM$
and $\R$\+contramodule $\P$ there is a natural isomorphism of
$\R/\J$\+contramodules $\Cohom_\R(\cM,\P)/\J\Cohom_\R(\cM,\P)
\simeq\Cohom_{\R/\J}({}_\J\cM\;\P/\J\P)$ (since this is so for
a cofree $\R$\+comodule $\cM$ and a projective $\R$\+contramodule $\P$,
and all the functors are exact ``from the same side'').

 For any $\R$\+comodule $\cM$, there is a natural isomorphism of
$\R$\+contramodules $\Cohom_\R(\cM,\R)\simeq\cM^\op$.
 It follows that the functor $\Cohom_\R({-},\P)$ into a projective
$\R$\+contramodule $\P$ is exact, since any projective
$\R$\+contramodule is a direct summand of a product of copies
of~$\R$.
 More generally, for any $\R$\+comodules $\cN$ and $\cM$ there is
a natural isomorphism of $\R$\+contramodules $\Cohom_\R(\cM,\cN^\op)
\simeq(\cN\oc_\R\cM)^\op$.

 Assume that $\R$ is a complete Noetherian local ring.
 Then, for any $\R$\+comodule $\cN$ of finite length, there are natural
isomorphisms of $\R$\+contramodules $\Cohom_\R(\cN,\P)\simeq
\cN^\op\ot_\R\P\simeq\cN^\op\ot^\R\P$, where the $\R$\+contramodule
structure on the middle term is obtained by contrarestriction of
scalars from the $R$\+module structure, $R$ being a discrete Artinian
quotient ring of $\R$ such that $\cN$ is a comodule over~$R$.
 Indeed, the second isomorphism follows from the results of
Section~\ref{contra-operations}; and to prove the first one,
it suffices to notice that both functors of the second
argument~$\P$ are right exact and preserve infinite products, and
they produce the same $\R$\+contramodule for $\P=\R$.

 Let $\R$ be a pro-Artinian topological ring.
 Given a closed ideal $\J\subset\R$, for any $\R$\+contramodule $\P$
there is a natural isomorphism of $\R$\+contramodules
$\Cohom_\R\allowbreak(\cC(\R/\J),\P)\simeq\P/\J\P$, where
the $\R$\+contramodule structure on the right hand side is obtained by
the contrarestriction of scalars from the $\R/\J$\+contramodule structure.
 One proves this in the way similar to the above: both functors
of the argument $\P$ are right exact and preserve infinite products,
and the isomorphism holds for $\P=\R$. {\emergencystretch=1em\par}

\begin{lem}  \label{hom-co-associativity}
\textup{(a)} For any\/ $\R$\+contramodule\/ $\P$ and\/
$\R$\+comodules\/ $\cN$ and\/ $\cM$, there is a natural morphism of\/
$\R$\+comodules\/ $\P\ocn_\R(\cN\oc_\R\cM)\rarrow
(\P\ocn_\R\cN)\oc_\R\cM$.
 This morphism is an isomorphism whenever either
the\/ $\R$\+contramodule\/ $\P$ is projective, or
the\/ $\R$\+comodule\/ $\cM$ is injective. \par
\textup{(b)} For any\/ $\R$\+comodules $\cL$, $\cK$ and\/ $\cM$,
there is a natural morphism of\/ $\R$\+contra\-modules
$\Cohom_\R(\cL,\Hom_\R(\cK,\cM))\rarrow\Hom_\R(\cL\oc_\R\cK\;\cM)$.
 This morphism is an isomorphism whenever one of
the\/ $\R$\+comodules\/ $\cL$ and\/ $\cM$ is injective. \par
\textup{(c)} For any\/ $\R$\+contramodules\/ $\P$ and\/ $\Q$
and\/ $\R$\+comodule\/ $\cM$, there is a natural morphism of\/
$\R$\+contramodules\/ $\Cohom_\R(\P\ocn_\R\cM\;\Q)\rarrow
\Hom^\R(\P,\Cohom_\R(\cM,\Q))$.
 This morphism is an isomorphism whenever one of
the\/ $\R$\+contramodules\/ $\P$ and\/ $\Q$ is projective. \par
\textup{(d)} For any\/ $\R$\+contramodules\/ $\P$ and\/ $\Q$
and\/ $\R$\+comodule\/ $\cM$, there is a natural morphism of\/
$\R$\+contramodules\/ $\Cohom_\R(\cM,\Hom^\R(\P,\Q))\rarrow
\Hom^\R(\P,\Cohom_\R(\cM,\Q))$.
 This morphism is an isomorphism whenever either
the\/ $\R$\+contramodule\/ $\P$ is projective, or
the\/ $\R$\+comodule\/ $\cM$ is injective. \hfuzz=8.5pt
\end{lem}

\begin{proof}
 Part~(a): presenting $\cM$ and $\cN$ as filtered inductive
limits of comodules of finite length, we reduce the problem of
constructing the desired morphism to the case of a discrete
Artinian ring~$R$ and modules $P$, $\cM^\op$, $\cN^\op$ over it.
 Presenting the $R$\+module $P$ as a filtered inductive limit
of finitely generated modules, we can assume all the three
modules to have finite length.
 Then we have $(P\ocn_R(\cN\oc_R\cM))^\op = \Hom_R(P\;
\cN^\op\ot_R\cM^\op)$ and $((P\ocn_R\cN)\oc_R\cM)^\op =
\Hom_R(P,\cN^\op)\ot_R\cM^\op$, and there is a natural morphism
$\Hom_R(P,N)\ot_R M\rarrow\Hom_R(P\;N\ot_R M)$ for any
$R$\+modules $P$, $M$, $N$.
 To check the second assertion, notice that both sides commute
with infinite direct sums in all the three arguments, so it remains
to consider the cases $\P=\R$ or $\cM=\cC(\R)$, which are
straightforward.

 Part~(b): it suffices to construct the desired functorial morphism
in the case when the $\R$\+comodule $\cL$ is injective, because
the left hand side takes kernels in the argument $\cL$ to cokernels.
 So we can assume $\cL=\Phi_\R(\F)$ with $\F$ projective; then
$\cL\oc_\R\cK = \F\ocn_\R\cK$ by part~(a).
 There are functorial isomorphisms $\Cohom_\R(\cL,\Hom_\R(\cK,\cM))
\simeq\Hom^\R(\Psi_\R(\cL),\Hom_\R(\cK,\cM))\simeq\Hom^\R(\F,
\Hom_\R(\cK,\cM))\simeq\Hom_\R(\F\ocn_\R\cK\;\cM)\simeq
\Hom_\R(\cL\oc_\R\cK\;\cM)$.
 We have already proven the second assertion in the case when
$\cL$ is injective; to deal with the remaining case, notice that
both sides commute with infinite products in the argument $\cM$,
so it remains to consider the case $\cM=\cC(\R)$.
 In this case we have $\Cohom_\R(\cL,\Hom_\R(\cK,\cC(\R))
\simeq\Cohom_\R(\cL,\cK^\op)\simeq(\cL\oc_\R\cK)^\op\simeq
\Hom_\R(\cL\oc_\R\cK\;\cC(\R))$.

 Part~(c): it suffices to construct the desired functorial morphism
in the case when the $\R$\+contramodule $\P$ is projective, because
the right hand side takes cokernels in the argument $\P$ to
kernels.
 Assuming $\P$ is projective, the left hand side takes kernels in
the argument $\cM$ to cokernels, so it suffices to consider
the case when $\cM$ is injective.
 Then we can assume $\cM=\Phi_\R(\F)$ with $\F$ projective; hence
$\P\ocn_\R\cM=\Phi_\R(\P\ot^\R\F)$.
 The functorial isomorphisms $\Cohom_\R(\P\ocn_\R\cM\;\Q)
\simeq\Hom^\R(\Psi_\R(\P\ocn_\R\cM)\;\Q)\simeq
\Hom^\R(\P\ot^\R\F\;\Q)\simeq\Hom^\R(\P,\Hom^\R(\F,\Q))\simeq
\Hom^\R(\P,\Cohom_\R(\cM,\Q))$ provide the desired morphism.
 To prove the second assertion, notice that both sides take
infinite direct sums in the argument $\P$ and infinite products
in the argument $\Q$ to infinite products.
 The case $\P=\R$ is easy, and in the case $\Q=\R$ it remains
to recall that $\Cohom_\R(\cM,\R)=\cM^\op$ and
$(\P\ocn_\R\cM)^\op=\Hom^\R(\P,\cM^\op)$.

 The proof of the first assertion of part~(d) is similar to that
of the first assertions of parts~(b\+c); and the second assertion
of part~(d) is easy.
\end{proof}

\subsection{Change of ring}  \label{change-of-ring}
 Let $\eta\:\R\rarrow\S$ be a continuous homomorphism of topological
rings.
 Then for any set $X$ there is a natural map of sets
$\R[[X]]\rarrow\S[[X]]$ taking $\sum_x r_x x$ to $\sum_x \eta(r_x)x$.
 One easily checks that this is a morphism of monads, so any
$\S$\+contramodule $\Q$ acquires an induced $\R$\+contramodule
structure.
 We denote the $\R$\+contramodule so obtained by $R^\eta(\Q)$.
 This defines the functor of contrarestriction of scalars
$$
 R^\eta\:\S\contra\lrarrow\R\contra.
$$

 The functor $R^\eta$ is exact and preserves infinite products.
 It has a left adjoint functor of contraextension of scalars
$$
 E^\eta\:\R\contra\lrarrow\S\contra,
$$
which can be defined on free contramodules by the rule
$E^\eta(\R[[X]])=\S[[X]]$ and extended to arbitrary contramodules as
a right exact functor.
 The functor $E^\eta$ preserves infinite direct sums and tensor
products of contramodules.
 For any $\R$\+contramodule $\P$ and $\S$\+contramodule $\Q$,
there is a natural isomorphism of $\R$\+contramodules
$\Hom^\R(\P,R^\eta\Q)\simeq R^\eta\Hom^\S(E^\eta\P,\Q)$.

 Now assume that $\eta$ is a profinite morphism of pro-Artinian
topological rings, i.~e., the discrete quotient rings of $\S$ are
finite algebras over the discrete quotient rings of~$\R$.
 Then any discrete module of finite length over $\S$ is also
a discrete module of finite length over $\R$ in the induced
structure.
 Passing to the opposite categories and the ind-objects, we
obtain the functor of corestriction of scalars
$$
 R_\eta\:\S\comod\lrarrow\R\comod.
$$
 The functor $R_\eta$ is exact and preserves infinite direct sums.

 The functor $R_\eta$ has a right adjoint functor of coextension
of scalars
$$
 E_\eta\:\R\comod\lrarrow\S\comod,
$$
which can be defined on injective cogenerators by the rule
$E_\eta(\prod_{x\in X}\cC(\R)) = \prod_{x\in X}\cC(\S)$ and
extended to arbitrary comodules as a left exact functor.
 As a right adjoint to a functor taking Noetherian objects to
Noetherian objects, the functor $E_\eta$ preserves both
infinite direct sums and infinite products.
 For any $\R$\+contramodule $\P$ and $\R$\+comodule $\cM$, there is
a natural isomorphism of $\S$\+comodules $\Ctrhom_\S(E^\eta\P,
E_\eta\cM)\simeq E_\eta\Ctrhom_\R(\P,\cM)$.
 For any $\R$\+comodule $\cM$ and $\S$\+comodule $\cN$, there is
a natural isomorphism of $\R$\+contramodules
$\Hom_\R(R_\eta\cM,\cN)\simeq R^\eta\Hom_\S(\cM,E_\eta\cN)$.
{\hbadness=1650\par}

 Using Lemma~\ref{nakayama-reduct-products}, one can check that
the functor $E^\eta$ also preserves infinite products when $\eta$~is
a profinite morphism of pro-Artinian topological rings.
 For any $\R$\+contramodules $\P$ and $\K$, there is a natural
morphism of $\S$\+contramodules $E^\eta\Hom^\R(\P,\K)\rarrow
\Hom^\S(E^\eta\P,E^\eta\K)$, which is an isomorphism whenever
$\P$ is a projective $\R$\+contramodule.
 For any $\R$\+comodules $\cL$ and $\cM$, there is a natural
morphism of $\S$\+contramodules
$E^\eta\Hom_\R(\cL,\cM)\rarrow\Hom_\S(E_\eta\cL,E_\eta\cM)$,
which is an isomorphism whenever $\cM$ is an injective
$\R$\+comodule (see the argument for the case $\S=\R/\J$ in
Section~\ref{hom-operations}).
 For any $\R$\+contramodule $\P$ and $\S$\+comodule $\cN$,
there is a natural isomorphism of $\R$\+comodules
$\P\ocn_\R R_\eta\cN\simeq R_\eta(E^\eta\P\ocn_\S\cN)$.
 For any $\R$\+contramodule $\P$ and $\R$\+comodule $\cM$,
there is a natural morphism of $\S$\+comodules
$E^\eta(\P)\ocn_\S E_\eta(\cM) \rarrow E_\eta(\P\ocn_\R\cM)$,
which is an isomorphism whenever $\P$ is a projective
$\R$\+contramodule.

 For any $\R$\+contramodule $\P$ and $\S$\+contramodule $\Q$,
there is a natural morphism of $\R$\+contramodules
$\P\ot^\R R^\eta\Q\rarrow R^\eta(E^\eta\P\ot^\S\Q)$.
 For any $\R$\+contramodule $\P$ and $\S$\+comodule $\cN$,
there is a natural morphism of $\R$\+comodules
$R_\eta\Ctrhom_\S(E^\eta\P,\cN)\allowbreak\rarrow
\Ctrhom_\R(\P,R_\eta\cN)$.
 For any $\S$\+contramodule $\Q$ and $\R$\+comodule $\cM$,
there is a natural morphism of $\R$\+comodules
$R_\eta\Ctrhom_\S(\Q,E_\eta\cM)\rarrow\Ctrhom_\R(R^\eta\Q,\cM)$.

 For any $\R$\+comodule $\cM$ and $\S$\+comodule $\cN$, there is
a natural isomorphism of $\R$\+comodules $\cM\oc_\R R_\eta\cN
\simeq R_\eta(E_\eta\cM\oc_\S\cN)$.
 For any $\R$\+comodules $\cL$ and $\cM$, there is a natural
isomorphism of $\S$\+comodules $E_\eta(\cL\oc_\R\cM)\simeq
E_\eta(\cL)\oc_\S E_\eta(\cM)$.
 For any $\R$\+comodule $\cM$ and $\S$\+contramodule $\Q$, there is
a natural isomorphism of $\R$\+contramodules
$\Cohom_\R(\cM,R^\eta\Q)\simeq R^\eta\Cohom_\S(E_\eta\cM,\Q)$.
 For any $\S$\+comodule $\cN$ and $\R$\+contramodule $\P$, there is
a natural isomorphism of $\R$\+contramodules
$\Cohom_\R(R_\eta\cN,\P)\simeq R^\eta\Cohom_\S(\cN,E^\eta\P)$.
 For any $\R$\+comodule $\cM$ and $\R$\+contra\-module $\P$, there is
a natural isomorphism of $\S$\+contramodules $E^\eta\Cohom_\R(\cM,\P)
\allowbreak\simeq\Cohom_\S(E_\eta\cM,E^\eta\P)$.

 Let $\R\contra^\free$ and $\R\comod^\cofr$ denote the additive
categories of, respectively, free $\R$\+contramodules and
cofree $\R$\+comodules; the similar notation applies to~$\S$.

\begin{prop}  \label{r-s-extension}
 The equivalences of categories\/ $\Phi_\R=\Psi_\R^{-1}\:
\R\contra^\free\simeq\R\comod^\cofr$ and\/
$\Phi_\S=\Psi_\S^{-1}\:\S\contra^\free\simeq\S\comod^\cofr$
from Proposition~\textup{\ref{r-co-contra}} transform
the contraextension-of-scalars functor
$$
 E^\eta\:\R\contra^\free\lrarrow\S\contra^\free
$$
into the coextension-of-scalars functor
$$
 E_\eta\:\R\comod^\cofr\lrarrow\S\comod^\cofr
$$
and back.
 In other words, the following diagram of categories, functors,
and equivalences is commutative:
$$
\begin{diagram}
 \node{\llap{$\Phi_\R$}\:\R\contra^\free}\arrow{e,=}
 \arrow{s,l}{E^\eta}\node{\R\comod^\cofr\,\.\:\!\rlap{$\Psi_\R$}}
 \arrow{s,r}{E_\eta}\\
 \node{\llap{$\Phi_\S$}\:\S\contra^\free}\arrow{e,=}
 \node{\S\comod^\cofr\,\.\:\!\rlap{$\Psi_\S$}}
\end{diagram}
$$
\end{prop}

\begin{proof}
 It follows from the above that for any $\R$\+contramodule $\P$
there is a natural morphism of $\S$\+comodules $\Phi_\S E^\eta(\P)
\rarrow E_\eta\Phi_\R(\P)$, which is an isomorphism whenever
$\P$ is projective.
 Similarly, for any $\R$\+comodule $\cM$ there is a natural morphism
of $\S$\+contramodules $E^\eta\Psi_\R(\cM)\rarrow
\Psi_\S E_\eta(\cM)$, which is an isomorphism whenever
$\cM$ is injective.
\end{proof}

\subsection{Discrete $\R$-modules}  \label{discrete-modules}
 Let $\R$ be a pro-Artinian topological ring.
 Denote by $\R\discr$ the abelian category of discrete
$\R$\+modules.
 Clearly, it is a locally Noetherian (and even locally finite)
Grothendieck abelian category.

 For each irredicible discrete $\R$\+module, pick its injective
envelope in $\R\discr$, and denote by $C$ the direct sum of
all the injective objects obtained in this way.
 For any discrete $\R$\+module $M$ and any closed ideal $\J\subset\R$,
denote by ${}_\J M$ the maximal submodule of $M$ annihilated by~$\J$.
 Then ${}_\J C$ is a direct sum of injective hulls of
all the irreducible objects in $\R/\J\discr$.

 Applying this assertion in the case of an open ideal $\I\subset\R$
and using the standard results in commutative algebra of Artinian
rings~\cite[Theorem~18.6]{Mats}, one concludes that the functor
$M\mpsto\Hom_{\R\discr}(M,C)$ is an exact auto-anti-equivalence of
the abelian category of discrete $\R$\+modules of finite length.
 Passing to the ind-objects, we obtain an equivalence of abelian
categories $\R\discr\simeq\R\comod$ identifying the injective object
$C\in\R\discr$ with the canonical injective object $\cC\in\R\comod$.
 For discrete $\R$\+modules and $\R$\+comodules $M$ and $\cM$
of finite length, the functor $M\mpsto\Hom_{\R\discr}(M,C)$
is then identified with the functor $\cM\mpsto\cM^\op$ (both
functors taking values in the category of discrete $\R$\+modules
of finite length).

 Using this description of $\R$\+comodules, the constructions of
the operations of contratensor product and contrahomomorphisms in
Sections~\ref{hom-operations}--\ref{contra-operations}
can be presented in the following much simpler form.

 For any $\R$\+contramodule $\P$ and discrete $\R$\+module $M$,
the contratensor product $\P\ocn_\R M$ is a discrete $\R$\+module
defined by the rule $\P\ocn_\R\limin_\alpha M_\alpha =
\limin_\alpha \P/\I_\alpha\P\ot_{R_\alpha}M_\alpha$,
where $M_\alpha$ is a module over a discrete Artinian quotient
ring $R_\alpha=\R/\I_\alpha$ of the topological ring $\R$
and the tensor product in the right hand side is taken in
the category of $R_\alpha$\+modules.
 The discrete $\R$\+module of contrahomomorphisms
$\Ctrhom_\R(\P,M)$ is defined as the inductive limit
$\limin_\I\Hom_{\R/\I}(\P/\I\P\;{}_\I M)$ taken over all the open
ideals $\I\subset\R$, where $\Hom_{\R/\I}$ denotes the internal
Hom in the abelian tensor category of $\R/\I$\+modules.

 The cotensor product of discrete $\R$\+modules can be constructed
by the rule $\limin_\alpha N_\beta\oc_\R\limin_\alpha M_\alpha =
\limin_{\alpha,\beta}\Hom_\R(\Hom_\R(N_\beta,C)\ot_\R
\Hom_\R(M_\alpha,C)\;C)$, where $M_\alpha$ and $N_\beta$ are
discrete $\R$\+modules of finite length.
 In particular, the cotensor product with a discrete $\R$\+module
$N$ of finite length is described as $N\oc_\R M =
\Hom_\R(\Hom_\R(N,C)\;M)$.

\begin{lem} \label{free-r-contramodules}
 A contramodule\/ $\P$ over a pro-Artinian topological ring\/ $\R$
is projective if and only if either of the following equivalent
conditions holds:
\begin{enumerate}
\renewcommand{\theenumi}{\alph{enumi}}
\item the functor\/ $\cN\mpsto\P\ocn_\R\cN$ is exact on
      the abelian category of\/ $\R$\+comodules;
\item the contravariant functor\/ $\cM\mpsto\Cohom_\R(\cM,\P)$
      from the abelian category of\/ $\R$\+comodules to
      the abelian category of\/ $\R$\+contramodules is exact. 
\end{enumerate}
\end{lem}

\begin{proof}
 In both cases, the ``only if'' assertions have been already proven
in Sections~\ref{hom-operations} and~\ref{co-operations}; so
it remains to prove the ``if''.
 In both cases, restricting oneself to $\R$\+comodules $\cN$ or $\cM$
obtained by corestriction of scalars from comodules over $\R/\I$, where
$\I\subset\R$ are open ideals, and using
Lemma~\ref{nakayama-reduct-proj}, one reduces the question to the case
of modules over a discrete Artinian ring~$R$.

 In the latter situation, identifying $R$\+comodules with
$R$\+modules as explained above in this section, we see that
the functor in part~(a) is simply the functor $\ot_R$, so
the assertion follows from the fact that a flat module over
an Artinian ring is projective~\cite{Bas}.
 Finally, restricting ourselves to $R$\+comodules $\cM$ of
finite length in part~(b), we also identify the functor under
consideration with the functor of tensor product over~$R$
(see Section~\ref{co-operations}).
 (Cf.~\cite[Section~0.2.9]{Psemi}.)
\end{proof}

\begin{lem} \label{cofree-r-comodules}
 A comodule\/ $\cM$ over a pro-Artinian topological ring\/ $\R$
is injective if and only if either of the following equivalent
conditions holds:
\begin{enumerate}
\renewcommand{\theenumi}{\alph{enumi}}
\item the functor\/ $\cN\mpsto\cN\oc_\R\cM$ is exact on
      the abelian category of\/ $\R$\+comodules;
\item the functor\/ $\cM\mpsto\Cohom_\R(\cM,\P)$ is exact on
      the abelian category of\/ $\R$\+contramodules. 
\end{enumerate}
\end{lem}

\begin{proof}
 In both cases, the ``only if'' assertions have been already proven
in Section~\ref{co-operations}; it remains to prove the ``if''.
 Restricting ourselves to $\R$\+comodules $\cN$ or $\R$\+contramodules
$\P$ obtained by restriction of scalars from contramodules or
comodules over $\R/\I$, where $\I\subset\R$ are open ideals,
and using the relevant assertion from Section~\ref{r-comodules-sect},
we can reduce the question to the case of comodules over a discrete
Artinian ring~$R$.
 Restricting to $\R$\+contramodules $\P$ of the form $\P=\cN^\op$,
we conclude that it suffices to prove part~(a).
 Finally, it remains to restrict to $\R$\+comodules $\cN$ of finite
length and use the above description of $\oc_\R$ in terms of
discrete $\R$\+modules.
\end{proof}

\begin{rem} \label{contra-operations-not-exact}
 The assertions of Lemma~\ref{free-r-contramodules} do not apply to
the functors $\ot^\R$ and $\Ctrhom_\R$, because the substitution of
a projective contramodule at the first argument does not make
these functors exact.
 The reason is that infinite direct sums of $\R$\+contramodules
are not exact functors, and neither are infinite products of
$\R$\+comodules.

 Similarly, the assertions of Lemma~\ref{cofree-r-comodules} do not
apply to the functors $\ocn_\R$ and $\Ctrhom_\R$, because
the substitution of an injective comodule at the second argument
does not make these functors exact.
 In the case of the functor $\ocn_\R$, it suffices to consider
a discrete Artinian ring~$\R$ (essentially, one would like to
put projective comodules in the second argument of
$\ocn_\R$, but these do not exist in general).
 A counterexample for the functor $\Ctrhom_\R$ with
$\R=k[[\epsilon]]$ can be obtained from the examples of
$k[[\epsilon]]$\+contramodules with divisible
elements~\cite[Section~A.1.1]{Psemi}.
\end{rem}

\begin{cor} \label{r-homol-dim}
 For any pro-Artinian topological ring\/ $\R$, the homological
dimensions of the abelian categories of\/ $\R$\+contramodules and\/
$\R$\+comodules coincide.
\end{cor}

\begin{proof}
 The right derived functor
$$
 \Coext_\R^{-i}\:\R\comod^\sop\times
 \R\contra\lrarrow\R\contra, \quad i=0,\ 1,\ 2,\ \dots
$$
of the functor $\Cohom_\R$ is constructed by either replacing
the first argument of $\Cohom_\R$ with its right injective
resolution, or replacing the second argument with its left
projective resolution.
 The same derived functor is obtained in both ways, since
$\Cohom_\R$ from an injective $\R$\+comodule or into a projective
$\R$\+contramodule is an exact functor.
 It follows from Lemmas~\ref{free-r-contramodules}(b)
and~\ref{cofree-r-comodules}(b) that the homological
dimensions of both abelian categories $\R\contra$ and
$\R\comod$ are equal to the homological dimension of the derived
functor $\Coext_\R$ (i.~e., the maximal integer~$i$ for which
$\Coext_\R^{-i}$ is a nonzero functor).
\end{proof}

 The common value of the homological dimensions of the abelian
categories $\R\contra$ and $\R\comod$ is called the \emph{homological
dimension} of a pro-Artinian topological ring~$\R$.
 Notice that this dimension is also equal to the homological
dimension of the abelian category of discrete $\R$\+modules
of finite length.

 In particular, when $\R$ is a complete Noetherian local ring,
the category of discrete $\R$\+modules of finite length is
contained in the category of finitely generated $\R$\+modules,
and the latter is contained in the category of $\R$\+contramodules.
 It follows that the homological dimension of $\R$ as
a pro-Artinian topological ring is equal to the homological
dimension of $\R$ as an abstract commutative ring, i.~e.,
the Krull dimension of $\R$ when $\R$ is regular and infinity
otherwise.

 Let $\T\rarrow\R$ be a profinite morphism of pro-Artinian topological
rings (see Section~\ref{change-of-ring}).
 Let $C_\T$ be a direct sum of injective hulls of the irreducible
objects of $\T\discr$.
 Then the set $C_\R$ of all continuous $\T$\+module homomorphisms
$\R\rarrow C_\T$ (with the discrete topology on $C_\T$) endowed with
its natural $\R$\+module structure is a direct sum of injective hulls
of all the irreducible objects in $\R\discr$.
 This construction agrees with the compositions of profinite morphisms
$\T\rarrow\R\rarrow\S$.

 Assuming such morphisms to be given and the modules $C_\R$ and $C_\S$
being defined in terms of $C_\T$, we have the equivalences of
categories $\R\discr\simeq\R\comod$ and $\S\discr\simeq\S\comod$
corresponding to $C_\R$ and~$C_\S$.
 These identifications being presumed, the functor of corestriction of
scalars $R_\eta$ for the morphism $\eta\:\R\rarrow\S$ assigns to
a discrete module $N$ over $\S$ the same abelian group $N$ considered
as a discrete module over $\R$ in the module structure induced
by~$\eta$.
 The functor of coextension of scalars $E_\eta$ assigns to
a discrete $\R$\+module $M$ the set of all continuous $\R$\+module
homomorphisms $\S\rarrow M$ endowed with its natural discrete
$\S$\+module structure.

\subsection{Special classes of rings}
 The above construction provides a natural choice of an object $C_\R$
for some special classes of pro-Artinian topological rings $\R$,
such as the profinite-dimensional algebras over a field~$k$ (for which
one can use $\T=k=C_\T$) or the profinite rings (for which one
can take $\T=\widehat{\mathbb Z}=\prod_l\mathbb Z_l$ to be the product
of all the rings of $l$\+adic integers and $C_\T=\mathbb Q/\mathbb Z$).

 Let us discuss the case of a topological algebra $\R$ over
a field~$k$ in some more detail.
 First of all, in this case there is the following simplified
definition of $\R$\+contramodules.

 Given a vector space $V$ over~$k$, set
$\R\ot\comp V =\limpr_\I\R/\I\ot_k V$, where the projective limit
is taken over all the open ideals $\I\subset\R$.
 The unit element of $\R$ induces the natural map
$V\rarrow\R\ot\comp V$.
 The natural map $\R\ot\comp(\R\ot\comp V)\rarrow\R\ot\comp V$
is obtained by passing to the projective limit of the natural
maps $\R/\I\ot_k(\R\ot\comp V)\rarrow\R/\I\ot_k\R/\I\ot_k V
\rarrow \R/\I\ot_k V$.
 This construction can be extended to any noncommutative topological
algebra $\R$ over~$k$ where open right ideals form a base of
neighborhoods of zero~\cite[Section~D.5.2]{Psemi}.

 The above two natural transformations define a monad structure on
the endofunctor $V\mpsto \R\ot\comp V$ of the category of
$k$\+vector spaces.
 The category of algebras/modules over this monad is equivalent to
the category of $\R$\+contramodules as defined in
Section~\ref{contramodules-sect}.
 Indeed, recall from the discussion in the beginning of
Section~\ref{contramodules-sect} that the category of $k$\+vector
spaces can be defined as the category of modules over the monad
$X\mpsto k[X]$ on the category of sets.
 As with any monad, this provides for any $k$\+vector space $V$
the ``monadic bar-resolution'' with the initial fragment
$k[k[V]]\rarrow k[V]\rarrow V\rarrow 0$, where the middle arrow
is the natural map and the leftmost one is the difference of two
natural maps.

 Applying the additive functor $\R\ot\comp{-}$ to this exact sequence
of vector spaces and taking into account the natural isomorphism
$\R\ot\comp k[X]\simeq \R[[X]]$ for any set $X$, we obtain 
an exact sequence
$$
 \R[[k[V]]]\lrarrow\R[[V]]\lrarrow\R\ot\comp V\lrarrow0
$$
for any $k$\+vector space~$V$.
 Now to any map $\R\ot\comp\P\rarrow\P$ for a $k$\+vector space $\P$
one naturally assigns a map $\R[[\P]]\rarrow\P$; the latter
satisfies the contraassociativity axiom whenever the former does.
 Conversely, given a map $\R[[\P]]\rarrow\P$ satisfying
the contraassociativity axiom for a set $\P$, one introduces
a $k$\+vector space structure on $\P$ by considering the composition
$k[\P]\rarrow\R[[\P]]\rarrow\P$.
 Then the map $\R[[\P]]\rarrow\P$ factorizes uniquely through
the surjection $\R[[\P]]\rarrow\R\ot\comp\P$, as one can readily see
from the above exact sequence.

 The free $\R$\+contramodules are exactly the $\R$\+contramodules
of the form $\R\ot\comp V$, where $V$ is a $k$\+vector space.
 The infinite direct sums and the tensor product of free
$\R$\+contramodules are described by the rules $\bigoplus_\alpha
\R\ot\comp V_\alpha=\R\ot\comp\bigoplus_\alpha V_\alpha$ and
$(\R\ot\comp U)\ot^\R(\R\ot\comp V)=\R\ot\comp(U\ot_k V)$, where
$V_\alpha$, $U$, $V$ are $k$\+vector spaces.
 Given a morphism $\eta\:\R\rarrow\S$ of topological algebras over~$k$,
the functor of contraextension of scalars $E^\eta$ takes
$\R\ot\comp V$ to $\S\ot\comp V$.

 Now assume that $\R$ is a profinite-dimensional algebra over~$k$;
in other words, $\R=\cC^*$ is the dual vector space to
a coassociative, cocommutative, and counital coalgebra~$\cC$ over~$k$.
 Then one has $\R\ot\comp V\simeq\Hom_k(\cC,V)$, so our definition
of $\R$\+contramodules is equivalent to the classical definition
of $\cC$\+contramodules in~\cite{EM}.
 As explained above, the category of $\R$\+comodules is naturally
equivalent to the category of discrete $\R$\+modules, which are
clearly the same as $\cC$\+comodules in the conventional sense.
 In fact, these assertions do not depend on the assumption that
$\R$ is commutative and $\cC$ is cocommutative.

 The distinguished $\R$\+comodule $\cC(\R)$ is identified with the
comodule $\cC$ over the coalgebra $\cC$, and the cofree $\R$\+comodules
are the same as the cofree $\cC$\+comodules, i.~e., those of the form
$\cC\ot_k V$, where $V$ is a $k$\+vector space.
 The infinite products of cofree $\R$\+comodules and
free $\R$\+contramodules are described by the rules
$\prod_\alpha\cC\ot_k V_\alpha = \cC\ot_k\prod_\alpha V_\alpha$ and
$\prod_\alpha\Hom_k(\cC,V_\alpha) =
\Hom_k(\cC\;\prod_\alpha V_\alpha)$;
there are similar formulas for the infinite direct sums.

 Finally, the operations of comodule $\Hom$ (taking values in
contramodules), contratensor product $\ocn_\R$, cotensor product
$\oc_\R$, and cohomomorphisms $\Cohom_\R$, as defined in
Sections~\ref{hom-operations} and~\ref{co-operations}, correspond
to the operations of comodule $\Hom$, contratensor product
$\ocn_\cC$, cotensor product $\oc_\cC$, and cohomomorphisms
$\Cohom_\cC$, as defined in~\cite[Section~0.2]{Psemi}, under
the equivalences of categories $\R\contra\simeq\cC\contra$
and $\R\comod\simeq\cC\comod$.
 Given a morphism of profinite-dimensional algebras $\R\rarrow\S$
dual to a morphism of coalgebras $\cD\rarrow\cC$, where $\R=\cC^*$
and $\S=\cD^*$, the functors of contra/coextension of scalars defined
in Section~\ref{change-of-ring} correspond under our equivalences of
categories to the similar functors defined in
\cite[Section~7.1.2]{Psemi} (in the case of coalgebras over a field).

\Section{$\R$-Free and $\R$-Cofree wcDG-Modules}

 For the rest of the paper (with the exception of the appendices),
unless otherwise specified, $\R$ denotes a pro-Artinian commutative
topological local ring with the maximal ideal~$\m$ and the residue
field $k=\R/\m$.

 We fix once and for all a grading group datum $(\Gamma,\sigma,1)$
\cite[Section~1.1]{PP2}; unless a grading by the integers is
specifically mentioned, all our gradings (on algebras, modules,
complexes, etc.)\ are presumed to be $\Gamma$\+gradings.
 The conventions and notation of \emph{loc.\ cit.} are being used
in connection with the $\Gamma$\+gradings.

\subsection{$\R$-free graded algebras and modules}
\label{r-free-graded}
 A \emph{graded\/ $\R$\+contramodule} is a family of
$\R$\+contramodules indexed by elements of the grading
group~$\Gamma$.
 A \emph{free graded\/ $\R$\+contramodule} is
a graded $\R$\+contramodule whose grading components are free
$\R$\+contramodules.

 The operations of tensor product $\ot^\R$ and internal
homomorphisms $\Hom^\R$ are extended to graded
$\R$\+contramodules in the conventional way.
 Specifically, the components of $\P\ot^\R\Q$ are the contramodule
direct sums of the tensor products of the appropriate components
of $\P$ and $\Q$, that is $(\P\ot^\R\Q)^n = \bigoplus_{i+j=n}
\P^i\ot^\R\Q^j$, while the components of $\Hom^\R(\P,\Q)$ are
the direct products of the $\Hom$ contramodules between
the components of $\P$ and~$\Q$, i.~e., $\Hom_\R(\P,\Q)^n=
\prod_{j-i=n}\Hom_\R(\P^i,\Q^j)$.
 Notice that the operations $\ot^\R$ and $\Hom^\R$ preserve
the class of free (graded) $\R$\+contramodules
(see Sections~\ref{contramodules-sect}--\ref{nakayama-sect}).

 An \emph{$\R$\+free graded algebra} $\B$ is, by the definition,
a graded algebra object in the tensor category of free
$\R$\+contramodules; in other words, it is a free graded
$\R$\+contramodule endowed with an (associative, noncommutative)
homogeneous multiplication map $\B\ot^\R\B\rarrow\B$ and
a homogeneous unit map $\R\rarrow\B$ (or, equivalently,
a unit element $1\in\B^0$).

 An \emph{$\R$\+free graded left module} $\M$ over $\B$ is
a graded left module over $\B$ in the tensor category of free
$\R$\+contramodules, i.~e., a free graded $\R$\+contramodule
endowed with an (associative and unital) homogeneous
$\B$\+action map $\B\ot^\R\M\rarrow\M$.
 \ \emph{$\R$\+free graded right modules} $\N$ over $\B$ are
defined in the similar way.
 Alternatively, one can define $\B$\+module structures
on free graded $\R$\+contramodules in terms of the action maps
$\M\rarrow\Hom^\R(\B,\M)$, and similarly for~$\N$.
 The conventional super sign rule has to be used: e.~g.,
the two ways to represent a left action are related by the rule
$f_x(b)=(-1)^{|x||b}bx$, while for a right action it is
$f_y(b)=yb$, etc.

 In fact, the category of $\R$\+free graded $\B$\+modules is
enriched over the tensor category of (graded) $\R$\+contramodules,
so the abelian group of morphisms between two $\R$\+free
graded $\B$\+modules $\L$ and $\M$ is the underlying abelian
group of the degree-zero component of a certain naturally
defined (not necessarily free) graded $\R$\+contramodule
$\Hom_\B(\L,\M)$.
 This $\R$\+contramodule is constructed as the kernel of
the pair of morphisms of graded $\R$\+contramodules
$\Hom^\R(\L,\M)\birarrow\Hom^\R(\B\ot^\R\L\;\M)$ induced by
the actions of $\B$ in $\L$ and~$\M$.
 There is a nonobvious sign rule here, which will be needed when
working with differential modules, and which is different for
the left and the right graded modules;
see~\cite[Section~1.1]{Pkoszul}.

 The \emph{tensor product} $\N\ot_\B\M$ of an $\R$\+free graded right
$\B$\+module $\N$ and an $\R$\+free graded left $\B$\+module $\M$ 
is a (not necessarily free) graded $\R$\+contramodule
constructed as the cokernel of the pair of morphisms of
graded $\R$\+contramodules $\N\ot^\R\B\ot^\R\M\birarrow\N\ot^\R\M$.

 The additive category of $\R$\+free graded (left or right)
$\B$\+modules has a natural exact category structure:
a short sequence of $\R$\+free graded $\B$\+modules is said
to be exact if it is (split) exact as a short sequence of
free graded $\R$\+contramodules.
 The exact category of $\R$\+free graded $\B$\+modules admits
infinite direct sums and products, which are preserved by
the forgetful functors to the category of free graded
$\R$\+contramodules and preserve exact sequences.

 If a morphism of $\R$\+free graded $\B$\+modules has
an $\R$\+free kernel (resp., cokernel) in the category of
graded $\R$\+contramodules, then this kernel (resp., cokernel)
is naturally endowed with an $\R$\+free graded $\B$\+module
structure, making it the kernel (resp., cokernel) of
the same morphism in the additive category of $\R$\+free
graded $\B$\+modules.

 There are enough projective objects in the exact category
of $\R$\+free graded left $\B$\+modules; these are the direct
summands of the graded $\B$\+modules $\B\ot^\R\U$ freely
generated by free graded $\R$\+contramodules~$\U$.
 Similarly, there are enough injectives in this exact
category, and these are the direct summands of the graded
$\B$\+modules $\Hom^\R(\B,\V)$ cofreely cogenerated 
by free graded $\R$\+contramodules~$\V$ (for the sign
rule, see~\cite[Section~1.5]{Pkoszul}).

\begin{lem}  \label{r-free-hom-reduction}
 \textup{(a)} Let\/ $\P$ be a projective\/ $\R$\+free graded left\/
$\B$\+module and\/ $\M$ be an arbitrary\/ $\R$\+free graded left\/
$\B$\+module.
 Then the graded\/ $\R$\+contramodule\/ $\Hom_\B(\P,\M)$ is free,
and the functor of reduction modulo\/~$\m$ induces an isomorphism
of graded $k$\+vector spaces
$$
 \Hom_\B(\P,\M)/\m\Hom_\B(\P,\M)\simeq
 \Hom_{\B/\m\B}(\P/\m\P\;\M/\m\M).
$$ \par
\textup{(b)} Let\/ $\L$ be an arbitrary\/ $\R$\+free graded left\/
$\B$\+module and\/ $\J$ be an injective\/ $\R$\+free graded left\/
$\B$\+module.
 Then the graded\/ $\R$\+contramodule\/ $\Hom_\B(\L,\J)$ is free,
and the functor of reduction modulo\/~$\m$ induces an isomorphism of
graded $k$\+vector spaces
$$
 \Hom_\B(\L,\J)/\m\Hom_\B(\L,\J)\simeq
 \Hom_{\B/\m\B}(\L/\m\L\;\J/\m\J).
$$ \par
\textup{(c)} For any\/ $\R$\+free graded right\/ $\B$\+module\/ $\N$
and\/ $\R$\+free graded left\/ $\B$\+module\/ $\M$, there is a natural
isomorphism of graded $k$\+vector spaces
$$
 (\N\ot_\B\M)/\m(\N\ot_\B\M)\simeq(\N/\m\N)\ot_{\B/\m\B}(\M/\m\M).
$$
 For any $\R$\+free graded right\/ $\B$\+module\/ $\N$ and any
projective\/ $\R$\+free graded left\/ $\B$\+module $\P$, the graded\/
$\R$\+contramodule\/ $\N\ot_\B\P$ is free.
\end{lem}

\begin{proof}
 Clearly, $\B/\m\B$ is a graded $k$\+algebra, and $\M/\m\M$ is
a graded left $\B/\m\B$\+module for any $\R$\+free graded
left $\B$\+module $\M$, since the functor of reduction modulo~$\m$
preserves tensor products of $\R$\+contramodules.
 For any $\R$\+free left $\B$\+modules $\L$ and $\M$, this functor
induces a natural map of graded abelian groups
{\hbadness=1400
$$
 \Hom_\B(\L,\M)\lrarrow\Hom_{\B/\m\B}(\L/\m\L\;\M/\m\M).
$$
 One} easily checks that this map is a morphism of graded
$\R$\+contramodules, hence there is a natural morphism of
graded $k$\+vector spaces
$$
 \Hom_\B(\L,\M)/\m\Hom_\B(\L,\M)\lrarrow
 \Hom_{\B/\m\B}(\L/\m\L\;\M/\m\M).
$$

 Now if $\L=\B\ot^\R\U$ is a graded $\B$\+module freely generated
by a free $\R$\+contramodule $\U$, then
$\Hom_\B(\L,\M)\simeq\Hom^\R(\U,\M)$ and $\L/\m\L\simeq
\B/\m\B\ot_k\U/\m\U$, hence $\Hom_{\B/\m\B}(\L/\m\L\;\M/\m\M)
\simeq\Hom_k(\U/\m\U\;\M/\m\M)$, and the desired isomorphism
follows from Lemma~\ref{nakayama-reduct-products} and the related
result of Section~\ref{hom-operations}.

 Similarly, if $\M=\Hom^\R(\B,\V)$ is a graded $\B$\+module cofreely
cogenerated by a free $\R$\+contramodule $\V$, then
$\Hom_\B(\L,\M)\simeq\Hom^\R(\L,\V)$ and $\M/\m\M\simeq
\Hom_k(\B/\m\B\;\V/\m\V)$, hence $\Hom_{\B/\m\B}(\L/\m\L\;\M/\m\M)
\simeq\Hom_k(\L/\m\L\;\V/\m\V)$ and the desired isomorphism follows.

 Parts~(a) and~(b) are proven.
 The proof of the second assertion of part~(c), based on the natural
isomorphism $\N\ot_\B(\B\ot^\R\U)\simeq \N\ot^\R\U$, is similar;
and the first assertion is easy, since the functor of reduction
modulo~$\m$ preserves tensor products and cokernels in the category
of $\R$\+contramodules.
\end{proof}

\begin{lem} \label{r-free-proj-inj-reduction}
 An\/ $\R$\+free graded module\/ $\M$ over an\/ $\R$\+free graded
algebra\/ $\B$ is projective (resp., injective) if and only if
the graded module\/ $\M/\m\M$ over the graded $k$\+algebra\/
$\B/\m\B$ is projective (resp., injective).
\end{lem}

\begin{proof}
 The ``only if'' part follows from the above description of
the projective and injective $\R$\+free graded $\B$\+modules.
 Conversely, assume that $\M/\m\M$ is a projective
$\B/\m\B$\+module.
 Then it is the image of a homogeneous idempotent endomorphism of
a $\B/\m\B$\+module freely generated by some graded
$k$\+vector space~$U$.
 The latter can be presented as $\U/\m\U$, where $\U$ is some
free graded $\R$\+contramodule.
 Set $\F=\B\ot^\R\U$.
 According to Lemma~\ref{r-free-hom-reduction}(a), the graded
$k$\+algebra $\Hom_{\B/\m\B}(\F/\m\F\;\F/\m\F)$ can be obtained by
reducing modulo~$\m$ the $\R$\+free graded algebra $\Hom_\B(\F,\F)$.

 Using Lemma~\ref{idempotent-lifting}, we can lift our idempotent
endomorphism of $\F/\m\F$ to a homogeneous idempotent endomorphism
of~$\F$.
 Let $\P$ be the image of the latter endomorphism; it is
a projective $\R$\+free graded $\B$\+module.
 The graded $k$\+vector space $\Hom_{\B/\m\B}(\P/\m\P\;\M/\m\M)$
can also be obtained by reducing the free graded $\R$\+contramodule
$\Hom_\B(\P,\M)$; in particular, our isomorphism $\P/\m\P\simeq
\M/\m\M$ can be lifted to a morphism of $\R$\+free graded
$\B$\+modules $\P\rarrow\M$.
 The latter, being an isomorphism modulo~$\m$, is consequently
itself an isomorphism.
 The case of injective graded modules is similar.
\end{proof}

\begin{cor}  \label{r-free-homol-dim}
 The homological dimension of the exact category of\/ $\R$\+free
graded left\/ $\B$\+modules does not exceed that of the abelian
category of graded left\/ $\B/\m\B$\+modules.
\end{cor}

\begin{proof}
 This follows immediately from the observation that the exact
category of $\R$\+free graded $\B$\+modules has enough projective
or injective objects together with
Lemma~\ref{r-free-proj-inj-reduction}.
\end{proof}

\subsection{Absolute derived category of $\R$-free CDG-modules}
\label{r-free-absolute}
 Let $\U$ and $\V$ be graded $\R$\+contramodules endowed
with homogeneous $\R$\+contramodule endomorphisms (differentials)
$d_\U\:\U\rarrow\U$ and $d_\V\:\V\rarrow\V$ of degree~$1$.
 Then the graded $\R$\+contramodules $\U\ot^\R\V$ and
$\Hom^\R(\U,\V)$ are endowed with their differentials~$d$
defined by the conventional rules $d(u\ot v)=d_\U(u)\ot v +
(-1)^{|u|}u\ot d_\V(v)$ and $d(f)(u) = 
d_\V(f(u)) - (-1)^{|f|}f(d_\U(u))$.

 An \emph{odd derivation}~$d$ (of degree~1) of an $\R$\+free
graded algebra $\B$ is a differential ($\R$\+contramodule
endomorphism) such that the multiplication map $\B\ot^\R\B
\rarrow\B$ forms a commutative diagram with the differential~$d$
on $\B$ and the induced differential on $\B\ot^\R\B$.
 An \emph{odd derivation}~$d_\M$ of an $\R$\+free graded
left $\B$\+module $\M$ \emph{compatible with} the derivation~$d$
on $\B$ is a differential ($\R$\+contramodule endomorphism of
degree~$1$) such that the action map $\B\ot^\R\M\rarrow\M$
forms a commutative diagram with $d_\M$ and the differential
on $\B\ot^\R\M$ induced by $d$ and~$d_\M$.
 Odd derivations of $\R$\+free graded right $\B$\+modules
are defined similarly.
 Alternatively, one requires the action map $\M\rarrow\Hom^\R(\B,\M)$
to commute with the differentials.

 An \emph{$\R$\+free CDG\+al\-gebra} is, by the definition,
a CDG\+algebra object in the tensor category of free
$\R$\+contra\-modules; in other words, it is an $\R$\+free graded
algebra endowed with an odd derivation $d\:\B\rarrow\B$
of degree~$1$ and a curvature element $h\in\B^2$ satisfying
the conventional equations $d^2(b)=[h,b]$ for all $b\in\B$
and $d(h)=0$.
 Morphisms $\B\rarrow\A$ of $\R$\+free CDG\+algebras are defined
as pairs $(f,a)$, with $f\:\B\rarrow\A$ being a morphism of
$\R$\+free graded algebras and $a\in\A^1$, satisfying the conventional
equations (see~\cite[Section~3.1]{Pkoszul}
or~\cite[Section~1.1]{Psing}).
 
 An \emph{$\R$\+free left CDG\+module} $\M$ over $\B$ is, by
the definition, an $\R$\+free graded left $\B$\+module endowed with
an odd derivation $d_\M\:\M\rarrow\M$ of degree~$1$ compatible with
the derivation~$d$ on $\B$ and satisfying the equation
$d_\M^2(x)=hx$ for all $x\in\M$.
 The definition of an \emph{$\R$\+free right CDG\+module} $\N$ over
$\B$ is similar; the only nonobvious difference is that the equation
for the square of the differential has the form $d_\N^2(y)=-yh$
for all $y\in\N$.

 $\R$\+free left (resp., right) CDG\+modules over $\B$ naturally form
a DG\+category, which we will denote by $\B\mod\Rfr$
(resp., $\modrRfr\B$).
 In fact, these DG\+categories are enriched over the tensor
category of (complexes of) $\R$\+contramodules, so the complexes
of morphisms in $\B\mod\Rfr$ and $\modrRfr\B$ are the underlying
complexes of abelian groups for naturally defined complexes of
$\R$\+contramodules.

 The underlying graded $\R$\+contramodules of these complexes
were defined in Section~\ref{r-free-graded}.
 Given two left CDG\+modules $\L$ and $\M\in\B\mod\Rfr$,
the differential in the graded $\R$\+contramodule $\Hom_\B(\L,\M)$
is defined by the conventional formula $d(f)(x)=d_\M(f(x)) -
(-1)^{|f|}f(d_\L(x))$; one easily checks that $d^2(f)=0$.
 For any two right CDG\+modules $\K$ and $\N\in\modrRfr\B$,
the complex of $\R$\+contramodules $\Hom_{\B^\op}(\K,\N)$ is
defined in the similar way.

 Passing to the zero cohomology of the complexes of morphisms,
we construct the homotopy categories of $\R$\+free CDG\+modules
$H^0(\B\mod\Rfr)$ and $H^0(\modrRfr\B)$.
 Even though these are also naturally enriched over
$\R$\+contramodules, we will mostly consider them as conventional
categories with abelian groups of morphisms.
 Since the DG\+categories $\B\mod\Rfr$ and $\modrRfr\B$ have
shifts, twists, and infinite direct sums and products, their
homotopy categories are triangulated categories with infinite
direct sums and products.

\begin{rem}  \label{contra-enriched}
 Our reason for hesitating to work with triangulated categories
enriched over $\R$\+contramodules \emph{as} our main setting is that
we would like to be able to take the Verdier quotient
categories freely and without always worrying about the existence of
adjoint functors to the localization functor.
 The problem is that the Verdier localization acts on morphisms
as a kind of inductive limit, and inductive limits of
$\R$\+contramodules are \emph{not} well-behaved; in particular,
they do not commute with the forgetful functors to abelian groups.
 When the localization functor \emph{does} have an adjoint, though,
the inductive limits involved in the quotient category
construction are \emph{stabilizing} inductive limits, which
have all the good properties, of course.

 So we generally consider triangulated categories with abelian
groups of morphisms; alternatively, our triangulated categories
can be viewed as being \emph{$\R$\+linear}, i.~e., having
(abstract, nontopological) $\R$\+modules of morphisms.
 The triangulated categories that we are most interested in will
be eventually presented as appropriate resolution subcategories
of the homotopy categories and consequently endowed with
DG\+enhancements; these will be naturally enriched over (complexes
of) $\R$\+contramodules, and in fact, even \emph{free}
$\R$\+contramodules.
 In our notation, these will be spoken of as the ``derived
functors $\Ext$'' defined on our triangulated categories and
taking values in the homotopy category of free $\R$\+contramodules.
\end{rem}

 The tensor product $\N\ot_\B\M$ of a right $\R$\+free
CDG\+module $\N$ and a left $\R$\+free CDG\+module $\M$ over $\B$
is a complex of $\R$\+contramodules obtained by endowing
the graded $\R$\+contramodule $\N\ot_\B\M$ constructed in
Section~\ref{r-free-graded} with the conventional differential
$d(y\ot x)=d_\N(y)\ot x + (-1)^{|y|}y\ot d_\M(x)$.
 The tensor product of $\R$\+free CDG\+modules over $\B$ is
a triangulated functor of two arguments
$$
 \ot_\B\:H^0(\modrRfr\B)\times H^0(\B\mod\Rfr)\lrarrow
 H^0(\R\contra),
$$
where $H^0(\R\contra)$ denotes, by an abuse of notation,
the homotopy category of complexes of $\R$\+contramodules.

 An $\R$\+free left CDG\+module over $\B$ is said to be
\emph{absolutely acyclic} if it belongs to the minimal thick
subcategory of the homotopy category $H^0(\B\mod\Rfr)$ containing
the totalizations of short exact sequences of $\R$\+free
CDG\+modules over~$\B$.
 The quotient category of $H^0(\B\mod\Rfr)$ by the thick
subcategory of absolutely acyclic $\R$\+free CDG\+modules is called
the \emph{absolute derived category} of $\R$\+free left
CDG\+modules over~$\B$ and denoted by $\sD^\abs(\B\mod\Rfr)$.
 The absolute derived category of $\R$\+free right CDG\+modules
over $\B$, denoted by $\sD^\abs(\modrRfr\B)$, is defined
similarly (see~\cite[Sections~1.2 and~3.3]{Pkoszul} or
\cite[Section~3.2]{PP2}).

 An $\R$\+free left CDG\+module over $\B$ is said to be
\emph{contraacyclic} if it belongs to the minimal triangulated
subcategory of the homotopy category $H^0(\B\mod\Rfr)$
containing the totalizations of short exact sequences of
$\R$\+free CDG\+modules over $\B$ and closed under infinite
products.
 The quotient cateogry of $H^0(\B\mod\Rfr)$ by the thick
subcategory of contraacyclic $\R$\+free CDG\+modules is called
the \emph{contraderived category} of $\R$\+free left
CDG\+modules over $\B$ and denoted by $\sD^\ctr(\B\mod\Rfr)$.
 The (similarly defined) contraderived category of $\R$\+free
right CDG\+modules over~$\B$ is denoted by $\sD^\ctr(\modrRfr\B)$.

 An $\R$\+free left CDG\+module over $\B$ is said to be
\emph{coacyclic} if it belongs to the minimal triangulated
subcategory of the homotopy category $H^0(\B\mod\Rfr)$
containing the totalizations of short exact sequences of
$\R$\+free CDG\+modules over $\B$ and closed under infinite
direct sums.
 The quotient category of $H^0(\B\mod\Rfr)$ by the thick
subcategory of coacyclic $\R$\+free CDG\+modules is called
the \emph{coderived category} of $\R$\+free left
CDG\+modules over $\B$ and denoted by $\sD^\co(\B\mod\Rfr)$.
 The coderived category of $\R$\+free right CDG\+modules over~$\B$
is denoted by $\sD^\co(\modrRfr\B)$.

 Denote by $\B\mod\Rfr_\proj\subset\B\mod\Rfr$ the full
DG\+subcategory formed by all the $\R$\+free CDG\+modules
over $\B$ whose underlying $\R$\+free graded $\B$\+modules
are projective.
 Similarly, let $\B\mod\Rfr_\inj\subset\B\mod\Rfr$ be the full
DG\+subcategory formed by all the $\R$\+free CDG\+modules
over $\B$ whose underlying $\R$\+free graded $\B$\+modules
are injective.
 The corresponding homotopy categories are denoted by
$H^0(\B\mod\Rfr_\proj)$ and $H^0(\B\mod\Rfr_\inj)$, respectively.

\begin{lem} \label{contractible-reduction}
\textup{(a)} Let\/ $\P$ be a CDG\+module from\/ $\B\mod\Rfr_\proj$.
 Then\/ $\P$ is contractible (i.~e., represents a zero object in
$H^0(\B\mod\Rfr)$) if and only if the CDG\+module\/ $\P/\m\P$
over the CDG\+algebra $\B/\m\B$ over~$k$ is contractible. \par
\textup{(b)} Let\/ $\J$ be a CDG\+module from\/ $\B\mod\Rfr_\inj$.
 Then\/ $\J$ is contractible if and only if the CDG\+module\/
$\J/\m\J$ over the CDG\+algebra $\B/\m\B$ is contractible.
\end{lem}

\begin{proof}
 According to Lemma~\ref{r-free-hom-reduction}(a\+b), we have
$\R$\+free DG\+algebras $\Hom_\B(\P,\P)$ and $\Hom_\B(\J,\J)$
whose reductions modulo~$\m$ are naturally isomorphic to
the DG\+algebras $\Hom_{\B/\m\B}(\P/\m\P\;\P/\m\P)$ and
$\Hom_{\B/\m\B}(\J/\m\J\;\J/\m\J)$.
 If, e.~g., the CDG\+module $\P/\m\P$ over $\B/\m\B$ is
contractible, then the DG\+algebra $\Hom_{\B/\m\B}(\P/\m\P\;\P/\m\P)$
over~$k$ is acyclic, and it follows by virtue of
Lemma~\ref{nakayama-acycl-contract} that the $\R$\+free
DG\+algebra $\Hom_\B(\P,\P)$ is contractible as a complex of
$\R$\+contramodules.
 Hence, in particular, the complex $\Hom_\B(\P,\P)$ is acyclic
and $\P$ is a contractible CDG\+module over~$\B$.
\end{proof}

\begin{thm}  \label{r-free-orthogonality}
 Let\/ $\B$ be an\/ $\R$\+free CDG\+algebra.  Then \par
\textup{(a)} for any CDG\+module\/ $\P\in H^0(\B\mod\Rfr_\proj)$
and any contraacyclic\/ $\R$\+free left CDG\+module $\M$ over $\B$,
the complex of\/ $\R$\+contramodules\/ $\Hom_\B(\P,\M)$ is
contractible; \par
\textup{(b)} for any coacyclic $\R$\+free left CDG\+module $\L$
over $\B$ and any CDG\+module $\J\in H^0(\B\mod\Rfr_\inj)$,
the complex of\/ $\R$\+contramodules\/ $\Hom_\B(\L,\J)$
is contractible.
\end{thm}

\begin{proof}
 By the first assertions of Lemma~\ref{r-free-hom-reduction}(a\+b),
the complexes $\Hom_\B(\P,\N)$ and $\Hom_\B(\K,\J)$ are complexes
of free $\R$\+contramodules for any $\R$\+free CDG\+module $\N$
or $\K$ over~$\B$.
 To any short exact sequence of $\R$\+free CDG\+modules $\N$ or $\K$,
short exact sequences of such complexes of $\R$\+contramodules
$\Hom_\B$ correspond; and the totalization of a short exact 
sequence of complexes of free $\R$\+contramodules is a contractible
complex of $\R$\+contramodules.

 Finally, to infinite products of the CDG\+modules $\N$ or
infinite direct sums of the CDG\+modules $\K$, our functors
$\Hom_\B$ assign infinite products of complexes of
$\R$\+contramodules, and to cones of morphisms of the CDG\+modules
$\N$ or $\K$ they assign the (co)cones of morphisms of complexes
of $\R$\+contramodules.
 The class of contractible complexes of $\R$\+contramodules is
closed under both operations.
 (Cf.~\cite[Theorem~3.5 and Remark~3.5]{Pkoszul}, and
Theorem~\ref{non-adj-orthogonality} below.)
\end{proof}

\begin{thm} \label{r-free-absolute-derived}
 Let\/ $\B$ be an\/ $\R$\+free CDG\+algebra.
 Assume that the exact category of\/ $\R$\+free graded left\/
$\B$\+modules has finite homological dimension.
 Then the compositions of natural functors
$$
 H^0(\B\mod\Rfr_\proj)\lrarrow H^0(\B\mod\Rfr)\lrarrow
 \sD^\abs(\B\mod\Rfr)
$$
and
$$
 H^0(\B\mod\Rfr_\inj)\lrarrow H^0(\B\mod\Rfr)\lrarrow
\sD^\abs(\B\mod\Rfr)
$$
are equivalences of triangulated categories.
\end{thm}

\begin{proof}
 The proof is completely analogous to that in the case of
a CDG\+algebra over a field or a CDG\+ring
\cite[Theorem~3.6]{Pkoszul}, and is a particular case
of the general scheme of~\cite[Remark~3.6]{Pkoszul}.
 The only aspect of these arguments that needs some additional
comments here is the constructions of the $\R$\+free CDG\+module
$G^+(\L)$ over $\B$ freely generated by an $\R$\+free graded
$\B$\+module $\L$ and the $\R$\+free CDG\+module $G^-(\L)$
cofreely cogenerated by~$\L$.

 The free graded $\R$\+contramodule $G^+(\L)$ is defined by
the rule $G^+(\L)^i=\L^i\oplus\L^{i-1}$, the elements of
$G^+(\L)^i$ are denoted formally by $x+dy$, where $x\in\L^i$
and $y\in\L^{i-1}$, and the differential and the left action of $\B$
in $G^+(\L)$ are defined by the formulas from~\cite[proof of
Theorem~3.6]{Pkoszul} (expressing the CDG\+module axioms).
 Similarly, one sets $G^-(\L)^i=\L^{i+1}\oplus\L^i$ as a graded
$\R$\+contramodule, and defines the differential and
the action of $\B$ in $G^-(\L)$ by the formulas from~\cite{Pkoszul}.
\end{proof}

\begin{cor} \label{r-free-abs-acycl-reduction}
 Let\/ $\B$ be an\/ $\R$\+free CDG\+algebra.
 Assume that the exact category of\/ $\R$\+free graded left\/
$\B$\+modules has finite homological dimension.
 Then an\/ $\R$\+free left CDG\+module\/ $\M$ over\/ $\B$ is
absolutely acyclic if and only if the CDG\+module\/
$\M/\m\M$ over the CDG\+algebra\/ $\B/\m\B$ over the field~$k$
is absolutely acyclic.
\end{cor}

\begin{proof}
 The ``only if'' part is clear, because the functor of reduction
modulo~$\m$ preserves exact triples of $\R$\+free CDG\+modules
over~$\B$.
 To prove ``if'', notice that by
Theorems~\ref{r-free-orthogonality}\+-\ref{r-free-absolute-derived},
$\R$\+free CDG\+modules over $\B$ with projective underlying graded
$\B$\+modules and absolutely acyclic $\R$\+free CDG\+modules over
$\B$ form a semiorthogonal decomposition of the homotopy category of
$\R$\+free CDG\+modules $H^0(\B\mod\Rfr)$.
 So it suffices to show that for any CDG\+module
$\P\in H^0(\B\mod\Rfr_\proj)$ the complex $\Hom_\B(\P,\M)$
is acyclic (as a complex of abelian groups).

 By Lemma~\ref{r-free-hom-reduction}(a), this complex is
the underlying complex of abelian groups to a complex of
free $\R$\+contramodules, and reducing this complex of
$\R$\+contramodules modulo~$\m$ one obtains the complex of
$k$\+vector spaces $\Hom_{\B/\m\B}(\P/\m\P\;\M/\m\M)$.
 Since the latter complex is acyclic, the desired assertion
follows from Lemma~\ref{nakayama-acycl-contract}.
 Alternatively, one can use Theorem~\ref{r-free-absolute-derived}
together with Lemma~\ref{contractible-reduction}.
\end{proof}

 The following theorem is to be compared with
Theorem~\ref{r-free-co-derived-thm} below.

\begin{thm} \label{r-free-star-conditions}
 Let\/ $\B$ be an\/ $\R$\+free CDG\+algebra.
 Then the composition of natural functors
$$
 H^0(\B\mod\Rfr_\proj)\lrarrow H^0(\B\mod\Rfr)\lrarrow
 \sD^\ctr(\B\mod\Rfr)
$$
is an equivalence of triangulated categories whenever the graded
$k$\+algebra\/ $\B/\m\B$ satisfies the condition~\textup{(${*}{*}$)}
from \textup{\cite[\emph{Section}~3.8]{Pkoszul}}.
 The composition of natural functors
$$
 H^0(\B\mod\Rfr_\inj)\lrarrow H^0(\B\mod\Rfr)\lrarrow
 \sD^\co(\B\mod\Rfr)
$$
is an equivalence of triangulated categories whenever the graded
$k$\+algebra\/ $\B/\m\B$ satisfies the condition~\textup{(${*}$)}
from \textup{\cite[\emph{Section}~3.7]{Pkoszul}}.
\end{thm}

\begin{proof}
 The semiorthogonality being known from
Theorem~\ref{r-free-orthogonality}, it remains to show the existence
of resolutions.
 These are provided by the constructions of
\cite[Sections~3.7\+-3.8]{Pkoszul} applied to $\R$\+free
CDG\+modules over $\B$.
 The injectivity/projectivity properties
of the $\R$\+free graded $\B$\+modules so obtained follow from
Lemma~\ref{r-free-proj-inj-reduction}.
\end{proof}

 Clearly, the assertion of Corollary~\ref{r-free-abs-acycl-reduction}
holds with the absolute acyclicity property replaced with the contra-
or coacyclicity whenever the corresponding assertion of
Theorem~\ref{r-free-star-conditions} holds.

\medskip
 Let $f=(f,a)\:\B\rarrow\A$ be a morphism of $\R$\+free CDG\+algebras.
 Then any $\R$\+free graded module over $\A$ can be endowed with
a graded $\B$\+module structure via~$f$, and any homogeneous morphism
(of any degree) between graded $\A$\+modules can be also considered
as a homogeneous morphism (of the same degree) between graded
$\B$\+modules.
 With any $\R$\+free left CDG\+module $(\M,d_\M)$ over $\A$ one can
associate an $\R$\+free left CDG\+module $(\M,d'_\M)$ over $\B$ with
the modified differential $d'_\M(x)=d_\M(x)+ax$.
 Similarly, for any $\R$\+free right CDG\+module $(\N,d_\N)$ over $\A$
the formula $d'_\N(y)=d_\N(y)-(-1)^{|y|}ya$ defines a modified
differential on $\N$ making $(\N,d'_\N)$ an $\R$\+free right
CDG\+module over~$\B$.
 We have constructed the DG\+functors of restriction of scalars
$R_f\:\A\mod\Rfr\rarrow\B\mod\Rfr$ and $\modrRfr\A\rarrow
\modrRfr\B$; passing to the homotopy categories, we obtain
the triangulated functors
$$
 R_f\:H^0(\A\mod\Rfr)\lrarrow H^0(\B\mod\Rfr)
$$
and $H^0(\modrRfr\A)\rarrow H^0(\modrRfr\B)$.

\subsection{Semiderived category of $\R$-free wcDG-modules}
\label{r-free-semi}
 A \emph{weakly curved differential graded algebra}, or
\emph{wcDG\+algebra} over~$\R$ is, by the definition,
an $\R$\+free CDG\+algebra $(\A,d,h)$ with the curvature element
$h$ belonging to~$\m\A^2$.
 A \emph{morphism of wcDG\+algebras} over $\R$ (\emph{wcDG\+morphism})
$\B\rarrow\A$ is a morphism $(f,a)$ between $\B$ and $\A$
considered as $\R$\+free CDG\+algebras such that
the change-of-connection element $a$ belongs to~$\m\A^1$.
 So the reduction modulo~$\m$ of a wcDG\+algebra $\A$ over $\R$
is a DG\+algebra $\A/\m\A$ over the field~$k$, and the reduction
of a wcDG\+algebra morphism is a conventional DG\+algebra morphism.

 CDG\+modules over a wcDG\+algebra will be referred to as
\emph{wcDG\+modules}.
 All the above definitions, constructions, and notation related
to $\R$\+free CDG\+modules will be applied to the particular case
of $\R$\+free wcDG\+modules as well.

 An $\R$\+free wcDG\+module $\M$ over a wcDG\+algebra $\A$ is
said to be \emph{semiacyclic} if the complex of $k$\+vector
spaces $\M/\m\M$ (which is in fact a DG\+module over
the DG\+algebra $\A/\m\A$) is acyclic.
 In particular, it follows from Lemma~\ref{nakayama-acycl-contract}
that when $\A$ itself is an $\R$\+free DG\+algebra
(i.~e., the curvature element~$h$ of $\A$ vanishes),
an $\R$\+free CDG\+module $\M$ over $\A$ is semiacyclic if
and only if $\M$ is contractible as a complex of free
$\R$\+contramodules.

 Clearly, the property of semiacyclicity of an $\R$\+free
wcDG\+module over $\A$ is preserved by shifts, cones and
homotopy equivalences (i.~e., isomorphisms in $H^0(\A\mod\Rfr)$).
 The class of semiacyclic $\R$\+free wcDG\+modules is also
closed with respect to infinite direct sums and products.
 The quotient category of the homotopy category 
$H^0(\A\mod\Rfr)$ by the thick subcategory of semiacyclic
$\R$\+free wcDG\+modules is called the \emph{semiderived
category} of $\R$\+free left wcDG\+modules over $\A$ and
denoted by $\sD^\si(\A\mod\Rfr)$.
 The semiderived category of $\R$\+free right wcDG\+modules
$\sD^\si(\modrRfr\A)$ is defined similarly.

 Every absolutely acyclic $\R$\+free wcDG\+module, and in fact
every contraacyclic and every coacyclic $\R$\+free wcDG\+module
is semiacyclic.
 Indeed, the functor of reduction modulo~$\m$ preserves the absolute
acyclicity, contra- and coacyclicity properties, and any
contraacyclic or coacyclic DG\+module over $\A/\m\A$ is acyclic.

 Notice that any acyclic complex of free $\R$\+contramodules is
contractible when (the category of contramodules over)
$\R$ has finite homological dimension.
 In this case, we call semiacyclic $\R$\+free wcDG\+modules over
$\A$ over simply \emph{acyclic}, and refer to the semiderived
category of $\R$\+free wcDG\+modules as
their \emph{derived category}.

 An $\R$\+free wcDG\+module $\P$ over $\A$ is called
\emph{homotopy projective} if the complex $\Hom_\A(\P,\M)$
is acyclic for any semiacyclic $\R$\+free wcDG\+module
$\M$ over~$\A$.
 Similarly, an $\R$\+free wcDG\+module $\J$ over $\A$ is
called \emph{homotopy injective} if the complex 
$\Hom_\A(\M,\J)$ is acyclic for any semiacyclic $\R$\+free
wcDG\+module $\M$ over~$\A$.
 The full triangulated subcategories in $H^0(\A\mod\Rfr)$ formed
by the homotopy projective and homotopy injective $\R$\+free
wcDG\+modules are denoted by $H^0(\A\mod\Rfr)_\proj$ and
$H^0(\A\mod\Rfr)_\inj$, respectively.
 (Notice the difference with the notation for graded
projective/injective CDG\+modules introduced in
Section~\ref{r-free-absolute}.)

\begin{lem} \label{homotopy-proj-inj-reduction}
\textup{(a)}
 A wcDG\+module\/ $\P\in \A\mod\Rfr_\proj$ is homotopy
projective if and only if the DG\+module\/ $\P/\m\P$ over
the DG\+algebra\/ $\A/\m\A$ over the field~$k$ is homotopy
projective. \par
\textup{(b)}
 A wcDG\+module $\J\in \A\mod\Rfr_\inj$ is homotopy
injective if and only if the DG\+module\/ $\J/\m\J$ over
the DG\+algebra\/ $\A/\m\A$ is homotopy injective.
\end{lem}

\begin{proof}
 We will prove the ``if'' assertions here.
 Then it will follow from the proof of
Theorem~\ref{r-free-semi-resolutions} below that any
homotopy projective/injective $\R$\+free wcDG\+module $\M$
is homotopy equivalent to a graded projective/injective
$\R$\+free wcDG\+module with a homotopy projective/injective
reduction modulo~$\m$; hence the reduction $\M/\m\M$ of
the wcDG\+module $\M$, being homotopy equivalent to
a homotopy projective/injective DG\+module, is itself
homotopy projective/injective.

 Indeed, let $\N$ be a semiacyclic wcDG\+module over~$\A$.
 By Lemma~\ref{r-free-hom-reduction}(a\+b), the complexes
$\Hom_\A(\P,\N)$ and $\Hom_\A(\N,\J)$ are the underlying
complexes of abelian groups to complexes of free
$\R$\+contramodules whose reductions modulo~$\m$ are
the complexes of $k$\+vector spaces
$\Hom_{\A/\m\A}(\P/\m\P\;\N/\m\N)$ and
$\Hom_{\A/\m\A}(\N/\m\N\;\J/\m\J)$.
 Since the latter two complexes are acyclic by assumption,
so are the former two complexes
(see Lemma~\ref{nakayama-acycl-contract}).
\end{proof}

 Introduce the notation $H^0(\A\mod\Rfr_\proj)_\proj$ for
the intersection $H^0(\A\mod\Rfr)_\proj\cap H^0(\A\mod\Rfr_\proj)$,
i.~e., the homotopy category of homotopy projective $\R$\+free
left CDG\+modules over $\A$ whose underlying graded $\A$\+modules
are also projective.
 Similarly, let $H^0(\A\mod\Rfr_\inj)_\inj = H^0(\A\mod\Rfr)_\inj
\cap H^0(\A\mod\Rfr_\inj)$ denote the homotopy category of homotopy
injective $\R$\+free left CDG\+modules over $\A$ whose underlying
graded $\A$\+modules are injective.

\begin{thm} \label{r-free-semi-resolutions}
\textup{(a)} For any wcDG\+algebra\/ $\A$ over\/ $\R$,
the compositions of functors
$$
H^0(\A\mod\Rfr)_\proj\lrarrow H^0(\A\mod\Rfr)\lrarrow
\sD^\si(\A\mod\Rfr)
$$
and
$$
H^0(\A\mod\Rfr_\proj)_\proj\lrarrow H^0(\A\mod\Rfr)\lrarrow
\sD^\si(\A\mod\Rfr)
$$
are equivalences of triangulated categories. \par
\textup{(b)} For any wcDG\+algebra\/ $\A$ over\/ $\R$,
the compositions of functors
$$
H^0(\A\mod\Rfr)_\inj\lrarrow
H^0(\A\mod\Rfr)\lrarrow\sD^\si(\A\mod\Rfr)
$$
and
$$
H^0(\A\mod\Rfr_\inj)_\inj\lrarrow
H^0(\A\mod\Rfr)\lrarrow\sD^\si(\A\mod\Rfr).
$$
are equivalences of triangulated categories.
\end{thm}

\begin{proof}
 It suffices to construct for any $\R$\+free wcDG\+module $\M$
over $\A$ two closed morphisms of $\R$\+free wcDG\+modules $\P\rarrow
\M\rarrow\J$ with semiacyclic cones such that the wcDG\+module $\P$
is homotopy projective and the wcDG\+module $\J$ is homotopy
injective.
 In fact, we will have $\P\in \A\mod\Rfr_\proj$ with
$\P/\m\P$ homotopy projective and $\J\in \A\mod\Rfr_\inj$
with $\J/\m\J$ homotopy injective over $\A/\m\A$.
 Then we will use the ``if'' assertions of
Lemma~\ref{homotopy-proj-inj-reduction}, and
have proven the ``only if'' assertions.

 To obtain the wcDG\+modules $\P$ and $\J$, we will use a curved
version of the bar- and cobar-resolutions.
 Consider the bigraded $\A$\+module whose component of\
the internal grading~$n\ge1$ is the graded $\A$\+module
$\A^{\ot^\R\,n}\ot^\R\M$ with the action of $\A$ given by
the rule $b(b_1\ot\dotsb\ot b_n\ot x)=(-1)^{|b|(n-1)}bb_1
\ot b_2\ot\dotsb\ot b_n\ot x$ for $b$, $b_s\in\A$ and $x\in\M$.
 Define the differentials~$\d$, $d$, and~$\delta$ on this
$\A$\+module by the rules
\begin{multline*}
 \d(b_1\ot\dotsb\ot b_n\ot x) = b_1b_2\ot b_3\ot\dotsb b_n\ot x
 - b_1\ot b_2b_3\ot b_4\ot\dotsb\ot b_n\ot x \\ +\dotsb
 + (-1)^n b_1\ot\dotsb\ot b_{n-2}\ot b_{n-1}b_n \ot x
 + (-1)^{n+1} b_1\ot\dotsb\ot b_{n-1}\ot b_nx,
\end{multline*}
\begin{multline*}
 (-1)^{n-1} d(b_1\ot\dotsb\ot b_n\ot x) =
 d(b_1)\ot b_2\ot\dotsb\ot b_n\ot x \\
 + (-1)^{|b_1|} b_1\ot d(b_2)\ot\dotsb\ot b_n\ot x + \dotsb
 + (-1)^{|b_1|+\dotsb+|b_n|} b_1\ot\dotsb b_n \ot d(x),
\end{multline*}
and
\begin{multline*}
 \delta(b_1\ot\dotsb\ot b_n\ot x) = b_1\ot h\ot b_2\ot
 \dotsb\ot b_n\ot x - \dotsb \\ + (-1)^n b_1\ot\dotsb\ot b_{n-1}\ot
 h\ot b_n\ot x + (-1)^{n+1} b_1\ot\dotsb\ot b_n\ot h\ot x,
\end{multline*}
and endow it with the total differential $\d+d+\delta$.
 Totalizing our bigraded $\A$\+module by taking infinite
direct sums (in the category of $\R$\+contramodules)
along the diagonals where the difference $i=|b_1|+\dotsb+|b_n|+|x|
-n+1$ is constant, we obtain the desired wcDG\+module $\P$ 
over~$\A$.
 The closed morphism of wcDG\+modules $\P\rarrow\M$ is defined
by the rules $b\ot x\mpsto bx$ and $b_1\ot\dotsb\ot b_n\ot x
\mpsto 0$ for $n\ge2$.

 Similarly, consider the bigraded $\A$\+module whose component
of the internal grading $n\ge1$ is the graded $\A$\+module
$\Hom^\R(\A^{\ot^\R\,n}\;\M)$ with the action of $\A$ given
by the rule $(bf)(b_1,\dotsc,b_n)=(-1)^{|b|
(|f|+|b_1+\dotsb+|b_n|+n-1)}f(b_1,\dotsc,b_nb)$, where
$f\:\A^{\ot^\R\,n}\rarrow\M$ is a homogeneous graded
$\R$\+contramodule morphism of degree~$|f|$, while $b$, $b_s\in\A$
and $f(b_1,\dotsc,b_n)\in\M$.
 Define the differentials~$\d$, $d$, and~$\delta$ on this
$\A$\+module by the rules
\begin{multline*}
 (\d f)(b_1,\dotsc,b_n) = (-1)^{|f||b_1|}b_1f(b_2,\dotsc,b_n)
 - f(b_1b_2,b_3,\dotsc,b_n) \\ + f(b_1,b_2b_3,b_4,\dotsc,b_n)
 - \dotsb + (-1)^{n-1}f(b_1,\dotsc,b_{n-2},b_{n-1}b_n),
\end{multline*}
\begin{multline*}
 (-1)^{n-1}(df)(b_1,\dotsc,b_n) = d(f(b_1,\dotsc,b_n))
 - (-1)^{|f|}f(db_1,b_2,\dotsc,b_n) \\ - (-1)^{|f|+|b_1|}
 f(b_1,db_2,\dotsc,b_n) - \dotsb - (-1)^{|f|+|b_1|+\dotsb+|b_{n-1}|}
 f(b_1,\dotsc,b_{n-1},db_n),
\end{multline*}
and
\begin{multline*}
 (\delta f)(b_1,\dotsc,b_n) = - f(h,b_1,\dotsc,b_n) \\
 + f(b_1,h,b_2,\dotsc,b_n) - \dotsb + (-1)^n
 f(b_1,\dotsc,b_{n-1},h,b_n)
\end{multline*}
and endow it with the total differential $\d+d+\delta$.
 Totalizing our bigraded $\A$\+module by taking infinite
products along the diagonals where the sum $i=|f|+n-1$
is constant, we obtain the desired wcDG\+module $\J$ over~$\A$.
 The closed morphism of wcDG\+modules $\M\rarrow\J$ is
defined by the rule $x\mpsto f_x$, \ $f_x(b)=(-1)^{|x||b|}bx$
and $f_x(b_1,\dotsb,b_n)=0$ for $n\ge2$, where $x\in\M$.

 The reductions $\P/\m\P$ and $\J/\m\J$ of $\P$ and $\J$ modulo~$\m$
are the conventional bar- and cobar-resolutions of the DG\+module
$\M/\m\M$ over the DG\+algebra $\A/\m\A$ over the field~$k$.
 Hence DG\+module $\P/\m\P$ is homotopy projective,
the DG\+module $\J/\m\J$ is homotopy injective, and the cones
of the morphisms $\P/\m\P\rarrow\M/\m\M\rarrow\J/\m\J$ are
acyclic (see~\cite[proofs of Theorems~1.4\+-1.5]{Pkoszul}).
\end{proof}

 Notice that it follows from the above arguments that for any
homotopy projective $\R$\+free wcDG\+module $\P$ and semiacyclic
$\R$\+free wcDG\+module $\N$ over $\A$ the complex of
$\R$\+contramodules $\Hom_\A(\P,\N)$ is contractible.
 Similarly, for any semiacyclic $\R$\+free wcDG\+module $\N$ and
homotopy injective $\R$\+free wcDG\+module $\J$ over $\A$
the complex of $\R$\+contramodules $\Hom_\A(\N,\J)$ is
contractible (cf.\ Theorem~\ref{non-adj-semiderived-res} below).

\begin{thm}\hfuzz=2pt \label{r-free-cofibrant}
 Let\/ $\A$ be a wcDG\+algebra over\/~$\R$.
 Assume that the DG\+algebra\/ $\A/\m\A$ is cofibrant (in
the standard model structure on the category of DG\+algebras
over~$k$).
 Then an\/ $\R$\+free wcDG\+module over\/ $\A$ is semiacyclic if and
only if it is absolutely acyclic.
 So the semiderived category of\/ $\R$\+free wcDG\+modules\/
$\sD^\si(\A\mod\Rfr)$ over\/~$\A$ coincides with their
absolute derived category\/ $\sD^\abs(\A\mod\Rfr)$.
\end{thm}

\begin{proof}
 Clearly, the DG\+algebra $\A/\m\A$ has finite homological dimension,
so Corollary~\ref{r-free-abs-acycl-reduction} is applicable in view of
Corollary~\ref{r-free-homol-dim}.
 By~\cite[Theorem~9.4]{Pkoszul}, a DG\+module over $\A/\m\A$ is
acyclic if and only if it is absolutely acyclic, so the assertion
follows.
\end{proof}

\begin{lem} \label{r-free-homotopy-proj-tensor}
 Let\/ $\P$ be a homotopy projective\/ $\R$\+free left wcDG\+module
over a wcDG\+algebra\/ $\A$, and let\/ $\N$ be a semiacyclic\/
$\R$\+free right wcDG\+module over\/~$\A$.
 Then the tensor product\/ $\N\ot_\A\P$ is a contractible complex of\/
$\R$\+contramodules.
\end{lem}

\begin{proof} \hbadness=2300
 It follows from Theorem~\ref{r-free-semi-resolutions}(a) that any
homotopy projective $\R$\+free wcDG\+module $\P$ over $\A$ is
homotopy equivalent to a wcDG\+module with a projective underlying
$\R$\+free graded $\A$\+module; so we can assume that
$\P\in\A\mod\Rfr_\proj$.
 Then, by Lemma~\ref{r-free-hom-reduction}(c), the tensor product
$\N\ot_\A\P$ is a complex of free $\R$\+contramodules whose reduction
modulo~$\m$ is isomorphic to the complex of $k$\+vector spaces
$(\N/\m\N)\ot_{\A/\m\A}(\P/\m\P)$.
 The latter complex is acyclic, since the DG\+module $\N/\m\N$
is acyclic and the DG\+module $\P/\m\P$ is homotopy flat over $\A/\m\A$
by Lemma~\ref{homotopy-proj-inj-reduction}(a) and the results
of~\cite[Section~1.6]{Pkoszul}.
 It remains to apply Lemma~\ref{nakayama-acycl-contract}.
\end{proof}

 The left derived functor of tensor product of $\R$\+free
wcDG\+modules
$$
 \Tor^\A\:\sD^\si(\modrRfr\A)\times\sD^\si(\A\mod\Rfr)\lrarrow
 H^0(\R\contra^\free),
$$
where $H^0(\R\contra^\free)$ denotes, by an abuse of notation,
the homotopy category of complexes of free $\R$\+contramodules, is
constructed by restricting the functor $\ot_\A$ to either of
the full subcategories $H^0(\modrRfr\A)\times H^0(\A\mod\Rfr_\proj)
_\proj$ or $H^0(\modrRfrproj\A)_\proj\times H^0(\A\mod\Rfr)\subset
H^0(\modrRfr\A)\times H^0(\A\mod\Rfr)$.
 Here $H^0(\modrRfrproj\A)_\proj$ denotes the homotopy category of
homotopy projective $\R$\+free right wcDG\+modules over~$\A$ with
projective underlying graded $\A$\+modules.
 The restriction takes values in $H^0(\R\contra^\free)$ by
Lemma~\ref{r-free-hom-reduction}(c) and factorizes through
the Cartesian product of semiderived categories $\sD^\si(\modrRfr\A)
\times\sD^\si(\A\mod\Rfr)$ by Theorem~\ref{r-free-semi-resolutions}(a)
and Lemma~\ref{r-free-homotopy-proj-tensor}.
{\hbadness=1100\par}

 Similarly, the right derived functor of homomorphisms of $\R$\+free
wcDG\+modules
$$
 \Ext_\A\:\sD^\si(\A\mod\Rfr)^\sop\times
 \sD^\si(\A\mod\Rfr)\lrarrow H^0(\R\contra^\free)
$$
is constructed by restricting the functor $\Hom_\A$ to either of
the full subcategories $H^0(\A\mod\Rfr_\proj)_\proj^\sop\times
H^0(\A\mod\Rfr)$ or $H^0(\A\mod\Rfr)^\sop\times
H^0(\A\mod\Rfr_\inj)_\inj\subset H^0(\A\mod\Rfr)^\sop\times
H^0(\A\mod\Rfr)$.

 Let $(f,a)\:\B\rarrow\A$ be a morphism of wcDG\+algebras over~$\R$.
 Then we have the induced functor $R_f\:H^0(\A\mod\Rfr)\rarrow
H^0(\B\mod\Rfr)$.
 Notice that, by the definition of a wcDG\+morphism,
this functor agrees with the functor $R_{f/\m f}\:H^0(\A/\m\A\mod)
\rarrow H^0(\B/\m\B\mod)$ induced on the homotopy categories of
DG\+modules over DG\+algebras over~$k$ by the DG\+algebra morphism
$f/\m f\:\B/\m\B\rarrow\A/\m\A$.
 Hence the functor $R_f$ preserves semiacyclicity of wcDG\+modules
and therefore induces a triangulated functor 
$$
 \boI R_f\:\sD^\si(\A\mod\Rfr)\lrarrow\sD^\si(\B\mod\Rfr)
$$
on the semiderived categories.

 The triangulated functor $\boI R_f$ has adjoints on both sides.
 The DG\+functor $E_f\:\B\mod\Rfr_\proj\rarrow\A\mod\Rfr_\proj$
is defined on the level of graded modules by the rule
$\N\mpsto \A\ot_\B\N$; the differential on $\A\ot_\B\N$ induced
by the differentials on $\A$ and $\N$ is modified to obtain
the differential on $E_f(\N)$ using the change-of-connection
element~$a$.
  Restricting the induced triangulated functor $E_f\:
H^0(\B\mod\Rfr_\proj)\rarrow H^0(\A\mod\Rfr_\proj)$ to
$H^0(\B\mod\Rfr_\proj)_\proj$, composing it with the localization
functor $H^0(\A\mod\Rfr_\proj)\rarrow\sD^\si(\A\mod\Rfr)$, and
taking into account Theorem~\ref{r-free-semi-resolutions}(a),
we construct the left derived functor
$$
 \boL E_f\:\sD^\si(\B\mod\Rfr)\lrarrow\sD^\si(\A\mod\Rfr),
$$
which is left adjoint to the functor $\boI R_f$.

 Similarly, the DG\+functor $E^f\:\B\mod\Rfr_\inj\rarrow
\A\mod\Rfr_\inj$ is defined on the level of graded modules by
the rule $\N\mpsto\Hom_\B(\A,\N)$; the change-of-connection
element~$a$ is used to modify the differential on $\Hom_\B(\A,\N)$
induced by the differentials on $\A$ and~$\N$.
 Restricting the induced triangulated functor $E^f$ to
$H^0(\B\mod\Rfr_\inj)_\inj$, composing it
with the localization functor to $\sD^\si(\A\mod\Rfr)$ and
using Theorem~\ref{r-free-semi-resolutions}(b), we obtain
the right derived functor
$$
 \boR E^f\:\sD^\si(\B\mod\Rfr)\lrarrow\sD^\si(\A\mod\Rfr),
$$
which is right adjoint to the functor $\boI R_f$.

\begin{thm}\hbadness=3000  \label{r-free-reduction-quasi}
 The functor\/ $\boI R_f\:\sD^\si(\A\mod\Rfr)\rarrow
\sD^\si(\B\mod\Rfr)$ is an equivalence of triangulated categories
whenever the DG\+algebra morphism $f/\m f\:\B/\m\B\allowbreak\rarrow
\A/\m\A$ is a quasi-isomorphism.
\end{thm}

\begin{proof}
 It suffices to check that the adjunction morphisms for
the functors $\boL E_f$, $\boI R_f$, $\boR E^f$ are isomorphisms.
 However, a morphism in a semiderived category of $\R$\+free
wcDG\+modules is an isomorphism whenever it becomes
an isomorphism in the derived category of DG\+modules after
the reduction modulo~$\m$.
 Our three functors make commutative diagrams with the similar
functors $\boL E_{f/\m f}$, $\boI R_{f/\m f}$, $\boR E^{f/\m f}$
for DG\+modules over $\A/\m\A$ and $\B/\m\B$ and the functors of
reduction modulo~$\m$, and the adjunctions agree.
 Finally, it remains to use the DG\+algebra version of the desired
assertion, which is well-known~\cite[Theorem~1.7]{Pkoszul}.
\end{proof}

\begin{rem} \label{reduction-quasi-not-converse}
 Unlike in the situation of DG\+algebras over a field or DG\+rings
\cite[Theorem~1.7]{Pkoszul}, the converse assertion to
Theorem~\ref{r-free-reduction-quasi} is \emph{not} true for
wcDG\+algebras over a pro-Artinian topological local ring~$\R$.
 Moreover, there exist wcDG\+algebras $\A$ with the DG\+algebra
$\A/\m\A$ having a nonzero cohomology $k$\+algebra, while
the semiderived category $\sD^\si(\A\mod\Rfr)$ vanishes.
 A counterexample, essentially from~\cite{KLN}, will be discussed
below in Example~\ref{kln-counterex}.
\end{rem}

\subsection{$\R$-cofree graded modules}  \label{r-cofree-graded}
 A \emph{graded\/ $\R$\+comodule} is a family of $\R$\+comodules
indexed by elements of the grading group~$\Gamma$.
 A \emph{cofree graded\/ $\R$\+comodule} is a graded $\R$\+comodule
whose grading components are cofree $\R$\+comodules.

 The operations of contratensor product $\ocn_\R$,
contrahomomorphisms $\Ctrhom_\R$, and comodule homomorphisms
$\Hom_\R$ are extended to graded $\R$\+contramodules and
graded $\R$\+comodules in the following way.
 The components of $\P\ocn_\R\cM$ are the direct sums of
the contratensor products of the appropriate components of
$\P$ and $\cM$, i.~e., $(\P\ocn_\R\cM)^n=\bigoplus_{i+j=n}
\P^i\ocn_\R\cM^j$.
 The components of $\Ctrhom_\R(\P,\cM)$ are the comodule direct
products of the comodules $\Ctrhom$ between the appropriate
components of $\P$ and $\cM$, that it $\Ctrhom_\R(\P,\cM)^n 
= \prod_{j-i=n}\Ctrhom_\R(\P^i,\cM^j)$.
 The components of $\Hom_\R(\cL,\cM)$ are the direct products
of the contramodules of homomorphisms between the appropriate
components of $\cL$ and $\cM$, i.~e., $\Hom_\R(\cL,\cM)^n
\allowbreak= \prod_{j-i=n}\Hom_\R(\cL^i,\cM^j)$.

 Similarly, the components of the graded version of
the cotensor product $\cN\oc_\R\cM$ are the direct sums of
the cotensor products of the appropriate components of
$\cN$ and $\cM$, while the components of the graded
$\Cohom_\R(\cM,\P)$ are the direct products of
the contramodules $\Cohom_\R$ between the appropriate
components of $\cM$ and~$\P$.
 Notice that all of our operations produce free (graded)
$\R$\+contramodules or cofree (graded) $\R$\+comodules when
receiving free (graded) $\R$\+contramodules and cofree
(graded) $\R$\+comodules as their inputs
(see Section~\ref{r-comodules-sect}--\ref{co-operations}).

 Recall that the category of (cofree) $\R$\+comodules is
a module category with respect to the operation $\ocn_\R$
over the tensor category of (free) $\R$\+contramodules.
 The same applies to the categories of graded comodules and
contramodules.

 An \emph{$\R$\+cofree graded left module} $\cM$ over
an $\R$\+free graded algebra $\B$ is, by the definition,
a graded left module over $\B$ in the module category of
cofree $\R$\+comodules, i.~e., a cofree graded
$\R$\+comodule endowed with an (associative and unital)
homogeneous $\B$\+action map $\B\ocn_\R\cM\rarrow\cM$.
\ \emph{$\R$\+cofree graded right modules} $\cN$ over $\B$
are defined in the similar way.
 Alternatively, one can define $\B$\+module structures on
cofree graded $\R$\+comodules in terms of the action maps
$\cM\rarrow\Ctrhom_\R(\B,\cM)$, and similarly for~$\cN$.

 The category of $\R$\+cofree graded $\B$\+modules is enriched
over the tensor category of (graded) $\R$\+contramodules, so
the abelian group of morphisms between two $\R$\+cofree
graded $\B$\+modules $\cL$ and $\cM$ is the underlying abelian
group of the degree-zero component of a certain naturally
defined (not necessarily free) graded $\R$\+contramodule
$\Hom_\B(\cL,\cM)$.
 This $\R$\+contramodule is constructed as the kernel of
the pair of morphisms of graded $\R$\+contramodules
$\Hom_\R(\cL,\cM)\birarrow\Hom_\R(\B\ocn_\R\cL\;\cM)$ induced
by the actions of $\B$ in $\cL$ and~$\cM$.
 For the sign rule, see~\cite[Section~1.1]{Pkoszul}.

 The \emph{tensor product} $\N\ot_\B\cM$ of an $\R$\+free graded
right $\B$\+module $\N$ and an $\R$\+cofree graded left $\B$\+module
$\cM$ is a (not necessarily cofree) graded $\R$\+comodule
constructed as the cokernel of the pair of morphisms of graded
$\R$\+comodules $\N\ot^\R\B\ocn_\R\cM\birarrow\N\ocn_\R\cM$ induced
by the actions of $\B$ in $\N$ and~$\cM$.
 The graded $\R$\+comodule $\Hom_\B(\L,\cM)$ from an $\R$\+free graded
left $\B$\+module $\L$ to an $\R$\+cofree graded left $\B$\+module
$\cM$ is defined as the kernel of the pair of morphisms of graded
$\R$\+comodules $\Ctrhom_\R(\L,\cM)\birarrow\Ctrhom_\R(\B\ot^\R\L\;\cM)
\simeq\Ctrhom_\R(\L,\Ctrhom_\R(\B,\cM))$ induced by the $\B$\+action
morphisms $\B\ot^\R\L\rarrow\L$ and $\cM\rarrow\Ctrhom_\R(\B,\cM))$.

 The additive category of $\R$\+cofree graded (left or right)
$\B$\+modules has a natural exact category structure: a short
sequence of $\R$\+cofree graded $\B$\+modules is said to be
exact if it is (split) exact as a short sequence of cofree
graded $\R$\+comodules.
 The exact category of $\R$\+cofree graded $\B$\+modules admits
infinite direct sums and products, which are preserved by
the forgetful functors to the category of cofree graded
$\R$\+comodules and preserve exact sequences.

 If a morphism of $\R$\+cofree graded $\B$\+modules has
an $\R$\+cofree kernel (resp., cokernel) in the category of
graded $\R$\+comodules, then this kernel (resp., cokernel)
is naturally endowed with an $\R$\+cofree graded $\B$\+module
structure, making it the kernel (resp., cokernel) of
the same morphism in the additive category of $\R$\+cofree
graded $\B$\+modules.

 There are enough projective objects in the exact category of
$\R$\+cofree graded left $\B$\+modules; these are the direct
summands of the graded $\B$\+modules $\B\ocn_\R\cU$ freely
generated by cofree graded $\R$\+comodules~$\cU$.
 Similarly, there are enough injectives in this exact category,
and these are the direct summands of the graded $\B$\+modules
$\Ctrhom_\R(\B,\cV)$ cofreely cogenerated by cofree graded
$\R$\+comodules~$\cV$.

\begin{prop} \label{r-cofree-r-co-contra}
 The exact categories of\/ $\R$\+free graded\/ left $\B$\+modules and
of\/ $\R$\+cofree graded\/ left $\B$\+modules are naturally equivalent.
 The equivalence is provided by the functors\/ $\Phi_\R$ and\/
$\Psi_\R$ of comodule-contramodule correspondence over\/~$\R$,
defined in Section~\textup{\ref{hom-operations}}, which transform
free graded\/ $\R$\+contramodules with a\/ $\B$\+module structure
into cofree graded\/ $\R$\+comodules with a\/ $\B$\+module structure
and back.
\end{prop}

\begin{proof}
 The assertion follows from the fact that the functors $\Phi_\R$
and $\Psi_\R$ define mutually inverse equivalences between
the module categories $\R\contra^\free$ and $\R\comod^\cofr$
of free $\R$\+contramodules and cofree $\R$\+comodules over
the tensor category $\R\contra^\free$.
 In fact, there is a natural isomorphism
$\Phi_\R(\P\ot^\R\Q)\simeq\P\ocn_\R\Phi_\R(\Q)$ 
for any $\R$\+contramodules $\P$ and~$\Q$.
 The functors $\Phi_\R$ and $\Psi_\R$ also transform the functor
$\Hom^\R(\P,{-})$ from a fixed free $\R$\+contramodule $\P$
to varying free $\R$\+contramodules into the functor
$\Ctrhom_\R(\P,{-})$ from the same $\R$\+contramodule $\P$ to
varying cofree $\R$\+comodules.
 In fact, there is a natural isomorphism
$\Psi_\R(\Ctrhom_\R(\P,\cM))\simeq\Hom^\R(\P,\Psi_\R(\cM))$
for any $\R$\+contramodule $\P$ and $\R$\+comod\-ule~$\cM$
(see Sections~\ref{hom-operations}\+-\ref{contra-operations}).
\end{proof}

 The equivalence $\Phi_\R=\Psi_\R^{-1}$ between the categories
of $\R$\+free and $\R$\+cofree graded $\B$\+modules is also
an equivalence of categories enriched over graded
$\R$\+contramodules.
 Besides, for any $\R$\+free graded right $\B$\+module $\N$ and
$\R$\+free graded left $\B$\+module $\M$ there is a natural
isomorphism of graded $\R$\+comodules $\N\ot_\B\Phi_\R(\M)
\simeq\Phi_\R(\N\ot_\B\M)$; and for any $\R$\+free graded left
$\B$\+module $\L$ and any $\R$\+cofree graded left $\B$\+module
$\cM$ there is a natural isomorphism of graded $\R$\+contramodules
$\Hom_\B(\L,\Psi_\R(\cM))\simeq\Psi_\R(\Hom_\B(\L,\cM))$.

 Furthermore, the equivalence $\Phi_\R=\Psi_\R^{-1}$ transforms
graded $\B$\+modules $\B\ot^\R\U$ freely generated by
free graded $\R$\+contramodules $\U$ into graded $\B$\+modules
$\B\ocn_\R\cU$ freely generated by cofree graded $\R$\+comodules
$\cU$, with $\cU=\Phi_\R(\U)$, and graded $\B$\+modules
$\Hom^\R(\B,\V)$ cofreely cogenerated by free graded
$\R$\+contramodules $\V$ into graded $\B$\+modules
$\Ctrhom_\R(\B,\cV)$ cofreely cogenerated by cofree graded
$\R$\+comodules $\cV$, with $\Psi_\R(\cV)=\V$.

 Finally, the equivalence $\Phi_\R=\Psi_\R^{-1}$ between
the exact categories of $\R$\+free and $\R$\+cofree graded
$\B$\+modules transforms the reduction functor 
$\P\mpsto\P/\m\P$ into the reduction functor $\cM\mpsto
{}_\m\cM$ (both taking values in the abelian category of
graded $\B/\m\B$\+modules).
 This follows from the results of Section~\ref{hom-operations}.

 All the results of Section~\ref{r-free-graded} have their
analogues for $\R$\+cofree graded $\B$\+modules, which can be,
at one's choice, either proven directly in the similar way
to the proofs in~\ref{r-free-graded}, or deduced from
the results in~\ref{r-free-graded} using
Proposition~\ref{r-cofree-r-co-contra}.

\subsection{Absolute derived category of $\R$-cofree CDG-modules}
\label{r-cofree-absolute}
 Let $\cU$ and $\cV$ be graded $\R$\+comodules endowed
with homogeneous $\R$\+comodule endomorphisms (differentials)
$d_\cU\:\cU\rarrow\cU$ and $d_\cV\:\cV\rarrow\cV$ of degree~$1$.
 Then the graded $\R$\+contramodule $\Hom_\R(\cU,\cV)$ is
endowed with the differential~$d$ defined by the conventional
rule $d(f)(u) = d_\cV(f(u)) - (-1)^{|f|}f(d_\cU(u))$.

 Let $\W$ be a graded $\R$\+contramodule endowed with 
a differential~$d_\W$.
 Then the graded $\R$\+comodules $\W\ocn_\R\cU$ and
$\Ctrhom_\R(\W,\cV)$ are endowed with their differentials~$d$
defined by the conventional rules
$d(w\ocn u) = d_\W(w)\ocn u + (-1)^{|w|}w\ocn d_\cU(u)$
and $d(f)(w) = d_\cV(f(w)) - (-1)^{|f|}f(d_\W(w))$.

 Let $\B$ be an $\R$\+free graded algebra endowed with an odd
derivation~$d$ of degree~$1$.
 An \emph{odd derivation} $d_\cM$ of an $\R$\+cofree graded left
$\B$\+module $\cM$ \emph{compatible with} the derivation~$d$
on $\B$ is a differential ($\R$\+comodule endomorphism of
degree~$1$) such that the action map $\B\ocn_\R\cM\rarrow\cM$
forms a commutative diagram with $d_\cM$ and the differential
on $\B\ocn_\R\cM$ induced by $d$ and~$d_\cM$.
 Odd derivations of $\R$\+cofree graded right $\B$\+modules are
defined similarly.
 Alternatively, one requires the action map
$\cM\rarrow\Ctrhom_\R(\B,\cM)$ to commute with the differentials;
the adjunction of $\ocn_\R$ and $\Ctrhom_\R$ transforms one
equation into the other.

 An \emph{$\R$\+cofree left CDG\+module} $\cM$ over
an $\R$\+free CDG\+algebra $\B$ is, by the definition,
an $\R$\+cofree graded left $\B$\+module endowed with
an odd derivation $d_\cM\:\cM\rarrow\cM$ of degree~$1$
compatible with the derivation $d$ on~$\B$ and satisfying
the equation $d^2_\cM(x)=hx$ for all $x\in\cM$.
 \ \emph{$\R$\+cofree right CDG\+modules} $\cN$ over~$\B$ are
defined similarly, except for the sign.
 Here the element $h\in\B^2$ can be viewed as a homogeneous
morphism $\R\rarrow\B$ of degree~$2$; the notation $x\mpsto hx$
is interpreted as the composition $\cM\simeq\R\ocn_\R\cM\rarrow
\B\ocn_\R\cM\rarrow\cM$ defining a homogeneous endomorphism of
degree~$2$ of the graded $\R$\+comodule~$\cM$.

 $\R$\+cofree left (resp., right) CDG\+modules over $\B$ form
a DG\+category, which we will denote by $\B\mod\Rcof$
(resp., $\modrRcof\B$).
 In fact, these DG\+categories are enriched over the tensor
category of (complexes of) $\R$\+contramodules, so the complexes
of morphisms in them are the underlying complexes of abelian
groups for naturally defined complexes of $\R$\+contramodules,
which we will denote by $\Hom_\B(\cL,\cM)$.
 The underlying graded $\R$\+contramodules of these complexes
were defined in Section~\ref{r-cofree-graded}, and
the differentials in them are defined in the conventional way.

 Passing to the zero cohomology of the complexes of morphisms,
we construct the homotopy categories of $\R$\+cofree CDG\+modules
$H^0(\B\mod\Rcof)$ and $H^0(\modrRcof\B)$.
 These we will mostly view as conventional categories with
abelian groups of morphisms (see Remark~\ref{contra-enriched}).
 Since the DG\+categories $\B\mod\Rcof$ and $\modrRcof\B$
have shifts, twists, and infinite direct sums and products,
their homotopy categories are triangulated categories with
infinite direct sums and products. {\hbadness=1200\par}

 The tensor product $\N\ot_\B\cM$ of an $\R$\+free right
CDG\+module $\N$ and an $\R$\+cofree left CDG\+module $\cM$
over $\B$ is a complex of $\R$\+comodules obtained by
endowing the graded $\R$\+comodule $\N\ot_\B\cM$ constructed
in Section~\ref{r-cofree-graded} with the conventional
tensor product differential.
 The tensor product of $\R$\+free and $\R$\+cofree CDG\+modules
over $\B$ is a triangulated functor of two arguments
$$
 \ot_\B\: H^0(\modrRfr\B)\times H^0(\B\mod\Rcof)\lrarrow
 H^0(\R\comod),
$$
where $H^0(\R\comod)$ denotes, by an abuse of notation,
the homotopy category of complexes of $\R$\+comodules.
 The complex $\Hom_\B(\L,\cM)$ from an $\R$\+free left
CDG\+module $\L$ to an $\R$\+cofree left CDG\+module $\cM$
over $\B$ is a complex of $\R$\+comodules obtained by
endowing the graded $\R$\+comodule $\Hom_\B(\L,\cM)$ constructed
in Section~\ref{r-cofree-graded} with the conventional
$\Hom$ differential.
 The $\Hom$ between $\R$\+free and $\R$\+cofree CDG\+modules
over $\B$ is a triangulated functor of two arguments
$$
 \Hom_\B\: H^0(\B\mod\Rfr)^\sop\times
 H^0(\B\mod\Rcof)\lrarrow H^0(\R\comod).
$$

 An $\R$\+cofree left CDG\+module over $\B$ is said to be
\emph{absolutely acyclic} if it belongs to the minimal
thick subcategory of the homotopy category $H^0(\B\mod\Rcof)$
containing the totalizations of short exact sequences of
$\R$\+cofree CDG\+modules over~$\B$.
 The quotient category of $H^0(\B\mod\Rcof)$ by the thick
subcategory of absolutely acyclic $\R$\+cofree CDG\+modules is
called the \emph{absolute derived category} of $\R$\+cofree
left CDG\+modules over $\B$ and denoted by
$\sD^\abs(\B\mod\Rcof)$.
 The absolute derived category of $\R$\+cofree right
CDG\+modules over $\B$, denoted by $\sD^\abs(\modrRcof\B)$,
is defined similarly.

 An $\R$\+cofree left CDG\+module over $\B$ is said to be
\emph{contraacyclic} if it belongs to the minimal
triangulated subcategory of the homotopy category
$H^0(\B\mod\Rcof)$ containing the totalizations of short
exact sequences of $\R$\+cofree CDG\+modules over $\B$ and
closed under infinite products.
 The quotient category of $H^0(\B\mod\Rcof)$ by the thick
subcategory of contraacyclic $\R$\+cofree CDG\+modules is called
the \emph{contraderived category} of $\R$\+cofree left CDG\+modules
over $\B$ and denoted by $\sD^\ctr(\B\mod\Rcof)$.
 The (similarly defined) contraderived category of $\R$\+cofree
right CDG\+modules over $\B$ is denoted by
$\sD^\ctr(\modrRcof\B)$.

 An $\R$\+cofree left CDG\+module over $\B$ is said to be
\emph{coacyclic} if it belongs to the minimal triangulated
subcategory of the homotopy category $H^0(\B\mod\Rcof)$
containing the totalizations of short exact sequences of
$\R$\+cofree CDG\+modules over $\B$ and closed under
infinite direct sums.
 The quotient category of $H^0(\B\mod\Rcof)$ by the thick
subcategory of coacyclic $\R$\+cofree CDG\+modules is called
the \emph{coderived category} of $\R$\+cofree left CDG\+modules
over $\B$ and denoted by $\sD^\co(\B\mod\Rcof)$.
 The coderived category of $\R$\+cofree right CDG\+modules
over $\B$ is denoted by $\sD^\co(\modrRcof\B)$.

 As above, we denote by $\B\mod\Rcof_\proj\subset\B\mod\Rcof$
the full DG\+subcategory formed by all the $\R$\+cofree
CDG\+modules whose underlying $\R$\+cofree graded $\B$\+mod\-ules
are projective.
 Similarly, $\B\mod\Rcof_\inj\subset\B\mod\Rcof$ is the full
DG\+subcategory formed by all the $\R$\+cofree CDG\+modules
whose underlying $\R$\+cofree graded $\B$\+modules are injective.
 The corresponding homotopy categories are denoted by
$H^0(\B\mod\Rcof_\proj)$ and $H^0(\B\mod\Rcof_\inj)$.
{\hbadness=1500\par}

 The functors $\Phi_\R=\Psi_\R^{-1}$ from
Section~\ref{r-cofree-graded} define an equivalence of
DG\+categories $\B\mod\Rfr\simeq\B\mod\Rcof$ (and similarly
for right CDG\+modules).
 Being also an equivalence of exact categories, this
correspondence identifies absolutely acyclic (resp., contraacyclic,
coacyclic) $\R$\+free CDG\+modules with absolutely acyclic 
(resp., contraacyclic, coacyclic) $\R$\+cofree CDG\+modules
over~$\B$.
 So an equivalence of the absolute derived categories
$$
 \Phi_\R\:\sD^\abs(\B\mod\Rfr)\simeq
 \sD^\abs(\B\mod\Rcof)\,\,\:\!\Psi_\R,
$$
an equivalence of the contraderived categories 
$$
 \Phi_\R\:\sD^\ctr(\B\mod\Rfr)\simeq
 \sD^\ctr(\B\mod\Rcof)\,\,\:\!\Psi_\R,
$$
and an equivalence of the coderived categories
$$
 \Phi_\R\:\sD^\co(\B\mod\Rfr)\simeq
 \sD^\co(\B\mod\Rcof)\,\,\:\!\Psi_\R
$$
are induced.
 The above equivalence of DG\+categories also identifies
$\B\mod\Rfr_\proj$ with $\B\mod\Rcof_\proj$ and $\B\mod\Rfr_\inj$
with $\B\mod\Rcof_\inj$.

 For any $\R$\+free right CDG\+module $\N$ and $\R$\+free
left CDG\+module $\M$ over $\B$ the isomorphism
$\N\ot_\B\Phi_\R(\M)\simeq\Phi_\R(\N\ot_\B\M)$ from
Section~\ref{r-cofree-graded} is an isomorphism of complexes
of $\R$\+comodules.
 Similarly, for any $\R$\+free left CDG\+module $\L$ and
$\R$\+cofree left CDG\+module $\cM$ over $\B$ the isomorphism
$\Hom_\B(\L,\Psi_\R(\cM))\simeq\Psi_\R(\Hom_\B(\L,\cM))$
is an isomorphism of complexes of $\R$\+contramodules.
 All the results of Section~\ref{r-free-absolute} have
their analogues for $\R$\+cofree CDG\+modules over~$\B$.

 Let $f=(f,a)\:\B\rarrow\A$ be a morphism of $\R$\+free CDG\+algebras.
 Then with any $\R$\+cofree left CDG\+module $(\cM,d_\cM)$ over $\A$
one can associate an $\R$\+cofree left CDG\+module $(\cM,d'_\cM)$
over $\B$ with the graded $\B$\+module structure on $\cM$ defined
via~$f$ and the modified differential~$d'_\cM$ constructed in terms
of~$a$.
 A similar procedure applies to right CDG\+modules.
 So we obtain the DG\+functors of restriction of scalars
$R_f\:\A\mod\Rcof\rarrow\B\mod\Rcof$ and $\modrRcof\A\rarrow
\modrRcof\B$; passing to the homotopy categories, we have
the triangulated functors
$$
 R_f\:H^0(\A\mod\Rcof)\lrarrow H^0(\B\mod\Rcof)
$$
and $H^0(\modrRcof\A)\rarrow H^0(\modrRcof\B)$.
 The equivalences of categories $\Phi_\R=\Psi_\R^{-1}$ identify
these functors with the functors $R_f$ for $\R$\+free CDG\+modules
constructed in Section~\ref{r-free-absolute}.

\subsection{Semiderived category of $\R$-cofree wcDG-modules}
\label{r-cofree-semi}
 As above, we refer to $\R$\+cofree CDG\+modules over a wcDG\+algebra
$\A$ as \emph{$\R$\+cofree wcDG\+modules}.
 An $\R$\+cofree wcDG\+module $\cM$ over $\A$ is said to be
\emph{semiacyclic} if the complex of vector spaces ${}_\m\cM$
(which is in fact a DG\+module over the DG\+algebra $\A/\m\A$)
is acyclic.
 In particular, it follows from Lemma~\ref{comodule-acycl-contract}
that when $\A$ is an $\R$\+free DG\+algebra (i.~e., $h=0$),
an $\R$\+cofree CDG\+module over $\A$ is semiacyclic if and only if
it is contractible as a complex of cofree $\R$\+comodules.

 Clearly, the property of semiacyclicity of an $\R$\+cofree
wcDG\+module over $\A$ is preserved by shifts, cones, homotopy
equivalences, infinite direct sums, and infinite products.
 The quotient category of the homotopy category
$H^0(\A\mod\Rcof)$ by the thick subcategory of semiacyclic
$\R$\+cofree wcDG\+modules is called the \emph{semiderived
category} of $\R$\+cofree left wcDG\+modules over $\A$ and
denoted by $\sD^\si(\A\mod\Rcof)$.
 The semiderived category of $\R$\+cofree right wcDG\+modules
$\sD^\si(\modrRcof\A)$ is defined similarly.

 Notice that any acyclic complex of cofree $\R$\+comodules is
contractible when (the category of comodules over) $\R$ has
finite homological dimension.
 In this case, we call semiacyclic $\R$\+cofree wcDG\+modules
over $\A$ simply \emph{acyclic}, and refer to the semiderived
category of $\R$\+cofree wcDG\+modules as their
\emph{derived category}.

 An $\R$\+cofree wcDG\+module $\cP$ over $\A$ is called
\emph{homotopy projective} if the complex $\Hom_\A(\cP,\cM)$
is acyclic for any semiacyclic $\R$\+cofree wcDG\+module
$\cM$ over~$\A$.
 Similarly, an $\R$\+cofree wcDG\+module $\cJ$ over $\A$ is
called \emph{homotopy injective} if the complex
$\Hom_\A(\cM,\cJ)$ is acyclic for any semiacyclic $\R$\+cofree
wcDG\+module~$\cM$.
 The full triangulated subcategories in $H^0(\A\mod\Rcof)$
formed by the homotopy projective and homotopy injective
$\R$\+cofree wcDG\+modules are denoted by $H^0(\A\mod\Rcof)_\proj$
and $H^0(\A\mod\Rcof)_\inj$, respectively.
 The intersections $H^0(\A\mod\Rcof)_\proj\cap H^0(\A\mod\Rcof_\proj)$
and $H^0(\A\mod\Rcof)_\inj\cap H^0(\A\mod\Rcof_\inj)$ are
denoted by $H^0(\A\mod\Rcof_\proj)_\proj$ and
$H^0(\A\mod\Rcof_\inj)_\inj$. {\hbadness=1700\par}

 The equivalence of homotopy categories $\Phi_\R=\Psi_\R^{-1}\:
H^0(\B\mod\Rfr)\simeq H^0(\B\allowbreak\mod\Rcof)$ identifies
semiacyclic $\R$\+free wcDG\+modules with semiacyclic $\R$\+cofree
wcDG\+modules, and therefore induces an equivalence of
the semiderived categories
$$
 \Phi_\R\:\sD^\si(\A\mod\Rfr)\simeq\sD^\si(\A\mod\Rcof)\,\,\:\!\Psi_\R.
$$
 It also follows that the equivalence of homotopy categories
identifies homotopy projective (resp., injective) $\R$\+free
wcDG\+modules with homotopy projective (resp., injective)
$\R$\+cofree wcDG\+modules.

 All the results of Section~\ref{r-free-semi} have their
analogues for $\R$\+cofree wcDG\+modules, which can be,
at one's choice, either proven directly in the similar way
to the proofs in~\ref{r-free-semi}, or deduced from
the results in~\ref{r-free-semi} using the above observations
about the functors $\Phi_\R=\Psi_\R^{-1}$.
 In particular, the following assertions hold.

\begin{lem}
\textup{(a)}
 A wcDG\+module\/ $\cP\in \A\mod\Rcof_\proj$ is homotopy
projective if and only if the DG\+module\/ ${}_\m\cP$ over
the DG\+algebra\/ $\A/\m\A$ over the field~$k$ is homotopy
projective. \par
\textup{(b)}
 A wcDG\+module $\cJ\in \A\mod\Rcof_\inj$ is homotopy
injective if and only if the DG\+module\/ ${}_\m\cJ$ over
the DG\+algebra\/ $\A/\m\A$ is homotopy injective. \qed
\end{lem}

\begin{thm}  \label{r-cofree-semi-resolutions}
\textup{(a)} For any wcDG\+algebra\/ $\A$ over\/ $\R$,
the compositions of functors
$$
 H^0(\A\mod\Rcof)_\proj\lrarrow H^0(\A\mod\Rcof)
 \lrarrow\sD^\si(\A\mod\Rcof)
$$
and
$$
 H^0(\A\mod\Rcof_\proj)_\proj\lrarrow H^0(\A\mod\Rcof)
 \lrarrow\sD^\si(\A\mod\Rcof)
$$
are equivalences of triangulated categories. \par
\textup{(b)} For any wcDG\+algebra\/ $\A$ over\/ $\R$,
the compositions of functors
$$
 H^0(\A\mod\Rcof)_\inj\lrarrow
 H^0(\A\mod\Rcof)\lrarrow\sD^\si(\A\mod\Rcof)
$$
and
$$
 H^0(\A\mod\Rcof_\inj)_\inj\lrarrow
 H^0(\A\mod\Rcof)\lrarrow\sD^\si(\A\mod\Rcof).
$$
are equivalences of triangulated categories. \qed
\end{thm}

\begin{lem} \label{r-cofree-homotopy-proj-tensor}
\textup{(a)} Let\/ $\cP$ be a homotopy projective\/ $\R$\+cofree
left wcDG\+module over a wcDG\+algebra\/ $\A$, and let\/ $\N$
be a semiacyclic $\R$\+free right wcDG\+module over\/~$\A$.
 Then the tensor product\/ $\N\ot_\A\cP$ is a contractible
complex of\/ $\R$\+comodules. \par
\textup{(b)}
 Let\/ $\Q$ be a homotopy projective\/ $\R$\+free
right wcDG\+module and\/ $\cM$ be a semiacyclic\/ $\R$\+cofree
left wcDG\+module over a wcDG\+algebra\/~$\A$.
 Then the tensor product\/ $\Q\ot_\A\cM$ is a contractible
complex of\/ $\R$\+comodules.  \par
\textup{(c)} Let\/ $\P$ be a homotopy projective\/ $\R$\+free
left wcDG\+module over a wcDG\+algebra\/ $\A$, and let\/ $\cM$
be a semiacyclic\/ $\R$\+cofree left wcDG\+module over\/~$\A$.
 Then the complex of\/ $\R$\+comodules\/ $\Hom_\A(\P,\cM)$
is contractible. \par
\textup{(d)} Let\/ $\cJ$ be a homotopy injective\/ $\R$\+cofree
left wcDG\+module and\/ $\L$ be a semiacyclic\/ $\R$\+free
left wcDG\+module over a wcDG\+algebra\/~$\A$.
 Then the complex of\/ $\R$\+comodules\/ $\Hom_\A(\L,\cJ)$
is contractible. \qed
\end{lem}

 The left derived functor of tensor product of $\R$\+free and
$\R$\+cofree wcDG\+modules
$$
 \Tor^\A\:\sD^\si(\modrRfr\A)\times\sD^\si(\A\mod\Rcof)
 \lrarrow H^0(\R\comod^\cofr),
$$
where $H^0(\R\comod^\cofr)$ denotes, by an abuse of notation,
the homotopy category of complexes of cofree $\R$\+comodules,
is constructed by restricting the functor $\ot_\A$ to either
of the full subcategories $H^0(\modrRfr\A)\times
H^0(\A\mod\Rcof_\proj)_\proj$ or $H^0(\modrRfrproj\A)_\proj\times
H^0(\A\mod\Rcof)\subset H^0(\modrRfr\A)\times H^0(\A\mod\Rfr)$.
 The equivalences of categories $\Phi_\R=\Psi_\R^{-1}$
transform this functor into the derived functor $\Tor^\A$ of
tensor product of $\R$\+free wcDG\+modules over $\A$
defined in Section~\ref{r-free-semi}.

 Similarly, the right derived functor of $\Hom$ between $\R$\+free
and $\R$\+cofree wcDG\+mod\-ules {\hbadness=1400
$$
 \Ext_\A\:\sD^\si(\A\mod\Rfr)^\sop\times\sD^\si(\A\mod\Rcof)
 \lrarrow H^0(\R\comod^\cofr)
$$
is} constructed by restricting the functor $\Hom_\A$ to either of
the full subcategories $H^0(\A\mod\Rfr_\proj)_\proj^\sop\times
H^0(\A\mod\Rcof)$ or $H^0(\A\mod\Rfr)^\sop\times
H^0(\A\mod\Rcof_\inj)_\inj\subset H^0(\A\mod\Rfr)^\sop\times
H^0(\A\mod\Rcof)$.
 The right derived functor of homomorphisms of $\R$\+cofree
wcDG\+modules
$$
 \Ext_\A\:\sD^\si(\A\mod\Rcof)^\sop\times\sD^\si(\A\mod\Rcof)
 \lrarrow H^0(\R\contra^\free)
$$
is constructed by restricting the functor $\Hom_\A$ to either of
the full subcategories $H^0(\A\mod\Rcof_\proj)_\proj^\sop\times
H^0(\A\mod\Rcof)$ or $H^0(\A\mod\Rcof)^\sop\times
H^0(\A\mod\Rcof_\inj)_\inj\subset H^0(\A\mod\Rcof)^\sop\times
H^0(\A\mod\Rcof)$.
 The equivalences of categories $\Phi_\R=\Psi_\R^{-1}$ transform
these two functors into each other and the derived functor
$\Ext_\A$ of homomorphisms of $\R$\+free wcDG\+modules 
defined in Section~\ref{r-free-semi}.

 Let $(f,a)\:\B\rarrow\A$ be a morphism of wcDG\+algebras over~$\R$.
 Then the functor $R_f\:H^0(\A\mod\Rcof)\rarrow H^0(\B\mod\Rcof)$
induces a triangulated functor
$$
 \boI R_f\:\sD^\si(\A\mod\Rcof)\lrarrow\sD^\si(\B\mod\Rcof).
$$

 The functor $\boI R_f$ has adjoints on both sides.
 The DG\+functor $E_f\:\B\mod\Rcof_\proj\rarrow\A\mod\Rcof_\proj$
is defined on the level of graded modules by the rule
$\cN\mpsto\A\ot_\B\cN$; the differential on $\A\ot_\B\cN$ induced
by the differentials on $\A$ and $\cN$ is modified to obtain
the differential on $E_f(\cN)$ using the element~$a$.
 Restricting the induced functor between the homotopy categories to
$H^0(\B\mod\Rcof_\proj)_\proj$, we construct the left derived functor
$$
 \boL E_f\:\sD^\si(\B\mod\Rcof)\lrarrow\sD^\si(\A\mod\Rcof),
$$
which is left adjoint to the functor $\boI R_f$.

 Similarly, the DG\+functor $E^f\:\B\mod\Rcof_\inj\rarrow
\A\mod\Rcof_\inj$ is defined on the level of graded modules
by the rule $\cN\mpsto\Hom_\B(\A,\cN)$; the element~$a$ is
used to modify the differential on $\Hom_\B(\A,\cN)$ induced
by the differentials on $\A$ and~$\cN$.
 Restricting the induced functor between the homotopy categories to
$H^0(\B\mod\Rcof_\inj)_\inj$, we obtain the right derived functor
$$
 \boR E_f\:\sD^\si(\B\mod\Rcof)\lrarrow\sD^\si(\A\mod\Rcof),
$$
which is right adjoint to the functor $\boI R_f$.

 The equivalences of semiderived categories $\Phi_\R=\Psi_\R^{-1}$
transform the $\R$\+cofree wcDG\+module restriction- and
extension-of-scalars functors $\boL E_f$, $\boI R_f$, $\boR E^f$
defined above into the $\R$\+free wcDG\+module restriction-
and extension-of-scalars functors $\boL E_f$, $\boI R_f$, $\boR E^f$
defined in Section~\ref{r-free-semi}.

\begin{thm}\hbadness=2600 \label{r-cofree-reduction-quasi}
 The functor\/ $\boI R_f\:\sD^\si(\A\mod\Rcof)\rarrow
\sD^\si(\B\mod\Rcof)$ is an equivalence of triangulated categories
whenever the DG\+algebra morphism $f/\m f\:\B/\m\B\allowbreak
\rarrow\A/\m\A$ is a quasi-isomorphism. \qed
\end{thm}

 Just as in the $\R$\+free wcDG\+module situation of
Theorem~\ref{r-free-reduction-quasi}, the converse statement to
Theorem~\ref{r-cofree-reduction-quasi} is not true.

\Section{$\R$-Free and $\R$-Cofree CDG-Contramodules
and CDG-Comodules}

\subsection{$\R$-free graded coalgebras and contra/comodules}
\label{r-free-graded-co}
 An \emph{$\R$\+free graded coalgebra} $\C$ is, by the definition,
a graded coalgebra object in the tensor category of free
$\R$\+contramodules.
 In other words, it is a free graded $\R$\+contramodule endowed
with a (coassociative, noncocommutative) homogeneous comultiplication
map $\C\rarrow\C\ot^\R\C$ and a homogeneous counit map
$\C\rarrow\R$ satisfying the conventional axioms.

 An \emph{$\R$\+free graded left comodule} $\M$ over $\C$ is
a graded left comodule over $\C$ in the tensor category of free
$\R$\+contramodules, i.~e., a free graded $\R$\+contramodule
endowed with a (coassociative and counital) homogeneous
$\C$\+coaction map $\M\rarrow\C\ot^\R\M$. \
 \emph{$\R$\+free graded right comodules} $\N$ over $\C$ are
defined in the similar way.

 An \emph{$\R$\+free graded left contramodule} $\P$ over $\C$
is a free graded $\R$\+contramodule endowed with
a \emph{$\C$\+contraaction} map $\Hom^\R(\C,\P)\rarrow\P$,
which must be a morphism of graded $\R$\+contramodules satisfying
the conventional contraassociativity and counit axioms.
 This can be rephrased by saying that an $\R$\+free $\C$\+contramodule
is an object of the opposite category to the category of graded
$\C$\+comodules in the module category $(\R\contra^\free)^\sop$ over
the tensor category $\R\contra^\free$ \cite[Section~0.2.4]{Psemi}.

 Specifically, the contraassociativity axiom asserts that
the two compositions $\Hom^\R(\C,\Hom^\R(\C,\P))\simeq
\Hom^\R(\C\ot^\R\C\;\P)\birarrow\Hom^\R(\C,\P)\rarrow\P$
of the maps induced by the contraaction and comultiplication
maps with the contraaction map must be equal to each other.
 The counit axiom claims that the composition
$\P\rarrow\Hom^\R(\C,\P)\rarrow\P$ of the map induced by
the counit map with the contraaction map must be equal to
the identity endomorphism of~$\P$.
 For the details, including the sign rule,
see~\cite[Section~2.2]{Pkoszul}.

 In fact, the categories of $\R$\+free graded $\C$\+comodules and
$\C$\+contramodules are enriched over the tensor category of
(graded) $\R$\+contramodules.
 So the abelian group of morphisms between two $\R$\+free
graded left $\C$\+comodules $\L$ nad $\M$ is the underlying
abelian group of the degree-zero component of
the (not necessarily free) graded $\R$\+contramodule
$\Hom_\C(\L,\M)$ constructed as the kernel of the pair
of morphisms of graded $\R$\+contramodules
$\Hom^\R(\L,\M)\birarrow\Hom^\R(\L\;\C\ot^\R\M)$ induced
by the coactions of $\C$ in $\L$ and~$\M$.
 For the sign rules, which are different for the left and
the right comodules, see~\cite[Section~2.1]{Pkoszul}.

 Similarly, the abelian group of morphisms between two
$\R$\+free graded left $\C$\+contra\-modules $\P$ and $\Q$ is
the underlying abelian group of the zero-degree component
of the (not necessarily free) graded $\R$\+contramodule
$\Hom^\C(\P,\Q)$ constructed as the kernel of the pair
of morphisms of graded $\R$\+contramodules
$\Hom^\R(\P,\Q)\birarrow\Hom^\R(\Hom^\R(\C,\P),\Q)$
induced by the contraactions of $\C$ in $\P$ and~$\Q$.

 The \emph{contratensor product} $\N\ocn_\C\P$ of an $\R$\+free
graded right $\C$\+comodule $\N$ and an $\R$\+free graded left
$\C$\+contramodule $\P$ is the (not necessarily free) graded
$\R$\+contramodule constructed as the cokernel of the pair
of morphisms of graded $\R$\+contramodules
$\N\ot^\R\Hom^\R(\C,\P)\birarrow\N\ot^\R\P$ defined in
terms of the $\C$\+coaction in $\N$, the $\C$\+contraaction
in $\P$, and the evaluation map $\C\ot^\R\Hom^\R(\C,\P)
\rarrow\P$.
 For further details, see~\cite[Section~0.2.6]{Psemi}; for
the sign rule, see~\cite[Section~2.2]{Pkoszul}. 

 The \emph{cotensor product} $\N\oc_\C\M$ of an $\R$\+free
graded right $\C$\+comodule $\N$ and an $\R$\+free graded
left $\C$\+comodule $\M$ is a (not necessarily free) graded
$\R$\+contramodule constructed as the kernel of the pair
of morphisms $\N\ot^\R\M\birarrow\N\ot^\R\C\ot^\R\M$.
 The graded $\R$\+contramodule of \emph{cohomomorphisms}
$\Cohom_\C(\M,\P)$ from an $\R$\+free graded left
$\C$\+comodule $\M$ to an $\R$\+free graded left $\C$\+contramodule
$\P$ is constructed as the cokernel of the pair of morphisms
$\Hom^\R(\C\ot^\R\M\;\P)\simeq\Hom^\R(\M,\Hom^\R(\C,\P))
\birarrow\Hom^\R(\M,\P)$ induced by the $\C$\+coaction in $\M$
and the $\C$\+contraaction in~$\P$.

 For any free graded $\R$\+contramodule $\U$ and any $\R$\+free
graded right $\C$\+comodule $\N$, the free graded $\R$\+contramodule
$\Hom^\R(\N,\U)$ has a natural left $\C$\+contramodule structure
provided by the map $\Hom^\R(\C,\Hom^\R(\N,\U))\simeq
\Hom^\R(\N\ot^\R\C\;\U)\rarrow\Hom^\R(\N,\U)$ induced by
the coaction map $\N\rarrow\N\ot^\R\C$.
 For any $\R$\+free right $\C$\+comodule $\N$ and $\R$\+free
left $\C$\+contramodule $\P$ there is a natural isomorphism
of graded $\R$\+contramodules $\Hom^\R(\N\ocn_\C\P\;\U)\simeq
\Hom^\C(\P,\Hom^\R(\N,\U))$.

 The additive categories of $\R$\+free graded $\C$\+contramodules
and $\C$\+comodules have natural exact category structures:
a short sequence of $\R$\+free graded $\C$\+contra/co\-modules
is said to be exact if it is (split) exact as a short sequence
of free graded $\R$\+contramodules.
 The exact category of $\R$\+free graded $\C$\+comodules admits
infinite direct sums, while the exact category of $\R$\+free
graded $\C$\+contramodules admits infinite products.
 Both of these are preserved by the forgetful functors to
the category of free graded $\R$\+contramodules and preserve
exact sequences.

 If a morphism of $\R$\+free graded $\C$\+contra/comodules has
$\R$\+free kernel and cokernel in the category of
graded $\R$\+contramodules, then these kernel and cokernel
are naturally endowed with $\R$\+free graded
$\C$\+contra/comodule structures, making them the kernel
and cokernel of the same morphism in the additive category
of $\R$\+free graded $\C$\+contra/comodules.

 There are enough projective objects in the exact category of
$\R$\+free graded left $\C$\+contramodules; these are the direct
summands of the graded $\C$\+contramodules $\Hom^\R(\C,\U)$
freely generated by free graded $\R$\+contramodules~$\U$.
 Similarly, there are enough injective objects in the exact
category of $\R$\+free graded left $\C$\+comodules; these are
the direct summands of the graded $\C$\+comodules
$\C\ot^\R\V$ cofreely cogenerated by free graded
$\R$\+contramodules~$\V$.

 For any $\R$\+free graded left $\C$\+comodule $\L$ and any
free graded $\R$\+contramodule $\V$ there is a natural isomorphism
of graded $\R$\+contramodules $\Hom_\C(\L\;\C\ot^\R\V)\simeq
\Hom^\R(\L,\V)$.
 For any $\R$\+free graded left $\C$\+contramodule $\Q$
and any free graded $\R$\+contramodule $\U$, there is a natural
isomorphism of graded $\R$\+contramodules $\Hom^\C(\Hom^\R(\C,\U),
\Q)\simeq\Hom^\R(\U,\Q)$
\cite[Sections~1.1.1\+-2 and~3.1.1\+-2]{Psemi}.
 
 For any $\R$\+free graded right $\C$\+comodule $\N$ and any
free graded $\R$\+contramodule $\U$, there is a natural
isomorphism of graded $\R$\+contramodules $\N\ocn_\C\Hom^\R(\C,\U)
\simeq\N\ot^\R\U$.
 The proof following~\cite[Section~5.1.1]{Psemi}
requires considering non-$\R$-free $\C$\+contramodules
(see Section~\ref{non-adj-graded-co}); alternatively,
a direct argument can be made based on the idea of the proof
of~\cite[Lemma~1.1.2]{Psemi}.

 For any $\R$\+free graded right $\C$\+comodule $\N$ and
any free graded $\R$\+contramodule $\V$, there is a natural
isomorphism of graded $\R$\+contramodules $\N\oc_\C(\C\ot^\R\V)
\simeq\N\ot^\R\V$.
 For any $\R$\+free graded left $\C$\+contramodule $\P$ and
any free graded $\R$\+contramodule $\V$, there is a natural
isomorphism of graded $\R$\+contramodules $\Cohom_\C(\C\ot^\R\V\;
\P)\simeq\Hom^\R(\V,\P)$.
 For any $\R$\+free graded left $\C$\+comodule $\M$ and any
free graded $\R$\+contramodule $\U$, there is a natural
isomorphism of graded $\R$\+contramodules $\Cohom_\C(\M,
\Hom^\R(\C,\U))\simeq\Hom^\R(\M,\U)$
\cite[Sections~1.2.1 and~3.2.1]{Psemi}.

\begin{lem} \label{r-free-co-op-reduction}
 \textup{(a)} Let\/ $\F$ be a projective\/ $\R$\+free graded left\/
$\C$\+contramodule and\/ $\Q$ be an arbitrary\/ $\R$\+free graded
left\/ $\C$\+contramodule.
 Then the graded\/ $\R$\+contramodule\/ $\Hom^\C(\F,\Q)$ is
free, and the functor of reduction modulo\/~$\m$  induces
an isomorphism of graded $k$\+vector spaces
$$
 \Hom^\C(\F,\Q)/\m\Hom^\C(\F,\Q)\simeq\Hom^{\C/\m\C}
 (\F/\m\F\;\Q/\m\Q).
$$ \par
 \textup{(b)} Let\/ $\L$ be an arbitrary\/ $\R$\+free graded left\/
$\C$\+comodule and\/ $\J$ be an injective\/ $\R$\+free graded left\/
$\C$\+comodule.
 Then the graded\/ $\R$\+contramodule\/ $\Hom_\C(\L,\J)$ is
free, and the functor of reduction modulo\/~$\m$ induces
an isomorphism of graded $k$\+vector spaces
$$
 \Hom_\C(\L,\J)/\m\Hom_\C(\L,\J)\simeq\Hom_{\C/\m\C}
 (\L/\m\L\;\J/\m\J).
$$ \par
 \textup{(c)} For any\/ $\R$\+free graded right\/ $\C$\+comodule\/
$\N$ and\/ $\R$\+free graded left\/ $\C$\+contra\-module\/ $\P$,
there is a natural isomorphism of graded $k$\+vector spaces
$$
 (\N\ocn_\C\P)/\m(\N\ocn_\C\P)\simeq(\N/\m\N)\ocn_{\C/\m\C}
 (\P/\m\P).
$$
For any\/ $\R$\+free graded right\/ $\C$\+comodule\/ $\N$ and any
projective $\R$\+free graded left $\C$\+contramodule\/ $\F$,
the graded\/ $\R$\+contramodule\/ $\N\ocn_\C\F$ is free. \par
 \textup{(d)} Let\/ $\N$ be an arbitrary\/ $\R$\+free graded right\/
$\C$\+comodule and\/ $\J$ be an injective\/ $\R$\+free graded
left\/ $\C$\+comodule.
 Then the graded\/ $\R$\+contramodule\/ $\N\oc_\C\J$ is free,
and there is a natural isomorphism of graded $k$\+vector spaces
$$
 (\N\oc_\C\J)/\m(\N\oc_\C\J)\simeq(\N/\m\N)\oc_{\C/\m\C}(\J/\m\J).
$$ \par
 \textup{(e)} For any\/ $\R$\+free graded left\/ $\C$\+comodule\/
$\M$ and\/ $\R$\+free graded left\/ $\C$\+contra\-module\/ $\P$,
there is a natural isomorphism of graded $k$\+vector spaces
$$
 \Cohom_\C(\M,\P)/\m\Cohom_\C(\M,\P)\simeq
 \Cohom_{\C/\m\C}(\M/\m\M\;\P/\m\P).
$$
 For any injective\/ $\R$\+free graded left\/ $\C$\+comodule\/
$\J$ and any\/ $\R$\+free graded left\/ $\C$\+contramodule\/ $\P$,
the graded\/ $\R$\+contramodule\/ $\Cohom_\C(\J,\P)$ is free.
 For any\/ $\R$\+free graded left\/ $\C$\+comodule\/ $\M$ and
any projective\/ $\R$\+free graded left\/ $\C$\+contramodule\/ $\F$,
the graded\/ $\R$\+contramodule\/ $\Cohom_\C(\M,\F)$ is free.
\end{lem}

\begin{proof}
 We will spell out the proofs of parts~(a) and~(c\+d); part~(b)
is similar to~(a), and (e)~is similar to~(c) (see also the proof of
Lemma~\ref{r-free-hom-reduction}).
 Clearly, $\C/\m\C$ is a graded coalgebra over~$k$, and the reduction
modulo~$\m$ takes $\R$\+free $\C$\+comodules to comodules over
$\C/\m\C$ and $\R$\+free $\C$\+contramodules to contramodules over
$\C/\m\C$, since the functor of reduction modulo~$\m$ preserves
the tensor products of $\R$\+contramodules and the internal $\Hom$
from a free $\R$\+contramodule.

 For any $\R$\+free left $\C$\+contramodules $\P$ and $\Q$,
the reduction functor induces a natural morphism of graded
$k$\+vector spaces
$$
 \Hom^\C(\P,\Q)/\m\Hom^\C(\P,\Q)\lrarrow
 \Hom^{\C/\m\C}(\P/\m\P\;\Q/\m\Q).
$$
 Now if $\F=\Hom^\R(\C,\U)$ is a graded $\C$\+contramodule freely
generated by a free $\R$\+contramodule $\U$, then
$\Hom^\C(\F,\Q)\simeq\Hom^\R(\U,\Q)$ and $\F/\m\F\simeq
\Hom_k(\C/\m\C\;\allowbreak\U/\m\U)$, hence
$\Hom^{\C/\m\C}(\F/\m\F,\Q/\m\Q)\simeq\Hom_k(\U/\m\U\;\Q/\m\Q)$
and the desired isomorphism~(a) follows.

 Analogously, for any $\R$\+free left $\C$\+comodules $\L$ and $\M$,
the functor of reduction modulo~$\m$ induces a natural morphism
of graded $k$\+vector spaces
$$
 \Hom_\C(\L,\M)/\m\Hom_\C(\L,\M)\lrarrow
 \Hom_{\C/\m\C}(\L/\m\L\;\M/\m\M),
$$
and the rest of the argument is as above.

 To prove the first assertion of part~(c), it suffices to notice that
the functor of reduction modulo~$\m$ preserves the internal $\Hom$
from a free $\R$\+contramodule, the tensor products, and the cokernels
in the category of $\R$\+contramodules.
 The second assertion is similar to the above.

 For any $\R$\+free graded right $\C$\+comodule $\N$ and $\R$\+free
graded left $\C$\+comodule $\M$, there is a natural morphism of
$k$\+vector spaces
$$
 (\N\oc_\C\M)/\m(\N\oc_\C\M)\lrarrow(\N/\m\N)\oc_{\C/\m\C}(\M/\m\M) 
$$
induced by the composition $\N\oc_\C\M\rarrow\N\ot^\R\M\rarrow
\N/\m\N\ot_k\M/\m\M$.
 If $\J=\C\ot^\R\V$ is a graded $\C$\+comodule cofreely cogenerated
by a free $\R$\+contramodule $\V$, then $\N\oc_\C\J\simeq\N\ot^\R\V$
and $\J/\m\J\simeq\C/\m\C\ot_k\V/\m\V$, hence
$(\N/\m\N)\oc_\C(\J/\m\J)\simeq(\N/\m\N)\ot_k(\V/\m\V)$ and
the isomorphism~(d) follows.
\end{proof}

\begin{lem} \label{r-free-co-inj-proj-reduction}
\textup{(a)} An\/ $\R$\+free graded comodule\/ $\M$ over
an\/ $\R$\+free graded coalgebra\/ $\C$ is injective if and only if
the graded comodule\/ $\M/\m\M$ over the graded $k$\+coalgebra
$\C/\m\C$ is injective. \par
\textup{(b)} An\/ $\R$\+free graded contramodule\/ $\P$ over
an\/ $\R$\+free graded coalgebra\/ $\C$ is projective if and
only if the graded contramodule\/ $\P/\m\P$ over the graded
$k$\+coalgebra $\C/\m\C$ is projective.
\end{lem}

\begin{proof}
 Similar to the proof of Lemma~\ref{r-free-proj-inj-reduction}.
\end{proof}

 Recall that the homological dimensions of the abelian categories
of graded left comodules, graded right comodules, and graded
contramodules over a graded coalgebra $C$ over a field $k$ are
equal to each other~\cite[Section~4.5]{Pkoszul}.
 This number is called the \emph{homological dimension} of
a graded coalgebra~$C$.

\begin{cor} \label{r-free-co-homol-dim}
 The homological dimensions of the exact categories of\/
$\R$\+free graded left\/ $\C$\+comodules, $\R$\+free graded
right\/ $\C$\+comodules, and\/ $\R$\+free graded left
$\C$\+contramod\-ules do not exceed that of the graded
coalgebra\/ $\C/\m\C$ over~$k$.  \qed
\end{cor}

 Given a projective $\R$\+free graded left $\C$\+contramodule $\P$,
the $\R$\+free graded left $\C$\+comodule $\Phi_\C(\P)$ is defined
as $\Phi_\C(\P)=\C\ocn_\C\P$.
 Here the contratensor product is $\R$\+free by
Lemma~\ref{r-free-co-op-reduction}(c), and is endowed with a left
$\C$\+comodule structure as an $\R$\+free cokernel of a morphism
of $\R$\+free left $\C$\+comodules.
 Given an injective $\R$\+free graded left $\C$\+comodule $\M$,
the $\R$\+free graded left $\C$\+contramodule $\Psi_\C(\M)$ is
defined as $\Psi_\C(\M)=\Hom_\C(\C,\M)$.
 Here the $\Hom$ of $\C$\+comodules is $\R$\+free by
Lemma~\ref{r-free-co-op-reduction}(b), and is endowed with a left
$\C$\+contramodule structure as an $\R$\+free kernel of a morphism
of $\R$\+free left $\C$\+contramodules.

\begin{prop} \label{r-free-c-co-contra}
 The functors\/ $\Phi_\C$ and\/ $\Psi_\C$ are mutually inverse
equivalences between the additive categories of projective\/
$\R$\+free graded\/ $\C$\+contramodules and injective\/ $\R$\+free
graded left\/ $\C$\+comodules.
\end{prop}

\begin{proof}
 One has $\Phi_\C(\Hom^\R(\C,\U))=\C\ot^\R\U$ and
$\Psi_\C(\C\ot^\R\V)=\Hom^\R(\C,\V)$.
\end{proof}

\subsection{Contra/coderived category of $\R$-free
CDG-contra/comodules}  \label{r-free-co-derived}
 Let $\C$ be an $\R$\+free graded coalgebra.
 A homogeneous $\R$\+contramodule endomorphism (differential)
$d\:\C\rarrow\C$ (of degree~$1$) is said to be
an \emph{odd coderivation} of $\C$ if it forms a commutative
diagram with the comultiplication map $\C\rarrow\C\ot^\R\C$
and the induced differential on $\C\ot^\R\C$.

 An \emph{odd coderivation} $d_\M$ of an $\R$\+free graded left
$\C$\+comodule $\M$ \emph{compatible with} the coderivation~$d$
on $\C$ is a differential ($\R$\+contramodule endomorphism of
degree~$1$) such that the coaction map $\M\rarrow\C\ot^\R\M$
forms a commutative diagram with $d_\M$ and the differential
on $\C\ot^\R\M$ induced by $d$ and~$d_\M$.
 Odd coderivations of $\R$\+free graded right $\C$\+comodules are
defined similarly.

 An \emph{odd contraderivation} $d_\P$ of an $\R$\+free graded
left $\C$\+contramodule $\P$ \emph{compatible with}
the coderivation~$d$ on $\C$ is a differential ($\R$\+contramodule
endomorphism of degree~$1$) such that the contraaction map
$\Hom^\R(\C,\P)\rarrow\P$ forms a commutative diagram with $d_\P$
and the differential on $\Hom^\R(\C,\P)$ induced by $d$ and~$d_\P$.

 Given an $\R$\+free graded coalgebra $\C$, the free graded
$\R$\+contramodule $\C^*=\Hom^\R(\C,\R)$ is endowed with
an $\R$\+free graded algebra structure provided by
the multiplication map $\Hom^\R(\C,\R)\ot^\R\Hom^\R(\C,\R)
\rarrow\Hom^\R(\C\ot^\R\C\;\R)\rarrow\Hom^\R(\C,\R)$ and
the unit map $\R\rarrow\Hom^\R(\C,\R)$ induced by
the comultiplication and counit morphisms of~$\C$.

 An $\R$\+free graded left comodule $\M$ over $\C$ is endowed
with a graded left $\C^*$\+module structure provided by the action
map $\C^*\ot^\R\M\rarrow\C^*\ot^\R\C\ot^\R\M\rarrow\M$.
 \ $\R$\+free graded right comodules over $\C$ are endowed with
graded right $\C^*$\+module structures in the similar way.
 An $\R$\+free graded left contramodule $\P$ over $\C$ is
endowed with a graded left $\C^*$\+module structure provided by
the action map $\C^*\ot^\R\P\rarrow\Hom^\R(\C,\P)\rarrow\P$.
 For the left-right and sign rules, see~\cite[Section~4.1]{Pkoszul}.
 Following~\cite{Pkoszul}, we will denote the multiplication in $\C^*$
and its action in comodules and contramodules over $\C$ by
the~$*$ symbol.

 An \emph{$\R$\+free CDG\+coalgebra} $\C$ is an $\R$\+free graded
coalgebra endowed with an odd coderivation~$d$ of degree~$1$ and
a homogeneous $\R$\+contramodule morphism $h\:\C\rarrow\R$ of
degree~$2$ (that is, $h$~vanishes on all the components of $\C$
except $\C^{-2}$) satisfying the equations $d^2(c) = h*c-c*h$
and $h(d(c))=0$ for all $c\in\C$.
 Morphisms $\C\rarrow\D$ of $\R$\+free CDG\+coalgebras are defined
as pairs $(f,a)$, with $f\:\C\rarrow\D$ being a morphism of
$\R$\+free graded coalgebras and $a\:\C\rarrow\R$ a homogeneous
$\R$\+contramodule morphism of degree~$1$, satisfying
the conventional equations~\cite[Section~4.1]{Pkoszul}.

 An \emph{$\R$\+free left CDG\+comodule} $\M$ over $\C$ is
an $\R$\+free graded left $\C$\+comodule endowed with an odd
coderivation $d_\M\:\M\rarrow\M$ of degree~$1$ compatible with
the coderivation $d$ on $\C$ and satisfying the equation
$d^2_\M(x)=h*x$ for all $x\in\M$.
 The definition of an \emph{$\R$\+free right CDG\+comodule} $\N$
over $\C$ is similar; the only difference is that the equation
for the square of the differential has the form $d_\N^2(y) =
- y*h$ for all $y\in\N$.
 An \emph{$\R$\+free left CDG\+contramodule} $\P$ over $\C$ is
an $\R$\+free graded left $\C$\+contramodule endowed with an odd
contraderivation $d_\P\:\P\rarrow\P$ of degree~$1$ compatible with
the coderivation $d$ on $\C$ and satisfying the equation
$d^2_\P(p)=h*p$ for all $p\in\P$.

 An \emph{$\R$\+free DG\+coalgebra} $\C$ is an $\R$\+free
CDG\+coalgebra with $h=0$.
 A morphism of $\R$\+free DG\+coalgebras $\C\rarrow\D$ is
a morphism $f=(f,a)$ between $\C$ and $\D$ considered as
$\R$\+free CDG\+coalgebras such that $a=0$.
 A \emph{DG\+comodule} or \emph{DG\+contramodule} over an $\R$\+free
DG\+coalgebra $\C$ is the same thing as a CDG\+comodule or
CDG\+contramodule over $\C$ considered as a CDG\+coalgebra.

 Let $\C$ be an $\R$\+free CDG\+coalgebra.
 $\R$\+free left (resp., right) CDG\+comodules over $\C$ naturally
form a DG\+category, which we will denote by $\C\comod\Rfr$
(resp., $\comodrRfr\C$).
 The similar DG\+category of $\R$\+free left CDG\+contramodules
over $\C$ will be denoted by $\C\contra\Rfr$.
 In fact, these DG\+categories are enriched over the tensor
category of (complexes of) $\R$\+contramodules, so the complexes
of morphisms in $\C\comod\Rfr$, \ $\comodrRfr\C$, and
$\C\contra\Rfr$ are the underlying complexes of abelian groups
for naturally defined complexes of $\R$\+contramodules.

 The underlying graded $\R$\+contramodules of these complexes were
defined in Section~\ref{r-free-graded-co}.
 Given two left CDG\+comodules $\L$ and $\M\in\C\comod\Rfr$,
the differential in the graded $\R$\+contramodule $\Hom_\C(\L,\M)$
is defined by the conventional formula
$d(f)(x) = d_\M(f(x)) - (-1)^{|f|}f(d_\L(x))$; one easily
checks that $d^2(f)=0$.
 The complex of $\R$\+contramodules $\Hom_{\C^\op}(\K,\N)$ for
any two right CDG\+comodules $\K$ and $\N\in\comodrRfr\C$ and
the complex of $\R$\+contramodules $\Hom^\C(\P,\Q)$ for any two
left CDG\+contramodules $\P$ and $\Q\in\C\contra\Rfr$
are defined in the similar way.

 Passing to the zero cohomology of the complexes of morphisms, we
construct the homotopy categories of $\R$\+free CDG\+comodules
and CDG\+contramodules $H^0(\C\comod\Rfr)$, \ $H^0(\comodrRfr\C)$,
and $H^0(\C\contra\Rfr)$.
 Even though these are also naturally enriched over
$\R$\+contramodules, we will mostly consider them as conventional
categories with abelian groups of morphisms
(see Remark~\ref{contra-enriched}).
 Since the DG\+categories $\C\comod\Rfr$, \ $\comodrRfr\C$, and
$\C\contra\Rfr$ have shifts and twists, their homotopy categories
are triangulated.
 The DG- and homotopy categories of $\R$\+free CDG\+comodules
also have infinite direct sums, while the DG- and homotopy
categories of $\R$\+free CDG\+contramodules have infinite products.
{\hbadness=2500\par}

 The \emph{contratensor product} $\N\ocn_\C\P$ of an $\R$\+free 
right CDG\+comodule $\N$ and an $\R$\+free left CDG\+contramodule
$\P$ over $\C$ is a complex of $\R$\+contramodules obtained by
endowing the graded $\R$\+contramodule $\N\ocn_\C\P$ constructed
in Section~\ref{r-free-graded-co} with the conventional tensor
product differential $d(y\ocn p) = d_\N(y)\ocn p +
(-1)^{|y|}y\ocn d_\P(p)$.
 The contratensor product of $\R$\+free right CDG\+comodules and
$\R$\+free left CDG\+contramodules over $\C$ is a triangulated
functor of two arguments
$$
 \ocn_\C\:H^0(\comodrRfr\C)\times H^0(\C\contra\Rfr)\lrarrow
 H^0(\R\contra).
$$

 The \emph{cotensor product} $\N\oc_\C\M$ of an $\R$\+free
right CDG\+comodule $\N$ and an $\R$\+free left CDG\+comodule $\M$
over $\C$ is a complex of $\R$\+contramodules obtained by
endowing the graded $\R$\+contramodule $\N\oc_\C\M$ with
the conventional tensor product differential.
 The cotensor product of $\R$\+free CDG\+comodules over $\C$ is
a triangulated functor of two arguments
$$
 \oc_\C\:H^0(\comodrRfr\C)\times H^0(\C\comod\Rfr)\lrarrow
 H^0(\R\contra).
$$

 The complex $\Cohom_\C(\M,\P)$ from an $\R$\+free left
CDG\+comodule $\M$ to an $\R$\+free left CDG\+contramodule $\P$
over $\C$ is a complex of $\R$\+contramodules obtained by
endowing the graded $\R$\+contramodule $\Cohom_\C(\M,\P)$
constructed in Section~\ref{r-free-graded-co} with
the conventional $\Hom$ differential.
 The $\Cohom$ between $\R$\+free left CDG\+comodules and
CDG\+contramodules over $\C$ is a triangulated functor of
two arguments
$$
 \Cohom_\C\:H^0(\C\comod\Rfr)^\sop\times H^0(\C\contra\Rfr)
 \lrarrow H^0(\R\contra).
$$

 An $\R$\+free left CDG\+comodule over $\C$ is called \emph{coacyclic}
if it belongs to the minimal triangulated subcategory of the homotopy
category $H^0(\C\comod\Rfr)$ containing the totalizations of
short exact sequences of $\R$\+free CDG\+comodules over $\C$ and
closed under infinite direct sums.
 The quotient category of $H^0(\C\comod\Rfr)$ by the thick subcategory
of coacyclic $\R$\+free CDG\+comodules is called the \emph{coderived
category} of $\R$\+free left CDG\+comodules over $\C$ and denoted
by $\sD^\co(\C\comod\Rfr)$.
 The coderived category of $\R$\+free right CDG\+comodules over $\C$,
denoted by $\sD^\co(\comodrRfr\C)$, is defined similarly
(see~\cite[Sections~1.2 and~4.2]{Pkoszul} or~\cite[Sections~0.2.2
and~2.1]{Psemi}).

 An $\R$\+free left CDG\+contramodule over $\C$ is called
\emph{contraacyclic} if it belongs to the minimal triangulated
subcategory of the homotopy category $H^0(\C\contra\Rfr)$ containing
the totalizations of short exact sequences of $\R$\+free
CDG\+contramodules over $\C$ and closed under infinite products.
 The quotient category of $H^0(\C\contra\Rfr)$ by the thick subcategory
of contraacyclic $\R$\+free CDG\+contramodules is called
the \emph{contraderived category} of $\R$\+free left CDG\+contramodules
over $\C$ and denoted by $\sD^\ctr(\C\contra\Rfr)$
(see~\cite[Sections~1.2 and~4.2]{Pkoszul} or~\cite[Sections~0.2.5
and~4.1]{Psemi}).

 Denote by $\C\contra\Rfr_\proj\subset\C\contra\Rfr$ the full
DG\+subcategory formed by all the $\R$\+free CDG\+contramodules
over $\C$ whose underlying $\R$\+free graded $\C$\+contramodules
are projective.
 Similarly, denote by $\C\comod\Rfr_\inj\subset\C\comod\Rfr$ the full
DG\+subcate\-gory formed by all the $\R$\+free CDG\+comodules over $\C$
whose underlying $\R$\+free graded $\C$\+comodules are injective.
 The corresponding homotopy categories are denoted by
$H^0(\C\contra\Rfr_\proj)$ and $H^0(\C\comod\Rfr_\inj)$, respectively.
{\emergencystretch=0em\par}

\begin{lem} \label{contractible-reduction-co}
 \textup{(a)} Let\/ $\P$ be a CDG\+contramodule from\/
$\C\contra\Rfr_\proj$.
 Then\/ $\P$ is contractible (i.~e., represents a zero object in
$H^0(\C\contra\Rfr)$) if and only if the CDG\+contramodule\/
$\P/\m\P$ over the CDG\+coalgebra\/ $\C/\m\C$ over~$k$
is contractible. \par
\textup{(b)} Let\/ $\M$ be a CDG\+comodule from\/ $\C\comod\Rfr_\inj$.
 Then\/ $\M$ is contractible if and only if the CDG\+comodule\/
$\M/\m\M$ over the CDG\+coalgebra\/ $\C/\m\C$ is contractible.
\end{lem}

\begin{proof}
 See the proof of Lemma~\ref{contractible-reduction}.
\end{proof}

\begin{thm} \label{r-free-co-derived-thm}
 Let\/ $\C$ be an\/ $\R$\+free CDG\+coalgebra.  Then \par
\textup{(a)} for any CDG\+contramodule\/
$\P\in H^0(\C\contra\Rfr_\proj)$ and any contraacyclic\/ $\R$\+free
left CDG\+contramodule\/ $\Q$ over\/ $\C$, the complex of\/
$\R$\+contramodules\/ $\Hom^\C(\P,\Q)$ is contractible; \par
\textup{(b)} for any coacyclic\/ $\R$\+free left CDG\+comodule\/ $\L$
over\/ $\C$ and any CDG\+comodule\/ $\M\in H^0(\C\comod\Rfr_\inj)$,
the complex of\/ $\R$\+contramodules\/ $\Hom_\C(\L,\M)$ is
contractible; \par
\textup{(c)} the composition of natural functors
$$
 H^0(\C\contra\Rfr_\proj)\lrarrow H^0(\C\contra\Rfr)\lrarrow
 \sD^\ctr(\C\contra\Rfr)
$$
is an equivalence of triangulated categories; \par
\textup{(d)} the composition of natural functors
$$
 H^0(\C\comod\Rfr_\inj)\lrarrow H^0(\C\comod\Rfr)\lrarrow
 \sD^\co(\C\comod\Rfr)
$$
is an equivalence of triangulated categories.
\end{thm}

\begin{proof}
 Parts~(a\+b) follow from the first assertions of
Lemma~\ref{r-free-co-op-reduction}(a\+b) together with the observations
that the total complex of a short exact sequence of complexes of
free $\R$\+contramodules is contractible and the infinite products
of complexes of $\R$\+contramodules preserve contractibility
(cf.\ Theorem~\ref{non-adj-co-derived-res} below).

 Parts~(c\+d) can be proven in the way similar to the proof of
~\cite[Theorem~4.4(c\+d)]{Pkoszul}, or alternatively deduced
from~\cite[Remark~3.7]{Pkoszul}.
 The key observations are that the class of projective $\R$\+free
$\C$\+contramodules is closed under infinite products, and
the class of injective $\R$\+free $\C$\+comodules is closed under
infinite direct sums.
\end{proof}

\begin{cor}  \label{r-free-co-acycl-reduction}
 For any\/ $\R$\+free CDG\+coalgebra $\C$, an\/ $\R$\+free
CDG\+contramodule\/ $\P$ over\/ $\C$ is contraacyclic if and only if
the CDG\+contramodule\/ $\P/\m\P$ over the CDG\+coalgebra\/ $\C/\m\C$
over the field~$k$ is contraacyclic.
 An\/ $\R$\+free CDG\+comodule\/ $\M$ over\/ $\C$ is coacyclic if and
only if the CDG\+comodule\/ $\M/\m\M$ over the CDG\+coalgebra\/
$\C/\m\C$ over~$k$ is coacyclic.
\end{cor}

\begin{proof}
 See the proof of Corollary~\ref{r-free-abs-acycl-reduction}.
\end{proof}

 Given a CDG\+contramodule $\P\in\C\contra\Rfr_\proj$,
the graded left $\C$\+comodule $\Phi_\C(\P)=\C\ocn_\C\P$
is endowed with a CDG\+comodule structure with the conventional
tensor product differential.
 Conversely, given a CDG\+comodule $\M\in\C\comod\Rfr_\inj$,
the graded left $\C$\+contramodule $\Psi_\C(\M)=
\Hom_\C(\C,\M)$ is endowed with a CDG\+con\-tramodule structure
with the conventional $\Hom$ differential.
 One easily checks that $\Phi_\C$ and $\Psi_\C$ are
mutually inverse equivalences between the DG\+categories
$\C\contra\Rfr_\proj$ and $\C\comod\Rfr_\inj$
(see Proposition~\ref{r-free-c-co-contra}).

\begin{cor} \label{r-free-derived-co-contra}
 The derived functors
$$
 \boL\Phi_\C\:\sD^\ctr(\C\contra\Rfr)\lrarrow\sD^\co(\C\comod\Rfr)
$$
and
$$
 \boR\Psi_\C\:\sD^\co(\C\comod\Rfr)\lrarrow\sD^\ctr(\C\contra\Rfr)
$$
defined by identifying\/ $\sD^\ctr(\C\contra\Rfr)$ with
$H^0(\C\contra\Rfr_\proj)$ and\/ $\sD^\co(\C\comod\Rfr)$ with
$H^0(\C\comod\Rfr_\inj)$ are mutually inverse equivalences
between the contraderived category\/ $\sD^\ctr(\C\contra\Rfr)$
and the coderived category\/ $\sD^\co(\C\comod\Rfr)$. \qed
\hfuzz=1.5pt
\end{cor}

 An $\R$\+free CDG\+contra/comodule over $\C$ is said to be
\emph{absolutely acyclic} if it belongs to the minimal thick
subcategory of the homotopy category of $\R$\+free
CDG\+contra/comodules over $\C$ containing the totalizations of
short exact sequences of $\R$\+free CDG\+contra/comodules.
 The quotient category of the homotopy category $H^0(\C\comod\Rfr)$,
\ $H^0(\comodrRfr\C)$, or $H^0(\C\contra\Rfr)$ by the thick
subcategory of absolutely acyclic CDG\+contra/comodules is 
called the \emph{absolute derived category} of (left or right)
CDG\+contra/comodules over $\C$ and denoted by
$\sD^\abs(\C\comod\Rfr)$, \ $\sD^\abs(\comodrRfr\C)$, or
$\sD^\abs(\C\contra\Rfr)$, respectively.

 The following result is to be compared with
Theorem~\ref{r-free-absolute-derived}.

\begin{thm} \label{r-free-finite-dim-co-derived}
 Let\/ $\C$ be an\/ $\R$\+free CDG\+coalgebra.
 Then whenever the exact category of\/ $\R$\+free graded left\/
$\C$\+contramodules has finite homological dimension, the classes of
contraacyclic and absolutely acyclic\/ $\R$\+free left
CDG\+contramodules over\/ $\C$ coincide, so\/
$\sD^\ctr(\C\contra\Rfr)=\sD^\abs(\C\contra\Rfr)$.
 Whenever the exact category of\/ $\R$\+free graded left\/
$\C$\+comodules has finite homological dimension, the classes of
coacyclic and absolutely acyclic\/ $\R$\+free left CDG\+comodules
over\/ $\C$ coincide, so\/
$\sD^\co(\C\comod\Rfr)=\sD^\abs(\C\comod\Rfr)$.
\end{thm}

\begin{proof}
 See~\cite[Theorem~4.5 or Remark~3.6]{Pkoszul}.
 A more general (but also much more difficult) argument can be
found in~\cite[Theorem~1.6]{Psing}.
\end{proof}

 The right derived functor of homomorphisms of $\R$\+free
CDG\+contramodules
$$
 \Ext^\C\:\sD^\ctr(\C\contra\Rfr)^\sop\times
 \sD^\ctr(\C\contra\Rfr)\lrarrow H^0(\R\contra^\free)
$$
is constructed by restricting the functor $\Hom^\C$ to the full
subcategory $H^0(\C\contra\Rfr_\proj)^\sop\allowbreak\times
H^0(\C\contra\Rfr)\subset H^0(\C\contra\Rfr)^\sop\times
H^0(\C\contra\Rfr)$.
 The restriction takes values in $H^0(\R\contra^\free)$ by
Lemma~\ref{r-free-co-op-reduction}(a) and factorizes through
the Cartesian product of contraderived categories by
Theorem~\ref{r-free-co-derived-thm}(a) and~(c).
{\hfuzz=12.6pt\par}

 The right derived functor of homomorphisms of $\R$\+free
CDG\+comodules 
$$
 \Ext_\C\:\sD^\co(\C\comod\Rfr)^\sop\times
 \sD^\co(\C\comod\Rfr)\lrarrow H^0(\R\contra^\free)
$$
is constructed by restricting the functor $\Hom_\C$ to the full
subcategory $H^0(\C\comod\Rfr)^\sop\allowbreak\times
H^0(\C\comod\Rfr_\inj)\subset H^0(\C\comod\Rfr)^\sop\times
H^0(\C\comod\Rfr)$.
 The restriction takes values in $H^0(\R\contra^\free)$ by
Lemma~\ref{r-free-co-op-reduction}(b) and factorizes through
the Cartesian product of coderived categories by
Theorem~\ref{r-free-co-derived-thm}(b) and~(d).
{\hfuzz=14.3pt\par}

 The left derived functor of contratensor product of $\R$\+free
CDG\+comodules and CDG\+contramodules
$$
 \Ctrtor^\C\:\sD^\co(\comodrRfr\C)\times\sD^\ctr(\C\contra\Rfr)
 \lrarrow H^0(\R\contra^\free)
$$
is constructed by restricting the functor $\ocn_\C$ to the full
subcategory $H^0(\comodrRfr\C)\times H^0(\C\contra\Rfr_\proj)
\subset H^0(\comodrRfr\C)\times H^0(\C\contra\Rfr)$.
 The restriction takes values in $H^0(\R\contra^\free)$ by
Lemma~\ref{r-free-co-op-reduction}(c) and factorizes through
the coderived category in the first argument because
the contratensor product with a projective $\R$\+free graded
left $\C$\+contramodule preserves short exact sequences and
infinite direct sums of $\R$\+free graded right $\C$\+comodules.

 The right derived functor of cotensor product of $\R$\+free
CDG\+comodules
$$
 \Cotor^\C\:\sD^\co(\comodrRfr\C)\times\sD^\co(\C\comod\Rfr)
 \lrarrow H^0(\R\contra^\free)
$$
is constructed by restricting the functor $\oc_\C$ to either of
the full subcategories $H^0(\comodrRfr\C)\times H^0(\C\comod\Rfr_\inj)$
or $H^0(\comodrRfrinj\C)\times H^0(\C\comod\Rfr)\subset
H^0(\comodrRfr\C)\times H^0(\C\comod\Rfr)$.
 Here $\comodrRfrinj\C$ denotes the DG\+category of $\R$\+free
right CDG\+comodules over $\C$ with injective underlying
$\R$\+free graded $\C$\+comodules, and $H^0(\comodrRfrinj\C)$
is the corresponding homotopy category.
 The restriction takes values in $H^0(\R\contra^\free)$ by
Lemma~\ref{r-free-co-op-reduction}(d) and factorizes through
the Cartesian product of coderived categories because
the cotensor product with an injective $\R$\+free graded
$\C$\+comodule preserves short exact sequences and infinite
direct sums of $\R$\+free graded $\C$\+comodules.

 The left derived functor of cohomomorphisms of $\R$\+free
CDG\+comodules and CDG\+contramodules
$$
 \Coext_\C\:\sD^\co(\C\comod\Rfr)^\sop\times
 \sD^\ctr(\C\contra\Rfr)\lrarrow H^0(\R\contra^\free)
$$
is constructed by restricting the functor $\Cohom_\C$ to either of
the full subcategories $H^0(\C\comod\Rfr_\inj)^\sop\times
H^0(\C\contra\Rfr)$ or $H^0(\C\comod\Rfr)^\sop\times
H^0(\C\contra\Rfr_\proj)\subset H^0(\C\comod\Rfr)\times
H^0(\C\contra\Rfr)$.
 The restriction takes values in $H^0(\R\contra^\free)$ by
Lemma~\ref{r-free-co-op-reduction}(e) and factorizes through
the Cartesian product of the coderived and contraderived
categories for the reasons similar to the ones explained above.

\begin{prop}  \label{r-free-cotor-ctrtor}
\textup{(a)} The equivalence of triangulated categories
$$
 \boL\Phi_\C\:\sD^\ctr(\C\contra\Rfr)\simeq
 \sD^\co(\C\comod\Rfr)\,\,\:\!\boR\Psi_\C
$$
from Corollary~\textup{\ref{r-free-derived-co-contra}} transforms
the left derived functor
$$
 \Coext_\C\:\sD^\co(\C\comod\Rfr)^\sop\times\sD^\ctr(\C\contra\Rfr)
 \lrarrow H^0(\R\contra^\free)
$$
into the right derived functors
$$
 \Ext^\C\:\sD^\ctr(\C\contra\Rfr)^\sop\times\sD^\ctr(\C\contra\Rfr)
 \lrarrow H^0(\R\contra^\free)
$$
and
$$
 \Ext_\C\:\sD^\co(\C\comod\Rfr)^\sop\times\sD^\co(\C\comod\Rfr)
 \lrarrow H^0(\R\contra^\free).
$$
 In other words, the following diagram of categories, functors,
and equivalences is commutative:
$$
\begin{diagram}
\node{\sD^\ctr(\C\contra\Rfr)^\sop\times\sD^\ctr(\C\contra\Rfr)}
\arrow{s,=}\arrow[4]{e,t}{\Ext^\C}\node[4]{H^0(\R\contra^\free)}
\arrow{s,=} \\
\node{\sD^\co(\C\comod\Rfr)^\sop\times\sD^\ctr(\C\contra\Rfr)}
\arrow{s,=}\arrow[4]{e,t}{\Coext_\C}\node[4]{H^0(\R\contra^\free)}
\arrow{s,=} \\
\node{\sD^\co(\C\comod\Rfr)^\sop\times\sD^\co(\C\comod\Rfr)}
\arrow[4]{e,t}{\Ext_\C}\node[4]{H^0(\R\contra^\free)}
\end{diagram}
$$ \par
\textup{(b)} The same equivalence of triangulated categories\/
$\boL\Phi_\C=\boR\Psi_\C^{-1}$ transforms the right derived functor\/
$$
 \Cotor^\C\:\sD^\co(\comodrRfr\C)\times\sD^\co(\C\comod\Rfr)
 \lrarrow H^0(\R\contra^\free)
$$
into the left derived functor
$$
 \Ctrtor^\C\:\sD^\co(\comodrRfr\C)\times\sD^\ctr(\C\contra\Rfr)
 \lrarrow H^0(\R\contra^\free).
$$
 In other words, the following diagram of categories, functors,
and equivalences is commutative:
$$
\begin{diagram}
\node{\sD^\co(\comodrRfr\C)\times\sD^\co(\C\comod\Rfr)}
\arrow{s,=}\arrow[4]{e,t}{\Cotor^\C}\node[4]{H^0(\R\contra^\free)}
\arrow{s,=} \\
\node{\sD^\co(\comodrRfr\C)\times\sD^\ctr(\C\contra\Rfr)}
\arrow[4]{e,t}{\Ctrtor^\C}\node[4]{H^0(\R\contra^\free)}
\end{diagram}
$$
\end{prop}

\begin{proof}
 For any projective $\R$\+free left CDG\+contramodule $\P$ and any
$\R$\+free left CDG\+contramodule $\Q$ over $\C$, there is
a natural isomorphism of complexes of free $\R$\+contramodules
$\Cohom_\C(\Phi_\C(\P),\Q)\simeq\Hom^\C(\P,\Q)$.
 For any $\R$\+free left CDG\+comodule $\L$ and any injective
$\R$\+free left CDG\+comodule $\M$ over $\C$, there is
a natural isomorphism of complexes of free $\R$\+contramodules
$\Cohom_\C(\L,\Psi_\C(\M))\simeq\Hom_\C(\L,\M)$.
 For any $\R$\+free right CDG\+comodule $\N$ and any projective
$\R$\+free left CDG\+contramodule $\P$ over $\C$, there is
a natural isomorphism of complexes of free $\R$\+contramodules
$\N\ocn_\C\P\simeq\N\oc_\C\Phi_\C(\P)$.
 (Cf.~\cite[Section~5.6]{Psemi}.) \emergencystretch=1em
\end{proof}

 Let $f=(f,a)\:\C\rarrow\D$ be a morphism of $\R$\+free
CDG\+coalgebras.
 Then any $\R$\+free graded contramodule (resp., comodule)
over $\C$ can be endowed with a graded $\D$\+contramodule
(resp., $\D$\+comodule) structure via~$f$, and any homogeneous
morphism (of any degree) between graded $\C$\+contramodules
(resp., $\C$\+comodules) can be also considered as a homogeneous
morphism (of the same degree) between graded $\D$\+contramodules
(resp., $\D$\+comodules).

 With any $\R$\+free left CDG\+contramodule $(\P,d_\P)$ over $\C$
one can associate an $\R$\+free left CDG\+contramodule $(\P,d'_\P)$
over $\D$ with the modified differential $d'_\P(p) = d_\P(p) + a*p$.
 Similarly, for any $\R$\+free left CDG\+comodule $(\M,d_\M)$ over
$\C$ the formula $d'_\M(x) = d_\M(x) + a*x$ defines a modified
differential on $\M$ making $(\M,d'_\M)$ an $\R$\+free left
CDG\+comodule over~$\D$.
 Finally, with any $\R$\+free right CDG\+comodule $(\N,d_\N)$ over
$\C$ one associate an $\R$\+free right CDG\+comodule $(\N,d'_\N)$
over $\D$ with the modified differential
$d'_\N(y)=d_\N(y)-(-1)^{|y|}y*a$.

 We have constructed the DG\+functors of contrarestriction of
scalars $R^f\:\C\contra\Rfr\allowbreak\rarrow\D\contra\Rfr$, and
corestriction of scalars $R_f\:\C\comod\Rfr\rarrow\D\comod\Rfr$ and
$\comodrRfr\C\rarrow\comodrRfr\D$.
 Passing to the homotopy categories, we obtain the triangulated
functors $R^f\:H^0(\C\contra\Rfr)\rarrow H^0(\D\contra\Rfr)$
and $R_f\:H^0(\C\comod\Rfr)\rarrow H^0(\D\comod\Rfr)$.
 Since the contra/corestriction of scalars clearly preserves
contra/coacyclicity, we have the induced functors on
the contra/coderived categories {\hfuzz=5.9pt
$$
 \boI R^f\:\sD^\ctr(\C\contra\Rfr)\lrarrow\sD^\ctr(\D\contra\Rfr)
$$
and }
$$
 \boI R_f\:\sD^\co(\C\comod\Rfr)\lrarrow\sD^\co(\C\comod\Rfr).
$$

 The triangulated functor $\boI R^f$ has a left adjoint.
 The DG\+functor $E^f\:\D\contra\Rfr_\proj\allowbreak\rarrow
\C\contra\Rfr_\proj$ is defined on the level of graded contramodules
by the rule $\Q\mpsto\Cohom_\D(\C,\Q)$
(see Lemma~\ref{r-free-co-op-reduction}(e)); the differential on
$\Cohom_\D(\C,\Q)$ induced by the differentials on $\C$ and $\Q$ is
modified to obtain the differential on $E^f(\Q)$ using
the change-of-connection linear function~$a$.
 Passing to the homotopy categories and taking into account
Theorem~\ref{r-free-co-derived-thm}(c), we obtain the left derived
functor of contraextension of scalars
$$
 \boL E^f\:\sD^\ctr(\D\contra\Rfr)\lrarrow\sD^\ctr(\C\contra\Rfr),
$$
which is left adjoint to the functor~$\boI R^f$.

 Similarly, the triangulated functor $\boI R_f$ has a right adjoint.
 The DG\+functor $E_f\:\D\comod\Rfr_\inj\rarrow\C\comod\Rfr_\inj$
is defined on the level of graded comodules by the rule
$\N\mpsto\C\oc_\D\N$ (see Lemma~\ref{r-free-co-op-reduction}(d));
the change-of-connection linear function~$a$ is used to modify
the differential on $\C\oc_\D\N$ induced by the differentials on
$\C$ and~$\N$.
 Passing to the homotopy categories and taking into account
Theorem~\ref{r-free-co-derived-thm}(d), we obtain the right derived
functor of coextension of scalars
$$
 \boR E_f\:\sD^\co(\D\comod\Rfr)\lrarrow\sD^\co(\C\comod\Rfr).
$$

\begin{prop} \label{r-free-co-extension}
 The equivalences of triangulated categories\/ $\boL\Phi_\C=
\boR\Psi_\C^{-1}$ and\/ $\boL\Phi_\D=\boR\Psi_\D^{-1}$ from
Corollary~\textup{\ref{r-free-derived-co-contra}} transform
the left derived functor\/ $\boL E^f$ into the right derived
functor\/ $\boR E_f$ and back.
 In other words, the following diagram of categories, functors,
and equivalences is commutative:
$$
\begin{diagram}
 \node{\llap{$\boL\Phi_\D$}\:\sD^\ctr(\D\contra\Rfr)}
 \arrow{e,=}\arrow{s,l}{\boL E^f}
 \node{\sD^\co(\D\comod\Rfr)\,\.\:\!\rlap{$\boR\Psi_\D$}}
 \arrow{s,r}{\boR E_f}\\
 \node{\llap{$\boL\Phi_\C$}\:\sD^\ctr(\C\contra\Rfr)}
 \arrow{e,=}
 \node{\sD^\co(\C\comod\Rfr)\,\.\:\!\rlap{$\boR\Psi_\C$}}
\end{diagram}
$$
\end{prop}

\begin{proof}
 See~\cite[Section~5.4]{Pkoszul} or \cite[Section~7.1.4
and Corollary~8.3.4]{Psemi}.
\end{proof}

\begin{thm}  \label{r-free-co-restriction}
\textup{(a)} Let $f=(f,a)\:\C\rarrow\D$ be a morphism of\/
$\R$\+free CDG\+co\-algebras.
 Then the functor\/ $\boI R^f\:\sD^\ctr(\C\contra\Rfr)\rarrow
\sD^\ctr(\D\contra\Rfr)$ is an equivalence of triangulated
categories if and only if the functor\/ $\boI R_f\:
\sD^\co(\C\comod\Rfr)\allowbreak\rarrow\sD^\co(\D\comod\Rfr)$ is
such an equivalence. \par
\textup{(b)} A morphism~$f=(f,a)$ of\/ $\R$\+free CDG\+coalgebras
has the above two equivalent properties~\textup{(a)} provided that
the morphism $f/\m f=(f/\m f\;a/\m a)\:\C/\m\C\rarrow\D/\m\D$
of CDG\+coalgebras over~$k$ has the similar properties.
\end{thm} 

\begin{proof}
 To prove part~(a), it suffices to notice that, according to
Proposition~\ref{r-free-co-extension}, the two triangulated
functors in question are the adjoints on different sides to
the same functor $\boL E^f=\boR E_f$.
 To prove part~(b), consider the adjunction morphisms
for the functors $\boL E^f$ and $\boI R^f$, or alternatively,
for the functors $\boI R_f$ and $\boR E_f$.
 The functors of reduction modulo $\m$ transform these into
the adjunction morphisms for the similar functors related
to the morphism of CDG\+algebras~$f/\m f$.
 By Corollary~\ref{r-free-co-acycl-reduction}, if the latter
are isomorphisms, then so are the former.
\end{proof}

\subsection{$\R$-cofree graded contra/comodules}
\label{r-cofree-graded-co}
 Let $\C$ be an $\R$\+free graded coalgebra.
 An \emph{$\R$\+cofree graded left comodule} $\cM$ over $\C$ is
a graded left $\C$\+comodule in the module category of cofree
$\R$\+comodules over the tensor category of free $\R$\+contramodules.
 In other words, it is a cofree graded $\R$\+comodule endowed with
a (coassociative and counital) homogeneous $\C$\+coaction map
$\cM\rarrow\C\ocn_\R\cM$.
 \ \ \emph{$\R$\+free graded right comodules} $\cN$ over $\C$ are
defined in the similar way.

 An \emph{$\R$\+cofree graded left contramodule} $\cP$ over $\C$
is a cofree graded $\R$\+comodule endowed with
a $\C$\+contraaction map $\Ctrhom_\R(\C,\cP)\rarrow\cP$, which
must be a morphism of graded $\R$\+comodules satisfying
the conventional contraassociativity and counit axioms.
 This can be rephrased by saying that an $\R$\+cofree
$\C$\+contramodule is an object of the opposite category to
the category of graded $\C$\+comodules in the module category
$(\R\comod^\cofr)^\sop$ over the tensor category
$\R\contra^\free$ (see Section~\ref{contra-operations} and
\cite[Section~0.2.4 or~3.1.1]{Psemi}).

 Specifically, the contraassociativity axiom asserts that the two
compositions $\Ctrhom_\R(\C,\Ctrhom_\R(\C,\cP))\simeq
\Ctrhom_\R(\C\ot^\R\C\;\cP)\birarrow\Ctrhom_\R(\C,\cP)\rarrow\cP$
of the maps induced by the contraaction and comultiplication maps
with the contraaction map must be equal to each other.
 The counit axiom claims that the composition $\cP\rarrow
\Ctrhom_\R(\C,\cP)\rarrow\cP$ of the map induced by the counit
map with the contraaction map must be equal to the identity
endomorphism of~$\cP$.
 For the sign rule, see~\cite[Section~2.2]{Pkoszul}.

 The categories of $\R$\+cofree graded $\C$\+comodules and
$\C$\+contramodules are enriched over the tensor category of
(graded) $\R$\+contramodules.
 So the abelian group of morphisms between two $\R$\+cofree
graded left $\C$\+comodules $\cL$ and $\cM$ is the underlying
abelian group of the degree-zero component of the (not
necessarily free) graded $\R$\+contramodule $\Hom_\C(\cL,\cM)$
constructed as the kernel of the pair of morphisms of graded
$\R$\+contramodules $\Hom_\R(\cL,\cM)\birarrow\Hom_\R(\cL\;
\C\ocn_\R\cM)$ induced by the coactions of $\C$ in $\cL$ and~$\cM$.
 For the sign rules, which are different for the left and
the right comodules, see~\cite[Section~2.1]{Pkoszul}.
 Similarly, the abelian group of morphisms between two
$\R$\+cofree graded left $\C$\+contramodules $\cP$ and $\cQ$ is
the underlying abelian group of the zero-degree component of
the graded $\R$\+contramodule $\Hom^\C(\cP,\cQ)$ constructed as
the kernel of the pair of morphisms of graded $\R$\+contramodules
$\Hom_\R(\cP,\cQ)\birarrow\Hom_\R(\Ctrhom_\R(\C,\cP),\cQ)$
induced by the contraactions of $\C$ in $\cP$ and~$\cQ$.

 The graded $\R$\+comodule $\Hom_\C(\L,\cM)$ from an $\R$\+free
graded left $\C$\+comodule $\L$ to an $\R$\+cofree graded left
$\C$\+comodule $\cM$ is defined as the kernel of the pair of
morphisms of graded $\R$\+comodules $\Ctrhom_\R(\L,\cM)\birarrow
\Ctrhom_\R(\L\;\C\ocn_\R\cM)$, one of which is induced by
the coaction morphism $\cM\rarrow\C\ocn_\R\cM$, while the other
is constructed in terms of the coaction morphism $\L\rarrow
\C\ot^\R\L$ in the following way.
 A given element of $\Ctrhom_\R(\L,\cM)$ is composed with
the adjunction morphism $\cM\rarrow\Ctrhom_\R(\C\;\C\ocn_\R\cM)$,
resulting in an element of $\Ctrhom_\R(\L,\Ctrhom_\R(\C\;
\C\ocn_\R\cM))\simeq\Ctrhom_\R(\C\ot^\R\L\;\C\ocn_\R\cM)$.
 The latter element is composed with the morphism of $\C$\+coaction
in~$\L$.
 Similarly, the graded $\R$\+comodule $\Hom^\C(\P,\cQ)$ from
an $\R$\+free graded left $\C$\+contramodule $\P$ to an $\R$\+cofree
graded left $\C$\+contramodule $\cQ$ is defined as the kernel of
the pair of morphisms of graded $\R$\+comodules
$\Ctrhom_\R(\P,\cQ)\birarrow\Ctrhom_\R(\Hom^\R(\C,\P),\cQ)$,
one of which is induced by the contraaction morphism
$\Hom^\R(\C,\P)\rarrow\P$, while the other is constructed in
terms of the contraaction morphism $\Ctrhom_\R(\C,\cQ)\rarrow\cQ$
in the following way.
 A given element of $\Ctrhom_\R(\P,\cQ)$ is composed with
the evaluation morphism $\C\ot^\R\Hom^\R(\C,\P)\rarrow\P$,
resulting in an element of $\Ctrhom_\R(\C\ot^\R\Hom^\R(\C,\P)\;
\cQ)\simeq\Ctrhom_\R(\Hom^\R(\C,\P)\;\Ctrhom_\R(\C,\cQ))$.
 The latter element is composed with the morphism of
$\C$\+contraaction in~$\cQ$.

 The \emph{contratensor product} $\N\ocn_\C\cP$ of an $\R$\+free
right $\C$\+comodule $\N$ and an $\R$\+cofree left $\C$\+contramodule
$\cP$ is the (not necessarily cofree) graded $\R$\+comodule
constructed as the cokernel of the pair of morphisms of graded
$\R$\+comodules $\N\ocn_\R\Ctrhom_\R(\C,\cP)\birarrow
\N\ocn_\R\cP$ defined in terms of the $\C$\+coaction in $\N$,
the $\C$\+contraaction in $\P$, and the evaluation map
$\C\ocn_\R\Ctrhom_\R(\C,\cP)\rarrow\cP$.
 The latter can be obtained from the natural isomorphism
$\Hom_\R(\C\ocn_\R\Ctrhom_\R(\C,\cP)\;\cP)\simeq
\Hom_\R(\Ctrhom_\R(\C,\cP),\Ctrhom_\R(\C,\cP))$
(see Section~\ref{contra-operations}).
 For further details, see~\cite[Section~0.2.6]{Psemi};
for the sign rule, see~\cite[Section~2.2]{Pkoszul}.

 The \emph{cotensor product} $\N\oc_\C\cM$ of an $\R$\+free graded
right $\C$\+comodule $\N$ and an $\R$\+cofree graded left
$\C$\+comodule $\cM$ is a graded $\R$\+comodule constructed as
the kernel of the pair of morphisms of graded $\R$\+comodules
$\N\ocn_\R\cM\birarrow\N\ot^\R\C\ocn_\R\cM$.
 The graded $\R$\+comodule of \emph{cohomomorphisms}
$\Cohom_\C(\M,\cP)$ from an $\R$\+free graded left $\C$\+comodule
$\M$ to an $\R$\+cofree graded left $\C$\+contramodule $\cP$ is
constructed as the cokernel of the pair of morphisms
$\Ctrhom_\R(\C\ot^\R\M\;\cP)\simeq\Ctrhom_\R(\M,\Ctrhom_\R(\C,\cP))
\birarrow\Ctrhom_\C(\M,\cP)$ induced by the $\C$\+coaction in
$\M$ and the $\C$\+contraaction in~$\cP$.
 The graded $\R$\+contramodule of \emph{cohomomorphisms}
$\Cohom_\C(\cM,\cP)$ from an $\R$\+cofree graded left $\C$\+comodule
$\cM$ to an $\R$\+cofree graded left $\C$\+contramodule $\cP$ is
constructed as the cokernel of the pair of morphisms
$\Hom_\R(\C\ocn_\R\cM\;\cP)\simeq\Hom_\R(\cM,\Ctrhom_\R(\C,\cP))
\rarrow\Hom_\R(\cM,\cP)$ induced by the $\C$\+coaction in $\cM$
and the $\C$\+contraaction in~$\cP$.

 For any cofree graded $\R$\+comodule $\cU$ and any $\R$\+free
graded right $\C$\+comodule $\N$, the cofree graded $\R$\+comodule
$\Ctrhom_\R(\N,\cU)$ has a natural left $\C$\+contramodule
structure provided by the map $\Ctrhom_\R(\C,\Ctrhom_\R(\N,\cU))
\simeq\Ctrhom_\C(\N\ot^\R\C\;\cU)\allowbreak\rarrow
\Ctrhom_\C(\N,\cU)$ induced by the $\C$\+coaction in~$\N$.
 For any $\R$\+free graded right $\C$\+comodule $\N$ and $\R$\+free
graded left $\C$\+contramodule $\P$ there is a natural isomorphism
of graded $\R$\+comodules $\Ctrhom_\R(\N\ocn_\C\P\;\cU)\simeq
\Hom^\C(\P,\Ctrhom_\R(\N,\cU))$.
 For any $\R$\+free graded right $\C$\+comodule $\N$ and $\R$\+cofree
graded left $\C$\+contramodule $\cP$ there is a natural isomorphism of
graded $\R$\+contramodules $\Hom_\R(\N\ocn_\C\cP\;\cU)\simeq
\Hom^\C(\cP,\Ctrhom_\R(\N,\cU))$.
 For any $\R$\+free graded right $\C$\+comodule $\N$ and $\R$\+cofree
graded left $\C$\+comodule $\cM$ there is a natural isomorphism of
graded $\R$\+contramodules $\Hom_\R(\N\oc_\C\cM\;\cU)\simeq
\Cohom^\C(\cM,\Ctrhom_\R(\N,\cU))$.
 
 The additive categories of $\R$\+cofree graded $\C$\+contramodules
and $\C$\+comodules have natural exact category structures: a short
sequence of $\R$\+cofree graded $\C$\+contra/comodules is said to be
exact is it is (split) exact as a short sequence of cofree graded
$\R$\+comodules.
 The exact category of $\R$\+cofree graded $\C$\+comodules admits
infinite direct sums, while the exact category of $\R$\+cofree
graded $\C$\+contramodules admits infinite products.
 Both operations are preserved by the forgetful functors to
the category of cofree graded $\R$\+comodules and preserve exact
sequences.

 If a morphism of $\R$\+cofree graded $\C$\+contra/comodules has
$\R$\+cofree kernel and cokernel in the category of graded
$\R$\+comodules, then these kernel and cokernel are naturally
endowed with $\R$\+cofree graded $\C$\+contra/comodule structures,
making them the kernel and cokernel of the same morphism in
the additive category of $\R$\+cofree graded $\C$\+contra/comodules.

 There are enough projective objects in the exact category of
$\R$\+cofree graded left $\C$\+contramodules; these are the direct
summands of the graded $\C$\+contramodules $\Ctrhom_\R(\C,\cU)$ 
freely generated by cofree graded $\R$\+comodules~$\cU$.
 Similarly, there are enough injective objects in the exact category
of $\R$\+cofree graded left $\C$\+comodules; these are the direct
summands of the graded $\C$\+comodules $\C\ocn_\R\cV$ cofreely
cogenerated by cofree graded $\R$\+comodules~$\cV$.

 For any $\R$\+cofree graded left $\C$\+comodule $\cL$ and any
cofree graded $\R$\+comodule $\cV$ there is a natural isomorphism
of graded $\R$\+contramodules $\Hom_\C(\cL\;\C\ocn_\R\cV)\simeq
\Hom_\R(\cL,\cV)$.
 For any $\R$\+free graded left $\C$\+comodule $\L$ and any
cofree graded $\R$\+co\-module $\cV$ there is a natural isomorphism
of graded $\R$\+comodules $\Hom_\C(\L\;\C\ocn_\R\cV)\allowbreak\simeq
\Ctrhom_\R(\L,\cV)$.
 For any $\R$\+cofree graded left $\C$\+contramodule $\cQ$ and
any cofree graded $\R$\+comodule $\cU$, there is a natural isomorphism
of graded $\R$\+contramodules $\Hom^\C(\Ctrhom_\R(\C,\cU),\cQ)
\simeq\Hom_\R(\cU,\cQ)$.
 For any $\R$\+cofree graded left $\C$\+contramod\-ule $\cQ$ and
any free graded $\R$\+contramodule $\U$, there is a natural
isomorphism of graded $\R$\+comodules $\Hom^\C(\Hom^\R(\C,\U),\cQ))
\simeq\Ctrhom_\R(\U,\cQ)$ \cite[Lemma~1.1.2]{Psemi}.
{\emergencystretch=0em\par}

 For any $\R$\+free graded right $\C$\+comodule $\N$ and any cofree
graded $\R$\+comodule $\cU$, there is a natural isomorphism of
graded $\R$\+comodules $\N\ocn_\C\Ctrhom_\R(\C,\cU)\simeq
\N\ocn_\R\cU$ \cite[Section~5.1.1]{Psemi}.
 For any $\R$\+free graded right $\C$\+comodule $\N$ and any cofree
graded $\R$\+comodule $\cV$, there is a natural isomorphism of
graded $\R$\+comodules $\N\oc_\C(\C\ocn_\R\cV)\simeq\N\ocn_\R\cV$.
 For any $\R$\+cofree graded left $\C$\+contramodule $\cP$ and any
free graded $\R$\+contramodule $\V$, there is a natural isomorphism
of graded $\R$\+comodules $\Cohom_\C(\C\ot^\R\V\;\cP)\simeq
\Ctrhom_\R(\V,\cP)$.
 For any $\R$\+cofree graded left $\C$\+contramodule $\cP$ and any
cofree graded $\R$\+comodule $\cV$, there is a natural isomorphism
of graded $\R$\+contramodules $\Cohom_\C(\C\ocn_\R\cV\;\cP)\simeq
\Hom_\R(\cV,\cP)$.
 For any $\R$\+free graded left $\C$\+comodule $\M$ and any cofree
graded $\R$\+comodule $\cU$, there is a natural isomorphism of
graded $\R$\+comodules $\Cohom_\C(\M,\Ctrhom_\R(\C,\cU))\simeq
\Ctrhom_\R(\M,\cU)$.
 For any $\R$\+cofree graded left $\C$\+comodule $\cM$ and any cofree
graded $\R$\+comodule $\cU$, there is a natural isomorphism of
graded $\R$\+contramodules $\Cohom_\C(\cM,\Ctrhom_\R(\C,\cU))\simeq
\Hom_\R(\cM,\cU)$ \cite[Sections~1.2.1 and~3.2.1]{Psemi}.

 All the results of Section~\ref{r-free-graded-co} have their
analogues for $\R$\+cofree graded $\C$\+contra\-modules and
$\C$\+comodules, which can be, at one's choice, either proven
directly in the similar way to the proofs in~\ref{r-free-graded-co},
or deduced from the results in~\ref{r-free-graded-co} using
Proposition~\ref{r-cofree-co-r-co-contra} below.
 In particular, we have the following constructions.

 Given a projective $\R$\+cofree graded left $\C$\+contramodule
$\cP$, the $\R$\+cofree graded left $\C$\+comodule $\Phi_\C(\cP)$
is defined as $\Phi_\C(\cP)=\C\ocn_\C\cP$.
 Here the contratensor product is endowed with a left $\C$\+comodule
structure as an $\R$\+cofree cokernel of a morphism of $\R$\+cofree
$\C$\+comodules.
 Given an injective $\R$\+cofree graded left $\C$\+comodule $\cM$,
the $\R$\+cofree graded left $\C$\+contramodule $\Psi_\C(\cM)$
is defined as $\Psi_\C(\cM)=\Hom_\C(\C,\cM)$.
 Here the $\Hom$ of $\C$\+comodules is endowed with a left
$\C$\+contramodule structure as an $\R$\+cofree kernel of
a morphism of $\R$\+cofree left $\C$\+contramodules.

\begin{prop} \label{r-cofree-c-co-contra}
 The functors\/ $\Phi_\C$ and\/ $\Psi_\C$ are mutually inverse
equivalences between the additive categories of projective\/
$\R$\+cofree graded left\/ $\C$\+contramodules and injective\/
$\R$\+cofree graded left\/ $\C$\+comodules.
\end{prop}

\begin{proof}
 One has $\Phi_\C(\Ctrhom_\R(\C,\cU))\simeq\C\ocn_\R\cU$ and
$\Psi_\C(\C\ocn_\R\cV)\simeq\Ctrhom_\R(\C,\cV)$.
\end{proof}

\begin{prop} \label{r-cofree-co-r-co-contra}
\textup{(a)} The exact categories of\/ $\R$\+free graded left\/
$\C$\+contramodules and of\/ $\R$\+cofree graded left\/
$\C$\+contramodules are naturally equivalent.
 The equivalence is provided by the functors\/ $\Phi_\R$ and\/
$\Psi_\R$ of comodule-contramodule correspondence over\/ $\R$,
defined in Section\/~\textup{\ref{hom-operations}}, which
transform free graded\/ $\R$\+contramodules with a\/ $\C$\+contramodule
structure into cofree graded\/ $\R$\+comodules with
a\/ $\C$\+contramodule structure and back. \par
\textup{(b)} The exact categories of\/ $\R$\+free graded left\/
$\C$\+comodules and of\/ $\R$\+cofree graded left\/ $\C$\+comodules
are naturally equivalent.
 The equivalence is provided by the functors\/ $\Phi_\R$ and\/
$\Psi_\R$, which transform free graded\/ $\R$\+contramodules with
a\/ $\C$\+comodule structure into cofree graded\/ $\R$\+comodules
with a\/ $\C$\+comodule structure and back.
\end{prop}

\begin{proof}
 See the proof of Proposition~\ref{r-cofree-r-co-contra}.
\end{proof}

 The equivalence $\Phi_\R=\Psi_\R^{-1}$ between the categories of
$\R$\+free and $\R$\+cofree graded $\C$\+contra/comodules is also
an equivalence of categories enriched over graded
$\R$\+contramodules.
 Besides, for any $\R$\+free graded left $\C$\+comodule $\L$ and
$\R$\+cofree graded left $\C$\+comodule $\cM$ there is a natural
isomorphism of graded $\R$\+contramodules
$\Hom_\C(\L,\Psi_\R(\cM))\simeq\Psi_\R(\Hom_\C(\L,\cM))$.
 For any $\R$\+free graded left $\C$\+contramodule $\P$ and
$\R$\+cofree graded left $\C$\+contramodule $\cQ$ there is a natural
isomorphism of graded $\R$\+contramodules
$\Hom^\C(\P,\Psi_\R(\cQ))\simeq\Psi_\R(\Hom^\C(\P,\cQ))$.
 For any $\R$\+free graded right $\C$\+comodule $\N$ and
$\R$\+free graded left $\C$\+contramodule $\P$ there is a natural
isomorphism of graded $\R$\+comodules $\N\ocn_\C\Phi_\R(\P)
\simeq\Phi_\R(\N\ocn_\C\P)$.

 For any $\R$\+free graded right $\C$\+comodule $\N$ and
$\R$\+cofree graded left $\C$\+comodule $\cM$ there is a natural
isomorphism of graded $\R$\+contramodules
$\N\oc_\C\Psi_\R(\cM)\simeq\Psi_\R(\N\oc_\C\cM)$.
 For any $\R$\+free graded left $\C$\+comodule $\M$ and 
$\R$\+free graded left $\C$\+contramodule $\P$ there is a natural
isomorphism of graded $\R$\+contramodules
$\Cohom_\C(\Phi_\R(\M),\Phi_\R(\P))\simeq\Cohom_\C(\M,\P)$
and a natural isomorphism of graded $\R$\+comodules
$\Cohom_\C(\M,\Phi_\R(\P))\simeq\Phi_\R(\Cohom_\C(\M,\P))$.

 Furthermore, the equivalences $\Phi_\R=\Psi_\R^{-1}$ transform
graded $\C$\+contramodules $\Hom_\R(\C,\U)$ freely generated by
free graded $\R$\+contramodules $\U$ into graded $\C$\+contramodules
$\Ctrhom_\R(\C,\cU)$ freely generated by cofree graded
$\R$\+comodules $\cU$, with $\Psi_\R(\cU)=\U$, and graded
$\C$\+comodules $\C\ot^\R\V$ cofreely cogenerated by free graded
$\R$\+contramodules $\V$ into graded $\C$\+comodules $\C\ocn_\R\cV$
cofreely cogenerated by cofree graded $\R$\+comodules $\cV$,
with $\cV=\Phi_\R(\V)$.
{\hbadness=1350\par}

 In particular, the equivalences of additive categories
$\Phi_\C=\Psi_\C^{-1}$ from Propositions~\ref{r-free-c-co-contra}
and~\ref{r-cofree-c-co-contra} form a commutative diagram with
(the appropriate restriction of) the equivalence of exact
categories $\Phi_\R=\Psi_\R^{-1}$ from
Proposition~\ref{r-cofree-co-r-co-contra}.

 Finally, the equivalences $\Phi_\R=\Psi_\R^{-1}$ between the exact
categories of $\R$\+free and $\R$\+cofree graded
$\C$\+contra/comodules transform the reduction functors
$\P\mpsto\P/\m\P$ and $\M\mpsto\M/\m\M$ into the reduction functors
$\cP\mpsto{}_\m\cP$ and $\cM\mpsto{}_\m\cM$, respectively,
the functors taking values in the abelian category of
$\C/\m\C$\+contra/comodules.
 This follows from the results of
Sections~\ref{hom-operations}\+-\ref{contra-operations}.

\subsection{Contra/coderived category of $\R$-cofree
CDG-contra/comodules}  \label{r-cofree-co-derived}
 Let $\C$ be an $\R$\+free graded coalgebra.
 Suppose that it is endowed with an odd coderivation~$d$ of degree~$1$.
 An \emph{odd coderivation} $d_\cM$ of an $\R$\+cofree graded
left $\C$\+comodule $\cM$ \emph{compatible with} the coderivation~$d$
on $\C$ is a differential ($\R$\+comodule endomorphism of degree~$1$)
such that the coaction map $\cM\rarrow\C\ocn_\R\cM$ forms a commutative
diagram with~$d_\cM$ and the differential on $\C\ocn_\R\cM$ induced
by $d$ and~$d_\cM$ (see Section~\ref{r-cofree-absolute}).
 Odd coderivations of $\R$\+cofree graded right $\C$\+comodules are
defined similarly.

 An \emph{odd contraderivation} $d_\cP$ of an $\R$\+cofree graded
left $\C$\+contramodule $\cP$ \emph{compatible with}
the coderivation~$d$ on $\C$ is a differential ($\R$\+comodule
endomorphism of degree~$1$) such that the contraaction map
$\Ctrhom_\R(\C,\cP)\rarrow\cP$ forms a commutative diagram
with~$d_\cP$ and the differential on $\Ctrhom_\R(\C,\cP)$ induced
by $d$ and~$d_\cP$.

 An $\R$\+cofree graded comodule $\cM$ over $\C$ is endowed with
the structure of a graded left module over the $\R$\+free graded
algebra $\C^*$ (see Section~\ref{r-free-co-derived}) provided
by the action map $\C^*\ocn_\R\cM\rarrow\C^*\ocn_\R(\C^*\ocn_\R\cM)
\simeq(\C^*\ot^\R\C)\ocn_\R\cM\rarrow\cM$. 
 \ $\R$\+cofree graded right comodules over $\C$ are endowed with
graded right $\C^*$\+module structures in the similar way.
 An $\R$\+cofree graded left contramodule $\cP$ over $\C$ is endowed
with a graded left $\C^*$\+module structure provided by the action
map $\C^*\ocn_\R\cP\rarrow\Ctrhom_\R(\C,\cP)\rarrow\cP$.
 For the sign rules, see~\cite[Section~4.1]{Pkoszul}.

 An \emph{$\R$\+cofree left CDG\+comodule} $\cM$ over $\C$ is
an $\R$\+cofree graded left $\C$\+comodule endowed with
an odd coderivation $\d_\cM\:\cM\rarrow\cM$ of degree~$1$
compatible with the coderivation~$d$ on $\C$ and satisfying
the equation $d^2_\cM(x)=h*x$ for all $x\in\cM$.
 The definition of an \emph{$\R$\+cofree right CDG\+comodule}
$\cN$ over $\C$ is similar; the only difference is that
the equation for the square of the differential has the form
$d^2_\cN(y)=-y*h$ for all $y\in\cN$.
 An \emph{$\R$\+cofree left CDG\+contramodule} $\cP$ over $\C$ is
an $\R$\+cofree graded left $\C$\+contramodule endowed with
an odd contraderivation $d_\cP\:\cP\rarrow\cP$ of degree~$1$
compatible with the coderivation~$d$ on $\C$ and satisfying
the equation $d^2_\cP(p)=h*p$ for all $p\in\cP$.
 Here the element $h\in\C^*{}^2$ is viewed as a homogeneous
morphism $\R\rarrow\C^*$ of degree~$2$ and the notation
$x\mpsto h*x$ is interpreted as the composition $\cM\simeq
\R\ocn_\R\cM\rarrow\C^*\ocn_\R\cM\rarrow\cM$ defining a homogeneous
endomorphism of degree~$2$ of the graded $\R$\+comodule $\cM$;
and similarly for $\cN$ and~$\cP$.

 $\R$\+cofree left (resp., right) CDG\+comodules over $\C$ naturally
form a DG\+category, which we will denote by $\C\comod\Rcof$
(resp., $\comodrRcof\C$).
 The similar DG\+cat\-egory of $\R$\+cofree left CDG\+contramodules
over $\C$ will be denoted by $\C\contra\Rcof$.
 In fact, these DG\+categories are enriched over the tensor category
of (complexes of) $\R$\+contramodules, so the complexes of morphisms
in $\C\comod\Rcof$, \ $\comodrRcof\C$, and $\C\contra\Rcof$ are
the underlying complexes of abelian groups for naturally defined
complexes of $\R$\+contramodules.
 The underlying graded $\R$\+contramodules of these complexes were
defined in Section~\ref{r-cofree-graded-co}, and the differentials
in them are defined in the conventional way.

 Passing to the zero cohomology of the complexes of morphisms,
we construct the homotopy categories of $\R$\+cofree
CDG\+comodules and CDG\+contramodules $H^0(\C\comod\Rcof)$, \
$H^0(\comodrRcof\C)$, and $H^0(\C\contra\Rcof)$.
 These we will mostly view as conventional categories with
abelian groups of morphisms (see Remark~\ref{contra-enriched}).
 Since the DG\+categories $\C\comod\Rcof$, \ $\comodrRcof\C$,
and $\C\contra\Rcof$ have shifts and twists, their homotopy
categories are triangulated.
 The DG- and homotopy categories of $\R$\+cofree CDG\+comodules
also have infinite direct sums, while the DG- and homotopy categories
of $\R$\+cofree CDG\+contramodules have infinite products.

 The complex $\Hom_\C(\L,\cM)$ from an $\R$\+free left
CDG\+comodule $\L$ to an $\R$\+cofree left CDG\+comodule $\cM$
over $\C$ is a complex of $\R$\+comodules obtained by endowing
the graded $\R$\+comodule $\Hom_\C(\L,\cM)$ constructed in
Section~\ref{r-cofree-graded-co} with the conventional $\Hom$
differential.
 The $\Hom$ between $\R$\+free and $\R$\+cofree CDG\+comodules
over $\C$ is a triangulated functor of two arguments
$$
 \Hom_\C\: H^0(\C\comod\Rfr)^\sop\times H^0(\C\comod\Rcof)
 \lrarrow H^0(\R\comod).
$$
 Similarly, the complex $\Hom^\C(\P,\cQ)$ from an $\R$\+free left
CDG\+contramodule $\P$ to an $\R$\+cofree left CDG\+contramodule
$\cQ$ over $\C$ is a complex of $\R$\+comodules obtained by
endowing the graded $\R$\+comodule $\Hom^\C(\P,\cQ)$ with
the conventional $\Hom$ differential.
 The $\Hom$ between $\R$\+free and $\R$\+cofree CDG\+contramodules
over $\C$ is a triangulated functor of two arguments
$$
 \Hom^\C\: H^0(\C\contra\Rfr)^\sop\times H^0(\C\contra\Rcof)
 \lrarrow H^0(\R\comod).
$$
 The \emph{contratensor product} $\N\ocn_\C\cP$ of an $\R$\+free
right CDG\+comodule $\N$ and an $\R$\+cofree left CDG\+contramodule
$\cP$ over $\C$ is a complex of $\R$\+comodules obtained by
endowing the graded $\R$\+comodule $\N\ocn_\C\cP$ with
the conventional tensor product differential.
 The contratensor product of $\R$\+free right CDG\+comodules and
$\R$\+cofree left CDG\+contramodules over $\C$ is a triangulated
functor of two arguments
$$
 \ocn_\C\: H^0(\comodrRfr\C)\times H^0(\C\contra\Rcof)\lrarrow
 H^0(\R\comod).
$$

 The cotensor product $\N\oc_\C\cM$ of an $\R$\+free right
CDG\+comodule $\N$ and an $\R$\+cofree left CDG\+comodule $\cM$
over $\C$ is a complex of $\R$\+comodules obtained by endowing
the graded $\R$\+comodule $\N\oc_\C\cM$ defined in
Section~\ref{r-cofree-graded-co} with the conventional tensor
product differential.
 The cotensor product of $\R$\+free right CDG\+comodules and
$\R$\+cofree left CDG\+comodules over $\C$ is a triangulated functor
of two arguments
$$
 \oc_\C\: H^0(\comodrRfr\C)\times H^0(\C\comod\Rcof)\lrarrow
 H^0(\R\comod).
$$
 The complex $\Cohom_\C(\M,\cP)$ from an $\R$\+free left
CDG\+comodule $\M$ to an $\R$\+cofree left CDG\+contramodule
$\cP$ over $\C$ is a complex of $\R$\+comodules obtained by
endowing the graded $\R$\+comodule $\Cohom_\C(\M,\cP)$ with
the conventional $\Hom$ differential.
 The $\Cohom$ from $\R$\+free CDG\+comodules to $\R$\+cofree
CDG\+contramodules over $\C$ is a triangulated functor of
two arguments
$$
 \Cohom_\C\: H^0(\C\comod\Rfr)^\sop\times H^0(\C\contra\Rcof)
 \lrarrow H^0(\R\comod).
$$
 The complex $\Cohom_\C(\cM,\cP)$ from an $\R$\+cofree left
CDG\+comodule $\cM$ to an $\R$\+cofree left CDG\+contramodule
$\cP$ over $\C$ is a complex of $\R$\+contramodules obtained
by endowing the graded $\R$\+contramodule $\Cohom_\C(\cM,\cP)$
with the conventional $\Hom$ differential.
 The $\Cohom$ from $\R$\+cofree CDG\+comodules to $\R$\+cofree
CDG\+contramodules over $\C$ is a triangulated functor of
two arguments
$$
 \Cohom_\C\: H^0(\C\comod\Rcof)^\sop\times H^0(\C\comod\Rcof)
 \lrarrow H^0(\R\contra).
$$

 An $\R$\+cofree left CDG\+comodule over $\C$ is called
\emph{coacyclic} if it belongs to the minimal triangulated 
subcategory of the homotopy category $H^0(\C\comod\Rcof)$
containing the totalizations of short exact sequences of
$\R$\+cofree CDG\+comodules over $\C$ and closed under infinite
direct sums.
 The quotient category of $H^0(\C\comod\Rcof)$ by the thick
subcategory of coacyclic $\R$\+cofree CDG\+comodules is called
the \emph{coderived category} of $\R$\+cofree left CDG\+comodules
over $\C$ and denoted by $\sD^\co(\C\comod\Rcof)$.
 The coderived category of $\R$\+cofree right CDG\+comodules over
$\C$, denoted by $\sD^\co(\comodrRcof\C)$, is defined similarly.

 An $\R$\+cofree left CDG\+contramodule over $\C$ is called
\emph{contraacyclic} if it belongs to the minimal triangulated
subcategory of the homotopy category $H^0(\C\contra\Rcof)$
containing the totalizations of short exact sequences of
$\R$\+cofree CDG\+contramodules over $\C$ and closed under
infinite products.
 The quotient category of $H^0(\C\contra\Rcof)$ by the thick
subcategory of contraacyclic $\R$\+cofree CDG\+contramodules
is called the \emph{contraderived category} of $\R$\+cofree left
CDG\+contramodules over $\C$ and denoted by
$\sD^\ctr(\C\contra\Rcof)$.

 Denote by $\C\contra\Rcof_\proj\subset\C\contra\Rcof$ the full
DG\+subcategory formed by all the $\R$\+cofree CDG\+contramodules
over $\C$ whose underlying $\R$\+cofree graded $\C$\+contramodules
are projective.
 Similarly, denote by $\C\comod\Rcof_\inj\subset\C\comod\Rcof$
the full DG\+subcategory formed by all the $\R$\+cofree
CDG\+comodules whose underlying $\R$\+cofree graded $\C$\+comodules
are injective.
 The corresponding homotopy categories are denoted by
$H^0(\C\contra\Rcof_\proj)$ and $H^0(\C\comod\Rcof_\inj)$.

 The functors $\Phi_\R=\Psi_\R^{-1}$ from
Section~\ref{r-cofree-graded-co} define equivalences of
DG\+categories $\C\comod\Rfr\simeq\C\comod\Rcof$ (and similarly
for right CDG\+comodules), and
$\C\contra\Rfr\simeq\C\contra\Rcof$.
 Being also equivalences of exact categories, these correspondences
idenfity coacyclic $\R$\+free CDG\+comodules with coacyclic
$\R$\+cofree CDG\+comodules and contraacyclic $\R$\+free
CDG\+contramodules with contraacyclic $\R$\+cofree CDG\+contramodules.
 So an equivalence of the coderived categories {\hbadness=1200
$$
 \Phi_\R\:\sD^\co(\C\comod\Rfr)\simeq
 \sD^\co(\C\comod\Rcof)\,\,\:\!\Psi_\R
$$
and} an equivalence of the contraderived categories
$$
 \Phi_\R\:\sD^\ctr(\C\contra\Rfr)\simeq
 \sD^\ctr(\C\contra\Rcof)\,\,\:\!\Psi_\R
$$
are induced.
 The above equivalences of DG\+categories also identify
$\C\contra\Rfr_\proj$ with $\C\contra\Rcof_\proj$ and
$\C\comod\Rfr_\inj$ with $\C\comod\Rcof_\inj$.
 The comparison isomorphisms of graded $\R$\+comodules and
$\R$\+contramodules from Section~\ref{r-cofree-graded-co}
describing the compatibilities of the functors $\Phi_\R$ and
$\Psi_\R$ with the operations $\Hom_\C$, $\Hom^\C$, $\ocn_\C$,
$\oc_\C$, $\Cohom_\C$ on $\R$\+(co)free graded $\C$\+comodules and
$\C$\+contramodules all become isomorphisms of complexes of
$\R$\+comodules and $\R$\+contramodules when applied to
$\R$\+(co)free CDG\+comodules and CDG\+contramodules over~$\C$.

 All the results of Section~\ref{r-free-co-derived} have their
analogues for $\R$\+cofree CDG\+comodules and CDG\+contramodules,
which can be, at one's choice, either proven directly in
the similar way to the proofs in~\ref{r-free-co-derived},
or deduced from the results in~\ref{r-free-co-derived} using
the above observations about the functors $\Phi_\R=\Psi_\R^{-1}$.
 In particular, we have the following results.

\begin{thm} \label{r-cofree-co-derived-thm}
 Let\/ $\C$ be an\/ $\R$\+free CDG\+coalgebra.  Then \par
\textup{(a)} for any CDG\+contramodule\/
$\cP\in H^0(\C\contra\Rcof_\proj)$ and any contraacyclic\/ $\R$\+cofree
left CDG\+contramodule\/ $\cQ$ over\/ $\C$, the complex of\/
$\R$\+contramodules\/ $\Hom^\C(\cP,\cQ)$ is contractible;
{\hbadness=1800\par}
\textup{(b)} for any coacyclic\/ $\R$\+cofree left CDG\+comodule\/
$\cL$ over\/ $\C$ and any CDG\+co\-module\/
$\cM\in H^0(\C\comod\Rcof_\inj)$, the complex of\/
$\R$\+contramodules\/ $\Hom_\C(\cL,\cM)$ is contractible; \par
\textup{(c)} the composition of natural functors 
$$
 H^0(\C\contra\Rcof_\proj)\lrarrow H^0(\C\contra\Rcof)
 \lrarrow\sD^\ctr(\C\contra\Rcof)
$$
is an equivalence of triangulated categories; \par
\textup{(d)} the composition of natural functors
$$
 H^0(\C\comod\Rcof_\inj)\lrarrow H^0(\C\comod\Rcof)\lrarrow
 \sD^\co(\C\comod\Rcof)
$$
is an equivalence of triangulated categories. \qed
\end{thm}

 Given a CDG\+contramodule $\cP\in\C\contra\Rcof_\proj$, the graded
left $\C$\+comodule $\Phi_\C(\cP)=\C\ocn_\C\cP$ is endowed with
a CDG\+comodule structure with the conventional tensor product
differential.
 Conversely, given a CDG\+comodule $\cM\in\C\comod\Rcof_\inj$,
the graded left $\C$\+contramodule $\Psi_\C(\cM)=\Hom_\C(\C,\cM)$
is endowed with a CDG\+con\-tramodule structure with
the conventional $\Hom$ differential.
 One easily checks that $\Phi_\C$ and $\Psi_\C$ are mutually
inverse equivalences between the DG\+categories $\C\contra\Rcof_\proj$
and $\C\comod\Rcof_\inj$ (see Proposition~\ref{r-cofree-c-co-contra}).

\begin{cor} \label{r-cofree-derived-co-contra}
 The derived functors
$$
 \boL\Phi_\C\:\sD^\ctr(\C\contra\Rcof)\lrarrow
 \sD^\co(\C\comod\Rcof)
$$
and
$$
 \boR\Psi_\C\:\sD^\co(\C\comod\Rcof)\lrarrow
 \sD^\ctr(\C\contra\Rcof)
$$
defined by identifying\/ $\sD^\ctr(\C\contra\Rcof)$ with
$H^0(\C\contra\Rcof_\proj)$ and\/ $\sD^\co(\C\comod\Rcof)$
with $H^0(\C\comod\Rcof_\inj)$ are mutually inverse equivalences
between the contraderived category\/ $\sD^\ctr(\C\contra\Rcof)$
and the coderived category\/ $\sD^\co(\C\comod\Rcof)$. \qed
\hfuzz=16.2pt
\end{cor}

 The equivalences of DG\+categories $\Phi_\C=\Psi_\C^{-1}$ form
a commutative diagram with the (appropriate restrictions of)
the equivalences of DG\+categories $\Phi_\R=\Psi_\R^{-1}$, hence
the derived functors $\boL\Phi_\C=\boR\Psi_\C^{-1}$ commute
with the triangulated functors induced by $\Phi_\R=\Psi_\R^{-1}$
(cf.\ the construction of functors $\boL\Phi_{\R,\C}=
\boR\Psi_{\R,\C}^{-1}$ in Section~\ref{non-adj-co-derived} below). 

 The right derived functor of homomorphisms of $\R$\+cofree
CDG\+contramodules
$$
 \Ext^\C\:\sD^\ctr(\C\contra\Rcof)^\sop\times
 \sD^\ctr(\C\contra\Rcof)\lrarrow H^0(\R\contra^\free)
$$
is constructed by restricting the functor $\Hom^\C$ to the full
subcategory $H^0(\C\contra\Rcof_\proj)^\sop\allowbreak\times
H^0(\C\contra\Rcof)\subset H^0(\C\contra\Rcof)^\sop
\times H^0(\C\contra\Rcof)$.
 Similarly, the right derived functor of $\Hom$ between
$\R$\+free and $\R$\+cofree CDG\+contramodules
{\hfuzz=18.2pt
$$
 \Ext^\C\:\sD^\ctr(\C\contra\Rfr)^\sop\times
 \sD^\ctr(\C\contra\Rcof)\lrarrow H^0(\R\comod^\cofr)
$$
is} constructed by restricting the functor $\Hom^\C$ to the full
subcategory $H^0(\C\contra\Rfr_\proj)^\sop\allowbreak\times
H^0(\C\contra\Rcof)$.
 The equivalences of categories $\Phi_\R=\Psi_\R^{-1}$ transform
these two functors into each other and the derived functor
$\Ext^\C$ of homomorphisms of $\R$\+free CDG\+contramodules
defined in Section~\ref{r-free-co-derived}.
{\hfuzz=12.6pt\par}

 Analogously, the right derived functor of homomorphisms of
$\R$\+cofree CDG\+co\-modules
$$
 \Ext_\C\:\sD^\co(\C\comod\Rcof)^\sop\times
 \sD^\co(\C\comod\Rcof)\lrarrow H^0(\R\contra^\free)
$$
is constructed by restricting the functor $\Hom_\C$ to the full
subcategory $H^0(\C\comod\Rcof)^\sop\allowbreak\times
H^0(\C\comod\Rcof_\inj)\subset H^0(\C\comod\Rcof)^\sop
\times H^0(\C\comod\Rcof)$.
 Similarly, the right derived functor of $\Hom$ between
$\R$\+free and $\R$\+cofree CDG\+comodules
{\hfuzz=19.9pt
$$
 \Ext_\C\:\sD^\co(\C\comod\Rfr)^\sop\times
 \sD^\co(\C\comod\Rcof)\lrarrow H^0(\R\comod^\cofr)
$$
is} constructed by restricting the functor $\Hom_\C$ to the full
subcategory $H^0(\C\comod\Rfr)^\sop\allowbreak\times
H^0(\C\comod\Rcof_\inj)$.
 The equivalences of categories $\Phi_\R=\Psi_\R^{-1}$ transform
these two functors into each other and the derived functor
$\Ext_\C$ of homomorphisms of $\R$\+free CDG\+comodules defined
in Section~\ref{r-free-co-derived}.
{\hfuzz=14.3pt\par}

 The left derived functor of contratensor product of $\R$\+free
CDG\+comodules and $\R$\+cofree CDG\+contramodules
$$
 \Ctrtor^\C\:\sD^\co(\comodrRfr\C)\times\sD^\ctr(\C\contra\Rcof)
 \lrarrow H^0(\R\comod^\cofr)
$$
is constructed by restricting the functor $\ocn_\C$ to the full
subcategory $H^0(\comodrRfr\C)\times H^0(\C\contra\Rcof_\proj)$.
 The equivalences of categories $\Phi_\R=\Psi_\R^{-1}$ transform
this functor into the derived functor $\Ctrtor^\C$ of
contratensor product of $\R$\+free CDG\+comodules and
CDG\+contramodules constructed in Section~\ref{r-free-co-derived}.

 The right derived functor of cotensor product of $\R$\+free
and $\R$\+cofree CDG\+co\-modules
$$
 \Cotor^\C\:\sD^\co(\comodrRfr\C)\times \sD^\co(\C\comod\Rcof)
 \lrarrow H^0(\R\comod^\cofr)
$$
is constructed by restricting the functor $\oc_\C$ to either of
the full subcategories $H^0(\comodrRfr\C)\times
H^0(\C\comod\Rcof_\inj)$ or $H^0(\comodrRfrinj\C)\times
H^0(\C\comod\Rcof)\subset H^0(\comodrRfr\C)\times
H^0(\C\comod\Rcof)$.
 The equivalences of categories $\Phi_\R=\Psi_\R^{-1}$ transform
this functor into the derived functor $\Cotor^\C$ of cotensor
product of $\R$\+free CDG\+comodules constructed in
Section~\ref{r-free-co-derived}.

 The left derived functor of cohomomorphisms from $\R$\+free
CDG\+comodules to $\R$\+cofree CDG\+contramodules
$$
 \Coext_\C\:\sD^\co(\C\comod\Rfr)^\sop\times\sD^\ctr(\C\contra\Rcof)
 \lrarrow H^0(\R\comod^\cofr)
$$
is constructed by restricting the functor $\Cohom_\C$ to either of
the full subcategories $H^0(\C\comod\Rfr_\inj)^\sop\times
H^0(\C\contra\Rcof)$ or $H^0(\C\comod\Rfr)^\sop\times
H^0(\C\contra\Rcof_\proj)\subset H^0(\C\comod\Rfr)^\sop\times
H^0(\C\contra\Rcof)$.
 Similarly, the left derived functor of cohomomorphisms of
$\R$\+cofree CDG\+comodules and CDG\+contramodules
$$
 \Coext_\C\:\sD^\co(\C\comod\Rcof)^\sop\times\sD^\ctr(\C\contra\Rcof)
 \lrarrow H^0(\R\contra^\free)
$$
is constructed by restricting the functor $\Cohom_\C$ to either of
the full subcategories $H^0(\C\comod\Rcof_\inj)^\sop\times
H^0(\C\contra\Rcof)$ or $H^0(\C\comod\Rcof)^\sop\times
H^0(\C\contra\Rcof_\proj)\allowbreak\subset H^0(\C\comod\Rcof)^\sop
\times H^0(\C\contra\Rcof)$.
 The equivalences of categories $\Phi_\R=\Psi_\R^{-1}$ transform
these two functors into each other and the derived functor
$\Coext_\C$ of cohomomorphisms of $\R$\+free CDG\+comodules and
CDG\+contramodules constructed in Section~\ref{r-free-co-derived}.

\begin{prop}  \label{r-cofree-cotor-ctrtor}
\textup{(a)} The equivalences of triangulated categories\/
$$
 \boL\Phi_\C\:\sD^\ctr(\C\contra\Rfr)\simeq
 \sD^\co(\C\comod\Rfr)\,\,\:\!\boR\Psi_\C
$$
and 
$$
 \boL\Phi_\C\:\sD^\ctr(\C\contra\Rcof)\simeq
 \sD^\co(\C\comod\Rcof)\,\,\:\!\boR\Psi_\C
$$
from Corollaries~\textup{\ref{r-free-derived-co-contra}}
and~\textup{\ref{r-cofree-derived-co-contra}} transform
the left derived functor
$$
 \Coext_\C\:\sD^\co(\C\comod\Rfr)^\sop\times\sD^\ctr(\C\contra\Rcof)
 \lrarrow H^0(\R\comod^\cofr)
$$
into the right derived functors
$$
 \Ext^\C\:\sD^\ctr(\C\contra\Rfr)^\sop\times\sD^\ctr(\C\contra\Rcof)
 \lrarrow H^0(\R\comod^\cofr)
$$
and
$$
\Ext_\C\:\sD^\co(\C\comod\Rfr)^\sop\times\sD^\co(\C\comod\Rcof)
\lrarrow H^0(\R\comod^\cofr).
$$
 In other words, the following diagram of categories, functors,
and equivalences is commutative:
$$
\begin{diagram}
\node{\sD^\ctr(\C\contra\Rfr)^\sop\times\sD^\ctr(\C\contra\Rcof)}
\arrow{s,=}\arrow[4]{e,t}{\Ext^\C}\node[4]{H^0(\R\comod^\cofr)}
\arrow{s,=} \\
\node{\sD^\co(\C\comod\Rfr)^\sop\times\sD^\ctr(\C\contra\Rcof)}
\arrow{s,=}\arrow[4]{e,t}{\Coext_\C}\node[4]{H^0(\R\comod^\cofr)}
\arrow{s,=} \\
\node{\sD^\co(\C\comod\Rfr)^\sop\times\sD^\co(\C\comod\Rcof)}
\arrow[4]{e,t}{\Ext_\C}\node[4]{H^0(\R\comod^\cofr)}
\end{diagram}
$$ \par
\textup{(b)} The equivalence of triangulated categories\/
$\boL\Phi_\C=\boR\Psi_\C^{-1}$ from
Corollary~\textup{\ref{r-cofree-derived-co-contra}} transforms
the left derived functor
$$
 \Coext_\C\:\sD^\co(\C\comod\Rcof)^\sop\times\sD^\ctr(\C\contra\Rcof)
 \lrarrow H^0(\R\contra^\free)
$$
into the right derived functors
$$
 \Ext^\C\:\sD^\ctr(\C\contra\Rcof)^\sop\times\sD^\ctr(\C\contra\Rcof)
 \lrarrow H^0(\R\contra^\free)
$$
and
$$
 \Ext_\C\:\sD^\co(\C\comod\Rcof)^\sop\times\sD^\co(\C\comod\Rcof)
 \lrarrow H^0(\R\contra^\free).
$$
 In other words, the following diagram of categories, functors,
and equivalences is commutative:
$$
\begin{diagram}
\node{\sD^\ctr(\C\contra\Rcof)^\sop\times\sD^\ctr(\C\contra\Rcof)}
\arrow{s,=}\arrow[4]{e,t}{\Ext^\C}\node[4]{H^0(\R\contra^\free)}
\arrow{s,=} \\
\node{\sD^\co(\C\comod\Rcof)^\sop\times\sD^\ctr(\C\contra\Rcof)}
\arrow{s,=}\arrow[4]{e,t}{\Coext_\C}\node[4]{H^0(\R\contra^\free)}
\arrow{s,=} \\
\node{\sD^\co(\C\comod\Rcof)^\sop\times\sD^\co(\C\comod\Rcof)}
\arrow[4]{e,t}{\Ext_\C}\node[4]{H^0(\R\contra^\free)}
\end{diagram}
$$ \par
\textup{(c)} The equivalence of triangulated categories\/
$\boL\Phi_\C=\boR\Psi_\C^{-1}$ from
Corollary~\textup{\ref{r-cofree-derived-co-contra}} transforms
the right derived functor
$$
 \Cotor^\C\:\sD^\co(\comodrRfr\C)\times\sD^\co(\C\comod\Rcof)
 \lrarrow H^0(\R\comod^\cofr)
$$
into the left derived functor
$$
 \Ctrtor^\C\:\sD^\co(\comodrRfr\C)\times\sD^\ctr(\C\contra\Rcof)
 \lrarrow H^0(\R\comod^\cofr).
$$
 In other words, the following diagram of categories, functors,
and equivalences is commutative:
$$
\begin{diagram}
\node{\sD^\co(\comodrRfr\C)\times\sD^\co(\C\comod\Rcof)}
\arrow{s,=}\arrow[4]{e,t}{\Cotor^\C}\node[4]{H^0(\R\comod^\cofr)}
\arrow{s,=} \\
\node{\sD^\co(\comodrRfr\C)\times\sD^\ctr(\C\contra\Rcof)}
\arrow[4]{e,t}{\Ctrtor^\C}\node[4]{H^0(\R\comod^\cofr)}
\end{diagram}
$$ \qed
\end{prop}

 Let $f=(f,a)\:\C\rarrow\D$ be a morphism of $\R$\+free
CDG\+coalgebras.
 Then with any $\R$\+cofree left CDG\+contramodule $(\cP,d_\cP)$
over $\C$ one can associate an $\R$\+cofree left CDG\+contramodule
$(\cP,d'_\cP)$ over $\D$ with the graded $\C$\+contramodule
structure on $\cP$ defined via~$f$ and the modified
differential~$d'_\cP$ constructed in terms of~$a$.
 Similar procedures apply to left and right CDG\+comodules.

 So we obtain the DG\+functors of contrarestriction of scalars
$R^f\:\C\contra\Rcof\rarrow\D\contra\Rcof$, and corestriction
of scalars $R_f\:\C\comod\Rcof\rarrow\D\comod\Rcof$ and
$\comodrRcof\C\rarrow\comodrRcof\D$.
 Passing to the homotopy categories, we have the triangulated
fuctors $R^f\:H^0(\C\contra\Rcof)\rarrow H^0(\D\contra\Rcof)$
and $R_f\:H^0(\C\comod\Rcof)\rarrow H^0(\D\comod\Rcof)$.
 The equivalences of categories $\Phi_\R=\Psi_\R^{-1}$ identify
these functors with the functors $R^f$ and $R_f$ for
$\R$\+free CDG\+contra\-modules and CDG\+comodules constructed
in Section~\ref{r-free-co-derived}.
{\emergencystretch=1em\par}

 Since the contra/corestriction of scalars clearly preserves
contra/coacyclicity, we have the induced functors on
the contra/coderived categories
\begin{alignat*}{3}
 &\boI R^f &&\:\sD^\ctr(\C\contra\Rcof)&&\lrarrow
 \sD^\ctr(\D\contra\Rcof), \\
 &\boI R_f &&\:\sD^\co(\C\comod\Rcof)&&\lrarrow
 \sD^\co(\D\comod\Rcof).
\end{alignat*}

 The triangulated functor $\boI R^f$ has a left adjoint.
 The DG\+functor $E^f\:\D\contra\Rcof_\proj\allowbreak\rarrow
\C\contra\Rcof_\proj$ is defined on the level of graded contramodules
by the rule $\cQ\mpsto\Cohom_\D(\C,\cQ)$; the differential on
$\Cohom_\D(\C,\cQ)$ induced by the differentials on $\C$
and $\cQ$ is modified to obtain the differential on $E^f(\cQ)$
using the linear function~$a$.
 Passing to the homotopy categories and taking into account
Theorem~\ref{r-cofree-co-derived-thm}(c), we obtain the left
derived functor
$$
 \boL E^f\:\sD^\ctr(\D\contra\Rcof)\lrarrow\sD^\ctr(\C\contra\Rcof),
$$
which is left adjoint to the functor $\boI R^f$.
 
 Similarly, the triangulated functor $\boI R^f$ has a right adjoint.
 The DG\+functor $E_f\:\D\comod\Rcof\rarrow\C\comod\Rcof$ is defined
on the level of graded comodules by the rule
$\cN\mpsto\C\oc_\D\cN$; the linear function~$a$ is used to modify
the differential on $\C\oc_\D\cN$ induced by the differentials
on $\C$ and~$\cN$.
 Passing to the homotopy categories and taking into account
Theorem~\ref{r-cofree-co-derived-thm}(d), we obtain the right
derived functor
$$
 \boR E_f\:\sD^\co(\D\comod\Rcof)\lrarrow\sD^\co(\C\comod\Rcof).
$$

\begin{prop} \label{r-cofree-co-extension}
 The equivalences of triangulated categories\/ $\boL\Phi_\C=
\boR\Psi_\C^{-1}$ and\/ $\boL\Phi_\D=\boR\Psi_\D^{-1}$ from
Corollary~\textup{\ref{r-cofree-derived-co-contra}} transform
the left derived functor\/ $\boL E^f$ into the right derived
functor\/ $\boR E_f$ and back.
 In other words, the following diagram of categories, functors,
and equivalences is commutative:
$$
\begin{diagram}
 \node{\llap{$\boL\Phi_\D$}\:\sD^\ctr(\D\contra\Rcof)}
 \arrow{e,=}\arrow{s,l}{\boL E^f}
 \node{\sD^\co(\D\comod\Rcof)\,\.\:\!\rlap{$\boR\Psi_\D$}}
 \arrow{s,r}{\boR E_f}\\
 \node{\llap{$\boL\Phi_\C$}\:\sD^\ctr(\C\contra\Rcof)}
 \arrow{e,=}
 \node{\sD^\co(\C\comod\Rcof)\,\.\:\!\rlap{$\boR\Psi_\C$}}
\end{diagram}
$$ \qed
\end{prop}

 The equivalences of contra/coderived categories
$\Phi_\R=\Psi_\R^{-1}$ transform the $\R$\+cofree CDG\+contra/comodule
contra/corestriction- and contra/coextension-of-scalars functors
$\boL E^f$, $\boI R^f$, $\boI R_f$, $\boR E_f$ defined above into
the $\R$\+free CDG\+contra/comodule contra/corestriction-
and contra/coextension-of-scalars functors
$\boL E^f$, $\boI R^f$, $\boI R_f$, $\boR E_f$ defined
in Section~\ref{r-free-co-derived}.

\Section{Non-$\R$-Free and Non-$\R$-Cofree wcDG-Modules,
\protect\\ CDG-Contramodules, and CDG-Comodules}

\subsection{Non-$\R$-free and non-$\R$-cofree graded modules}
\label{non-adj-graded}
 Let $\B$ be an $\R$\+free graded algebra.
 An \emph{$\R$\+contramodule graded left module} $\M$ over $\B$
is, by the definition, a graded left $\B$\+module in the module
category of $\R$\+contramodules over the tensor category of
free $\R$\+contramodules.
 In other words, it is a graded $\R$\+contramodule endowed with
an (associative and unital) homogeneous $\B$\+action map
$\B\ot^\R\M\rarrow\M$.
 \ \emph{$\R$\+contramodule graded right modules} $\N$ over $\B$
are defined in the similar way.

 Alternatively, one can define $\B$\+module structures on
graded $\R$\+contramodules in terms of the action maps
$\M\rarrow\Hom^\R(\B,\M)$, and similarly for~$\N$.
 The latter point of view may be preferable in that the functor
$\Hom^\R(\B,{-})$ is exact, while the functor
$\B\ot^\R{-}$ is only right exact (see
Remark~\ref{contra-operations-not-exact}).

 In fact, the category of $\R$\+contramodule graded (left or right)
$\B$\+modules is enriched over the tensor category of
(graded) $\R$\+contramodules, so the abelian group of morphisms
between two $\R$\+contramodule graded $\B$\+modules $\L$ and $\M$
is the underlying abelian group of the degree-zero component of
the graded $\R$\+contramodule $\Hom_\B(\L,\M)$ constructed as
the kernel of the pair of morphisms of graded $\R$\+contra\-modules
$\Hom^\R(\L,\M)\birarrow\Hom^\R(\B\ot^\R\L\;\M)\simeq
\Hom^\R(\L,\Hom^\R(\B,\M))$ induced by the actions of $\B$ in
$\L$ and~$\M$.
 The \emph{tensor product} $\N\ot_\B\M$ of an $\R$\+contramodule
graded right $\B$\+module $\N$ and an $\R$\+contramodule graded
left $\B$\+module $\M$ is a graded $\R$\+contramodule constructed
as the cokernel of the pair of morphisms of graded
$\R$\+contramodules  $\N\ot^\R\B\ot^\R\M\birarrow\N\ot^\R\M$.

 The category of $\R$\+contramodule graded $\B$\+modules is
abelian with infinite direct sums and products; and the forgetful
functor from it to the category of graded $\R$\+contramodules
is exact and preserves both infinite direct sums and products.
 There are enough projective objects in the abelian category of
$\R$\+contramodule graded $\B$\+modules; these are the same as
the projective objects in the exact subcategory of $\R$\+free
graded $\B$\+modules.
 The latter exact subcategory in the abelian category of
$\R$\+contramodule graded $\B$\+modules is closed under
extensions and infinite direct sums and products.
 Infinite products of $\R$\+contramodule graded $\B$\+modules
are exact functors (because infinite products of $\R$\+contramodules
are); infinite direct sums are not, in general (because infinite
direct sums of $\R$\+contramodules are not).

 For any graded $\R$\+contramodule $\U$ and any $\R$\+contramodule
graded left $\B$\+module $\M$, there is a natural isomorphism of
graded $\R$\+contramodules $\Hom_\B(\B\ot^\R\U\;\M)\simeq
\Hom^\R(\U,\M)$ \cite[Lemma~1.1.2]{Psemi}.
 For any graded $\R$\+contramodule $\V$, the $\R$\+contramodule
graded left $\B$\+module $\Hom^\R(\B,\V)$ has a similar property
\cite[Section~3.1.1]{Psemi}.
 Given an $\R$\+contramodule graded left $\B$\+module $\N$,
an epimorphism onto in from a projective $\R$\+free graded
$\B$\+module can be obtained as the composition
$\B\ot^\R\U\rarrow\B\ot^\R\N\rarrow\N$, where $\U$ is a free
graded $\R$\+contramodule and $\U\rarrow\N$ is an epimorphism
of graded $\R$\+contramodules.
 For any $\R$\+contramodule graded right $\B$\+module $\N$ and
any graded $\R$\+contramodule $\U$, there is a natural isomorphism
of $\R$\+contramodules $\N\ot_\B(\B\ot^\R\U)\simeq\N\ot^\R\U$
\cite[Lemma~1.2.1]{Psemi}.

 An \emph{$\R$\+comodule graded left module} $\cM$ over $\B$ is,
by the definition, a graded $\B$\+module in the module category
of $\R$\+comodules over the tensor category of free 
$\R$\+contramodules.
 In other words, it is a graded $\R$\+comodule endowed with
an (associative and unital) homogeneous $\B$\+action map
$\B\ocn_\R\cM\rarrow\cM$.
 \ \emph{$\R$\+comodule graded right modules} $\cN$ over $\B$
are defined in the similar way.

 Alternatively, one can define $\B$\+module structures on graded
$\R$\+comodules in terms of the action maps $\cM\rarrow
\Ctrhom_\R(\B,\cM)$, and similarly for~$\cN$.
 The former point of view may be preferable in that the functor
$\B\ocn_\R{-}$ is exact, while the functor
$\Ctrhom_\R(\B,{-})$ is only left exact (see
Remark~\ref{contra-operations-not-exact}).

 In fact, the category of $\R$\+comodule graded $\B$\+modules is
enriched over the tensor category of (graded) $\R$\+contramodules,
so the abelian group of morphisms between two $\R$\+comodule
graded $\B$\+modules $\cL$ and $\cM$ is the underlying abelian group
of the degree-zero component of the graded $\R$\+contramodule
$\Hom_\B(\cL,\cM)$ constructed as the kernel of the pair of
morphisms of graded $\R$\+contramodules
$\Hom_\R(\cL,\cM)\birarrow\Hom_\R(\B\ocn_\R\cL\;\cM)\simeq
\Hom_\R(\cL,\Ctrhom_\R(\B,\cM))$.
 The \emph{tensor product} $\N\ot_\B\cM$ of an $\R$\+contramodule
graded right $\B$\+module $\N$ and an $\R$\+comodule graded left
$\B$\+module $\cM$ is a graded $\R$\+comodule constructed as
the cokernel of the pair of morphisms of graded $\R$\+comodules
$\N\ot^\R\B\ocn_\R\cM\birarrow\N\ocn_\R\cM$.
 The graded $\R$\+comodule $\Hom_\B(\L,\cM)$ from
an $\R$\+contramodule graded left $\B$\+module $\L$ to
an $\R$\+comodule graded left $\B$\+module $\cM$ is defined
as the kernel of the pair of morphisms of graded $\R$\+comodules
$\Ctrhom_\R(\L,\cM)\birarrow\Ctrhom_\R(\B\ot^\R\L\;\cM)\simeq
\Ctrhom_\R(\L,\Ctrhom_\R(\B,\cM))$.

 The contratensor product $\cK\ocn_\R\Q$ of an $\R$\+comodule
$\cK$ and an $\R$\+contramodule $\Q$ is set to be equal to
the contratensor product $\Q\ocn_\R\cK$ as defined in
Section~\ref{hom-operations}.
 This operation is extended to graded $\R$\+comodules and
$\R$\+contramodules by taking infinite direct sums of
$\R$\+comodules along the diagonals of the bigrading (as usually).

 The \emph{tensor product} $\cN\ot_\B\cM$ of an $\R$\+comodule
graded right $\B$\+module $\cN$ and and $\R$\+comodule graded left
$\B$\+module $\cM$ is a graded $\R$\+comodule constructed as
the cokernel of the pair of morphisms of graded $\R$\+comodules
$(\cN\ocn_\R\B)\oc\cM\simeq\cN\oc_\R(\B\ocn_\R\cM)\birarrow
\cN\oc_\R\cM$ (see Lemma~\ref{hom-co-associativity}).
 The graded $\R$\+contramodule $\Hom_\B(\cL,\M)$ from
an $\R$\+comodule graded left $\B$\+module $\cL$ to
an $\R$\+contramodule graded left $\B$\+module $\M$ is defined
as the kernel of the pair of morphisms of graded $\R$\+contramodules
$\Cohom_\R(\cL,\M)\birarrow\Cohom_\R(\B\ocn_\R\cL\;\M)\simeq
\Cohom_\R(\cL,\Hom^\R(\B,\M))$.

 The category of $\R$\+comodule graded $\B$\+modules is abelian
with infinite direct sums and products; and the forgetful functor
from it to the category of graded $\R$\+comodules is exact and
preserves both infinite direct sums and products.
 There are enough injective objects in the abelian category of
$\R$\+comodule graded $\B$\+modules; these are the same as
the injective objects in the exact subcategory of $\R$\+cofree
graded $\B$\+modules.
 The latter exact subcategory in the abelian category of
$\R$\+comodule graded $\B$\+modules is closed under extensions
and infinite direct sums and products.
 Filtered inductive limits of $\R$\+comodule graded $\B$\+modules
are exact functors (because filtered inductive limits of
$\R$\+comodules are); infinite products are not, in general
(because infinite products of $\R$\+comodules are not).

 For any graded $\R$\+comodule $\cV$ and any $\R$\+comodule graded
left $\B$\+module $\cL$, there is a natural isomorphism of graded
$\R$\+contramodules $\Hom_\B(\cL,\Ctrhom_\R(\B,\cV))\simeq
\Hom_\R(\cL,\cV)$.
 Given an $\R$\+comodule graded left $\B$\+module $\cM$,
a monomorphism from it into an injective $\R$\+cofree graded
$\B$\+module can be obtained as the composition
$\cM\rarrow\Ctrhom_\R(\B,\cM)\rarrow\Ctrhom_\R(\B,\cV)$,
where $\cV$ is a cofree graded $\R$\+comodule and 
$\cM\rarrow\cV$ is a monomorphism of graded $\R$\+comodules.

 For any graded $\R$\+contramodule $\U$ and $\R$\+comodule graded
left $\B$\+module $\cM$, there is a natural isomorphism of graded
$\R$\+comodules $(\U\ot^\R\B)\ot_\B\cM\simeq\U\ocn_\R\cM$.
 For any graded $\R$\+contramodule $\U$ and $\R$\+comodule
graded left $\B$\+module $\cM$, there is a natural isomorphism of
graded $\R$\+comodules $\Hom_\B(\B\ot^\R\U\;\cM)\simeq
\Ctrhom_\R(\U,\cM)$.
 Similarly, for any $\R$\+contramodule graded left $\B$\+module
$\L$ and any graded $\R$\+comodule $\cV$, there is a natural
isomorphism of graded $\R$\+comodules $\Hom_\B(\L,\Ctrhom_\R(\B,\cV))
\simeq\Ctrhom_\R(\L,\cV)$.

 For any graded $\R$\+comodule $\cU$, the $\R$\+comodule
graded left $\B$\+module $\B\ocn_\R\cU$ has similar properties.
 In particular, for any graded $\R$\+comodule $\cU$ and
$\R$\+comodule graded right $\B$\+module $\cN$, there is
a natural isomorphism of graded $\R$\+comodules
$\cN\ot_\B(\B\ocn_\R\cU)\simeq\cN\oc_\R\cU$.
 For any graded $\R$\+comodule $\cU$ and $\R$\+contramodule
graded left $\B$\+module $\M$, there is a natural isomorphism
of graded $\R$\+contramodules $\Hom_\B(\B\ocn_\R\cU\;\M)\simeq
\Cohom_\R(\cU,\M)$.
 For any $\R$\+comodule graded left $\B$\+module $\cL$ and
graded $\R$\+contramodule $\V$, there is a natural isomorphism
of graded $\R$\+contramodules $\Hom_\B(\cL,\Hom^\R(\B,\V))
\simeq\Cohom_\R(\cL,\V)$.

 For any $\R$\+contramodule graded left $\B$\+module $\M$,
the graded $\R$\+comodule $\Phi_\R(\M)=\cC(\R)\ocn_\R\M$
has a natural $\R$\+comodule graded left $\B$\+module structure.
 The similar construction applies to right $\B$\+modules.
 For any $\R$\+comodule graded left $\B$\+module $\cM$,
the graded $\R$\+contramodule $\Psi_\R(\cM)=\Hom_\R(\cC(\R),\cM)$
has a natural $\R$\+contramodule graded left $\B$\+module
structure (see the proof of Proposition~\ref{r-cofree-r-co-contra}).

 The functors $\Phi_\R$ and $\Psi_\R$ between the abelian
categories of $\R$\+contramodule and $\R$\+comodule graded
left $\B$\+modules are adjoint to each other.
 Their restrictions to the exact subcategories of $\R$\+free
and $\R$\+cofree graded $\B$\+modules provide the equivalence
$\Phi_\R=\Psi_\R^{-1}$ between these exact categories defined
in Section~\ref{r-cofree-graded}.

 For any $\R$\+contramodule graded left $\B$\+module $\L$ and any
$\R$\+comodule graded left $\B$\+module $\cM$ there are natural
isomorphisms of graded $\R$\+contramodules
$\Hom_\B(\L,\Psi_\R(\cM))\simeq\Psi_\R(\Hom_\B(\L,\cM))\simeq
\Hom_\B(\Phi_\R(\L),\cM))$.
 For any $\R$\+contra\-module graded right $\B$\+module $\N$ and
any $\R$\+contramodule graded left $\B$\+module $\M$ there is
a natural isomorphism of graded $\R$\+comodules
$\N\ot_\B\Phi_\R(\M)\simeq\Phi_\R(\N\ot_\B\M)$.

 For any $\R$\+comodule graded left $\B$\+modules $\cL$ and $\cM$,
there is a natural morphism of graded $\R$\+contramodules
$\Hom_\B(\cL,\Psi_\R(\cM))\rarrow\Hom_\B(\cL,\cM)$, which is
an isomorphism whenever one of the graded $\B$\+modules $\cL$ and
$\cM$ is $\R$\+cofree.
 For any $\R$\+contramodule graded left $\B$\+modules $\L$ and
$\M$, there is a natural morphism of graded $\R$\+contramodules
$\Hom_\B(\Phi_\R(\L),\M)\rarrow\Hom_\B(\L,\M)$, which is
an isomorphism whenever one of the graded $\B$\+modules $\L$
and $\M$ is $\R$\+free.
 For any $\R$\+contramodule graded right $\B$\+module $\N$ and
any $\R$\+comodule graded left $\B$\+mod\-ule $\cM$, there is
a natural morphism of graded $\R$\+comodules
$\N\ot_\B\cM\rarrow\Phi_\R(\N)\ot_\B\cM$, which is 
an isomorphism whenever either the graded $\B$\+module $\N$
is $\R$\+free, or the graded $\B$\+module $\cM$ is $\R$\+cofree.

\subsection{Contra/coderived category of CDG-modules}
\label{cdg-coderived}
 The definitions of
\begin{itemize}
\item odd derivations of $\R$\+contramodule and
      $\R$\+comodule $\B$\+modules compatible with a given odd
      derivation of an $\R$\+free graded algebra~$\B$,
\item $\R$\+contramodule and $\R$\+comodule left and right
      CDG\+modules over an $\R$\+free CDG\+algebra~$\B$,
\item the complexes of $\R$\+contramodules $\Hom_\B(\L,\M)$ for
      given $\R$\+contramodule left CDG\+modules $\L$ and $\M$
      over~$\B$,
\item the complexes of $\R$\+contramodules $\Hom_\B(\cL,\cM)$ for
      given $\R$\+comodule left CDG\+modules $\cL$ and $\cM$
      over~$\B$,
\item the complex of $\R$\+comodules $\Hom_\B(\L,\cM)$ for
      a given $\R$\+contramodule left CDG\+module $\L$ and
      $\R$\+comodule left CDG\+module $\cM$ over~$\B$,
\item the complex of $\R$\+contramodules $\Hom_\B(\cL,\M)$ for
      a given $\R$\+comodule left CDG\+module $\cL$ and
      $\R$\+contramodule left CDG\+module $\M$ over~$\B$,
\item the complex of $\R$\+contramodules $\N\ot_\B\M$ for
      a given $\R$\+contramodule right CDG\+module $\N$ and
      $\R$\+contramodule left CDG\+module $\M$ over~$\B$,
\item the complex of $\R$\+comodules $\N\ot_\B\cM$ for
      a given $\R$\+contramodule right CDG\+module $\N$ and
      $\R$\+comodule left CDG\+module $\cM$ over~$\B$,
\item the complex of $\R$\+comodules $\cN\ot_\B\cM$ for
      a given $\R$\+comodule right CDG\+mod\-ule $\cN$ and
      $\R$\+comodule left CDG\+module $\cM$ over~$\B$,
\item the $\R$\+contramodule or $\R$\+comodule CDG\+modules
      over $\B$ obtained by restriction of scalars via a morphism
      of $\R$\+free CDG\+algebras $\B\rarrow\A$ from
      $\R$\+contramodule or $\R$\+comodule CDG\+modules over~$\A$
\end{itemize}
repeat the similar definitions
for $\R$\+free and $\R$\+cofree $\B$\+modules given in
Sections~\ref{r-free-absolute} and~\ref{r-cofree-absolute}
\emph{verbatim} (with the definitions and constructions of
Section~\ref{non-adj-graded} being used in place of those from
Sections~\ref{r-free-graded} and~\ref{r-cofree-graded} as applicable),
so there is no need to spell them out here again.
 We restrict ourselves to introducing the new notation for
our new and more general classes of objects.

 The DG\+categories of $\R$\+contramodule left and right
CDG\+modules over $\B$ are denoted by $\B\mod\Rctr$ and
$\modrRctr\B$, and their homotopy categories are
$H^0(\B\mod\Rctr)$ and $H^0(\modrRctr\B)$.
 Similarly, the DG\+categories of $\R$\+comodule left and right
CDG\+modules over $\B$ are denoted by $\B\mod\Rco$ and
$\modrRco\B$, and their homotopy categories are
$H^0(\B\mod\Rco)$ and $H^0(\modrRco\B)$.
 The tensor products of CDG\+modules over $\B$ are triangulated
functors of two arguments
\begin{alignat*}{3}
 &\ot_\B\: H^0(\modrRctr\B)&&\times H^0(\B\mod\Rctr)&&\lrarrow
 H^0(\R\contra), \\
 &\ot_\B\: H^0(\modrRctr\B)&&\times H^0(\B\mod\Rco)&&\lrarrow
 H^0(\R\comod), \\
 &\ot_\B\: H^0(\modrRco\B)&&\times H^0(\B\mod\Rco)&&\lrarrow
 H^0(\R\comod).
\end{alignat*}
 The $\Hom$ from $\R$\+contramodule to $\R$\+comodule CDG\+modules
over $\B$ is a triangulated functor
$$
 \Hom_\B\: H^0(\B\mod\Rctr)^\sop\times H^0(\B\mod\Rco)\lrarrow
 H^0(\R\comod),
$$ 
and the $\Hom$ from $\R$\+comodule to $\R$\+contramodule CDG\+modules
over $\B$ is a triangulated functor
$$
 \Hom_\B\: H^0(\B\mod\Rco)^\sop\times H^0(\B\mod\Rctr)\lrarrow
 H^0(\R\contra).
$$
 Given a morphism of $\R$\+free CDG\+algebras $f=(f,a)\:\B\rarrow\A$,
the functors of restriction of scalars are denoted by
\begin{align*}
 R_f\: H^0(\A\mod\Rctr)&\lrarrow H^0(\B\mod\Rctr), \\
 R_f\: H^0(\A\mod\Rco)&\lrarrow H^0(\B\mod\Rco),
\end{align*}
and similarly for right CDG\+modules.

 An $\R$\+contramodule left CDG\+module over $\B$ is said to be
\emph{contraacyclic} if it belongs to the minimal triangulated
subcategory of the homotopy category $H^0(\B\allowbreak\mod\Rctr)$
containing the totalizations of short exact sequences of
$\R$\+contramodule CDG\+modules over $\B$ and closed under
infinite products.
 The quotient category of $H^0(\B\mod\Rctr)$ by the thick
subcategory of contraacyclic $\R$\+contramodule CDG\+mod\-ules
is called the \emph{contraderived category} of $\R$\+contramodule
left CDG\+modules over $\B$ and denoted by $\sD^\ctr(\B\mod\Rctr)$.
 The contraderived category of $\R$\+contramodule right
CDG\+modules over $\B$, denoted by $\sD^\ctr(\modrRctr\B)$,
is defined similarly. {\emergencystretch=0em\hfuzz=6pt\par}

 An $\R$\+comodule left CDG\+module over $\B$ is said to be
\emph{coacyclic} if it belongs to the minimal triangulated
subcategory of the homotopy category $H^0(\B\mod\Rco)$ containing
the totalizations of short exact sequences of $\R$\+comodule
CDG\+modules over $\B$ and closed under infinite direct sums.
 The quotient category of $H^0(\B\mod\Rco)$ by the thick
subcategory of coacyclic $\R$\+comodule CDG\+modules is called
the \emph{coderived category} of $\R$\+comodule left
CDG\+modules over $\B$ and denoted by $\sD^\co(\B\mod\Rco)$.
 The coderived category of $\R$\+comodule right CDG\+modules
over $\B$, denoted by $\sD^\co(\modrRco\B)$, is defined similarly.

 It follows from the next theorem, among other things, that our
terminology is not ambigous: an $\R$\+free CDG\+module over $\B$
is contraacyclic in the sense of Section~\ref{r-free-absolute}
if and only if it is contraacyclic as an $\R$\+contramodule
CDG\+module, in the sense of the above definition.
 Similarly, an $\R$\+cofree CDG\+module over $\B$ is coacyclic
in the sense of Section~\ref{r-cofree-absolute} if and only if
it is coacyclic as an $\R$\+comodule CDG\+module, in the sense
of the above definition.

\begin{thm}  \label{co-derived-mod}
 For any\/ $\R$\+free CDG\+algebra\/ $\B$, the functors
$$
 \sD^\ctr(\B\mod\Rfr)\lrarrow\sD^\ctr(\B\mod\Rctr)
$$
and
$$
 \sD^\co(\B\mod\Rcof)\lrarrow\sD^\co(\B\mod\Rco)
$$
induced by the natural embeddings of DG\+categories\/
$\B\mod\Rfr\rarrow\B\mod\Rctr$ and\/ $\B\mod\Rcof\rarrow\B\mod\Rco$
are equivalences of triangulated categories.
\end{thm}

\begin{proof}
 We follow the idea of~\cite[proof of Theorem~1.5 and
Remark~1.5]{Psing}.
 Let $\M$ be an $\R$\+contramodule left CDG\+module over~$\B$.
 As explained in Section~\ref{non-adj-graded}, there exists
a surjective morphism onto the $\R$\+contramodule graded
$\B$\+module $\M$ from an $\R$\+free (and even projective)
$\R$\+contramodule graded $\B$\+module~$\P_0$.
 Let $G^+(\P_0)$ be the $\R$\+contramodule CDG\+module over $\B$
freely generated by $\P_0$ (see the proof of
Theorem~\ref{r-free-absolute-derived}); then there is a surjective
closed morphism of CDG\+modules $G^+(\P_0)\rarrow\M$.
 Applying the same procedure to the kernel of the latter morphism,
etc., we obtain a left resolution of $\M$ by $\R$\+free
CDG\+modules and closed morphisms $\dotsb\rarrow G^+(\P_1)
\rarrow G^+(\P_0)\rarrow\M\rarrow0$.

 Totalizing this complex of $\R$\+free CDG\+modules by taking
infinite products along the diagonals, we get a closed morphism
of $\R$\+contramodule CDG\+modules $\P\rarrow\M$ with
an $\R$\+free CDG\+module~$\P$.
 It follows from the next lemma that the cone of this morphism
is a contraacyclic $\R$\+contramodule CDG\+module.

\begin{lem} \label{bounded-above-lem}
 For any bounded above exact sequence of\/ $\R$\+contramodule
CDG\+mod\-ules over\/ $\B$ and closed morphisms between them
$\dotsb\rarrow\K_2\rarrow\K_1\rarrow\K_0\rarrow0$, the total
CDG\+module\/ $\L$ of the complex of CDG\+modules\/ $\K_\bu$,
constructed by taking infinite products along the diagonals,
is a contraacyclic\/ $\R$\+contramodule CDG\+module.
\end{lem}

\begin{proof}
 Let $\L_n$ denote the totalizations of the finite quotient
complexes of canonical filtration of the exact complex~$\K_\bu$.
 Clearly, the $\R$\+contramodule CDG\+modules $\L_n$ are
contraacyclic.
 Consider the ``telescope'' short sequence of $\R$\+contramodule
CDG\+modules $\L\rarrow\prod_n\L_n\rarrow\prod_n\L_n$, the second
morphism being constructed in terms of the identity morphisms
$\L_n\rarrow\L_n$ and the natural closed surjections
$\L_{n+1}\rarrow\L_n$.
 Forgetting the differentials (that is considering our short
sequence as a sequence of $\R$\+contramodule graded $\B$\+modules),
we discover that this short sequence is a product of telescope
sequences related to stabilizing projective systems of
$\R$\+contramodule graded $\B$\+modules.
 The latter being always split exact, our short sequence of
$\R$\+contramodule CDG\+modules is also exact; and $\L_n$ being
contraacyclic, it follows that $\L$ is contraacyclic, too.
\end{proof}

 By~\cite[Lemma~1.6]{Pkoszul}, it follows that the contraderived
category $\sD^\ctr(\B\mod\Rctr)$ is equivalent to the quotient
category of the homotopy category of $\R$\+free left
CDG\+modules over $\B$ by the thick subcategory of $\R$\+free
CDG\+modules \emph{contraacyclic as\/ $\R$\+contramodule
CDG\+modules}.
 It remains to show that any $\R$\+free CDG\+module contraacyclic
as an $\R$\+contramodule CDG\+module is also contraacyclic as
an $\R$\+free CDG\+module.

 In fact, we will prove that any closed morphism from an $\R$\+free
CDG\+module $\P$ to a contraacyclic $\R$\+contramodule CDG\+module
$\L$ factorizes through a contraacyclic $\R$\+free CDG\+module $\F$
as a morphism in the homotopy category $H^0(\B\mod\Rctr)$.
 Indeed, consider the class of all $\R$\+contramodule CDG\+modules
$\L$ satisfying the above condition with respect to all
$\R$\+free CDG\+modules~$\P$.
 We will check that this class of CDG\+modules is closed under
cones and infinite products, and contains the total CDG\+modules
of short exact sequences of $\R$\+contramodule CDG\+modules.

 Let $\L_\alpha$ be a family of $\R$\+contramodule CDG\+modules,
$\P$ be an $\R$\+free CDG\+module, and $\P\rarrow\prod_\alpha
\L_\alpha$ be a morphism in $H^0(\B\mod\Rctr)$.
 Assuming that each component $\P\rarrow\L_\alpha$ of our
morphism factorizes through a contraacyclic $\R$\+free CDG\+module
$\F_\alpha$, the morphism $\P\rarrow\prod_\alpha\L_\alpha$
factorizes though the CDG\+module $\prod_\alpha\F_\alpha$,
which is also a contraacyclic $\R$\+free CDG\+module.
 This proves the closedness under infinite products.

 Now let us reformulate the property of CDG\+modules that we are
interested in as follows: an $\R$\+contramodule CDG\+module $\L$
belongs to our class of CDG\+modules if and only if for any
closed morphism $\P\rarrow\L$ into $\L$ from an $\R$\+free
CDG\+module $\P$ there exists a closed morphism of $\R$\+free
CDG\+modules $\Q\rarrow\P$ whose cone is a contraacyclic
$\R$\+free CDG\+module and whose composition with the morphism
$\P\rarrow\L$ is homotopic to zero.
 Assume that $\R$\+contramodule CDG\+modules $\K$ and $\M$ have
this property, and let $\K\rarrow\L\rarrow\M$ be a distinguished
triangle in $H^0(\B\mod\Rctr)$.

 Given a closed morphism $\P\rarrow\L$ with an $\R$\+free
CDG\+module $\P$, consider the composition $\P\rarrow\L\rarrow\M$
and find a closed morphism of $\R$\+free CDG\+modules
$\T\rarrow\P$ whose cone is a contraacyclic $\R$\+free CDG\+module
and such that the composition $\T\rarrow\P\rarrow\L\rarrow\M$ is
homotopic to zero.
 Then the composition $\T\rarrow\P\rarrow\L$ factorizes
through~$\K$ in the homotopy category.
 Find a closed morphism of $\R$\+free CDG\+modules $\Q\rarrow\T$
whose cone is a contraacyclic $\R$\+free CDG\+module and
such that the composition $\Q\rarrow\T\rarrow\K$ is homotopic to zero.
 Then the composition $\Q\rarrow\T\rarrow\P$ provides the desired
closed morphism of $\R$\+free CDG\+modules with a contraacyclic
$\R$\+free cone annihilating the morphism $\P\rarrow\L$.
 This proves the closedness with respect to cones.

 Finally, let $\M$ be the total CDG\+module of a short exact
sequence of $\R$\+contra\-module CDG\+modules $\U\rarrow\V\rarrow\W$.
 Let $\P\rarrow\M$ be a closed morphism into $\M$ from an $\R$\+free
CDG\+module~$\P$.
 It remains to construct a closed morphism $\Q\rarrow\P$ of
$\R$\+free CDG\+modules whose cone is a contraacyclic $\R$\+free
CDG\+module and whose composition with the morphism $\P\rarrow\M$
is homotopic to zero.

 As a graded $\R$\+contramodule $\B$\+module, $\M$ is the direct
sum of three modules $\U[1]$, \ $\V$, and $\W[-1]$; so any
graded $\B$\+module morphism $\N\rarrow\M$ into $\M$ from
an $\R$\+contramodule graded $\B$\+module $\N$ can be viewed as
a triple of graded $\B$\+module morphisms $f\:\N\rarrow\U[1]$, \ 
$g\:\N\rarrow\V$, and $h\:\N\rarrow\W[-1]$.

\begin{lem}  \label{lifting-contract-lem}
 Let\/ $\N\rarrow\M$ be a closed morphism of\/ $\R$\+contramodule
CDG\+mod\-ules represented by a triple $(f,g,h)$ as above.
 Then whenever the morphism $h\:\N\rarrow\W[-1]$ can be lifted to
a morphism of\/ $\R$\+contramodule graded\/ $\B$\+modules
$t\:\N\rarrow\V[-1]$, the morphism\/ $\N\rarrow\M$ is homotopic
to zero.
\end{lem}

\begin{proof}
 See~\cite[Lemma~1.5.E(b)]{Psing}.
\end{proof}

\begin{lem}  \label{generated-lifting-lem}
 Let\/ $\N\rarrow\M$ be a morphism of\/ $\R$\+contramodule graded\/
$\B$\+modules with the components $(f,g,h)$.
 Let $G^+(\N)\rarrow\M$ be the induced closed morphism of\/
$\R$\+contramodule CDG\+modules over\/~$\B$; denote its components
by $(\tilde f,\tilde g,\tilde h)$.
 Then the morphism of\/ $\R$\+contramodule graded\/ $\B$\+modules
$\tilde h\:G^+(\N)\rarrow\W[-1]$ can be lifted to a morphism
of\/ $\R$\+contramodule graded\/ $\B$\+modules $G^+(\N)\rarrow\V[-1]$
whenever the morphism $h\:\N\rarrow\W[-1]$ can be lifted to
a morphism\/ $\N\rarrow\V[-1]$.
\end{lem}

\begin{proof}
 See~\cite[Lemma~1.5.F]{Psing}.
\end{proof}

 Let $\N$ be an $\R$\+free graded $\B$\+module mapping surjectively
onto the fibered product of the morphisms of $\R$\+contramodule
graded $\B$\+modules $\P\rarrow\W[-1]$ and $\V[-1]\rarrow\W[-1]$.
 Then $\N\rarrow\P$ is a surjective morphism of $\R$\+free graded
$\B$\+modules; consider the induced surjective closed morphism of
$\R$\+free CDG\+modules $G^+(\N)\rarrow\P$ over~$\B$.
 Let $\T$ be the kernel of the latter morphism and $\Q$ be the cone
of the closed embedding $\T\rarrow G^+(\N)$.
 Then there is a natural closed morphism of $\R$\+free CDG\+modules
$\Q\rarrow\P$.
 Its cone, being the total CDG\+module of a short exact sequence of
$\R$\+free CDG\+modules $\T\rarrow\Q\rarrow G^+(\N)$ over $\B$, is
a contraacyclic $\R$\+free CDG\+module.

 Consider the composition $\Q\rarrow\P\rarrow\W[-1]$; it is
a morphism of $\R$\+contra\-module graded $\B$\+modules
$G^+(\N)\oplus\T[1]\rarrow\W[-1]$ vanishing on $\T[1]$.
 The morphism $G^+(\N)\rarrow\W[-1]$ is the component $\tilde h$
of the closed morphism $G^+(\N)\rarrow\M$ induced by the morphism of
$\R$\+contramodule graded $\B$\+modules $\N\rarrow\M$ equal to
the composition $\N\rarrow\P\rarrow\M$.
 By the construction, the component $h\:\N\rarrow\W[-1]$ of
the latter morphism lifts to an $\R$\+contramodule graded $\B$\+module
morphism $\N\rarrow\V[-1]$.
 By Lemma~\ref{generated-lifting-lem}, the $\R$\+contramodule graded
$\B$\+module morphism $\tilde h\:G^+(\N)\rarrow\W[-1]$ can be also
lifted into~$\V[-1]$.
 Hence the same applies to the composition $\Q\rarrow\P\rarrow\W[-1]$;
and by Lemma~\ref{lifting-contract-lem} it follows that
the composition $\Q\rarrow\P\rarrow\M$ is homotopic to zero.

 We have proven the first assertion of Theorem; the proof of
the second one is analogous up to duality.
\end{proof}

 Notice that when the pro-Artinian topological local ring $\R$ has
finite homological dimension (see Section~\ref{discrete-modules}),
the above argument also proves that the embeddings of DG\+categories
$\B\mod\Rfr\rarrow\B\mod\Rctr$ and $\B\mod\Rcof\rarrow\B\mod\Rco$
induce also equivalences of the absolute derived categories
$\sD^\abs(\B\mod\Rfr)\simeq\sD^\abs(\B\mod\Rctr)$ and
$\sD^\abs(\B\mod\Rcof)\simeq\sD^\abs(\B\mod\Rco)$.
 Here the absolute derived categories $\sD^\abs(\B\mod\Rctr)$ and
$\sD^\abs(\B\mod\Rco)$ are defined as the quotient categories of
the homotopy categories $H^0(\B\mod\Rctr)$ and $H^0(\B\mod\Rco)$
by the thick subcategories of \emph{absolutely acyclic}
$\R$\+contramodule and $\R$\+comodule CDG\+modules, i.~e.,
the minimal thick subcategories containing the totalizations of
short exact sequences of $\R$\+contramodule or $\R$\+comodule
CDG\+modules, respectively.

 As a particular case of the above definitions in the case
$\B=\R$, we have the classes of \emph{contraacyclic} and
\emph{absolutely acyclic} complexes of $\R$\+contramodules, and
similarly, \emph{coacyclic} and \emph{absolutely acyclic}
complexes of $\R$\+comodules.
 The corresponding quotient categories of the homotopy categories
$H^0(\R\contra)$ and $H^0(\R\comod)$ are the \emph{contraderived
category} of complexes of $\R$\+contramodules $\sD^\ctr(\R\contra)$,
the \emph{absolute derived category} of complexes of
$\R$\+contramodules $\sD^\abs(\R\contra)$, the \emph{coderived
category} of complexes of $\R$\+comodules $\sD^\co(\R\comod)$,
and the \emph{absolute derived category} of complexes of
$\R$\+comodules $\sD^\abs(\R\comod)$.

 When $\R$ has finite homological dimension, the quotient categories
$\sD^\ctr(\R\contra)$ and $\sD^\abs(\R\contra)$ coincide with each
other and with the conventional derived category of (complexes of)
$\R$\+contramodules $\sD(\R\contra)$, and similarly, the quotient
categories $\sD^\co(\R\comod)$ and $\sD^\abs(\R\comod)$ coincide
with each other and with the conventional derived category of
(complexes of) $\R$\+comodules $\sD(\R\comod)$
\cite[Remark~2.1]{Psemi}.

\begin{thm}  \label{non-adj-orthogonality}
 Let\/ $\B$ be an\/ $\R$\+free CDG\+algebra.  Then \par
\textup{(a)} for any CDG\+module\/ $\P\in H^0(\B\mod\Rfr_\proj)$
and any contraacyclic\/ $\R$\+contra\-module left CDG\+module\/ $\M$
over\/ $\B$, the complex of\/ $\R$\+contramodules\/
$\Hom_\B(\P,\M)$ is contraacyclic; \par
\textup{(b)} for any coacyclic\/ $\R$\+comodule left CDG\+module\/
$\cL$ over\/ $\B$ and any CDG\+module\/ $\cJ\in H^0(\B\mod\Rcof_\inj)$,
the complex of\/ $\R$\+contramodules\/ $\Hom_\B(\cL,\cJ)$ is
contraacyclic.
\end{thm}

\begin{proof}
 To prove part~(a), notice that the functor $\Hom_\B(\P,{-})$
takes short exact sequences and infinite products of
$\R$\+contramodule CDG\+modules to short exact sequences and
infinite products of complexes of $\R$\+contramodules.
 It also takes cones of closed morphisms of $\R$\+contramodule
CDG\+modules to cones of closed morphisms of complexes of
$\R$\+contramodules.
 The proof of part~(b) is similar up to duality
(cf.\ \cite[Theorem~3.5 and Remark~3.5]{Pkoszul}; see also
Theorem~\ref{r-free-orthogonality}).
\end{proof}

\begin{cor}  \label{non-adj-fin-dim-reduct-co-abs}
 Let\/ $\B$ be an\/ $\R$\+free CDG\+algebra such that the exact
category of\/ $\R$\+(co)free graded left\/ $\B$\+modules has
finite homological dimension
(cf.\ Proposition~\textup{\ref{r-cofree-r-co-contra}} and
Corollary~\textup{\ref{r-free-homol-dim}}).
 Then the functors
$$
 \sD^\abs(\B\mod\Rfr)\lrarrow\sD^\ctr(\B\mod\Rctr)
$$
and
$$
 \sD^\abs(\B\mod\Rcof)\lrarrow\sD^\co(\B\mod\Rco)
$$
induced by the natural embeddings of DG\+categories\/
$\B\mod\Rfr\rarrow\B\mod\Rctr$ and\/ $\B\mod\Rcof\rarrow
\B\mod\Rco$ are equivalences of triangulated categories. 
\end{cor}

\begin{proof}
 It follows from Theorems~\ref{r-free-orthogonality}(a)
and~\ref{r-free-absolute-derived} that the exotic derived
categories $\sD^\abs(\B\mod\Rfr)$ and $\sD^\ctr(\B\mod\Rfr)$
coincide.
 Similarly one can show that the derived categories of the second
kind $\sD^\abs(\B\mod\Rcof)$ and $\sD^\co(\B\mod\Rcof)$ coincide
(see Section~\ref{r-cofree-absolute}).
 So it remains to apply Theorem~\ref{co-derived-mod}.
\end{proof}

\begin{cor}  \label{non-adj-fin-dim-reduct-r-co-contra}
 Let\/ $\B$ be an\/ $\R$\+free CDG\+algebra such that the exact
category of\/ $\R$\+(co)free graded left\/ $\B$\+modules has
finite homological dimension.
 Then the derived functors
$$
 \boL\Phi_\R\:\sD^\ctr(\B\mod\Rctr)\lrarrow\sD^\co(\B\mod\Rco)
$$
and
$$
 \boR\Psi_\R\:\sD^\co(\B\mod\Rco)\lrarrow\sD^\ctr(\B\mod\Rctr)
$$
defined by identifying\/ $\sD^\ctr(\B\mod\Rctr)$ with\/
$\sD^\abs(\B\mod\Rfr)$ and\/ $\sD^\co(\B\mod\Rco)$ with\/
$\sD^\abs(\B\mod\Rcof)$
(see Corollary~\textup{\ref{non-adj-fin-dim-reduct-co-abs}})
are mutually inverse equivalences of triangulated categories.
\end{cor}

\begin{proof}
 Follows from the results of Section~\ref{r-cofree-absolute}.
\end{proof}

\subsection{Semiderived category of wcDG-modules}
\label{wcdg-semiderived}
 Let $\A$ be a wcDG\+algebra over~$\R$.
 An \emph{$\R$\+contramodule} (left or right) \emph{wcDG\+module}
over $\A$ is, by the definition, an $\R$\+contramodule CDG\+module
over $\A$ considered as an $\R$\+free CDG\+algebra.
 \ \emph{$\R$\+comodule wcDG\+modules} over $\A$ are defined
similarly.
 All the definitions, constructions, and notation of
Section~\ref{cdg-coderived} related to CDG\+modules will be applied
to wcDG\+modules as a particular case.

 An $\R$\+contramodule left wcDG\+module over $\A$ is called
\emph{semiacyclic} if it belongs to the minimal thick subcategory
of $H^0(\A\mod\Rctr)$ containing both the contraacyclic
$\R$\+contramodule wcDG\+modules and the semiacyclic $\R$\+free
wcDG\+modules.
 Equivalently, an $\R$\+contramodule wcDG\+module is semiacyclic
if the equivalence of contraderived categories from
Theorem~\ref{co-derived-mod} assigns a semiacyclic $\R$\+free
wcDG\+module to it.
 It is clear from the latter definition that the class of
semiacyclic $\R$\+contramodule CDG\+modules is closed under
infinite products.
 It is important for these arguments that the class of semiacyclic
$\R$\+free wcDG\+modules contains the class of contraacyclic
$\R$\+free wcDG\+modules.

 The \emph{semiderived category} $\sD^\si(\A\mod\Rctr)$ of
$\R$\+contramodule left wcDG\+modules over $\A$ is defined as
the quotient category of the homotopy category $H^0(\A\mod\Rctr)$
by the thick subcategory of semiacyclic $\R$\+contramodule
wcDG\+modules.
 It is obvious from the definition that the embedding of
DG\+categories $\A\mod\Rfr\rarrow\A\mod\Rctr$ induces
an equivalence of semiderived categories $\sD^\si(\A\mod\Rfr)
\rarrow\sD^\si(\A\mod\Rctr)$.
 The semiderived category of $\R$\+contramodule right wcDG\+modules
$\sD^\si(\modrRctr\A)$ over $\A$ is defined similarly.

 Analogously, an $\R$\+comodule left wcDG\+module over $\A$ is
called \emph{semiacyclic} if it belongs to the minimal thick
subcategory of $H^0(\A\mod\Rco)$ containing both the coacyclic
$\R$\+comodule wcDG\+modules and the semiacyclic $\R$\+cofree
wcDG\+modules.
 Equivalently, an $\R$\+comodule wcDG\+module is semiacyclic
if the equivalence of coderived categories from
Theorem~\ref{co-derived-mod} assigns a semiacyclic $\R$\+cofree
wcDG\+module to it.
 It is clear from the latter definition that the class of
semiacyclic $\R$\+comodule wcDG\+modules is closed under
infinite direct sums.
 It is important for these arguments that the class of semiacyclic
$\R$\+cofree wcDG\+modules contains the class of coacyclic
$\R$\+cofree wcDG\+modules.

 The \emph{semiderived category} $\sD^\si(\A\mod\Rco)$ of
$\R$\+comodule left wcDG\+modules over $\A$ is defined as the quotient
category of the homotopy category $H^0(\A\mod\Rco)$ by the thick
subcategory of semiacyclic $\R$\+comodule wcDG\+modules.
 It is obvious from the definition that the embedding of
DG\+categories $\A\mod\Rcof\rarrow\A\mod\Rco$ induces an equivalence
of semiderived categories $\sD^\si(\A\mod\Rcof)\rarrow
\sD^\si(\A\mod\Rco)$.
 The semiderived category of $\R$\+comodule right wcDG\+modules
$\sD^\si(\modrRco\A)$ is defined similarly.

 When $\A$ is actually a DG\+algebra (i.~e., $h=0$),
an $\R$\+contramodule wcDG\+module over $\A$ is semiacyclic
if and only if its underlying complex of $\R$\+contramodules
is contraacyclic.
 This follows from the discussion in Section~\ref{r-free-semi}
together with the facts that the forgetful functor takes
contraacyclic $\R$\+contramodule wcDG\+modules to contraacyclic
complexes of $\R$\+contramodules, and a complex of
free $\R$\+contramodules is contraacyclic (with respect to
the class of complexes of arbitrary $\R$\+contramodules) if and
only if it is contractible.

 Similarly, an $\R$\+comodule wcDG\+module over $\A$ is semiacyclic
if and only if its underlying complex of $\R$\+comodules is coacyclic
(see Section~\ref{r-cofree-semi}).
 These assertions explain the ``semiderived category'' terminology
(cf.~\cite{Psemi}).

 When the topological local ring $\R$ has finite homological
dimension, the semiacyclic $\R$\+contramodule or $\R$\+comodule
wcDG\+modules over $\A$ can be simply called \emph{acyclic},
and the semiderived categories of $\R$\+contramodule or
$\R$\+comodule wcDG\+modules over $\A$ can be simply called
their \emph{derived categories}.

\begin{thm} \label{non-adj-semiderived-res}
 Let\/ $\A$ be a wcDG\+algebra over\/~$\R$.  Then \par
\textup{(a)} for any wcDG\+module\/
$\P\in H^0(\A\mod\Rfr_\proj)_\proj$ and any semiacyclic\/
$\R$\+contra\-module left wcDG\+module\/ $\M$ over\/ $\A$,
the complex of\/ $\R$\+contramodules\/ $\Hom_\A(\P,\M)$
is contraacyclic; \par
\textup{(b)} for any semiacyclic\/ $\R$\+comodule left wcDG\+module\/
$\cL$ over\/ $\A$ and any wcDG\+module\/
$\cJ\in H^0(\A\mod\Rcof_\inj)_\inj$, the complex of\/
$\R$\+contramodules\/ $\Hom_\A(\cL,\cJ)$ is contraacyclic;
{\hbadness=1400\par}
\textup{(c)} the composition of natural functors 
$$
 H^0(\A\mod\Rfr_\proj)_\proj\lrarrow H^0(\A\mod\Rctr)
 \lrarrow\sD^\si(\A\mod\Rctr)
$$
is an equivalence of triangulated categories; \par
\textup{(d)} the composition of natural functors
$$
 H^0(\A\mod\Rcof_\inj)_\inj\lrarrow H^0(\A\mod\Rco)
 \lrarrow\sD^\si(\A\mod\Rco)
$$
is an equivalence of triangulated categories.
\end{thm}

\begin{proof}
 Part~(a): it suffices to consider the cases when $\M$ is
a contraacyclic $\R$\+contra\-module wcDG\+module or $\M$ is
a semiacyclic $\R$\+free wcDG\+module.
 The former case holds for any $\P\in H^0(\A\mod\Rfr_\proj)$
by Theorem~\ref{non-adj-orthogonality}(a).
 In the latter case, the complex of $\R$\+contramodules
$\Hom_\A(\P,\M)$ is even contractible; see the remarks after
the proof of Theorem~\ref{r-free-semi-resolutions}.

 The proof of part~(b) is similar up to duality.
 Parts~(c) and~(d) follow from
Theorems~\ref{r-free-semi-resolutions}(a)
and~\ref{r-cofree-semi-resolutions}(b), respectively.
\end{proof}

\begin{prop} \label{non-adj-r-co-contra}
 The derived functors
$$
 \boL\Phi_\R\:\sD^\si(\A\mod\Rctr)\lrarrow\sD^\si(\A\mod\Rco)
$$
and
$$
 \boR\Psi_\R\:\sD^\si(\A\mod\Rco)\lrarrow\sD^\si(\A\mod\Rctr)
$$
defined by identifying\/ $\sD^\si(\A\mod\Rctr)$ with\/
$\sD^\si(\A\mod\Rfr)$ and\/ $\sD^\si(\A\mod\Rco)$ with\/
$\sD^\si(\A\mod\Rcof)$ are mutually inverse equivalences between
the semiderived categories\/ $\sD^\si(\A\mod\Rctr)$ and\/
$\sD^\si(\A\mod\Rco)$ of\/ $\R$\+contramodule and\/
$\R$\+comodule wcDG\+modules.
\end{prop}

\begin{proof}
 Follows from the results of Section~\ref{r-cofree-semi}.
\end{proof}

 One uses the equivalences of categories $\sD^\si(\A\mod\Rctr)
\simeq\sD^\si(\A\mod\Rfr)$ and $\sD^\si(\A\mod\Rco)\simeq\sD^\si
(\A\mod\Rcof)$ in order to extend the constructions and results
of Sections~\ref{r-free-semi} and~\ref{r-cofree-semi} related to
$\R$\+free and $\R$\+cofree wcDG\+modules over $\A$ to arbitrary
$\R$\+contramodule and $\R$\+comodule wcDG\+modules.
 Let us state some of the assertions which one can obtain
in this way.

\begin{cor}  \label{non-adj-cofibrant}
 Assume that the DG\+algebra\/ $\A/\m\A$ is cofibrant (in the standard
model structure on the category of DG\+algebras over~$k$).
 Then an\/ $\R$\+contramodule wcDG\+module over\/ $\A$ is semiacyclic 
if and only if it is contraacyclic, that is\/
$\sD^\ctr(\A\mod\Rctr)\simeq\sD^\si(\A\mod\Rctr)$.
 Similarly, an\/ $\R$\+comodule wcDG\+module over\/ $\A$ is semiacyclic
if and only if it is coacyclic, that is\/
$\sD^\co(\A\mod\Rco)\simeq\sD^\si(\A\mod\Rco)$.
 Assuming additionally that\/ $\R$ has finite homological dimension,
an\/ $\R$\+contramodule or\/ $\R$\+comodule wcDG\+module over\/ $\A$
is semiacyclic if and only if it is absolutely acyclic.
\end{cor}

\begin{proof}
 Follows from Theorem~\ref{r-free-cofibrant}.
\end{proof}

 In line with our usual notation, let $H^0(\modrRcofproj\A)_\proj$
denote the homotopy category of homotopy projective $\R$\+cofree
right wcDG\+modules over $\A$ with projective underlying
$\R$\+cofree graded $\A$\+modules.

\begin{lem} 
 Let\/ $\A$ be a wcDG\+algebra over\/~$\R$. Then \par
\textup{(a)} for any wcDG\+module\/ $\Q\in H^0(\modrRfrproj\A)_\proj$
and any semiacyclic\/ $\R$\+comodule left wcDG\+module\/ $\cM$ over\/
$\A$, the complex of\/ $\R$\+comodules\/ $\Q\ot_\A\cM$ is
coacyclic; \par
\textup{(b)} for any wcDG\+module\/ $\cQ\in H^0(\modrRcofproj\A)_\proj$
and any semiacyclic\/ $\R$\+comodule left wcDG\+module\/ $\cM$ over\/
$\A$, the complex of\/ $\R$\+comodules\/ $\cQ\ot_\A\cM$ is
coacyclic; {\hfuzz=1.3pt\par}
\textup{(c)} for any wcDG\+module\/ $\cP\in
H^0(\A\mod\Rcof_\proj)_\proj$ and any semiacyclic\/ $\R$\+contra\-module
left wcDG\+module\/ $\M$ over\/ $\A$, the complex of\/
$\R$\+contramodules\/ $\Hom_\A(\cP,\M)$ is contraacyclic; \par
\textup{(d)} for any wcDG\+module\/ $\J\in H^0(\A\mod\Rfr_\inj)_\inj$
and any semiacyclic\/ $\R$\+comodule left wcDG\+module $\cL$ over\/
$\A$, the complex of\/ $\R$\+contramodules\/ $\Hom_\A(\cL,\J)$ is
contraacyclic.
\end{lem}

\begin{proof}
 The proof is similar to that of
Theorem~\ref{non-adj-semiderived-res}(a\+b).
 The cases of a coacyclic $\R$\+comodule wcDG\+module $\cM$ or $\cL$,
or a contraacyclic $\R$\+contramodule wcDG\+module $\M$, are
straightforward in view of the results of Section~\ref{non-adj-graded}.
 The cases of a semiacyclic $\R$\+cofree wcDG\+module $\cM$ or $\cL$,
or a semiacyclic $\R$\+free wcDG\+module $\M$, can be dealt with using
the techniques of Sections~\ref{r-free-semi} and~\ref{r-cofree-semi}.
 Alternatively, the isomorphisms $\Hom_\A(\cP,\M)\simeq
\Hom_\A(\Psi_\R(\cP),\M)$ and $\Hom_\A(\cL,\J)\simeq
\Hom_\A(\cL,\Phi_\R(\J))$ from Section~\ref{non-adj-graded} reduce
parts (c) and~(d) to Theorem~\ref{non-adj-semiderived-res}(a\+b).
 The isomorphism $\cQ\ot_\A\cM\simeq\Psi_\R(\cQ)\ot_\A\cM$
reduces part~(b) to part~(a); and part~(a) in the case of
a semiacyclic $\R$\+cofree wcDG\+module $\cM$ is 
Lemma~\ref{r-cofree-homotopy-proj-tensor}(b). \emergencystretch=0em
\end{proof}

\begin{rem}
 The analogues of the assertions of Lemma do \emph{not} hold for
the other tensor product and $\Hom$ operations on wcDG\+modules
over~$\A$.
 In particular, the tensor product of a projective $\R$\+free 
wcDG\+module $\Q$ and a semiacyclic $\R$\+contramodule wcDG\+module
$\M$ is \emph{not} in general a contraacyclic complex of
$\R$\+contramodules (because this is not true for contraacyclic
$\R$\+contramodule wcDG\+modules~$\M$).
 The tensor product of a projective $\R$\+cofree wcDG\+module $\cQ$
and a semiacyclic $\R$\+contra\-module wcDG\+module $\M$ is \emph{not}
in general a coacyclic complex of $\R$\+comodules.
 The $\Hom$ from a projective $\R$\+free wcDG\+module $\P$ to
a semiacyclic $\R$\+comodule wcDG\+module $\cM$ is \emph{not} in
general a coacyclic complex of $\R$\+comodules (because this is not
true for coacyclic $\R$\+comodule wcDG\+modules~$\cM$).
 The $\Hom$ from a semiacyclic $\R$\+contramodule wcDG\+module $\L$
to an injective $\R$\+cofree wcDG\+module $\cJ$ is \emph{not} in
general a coacyclic complex of $\R$\+comodules (because this is not
true for contraacyclic $\R$\+contramodule wcDG\+modules~$\L$).
\end{rem}

 Restricting the functor $\ot_\A$ from Section~\ref{cdg-coderived}
to either of the full subcategories
$H^0(\modrRco\A)\times H^0(\A\mod\Rcof_\proj)_\proj$ or
$H^0(\modrRcofproj\A)_\proj\times H^0(\A\mod\Rco)\subset
H^0(\modrRco\A)\times H^0(\A\mod\Rco)$, composing it with
the localization functor $H^0(\R\comod)\rarrow\sD^\co(\R\comod)$,
and identifying the coderived category $\sD^\co(\R\comod)$ with
the homotopy category $H^0(\R\comod^\cofr)$ (see, e.~g.,
Theorem~\ref{non-adj-co-derived-res}(d) below),
we construct the double-sided derived functor {\hbadness=2650
$$
 \Tor^\A\:\sD^\si(\modrRco\A)\times\sD^\si(\A\mod\Rco)
 \lrarrow H^0(\R\comod^\cofr).
$$ 
 Using} the above identifications of the semiderived categories and
the constructions of the derived functors $\Tor^\A$ from
Sections~\ref{r-free-semi} and~\ref{r-cofree-semi}, we obtain
the left derived functors
\begin{alignat*}{2}
 &\Tor^\A\:\sD^\si(\modrRctr\A)\times\sD^\si(\A\mod\Rctr)
 &&\lrarrow H^0(\R\contra^\free), \\
 &\Tor^\A\:\sD^\si(\modrRctr\A)\times\sD^\si(\A\mod\Rco)
 &&\lrarrow H^0(\R\comod^\cofr).
\end{alignat*}
 The latter functor can be also constructed by restricting the functor
$\ot_\A$ from Section~\ref{cdg-coderived} to the full subcategory
$H^0(\modrRfrproj\A)_\proj\times H^0(\A\mod\Rco)$, composing it with
the localization functor $H^0(\R\comod)\rarrow\sD^\co(\R\comod)$, and
identifying $\sD^\co(\R\comod)$ with $H^0(\R\comod^\cofr)$.
 The three functors $\Tor^\A$ are transformed into each other by
the equivalences of categories $\boL\Phi_\R=\boR\Psi_\R^{-1}$ from
Proposition~\ref{non-adj-r-co-contra}.

 Similarly, restricting the functor $\Hom_\A$ from
Section~\ref{cdg-coderived} to either of the full subcategories
$H^0(\A\mod\Rcof_\proj)_\proj^\sop\times H^0(\A\mod\Rctr)$ or
$H^0(\A\mod\Rco)^\sop\times H^0(\A\mod\Rfr_\inj)_\inj\subset
H^0(\A\mod\Rco)^\sop\times H^0(\A\mod\Rctr)$, composing it
with the localization functor $H^0(\R\contra)\rarrow
\sD^\ctr(\R\contra)$, and identifying the contraderived category
$\sD^\ctr(\R\contra)$ with the homotopy category
$H^0(\R\contra^\free)$ (see, e.~g.,
Theorem~\ref{non-adj-co-derived-res}(c)),
we construct the double-sided derived functor {\hbadness=1250
$$
 \Ext_\A\:\sD^\si(\A\mod\Rco)^\sop\times\sD^\si(\A\mod\Rctr)
 \lrarrow H^0(\R\contra^\free).
$$
 From} the constructions of the derived functors
$\Ext_\A$ in Sections~\ref{r-free-semi} and~\ref{r-cofree-semi}
we obtain the right derived functors 
\begin{alignat*}{3}
 &\Ext_\A\:\sD^\si(\A\mod\Rctr)^\sop&&\times\sD^\si(\A\mod\Rctr)
 &&\lrarrow H^0(\R\contra^\free), \\
 &\Ext_\A\:\sD^\si(\A\mod\Rco)^\sop&&\times\sD^\si(\A\mod\Rco)
 &&\lrarrow H^0(\R\contra^\free), \\
 &\Ext_\A\:\sD^\si(\A\mod\Rctr)^\sop&&\times\sD^\si(\A\mod\Rco)
 &&\lrarrow H^0(\R\comod^\cofr),
\end{alignat*}
 The second of the four can be also constructed by restricting
the functor $\Hom_\A$ from Section~\ref{cdg-coderived} to
the full subcategory $H^0(\A\mod\Rfr_\proj)_\proj^\sop\times
H^0(\A\mod\Rctr)$, composing it with the localization functor
$H^0(\R\contra)\rarrow\sD^\ctr(\R\contra)$, and identifying
$\sD^\ctr(\R\contra)$ with $H^0(\R\contra^\free)$.
 Similarly, the third functor can be constructed by restricting
the functor $\Hom_\A$ from Section~\ref{cdg-coderived} to the full
subcategory $H^0(\A\mod\Rco)^\sop\times H^0(\A\mod\Rcof_\inj)_\inj$,
composing it with the localization functor
$H^0(\R\contra)\rarrow\sD^\ctr(\R\contra)$, and identifying
$\sD^\ctr(\R\contra)$ with $H^0(\R\contra^\free)$.
 The four functors $\Ext_\A$ are transformed into each other by
the equivalences of categories $\boL\Phi_\R=\boR\Psi_\R^{-1}$.

 Let $(f,a)\:\B\rarrow\A$ be a morphism of wcDG\+algebras over~$\R$.
 Then we have the induced functor $R_f\:H^0(\A\mod\Rctr)\rarrow
H^0(\B\mod\Rctr)$.
 This functor obviously takes contraacyclic $\R$\+contramodule
wcDG\+modules over $\A$ to contraacyclic $\R$\+contramodule
wcDG\+modules over $\B$; it was explained in
Section~\ref{r-free-semi} that it takes semiacyclic $\R$\+free
wcDG\+modules over $\A$ to semiacyclic $\R$\+free wcDG\+modules
over~$\B$.
 Hence it takes semiacyclic $\R$\+contramodule wcDG\+modules to
semiacyclic $\R$\+contramodule wcDG\+modules, and therefore
induces a triangulated functor
$$
 \boI R_f\:\sD^\si(\A\mod\Rctr)\lrarrow\sD^\si(\B\mod\Rctr).
$$
 The constructions of Section~\ref{r-free-semi} provide, via
the identification of the semiderived categories of $\R$\+free
and arbitrary $\R$\+contramodule CDG\+modules, the left and
right adjoint functors to $\boI R_f$
\begin{align*}
 \boL E_f\:\sD^\si(\B\mod\Rctr)&\lrarrow\sD^\si(\A\mod\Rctr), \\
 \boR E^f\:\sD^\si(\B\mod\Rctr)&\lrarrow\sD^\si(\A\mod\Rctr).
\end{align*}

 Similarly, the functor $R_f\:H^0(\A\mod\Rco)\rarrow
H^0(\B\mod\Rco)$ takes coacyclic $\R$\+comodule wcDG\+modules
over $\A$ to coacyclic $\R$\+comodule wcDG\+modules over $\B$,
and semiacyclic $\R$\+cofree wcDG\+modules over $\A$ to
semiacyclic $\R$\+cofree wcDG\+mod\-ules over $\B$.
 Hence it also takes semiacyclic $\R$\+comodule wcDG\+modules
to semiacyclic $\R$\+comodule wcDG\+modules, and therefore
induces a triangulated functor 
$$
 \boI R_f\:\sD^\si(\A\mod\Rco)\lrarrow\sD^\si(\B\mod\Rco).
$$
 The constructions of Section~\ref{r-cofree-semi} provide
the left and right adjoint functors
\begin{align*}
 \boL E_f\:\sD^\si(\B\mod\Rco)&\lrarrow\sD^\si(\A\mod\Rco), \\
 \boR E^f\:\sD^\si(\B\mod\Rco)&\lrarrow\sD^\si(\A\mod\Rco).
\end{align*}
 The equivalences of semiderived categories $\boL\Phi_\R = 
\boR\Psi_\R^{-1}$ transform the $\R$\+comodule wcDG\+module
restriction- and extension-of-scalars functors into the similar
$\R$\+contra\-module wcDG\+module functors.

\begin{cor}\emergencystretch=0em\hfuzz=2.6pt  \label{non-adj-quasi}
 The functors\/ $\boI R_f\:\sD^\si(\A\mod\Rctr)\rarrow
\sD^\si(\B\mod\Rctr)$ and\/ $\boI R_f\:\allowbreak\sD^\si(\A\mod\Rco)
\rarrow\sD^\si(\B\mod\Rco)$ are equivalences of triangulated
categories whenever the DG\+algebra morphism $f/\m f\:
\B/\m\B\rarrow\A/\m\A$ is a quasi-isomorphism.
\end{cor}

\begin{proof}
 Follows from Theorems~\ref{r-free-reduction-quasi}
and/or~\ref{r-cofree-reduction-quasi}.
\end{proof}

\subsection{Non-$\R$-free graded contramodules and non-$\R$-cofree
graded comodules} \label{non-adj-graded-co}
 Let $\C$ be an $\R$\+free graded coalgebra.
 An \emph{$\R$\+comodule graded left\/ $\C$\+comodule} $\cM$ is,
by the definition, a graded left $\C$\+comodule in the module
category of $\R$\+comodules over the tensor category of free
$\R$\+contramodules.
 In other words, it is a graded $\R$\+comodule endowed with
a (coassociative and counital) homogeneous $\C$\+coaction map
$\cM\rarrow\C\ocn_\R\cM$.
 \ \emph{$\R$\+comodule graded right\/ $\C$\+comodules} $\cN$ are
defined in the similar way.

 An \emph{$\R$\+contramodule graded left\/ $\C$\+contramodule} $\P$
is, by the definition, a graded left $\C$\+contramodule in
the ``internal Hom category'' of $\R$\+contramodules over
the tensor category of free $\R$\+contramodules.
 In other words, it is a graded $\R$\+contramodule endowed with
a (contraassociative and counital) homogeneous $\C$\+contraaction
map $\Hom^\R(\C,\P)\rarrow\P$.
 For the details, see Section~\ref{r-free-graded-co} and
the references therein.

 The categories of $\R$\+contramodule graded $\C$\+contramodules
and $\R$\+comodule graded $\C$\+comodules are enriched over
the tensor category of (graded) $\R$\+contra\-modules.
 So the abelian group of morphisms between two $\R$\+contramodule
graded left $\C$\+contramodules $\P$ and $\Q$ is the underlying
abelian group of the degree-zero component of the graded
$\R$\+contramodule $\Hom^\C(\P,\Q)$ constructed as the kernel of
the pair of morphisms of graded $\R$\+contramodules
$\Hom^\R(\P,\Q)\birarrow\Hom^\R(\Hom^\R(\C,\P),\Q)$.
 Similarly, the abelian group of morphisms between two $\R$\+comodule
graded left $\C$\+comodules $\cL$ and $\cM$ is the underlying
abelian group of the degree-zero component of the graded
$\R$\+contramodule $\Hom_\C(\cL,\cM)$ constructed as the kernel
of the pair of morphisms of graded $\R$\+contramodules
$\Hom_\R(\cL,\cM)\birarrow\Hom_\R(\cL\;\C\ocn_\R\cM)$.

 The graded $\R$\+comodule $\Hom_\C(\L,\cM)$ from an $\R$\+free
graded left $\C$\+comodule $\L$ to an $\R$\+comodule graded
left $\C$\+comodule $\cM$ is defined as the kernel of the pair
of morphisms of graded $\R$\+comodules $\Ctrhom_\R(\L,\cM)
\birarrow\Ctrhom_\R(\L\;\C\ocn_\R\cM)$ constructed in the way
explained in Section~\ref{r-cofree-graded-co}.
 The graded $\R$\+comodule $\Hom^\C(\P,\cQ)$ from
an $\R$\+contramodule graded left $\C$\+contramodule $\P$ to
an $\R$\+cofree graded left $\C$\+contramodule $\cQ$ is defined
as the kernel of the pair of morphisms of graded $\R$\+comodules
$\Ctrhom_\R(\P,\cQ)\birarrow\Ctrhom_\R(\Hom^\R(\C,\P),\cQ)$
constructed in the way explained in Section~\ref{r-cofree-graded-co}.

 The \emph{contratensor product} $\cN\ocn_\C\P$ of an $\R$\+comodule
graded right $\C$\+comodule $\cN$ and an $\R$\+contramodule graded
left $\C$\+contramodule $\P$ is the graded $\R$\+comodule
constructed as the cokernel of the pair of morphisms of graded
$\R$\+comodules $\cN\ocn_\R\Hom^\R(\C,\P)\birarrow\cN\ocn_\R\P$
defined in terms of the $\C$\+coaction in $\cN$,
the $\C$\+contraaction in $\P$, and the evaluation map
$\C\ot^\R\Hom^\R(\C,\P)\rarrow\P$ (see Section~\ref{non-adj-graded}
for the definition of $\cN\ocn_\R\P$).
 The \emph{contratensor product} $\N\ocn_\C\P$ of an $\R$\+free
graded right $\C$\+comodule $\N$ and and $\R$\+contramodule
graded left $\C$\+contramodule $\P$ is the graded $\R$\+contramodule
constructed as the cokernel of the pair of morphisms of graded
$\R$\+contramodules $\N\ot^\R\Hom^\R(\C,\P)\birarrow\N\ot^\R\P$
defined in terms of the $\C$\+coaction in $\N$,
the $\C$\+contraaction in $\P$, and the evaluation map
$\C\ot^\R\Hom^\R(\C,\P)\rarrow\P$.

 The \emph{cotensor product} $\N\oc_\C\cM$ of an $\R$\+free
graded right $\C$\+comodule $\N$ and an $\R$\+comodule graded
left $\C$\+comodule $\cM$ is a graded $\R$\+comodule constructed
as the kernel of the pair of morphisms $\N\ocn_\R\cM\birarrow
\N\ot^\R\C\ocn_\R\cM$.
 Similarly one defines the cotensor product $\cN\oc_\C\M$ of
an $\R$\+comodule graded right $\C$\+comodule $\cN$ and an $\R$\+free
graded left $\C$\+comodule~$\M$.
 The graded $\R$\+contramodule of \emph{cohomomorphisms}
$\Cohom_\C(\M,\P)$ from an $\R$\+free graded left $\C$\+comodule
$\M$ to an $\R$\+contramodule graded left $\C$\+contramodule $\P$
is constructed as the cokernel of the pair of morphisms
$\Hom^\R(\C\ot^\R\M\;\P)\simeq\Hom^\R(\M,\Hom^\R(\C,\P))\birarrow
\Hom^\R(\M,\P)$ induced by the $\C$\+coaction in $\M$ and
the $\C$\+contraaction in~$\P$.
 The graded $\R$\+contramodule of \emph{cohomomorphisms}
$\Cohom_\C(\cM,\cP)$ from an $\R$\+comodule graded left $\C$\+comodule
$\cM$ to an $\R$\+cofree graded left $\C$\+contramodule $\cP$
is constructed as the cokernel of the pair of morphisms
$\Hom_\R(\C\ocn_\R\cM\;\cP)\simeq\Hom_\R(\cM,\Ctrhom_\R(\C,\cP))
\birarrow\Hom_\R(\cM,\cP)$.

 The \emph{cotensor product} $\cN\oc_\C\cM$ of an $\R$\+comodule
graded right $\C$\+comodule $\cN$ and an $\R$\+comodule graded
left $\C$\+comodule $\cM$ is a graded $\R$\+comodule constructed
as the kernel of the pair of morphisms
$\cN\oc_\R\cM\birarrow(\cN\ocn_\R\C)\oc_\R\cM\simeq
\cN\oc_\R(\C\ocn_\R\cM)$ (see Lemma~\ref{hom-co-associativity}).
 The graded $\R$\+contramodule of \emph{cohomomorphisms}
$\Cohom_\C(\cM,\P)$ from an $\R$\+comodule graded left
$\C$\+comodule $\cM$ to an $\R$\+contramodule graded left
$\C$\+contramodule $\P$ is constructed as the cokernel of the pair
of morphisms $\Cohom_\R(\C\ocn_\R\cM,\P)\simeq
\Cohom_\R(\cM,\Hom^\R(\C,\P))\birarrow\Cohom_\R(\cM,\P)$
induced by the $\C$\+coaction in $\cM$ and the $\C$\+contraaction
in~$\P$.

 The category of $\R$\+contramodule graded $\C$\+contramodules is
abelian and admits infinite products; the forgetful functor
from it to the category of graded $\R$\+con\-tramodules is exact and
preserves infinite products.
 There are enough projective objects in the abelian category of
$\R$\+contramodule graded $\C$\+contramodules; these are the same
as the projective objects in the exact subcategory of $\R$\+free
graded $\C$\+contramodules.
 The latter exact subcategory in the abelian category of
$\R$\+contra\-module graded $\C$\+contramodules is closed under
extensions and infinite products.

 The category of $\R$\+comodule graded $\C$\+comodules is abelian
and admits infinite direct sums; the forgetful functor from it
to the category of $\R$\+comodules is exact and preserves infinite
direct sums.
 There are enough injective objects in the abelian category of
$\R$\+comodule graded $\C$\+comodules; these are the same as
the injective objects in the exact subcategory of $\R$\+cofree
graded $\C$\+contramodules.
 The latter exact subcategory in the abelian category of
$\R$\+comodule graded $\C$\+comodules is closed under extensions
and infinite direct sums.

\begin{rem}
 $\R$\+contramodule $\C$\+comodules and $\R$\+comodule
$\C$\+contramodules can be defined in what is our usual way
in this paper.
 The reason we do not consider these is because the kernels of
morphisms of $\R$\+contramodule $\C$\+comodules and the cokernels
of morphisms of $\R$\+comodule $\C$\+contramodules are not
well-behaved, the functors of contramodule tensor product with
a free $\R$\+contramodule and $\Ctrhom$ from a free
$\R$\+contramodule being not exact (see
Remark~\ref{contra-operations-not-exact}).
 So we do not know of any reason for the categories of
$\R$\+contramodule $\C$\+comodules and $\R$\+comodule 
$\C$\+contramodules to be even abelian.
 $\R$\+contramodule $\C$\+comodules are reasonably behaved,
though, when the pro-Artinian topological ring $\R$ has homological
dimension~$1$ (see Remark~\ref{dim-1-contramodules} and
Section~\ref{discrete-modules}).
\end{rem}

 For any graded $\R$\+contramodule $\V$ and any $\R$\+comodule
graded left $\C$\+comodule $\cM$, the contratensor product
$\V\ocn_\R\cM$ has a natural left $\C$\+comodule structure.
 For any graded $\R$\+comodule $\cV$ and any $\R$\+free graded
left $\C$\+comodule $\M$, the contratensor product
$\M\ocn_\R\cV$ has a natural left $\C$\+comodule structure.
 The same applies to right $\C$\+comodules.
 For any graded $\R$\+comodule $\cU$ and any $\R$\+comodule
graded right $\C$\+comodule $\cN$, the graded $\R$\+contramodule
$\Hom_\R(\cN,\cU)$ has a natural left $\C$\+contramodule
structure provided by the map $\Hom^\R(\C,\Hom_\R(\cN,\cU))
\simeq\Hom_\R(\cN\ocn_\R\C\;\cU)\rarrow\Hom_\R(\cN,\cU)$.
 For any graded $\R$\+contramodule $\U$ and any $\R$\+free
graded right $\C$\+comodule $\N$, the graded $\R$\+contramodule
$\Hom^\R(\N,\U)$ has a natural left $\C$\+contramodule structure.

 The left $\C$\+contramodule $\Hom^\R(\C,\U)$, where $\U$ is
a graded $\R$\+contramodule, is called the left $\C$\+contramodule
\emph{induced} from the graded $\R$\+contramodule~$\U$.
 The left $\C$\+comodule $\C\ocn_\R\cV$, where $\cV$ is a graded
$\R$\+comodule, is called the left $\C$\+comodule \emph{coinduced}
from the graded $\R$\+comodule~$\cV$.

 For any $\R$\+comodule graded right $\C$\+comodule~$\cN$, \
$\R$\+contramodule graded left $\C$\+contramodule $\P$, and 
graded $\R$\+comodule $\cU$ there is a natural isomorphism
of graded $\R$\+contramodules $\Hom_\R(\cN\ocn_\C\P\;\cU)
\simeq\Hom^\C(\P,\Hom_\R(\cN,\cU))$.
 For any $\R$\+free graded right $\C$\+comodule~$\N$, \
$\R$\+contramodule graded left $\C$\+contramodule $\P$, and
graded $\R$\+contramodule $\U$ there is a natural isomorphism
of graded $\R$\+contra\-modules $\Hom^\R(\N\ocn_\C\P\;\U)
\simeq\Hom^\C(\P,\Hom^\R(\N,\U))$.
 For any $\R$\+comodule graded right $\C$\+comodule~$\cN$, \
$\R$\+free graded left $\C$\+comodule $\M$, and graded
$\R$\+comodule $\cU$ such that the cotensor product
$\cN\oc_\C\M$ is a cofree graded $\R$\+comodule, there is
a natural isomorphism of graded $\R$\+contramodules
$\Hom_\R(\cN\oc_\C\M\;\cU)\simeq\Cohom_\C(\M,\Hom_\R(\cN,\cU))$.

 For any $\R$\+comodule graded left $\C$\+comodule $\cL$ and
any graded $\R$\+comodule $\cV$ there is a natural isomorphism
of graded $\R$\+contramodules $\Hom_\C(\cL\;\C\ocn_\R\cV)\simeq
\Hom_\R(\cL,\cV)$.
 For any $\R$\+contramodule graded left $\C$\+contramodule $\Q$
and any graded $\R$\+contramodule $\U$ there is a natural isomorphism
of graded $\R$\+contramodules
$\Hom^\C(\Hom^\R(\C,\U),\Q)\simeq\Hom^\R(\U,\Q)$
\cite[Sections~1.1.1\+-2 and~3.1.1\+-2]{Psemi}.
 For any $\R$\+free graded left $\C$\+comodule $\L$ and any
graded $\R$\+comodule $\cV$ there is a natural isomorphism of
graded $\R$\+comodules $\Hom_\C(\L\;\C\ocn_\R\cV)\simeq
\Ctrhom_\R(\L,\cV)$.
 For any $\R$\+cofree graded left $\C$\+contramodule $\cQ$ and any
graded $\R$\+contramodule $\U$ there is a natural isomorphism of
graded $\R$\+comodules $\Hom^\C(\Hom^\R(\C,\U),\cQ)\simeq
\Ctrhom_\R(\U,\cQ)$.
 For any $\R$\+comodule graded right $\C$\+comodule $\cN$ and
any graded $\R$\+contramodule $\U$, there is a natural
isomorphism of graded $\R$\+comodules
$\cN\ocn_\C\Hom^\R(\C,\U)\simeq\cN\ocn_\R\U$.
 For any $\R$\+free graded right $\C$\+comodule $\N$ and any
graded $\R$\+contramodule $\U$, there is a natural isomorphism
of graded $\R$\+contramodules $\N\ocn_\C\Hom^\R(\C,\U)\simeq\N\ot^\R\U$
\cite[Section~5.1.1]{Psemi}. {\hbadness=1400\par}

 For any $\R$\+free graded right $\C$\+comodule $\N$ and any
graded $\R$\+comodule $\cV$, there is a natural isomorphism of
graded $\R$\+comodules $\N\oc_\C(\C\ocn_\R\cV)\simeq\N\ocn_\R\cV$.
 For any free graded $\R$\+contramodule $\V$ and any
$\R$\+comodule graded left $\C$\+comodule $\cM$, there is
a natural isomorphism of graded $\R$\+comodules
$(\V\ot^\R\C)\oc_\C\cM\simeq\V\ocn_\R\cM$.
 For any free graded $\R$\+contramodule $\V$ and any
$\R$\+contramodule graded left $\C$\+contramodule $\P$, there is
a natural isomorphism of graded $\R$\+contramodules
$\Cohom_\C(\C\ot^\R\V\;\P)\simeq\Hom^\R(\V,\P)$.
 For any $\R$\+free graded left $\C$\+comodule $\M$ and any
graded $\R$\+contramodule $\U$, there is a natural isomorphism
of graded $\R$\+contramodules $\Cohom_\C(\M,\Hom^\R(\C,\U))
\simeq\Hom^\R(\M,\U)$.
 For any graded $\R$\+comodule $\cV$ and any $\R$\+cofree graded
left $\C$\+contramodule $\cP$ there is a natural isomorphism of
graded $\R$\+contramodules $\Cohom_\C(\C\ocn_\R\cV\;\cP)\simeq
\Hom_\R(\cV,\cP)$.
 For any $\R$\+comodule graded left $\C$\+comodule $\cM$ and
any cofree graded $\R$\+comodule $\cU$, there is a natural
isomorphism of graded $\R$\+contramodules $\Cohom_\C(\cM,
\Ctrhom_\R(\C,\cU))\simeq\Hom_\R(\cM,\cU)$
\cite[Sections~1.2.1 and~3.2.1]{Psemi}.
{\hbadness=1200\par}

 For any $\R$\+comodule graded right $\C$\+comodule $\cN$ and
any graded $\R$\+comodule $\cV$, there is a natural isomorphism
of graded $\R$\+comodules $\cN\oc_\C(\C\ocn_\R\cV)\simeq
\cN\oc_\R\cV$.
 For any $\R$\+contramodule graded left $\C$\+contramodule $\P$
and any graded $\R$\+comodule $\cV$, there is a natural
isomorphism of graded $\R$\+contramodules
$\Cohom_\C(\C\ocn_\R\cV\;\P)\simeq\Cohom_\R(\cV,\P)$.
 For any $\R$\+comodule graded left $\C$\+comodule $\cM$ and
any graded $\R$\+contramodule $\U$, there is a natural
isomorphism of graded $\R$\+contramodules
$\Cohom_\C(\cM,\Hom^\R(\C,\U))\simeq\Cohom_\R(\cM,\U)$.

 Given an $\R$\+contramodule graded left $\C$\+contramodule
$\P$, an epimorphism onto it from a projective $\R$\+free
graded $\C$\+contramodule can be obtained as the composition
$\Hom^\R(\C,\U)\rarrow\Hom^\R(\C,\P)\rarrow\P$, where $\U$
is a free graded $\R$\+contramodule and $\U\rarrow\P$ is
an epimorphism of graded $\R$\+contramodules.
 Given an $\R$\+comodule graded left $\C$\+comodule $\cM$,
a monomorphism from it into an injective $\R$\+cofree graded
$\C$\+comodule can be obtained as the composition
$\cM\rarrow\C\ocn_\R\cM\rarrow\C\ocn_\R\cV$, where $\cV$
is a cofree graded $\R$\+comodule and $\cM\rarrow\cV$ is
a monomorphism of graded $\R$\+comodules.

 Let $\C$ and $\D$ be $\R$\+free graded coalgebras.
 An \emph{$\R$\+comodule graded\/ $\D$\+$\C$\+bicomodule} $\cK$
is a graded $\R$\+comodule endowed with commuting graded
left $\D$\+comodule and graded right $\C$\+comodule structures.
 Equivalently, $\cK$ should be endowed with a coassociative
and counital homogeneous \emph{$\D$\+$\C$\+bicoaction} map
$\cK\rarrow\D\ocn_\R\cK\ocn_\R\C$.

 Given an $\R$\+comodule graded $\D$\+$\C$\+bicomodule $\cK$
and an $\R$\+contramodule graded left $\C$\+contramodule $\P$,
the contratensor product $\cK\ocn_\C\P$ is endowed with
a graded left $\D$\+comodule structure as the cokernel of a pair
of morphisms of $\R$\+comodule graded left $\D$\+comodules.
 Given an $\R$\+comodule graded $\D$\+$\C$\+bicomodule $\cK$
and an $\R$\+comodule graded left $\D$\+comodule $\cM$,
the graded $\R$\+contramodule $\Hom_\D(\cK,\cM)$ is endowed with
a graded left $\C$\+contramodule structure as the kernel of a pair
of morphisms of $\R$\+contramodule graded left $\C$\+contramodules.

\begin{lem} \hbadness=1125 \label{r-c-co-contra-adjunction}
 Let\/ $\cK$ be an\/ $\R$\+comodule graded\/ $\D$\+$\C$\+bicomodule.
 Then the functor taking an\/ $\R$\+contramodule graded left\/
$\C$\+contramodule\/ $\P$ to the\/ $\R$\+comodule graded left\/
$\D$\+comodule\/ $\cK\ocn_\C\P$ is naturally left adjoint to
the functor taking an\/ $\R$\+comodule graded left\/
$\D$\+comodule\/ $\cM$ to the\/ $\R$\+contramodule graded left\/
$\C$\+contramodule\/ $\Hom_\D(\cK,\cM)$.
 Moreover, there is a natural isomorphism of graded\/
$\R$\+contramodules\/ $\Hom_\D(\cK\ocn_\C\P\;\cM)\simeq
\Hom^\C(\P,\Hom_\D(\cK,\cM))$.
\end{lem}

\begin{proof}
 According to Section~\ref{hom-operations}, there is a natural
isomorphism of $\R$\+contramodules $\Hom_\R(\cK\ocn_\R\P\;\cM)
\simeq\Hom^\R(\P,\Hom_\R(\cK,\cM))$.
 It is straightforward to check that a graded $\R$\+comodule
map $\cK\ocn_\R\P\rarrow\cM$ factorizes through $\cK\ocn_\C\P$
if and only if the corresponding graded $\R$\+contramodule map
$\P\rarrow\Hom_\R(\cK,\cM)$ is a $\C$\+contramodule morphism,
and a graded $\R$\+contramodule map $\P\rarrow\Hom_\R(\cK,\cM)$
factorizes through $\Hom_\D(\cK,\cM)$ if and only if
the corresponding graded $\R$\+comodule map $\cK\ocn_\R\P
\rarrow\cM$ is a $\D$\+comodule morphism.
 (Cf.\ \cite[Section~5.1.2]{Psemi}.)
\end{proof}

 Set $\cC(\R,\C)=\cC(\R)\ocn_\R\C$.
 Obviously, the graded $\R$\+comodule $\cC(\R,\C)$ has a natural
$\R$\+comodule graded $\C$\+bicomodule structure.
 For any $\R$\+contramodule graded left $\C$\+contramodule $\P$
and $\R$\+comodule graded left $\C$\+comodule $\cM$, set
$\Phi_{\R,\C}(\P)=\cC(\R,\C)\ocn_\C\P$ and
$\Psi_{\R,\C}(\cM)=\Hom_\C(\cC(\R,\C)\;\cM)$.
 According to Lemma~\ref{r-c-co-contra-adjunction}, $\,\Phi_{\R,\C}$
and $\Psi_{\R,\C}$ are adjoint functors between the abelian
categories of $\R$\+contramodule graded left $\C$\+contramodules
and $\R$\+comodule graded left $\C$\+comodules.

 One easily checks that for any projective $\R$\+free graded
left $\C$\+contramodule $\F$ there is a natural isomorphism of
$\R$\+comodule graded $\C$\+comodules $\Phi_\R\Phi_\C(\F)\simeq
\Phi_{\R,\C}(\F)$.
 Similarly, for any injective $\R$\+cofree graded left
$\C$\+comodule $\cJ$ there is a natural isomorphism of
$\R$\+contramodule graded $\C$\+contramodules $\Psi_\R\Psi_\C(\cJ)
\simeq\Psi_{\R,\C}(\cJ)$ (see Sections~\ref{r-free-graded-co}
and~\ref{r-cofree-graded-co} for the definitions of
$\Phi_\R=\Psi_\R^{-1}$ and $\Phi_\C=\Psi_\C^{-1}$).
 So the functors $\Phi_{\R,\C}$ and $\Psi_{\R,\C}$ restrict to
mutually inverse equivalences between the additive categories of
projective $\R$\+free left $\C$\+contramodules and injective
$\R$\+cofree left $\C$\+comodules.

 The equivalence $\Phi_\R=\Psi_\R^{-1}$ between the exact categories
of $\R$\+free and $\R$\+cofree graded right $\C$\+comodules is
constructed in the way similar to the case of left $\C$\+comodules
(see Proposition~\ref{r-cofree-co-r-co-contra}(b)).
 For any $\R$\+free graded right $\C$\+comodule $\N$ and any
$\R$\+contramodule graded left $\C$\+contramodule $\P$ there is
a natural isomorphism of graded $\R$\+comodules
$\Phi_\R(\N)\ocn_\C\P\simeq\Phi_\R(\N\ocn_\C\P)$.

 For any $\R$\+free graded right $\C$\+comodule $\N$ and any
$\R$\+comodule graded left $\C$\+comodule $\cM$, there is
a natural isomorphism of graded $\R$\+comodules
$\Phi_\R(\N)\oc_\C\cM\simeq\N\oc_\C\cM$.
 For any $\R$\+free graded left $\C$\+comodule $\M$ and any
$\R$\+contramodule graded left $\C$\+contramodule $\P$,
there is a natural isomorphism of graded $\R$\+contramodules
$\Cohom_\C(\Phi_\R(\M),\P)\simeq\Cohom_\C(\M,\P)$.
 For any $\R$\+comodule graded left $\C$\+comodule $\cM$ and any
$\R$\+cofree graded left $\C$\+contramodule $\cP$, there is
a natural isomorphsm of graded $\R$\+contramodules
$\Cohom_\C(\cM,\Psi_\R(\cP))\simeq\Cohom_\C(\cM,\cP)$.
{\emergencystretch=1em\par}

\subsection{Contra/coderived category of CDG-contra/comodules}
\label{non-adj-co-derived}
 The definitions of
\begin{itemize}
\item odd contra- and coderivations of $\R$\+contramodule
      $\C$\+contramodules and $\R$\+co\-module $\C$\+comodules
      compatible with a given odd derivation of an $\R$\+free
      graded coalgebra~$\C$,
\item $\R$\+contramodule CDG\+contramodules and $\R$\+comodule
      CDG\+comodules over an $\R$\+free CDG\+coalgebra~$\C$,
\item the complexes of $\R$\+contramodules $\Hom^\C(\P,\Q)$ for
      given $\R$\+contramodule left CDG\+contramodules $\P$
      and $\Q$ over~$\C$,
\item the complexes of $\R$\+contramodules $\Hom_\C(\cL,\cM)$
      for given $\R$\+comodule left CDG\+comodules $\cL$ and
      $\cM$ over~$\C$,
\item the complex of $\R$\+comodules $\Hom^\C(\P,\cQ)$ for
      an $\R$\+contramodule left CDG\+contramodule $\P$
      and $\R$\+cofree left CDG\+contramodule $\cQ$ over~$\C$,
{\hbadness=1550
\item the complex} of $\R$\+comodules $\Hom_\C(\L,\cM)$ for
      an $\R$\+free left CDG\+comodule $\L$ and
      $\R$\+comodule left CDG\+comodule $\cM$ over~$\C$,
\item the complex of $\R$\+comodules $\cN\ocn_\C\P$ for
      an $\R$\+comodule right CDG\+comodule $\cN$ and
      $\R$\+contramodule left CDG\+contramodule $\P$ over~$\C$,
\item the complex of $\R$\+contramodules $\N\ocn_\C\P$ for
      an $\R$\+free right CDG\+comodule $\N$ and $\R$\+contramodule
      left CDG\+contramodule $\P$ over~$\C$,
\item the complex of $\R$\+comodules $\N\oc_\C\cM$ for 
      an $\R$\+free right CDG\+comodule $\N$ and $\R$\+comodule
      left CDG\+comodule $\cM$ over~$\C$,
\item the complex of $\R$\+comodules $\cN\oc_\C\cM$ for
      an $\R$\+comodule right CDG\+comodule $\cN$ and
      $\R$\+comodule left CDG\+comodule $\cM$ over~$\C$,
\item the complex of $\R$\+contramodules $\Cohom_\C(\M,\P)$
      for an $\R$\+free left CDG\+co\-module $\M$ and
      $\R$\+contramodule left CDG\+contramodule $\P$ over~$\C$,
\item the complex of $\R$\+contramodules $\Cohom_\C(\cM,\cP)$
      for an $\R$\+comodule left CDG\+comodule $\cM$ and
      $\R$\+cofree left CDG\+contramodule $\cP$ over~$\C$,
\item the complex of $\R$\+contramodules $\Cohom_\C(\cM,\P)$
      for an $\R$\+comodule left CDG\+comodule $\cM$ and
      $\R$\+contramodule left CDG\+contramodule $\P$ over~$\C$,
\item the $\R$\+contramodule CDG\+contramodules or $\R$\+comodule
      CDG\+comodules over $\D$ obtained by contra- or corestriction
      of scalars via a morphism of $\R$\+free CDG\+coalgebras
      $\C\rarrow\D$ from $\R$\+contramodule CDG\+contramodules
      or $\R$\+comodule CDG\+comodules over~$\C$
\end{itemize}
repeat the similar definition for $\R$\+free CDG\+contramodules
and $\R$\+cofree CDG\+co\-modules given in
Sections~\ref{r-free-co-derived} and~\ref{r-cofree-co-derived}
\emph{verbatim} (with the definitions and constructions of
Section~\ref{non-adj-graded-co} being used in place of those from
Sections~\ref{r-free-graded-co} and~\ref{r-cofree-graded-co}
as applicable), so there is no need to spell them out here again.
 We restrict ourselves to introducing the new notation for our
new classes of objects.

 The DG\+categories of $\R$\+comodule left and right CDG\+comodules
over $\C$ are denoted by $\C\comod\Rco$ and $\comodrRco\C$,
and their homotopy categories are
$H^0(\C\comod\Rco)$ and
$H^0(\comodrRco\C)$.
 Similarly, the DG\+category of $\R$\+contra\-module left
CDG\+contramodules over $\C$ is denoted by $\C\contra\Rctr$,
and its homotopy category is $H^0(\C\contra\Rctr)$.
 The $\Hom$ from $\R$\+contramodule to $\R$\+cofree CDG\+contramodules
over $\C$ is a triangulated functor of two arguments
$$
 \Hom^\C\: H^0(\C\contra\Rctr)^\sop\times H^0(\C\contra\Rcof)
 \lrarrow H^0(\R\comod).
$$
 The $\Hom$ from $\R$\+free to $\R$\+comodule CDG\+comodules over $\C$
is a triangulated functor of two arguments
$$
 \Hom_\C\: H^0(\C\comod\Rfr)^\sop\times H^0(\C\comod\Rco)
 \lrarrow H^0(\R\comod).
$$
 The contratensor product of $\R$\+comodule right CDG\+comodules
and $\R$\+contramodule left CDG\+contramodules over $\C$ is
a triangulated functor of two arguments
$$
 \ocn_\C\: H^0(\comodrRco\C)\times H^0(\C\contra\Rctr)\lrarrow
 H^0(\R\comod).
$$
 The contratensor product of $\R$\+free right CDG\+comodules and
$\R$\+contramodule left CDG\+contramodules over $\C$ is
a triangulated functor of two arguments
$$
 \ocn_\C\: H^0(\comodrRfr\C)\times H^0(\C\contra\Rctr)\lrarrow
 H^0(\R\contra).
$$
 The cotensor product of $\R$\+free right CDG\+comodules and
$\R$\+comodule left CDG\+co\-modules over $\C$ is a triangulated
functor of two arguments
$$
 \oc_\C\: H^0(\comodrRfr\C)\times H^0(\C\comod\Rco)\lrarrow
 H^0(\R\comod).
$$
 The cotensor product of $\R$\+comodule left and right CDG\+comodules
over $\C$ is a triangulated functor of two arguments
$$
 \oc_\C\: H^0(\comodrRco\C)\times H^0(\C\comod\Rco)\lrarrow
 H^0(\R\comod).
$$
 The $\Cohom$ from $\R$\+free left CDG\+comodules to
$\R$\+contramodule left CDG\+contra\-modules over $\C$ is
a triangulated functor of two arguments
$$
 \Cohom_\C\: H^0(\C\comod\Rfr)^\sop\times H^0(\C\contra\Rctr)
 \lrarrow H^0(\R\contra).
$$
 The $\Cohom$ from $\R$\+comodule left CDG\+comodules to
$\R$\+cofree left CDG\+contra\-modules over $\C$ is
a triangulated functor of two arguments
$$
 \Cohom_\C\: H^0(\C\comod\Rco)^\sop\times H^0(\C\contra\Rcof)
 \lrarrow H^0(\R\contra).
$$
 The $\Cohom$ from $\R$\+comodule left CDG\+comodules to
$\R$\+contramodule left CDG\+con\-tramodules over $\C$ is
a triangulated functor of two arguments
$$
 \Cohom_\C\: H^0(\C\comod\Rco)^\sop\times H^0(\C\contra\Rctr)
 \lrarrow H^0(\R\contra).
$$
 Given a morphism of $\R$\+free CDG\+coalgebras $f=(f,a)\:
\C\rarrow\D$, the functors of contra- and corestriction of
scalars are denoted by
$$
 R^f\: H^0(\C\contra\Rctr)\lrarrow H^0(\D\contra\Rctr)
$$
and
$$
 R_f\: H^0(\C\comod\Rco)\lrarrow H^0(\D\comod\Rco).
$$

 An $\R$\+contramodule left CDG\+contramodule over $\C$ is called
\emph{contraacyclic} if it belongs to the minimal triangulated
subcategory of the homotopy category $H^0(\C\contra\Rctr)$
containing the totalizations of short exact sequences of
$\R$\+contra\-module CDG\+contramodules over $\C$ and closed under
infinite products.
 The quotient category of $H^0(\C\contra\Rctr)$ by the thick
subcategory of contraacyclic $\R$\+contra\-module CDG\+contramodules
is called the \emph{contraderived category} of $\R$\+contramodule
left CDG\+contramodules over $\C$ and denoted by
$\sD^\ctr(\C\contra\Rctr)$.

 An $\R$\+comodule left CDG\+comodule over $\C$ is called
\emph{coacyclic} if it belongs to the minimal triangulated
subcategory of the homotopy category $H^0(\C\comod\Rco)$
containing the totalizations of short exact sequences of
$\R$\+comodule CDG\+comodules over $\C$ and closed under infinite
direct sums.
 The quotient category of $H^0(\C\comod\Rco)$ by the thick
subcategory of coacyclic $\R$\+comodule CDG\+comodules is called
the \emph{coderived category} of $\R$\+comodule left
CDG\+comodules over $\C$ and denoted by $\sD^\co(\C\comod\Rco)$.
 The coderived category of $\R$\+comodule right CDG\+comodules
over $\C$, denoted by $\sD^\co(\comodrRco\C)$, is defined similarly.

 It follows from the next theorem, among other things, that our
terminology is not ambigous: an $\R$\+free CDG\+contramodule over
$\C$ is contraacyclic in the sense of Section~\ref{r-free-co-derived}
if and only if it is contraacyclic as an $\R$\+contramodule
CDG\+contramodule, in the sense of the above definition.
 Similarly, an $\R$\+cofree CDG\+comodule over $\C$ is coacyclic
in the sense of Section~\ref{r-cofree-co-derived} if and only if
it is coacyclic as an $\R$\+comodule CDG\+comodule, in
the sense of the above definition.

\begin{thm} \label{non-adj-co-derived-comp}
 For any\/ $\R$\+free CDG\+coalgebra\/ $\C$, the functors
$$
 \sD^\ctr(\C\contra\Rfr)\lrarrow\sD^\ctr(\C\contra\Rctr)
$$
and
$$
 \sD^\co(\C\comod\Rcof)\lrarrow\sD^\co(\C\comod\Rco)
$$
induced by the natural embeddings of DG\+categories\/ $\C\contra\Rfr
\rarrow\C\contra\Rctr$ and\/ $\C\comod\Rcof\rarrow\C\comod\Rco$
are equivalences of triangulated categories.
\end{thm}

\begin{proof}
 The argument from the proof of Theorem~\ref{co-derived-mod} is
applicable, but due to the fact that there are always enough
projective resolutions for the contraderived category of
CDG\+contramodules (and injective resolutions for the coderived
category of CDG\+comodules), there is a much simpler proof.
 The desired assertions follow immediately from the next
Theorem~\ref{non-adj-co-derived-res} together with
Theorems~\ref{r-free-co-derived-thm}(c)
and~\ref{r-cofree-co-derived-thm}(d).
\end{proof}

\begin{thm} \label{non-adj-co-derived-res}
 Let\/ $\C$ be an\/ $\R$\+free CDG\+coalgebra.  Then \par
\textup{(a)} for any CDG\+contramodule\/
$\P\in H^0(\C\contra\Rfr_\proj)$ and any contraacyclic\/
$\R$\+contramodule left CDG\+contramodule\/ $\Q$ over\/ $\C$,
the complex of\/ $\R$\+contramodules\/ $\Hom^\C(\P,\Q)$
is contraacyclic; \par
\textup{(b)} for any coacyclic\/ $\R$\+comodule left CDG\+comodule\/
$\cL$ over\/ $\C$ and any CDG\+co\-module\/
$\cM\in H^0(\C\comod\Rcof_\inj)$, the complex of\/
$\R$\+contramodules\/ $\Hom_\C(\cL,\cM)$ is contraacyclic; \par
\textup{(c)} the composition of natural functors 
$$
 H^0(\C\contra\Rfr_\proj)\lrarrow H^0(\C\contra\Rctr)
 \lrarrow\sD^\ctr(\C\contra\Rctr)
$$
is an equivalence of triangulated categories; \par
\textup{(d)} the composition of natural functors
$$
 H^0(\C\comod\Rcof_\inj)\lrarrow H^0(\C\comod\Rco)\lrarrow
 \sD^\co(\C\comod\Rco)
$$
is an equivalence of triangulated categories.
\end{thm}

\begin{proof}
 Part~(a) holds, since the functor $\Hom^\C(\P,{-})$ takes
short exact sequences and infinite products of $\R$\+contramodule
CDG\+contramodules to short exact sequences and infinite products
of complexes of $\R$\+contramodules.
 The proof of part~(b) is similar up to duality
(cf.\ \cite[Theorem~4.4, Theorem~3.5 and Remark~3.5]{Pkoszul}).

 To prove part~(c), we will need to use the construction of
the CDG\+contramodule $G^+(\F)$ over $\C$ freely generated by
a given $\R$\+contramodule graded left $\C$\+contramodule $\F$
(cf.\ the proof of Theorem~\ref{r-free-absolute-derived}).
 The graded $\R$\+contramodule $G^+(\F)$ is defined by the rule
$G^+(\F)^i=\F^i\oplus\F^{i-1}$; the elements of $G^+(\F)^i$
are denoted formally by $x+d_Gy$, where $x\in\F^i$ and $y\in\F^{i-1}$.

 To define the left contraaction of $\C$ in $G^+(\F)$,
present an arbitrary element of degree~$i$ component
$\Hom^{\R,i}(\C\;G^+(\F))$ of the graded $\R$\+contramodule
$\Hom^\R(\C\;G^+(\F))$ in the form $c\mpsto f(c) + d_G(g(c))$,
where $f\in\Hom^{\R,i}(\C,\F)$ and $g\in\Hom^{\R,\,i-1}(\C,\F)$.
 Set $\pi_{G^+(\F)}(f)=\pi_\F(f)$ and $\pi_{G^+(\F)}(d_G\circ g)=
d(\pi_\F(g)) + (-1)^{|g|}\pi_\F(g\circ d_\C)$, where
$\pi_\F\:\Hom^\R(\C,\F)\rarrow\F$ and $\pi_{G^+(\F)}\:
\Hom^\R(\C\;G^+(\F))\rarrow G^+(\F)$ denote
the $\C$\+contraaction map and $d_\C\:\C\rarrow\C$ is
the differential in~$\C$.
 Finally, define the differential in $G^+(\F)$ by the rule
$d(x+d_G(y)) = h*y + d_G(x)$, where $h\:\C^{-2}\rarrow\R$
is the curvature linear function of~$\C$.

 There is a natural morphism of graded $\C$\+contramodules
$\F\rarrow G^+(\F)$ defined by the rule $x\mpsto x = x + d_G(0)$.
 The cokernel of this morphism is isomorphic to $\F[-1]$ as
an $\R$\+contra\-module graded left $\C$\+contramodule.
 Any morphism of graded $\C$\+contramodules $p\:\F\rarrow\P$
from $\F$ to an $\R$\+contramodule CDG\+contramodule $\P$ over $\C$
factorizes uniquely as the composition of the natural embedding
$\F\rarrow G^+(\F)$ and a closed morphism of CDG\+contramodules
$G^+(\F)\rarrow\P$; the latter is given by the rule $x+d_Gy \mpsto
p(x) + d_\P(p(y))$, where $d_\P$~is the differential in~$\P$.

 Now let $\P$ be an $\R$\+contramodule left CDG\+contramodule
over~$\C$.
 As explained in Section~\ref{non-adj-graded-co}, there exists
a surjective morphism onto the $\R$\+contramodule graded
$\C$\+contramodule $\P$ from a projective $\R$\+free graded
$\C$\+contramodule~$\F_0$.
 The above construction then provides a surjective closed morphism
$G^+(\F_0)\rarrow\P$ onto $\P$ from a CDG\+contramodule
$G^+(\F_0)\in H^0(\C\contra\Rfr_\proj)$.
 Applying the same procedure to the kernel of the latter morphism,
etc., we obtain a left resolution of $\P$ by
$\R$\+free CDG\+contramodules with projective underlying
graded $\C$\+contramodules $\dotsb\rarrow G^+(\F_1)\rarrow
G^+(\F_0)\rarrow\P\rarrow0$.

 Totalizing this complex of CDG\+contramodules from
$H^0(\C\contra\Rfr_\proj)$ by taking infinite products along 
the diagonals, we get a closed morphism of $\R$\+contramodule
CDG\+contramodules $\F\rarrow\P$ with
$\F\in H^0(\C\contra\Rfr_\proj)$ (cf.\ the proof of
Theorem~\ref{r-free-co-derived-thm}).
 It remains to use the $\C$\+contramodule analogue of
Lemma~\ref{bounded-above-lem}.

 We have proven part~(c).
 The proof of part~(d) is analogous up to the duality, so we
restrict ourselves to writing down the construction of
the $\R$\+comodule left CDG\+comodule $G^-(\L)$ cofreely
cogenerated by an $\R$\+comodule graded left $\C$\+comod\-ule~$\cL$.
 The graded $\R$\+comodule $G^-(\cL)$ is defined by the rule
$G^-(\cL)^i=\cL^{i+1}\oplus\cL^i$; the elements of $G^+(\cL)^i$
are denoted formally by $d_G^{-1}x+py$, where $x\in\cL^{i+1}$
and $y\in\cL^i$.

 We will need to use Sweedler's notation for the coaction maps:
given an $\R$\+comod\-ule left $\C$\+comodule $\cM$, the left
$\C$\+coaction map $\cM\rarrow\C\ocn_\R\cM$ is denoted by
by $z\mpsto z_{(-1)}\ocn z_{(0)}$, where $z$, $z_{(0)}\in\cM$,
\ $z_{(-1)}\in\C$, and $c\ocn z$ is a symbolic notation for
an element of $\C\ocn_\R\cM$.
 Then the $\C$\+coaction in $G^+(\cL)$ is expressed in terms of
the $\C$\+coaction in $\cL$ and the differential~$d_\C$ in $\C$
by the rules $d_G^{-1}(x)\mpsto (-1)^{|x_{(-1)}|}x_{(-1)}
\ocn d_G^{-1}(x_{(0)})$ and $py\mpsto (-1)^{|y_{(-1)}|}d_\C(y_{(-1)})
\ocn d_G^{-1}(y_{(0)}) + y_{(-1)}\ocn p(y_{(0)})$.
 The differential in $G^-(\cL)$ is
$d(d_G^{-1}x+py) = d_G^{-1}(h*y) + px$.

 There is a natural morphism of graded $\C$\+comodules
$G^-(\cL)\rarrow\cL$ defined by the rule $d_G^{-1}x+py\mpsto y$.
 The kernel of this morphism is isomorphic to $\cL[1]$ as
an $\R$\+comodule graded left $\C$\+comodule.
 Any morphism of graded $\C$\+comodules $f\:\cM\rarrow\cL$ from
an $\R$\+comodule left CDG\+comodule $\cM$ over $\C$ to $\cL$
factorizes uniquely as the composition of a closed morphism
of CDG\+comodules $\cM\rarrow G^-(\cL)$ and the natural
surjection $G^-(\cL)\rarrow\cL$; the former is given by
the formula $z\mpsto d_G^{-1}(f(d_\cM(z))) + p(f(z))$
(cf.\ \cite[proof of Theorem~3.6]{Pkoszul}).
\end{proof}

 Given an $\R$\+contramodule left CDG\+contramodule $\P$ over $\C$,
the $\R$\+comodule graded left $\C$\+comodule $\Phi_{\R,\C}(\P) =
\cC(\R,\C)\ocn_\C\P$ is endowed with a CDG\+comodule structure with
the conventional tensor product differential (where the differential
on $\cC(\R,\C)=\C\ocn_\R\cC(\R)$ is induced by the differential
on~$\C$).
 Similarly, given an $\R$\+comodule left CDG\+comodule $\cM$ over
$\C$, the $\R$\+contramodule graded left $\C$\+contramodule
$\Psi_{\R,\C}(\cM)=\Hom_\C(\cC(\R,\C)\;\cM)$ is endowed with
a CDG\+contra\-module structure with the conventional
$\Hom$ differential.

 For any CDG\+contramodule $\F\in\C\contra\Rfr_\proj$,
the natural isomorphism $\Phi_{\R,\C}(\F)\simeq\Phi_\R\Phi_\C(\F)$
from Section~\ref{non-adj-graded-co} agrees with the CDG\+comodule
structures on both sides.
 Similarly, for any CDG\+comodule $\cJ\in\C\comod\Rcof_\inj$,
the natural isomorphism $\Psi_{\R,\C}(\cJ)\simeq\Psi_\R\Psi_\C(\cJ)$
agrees with the CDG\+contramodule structures on both sides
(see Sections~\ref{r-free-co-derived} and~\ref{r-cofree-co-derived}
for the definitions).

\begin{cor} \label{non-adj-derived-co-contra}
 The derived functors
$$
 \boL\Phi_{\R,\C}\:\sD^\ctr(\C\contra\Rctr)\lrarrow\sD^\co(\C\comod\Rco)
$$
and
$$
 \boR\Psi_{\R,\C}\:\sD^\co(\C\comod\Rco)\lrarrow\sD^\ctr(\C\contra\Rctr)
$$
defined by identifying\/ $\sD^\ctr(\C\allowbreak\contra\Rctr)$ with
$H^0(\C\contra\Rfr_\proj)$ and\/ $\sD^\co(\C\comod\Rco)$ with
$H^0(\C\comod\Rcof_\inj)$ are mutually inverse equivalences
between the contraderived category\/ $\sD^\ctr(\C\contra\Rctr)$
and the coderived category\/ $\sD^\co(\C\comod\Rco)$. \qed
\hfuzz=7.4pt
\end{cor}

 Using the equivalence of triangulated categories from
Theorem~\ref{non-adj-co-derived-comp} and the construction of
the derived functor $\Ext^\C$ from Section~\ref{r-free-co-derived},
we obtain the right derived functor
$$
 \Ext^\C\:\sD^\ctr(\C\contra\Rctr)^\sop\times\sD^\ctr(\C\contra\Rctr)
 \lrarrow H^0(\R\contra^\free).
$$
 The same functor can be constructed by restricting the functor
$\Hom^\C$ to the full subcategory $H^0(\C\contra\Rfr_\proj)^\sop
\times H^0(\C\contra\Rctr)\subset H^0(\C\contra\Rctr)^\sop\times
H^0(\C\contra\Rctr)$, composing it with the localization functor
$H^0(\R\contra)\rarrow\sD^\ctr(\R\contra)$, and identifying
the contraderived category $\sD^\ctr(\R\contra)$ with
the homotopy category $H^0(\R\contra^\free)$.

 Similarly, from Theorem~\ref{non-adj-co-derived-comp} and
the construction of the derived functor $\Ext_\C$
in Section~\ref{r-cofree-co-derived} we obtain the right derived
functor
$$
 \Ext_\C\:\sD^\co(\C\comod\Rco)^\sop\times\sD^\co(\C\comod\Rco)
 \lrarrow H^0(\R\contra^\free).
$$
 The same functor can be constructed by restricting the functor
$\Hom_\C$ to the full subcategory $H^0(\C\comod\Rco)^\sop\times
H^0(\C\comod\Rcof_\inj)\subset H^0(\C\comod\Rco)^\sop\times
H^0(\C\comod\Rco)$, composing it with the localization functor
$H^0(\R\contra)\rarrow\sD^\ctr(\R\contra)$, and identifying
$\sD^\ctr(\R\contra)$ with $H^0(\R\contra^\free)$.

 Restricting the functor $\ocn_\C$ to the full subcategory
$H^0(\comodrRco\C)\times H^0(\C\contra\Rfr_\proj)\allowbreak\subset
H^0(\comodrRco\C)\times H^0(\C\contra\Rctr)$, composing it with
the localization functor $H^0(\R\comod)\rarrow\sD^\co(\R\comod)$,
and identifying the coderived category $\sD^\co(\R\comod)$ with
the homotopy category $H^0(\R\comod^\cofr)$, we obtain
the left derived functor {\hfuzz=18.2pt
$$
 \Ctrtor^\C\:\sD^\co(\comodrRco\C)\times \sD^\ctr(\C\contra\Rctr)
 \lrarrow H^0(\R\comod^\cofr).
$$
 The} functor obtained by restriction and composition factorizes
through the Cartesian product of coderived and contraderived
categories by Theorem~\ref{non-adj-co-derived-res}(c) and
because the contratensor product with a projective $\R$\+free
CDG\+contramodule preserves exact triples and infinite direct
sums of $\R$\+comodule right CDG\+comodules.
 The equivalences of categories from
Theorem~\ref{non-adj-co-derived-comp} together with the equivalences
of categories $\Phi_\R=\Psi_\R^{-1}$ transform
the derived functor $\Ctrtor^\C$ that we have constructed into
the functors $\Ctrtor^\C$ from Sections~\ref{r-free-co-derived}
and~\ref{r-cofree-co-derived}.

 Restricting the functor $\oc_\C$ to either of the full
subcategories $H^0(\comodrRco\C)\times H^0(\C\comod\Rcof_\inj)$
or $H^0(\comodrRcofinj\C)\times H^0(\C\comod\Rco)\subset
H^0(\comodrRco\C)\times H^0(\C\comod\Rcof_\inj)$, composing
it with the localization functor $H^0(\R\comod)\rarrow
\sD^\co(\R\comod)$, and identifying $\sD^\co(\R\comod)$ with
$H^0(\R\comod^\cofr)$, we obtain the right derived functor
$$
 \Cotor^\C\:\sD^\co(\comodrRco\C)\times\sD^\co(\C\comod\Rco)
 \lrarrow H^0(\R\comod^\cofr).
$$
 Here $\comodrRcofinj\C$ denotes the DG\+category of $\R$\+cofree
right CDG\+comodules over $\C$ with injective underlying
$\R$\+cofree graded $\C$\+comodules, and $H^0(\comodrRcofinj\C)$
is the corresponding homotopy category.
 From Theorem~\ref{non-adj-co-derived-comp} and the construction
of the derived functor $\Cotor^\C$ in Section~\ref{r-cofree-co-derived}
we obtain a triangulated functor
$$
 \Cotor^\C\:\sD^\co(\comodrRfr\C)\times\sD^\co(\C\comod\Rco)
 \lrarrow H^0(\R\comod^\cofr).
$$
 The latter functor can be also obtained by restricting
the functor $\oc_\C$ to the full subcategory
$H^0(\comodrRfrinj\C)\times H^0(\C\comod\Rco)\subset
H^0(\comodrRfr\C)\times H^0(\C\comod\Rco)$, composing it with
the localization functor $H^0(\R\comod)\rarrow \sD^\co(\R\comod)$,
and identifying $\sD^\co(\R\comod)$ with $H^0(\R\comod^\cofr)$.
 The equivalences of categories from
Theorem~\ref{non-adj-co-derived-comp} together with the equivalences
of categories $\Phi_\R=\Psi_\R^{-1}$ transform
the above two derived functors $\Ctrtor^\C$ into each other 
and the functor $\Cotor^\C$ from Sections~\ref{r-free-co-derived}.

 Restricting the functor $\Cohom_\C$ to either of the full
subcategories $H^0(\C\comod\Rcof_\inj)^\sop\allowbreak\times
H^0(\C\contra\Rctr)$ or $H^0(\C\comod\Rco)^\sop\times
H^0(\C\contra\Rfr_\proj)\subset H^0(\C\comod\Rco)^\sop\allowbreak
\times H^0(\C\contra\Rctr)$, composing it with the localization
functor $H^0(\R\contra)\rarrow\sD^\ctr(\R\contra)$, and
identifying $\sD^\ctr(\R\contra)$ with $H^0(\R\contra^\free)$,
we obtain the left derived functor {\hfuzz=12.5pt
$$
 \Coext_\C\:\sD^\co(\C\comod\Rco)^\sop\times\sD^\ctr(\C\contra\Rctr)
 \lrarrow H^0(\R\contra^\free).
$$
 From} Theorem~\ref{non-adj-co-derived-comp} and the constructions
of the derived functors $\Coext_\C$ in Sections~\ref{r-free-co-derived}
and~\ref{r-cofree-co-derived} we obtain triangulated functors
\begin{alignat*}{3}
 &\Coext_\C\:\sD^\co(\C\comod\Rfr)^\sop&&\times\sD^\ctr(\C\contra\Rctr)
 &&\lrarrow H^0(\R\contra^\free), \\
 &\Coext_\C\:\sD^\co(\C\comod\Rco)^\sop&&\times\sD^\ctr(\C\contra\Rcof)
 &&\lrarrow H^0(\R\contra^\free).
\end{alignat*}
 The former of these two functors can be also obtained by restricting
the functor $\Cohom_\C$ to the full subcategory
$H^0(\C\comod\Rfr_\inj)^\sop\times H^0(\C\contra\Rctr)\subset
H^0(\C\comod\Rfr)^\sop\times H^0(\C\contra\Rctr)$, composing it
with the localization functor $H^0(\R\contra)\rarrow
\sD^\ctr(\R\contra)$, and identifying $\sD^\ctr(\R\contra)$ with
$H^0(\R\contra^\free)$.
 The latter derived functor can be similarly obtained by restricting
the functor $\Cohom_\C$ to the full subcategory
$H^0(\C\comod\Rco)^\sop\times H^0(\C\contra\Rcof_\proj)\subset
H^0(\C\comod\Rco)^\sop\times H^0(\C\contra\Rcof)$.
 The equivalences of categories from
Theorem~\ref{non-adj-co-derived-comp} together with the equivalences
of categories $\Phi_\R=\Psi_\R^{-1}$ transform
these three derived functors $\Coext_\C$ into each other 
and the derived functor $\Coext_\C$ of cohomomorphisms from
$\R$\+free CDG\+comodules to $\R$\+cofree CDG\+comodules
from Section~\ref{r-cofree-co-derived} (taking values in
$H^0(\R\comod^\cofr)$). {\hbadness=4300\par}

\begin{prop}
\textup{(a)} The equivalence of triangulated categories
$$
 \boL\Phi_{\R,\C}\:\sD^\ctr(\C\contra\Rctr)
 \simeq\sD^\co(\C\comod\Rco)\,\,\:\!\boR\Psi_{\R,\C}
$$
from Corollary~\textup{\ref{non-adj-derived-co-contra}} transforms
the left derived functor
$$
 \Coext_\C\:\sD^\co(\C\comod\Rco)^\sop\times\sD^\ctr(\C\contra\Rctr)
 \lrarrow H^0(\R\contra^\free)
$$
into the right derived functors
$$
 \Ext^\C\:\sD^\ctr(\C\contra\Rctr)^\sop\times\sD^\ctr(\C\contra\Rctr)
 \lrarrow H^0(\R\contra^\free)
$$
and
$$
 \Ext_\C\:\sD^\co(\C\comod\Rco)^\sop\times\sD^\co(\C\comod\Rco)
 \lrarrow H^0(\R\contra^\free).
$$
 In other words, the following diagram of categories, functors,
and equivalences is commutative:
$$
\begin{diagram}
\node{\sD^\ctr(\C\contra\Rctr)^\sop\times\sD^\ctr(\C\contra\Rctr)}
\arrow{s,=}\arrow[4]{e,t}{\Ext^\C}\node[4]{H^0(\R\contra^\free)}
\arrow{s,=} \\
\node{\sD^\co(\C\comod\Rco)^\sop\times\sD^\ctr(\C\contra\Rctr)}
\arrow{s,=}\arrow[4]{e,t}{\Coext_\C}\node[4]{H^0(\R\contra^\free)}
\arrow{s,=} \\
\node{\sD^\co(\C\comod\Rco)^\sop\times\sD^\co(\C\comod\Rco)}
\arrow[4]{e,t}{\Ext_\C}\node[4]{H^0(\R\contra^\free)}
\end{diagram}
$$ \par
\textup{(b)} The same equivalence of triangulated categories\/
$\boL\Phi_{\R,\C}=\boR\Psi_{\R,\C}^{-1}$ transforms
the right derived functor
$$
 \Cotor^\C\:\sD^\co(\comodrRco\C)\times\sD^\co(\C\comod\Rco)
 \lrarrow H^0(\R\comod^\cofr)
$$
into the left derived functor
$$
 \Ctrtor^\C\:\sD^\co(\comodrRco\C)\times\sD^\ctr(\C\contra\Rctr)
 \lrarrow H^0(\R\comod^\cofr).
$$
 In other words, the following diagram of categories, functors,
and equivalences is commutative:
$$
\begin{diagram}
\node{\sD^\co(\comodrRco\C)\times\sD^\co(\C\comod\Rco)}
\arrow{s,=}\arrow[4]{e,t}{\Cotor^\C}\node[4]{H^0(\R\contra^\cofr)}
\arrow{s,=} \\
\node{\sD^\co(\comodrRco\C)\times\sD^\ctr(\C\contra\Rctr)}
\arrow[4]{e,t}{\Ctrtor^\C}\node[4]{H^0(\R\contra^\cofr)}
\end{diagram}
$$
\end{prop}

\begin{proof}
 For any $\R$\+contramodule left CDG\+contramodules $\P$ and $\Q$
over $\C$, there is a natural morphism of complexes of
$\R$\+contramodules $\Cohom_\C(\Phi_{\R,\C}(\P),\Q)=
\Cohom_\C(\cC(\R,\C)\ocn_\C\P\;\Q)\rarrow
\Hom^\C(\P,\Cohom_\C(\cC(\R,\C),\Q))\simeq\Hom^\C(\P,\Q)$,
which is an isomorphism whenever either of the graded
$\C$\+contramodules $\P$ or $\Q$ is a projective $\R$\+free
graded $\C$\+contramodule.
 For any $\R$\+comodule left CDG\+comodules $\cL$ and $\cM$ over $\C$,
there is a natural morphism of complexes of $\R$\+contramodules
$\Cohom_\C(\cL,\Psi_{\R,\C}(\cM))=\Cohom_\C(\cL,\Hom_\C(\cC(\R,\C),
\cM))\rarrow\Hom_\C(\cC(\R,\C)\oc_\C\cL\;\cM)\allowbreak\simeq
\Hom_\C(\cL,\cM)$, which is an isomorphism whenever either of
the graded $\C$\+comodules $\cL$ or $\cM$ is an injective
$\R$\+cofree graded $\C$\+comodule.

 For any $\R$\+comodule right CDG\+comodule $\cN$ and
$\R$\+contramodule left CDG\+contra\-module $\P$ over $\C$,
there is a natural morphism of complexes of $\R$\+comodules
$\cN\ocn_\C\P\simeq(\cN\oc_\C\cC(\R,\C))\ocn_\C\P\rarrow
\cN\oc_\C(\cC(\R,\C)\ocn_\C\P)=\cN\oc_\C\Phi_{\R,\C}(\P)$,
which is an isomorphism whenever either the graded right
$\C$\+comodule $\cN$ is an injective $\R$\+cofree graded
$\C$\+comodule, or the graded left $\C$\+contramodule $\P$
is a projective $\R$\+free graded $\C$\+contramodule.
 (Cf.\ \cite[Section~5.3]{Pkoszul} and
Lemma~\ref{hom-co-associativity}; see \cite[proof of
Corollary~5.6]{Psemi} for further details.)
\end{proof}

 Let $(f,a)\:\C\rarrow\D$ be a morphism of $\R$\+free CDG\+algebras.
 Then the functor $R^f\:H^0(\C\contra\Rctr)\rarrow H^0(\D\contra\Rctr)$
obviously takes contraacyclic $\R$\+con\-tramodule CDG\+contramodules
over $\C$ to contraacyclic $\R$\+contramodule CDG\+mod\-ules over $\D$,
and hence induces a triangulated functor
$$
 \boI R^f\:\sD^\ctr(\C\contra\Rctr)\lrarrow\sD^\ctr(\D\contra\Rctr).
$$
 Similarly, the functor $R_f\:H^0(\C\comod\Rco)\rarrow
H^0(\D\comod\Rco)$ takes coacyclic $\R$\+comodule CDG\+comodules
over $\C$ to coacyclic $\R$\+comodule CDG\+comodules over $\D$,
and hence induces a triangulated functor
$$
 \boI R_f\:\sD^\co(\C\comod\Rco)\lrarrow\sD^\co(\D\comod\Rco).
$$
 The constructions of Sections~\ref{r-free-co-derived}
and~\ref{r-cofree-co-derived} together with
Theorem~\ref{non-adj-co-derived-comp} provide the left adjoint
functor to $\boI R^f$
$$
 \boL E^f\:\sD^\ctr(\D\contra\Rctr)\lrarrow\sD^\ctr(\C\contra\Rctr)
$$
and the right adjoint functor to $\boI R_f$
$$
 \boR E_f\:\sD^\co(\D\comod\Rco)\lrarrow\sD^\co(\C\comod\Rco).
$$

\begin{prop}  \label{non-adj-co-extension}
 The equivalences of triangulated categories\/ $\boL\Phi_{\R,\C}=
\boR\Psi_{\R,\C}^{-1}$ and $\boL\Phi_{\R,\D}=\boR\Psi_{\R,\D}^{-1}$
from Corollary~\textup{\ref{non-adj-derived-co-contra}} transform
the left derived functor\/ $\boL E^f$ into the right derived
functor\/ $\boR E_f$ and back.
 In other words, the following diagram of categories, functors,
and equivalences is commutative:
$$
\begin{diagram}
 \node{\llap{$\boL\Phi_{\R,\D}$}\:\sD^\ctr(\D\contra\Rctr)}
 \arrow{e,=}\arrow{s,l}{\boL E^f}
 \node{\sD^\co(\D\comod\Rco)\,\.\:\!\rlap{$\boR\Psi_{\R,\D}$}}
 \arrow{s,r}{\boR E_f} \\
 \node{\llap{$\boL\Phi_{\R,\C}$}\:\sD^\ctr(\C\contra\Rctr)}
 \arrow{e,=}
 \node{\sD^\co(\C\comod\Rco)\,\.\:\!\rlap{$\boR\Psi_{\R,\C}$}}
\end{diagram}
$$
\end{prop}

\begin{proof}
 Follows from Proposition~\ref{r-free-co-extension}
or~\ref{r-cofree-co-extension}.
\end{proof}

\Section{Change of Coefficients and Compact Generation}

\subsection{Change of coefficients for wcDG-modules}
 Let $\eta\:\R\rarrow\S$ be a profinite morphism (see
Section~\ref{change-of-ring}) of pro-Artinian topological
local rings with the maximal ideals $\m_\R$ and $\m_\S$ and
the residue fields $k_\R$ and $k_\S$.
 Notice that such a morphism is always local, i.~e.,
$\eta(\m_\R)\subset\m_\S$; so $\eta$ induces a finite
field extension $\eta/\m\:k_\R\rarrow k_\S$.

 We will apply the functors $R^\eta$, $E^\eta$, $R_\eta$,
$E_\eta$ to graded contramodules and comodules over $\R$
and $\S$ termwise.
 The natural (iso)morphisms from Section~\ref{change-of-ring}
hold for graded contramodules and comodules in the same form
as in the ungraded case.

 Let $\B$ be an $\R$\+free graded algebra.
 Then the free graded $\S$\+contramodule $E^\eta(\B)$ has
a natural graded algebra structure provided by
the multiplication map $E^\eta(\B)\ot^\S E^\eta(\B)\simeq
E^\eta(\B\ot^\R\B)\rarrow E^\eta(\B)$ induced by
the multiplication in $\B$ and the unit map $\S\rarrow E^\eta(\B)$
induced by the unit map $\R\rarrow\B$.
 
 Let $\M$ be an $\R$\+contramodule graded left $\B$\+module.
 Then the graded $\S$\+contra\-module $E^\eta(\M)$ has
a natural graded left $E^\eta(\B)$\+module structure provided
by the action map $E^\eta(\B)\ot^\S E^\eta(\M)\simeq
E^\eta(\B\ot^\R\M)\rarrow E^\eta(\M)$.
 The same module structure can be defined in terms of the action
map $E^\eta(\M)\rarrow E^\eta\Hom^\R(\B,\M)\simeq
\Hom^\S(E^\eta(\B),E^\eta(\M))$.
 The similar construction applies to right $\B$\+modules.

 Let $\cM$ be an $\R$\+comodule graded left $\B$\+module.
 Then the graded $\S$\+comodule $E_\eta(\cM)$ has
a natural graded left $E^\eta(\B)$\+module structure provided
by the action map $E^\eta(\B)\ocn_\S E_\eta(\cM)\simeq
E_\eta(\B\ocn_\R\cM)\rarrow E_\eta(\cM)$.
 The same module structure can be defined in terms of the action
map $E_\eta(\cM)\rarrow E_\eta\Ctrhom_\R(\B,\cM)\simeq
\Ctrhom_\S(E^\eta(\B),E_\eta(\cM))$.
 The similar construction applies to right $\B$\+modules.

 Let $\N$ be an $\S$\+contramodule graded left $E^\eta(\B)$\+module.
 Then the graded $\R$\+contramodule $R^\eta(\N)$ has
a natural graded left $\B$\+module structure provided by
the action map $\B\ot^\R R^\eta(\N)\rarrow R^\eta(E^\eta(\B)
\ot^\S\N)\rarrow R^\eta(\N)$.
 The same module structure can be defined in terms of the action
map $R^\eta(\N)\rarrow R^\eta\Hom^\S(E^\eta(\B),\N)\simeq
\Hom^\R(\B,R^\eta(\N))$.
 The similar construction applies to right $E^\eta(\B)$\+modules.

 Let $\cN$ be an $\S$\+comodule graded left $E^\eta(\B)$\+module.
 Then the graded $\R$\+comodule $R_\eta(\cN)$ has a natural
graded left $\B$\+module structure provided by the action map
$\B\ocn_\R R_\eta(\cN)\simeq R_\eta(E^\eta(\B)\ocn_\S\cN)
\rarrow R_\eta(\cN)$.
 The same module structure can be defined in terms of the action
map $R_\eta(\cN)\rarrow R_\eta\Ctrhom_\S(E^\eta(\B),\cN)
\rarrow\Ctrhom_\S(\B,R_\eta(\cN))$.
 The similar construction applies to right $E^\eta(\B)$\+modules.

 Now let $\B$ be an $\R$\+free CDG\+algebra.
 Then the $\S$\+free graded algebra $E^\eta(\B)$ has a natural
CDG\+algebra structure with the differential and the curvature element
induced by the differential and the curvature element of~$\B$.
 The above constructions $E^\eta$, $E_\eta$, $R^\eta$, $R_\eta$
for graded $\B$- and $E^\eta(\B)$\+modules assign CDG\+modules
to CDG\+modules, defining DG\+functors
\begin{align*}
 E^\eta\:\B\mod\Rctr &\lrarrow E^\eta(\B)\mod\Sctr, \\
 E_\eta\:\B\mod\Rco &\lrarrow E^\eta(\B)\mod\Sco
\end{align*}
and
\begin{align*}
 R^\eta\: E^\eta(\B)\mod\Sctr &\lrarrow \B\mod\Rctr, \\
 R_\eta\: E^\eta(\B)\mod\Sco &\lrarrow \B\mod\Rco.
\end{align*}
 The DG\+functor $E^\eta$ is left adjoint to the DG\+functor
$R^\eta$, and the DG\+functor $E_\eta$ is right adjoint to
the DG\+functor~$R_\eta$.
 Passing to the homotopy categories, we obtain the induced
triangulated functors.

 Clearly, the functor $R^\eta$ takes short exact sequences and
infinite products of $\S$\+contramodule CDG\+modules to short
exact sequences and infinite products of $\R$\+contramodule
CDG\+modules, hence it takes contraacyclic $\S$\+contramodule
CDG\+modules to contraacyclic $\R$\+contramodule CDG\+modules
and therefore induces a triangulated functor
$$
 \boI R^\eta\:\sD^\ctr(E^\eta(\B)\mod\Sctr)\lrarrow
 \sD^\ctr(\B\mod\Rctr).
$$
 Similarly, the functor $R_\eta$ takes short exact sequences
and infinite direct sums of $\S$\+comodule CDG\+modules to short
exact sequences and infinite direct sums of $\R$\+comodule
CDG\+modules, hence it takes coacyclic $\S$\+comodule CDG\+modules
to coacyclic $\R$\+comodule CDG\+modules and induces
a triangulated functor
$$
 \boI R^\eta\:\sD^\co(E^\eta(\B)\mod\Sco)\lrarrow
 \sD^\co(\B\mod\Rco).
$$

 The functor $E^\eta\:H^0(\B\mod\Rfr)\rarrow H^0(E^\eta(\B)\mod\Sfr)$
takes short exact sequences and infinite products of $\R$\+free
CDG\+modules to short exact sequences and infinite products of
$\S$\+free CDG\+modules; hence it takes contraacyclic $\R$\+free
CDG\+modules to contraacyclic $\S$\+free CDG\+modules and induces
a triangulated functor $\sD^\ctr(\B\mod\Rfr)\rarrow
\sD^\ctr(E^\eta(\B)\mod\Sfr)$.
 Using Theorem~\ref{co-derived-mod}, we obtain the left derived functor
$$
 \boL E^\eta\:\sD^\ctr(\B\mod\Rctr)\lrarrow
 \sD^\ctr(E^\eta(\B)\mod\Sctr),
$$
which is left adjoint to the functor~$\boI R^\eta$.
 Similarly, the functor $E_\eta\:H^0(\B\mod\Rcof)\allowbreak\rarrow
H^0(E^\eta(\B)\mod\Scof)$ takes coacyclic $\R$\+cofree CDG\+modules
to coacyclic $\S$\+cofree CDG\+modules and induces
a triangulated functor $\sD^\co(\B\mod\Rcof)\rarrow
\sD^\co(E^\eta(\B)\mod\Scof)$.
 Using Theorem~\ref{co-derived-mod}, we obtain the right
derived functor {\hbadness=1100
$$
 \boR E_\eta\:\sD^\co(\B\mod\Rco)\lrarrow
 \sD^\co(E^\eta(\B)\mod\Sco),
$$
which} is right adjoint to the functor~$\boI R_\eta$.

 Finally, let $\A$ be a wcDG\+algebra over $\R$; then
$E^\eta(\A)$ is a wcDG\+algebra over~$\S$.
 Clearly, for any $\R$\+contramodule wcDG\+module $\M$ over $\A$,
one has $E^\eta(\M)/\m_\S E^\eta(\M)\allowbreak\simeq
E^{\eta/\m}(\M/\m_\R\M)$, and the DG\+modules $\M/\m_\R\M$ over
$\A/\m_\R\A$ and $E^{\eta/\m}(\M/\m_\R\M)$ over $E^{\eta/\m}(\A/\m\A)
\simeq E^\eta(\A)/\m_\S E^\eta(\A)$ are acyclic simultaneously.
 Hence an $\R$\+free wcDG\+module $\M$ over $\A$ is semiacyclic
if and only if the $\S$\+free wcDG\+module $E^\eta(\M)$ over
$E^\eta(\A)$ is semiacyclic, and the functor $E^\eta$ induces
a triangulated functor $\sD^\si(\A\mod\Rfr)\rarrow\sD^\si
(E^\eta(\A)\mod\Sfr)$.
 Identifying the semiderived categories of $\R$- or $\S$\+free
wcDG\+modules with the semiderived categories of $\R$- or
$\S$\+contramodule wcDG\+modules, we obtain the left derived functor
$$
 \boL E^\eta\:\sD^\si(\A\mod\Rctr)\lrarrow\sD^\si(E^\eta(\A)\mod\Sctr).
$$
 Similarly, for any $\R$\+comodule wcDG\+module $\cM$ over $\A$, one
has ${}_{\m_\S}E_\eta(\cM)\simeq E_{\eta/\m}({}_{\m_\R}\cM)$, hence
an $\R$\+cofree wcDG\+module $\cM$ over $\A$ is semiacyclic if and
only if the $\S$\+cofree wcDG\+module $E_\eta(\cM)$ over $\A$ is
semiacyclic.
 The functor $E_\eta$ induces a triangulated functor
$\sD^\si(\A\mod\Rcof)\rarrow\sD^\si(E^\eta(\A)\mod\Scof)$.
 Identifying the semiderived categories of $\R$- or $\S$\+cofree
wcDG\+modules with the semiderived categories of $\R$- or
$\S$\+comodule wcDG\+modules, we obtain the right derived functor
$$
 \boR E_\eta\:\sD^\si(\A\mod\Rco)\lrarrow\sD^\si(E^\eta(\A)\mod\Sco).
$$

\begin{prop}  \label{semi-extension-coefficients}
 The equivalences of triangulated categories\/ $\boL\Phi_\R =
\boR\Psi_\R^{-1}$ and\/ $\boL\Phi_\S=\boR\Psi_\S^{-1}$ from
Proposition~\textup{\ref{non-adj-r-co-contra}} transform the left
derived functor\/ $\boL E^\eta$ into the right derived functor\/
$\boR E_\eta$ and back.
 In other words, the following diagram of categories, functors,
and equivalences is commutative:
$$
\begin{diagram}
 \node{\llap{$\boL\Phi_\R$}\:\sD^\si(\A\mod\Rctr)}
 \arrow{e,=}\arrow{s,l}{\boL E^\eta}
 \node{\sD^\si(\A\mod\Rco)\,\.\:\!\rlap{$\boR\Psi_\R$}}
 \arrow{s,r}{\boR E_\eta}\\
 \node{\llap{$\boL\Phi_\S$}\:\sD^\si(E^\eta(\A)\mod\Sctr)}
 \arrow{e,=}
 \node{\sD^\si(E^\eta(\A)\mod\Sco)\,\.\:\!\rlap{$\boR\Psi_\S$}}
\end{diagram}
$$
\end{prop}

\begin{proof}
 For any $\R$\+contramodule left wcDG\+module $\M$ over $\A$,
the natural morphism of graded $\S$\+comodules $\Phi_\S E^\eta(\M)
\rarrow E_\eta\Phi_\R(\M)$ is a closed morphism of $\S$\+comodule
left wcDG\+modules over~$E^\eta(\A)$.
 Similarly, for any $\R$\+comodule left wcDG\+module $\cM$ over
$\A$, the natural morphism of graded $\S$\+contramodules
$E^\eta\Psi_\R(\cM)\rarrow\Psi_\S E_\eta(\cM)$ is a closed morphism
of $\S$\+contramodule left wcDG\+modules over~$E^\eta(\A)$
(cf.\ Proposition~\ref{r-s-extension}).
\end{proof}

 In order to define the functors induced by $R^\eta$ and $R_\eta$
on the semiderived categories of wcDG\+modules, we will need to prove
the following lemma first.

\begin{lem} \label{restriction-semi-preserves}
\textup{(a)} The triangulated functor of contrarestriction of
coefficients $R^\eta\:H^0(E^\eta(\A)\mod\Sctr)\rarrow H^0(\A\mod\Rctr)$
takes semiacyclic\/ $\S$\+contramodule wcDG\+modules to semiacyclic\/
$\R$\+contramodule wcDG\+modules. \par
\textup{(b)} The triangulated functor corestriction of coefficients
$R_\eta\:H^0(E^\eta(\A)\mod\Sco)\allowbreak\rarrow H^0(\A\mod\Rco)$
takes semiacyclic\/ $\S$\+comodule wcDG\+modules to semiacyclic\/
$\R$\+comodule wcDG\+modules.
\end{lem}

\begin{proof}
 Part~(a): in view of the semiorthogonal decomposition of
Theorem~\ref{non-adj-semiderived-res}(a,\,c) and the adjunction
of $R^\eta$ and $E^\eta$, the desired assertion is equivalent to 
the functor $E^\eta\:H^0(\A\mod\Rctr)\rarrow H^0(E^\eta(\A)\mod\Sctr)$
taking $H^0(\A\mod\Rctr_\proj)_\proj$ to
$H^0(E^\eta(\A)\mod\Sctr_\proj)_\proj$.
 It is clear that the functor $E^\eta$ takes $H^0(\A\mod\Rctr_\proj)$
to $H^0(E^\eta(\A)\mod\Sctr_\proj)$, and it remains to use
Lemma~\ref{homotopy-proj-inj-reduction}(a) in order to check
the preservation of homotopy projectivity (which is obviously
preserved by the functor $E^{\eta/\m}$).
 The proof of part~(b) is analogous up to duality.
\end{proof}

 According to Lemma~\ref{restriction-semi-preserves}, the functor
$R^\eta$ induces a triangulated functor of contrarestriction of
coefficients
$$
 \boI R^\eta\:\sD^\si(E^\eta(\A)\mod\Sctr)\lrarrow
 \sD^\si(\A\mod\Rctr).
$$
 Similarly, the functor $R_\eta$ induces a triangulated functor
of corestriction of coefficients
$$
 \boI R_\eta\:\sD^\si(E^\eta(\A)\mod\Sco)\lrarrow
 \sD^\si(\A\mod\Rco).
$$
 The functor $\boI R^\eta$ is right adjoint to the functor
$\boL E^\eta$, and the functor $\boI R_\eta$ is left adjoint to
the functor $\boR E_\eta$.
 In view of Proposition~\ref{semi-extension-coefficients},
identifying $\sD^\si(\A\mod\Rctr)$ with $\sD^\si(\A\mod\Rco)$ and
$\sD^\si(E^\eta(\A)\mod\Sctr)$ with $\sD^\si(E^\eta(\A)\mod\Sco)$
allows one to view the functors $\boI R^\eta$ and $\boI R_\eta$ as
the adjoints on two sides to the same triangulated functor
$\boL E^\eta=\boR E_\eta$.

\subsection{Change of coefficients for CDG-contra/comodules}
 Let $\C$ be an $\R$\+free graded coalgebra.
 Then the free graded $\S$\+contramodule $E^\eta(\C)$ has
a natural graded coalgebra structure provided by the comultiplication
map $E^\eta(\C)\rarrow E^\eta(\C\ot^\R\C)\simeq
E^\eta(\C)\ot^\S E^\eta(\C)$ induced by the comultiplication in $\C$
and the counit map $E^\eta(\C)\rarrow\S$ induced by the counit map
$\C\rarrow\R$.

 Let $\P$ be an $\R$\+contramodule graded left $\C$\+contramodule.
 Then the graded $\S$\+contramodule $E^\eta(\P)$ has a natural
left $E^\eta(\C)$\+contramodule structure provided by
the contraaction map $\Hom^\S(E^\eta(\C),E^\eta(\P))\simeq
E^\eta\Hom^\R(\C,\P)\rarrow E^\eta(\P)$.

 Let $\cP$ be an $\R$\+cofree graded left $\C$\+contramodule.
 Then the cofree graded $\S$\+comodule $E_\eta(\cP)$ has 
a natural left $E^\eta(\C)$\+contramodule structure provided by
the contraaction map $\Ctrhom^\S(E^\eta(\C),E_\eta(\cP))\simeq
E_\eta\Ctrhom_\R(\C,\cP)\rarrow E_\eta(\cP)$.

 Let $\cM$ be an $\R$\+comodule graded left $\C$\+comodule.
 Then the graded $\S$\+comodule $E_\eta(\cM)$ has a natural
left $E^\eta(\C)$\+comodule structure provided by the coaction map
$E_\eta(\cM)\rarrow E_\eta(\C\ocn_\R\cM)\simeq E^\eta(\C)\ocn_\S
E_\eta(\cM)$.
 The similar construction applies to right $\C$\+comodules.

 Let $\M$ be an $\R$\+free graded left $\C$\+comodule.
 Then the free graded $\S$\+contramodule $E^\eta(\M)$ has a natural
left $E^\eta(\C)$\+comodule structure provided by the coaction
map $E^\eta(\M)\rarrow E^\eta(\C\ot^\R\M)\simeq E^\eta(\C)\ot^\S
E^\eta(\M)$.
 The similar construction applies to right $\C$\+comodules.

 Let $\Q$ be an $\S$\+contramodule graded left
$E^\eta(\C)$\+contramodule.
 Then the graded $\R$\+contramodule $R^\eta(\Q)$ has a natural
graded left $\C$\+contramodule structure provided by
the contraaction map $\Hom^\R(\C,R^\eta(\Q))\simeq
R^\eta\Hom^\S(E^\eta(\C),\Q)\rarrow R^\eta(\Q)$.

 Let $\cN$ be an $\S$\+comodule graded left $E^\eta(\C)$\+comodule.
 Then the graded $\R$\+comodule $R_\eta(\cN)$ has a natural graded
left $\C$\+comodule structure provided by the coaction map
$R_\eta(\cN)\rarrow R_\eta(E^\eta(\C)\ocn_\S\cN)\simeq
\C\ocn_\R R_\eta(\cN)$.
 The similar construction applies to right $E^\eta(\C)$\+comodules.

 Now let $\C$ be an $\R$\+free CDG\+coalgebra.
 Then the $\S$\+free graded coalgebra $E^\eta(\C)$ has a natural
CDG\+coalgebra structure with the differential and the curvature
linear function induced by the differential and the curvature
linear function of~$\C$.
 The above constructions $E^\eta$, $E_\eta$, $R^\eta$, $R_\eta$
assign CDG\+contra/comodules to CDG\+contra/comodules, defining
DG\+functors
\begin{align*}
 E^\eta\:\C\contra\Rctr &\lrarrow E^\eta(\C)\contra\Sctr, \\
 E_\eta\:\C\contra\Rcof &\lrarrow E^\eta(\C)\contra\Scof,
\displaybreak[0] \\
 E_\eta\:\C\comod\Rco &\lrarrow E^\eta(\C)\comod\Sco,     \\
 E^\eta\:\C\comod\Rfr &\lrarrow E^\eta(\C)\comod\Sfr
\end{align*}
and
\begin{align*}
 R^\eta\: E^\eta(\C)\contra\Sctr &\lrarrow \C\contra\Rctr, \\
 R_\eta\: E^\eta(\C)\comod\Sco &\lrarrow \C\comod\Rco.
\end{align*}
 The DG\+functor $E^\eta\:\C\contra\Rctr\rarrow E^\eta(\C)\contra
\Sctr$ is left adjoint to the DG\+functor $R^\eta$, and
the DG\+functor $E_\eta\:\C\comod\Rco\rarrow E^\eta(\C)\comod\Sco$
is right adjoint to the DG\+functor~$R_\eta$.
 Passing to the homotopy categories, we obtain the induced
triangulated functors.

 Clearly, the functor $R^\eta$ takes short exact sequences and
infinite products of $\S$\+contramodule CDG\+contramodules to
short exact sequences and infinite products of $\R$\+contramodule
CDG\+contramodules, hence it takes contraacyclic $\S$\+contra\-module
CDG\+contramodules to contraacyclic $\R$\+contramodule
CDG\+contramodules and induces a triangulated functor
$$
 \boI R^\eta\:\sD^\ctr(E^\eta(\C)\contra\Sctr)\lrarrow
 \sD^\ctr(\C\contra\Rctr).
$$
 Similarly, the functor $R_\eta$ takes coacyclic $\S$\+comodule
CDG\+comodules to coacyclic $\R$\+comodule CDG\+comodules and
therefore induces a triangulated functor
$$
 \boI R_\eta\:\sD^\co(E^\eta(\C)\comod\Sco)\lrarrow
 \sD^\co(\C\comod\Rco).
$$

 The functor $E^\eta\:H^0(\C\contra\Rfr)\rarrow H^0(E^\eta(\C)
\contra\Sfr)$ takes short exact sequences and infinite products
of $\R$\+free CDG\+contramodules to short exact sequences and
infinite products of $\S$\+free CDG\+contramodules; hence it takes
contraacyclic $\R$\+free CDG\+contramodules to contraacyclic
$\S$\+free CDG\+contramodules and induces a triangulated functor
$\sD^\ctr(\C\contra\Rfr)\rarrow\sD^\ctr(E^\eta(\C)\contra\Sfr)$.
 Using Theorem~\ref{non-adj-co-derived-comp}, we obtain
the left derived functor
$$
 \boL E^\eta\:\sD^\ctr(\C\contra\Rctr)\lrarrow
 \sD^\ctr(E^\eta(\C)\contra\Sctr).
$$
 Similarly, the functor $E_\eta\:H^0(\C\comod\Rcof)\lrarrow
H^0(E^\eta(\C)\comod\Scof)$ takes coacyclic $\R$\+cofree
CDG\+comodules to coacyclic $\S$\+cofree CDG\+comodules and
induces a triangulated functor $\sD^\co(\C\comod\Rcof)\rarrow
\sD^\co(E^\eta(\C)\comod\Scof)$.
 Using Theorem~\ref{non-adj-co-derived-comp}, we obtain
the right derived functor
$$
 \boR E_\eta\:\sD^\co(\C\comod\Rco)\lrarrow
 \sD^\co(E^\eta(\C)\comod\Sco).
$$

 The functor $E_\eta\:H^0(\C\contra\Rcof)\rarrow H^0(E^\eta(\C)
\contra\Scof)$ takes short exact sequences and infinite products
of $\R$\+cofree CDG\+contramodules to short exact sequences and
infinite products of $\S$\+cofree CDG\+contramodules.
 Hence it takes contraacyclic $\R$\+cofree CDG\+contramodules to
contraacyclic $\S$\+cofree CDG\+contramodules and induces
a triangulated functor
$$
 \boI E_\eta\:\sD^\ctr(\C\contra\Rcof)\lrarrow
 \sD^\ctr(E^\eta(\C)\contra\Scof).
$$
 Similarly, the functor $E^\eta\:H^0(\C\comod\Rfr)\lrarrow
H^0(E^\eta(\C)\comod\Sfr)$ takes coacyclic $\R$\+free CDG\+comodules
to coacyclic $\S$\+free CDG\+comodules and induces
a triangulated functor
$$
 \boI E^\eta\:\sD^\co(\C\comod\Rfr)\lrarrow
 \sD^\co(E^\eta(\C)\comod\Sfr).
$$

\begin{prop}
\textup{(a)} The equivalences of triangulated categories\/
$\Phi_\R=\Psi_\R^{-1}$ and\/ $\Phi_\S=\Psi_\S^{-1}$ from
Section~\textup{\ref{r-cofree-co-derived}} together with
the equivalences of categories from
Theorem~\textup{\ref{non-adj-co-derived-comp}} transform
the left derived functor\/ $\boL E^\eta$ into the induced functor\/
$\boI E_\eta$ and the right derived functor\/ $\boR E_\eta$
into the induced functor\/~$\boI E^\eta$.
 In other words, the following diagrams of categories, functors,
and equivalences are commutative:
$$
\begin{diagram}
 \node{\sD^\ctr(\C\contra\Rctr)}\arrow{e,=}\arrow{s,l}{\boL E^\eta}
 \node{\sD^\ctr(\C\contra\Rcof)}\arrow{s,r}{\boI E_\eta}\\
 \node{\sD^\ctr(E^\eta(\C)\contra\Sctr)}\arrow{e,=}
 \node{\sD^\ctr(E^\eta(\C)\contra\Scof)}
\end{diagram}
$$
and
$$
\begin{diagram}
 \node{\sD^\co(\C\comod\Rco)}\arrow{e,=}\arrow{s,l}{\boR E_\eta}
 \node{\sD^\co(\C\comod\Rfr)}\arrow{s,r}{\boI E^\eta}\\
 \node{\sD^\co(E^\eta(\C)\comod\Sco)}\arrow{e,=}
 \node{\sD^\co(E^\eta(\C)\comod\Sfr)}
\end{diagram}
$$ \par
\textup{(b)} The equivalences of triangulated categories\/
$\boL\Phi_\C=\boR\Psi_\C^{-1}$ and\/ $\boL\Phi_{E^\eta(\C)}=
\boR\Psi_{E^\eta(\C)}^{-1}$ from
Corollaries~\textup{\ref{r-free-derived-co-contra}}
and~\textup{\ref{r-cofree-derived-co-contra}} together with
the equivalence of categories from
Theorem~\textup{\ref{non-adj-co-derived-comp}} transform
the left derived functor\/ $\boL E^\eta$ into the induced functor\/
$\boI E^\eta$ and the right derived functor\/ $\boR E_\eta$
into the induced functor\/~$\boI E_\eta$.
 In other words, the following diagrams of categories, functors,
and equivalences are commutative:
$$
\begin{diagram}
 \node{\sD^\ctr(\C\contra\Rctr)}\arrow{e,=}\arrow{s,l}{\boL E^\eta}
 \node{\sD^\co(\C\comod\Rfr)}\arrow{s,r}{\boI E^\eta}\\
 \node{\sD^\ctr(E^\eta(\C)\contra\Sctr)}\arrow{e,=}
 \node{\sD^\co(E^\eta(\C)\comod\Sfr)}
\end{diagram}
$$
and
$$
\begin{diagram}
 \node{\sD^\co(\C\comod\Rco)}\arrow{e,=}\arrow{s,l}{\boR E_\eta}
 \node{\sD^\ctr(\C\contra\Rcof)}\arrow{s,r}{\boI E_\eta}\\
 \node{\sD^\co(E^\eta(\C)\comod\Sco)}\arrow{e,=}
 \node{\sD^\ctr(E^\eta(\C)\contra\Scof)}
\end{diagram}
$$ \par
\textup{(c)} The equivalences of triangulated categories
$\boL\Phi_{\R,\C}=\boR\Psi_{\R,\C}^{-1}$ and\/
$\boL\Phi_{\S,\.E^\eta(\C)}=\boR\Psi_{\S,\.E^\eta(\C)}^{-1}$
from Corollary~\textup{\ref{non-adj-derived-co-contra}} transform
the left derived functor\/ $\boL E^\eta$ into the right derived
functor\/~$\boR E_\eta$.
 In other words, the following diagram of categories, functors,
and equivalences is commutative:
$$
\begin{diagram}
 \node{\llap{$\boL\Phi_{\R,\C}$}\:\sD^\ctr(\C\contra\Rctr)}
 \arrow{e,=}\arrow{s,l}{\boL E^\eta}
 \node{\sD^\co(\C\comod\Rco)\,\.\:\!\rlap{$\boR\Psi_{\R,\C}$}}
 \arrow{s,r}{\boR E_\eta}\\
 \node{\llap{$\boL\Phi_{\S,\.E^\eta(\C)}$}\:
 \sD^\ctr(E^\eta(\C)\contra\Sctr)}
 \arrow{e,=}\node{\sD^\co(E^\eta(\C)\comod\Sco)
 \,\.\:\!\rlap{$\Psi_{\S,\.E^\eta(\C)}$}}
\end{diagram}
$$
\end{prop}

\begin{proof}
 Part~(c): Notice that the functor $E_\eta$ takes the $\R$\+cofree
graded $\C$\+bicomodule $\cC(\R,\C)$ to the $\S$\+cofree graded
$E^\eta(\C)$\+bicomodule $\cC(\S,E^\eta(\C))$.
 Furthermore, for any $\R$\+contramodule left CDG\+contramodule $\P$
over $\C$ there is a natural closed morphism of $\S$\+comodule left
CDG\+comodules $\Phi_{\S,\.E^\eta(\C)}E^\eta(\P)\rarrow
E_\eta\Phi_{\R,\C}(\P)$ over $E^\eta(\C)$, which is an isomorphism
whenever $\P$ is a projective $\R$\+free graded $\C$\+contramodule.
 Similarly, for any $\R$\+comodule left CDG\+comodule $\cM$ over $\C$
there is a natural closed morphism of $\S$\+contramodule left
CDG\+contramodules $E^\eta\Psi_{\R,\C}(\cM)\rarrow
\Psi_{\S,\.E^\eta(\C)}E_\eta(\cM)$ over $E^\eta(\C)$, which is
an isomorphism whenever $\cM$ is an injective $\R$\+cofree graded
$\C$\+comodule.
\end{proof}

\subsection{Compact generator for wcDG-modules}
 Denote by $\kappa\:\R\rarrow k$ the natural surjection from
a pro-Artinian topological local ring $\R$ to its residue field~$k$.
 Given a CDG\+algebra $B$ over~$k$, we denote by
$\sD^\ctr(B\mod)$ and $\sD^\co(B\mod)$, respectively, the contraderived
and coderived category of left CDG\+modules over~$B$.

\begin{thm}  \label{restriction-generates}
 Let\/ $\B$ be an\/ $\R$\+free CDG\+algebra. Then \par
\textup{(a)} the contraderived category\/ $\sD^\ctr(\B\mod\Rctr)$ is
generated, as a triangulated category with infinite products, by
the image of the triangulated functor
$$
 \boI R^\kappa\:\sD^\ctr(\B/\m\B\mod)\lrarrow\sD^\ctr(\B\mod\Rctr);
$$ \par
\textup{(b)} the coderived category\/ $\sD^\co(\B\mod\Rco)$ is
generated, as a triangulated category with infinite direct sums, by
the image of the triangulated functor
$$
 \boI R_\kappa\:\sD^\co(\B/\m\B\mod)\lrarrow\sD^\co(\B\mod\Rco).
$$
\end{thm}

\begin{proof}
 Part~(a): by Theorem~\ref{co-derived-mod}, any object of
$\sD^\ctr(\B\mod\Rctr)$ can be represented by an $\R$\+free
left CDG\+module $\M$ over~$\B$.
 For any $n\ge0$, denote by $\om n\subset\R$ the topological
closure of the $n$\+th power of the ideal $\m\subset\R$; so we have
$\R=\om0\supset\m=\om1\supset\om2\supset\om3\supset\dotsb$
 Applying to $\M$ the contraextension and subsequently
the contrarestriction of scalars for the morphism of pro-Artinian
local rings $\R\rarrow\R/\om n$, we obtain a sequence of closed
morphisms of $\R$\+contramodule CDG\+modules $\M/\m\M\larrow
\M/\om2\M\larrow\M/\om3\M\larrow\dotsb$ over~$\B$.
 Since $\M$ is $\R$\+free and $\m$ is topologically nilpotent,
the projective limit of this sequence coincides with~$\M$.
 Since the contraaction morphism $\m[[\om n\M]]\rarrow\M$
lands in $\om{{n+1}}\M$, all the kernels of closed morphisms
of CDG\+modules in our sequence have their graded $\R$\+contramodule
structures obtained by the contrarestriction of scalars from
$k$\+vector space structures, i.~e., they belong to the image
of~$R^\kappa$. 
 Hence the CDG\+modules $\M/\om n\M$ belong to the triangulated
subcategory in $\sD^\ctr(\B\mod\Rctr)$ generated by the image
of~$R^\kappa$.
 It remains to notice that the telescope short sequence
$0\rarrow\M\rarrow\prod_n\M/\om n\M\rarrow\prod_n \M/\om n\M\rarrow0$
is exact, since it is exact as a sequence of abelian groups
(the morphisms $\M/\om{{n+1}}\M\rarrow\M/\om n\M$ being surjective
and the forgetful functor from $\R\contra$ to abelian groups
commuting with the infinite products).

 The proof of part~(b) is analogous up to duality (and even somewhat
simpler).
\end{proof}

 Given a DG\+algebra $A$ over the field~$k$, we denote by $\sD(A\mod)$
the (conventional) derived category of left DG\+modules over~$A$.

\begin{cor} \label{wcdg-restriction-generates}
 Let\/ $\A$ be a wcDG\+algebra over\/~$\R$.  Then \par
\textup{(a)} the semiderived category\/ $\sD^\si(\A\mod\Rctr)$ is
generated, as a triangulated category with infinite products, by
the image of the triangulated functor
$$
 \boI R^\kappa\:\sD(\A/\m\A\mod)\lrarrow\sD^\si(\A\mod\Rctr);
$$ \par
\textup{(b)} the semiderived category\/ $\sD^\si(\A\mod\Rco)$ is
generated, as a triangulated category with infinite direct sums,
by the image of the triangulated functor
$$
 \boI R_\kappa\:\sD(\A/\m\A\mod)\lrarrow\sD^\si(\A\mod\Rco).
$$
\end{cor}

\begin{proof}
 Follows from Theorem~\ref{restriction-generates}.
\end{proof}

 Any DG\+algebra $A$ over a field $k$ can be also considered, at
one's choice, as a left or a right DG\+module over itself.
 We will denote this DG\+module simply by~$A$.

\begin{thm} \label{semiderived-compact-generator-thm}
 For any wcDG\+algebra\/ $\A$ over\/ $\R$, the\/ $\R$\+comodule
left wcDG\+module \hfuzz=3.4pt
$$
 \boI R_\kappa(\A/\m\A)
$$
over\/ $\A$ is a compact generator of the semiderived category of\/
$\R$\+comodule left wcDG\+modules\/ $\sD^\si(\A\mod\Rco)$.
\end{thm}

\begin{proof}
 Let us show that $\boI R_\kappa(\A/\m\A)\in\sD^\si(\A\mod\Rco)$ is 
a compact object.
 Since the DG\+module $\A/\m\A$ is a compact object of
$\sD(\A/\m\A\mod)$ and the functor $\boI R_\kappa$ is left adjoint
to the functor $\boR E_\kappa$, it suffices to check that
the functor $\boR E_\kappa$ preserves infinite direct sums.
 The latter assertion is true for any profinite morphism
$\eta\:\R\rarrow\S$ in place of~$\kappa$.
 It suffices to identify $\sD^\si(\A\mod\Rco)$ with
$\sD^\si(\A\mod\Rcof)$ and notice that the DG\+functor
$E_\eta\:\A\mod\Rcof\rarrow E^\eta(\A)\mod\Scof$ preserves infinite
direct sums, as does the localization functor
$H^0(\A\mod\Rcof)\rarrow\sD^\si(\A\mod\Rcof)$.
 (Similarly, the functor $\boL E^\eta$ preserves infinite products.)

 More explicitly, the object $\boI R_\kappa(\A/\m\A)\in
\sD^\si(\A\mod\Rco)$ represents the functor assigning to
an $\R$\+cofree wcDG\+module $\cM$ over $\A$ the degree-zero
cohomology group of the DG\+module ${}_\m\cM$ over $\A/\m\A$.
 By the definition of the semiderived category of $\R$\+cofree
wcDG\+modules, an object $\cM$ annihilated, together with all of
its shift, by this cohomological functor, vanishes in
$\sD^\co(\A\mod\Rcof)$.
 Alternatively, one can use
Corollary~\ref{wcdg-restriction-generates}(b) together with the fact
that the DG\+module $\A/\m\A$ generates $\sD(\A/\m\A\mod)$ in order to
show that $\boI R_\eta(\A/\m\A)$ generates $\sD^\co(\A\mod\Rco)$.
\end{proof}

\begin{rem}
 Informally speaking, the semiderived category $\sD^\si(\A\mod\Rco)$ 
is a mixture of the conventional derived category in the direction
of $\A$ relative to $\R$ and the coderived category along the variables
from~$\R$.
 So the assertion of Theorem~\ref{semiderived-compact-generator-thm} is
a common generalization of two results.
 On the one hand, when $\R=k$ is a field and $\A=A$ is a DG\+algebra
over~$k$, Theorem~\ref{semiderived-compact-generator-thm} reduces to
the assertion that the left DG\+module $A$ over $A$ is a single compact
generator of the derived category of left DG\+modules $\sD(A\mod)$
(cf.~\cite[Section~5]{Kel}).
 On the other hand, when $\A=\R$,
Theorem~\ref{semiderived-compact-generator-thm} claims that
the irreducible $\R$\+comodule $k^\sop$ is a single compact generator
of the coderived category $\sD^\co(\R\comod)$.
 The coderived category $\sD^\co(\R\comod)$, which is equivalent to
the homotopy category $H^0(R\comod^\cofr)$ of complexes of injective or
cofree $\R$\+comodules, is essentially the zero-dimensional particular
case of what is called the category of ind-coherent sheaves on
an ind-scheme in~\cite[Section~10]{Gai} or~\cite[Chapter~3]{GR}.
 Thus the results of our Corollary~\ref{wcdg-restriction-generates}(b)
and Theorem~\ref{semiderived-compact-generator-thm} for $\A=\R$ can be
viewed as a particular case of~\cite[Corollary~3.2.2 in Chapter~3]{GR}.
 (See~\cite[Theorem~2.4]{Jor} and~\cite[Proposition~2.3]{Kra} for
historically first results of this kind,
and~\cite[Section~3.11]{Pkoszul} for an exposition based on coderived
categories.)
\end{rem}

\begin{cor} \label{semiderived-compact-generator-cor}
 For any wcDG\+algebra\/ $\A$ over\/ $\R$, the semiderived category\/
$\sD^\si(\A\mod\Rctr)$ of\/ $\R$\+contramodule left wcDG\+modules
over\/ $\A$ has a single compact generator.
 So do the semiderived category\/ $\sD^\si(\A\mod\Rfr)$ of\/
$\R$\+free left wcDG\+modules over\/ $\A$ and the semiderived category\/
$\sD^\si(\A\mod\Rcof)$ of\/ $\R$\+cofree left wcDG\+modules over\/~$\A$.
\end{cor} 

\begin{proof}
 All the three categories are naturally equivalent to the one whose
compact generator was constructed in
Theorem~\ref{semiderived-compact-generator-thm}.
 Use the construction of Theorem~\ref{co-derived-mod} to obtain
the object of $\sD^\si(\A\mod\Rcof)$ corresponding to our
$\R$\+comodule wcDG\+module $\boI R_\kappa(\A/\m\A)$, and
the construction of the functor $\Psi_\R$ from
Sections~\ref{r-cofree-graded}\+-\ref{r-cofree-semi} and
Proposition~\ref{non-adj-r-co-contra} to obtain
the corresponding object of $\sD^\si(\A\mod\Rfr)$.
\end{proof}

\begin{ex} \label{kln-counterex}
 Let $\R$ be a pro-Artinian topological local ring with the maximal
ideal~$\m$ and the residue field~$k$.
 Let $\epsilon$ be an element of $\m\setminus\om2\subset\R$.
 Then there exists an open ideal $\om2\subset\I\subset\R$ such that
the quotient ring $S=\R/\I$ is a module of length~$2$ over itself
and has its maximal ideal $m$ generated by the class
$0\ne\bar\epsilon\in S$ of the element~$\epsilon$.
 Denote the natural surjections between our local rings by $\eta\:
\R\rarrow S$, \ $\sigma\:S\rarrow k$, and $\kappa\:\R\rarrow k$.

 Consider the $\R$\+free graded algebra $\A$ with the components
$\A^i=\R$ for all $i$ divisible by~$2$ and $\A^i=0$ otherwise,
the multiplication maps in $\A$ being the identity maps
$\R\ot^\R\R=\R\rarrow\R$.
 Let $x\in\A^2$ denote the element corresponding to $1\in\R$
(so $\A=\R[x,x^{-1}]$ if $2$ has an infinite order in the grading
group~$\Gamma$ and $\A=\R[x]/(x^n-1)$ if $2\in\Gamma$ is
an element of order~$n$; in particular, $\A=\R$ and $x=1$ if $2=0$
in~$\Gamma$).
 Define the wcDG\+algebra structure on $\A$ with $d=0$ and
$h=\epsilon x$.
 Set $B=E^\eta(\A)$.

 It was noticed in~\cite[proof of Proposition~3.7]{KLN}
that the $S$\+contramodule wcDG\+mod\-ule $R^\sigma(B/mB)$ (or, which
is essentially the same, the $S$\+comodule wcDG\+module
$R_\sigma(B/mB)$) over $B$ is absolutely acyclic.
 Indeed, consider the algebra $B$ as an $S$\+contra/comodule graded
module over itself and apply the construction $G^+$ of the freely
generated wcDG\+module (see the proof of
Theorem~\ref{r-free-absolute-derived}) to it.

 Taking the tensor product of the exact triple of $S$\+modules
$k\rarrow S\rarrow k$ with the algebra $B$ over $S$, we obtain
an exact triple of $S$\+contra/comodule graded $B$\+modules
$B/mB\rarrow B\rarrow B/mB$; applying the construction $G^+$, we get
an exact triple of $S$\+contra/comodule CDG\+modules and closed
morphisms $G^+(B/mB)\rarrow G^+(B)\rarrow G^+(B/mB)$.
 Since the quotient module $B/mB$ has a natural structure of
$S$\+contra/comodule CDG\+module over $B$, there is a natural
closed morphism of CDG\+modules $G^+(B/mB)\rarrow B/mB$; hence
the induced exact triple of CDG\+modules and closed morphisms
$B/mB\rarrow M\rarrow G^+(B/mB)$.

 The closed morphism $B/mB\rarrow M$ is homotopic to zero,
the contracting homotopy being given by the rule $b\mpsto
x^{-1}d_G(b)$.
 The CDG\+modules $G^+(L)$ being always contractible
(see~\cite[proof of Theorem~1.4]{Psing}), it follows that both
the $S$\+contra/co\-module CDG\+modules $B/mB$ and $M$ are
absolutely acyclic.
 Consequently, the $\R$\+contramodule wcDG\+module
$R^\kappa(\A/\m\A)=R^\eta R^\sigma(B/mB)$ and the $\R$\+comodule
wcDG\+module $R_\kappa(\A/\m\A)=R_\eta R_\sigma(B/mB)$ are
absolutely acyclic, too.

 Notice that the graded algebra $\A/\m\A$ is a ``graded field'',
i.~e., every graded module over it is free.
 Since the differential on $\A/\m\A$ is zero, it follows easily
(cf.~\cite[proof of Proposition~5.10]{KLN}) that every acyclic
DG\+module over $\A/\m\A$ is contractible.
 Consequently, the DG\+module $\A/\m\A$ generates the homotopy
category $H^0(\A/\m\A\mod)$ considered either as a triangulated
category with infinite direct sums, or as a triangulated
category with infinite products.
 By Theorem~\ref{restriction-generates}, it follows that both
the contraderived category $\sD^\ctr(\A\mod\Rctr)$ and
the coderived category $\sD^\co(\A\mod\Rco)$ vanish.
 Thus $\sD^\si(\A\mod\Rctr)=0=\sD^\si(\A\mod\Rco)$
(cf.\ Example~\ref{kln-counterex2} below).

 Furthermore, every short exact sequence of $\R$\+free
graded modules over $\A$ splits, as does every short exact
sequence of $\R$\+cofree graded modules; so
$\sD^\ctr(\A\mod\Rfr)=H^0(\A\mod\Rfr)$ and
$\sD^\co(\A\mod\Rcof)=H^0(\A\mod\Rcof)$.
 Using Theorem~\ref{co-derived-mod}, we conclude that
$H^0(\A\mod\Rfr)=0=H^0(\A\mod\Rcof)$.

 In particular, the wcDG\+algebra morphisms like $\A\rarrow 0$
or $\A\rarrow\A\oplus\A$, etc., induce equivalences of
the semiderived categories of wcDG\+modules, while not being
quasi-isomorphisms modulo~$\m$ at all
(cf.~Remark~\ref{reduction-quasi-not-converse}).
\end{ex}

\begin{rem}
 The notion of compactness in application to the triangulated
categories we are dealing with in this paper in inherently
ambigous, because these categories (and their DG\+enhancements)
can be naturally viewed as being enriched over $\R$\+contramodules.
 The problem is that the definition of compactness involves
considering infinite direct sums of the groups/modules of
morphisms in the category, and the forgetful functor
$\R\contra\rarrow\R\mod$ does not preserve infinite direct sums.

 In the above discussion, as indeed everywhere in this paper,
we presume the conventional notion of compactness of triangulated
categories with abelian groups of morphisms (so the contramodule
enrichment is ignored).
 The following example illustrates the difference.

 Let $\A$ be an $\R$\+free DG\+algebra (i.~e., a wcDG\+algebra
with $h=0$); consider $\A$ as a left (wc)DG\+module over itself.
 Then the $\R$\+comodule wcDG\+module $\boI R_\kappa(\A/\m\A)$
over $\A$ has a right resolution by direct sums of copies of
$\Phi_\R(\A)$, so $\A$ generates $\sD^\si(\A\mod\Rctr)\simeq
\sD^\si(\A\mod\Rco)$ as a triangulated category with infinite
direct sums and products (and even as a triangulated category with
infinite direct sums when $\R$ has finite homological dimension).
 Besides, for any $\R$\+contramodule wcDG\+module $\M$ over $\A$,
the complex $\Hom_\A(\A,\M)$ computes $\Hom$ in the semiderived
category of wcDG\+modules (see Lemma~\ref{homotopy-proj-inj-reduction}
and Theorem~\ref{non-adj-semiderived-res}); this complex also
coincides with the complex of $\R$\+contramodules underlying
the DG\+module~$\M$.
 Suppose $\R$ has finite homological dimension; then a wcDG\+module
$\M$ is (semi)acyclic whenever the complex $\M=\Hom_\A(\A,\M)$
is acyclic.

 Furthermore, the functor $\Hom_\A(\A,{-})$ transforms infinite
direct sums in $\sD^\si(\A\mod\Rctr)$ (represented by infinite
direct sums in $H^0(\A\mod\Rfr)$; cf.\ the definition of
functor $\Ext_\A$ in Section~\ref{wcdg-semiderived}) into infinite
direct sums in the contraderived category $\sD^\ctr(\R\contra)$
(represented by infinite direct sums in $H^0(\R\contra^\free)$).
 When $\R$ has homological dimension~$1$, infinite direct sums
in $\sD^\ctr(\R\contra)$ even commute with the passage to
the $\R$\+contramodules of cohomology (see
Remark~\ref{dim-1-contramodules} and Section~\ref{discrete-modules}).
 Still, these do not commute with the forgetful functor to
$\R\mod$, and so the wcDG\+module $\A$ over $\A$ is \emph{not}
compact in our sense (as one can see already in the simplest case
$\A=\R=k[[\epsilon]]$).  \hbadness=2950
\end{rem}

\subsection{Compact generators for CDG-co/contramodules}
 Given a CDG\+coalge\-bra $C$ over the field~$k$, we denote by
$\sD^\ctr(C\contra)$ and $\sD^\co(C\comod)$, respectively,
the contraderived category of left CDG\+contramodules and
the coderived category of left CDG\+comodules over~$C$.

\begin{thm}  \label{restriction-generates-co}
 Let\/ $\C$ be an\/ $\R$\+free CDG\+coalgebra. Then \par
\textup{(a)} the contraderived category\/ $\sD^\ctr(\C\contra\Rctr)$
is generated, as a triangulated category with infinite products, by
the image of the triangulated functor
$$
 \boI R^\kappa\:\sD^\ctr(\C/\m\C\contra)\lrarrow\sD^\ctr(\C\contra\Rctr);
$$ \par
\textup{(b)} the coderived category\/ $\sD^\co(\C\comod\Rco)$ is
generated, as a triangulated category with infinite direct sums, by
the image of the triangulated functor
$$
 \boI R_\kappa\:\sD^\co(\C/\m\C\comod)\lrarrow\sD^\co(\C\comod\Rco).
$$
\end{thm}

\begin{proof}
 Similar to the proof of Theorem~\ref{restriction-generates}.
\end{proof}

\begin{lem} \label{union-of-finite-length}
\textup{(a)} Let\/ $\C$ be an\/ $\R$\+free graded coalgebra.
 Then any\/ $\R$\+comodule graded\/ $\C$\+comodule is
the union of its\/ $\R$\+comodule graded\/ $\C$\+subcomodules that
have finite length as graded\/ $\R$\+comodules (in particular, these
have a finite number of nonzero grading components only). \par
\textup{(b)} Let\/ $\C$ be an\/ $\R$\+free CDG\+coalgebra.
 Then any\/ $\R$\+comodule CDG\+comodule over\/ $\C$
is the union of its\/ $\R$\+comodule CDG\+subcomodules whose
underlying graded\/ $\R$\+comodules have finite length.
\end{lem}

\begin{proof}
 Part~(a): the key observation is that the functor of contratensor
product of $\R$\+contramodules and $\R$\+comodules $\ocn_\R$
commutes with the inductive limits in the comodule argument.
 In addition, filtered inductive limits are exact in $\R\comod$,
as is the functor $\C\ocn_\R{-}$. 
 Let $\cM$ be an $\R$\+comodule graded left $\C$\+comodule; pick
an $\R$\+subcomodule of finite length $\cV\subset\cM$ and consider
the full preimage $\cL\subset\cM$ of $\C\ocn_\R\cV\subset
\C\ocn_\R\cM$ under the $\C$\+coaction map $\cM\rarrow\C\ocn_\R\cM$.
 It follows from the counit axiom for $\cM$ that $\cL$ is contained
in~$\cV$, hence $\cL$ is an $\R$\+comodule of finite length.
 Furthermore, $\cL$ is a $\C$\+subcomodule in $\cM$, because
the $\C$\+coaction map $\cM\rarrow\C\ocn_\R\cM$ is a $\C$\+comodule
morphism (the coassociativity axiom for the coaction) and
$\C\ocn_\R\cV$ is a $\C$\+subcomodule in $\C\ocn_\R\cM$.
 Finally, $\cM$ is the filtered inductive limit of its
$\C$\+subcomodules $\cL$ indexed by all the $\R$\+subcomodules
$\cV\subset\cM$ of finite length, since $\C\ocn_\R\cM$ is
the inductive limit of $\C\ocn_\R\cV$ and inductive limits commute
with fibered products in $\R\comod$.

 Part~(b): Let $\cM$ be a left CDG\+comodule over $\C$ and 
$\cL\subset\cM$ be a graded $\C$\+subcomodule having finite length
as a graded $\R$\+comodule.
 Then $\cL+d_\cM(\cL)\subset\cM$ is a CDG\+subcomodule of $\cM$
with the same property.
\end{proof}

 It follows from Lemma~\ref{union-of-finite-length} that having
finite length as a graded $\R$\+comodule or as a graded $\C$\+comodule
(or even as a CDG\+comodule over~$\C$) are equivalent properties for
an $\R$\+comodule graded $\C$\+comodule (or an $\R$\+comodule
CDG\+comodule over~$\C$).
 Therefore, we will simply call the graded comodules (resp.,
CDG\+comodules) with this property the $\R$\+comodule graded
$\C$\+comodules (resp., CDG\+comodules over~$\C$) \emph{of
finite length}.

 The DG\+subcategory of $\C\comod\Rco$ formed by the CDG\+comodules
of finite length will be denoted by $\C\comod\Rco_\fin$;
the corresponding homotopy category is $H^0(\C\comod\Rco_\fin)$.
 The quotient category of $H^0(\C\comod\Rco_\fin)$ by its minimal
thick subcategory containing the total CDG\+comodules of short
exact sequences of CDG\+comodules of finite length and closed
morphisms between them is called the \emph{absolute derived
category} of $\R$\+comodule left CDG\+comodules of finite length
over $\C$ and denoted by $\sD^\abs(\C\comod\Rco_\fin)$.

\begin{thm}  \label{coderived-compact-generators-thm}
 Let\/ $\C$ be an\/ $\R$\+free CDG\+coalgebra.  Then \par
\textup{(a)} the triangulated functor
$$
 \sD^\abs(\C\comod\Rco_\fin)\lrarrow\sD^\co(\C\comod\Rco)
$$
induced by the embedding of DG\+categories\/
$\C\comod\Rco_\fin\rarrow\C\comod\Rco$ is fully faithful; and \par
\textup{(b)} the image of this functor (or, more precisely, a set
of representatives of the isomorphism classes in the image) is
a set of compact generators of the coderived category\/
$\sD^\co(\C\comod\Rco)$.
\end{thm}

\begin{proof}
 The proof of part~(a) is similar to that
of~\cite[Theorem~3.11.1]{Pkoszul}.
 There are two ways to prove part~(b): either one can use
the general argument from~\cite[proof of Theorem~3.11.2]{Pkoszul}
(due to D.~Arinkin), or alternatively the assertion can be deduced,
using Theorem~\ref{restriction-generates-co}(b), from the similar
result for CDG\+comodules over CDG\+coalgebras over
fields~\cite[Section~5.5]{Pkoszul}.
\end{proof}

\begin{cor}  \label{coderived-compact-generators-cor}
 For any\/ $\R$\+free\/ CDG\+coalgebra\/ $\C$, the contraderived
category\/ $\sD^\ctr(\C\contra\Rctr)$ of\/ $\R$\+contramodule left
CDG\+contramodules over\/ $\C$ is compactly generated.
 So are the contraderived categories\/ $\sD^\ctr(\C\contra\Rfr)$
and\/ $\sD^\ctr(\C\contra\Rcof)$ and the coderived categories\/
$\sD^\co(\C\comod\Rcof)$ and\/ $\sD^\co(\C\comod\Rfr)$.
\emergencystretch=1em\hfuzz=13pt
\end{cor}

\begin{proof}
 All the five mentioned categories are naturally equivalent to
the one whose compact generators were constructed in
Theorem~\ref{coderived-compact-generators-thm}.
 See Theorem~\ref{non-adj-co-derived-comp} and
Corollaries~\ref{r-free-derived-co-contra},
\ref{r-cofree-derived-co-contra},
and~\ref{non-adj-derived-co-contra}.
\end{proof}

\Section{Bar and Cobar Duality}

\subsection{Bar- and cobar-constructions}  \label{bar-cobar-sect}
 The bar-construction for nonaugmented $\R$\+free CDG\+algebras
and the cobar-construction for noncoaugmented $\R$\+free
CDG\+coalgebras are based on the following lemma.
{\hbadness=1950\par}

\begin{lem}  \label{unit-split}
\textup{(a)} If\/ $\B$ is a nonzero\/ $\R$\+free graded algebra,
then the unit map\/ $\R\rarrow\B$ is the embedding of a direct
summand in the category of free graded\/ $\R$\+contramodules. \par
\textup{(b)} If\/ $\C$ is a nonzero\/ $\R$\+free graded coalgebra,
then the counit map\/ $\C\rarrow\R$ is the projection onto
a direct summand in the category of free graded\/
$\R$\+contramodules.
\end{lem}

\begin{proof}
 Part~(a): reducing the unit map $\R\rarrow\B$ modulo~$\m$, we
obtain the unit map $k\rarrow\B/\m\B$ of the graded
$k$\+algebra $\B/\m\B$.
 If the latter map is zero, it follows that $\B/\m\B=0$ and $\B=0$.
 Otherwise, pick a homogeneous $k$\+linear map
$\bar v\:\B/\m\B\rarrow k$ such that the composition
$k\rarrow\B/\m\B\rarrow k$ is the identity map, and lift $\bar v$
to a homogeneous morphism of graded $\R$\+contramodules
$v\:\B\rarrow\R$.
 Then the composition $\R\rarrow\B\rarrow\R$ is invertible
(see the proof of Lemma~\ref{nakayama-acycl-contract}).

 Part~(b): reducing the counit map $\C\rarrow\R$ modulo~$\m$, we
obtain the counit map $\C/\m\C\rarrow k$ of the graded
coalgebra $\C$ over~$k$.
 If the latter map is zero, it follows that $\C/\m\C=0$ and $\C=0$.
 Otherwise, pick a homogeneous $k$\+linear map $\bar w\:k\rarrow
\C/\m\C$ such that the composition $k\rarrow\C/\m\C\rarrow k$
is the identity map, and continue to argue as above.
\end{proof}

 Let $\U$ be a free graded $\R$\+contramodule.
 Then the infinite direct sum of tensor powers $\bigoplus_{n=0}^
\infty\U^{\ot n}$ in the category of free graded
$\R$\+contramodules has a natural structure of $\R$\+free graded
algebra with the multiplication given by the conventional rule
$(u_1\ot\dotsb \ot u_j)(u_{j+1}\ot\dotsb\ot u_n) =
u_1\ot\dotsb\ot u_j\ot u_{j+1}\ot\dotsb\ot u_n$ and
the unit element provided by the embedding of the component
$\R=\U^{\ot 0}\rarrow\bigoplus_{n=0}^\infty\U^{\ot n}$.
 The same infinite direct sum of tensor powers also has a natural
$\R$\+free graded coalgebra structure with the comultiplication
$u_1\ot\dotsb\ot u_n\mpsto\sum_{j=0}^n (u_1\ot\dotsb\ot u_j)\ot
(u_{j+1}\ot\dotsb\ot u_n)$ and the counit map being the projection
onto the component $\U^{\ot 0}=\R$.

\begin{lem}  \label{free-algebra-derivations}
 \textup{(a)} Odd derivations of degree~$1$ on the\/ $\R$\+free
graded algebra\/ $\bigoplus_{n=0}^\infty \U^{\ot n}$ are determined
by their restrictions to the component\/ $\U^{\ot1}\simeq\U$.
 Conversely, any homogeneous\/ $\R$\+contramodule morphism\/
$\U\rarrow\bigoplus_{n=0}^\infty\U^{\ot n}$ of degree~$1$
gives rise to an odd derivation of degree~$1$ on\/
$\bigoplus_{n=0}^\infty \U^{\ot n}$. \par
 \textup{(b)} Odd coderivations of degree~$1$ on the\/ $\R$\+free
graded coalgebra\/ $\bigoplus_{n=0}^\infty \U^{\ot n}$ are determined
by their projections to the component\/ $\U^{\ot1}\simeq\U$.
 Conversely, any homogeneous\/ $\R$\+contramodule morphism\/
$\bigoplus_{n=0}^\infty \U^{\ot n}\rarrow\U$ of degree~$1$
gives rise to an odd coderivation of degree~$1$ on\/
$\bigoplus_{n=0}^\infty \U^{\ot n}$.
\end{lem}

\begin{proof}
 Straightforward and similar to the graded $k$\+(co)algebra case.
 In the case~(b), it is essential that the natural map
$\bigoplus_{n=0}^\infty\U^{\ot n}\rarrow\prod_{n=0}^\infty\U^{\ot n}$
is injective (since $\U$ is a \emph{free} graded $\R$\+contramodule)
and no coderivation of $\bigoplus_{n=0}^\infty\U^{\ot n}$ can raise
the tensor degree by more than~$1$ (since no map
$\bigoplus_{n=0}^\infty\U^{\ot n}\rarrow\U$ does).
\end{proof}

 A graded coalgebra $D$ without counit over a field~$k$ is called
\emph{conilpotent} if it is the union
$D=\bigcup_n\ker(D\to D^{\ot n+1})$ of the kernels of the iterated
comultiplication maps.
 A graded coalgebra $C$ over $k$ endowed with a coaugmentation
(morphism of coalgebras) $\bar w\:k\rarrow C$ is called
\emph{conilpotent} if the graded coalgebra without counit
$D/\bar w(k)$ is conilpotent.
 One can easily see that a conilpotent graded coalgebra has
a unique coaugmentation.

 The graded tensor coalgebra $\bigoplus_{n=0}^\infty U^{\ot n}$
over~$k$ is conilpotent for any graded vector space~$U$.
 For the reasons that are clear from the following lemma, this
coalgebra can be called the \emph{conilpotent graded coalgebra
cofreely cogenerated by} the graded vector space~$U$.
 More generally, the $\R$\+free graded coalgebra
$\bigoplus_n\U^{\ot n}$ is the $\R$\+free graded coalgebra with
conilpotent reduction modulo~$\m$ cofreely cogenerated by
the free graded $\R$\+contramodule~$\U$.

 On the other hand, the $\R$\+free graded algebra
$\bigoplus_n\U^{\ot n}$ is freely generated, just as an $\R$\+free
graded algebra, by the free graded $\R$\+contramodule~$\U$.

\begin{lem}  \label{free-algebra-morphisms}
\textup{(a)} Let\/ $\B$ be an\/ $\R$\+free graded algebra and\/ $\U$
be a free graded\/ $\R$\+contramodule.
 Then morphisms of\/ $\R$\+free graded algebras\/ $\bigoplus_{n=0}^
\infty \U^{\ot n}\rarrow\B$ are determined by their restrictions
to the component\/ $\U^{\ot1}\simeq\U$.
 Conversely, any homogeneous\/ $\R$\+contramodule morphism\/
$\U\rarrow\B$ of degree~$0$ gives rise to a morphism of\/
$\R$\+free graded algebras\/ $\bigoplus_{n=0}^\infty
\U^{\ot n}\rarrow\B$. \par
\textup{(b)}
 Let\/ $\C$ be an\/ $\R$\+free graded coalgebra and\/ $\U$ be a free
graded\/ $\R$\+contramodule.
 Then morphisms of\/ $\R$\+free graded coalgebras\/ $\C\rarrow
\bigoplus_{n=0}^\infty\U^{\ot n}$ are determined by
their projections to the component\/ $\U^{\ot1}\simeq\U$.
 Conversely, if the graded coalgebra\/ $\C/\m\C$ is conilpotent with
the coaugmentation\/ $k\rarrow\C/\m\C$, then a homogeneous\/
$\R$\+contramodule morphism\/ $\C\rarrow\U$ of degree~$0$ gives
rise to a morphism of\/ $\R$\+free graded coalgebras\/
$\C\rarrow\bigoplus_{n=0}^\infty\U^{\ot n}$ if and only if
the composition $k\rarrow\C/\m\C\rarrow\U/\m\U$ vanishes.
\end{lem}

\begin{proof}
 We will only prove part~(b), as the proof of part~(a) is similar
but much simpler.
 Since the map $\bigoplus_{n=0}^\infty\U^{\ot n}\rarrow
\prod_{n=0}^\infty\U^{\ot n}$ is injective, a morphism
$\C\rarrow\bigoplus_{n=2}^\infty\U^{\ot n}$ is determined by
its projections to $\U^{\ot n}$.
 For an $\R$\+free graded coalgebra morphism, the component
$\C\rarrow\U^{\ot n}$ is equal to the composition of the iterated
comultiplication map $\C\rarrow\C^{\ot n}$ and the $n$\+th
tensor power $\C^{\ot n}\rarrow\U^{\ot n}$ of
the component $\C\rarrow\U$.
 This proves the first assertion; to prove the ``only if'' part
of the second one, it suffices to notice that it holds for
$\R=k$ and apply the reduction modulo~$\m$.

 Assuming that $\C/\m\C$ is conilpotent, let us show that
an $\R$\+contramodule morphism $\C\rarrow\U$ of degree~$0$
for which the composition $k\rarrow\C/\m\C\rarrow\U/\m\U$
is zero gives rise to an $\R$\+free graded coalgebra morphism
into $\bigoplus_{n=0}^\infty\U^{\ot n}$.
 For any family of free $\R$\+contramodules $\V_\alpha$, one has
$\bigoplus_\alpha\V_\alpha=\limpr_\I\bigoplus_\alpha
\V_\alpha/\I\V_\alpha$, where the projective limit is taken over
all open ideals $\I\subset\R$.
 Hence it suffices to consider the case of a discrete Artinian
local ring $R$ in place of~$\R$.

 Let $m$ be the maximal ideal of~$R$; let $C$ be an $R$\+free
graded coalgebra such that $C/mC$ is conilpotent; let $U$ be
a free graded $R$\+module endowed with a graded $R$\+module
morphism $C\rarrow U$ such that the composition $k\rarrow C/mC
\rarrow U/mU$ vanishes.
 Given an element $c\in C$, we have to show that the composition
$C\rarrow C^{\ot n+1}\rarrow U^{\ot n+1}$ annihilates~$c$ for
$n$~large enough.

 Let $N\subset C$ be a free $R$\+submodule with one generator
such that the quotient module $C/N$ is free and the image of
the map $N/mN\rarrow C/mC$ coincides with the image of
$k\rarrow C/mC$ (see the proof of Lemma~\ref{unit-split}(b)).
 Then the image of the composition $N\rarrow C\rarrow U$ is
contained in~$mU$.
 Pick an integer $i\ge1$ such that $c$~is annihilated by
the composition $C\rarrow C/mC\rarrow ((C/mC)/k)^{\ot i+1}$
of the natural surjection with the iterated comultiplication map.
 Then the image of~$c$ in $C^{\ot i+1}$ belongs to the sum of
$mC^{\ot i+1}$ and $\sum_{j=0}^iC^{\ot j}\ot_RN\ot_RC^{\ot j-i}$.

 Repeat this procedure by applying the high enough iterated
comultiplication maps simultaneously to, e.~g., the first and
the last tensor factors in $C^{\ot i+1}$, etc.
 Proceeding in this way, we can find an integer~$l\ge1$ such that
the image of~$c$ in $C^{\ot l+1}$ belongs to the sum of
$m^rC^{\ot l+1}$ and all the tensor products of $l+1$ factors,
$r$~of which are the submodules $N$ and the remaining ones are
the whole of $C$, for any given~$r\ge1$.
 If $m^r=0$, it follows that the image of~$c$ in $U^{\ot l+1}$
vanishes.
\end{proof}

 The following constructions repeat those
of~\cite[Section~6.1]{Pkoszul}; the only difference is that 
$k$\+vector spaces are replaced with free $\R$\+contramodules.

 Let $\B=(\B,d,h)$ be an $\R$\+free CDG\+algebra; we assume that
$\B\ne0$.
 Let $v\:\B\rarrow\R$ be a homogeneous retraction onto
the image of the unit map $\R\rarrow\B$, i.~e., a morphism of
graded $\R$\+contramodules such that the composition
$\R\rarrow\B\rarrow\R$ is the identity map.
 Set $\V=\ker v\subset\B$, so $\B=\R\oplus\V$ as a graded
$\R$\+contramodule.
 Using this direct sum decomposition, we can split the multiplication
map $m\:\V\ot^\R\V\rarrow\B$, the differential $d\:\V\rarrow\B$,
and the curvature element $h\in\B$ into the components
$m=(m_\V,m_\R)$, \ $d=(d_\V,d_\R)$, and $h=(h_\V,h_\R)$, where
$m_\V\:\V\ot^\R\V\rarrow\V$, \ $m_\R\:\V\ot^\R\V\rarrow\R$, \
$d_\V\:\V\rarrow\V$, \ $d_\R\:\V\rarrow\R$, \ $h_\V\in\V$, and
$h_\R\in\R$.
 Notice that the restrictions of $m$ and~$d$ to the direct summands
$\R\ot^\R\V$, \ $\V\ot^\R\R$, \ $\R\ot^\R\R$, and $\R$ are uniquely
determined by the axioms of a graded algebra and its derivation.
 One has $h_\R=0$ for the dimension reasons unless $2=0$
in~$\Gamma$.

 Set $\B_+=\B/\R$.
 Let $\Br(\B)=\bigoplus_{n=0}^\infty\B_+^{\ot n}$ be the tensor
coalgebra of the free graded $\R$\+contramodule~$\B_+$.
 The $\R$\+free coalgebra $\Br(\B)$ is endowed with
the induced grading $|b_1\ot\dotsb\ot b_n| = |b_1|+\dotsb+|b_n|$,
the tensor grading~$n$, and the total grading $|b_1|+\dotsb+|b_n|-n$.
 All the gradings here are understood as direct sum decompositions
in $\R\contra^\free$.
 In the sequel, the total grading on $\Br(\B)$ will be generally
presumed.
 Equivalently, one can define the $\R$\+free (totally) graded
coalgebra $\Br(\B)$ as the tensor coalgebra of the free graded
$\R$\+contramodule $\B_+[1]$.

 Let $d_{\Br}$ be the odd coderivation of degree~$1$ on
$\Br(\B)$ whose compositions with the projection $\Br(\B)\rarrow\B_+$
are given by the rules $b_1\ot\dotsb\ot b_n\mpsto 0$ for $n\ge3$, \
$b_1\ot b_2\mpsto (-1)^{|b_1|+1}m_\V(b_1\ot b_2)$, \
$b\mpsto -d_\V(b)$, and $1\mpsto h_\V$, where $\B_+$ is identified
with $\V$ and $1\in\B_+^{\ot0}$.
 Let $h_{\Br}\:\Br(\B)\rarrow\R$ be the linear function given by
the formulas $h_{\Br}(b_1\ot\dotsb b_n)=0$ for $n\ge3$, \
$h_{\Br}(b_1\ot b_2)=(-1)^{|b_1|+1}h_\R(b_1\ot b_2)$, \
$h_{\Br}(b)=-d_\R(b)$, and $h_{\Br}(1)=h_\R$.
 Then the $\R$\+free graded coalgebra $\Br(\B)$ endowed with
the coderivation $d_{\Br}$ and the curvature linear function
$h_{\Br}$ is a CDG\+coalgebra.
 We will denote it $\Br_v(\B)$ and call the \emph{bar-construction}
of an $\R$\+free CDG\+algebra $\B$ endowed with a homogeneous
retraction $v\:\B\rarrow\R$ of the unit map.

 Given an $\R$\+free CDG\+algebra $\B$, changing a retraction
$v\:\B\rarrow\R$ to another one $v'\:\B\rarrow\R$ given by
the formula $v'(b)=v(b)+\alpha(b)$ leads to an isomorphism of
$\R$\+free CDG\+coalgebras $(\id,a)\:\Br_v(\B)\rarrow\Br_{v'}(\B)$,
where the linear function $a\:\Br(\B)\rarrow\R$ of degree~$1$
is obtained as the composition of the natural projection $\Br(\B)
\rarrow \B_+^{\ot 1}\simeq\B_+$ and the linear function
$\alpha\:\B_+\rarrow\R$ of degree~$0$.

 A morphism of $\R$\+free CDG\+algebras $(f,a)\:\B\rarrow\A$
is said to be \emph{weakly strict} if the element $a\in\A^1$
belongs to $\m\A_1$.
 In particular, if $\A$ and $\B$ are wcDG\+algebras, then, by
the definition, $(f,a)$ is a wcDG\+algebra morphism if and only
if it is a weakly strict CDG\+algebra morphism.

 To a weakly strict isomorphism of $\R$\+free CDG\+algebras
$(\id,a)\:(\B,d',h')\rarrow(\B,d,h)$ one can assign an isomorphism
of the corresponding bar-constructions of the form
$(f_{\Br},a_{\Br})\:\Br_v(\B,d',h')\rarrow\Br_v(\B,d,h)$
constructed as follows.
 Let $a_\V\in\V$ and $a_\R\in\R$ be the components of the element
$a\in\B$ with respect to the direct sum decomposition
$\B=\R\oplus\V$.

 Then the automorphism~$f_{\Br}$ of the $\R$\+free graded coalgebra
$\Br(\B)$ is defined by the rule that the composition of~$f_{\Br}$
with the projection $\Br(\B)\rarrow\B_+^{\ot 1}\simeq\B_+$
is equal to the sum of the same projection and minus the composition
of the projection $\Br(\B)\rarrow\B_+^{\ot 0}\simeq\R$
with the map $a_\V\:\R\rarrow\B_+$.
 A unique such automorphism~$f_{\Br}$ exists by
Lemma~\ref{free-algebra-morphisms}(b).
 The linear function $a_{\Br}\:\Br(\B)\rarrow\R$ is equal to
the composition of the projection $\Br(\B)\rarrow\R$ with
the map $a_\R\:\R\rarrow\R$.
 Notice that $a_\R$ and $a_{\Br}$ can be only nonzero when
$1=0$ in $\Gamma$, which can only happen when $2=0$ in~$\R$
(see~\cite[Section~1.1]{PP2}).

 Consequently, there is a functor from the category of
$\R$\+free CDG\+algebras and weakly strict morphisms between
them to the category of $\R$\+free CDG\+coalgebras assigning
to a CDG\+algebra $\B$ its bar-construction $\Br_v(\B)$.
 This functor takes wcDG\+algebras $\A$ over $\R$ to
$\R$\+free CDG\+coalgebras $\C=\Br(\A)$ whose reductions
$\C/\m\C$ are \emph{conilpotent CDG\+coalgebras} in the sense
of~\cite[Sections~6.1 and~6.4]{Pkoszul}.

 Recall the definition of the latter notion: a CDG\+coalgebra $C$
over a field~$k$ is called \emph{conilpotent} if it is conilpotent
as graded coalgebra and the homogeneous coaugmentation morphism
$\bar w\:k\rarrow C$ satisfies the equations $d\circ\bar w = 0 =
h\circ \bar w$ of compatibility with the CDG\+coalgebra structure.
 The definition of the category of conilpotent CDG\+coalgebras
requires a little care when the field~$k$ has characteristic~$2$:
a morphism of conilpotent CDG\+coalgebras $(f,a)\:C\rarrow D$ is 
a morphism of CDG\+coalgebras such that $a\circ\bar w=0$.

 Let $\C=(\C,d,h)$ be an $\R$\+free CDG\+coalgebra; we assume that
$\C\ne0$.
 Let $w\:\R\rarrow\C$ be a homogeneous section of the counit map
$\C\rarrow\R$, i.~e., a morphism of graded $\R$\+contramodules
such that the composition $\R\rarrow\C\rarrow\R$ is the identity map.
 Set $\W=\coker w$, so $\C=\R\oplus\W$ as a graded $\R$\+contramodule.
 Using this direct sum decomposition, we can split
the comultiplication map $\mu\:\C\rarrow\W\ot^\R\W$,
the differential $d\:\C\rarrow\W$, and the curvature linear function
$h\:\C\rarrow\R$ into the components $\mu=(\mu_\W,\mu_\R)$, \
$d=(d_\W,d_\R)$, and $h=(h_\W,h_\R)$, where
$\mu_\W\:\W\rarrow\W\ot^\R\W$, \ $\mu_\R\in\W\ot^\R\W$, \
$d_\W\:\W\rarrow\W$, \ $d_\R\in\W$, \ $h_\W\:\W\rarrow\R$, and
$h_\R\in\R$.
 Notice that compositions of $\mu$ and~$d$ with the projections
of $\C\ot^\R\C$ onto $\R\ot^\R\W$, \ $\W\ot^\R\R$, \ $\R\ot^\R\R$
and $\C$ onto $\R$ are uniquely determined by the axioms of
a graded coalgebra and a coderivation.
 One has $h_\R=0$ for the dimension reasons if $2\ne0$ in~$\Gamma$.

 Set $\C_+=\ker(\C\to\R)$ to be the kernel of the counit map.
 Let $\Cb(\C)=\bigoplus_{n=0}^\infty\C_+^{\ot n}$ be the tensor
algebra of the free graded $\R$\+contramodule~$\C_+$.
 The $\R$\+free algebra $\Cb(\C)$ is endowed with the induced
grading $|c_1\ot\dotsb\ot c_n|=|c_1|+\dotsb+|c_n|$, the tensor
grading~$n$, and the total grading $|c_1|+\dotsb+|c_n|+n$.
 All the gradings here are understood as direct sum decompositions
in $\R\contra^\free$.
 In the sequel, the total grading on $\Cb(\C)$ will be generally
presumed.
 Equivalently, one can define the $\R$\+free (totally) graded
coalgebra $\Cb(\C)$ as the tensor coalgebra of the free graded
$\R$\+contramodule $\C_+[-1]$.

 Let $d_{\Cb}$ be the odd derivation of degree~$1$ on $\Cb(\B)$
whose restriction to $\C_+\subset\Cb(\C)$ is given by
the formula $d(c) = (-1)^{|c_{(1,\W)}|+1}c_{(1,\W)}\ot c_{(2,\W)}
- d_\W(c) + h_\W(c)$, where $\C_+$ is identified with $\W$ and
$\mu_\W(c)=c_{(1,\W)}\ot c_{(2,\W)}$.
 Let $h_{\Cb}\in\Cb(\C)$ be the element given by the formula
$h_{\Cb} = (-1)^{|\mu_{(1,\R)}|+1}\mu_{(1,\R)}\ot\mu_{(2,\R)} - d_\R
+ h_\R$, where $\mu_\R=\mu_{(1,\R)}\ot\mu_{(2,\R)}$.
 Then the $\R$\+free graded algebra $\Cb(\C)$ endowed with
the derivation $d_{\Cb}$ and the curvature element $h_{\Cb}$ is
a CDG\+algebra.
 We will denote it $\Cb_w(\C)$ and call the \emph{cobar-construction}
of an $\R$\+free CDG\+coalgebra $\C$ endowed with a homogeneous
section $w\:\R\rarrow\C$ of the counit map.

 Given an $\R$\+free CDG\+coalgebra $\C$, changing a section
$w\:\R\rarrow\C$ to another one $w'\:\R\rarrow\C$
given by the rule $w'(1)=w(1)+\alpha$ leads to an isomorphism of
$\R$\+free CDG\+algebras $(\id,a)\:\Cb_{w'}(\C)\rarrow\Cb_w(\C)$,
where $a\in\Cb(\C)$ is the element of degree~$1$ corresponding to
$\alpha\in\C_+\subset\Cb(\C)$.

 To an isomorphism of $\R$\+free CDG\+coalgebras
$(\id,a)\:(\C,d,h)\rarrow(\C,d',h')$ one can assign an isomorphism
of the corresponding cobar-constructions of the form
$(f_{\Cb},a_{\Cb})\:\Cb_w(\C,d,h)\rarrow\Cb_w(\C,d',h')$
constructed as follows.
 Let $a_\W\:\W\rarrow\R$ and $a_\R\in\R$ be the components of
the linear function $a\:\C\rarrow\R$ with respect to the direct
sum decomposition $\C=\R\oplus\W$.
 Then the automorphism~$f_{\Cb}$ of the $\R$\+free graded algebra
$\Cb(\C)$ is given by the rule $c\mpsto c-a_\W(c)$, where
$c\in\C_+^{\ot 1}\simeq\C_+$.
 The element $a_{\Cb}\in\Cb(\C)$ is equal to $a_\R\in\R\simeq
\C_+^{\ot0}$.
 Notice that $a_\R$ and $a_{\Cb}$ are always zero unless $1=0$ in
$\Gamma$, which implies $2=0$ in~$\R$.

 Consequently, there is a functor from the category of $\R$\+free
CDG\+coalgebras to the category of $\R$\+free CDG\+algebras
assigning to a CDG\+coalgebra $\C$ its cobar-construction
$\Cb_w(\C)$.
 When the CDG\+coalgebra $\C/\m\C$ is \emph{coaugmented}
\cite[Section~6.1]{Pkoszul}, one can pick a section $w\:\R\rarrow\C$
such that its reduction $\bar w\:k\rarrow\C/\m\C$ is the
coaugmentation.
 This makes $\C\mpsto\Cb_w(\C)$ a functor from the category of
$\R$\+free CDG\+coalgebras $(\C,d,h)$ with coaugmented (and, in
particular, conilpotent) reductions $(\C/\m\C\;d/\m d\;h/\m h)$
to the category of wcDG\+algebras over~$\R$.

 Here we recall that a CDG\+coalgebra $C$ over~$k$ is said to be
\emph{coaugmented} if it is endowed with a morphism of CDG\+coalgebras
$(\bar w,0)\:(k,0,0)\rarrow(C,d,h)$, or equivalently, a morphism of
graded coalgebras $\bar w\:k\rarrow C$ such that
$d\circ \bar w = 0 = h\circ \bar w$.
 A morphism of coaugmented CDG\+coalgebras $(f,a)\:C\rarrow D$
is a morphism of CDG\+coalgebras such that $a\circ\bar w = 0$
(a condition nontrivial in characteristic~$2$ only).

\subsection{Twisting cochains}  \label{twisting-cochains-sect}
 Let $\C=(\C,d_\C,h_\C)$ be an $\R$\+free CDG\+coalgebra and
$\B=(\B,d_\B,h_\B)$ be an $\R$\+free CDG\+algebra.
 Introduce a CDG\+algebra structure on the graded $\R$\+contramodule
of homogeneous $\R$\+contramodule homomorphisms $\Hom^\R(\C,\B)$ in
the following way.
 The multiplication in $\Hom^\R(\C,\B)$ is defined as the composition
$\Hom^\R(\C,\B)\ot^\R\Hom^\R(\C,\B)\rarrow\Hom^\R(\C\ot^\R\C\;
\B\ot^\R\B)\rarrow\Hom^\R(\C,\B)$, the second map being induced by
the comultiplication in $\C$ and the multiplication in~$\B$.
 The left-right and sign rule is $(fg)(c)=(-1)^{|g||c_{(1)}|}
f(c_{(1)})g(c_{(2)})$.
 The differential is given by the conventional rule
$d(f)(c)=d_\B(f(c))-(1)^{|f|}f(d_\C(c))$.
 The curvature element in $\Hom^\R(\C,\B)$ is defined by the
formula $h(c)=\varepsilon(c)h_\B-h_\C(c)e$, where $\varepsilon\:
\C\rarrow\R$ is the counit map and $e\in\B$ is the unit element.

 A homogeneous $\R$\+contramodule map $\tau\:\C\rarrow\B$ of
degree~$1$ is called a \emph{twisting cochain} if it satisfies
the equation $\tau^2+d\tau+h=0$ with respect to the above-defined
CDG\+algebra structure on $\Hom^\R(\C,\B)$
(see~\cite[Section~6.2]{Pkoszul} and the references therein).
 Given a morphism of $\R$\+free CDG\+algebras $(f,a)\:\B\rarrow\A$
and a twisting cochain $\tau\:\C\rarrow\B$, one constructs
the twisting cochain $(f,a)\circ\tau\:\C\rarrow\A$ by the rule
$(f,a)\circ\tau = f\circ\tau + a\circ\varepsilon_\C$.
 Given a morphism of $\R$\+free CDG\+coalgebras
$(g,a)\:\D\rarrow\C$ and a twisting cochain $\tau\:\C\rarrow\B$,
one constructs the twisting cochain $\tau\circ(g,a)\:\D\rarrow\B$
by the rule $\tau\circ(g,a)=\tau\circ g - e_\B\circ a$.

 Let $\C$ be an $\R$\+free CDG\+coalgebra and $w\:\R\rarrow\C$ be
a homogeneous section of the counit map.
 Then the composition $\tau=\tau_{\C,w}\:\C\rarrow\Cb(\C)$ of
the maps $\C\rarrow\W\simeq\C_+\simeq\C_+^{\ot 1}\rarrow\Cb(\C)$
is a twisting cochain for $\C$ and $\Cb_w(\C)$.
 Let $\B$ be an $\R$\+free CDG\+algebra and $v\:\C\rarrow\R$ be
a homogeneous retraction of the unit map.
 Then minus the composition $\Br(\B)\rarrow\B_+^{\ot 1}\simeq
\B_+\simeq\V\rarrow\B$ is a twisting cochain $\tau=\tau_{\B,v}\:
\Br_v(\B)\rarrow\B$ for $\Br_v(\B)$ and~$\B$.

 Let $\tau\:\C\rarrow\B$ be a twisting cochain for an $\R$\+free
CDG\+coalgebra $\C$ and an $\R$\+free CDG\+algebra~$\B$.
 Then for any $\R$\+free left CDG\+module $\M$ over $\B$ there is
a natural structure of $\R$\+free left CDG\+comodule over $\C$ on
the tensor product $\C\ot^\R\M$.
 Namely, the coaction of $\C$ in $\C\ot^\R\M$ is induced by
the left coaction of $\C$ in itself, while the differential on
$\C\ot^\R\M$ is given by the formula
$d(c\ot x) = d(c)\ot x + (-1)^{|c|}c\ot d(x) + (-1)^{|c_{(1)}|}
c_{(1)}\ot \tau(c_{(2)})x$, where $c\mpsto c_{(1)}\ot c_{(2)}$
denotes the comultiplication in $\C$, while $b\ot x\mpsto bx$ 
is the left action of $\B$ in~$\M$.
 We will denote the free graded $\R$\+contramodule $\C\ot^\R\M$
with this CDG\+comodule structure by $\C\ot^\tau\M$.

 Furthermore, for any $\R$\+free left CDG\+comodule $\N$ over $\C$
there is a natural structure of $\R$\+free left CDG\+module over $\B$
on the tensor product $\B\ot^\R\N$.
 Namely, the action of $\B$ in $\B\ot^\R\N$ is induced by the left
action of $\B$ in itself, while the differential on $\B\ot^\R\N$
is given by the formula $d(b\ot y) = d(b)\ot y + (-1)^{|b|}b\ot d(y)
- (-1)^{|b|}b\tau(y_{(-1)})\ot y_{(0)}$, where $y\mpsto y_{(-1)}
\ot y_{(0)}$ denotes the left coaction of $\C$ in $\N$, while
$b\ot b'\mpsto bb'$ is the multiplication in~$\B$.
 We will denote the free graded $\R$\+contramodule $\B\ot^\R\N$
with this CDG\+module structure by $\B\ot^\tau\N$.

 The correspondences assigning to an $\R$\+free CDG\+module $\M$
over $\B$ the $\R$\+free CDG\+comodule $\C\ot^\tau\M$ over $\C$ and
to an $\R$\+free CDG\+comodule $\N$ over $\C$ the $\R$\+free
CDG\+module $\B\ot^\tau\N$ over $\B$ can be extended to
DG\+functors whose action on morphisms is given by the standard
formulas $f_*(c\ot x) = (-1)^{|f||c|}c\ot f(x)$ and
$g_*(b\ot y) = (-1)^{|g||b|}b\ot g(y)$.
 The DG\+functor $\C\ot^\tau{-}\:\B\mod\Rfr\rarrow\C\comod\Rfr$ is
right adjoint to the DG\+functor $\B\ot^\tau{-}\:\C\comod\Rfr
\rarrow\B\mod\Rfr$.

 Similarly, for any $\R$\+free right CDG\+module $\M$ over $\B$
there is a natural structure of right CDG\+comodule
over $\C$ on the tensor product $\M\ot^\R\C$.
 The coaction of $\C$ in $\M\ot^\R\C$ is induced by the right
coaction of $\C$ in itself, and the differential on $\M\ot^\R\C$
is given by the formula $d(x\ot c) = d(x)\ot c + (-1)^{|x|}
x\ot d(c) - (-1)^{|x|}x\tau(c_{(1)})\ot c_{(2)}$.
 We will denote the free graded $\R$\+contramodule $\M\ot^\R\C$
with this CDG\+comodule structure by $\M\ot^\tau\C$.
 For any $\R$\+free right CDG\+comodule $\N$ over $\C$ there is
a natural structure of right CDG\+module over $\B$ on the tensor
product $\N\ot^\R\B$.
 Namely, the action of $\B$ in $\N\ot^\R\B$ is induced by
the right action of $\B$ in itself, and the differential on
$\N\ot^\R\B$ is given by the formula $d(y\ot b) = d(y)\ot b +
(-1)^{|y|}y\ot d(b) + (-1)^{|y_{(0)}|}y_{(0)}\ot\tau(y_{(1)})b$,
where $y\mpsto y_{(0)}\ot y_{(1)}$ denotes the right coaction
of $\C$ in~$\N$.
 We will denote the free graded $\R$\+contramodule $\N\ot^\R\B$
with this CDG\+module structure by $\N\ot^\tau\B$.

 For any $\R$\+contramodule left CDG\+module $\P$ over $\B$ there
is a natural structure of left CDG\+contramodule over $\C$ on
the graded $\R$\+contramodule $\Hom^\R(\C,\P)$.
 The contraaction of $\C$ in $\Hom^\R(\C,\P)$ is induced by
the right coaction of $\C$ in itself as explained
in Sections~\ref{r-free-graded-co} and~\ref{non-adj-graded-co}
(for the sign rule, see~\cite[Section~2.2]{Pkoszul}).
 The differential on $\Hom^\R(\C,\P)$ is given by the formula
$d(f)(c) = d(f(c)) - (-1)^{|f|}f(d(c)) + (-1)^{|f||c_{(1)}|}
\tau(c_{(1)})f(c_{(2)})$ for $f\in\Hom^\R(\C,\P)$.
 Here the third summand is interpreted as the $\R$\+contramodule
morphism $\Hom^\R(\C,\P)\rarrow\Hom^\R(\C,\P)$ corresponding
to the $\R$\+contramodule morphism $\C\ot^\R\Hom^\R(\C,\P)\rarrow\P$
defined as the composition $\C\ot^\R\Hom^\R(\C,\P)\rarrow\C\ot^\R\C
\ot^\R\Hom^\R(\C,\P)\rarrow\B\ot^\R\C\ot^\R\Hom^\R(\C,\P)\rarrow
\B\ot^\R\P\rarrow\P$.
 We will denote the graded $\R$\+contramodule $\Hom^\R(\C,\P)$ 
with this CDG\+contramodule structure by $\Hom^\tau(\C,\P)$.

 For any $\R$\+contramodule left CDG\+contramodule $\Q$ over $\C$
there is a natural structure of left CDG\+module over $\B$ on
the graded $\R$\+contramodule $\Hom^\R(\B,\Q)$.
 The action of $\B$ in $\Hom^\R(\B,\Q)$ is induced by the right
action of $\B$ in itself (see Section~\ref{r-free-graded}).
 The differential on $\Hom^\R(\B,\Q)$ is given by the formula
$d(f)(b) = d(f(b)) - (-1)^{|f|}f(d(b)) - \pi(c\mapsto
(-1)^{|f|+|c||b|}f(\tau(c)b)$, where $\pi$ denotes the contraaction
map $\Hom^\R(\C,\Q)\rarrow\Q$.
 Here the third summand is interpreted as the $\R$\+contramodule
morphism $\Hom^\R(\B,\Q)\rarrow\Hom^\R(\B,\Q)$ corresponding to
the $\R$\+contramodule morphism $\B\ot^\R\Hom^\R(\B,\Q)\rarrow\Q$
defined as the composition $\B\ot^\R\Hom^\R(\B,\Q)\rarrow
\Hom^\R(\C,\B)\ot^\R\Hom^\R(\B,\Q)\rarrow\Hom^\R(\C,\Q)\rarrow\Q$,
where the morphism $\B\rarrow\Hom^\R(\C,\B)$ corresponds to
the composition $\C\ot^\R\B\rarrow\B\ot^\R\B\rarrow\B$.
 We denote the graded $\R$\+contramodule $\Hom^\R(\B,\Q)$ with
this CDG\+module structure by $\Hom^\tau(\B,\Q)$.
{\emergencystretch=1em\par}

 The correspondences assigning to an $\R$\+contramodule
CDG\+module $\P$ over $\B$ the $\R$\+contramodule CDG\+contramodule
$\Hom^\tau(\C,\P)$ over $\C$ and to an $\R$\+contramodule
CDG\+contramodule $\Q$ over $\C$ the $\R$\+contramodule
CDG\+module $\Hom^\tau(\B,\Q)$ over $\B$ can be extended to
DG\+functors whose action on morphisms is given by the standard
formula $g_*(f)=g\circ f$ for $f\:\C\rarrow\P$ or $f\:\B\rarrow\Q$.
 The DG\+functor $\Hom^\tau(\C,{-})\:\B\mod\Rctr\rarrow
\C\contra\Rctr$ is left adjoint to the DG\+functor
$\Hom^\tau(\B,{-})\:\C\contra\Rctr\rarrow\B\mod\Rctr$.

 Analogously, for any $\R$\+comodule left CDG\+module $\cM$ over $\B$
there is a natural structore of $\R$\+comodule left CDG\+comodule
over $\B$ on the contratensor product $\C\ocn_\R\cM$.
 Namely, the coaction of $\C$ in $\C\ocn_\R\cM$ is induced by
the left coaction of $\C$ in itself, and the differential on
$\C\ocn_\R\cM$ is defined in terms of the differentials on $\C$
and $\cM$, the twisting cochain~$\tau$, the comultiplication in $\C$,
and the action of $\B$ in $\cM$ by the formula above.
 We will denote the graded $\R$\+comodule $\C\ocn_\R\cM$ with
this CDG\+comodule structure by $\C\ocn^\tau\cM$.

 For any $\R$\+comodule left CDG\+comodule $\cN$ over $\C$ there is
a natural structure of $\R$\+comodule left CDG\+module over $\B$
on the contratensor product $\B\ocn_\R\cN$.
 The action of $\B$ in $\B\ocn_\R\cN$ is induced by the left
action of $\B$ in itself, and the differential on $\B\ocn_\R\cN$
is defined in terms of the differentials on $\B$ and $\cN$,
the twisting cochain~$\tau$, the multiplication in $\B$, and
the coaction of $\C$ in $\cN$ by the formula above.
 We will denote the graded $\R$\+comodule $\B\ocn_\R\cN$ with
this CDG\+module structure by $\B\ocn^\tau\cN$.

 The correspondences assigning to an $\R$\+comodule CDG\+module
$\cM$ over $\B$ the $\R$\+comodule CDG\+comodule $\C\ocn^\tau\cM$
over $\C$ and to an $\R$\+comodule CDG\+comodule $\cN$ over $\C$
the $\R$\+comodule CDG\+module $\B\ocn^\tau\cN$ over $\B$ can be
extended to DG\+functors whose action on morphisms is given
by the standard formulas above.
 The DG\+functor $\C\ocn^\tau{-}\:\B\mod\Rco\rarrow\C\comod\Rco$
is right adjoint to the DG\+functor $\B\ocn^\tau\:\C\comod\Rco
\rarrow\B\mod\Rco$.

 Similarly, for any $\R$\+comodule right CDG\+module $\cM$ over
$\B$ there is a natural structure of right CDG\+comodule over $\C$
on the contratensor product $\cM\ocn_\R\C$ (see
Section~\ref{non-adj-graded} for the definition of such
contratensor product).
 The coaction of $\C$ in $\cM\ocn_\R\C$ is induced by the right
coaction of $\C$ in itself, and the differential is given by
the formula above.
 We will denote the graded $\R$\+comodule $\cM\ocn_\R\C$ with
this CDG\+comodule structure by $\cM\ocn^\tau\C$.
 For any $\R$\+comodule right CDG\+comodule $\cN$ over $\C$ there is
a natural structure of right CDG\+module over $\B$ on
the contratensor product $\cN\ocn_\R\B$.
 The action of $\B$ in $\cN\ocn_\R\B$ is induced by the right action
of $\B$ in itself, and the differential is given by the formula
above.
 We will denote the graded $\R$\+comodule $\cN\ocn_\C\B$ with this
CDG\+module structure by $\cN\ocn^\tau\B$.

 For any $\R$\+cofree left CDG\+module $\cP$ over $\B$ there is
a natural structure of left CDG\+contramodule over $\C$ on
the cofree graded $\R$\+comodule $\Ctrhom_\R(\C,\cP)$.
 The contraaction of $\C$ in $\Ctrhom_\R(\C,\cP)$ is induced by
the right coaction of $\C$ in itself as explained in
Section~\ref{r-cofree-graded-co}.
 The differential on $\Ctrhom_\R(\C,\cP)$ is defined in terms of
the differentials on $\C$ and $\cP$, the twisting cochain~$\tau$,
the comultiplication in $\C$, and the action of $\B$ in $\cP$
by the formula above.
 The twisting term is constructed as the $\R$\+comodule morphism
$\Ctrhom_\R(\C,\cP)\rarrow\Ctrhom_\R(\C,\cP)$ corresponding to
the $\R$\+comodule morphism $\C\ocn_\R\Ctrhom_\R(\C,\cP)\rarrow
\cP$ defined as the composition $\C\ocn_\R\Ctrhom_\R(\C,\cP)\rarrow
\C\ot^\R\C\ocn_\R\Ctrhom^\R(\C,\cP)\rarrow\B\ot^\R\C\ocn_\R
\Ctrhom_\R(\C,\cP)\rarrow\B\ocn_\R\cP\rarrow\cP$.
 We will denote the cofree graded $\R$\+comodule $\Ctrhom_\R(\C,\cP)$
with this CDG\+contramodule structure by $\Ctrhom^\tau(\C,\cP)$.

 For any $\R$\+cofree left CDG\+contramodule $\cQ$ over $\C$ there is
a natural structure of left CDG\+module over $\B$ on the cofree
graded $\R$\+comodule $\Ctrhom_\R(\B,\cQ)$.
 The action of $\B$ in $\Ctrhom_\R(\B,\cQ)$ is induced by the right
action of $\B$ in itself (see Section~\ref{r-cofree-graded}).
 The differential on $\Ctrhom_\R(\B,\cQ)$ is defined in terms of
the differentials on $\B$ and $\cQ$, the twisting cochain~$\tau$,
the multiplication in $\B$, and the contraaction of $\C$ in $\cQ$
by the formula above.
 The twisting term is constructed as the $\R$\+comodule morphism
$\Ctrhom_\R(\B,\cQ)\rarrow\Ctrhom_\R(\B,\cQ)$ corresponding to
the $\R$\+comodule morphism $\B\ocn_\R\Ctrhom_\R(\B,\cQ)\rarrow\cQ$
defined as the composition $\B\ocn_\R\Ctrhom_\R(\B,\cQ)\rarrow
\Hom^\R(\C,\B)\ocn_\R\Ctrhom_\R(\B,\cQ)\rarrow\Ctrhom_\R(\C,\cQ)
\rarrow\cQ$, where the morphism $\B\rarrow\Hom^\R(\C,\B)$ was
defined above.
 The natural morphism $\Hom^\R(\C,\B)\ocn_\R\Ctrhom_\R(\B,\cQ)
\rarrow\Ctrhom_\R(\C,\cQ)$ corresponds to the composition
$\Ctrhom_\R(\B,\cQ)\rarrow\Ctrhom_\R(\C\ot^\R\Hom^\R(\C,\B)\;\cQ)
\simeq\Ctrhom_\R(\Hom^\R(\C,\B),\Ctrhom_\R(\C,\cQ))$.
 We will denote the cofree graded $\R$\+co\-module $\Ctrhom_\R(\B,\cQ)$
with this CDG\+module structure by $\Ctrhom^\tau(\B,\cQ)$.

 The correspondences assigning to an $\R$\+cofree CDG\+module
$\cP$ over $\B$ the $\R$\+cofree CDG\+contramodule
$\Ctrhom^\tau(\C,\cP)$ over $\C$ and to an $\R$\+cofree
CDG\+contramodule $\cQ$ over $\C$ the $\R$\+cofree CDG\+module
$\Ctrhom^\tau(\B,\cQ)$ over $\B$ can be extended to DG\+functors
whose action on morphisms is given by the standard formula above.
 The DG\+functor $\Ctrhom^\tau(\C,{-})\:\B\mod\Rcof\rarrow
\C\contra\Rcof$ is left adjoint to the DG\+functor
$\Ctrhom^\tau(\B,{-})\:\C\contra\Rcof\rarrow\B\mod\Rcof$.

\subsection{Conilpotent duality}  \label{conilpotent-sect}
 Let $\A$ be a wcDG\+algebra over $\R$ and $\C$ be an $\R$\+free
CDG\+coalgebra such that the CDG\+coalgebra $\C/\m\C$ over~$k$
is conilpotent with the coaugmentation map $\bar w\:k\rarrow\C/\m\C$.
 Let $\tau\:\C\rarrow\A$ be a twisting cochain such that
the composition $k\rarrow\C/\m\C\rarrow\A/\m\A$ vanishes.

\begin{thm}  \label{acyclic-cochain-duality}
 Assume that the twisting cochain $\bar\tau=\tau/\m\tau\:
\C/\m\C\rarrow\A/\m\A$ is \emph{acyclic} in the sense
of~\textup{\cite[\emph{Section}~6.5]{Pkoszul}}.
 Then \par
\textup{(a)} the functors\/ $\C\ot^\tau{-}\:H^0(\A\mod\Rfr)
\rarrow H^0(\C\comod\Rfr)$ and\/ $\A\ot^\tau{-}\:H^0(\C\comod\Rfr)
\rarrow H^0(\A\mod\Rfr)$ induce functors
$$
 \sD^\si(\A\mod\Rfr)\lrarrow\sD^\co(\C\comod\Rfr)
$$
and
$$
 \sD^\co(\C\comod\Rfr)\lrarrow\sD^\si(\A\mod\Rfr),
$$
which are mutually inverse equivalences of triangulated categories; \par
\textup{(b)} the functors\/ $\Hom^\tau(\C,{-})\:H^0(\A\mod\Rctr)
\rarrow H^0(\C\contra\Rctr)$ and\/ $\Hom^\tau(\A,{-})\:
H^0(\C\contra\Rctr)\rarrow H^0(\A\mod\Rctr)$ induce functors
{\hbadness=2600
$$
 \sD^\si(\A\mod\Rctr)\lrarrow\sD^\ctr(\C\contra\Rctr)
$$
and }
$$
 \sD^\ctr(\C\contra\Rctr)\lrarrow\sD^\si(\A\mod\Rctr),
$$
which are mutually inverse equivalences of triangulated categories; \par
\textup{(c)} the functors\/ $\C\ocn^\tau{-}\:H^0(\A\mod\Rco)
\rarrow H^0(\C\comod\Rco)$ and\/ $\A\ocn^\tau{-}\:H^0(\C\comod\Rco)
\rarrow H^0(\A\mod\Rco)$ induce functors
$$
 \sD^\si(\A\mod\Rco)\lrarrow\sD^\co(\C\comod\Rco)
$$
and
$$
 \sD^\co(\C\comod\Rco)\lrarrow\sD^\si(\A\mod\Rco),
$$
which are mutually inverse equivalences of triangulated categories; \par
\textup{(d)} the functors\/ $\Ctrhom^\tau(\C,{-})\:
H^0(\A\mod\Rcof)\rarrow H^0(\C\contra\Rcof)$ and\/
$\Ctrhom^\tau(\A,{-})\:H^0(\C\contra\Rcof)\rarrow H^0(\A\mod\Rcof)$
induce functors
$$
 \sD^\si(\A\mod\Rcof)\lrarrow\sD^\ctr(\C\contra\Rcof)
$$
and
$$
 \sD^\ctr(\C\contra\Rcof)\lrarrow\sD^\si(\A\mod\Rcof),
$$
which are mutually inverse equivalences of triangulated categories; \par
\textup{(e)} the above equivalences of triangulated categories
form a commutative diagram with the equivalences of categories\/
$\Phi_\R=\Psi_\R^{-1}$ of Sections~\textup{\ref{r-cofree-semi}}
and~\textup{\ref{r-cofree-co-derived}}, $\boL\Phi_\R=\boR\Psi_\R^{-1}$
of Proposition~\textup{\ref{non-adj-r-co-contra}},
$\boL\Phi_\C=\boR\Psi_\C^{-1}$ of
Corollaries~\textup{\ref{r-free-derived-co-contra}}
and~\textup{\ref{r-cofree-derived-co-contra}}, $\boL\Phi_{\R,\C}=
\boR\Psi_{\R,\C}^{-1}$ of
Corollary~\textup{\ref{non-adj-derived-co-contra}},
the equivalences of semiderived categories from
Section~\textup{\ref{wcdg-semiderived}}, and the equivalences
of contra/coderived categories from
Theorem~\textup{\ref{non-adj-co-derived-comp}}.
 In other words, the following square diagrams of triangulated
equivalences are commutative:
$$
\begin{diagram}
\node{\llap{$\C\ot^\tau{-}$}\:\sD^\si(\A\mod\Rfr)}
\arrow{e,=} \arrow{s,lr,=}{\Phi_\R}{\boR\Psi_\R}
\node{\sD^\co(\C\comod\Rfr)\,\,\:\!\rlap{$\A\ot^\tau{-}$}}
\arrow{s,lr,=}{\Phi_\R}{\boR\Psi_\R} \\
\node{\llap{$\C\ocn^\tau{-}$}\:\sD^\si(\A\mod\Rco)}
\arrow{e,=}
\node{\sD^\co(\C\comod\Rco)\,\,\:\!\rlap{$\A\ocn^\tau{-}$}}
\end{diagram}
$$
$$
\begin{diagram}
\node{\llap{$\Hom^\tau(\C,{-})$}\:\sD^\si(\A\mod\Rctr)}
\arrow{e,=} \arrow{s,lr,=}{\boL\Phi_\R}{\Psi_\R}
\node{\sD^\co(\C\contra\Rctr)\,\,\:\!\rlap{$\Hom^\tau(\A,{-})$}}
\arrow{s,lr,=}{\boL\Phi_\R}{\Psi_\R} \\
\node{\llap{$\Ctrhom^\tau(\C,{-})$}\:\sD^\si(\A\mod\Rcof)}
\arrow{e,=}
\node{\sD^\co(\C\contra\Rcof)\,\,\:\!\rlap{$\Ctrhom^\tau(\A,{-})$}}
\end{diagram}
$$
$$
\begin{diagram}
\node{\llap{$\C\ot^\tau{-}$}\:\sD^\si(\A\mod\Rfr)}
\arrow{e,=} \arrow{s,=}
\node{\sD^\co(\C\comod\Rfr)\,\,\:\!\rlap{$\A\ot^\tau{-}$}}
\arrow{s,lr,=}{\boR\Psi_\C}{\boL\Phi_\C} \\
\node{\llap{$\Hom^\tau(\C,{-})$}\:\sD^\si(\A\mod\Rctr)}
\arrow{e,=} 
\node{\sD^\co(\C\contra\Rctr)\,\,\:\!\rlap{$\Hom^\tau(\A,{-})$}}
\end{diagram}
$$
$$
\begin{diagram}
\node{\llap{$\C\ocn^\tau{-}$}\:\sD^\si(\A\mod\Rco)}
\arrow{e,=} \arrow{s,=}
\node{\sD^\co(\C\comod\Rco)\,\,\:\!\rlap{$\A\ocn^\tau{-}$}}
\arrow{s,lr,=}{\boR\Psi_\C}{\boL\Phi_\C} \\
\node{\llap{$\Ctrhom^\tau(\C,{-})$}\:\sD^\si(\A\mod\Rcof)}
\arrow{e,=}
\node{\sD^\co(\C\contra\Rcof)\,\,\:\!\rlap{$\Ctrhom^\tau(\A,{-})$}}
\end{diagram}
$$
$$
\begin{diagram}
\node{\llap{$\C\ocn^\tau{-}$}\:\sD^\si(\A\mod\Rco)}
\arrow{e,=} \arrow{s,lr,=}{\boR\Psi_\R}{\boL\Phi_\R}
\node{\sD^\co(\C\comod\Rco)\,\,\:\!\rlap{$\A\ocn^\tau{-}$}}
\arrow{s,lr,=}{\boR\Psi_{\R,\C}}{\boL\Phi_{\R,\C}} \\
\node{\llap{$\Hom^\tau(\C,{-})$}\:\sD^\si(\A\mod\Rctr)}
\arrow{e,=} 
\node{\sD^\co(\C\contra\Rctr)\,\,\:\!\rlap{$\Hom^\tau(\A,{-})$}}
\end{diagram}
$$
\end{thm}

\begin{proof}
 The assertions about the existence of induced functors in
parts~(a\+d) do not depend on the acyclicity assumption on~$\bar\tau$;
the assertions about the induced functors being equivalences of
categories do.
 Part~(a): the functor $\A\ot^\tau{-}$ takes coacyclic $\R$\+free
CDG\+comodules over $\C$ to contractible wcDG\+modules over~$\A$.
 Indeed, whenever $\N$ is the total CDG\+comodule of an short
exact sequence of $\R$\+free CDG\+comodules, $\A\ot^\tau\N$ is
the total wcDG\+module of a short exact sequence of wcDG\+modules
that is split as a short exact sequence of graded $\A$\+modules.
 Furthermore, the functor $\C\ot^\tau{-}$ takes semiacyclic
$\R$\+free wcDG\+modules $\M$ to contractible CDG\+comodules
$\N=\C\ot^\tau\M$.
 Indeed, the $\R$\+free graded $\C$\+comodule $\N$ is injective,
and the CDG\+comodule $\N/\m\N$ over $\C/\m\C$ is contractible by
\cite[proofs of Theorems~6.3 and~6.5]{Pkoszul}, so it remains to
apply Lemma~\ref{contractible-reduction-co}.
 The cone of the adjunction map $\A\ot^\tau\C\ot^\tau\M\rarrow\M$
is semiacyclic for any $\R$\+free wcDG\+module $\M$, because
the reduction modulo~$\m$ of this cone is acyclic
by~\cite[Theorem~6.5(a)]{Pkoszul}.
 The cone of the adjunction map $\N\rarrow\C\ot^\tau\A\ot^\tau\N$
is coacyclic for any $\R$\+free CDG\+comodule $\N$ by the same
result from~\cite{Pkoszul} and according to
Corollary~\ref{r-free-co-acycl-reduction}.

 Part~(b): the functor $\Hom^\tau(\A,{-})$ takes contraacyclic
$\R$\+contramodule CDG\+con\-tramodules to contraacyclic
$\R$\+contramodule wcDG\+modules, since it preserves short exact
sequences and infinite products of $\R$\+contramodule
CDG\+contramodules.
 The functor $\Hom^\tau(\C,{-})$ takes contraacyclic
$\R$\+contramodule wcDG\+modules to contraacyclic $\R$\+cotramodule
CDG\+contramodules for the same reason.
 It also takes semiacyclic $\R$\+free wcDG\+modules to contractible
CDG\+contramodules, for the reasons explained in the proof of
part~(a).
 Hence the induced adjoint functors $\sD^\si(\A\mod\Rctr)\rarrow
\sD^\ctr(\C\contra\Rctr)$ and $\sD^\ctr(\C\contra\Rctr)\rarrow
\sD^\si(\A\mod\Rctr)$ exist.
 To check that they are mutually inverse equivalences, it suffices
to consider them as functors $\sD^\si(\A\mod\Rfr)\rarrow
\sD^\ctr(\C\contra\Rfr)$ and $\sD^\ctr(\C\contra\Rfr)\rarrow
\sD^\si(\A\mod\Rfr)$, restricting both constructions to
$\R$\+free wcDG\+modules and $\R$\+free CDG\+contramodules.
 Then one can argue as in the proof of part~(a) above.

 The proof of parts (c) and~(d) are similar to the proofs of parts
(b) and~(a), respectively.
 To prove part~(e), notice that for any $\R$\+free wcDG\+module
$\M$ over $\A$, there is a natural isomorphism of $\R$\+cofree
CDG\+comodules $\Phi_\R(\C\ot^\tau\M)\simeq\C\ocn^\tau\Phi_\R(\M)$
over~$\C$.
 For any $\R$\+cofree wcDG\+module $\cP$ over $\A$, there is
a natural isomorphism of $\R$\+free CDG\+contramodules
$\Psi_\R(\Ctrhom^\tau(\C,\cP))\simeq\Hom^\tau(\C,\Psi_\R(\cP))$
over~$\C$.
 For any $\R$\+free wcDG\+module $\P$ over $\A$, the functors
$\Phi_\C=\Psi_\C^{-1}$ transform the CDG\+contramodule
$\Hom^\tau(\C,\P)\in\C\contra\Rfr_\proj$ into the CDG\+comodule
$\C\ot^\tau\P\in\C\comod\Rfr_\inj$ and back.
 For any $\R$\+cofree wcDG\+module $\cM$ over $\A$, the functors
$\Phi_\C=\Psi_\C^{-1}$ transform the CDG\+contramodule
$\Ctrhom^\tau(\C,\cM)\in\C\contra\Rcof_\proj$ into
the CDG\+comodule $\C\ocn^\tau\cM\in\C\comod\Rcof_\inj$ and back.
 For any $\R$\+contramodule wcDG\+module $\M$ over $\A$, there is
a natural isomorphism of $\R$\+comodule CDG\+comodules
$\Phi_{\R,\C}(\Hom^\tau(\C,\M))\simeq\C\ocn^\tau\Phi_\R(\M)$.
 For any $\R$\+comodule wcDG\+module $\cP$ over $\A$, there is
a natural isomorphism of $\R$\+contramodule CDG\+contramodules
$\Psi_{\R,\C}(\C\ocn^\tau\cP)\simeq\Hom^\tau(\C,\Psi_\R(\cP))$.
\end{proof}

 Let $\A$ be a wcDG\+algebra over $\R$, $v\:\A\rarrow\R$ be
a homogeneous retraction onto the unit map, $\C=\Br_v(\A)$ be
the corresponding $\R$\+free CDG\+coalgebra, and
$\tau_{\A,v}\:\C\rarrow\A$ be the natural twisting cochain.

\begin{cor}  \label{bar-duality}
 All the assertions of Theorem~\textup{\ref{acyclic-cochain-duality}}
hold for the twisting cochain $\tau=\tau_{\A,v}$.
\end{cor}

\begin{proof}
 The twisting cochain $\bar\tau_{\A,v}=\tau_{\A,v}/\m\tau_{\A,v}$
is acyclic, as one can see by
comparing~\cite[Theorems~6.3 and~6.5]{Pkoszul}, or more directly
from~\cite[Theorem~6.10(a)]{Pkoszul}.
\end{proof}

 Let $\C$ be an $\R$\+free CDG\+coalgebra; suppose the reduction
$\C/\m\C$ is a coaugmented CDG\+coalgebra over~$k$ with
the coaugmentation $\bar w\:k\rarrow\C/\m\C$.
 Let $w\:\R\rarrow\C$ be a homogeneous section of the counit map
lifting the coaugmentation~$\bar w$.
 Consider the corresponding wcDG\+algebra $\A=\Cb_w(\C)$ over~$\R$,
and let $\tau_{\C,w}\:\C\rarrow\A$ be the natural twisting cochain.

\begin{cor}  \label{conilpotent-cobar}
 All the assertions of Theorem~\textup{\ref{acyclic-cochain-duality}}
hold for the twisting cochain $\tau=\tau_{\C,w}$, provided that
the CDG\+coalgebra\/ $\C/\m\C$ is conilpotent.
\end{cor}

\begin{proof}
 The twisting cochain $\bar\tau_{\C,w}=\tau_{\C,w}/\m\tau_{\C,w}$
is acyclic by the definition (see also~\cite[Theorem~6.4]{Pkoszul}).
\end{proof}

\subsection{Nonconilpotent duality}  \label{nonconilpotent-sect}
 Let $\B$ be an $\R$\+free CDG\+algebra such that the graded
$k$\+algebra $\B/\m\B$ has finite left homological dimension.
 Let $\C$ be an $\R$\+free CDG\+coalgebra, $\tau\:\C\rarrow\B$ be
a twisting cochain, and $\bar\tau\:\C/\m\C\rarrow\B/\m\B$ be
its reduction modulo~$\m$.

\begin{thm}  \label{nonconilpotent-general-duality}
 Assume one (or, equivalently, all) of the functors\/
$\C/\m\C\ot^{\bar\tau}{-}$, \ $\B/\m\B\ot^{\bar\tau}{-}$, \
$\Hom^{\bar\tau}(\C/m\C,{-})$, and/or\/ $\Hom^{\bar\tau}(\B/\m\B,{-})$
induce equivalences between the triangulated categories\/
$\sD^\abs(\B/\m\B\mod)$, \ $\sD^\co(\C/\m\C\comod)$, and/or
$\sD^\ctr(\C/\m\C\contra)$
(see~\textup{\cite[\emph{Theorems}~6.7 \emph{and}~6.8]{Pkoszul}}).
 Then \par
\textup{(a)} the functors\/ $\C\ot^\tau{-}\:H^0(\B\mod\Rfr)
\rarrow H^0(\C\comod\Rfr)$ and\/ $\B\ot^\tau{-}\:H^0(\C\comod\Rfr)
\rarrow H^0(\B\mod\Rfr)$ induce functors
$$
 \sD^\abs(\B\mod\Rfr)\lrarrow\sD^\co(\C\comod\Rfr)
$$
and
$$
 \sD^\co(\C\comod\Rfr)\lrarrow\sD^\abs(\B\mod\Rfr),
$$
which are mutually inverse equivalences of triangulated categories; \par
\textup{(b)} the functors\/ $\Hom^\tau(\C,{-})\:H^0(\B\mod\Rctr)
\rarrow H^0(\C\contra\Rctr)$ and\/ $\Hom^\tau(\B,{-})\:
H^0(\C\contra\Rctr)\rarrow H^0(\B\mod\Rctr)$ induce functors
{\hbadness=2250
$$
 \sD^\ctr(\B\mod\Rctr)\lrarrow\sD^\ctr(\C\contra\Rctr)
$$
and }
$$
 \sD^\ctr(\C\contra\Rctr)\lrarrow\sD^\ctr(\B\mod\Rctr),
$$
which are mutually inverse equivalences of triangulated categories; \par
\textup{(c)} the functors\/ $\C\ocn^\tau{-}\:H^0(\B\mod\Rco)
\rarrow H^0(\C\comod\Rco)$ and\/ $\B\ocn^\tau{-}\:H^0(\C\comod\Rco)
\rarrow H^0(\B\mod\Rco)$ induce functors
$$
 \sD^\co(\B\mod\Rco)\lrarrow\sD^\co(\C\comod\Rco)
$$
and
$$
 \sD^\co(\C\comod\Rco)\lrarrow\sD^\co(\B\mod\Rco),
$$
which are mutually inverse equivalences of triangulated categories; \par
\textup{(d)} the functors\/ $\Ctrhom^\tau(\C,{-})\:
H^0(\B\mod\Rcof)\rarrow H^0(\C\contra\Rcof)$ and\/
$\Ctrhom^\tau(\B,{-})\:H^0(\C\contra\Rcof)\rarrow H^0(\B\mod\Rcof)$
induce functors
$$
 \sD^\abs(\B\mod\Rcof)\lrarrow\sD^\ctr(\C\contra\Rcof)
$$
and
$$
 \sD^\ctr(\C\contra\Rcof)\lrarrow\sD^\abs(\B\mod\Rcof),
$$
which are mutually inverse equivalences of triangulated categories; \par
\textup{(e)} the above equivalences of triangulated categories
form a commutative diagram with the equivalences of categories\/
$\Phi_\R=\Psi_\R^{-1}$ of Sections~\textup{\ref{r-cofree-absolute}}
and~\textup{\ref{r-cofree-co-derived}}, $\boL\Phi_\R=\boR\Psi_\R^{-1}$
of Corollary~\textup{\ref{non-adj-fin-dim-reduct-r-co-contra}},
$\boL\Phi_\C=\boR\Psi_\C^{-1}$ of
Corollaries~\textup{\ref{r-free-derived-co-contra}}
and~\textup{\ref{r-cofree-derived-co-contra}}, $\boL\Phi_{\R,\C}=
\boR\Psi_{\R,\C}^{-1}$ of
Corollary~\textup{\ref{non-adj-derived-co-contra}},
the equivalences of triangulated categories from
Corollary~\textup{\ref{non-adj-fin-dim-reduct-co-abs}}, and
the equivalences of contra/coderived categories from
Theorem~\textup{\ref{non-adj-co-derived-comp}}.
 In other words, the following square diagrams of triangulated
equivalences are commutative:
$$
\begin{diagram}
\node{\llap{$\C\ot^\tau{-}$}\:\sD^\abs(\B\mod\Rfr)}
\arrow{e,=} \arrow{s,lr,=}{\Phi_\R}{\boR\Psi_\R}
\node{\sD^\co(\C\comod\Rfr)\,\,\:\!\rlap{$\B\ot^\tau{-}$}}
\arrow{s,lr,=}{\Phi_\R}{\boR\Psi_\R} \\
\node{\llap{$\C\ocn^\tau{-}$}\:\sD^\co(\B\mod\Rco)}
\arrow{e,=}
\node{\sD^\co(\C\comod\Rco)\,\,\:\!\rlap{$\B\ocn^\tau{-}$}}
\end{diagram}
$$
$$
\begin{diagram}
\node{\llap{$\Hom^\tau(\C,{-})$}\:\sD^\ctr(\B\mod\Rctr)}
\arrow{e,=} \arrow{s,lr,=}{\boL\Phi_\R}{\Psi_\R}
\node{\sD^\co(\C\contra\Rctr)\,\,\:\!\rlap{$\Hom^\tau(\B,{-})$}}
\arrow{s,lr,=}{\boL\Phi_\R}{\Psi_\R} \\
\node{\llap{$\Ctrhom^\tau(\C,{-})$}\:\sD^\abs(\B\mod\Rcof)}
\arrow{e,=}
\node{\sD^\co(\C\contra\Rcof)\,\,\:\!\rlap{$\Ctrhom^\tau(\B,{-})$}}
\end{diagram}
$$
$$
\begin{diagram}
\node{\llap{$\C\ot^\tau{-}$}\:\sD^\abs(\B\mod\Rfr)}
\arrow{e,=} \arrow{s,=}
\node{\sD^\co(\C\comod\Rfr)\,\,\:\!\rlap{$\B\ot^\tau{-}$}}
\arrow{s,lr,=}{\boR\Psi_\C}{\boL\Phi_\C} \\
\node{\llap{$\Hom^\tau(\C,{-})$}\:\sD^\ctr(\B\mod\Rctr)}
\arrow{e,=} 
\node{\sD^\co(\C\contra\Rctr)\,\,\:\!\rlap{$\Hom^\tau(\B,{-})$}}
\end{diagram}
$$
$$
\begin{diagram}
\node{\llap{$\C\ocn^\tau{-}$}\:\sD^\co(\B\mod\Rco)}
\arrow{e,=} \arrow{s,=}
\node{\sD^\co(\C\comod\Rco)\,\,\:\!\rlap{$\B\ocn^\tau{-}$}}
\arrow{s,lr,=}{\boR\Psi_\C}{\boL\Phi_\C} \\
\node{\llap{$\Ctrhom^\tau(\C,{-})$}\:\sD^\abs(\B\mod\Rcof)}
\arrow{e,=}
\node{\sD^\co(\C\contra\Rcof)\,\,\:\!\rlap{$\Ctrhom^\tau(\B,{-})$}}
\end{diagram}
$$
$$
\begin{diagram}
\node{\llap{$\C\ocn^\tau{-}$}\:\sD^\co(\B\mod\Rco)}
\arrow{e,=} \arrow{s,lr,=}{\boR\Psi_\R}{\boL\Phi_\R}
\node{\sD^\co(\C\comod\Rco)\,\,\:\!\rlap{$\B\ocn^\tau{-}$}}
\arrow{s,lr,=}{\boR\Psi_{\R,\C}}{\boL\Phi_{\R,\C}} \\
\node{\llap{$\Hom^\tau(\C,{-})$}\:\sD^\ctr(\B\mod\Rctr)}
\arrow{e,=} 
\node{\sD^\co(\C\contra\Rctr)\,\,\:\!\rlap{$\Hom^\tau(\B,{-})$}}
\end{diagram}
$$
\end{thm}

\begin{proof}
 The assertions about the existence of induced functors in parts~(a\+d)
do not depend on the assumption on~$\bar\tau$; the assertions about
them being equivalences of categories do.
 Both functors in part~(a) take coacyclic objects to contractible ones,
since they transform infinite direct sums into infinite direct sums
and short exact sequences into short exact sequences whose underlying
sequences of graded objects are split exact.
 Both functors in part~(c) take coacyclic objects to coacyclic objects,
for the similar reasons.
 Analogously, both functors in part~(d) take contraacyclic objects to
contractible ones, and the functors in part~(b) preserve
contraacyclicity.
 It follows that the induced adjoint functors exist in all cases.

 To prove that the adjunction morphisms in~(a) are equivalences, one
reduces them modulo~$\m$ and uses the assumption of Theorem together
with Corollaries~\ref{r-free-abs-acycl-reduction}
and~\ref{r-free-co-acycl-reduction}.
 To demonstrate the same assertion in the case~(b), it suffices to
restrict both functors to $\R$\+free CDG\+modules and
CDG\+contramodules (see Theorems~\ref{co-derived-mod}
and~\ref{non-adj-co-derived-comp}) and apply the same argument as
in part~(a).

 Parts~(c) and~(d) are similar; and the proof of part~(e) is identical
to that of Theorem~\ref{acyclic-cochain-duality}(e).
\end{proof}

 Let $\C$ be an $\R$\+free CDG\+coalgebra and $w\:\R\rarrow\C$ be
a homogeneous section of the counit map.
 Consider the corresponding $\R$\+free CDG\+algebra $\B=\Cb_w(\C)$,
and let $\tau_{\C,w}\:\C\rarrow\B$ be the natural twisting cochain.

\begin{cor}  \label{nonconilpotent-cobar}
 All the assertions of
Theorem~\textup{\ref{nonconilpotent-general-duality}} are true
for the twisting cochain $\tau=\tau_{\C,w}$.
\end{cor}

\begin{proof}
 The assumption of Theorem~\ref{nonconilpotent-general-duality}
holds in this case by~\cite[Theorem~6.7]{Pkoszul}.
\end{proof}

 Notice that by comparing
Corollaries~\ref{conilpotent-cobar} and~\ref{nonconilpotent-cobar}
one can conclude that $\sD^\si(\A\mod\Rfr)=\sD^\abs(\A\mod\Rfr)$
and $\sD^\si(\A\mod\Rcof)=\sD^\abs(\A\mod\Rcof)$ when $\C/\m\C$ is
conilpotent, $w\:\R\rarrow\C$ reduces to the coaugmentation, and
$\A=\Cb_w(\C)$.
 In fact, this is a particular case of Theorem~\ref{r-free-cofibrant}.

\subsection{Transformation of functors under Koszul duality}
\label{transformation-koszul}
 Let $\A$ be a wcDG\+alge\-bra over $\R$ and $\C$ be an $\R$\+free
CDG\+coalgebra with a conilpotent reduction $\C/\m\C$.
 Let $\tau\:\C\rarrow\A$ be a twisting cochain such that
the twisting cochain $\bar\tau=\tau/\m\tau$ is acyclic.
 Notice that by the right version of
Theorem~\ref{acyclic-cochain-duality} the functors
$\M\mpsto\M\ot^\tau\C$ and $\N\mpsto\N\ot^\tau\A$ induce
an equivalence of triangulated categories
$\sD^\si(\modrRfr\A)\simeq\sD^\co(\comodrRfr\C)$, while
the functors $\cM\mpsto\cM\ocn^\tau\C$ and $\cN\mpsto\cN\ocn^\tau\A$
induce an equivalence of triangulated categories
$\sD^\si(\modrRco\A)\simeq\sD^\co(\comodrRco\C)$.

 The left derived functor
$$
 \Tor^\A\:\sD^\si(\modrRco\A)\times\sD^\si(\A\mod\Rctr)
 \lrarrow H^0(\R\comod^\cofr)
$$
is obtained by switching the left and right sides
in the construction of the derived functor $\Tor^\A$ from
Section~\ref{wcdg-semiderived}.

 The following theorem shows how Koszul duality transforms
the functor $\Tor^\A$ into the functor $\Ctrtor^\C$.

\begin{thm}  \label{ctrtor-and-tor}
\textup{(a)} The equivalences of categories
$$
\sD^\si(\modrRfr\A)\simeq\sD^\co(\comodrRfr\C)
$$
and
$$
 \sD^\si(\A\mod\Rfr)\simeq\sD^\ctr(\C\contra\Rfr)
$$
from Theorem~\textup{\ref{acyclic-cochain-duality}}
transform the functor
$$
 \Tor^\A\:\sD^\si(\modrRfr\A)\times\sD^\si(\A\mod\Rfr)
 \lrarrow H^0(\R\contra^\free)
$$
from Section~\textup{\ref{r-free-semi}}
into the functor
$$
 \Ctrtor^\C\:\sD^\co(\comodrRfr\C)\times\sD^\ctr(\C\contra\Rfr)
 \lrarrow H^0(\R\contra^\free)
$$
from Section~\textup{\ref{r-free-co-derived}}. \par
\textup{(b)} The equivalences of categories
$$
 \sD^\si(\modrRfr\A)\simeq\sD^\co(\comodrRfr\C)
$$
and
$$
 \sD^\si(\A\mod\Rcof)\simeq\sD^\ctr(\C\contra\Rcof)
$$
transform the functor
$$
 \Tor^\A\:\sD^\si(\modrRfr\A)\times\sD^\si(\A\mod\Rcof)\lrarrow
 H^0(\R\comod^\cofr)
$$
from Section~\textup{\ref{r-cofree-semi}} into the functor
$$
 \Ctrtor^\C\:\sD^\co(\comodrRfr\C)\times\sD^\ctr(\C\contra\Rcof)
 \lrarrow H^0(\R\comod^\cofr)
$$
from Section~\textup{\ref{r-cofree-co-derived}}. \par
\textup{(c)} The equivalences of categories
$$
 \sD^\si(\modrRco\A)\simeq\sD^\co(\comodrRco\C)
$$
and
$$
 \sD^\si(\A\mod\Rctr)\simeq\sD^\ctr(\C\contra\Rctr)
$$
transform the functor
$$
 \Tor^\A\:\sD^\si(\modrRco\A)\times\sD^\si(\A\mod\Rctr)\lrarrow
 H^0(\R\comod^\cofr)
$$
into the functor
$$
 \Ctrtor^\C\:\sD^\co(\comodrRco\C)\times\sD^\ctr(\C\contra\Rctr)\lrarrow
 H^0(\R\comod^\cofr)
$$
from Section~\textup{\ref{non-adj-co-derived}}.
\end{thm}

 The following theorem shows how Koszul duality transforms
the functor $\Tor^\A$ into the functor $\Cotor^\C$.

\begin{thm}  \label{cotor-and-tor}
\textup{(a)} The equivalences of categories
$$
 \sD^\si(\modrRfr\A)\simeq\sD^\co(\comodrRfr\C)
$$
and
$$
 \sD^\si(\A\mod\Rfr)\simeq \sD^\co(\C\comod\Rfr)
$$
from Theorem~\textup{\ref{acyclic-cochain-duality}}
transform the functor
$$
 \Tor^\A\:\sD^\si(\modrRfr\A)\times\sD^\si(\A\mod\Rfr)\lrarrow
 H^0(\R\contra^\free)
$$
from Section~\textup{\ref{r-free-semi}} into the functor
$$
 \Cotor^\C\:\sD^\co(\comodrRfr\C)\times\sD^\co(\C\comod\Rfr)\lrarrow
 H^0(\R\contra^\free)
$$
from Section~\textup{\ref{r-free-co-derived}}. \par
\textup{(b)} The equivalences of categories
$$
 \sD^\si(\modrRfr\A)\simeq\sD^\co(\comodrRfr\C)
$$
and
$$
 \sD^\si(\A\mod\Rcof)\simeq\sD^\co(\C\comod\Rcof)
$$
transform the functor
$$
 \Tor^\A\:\sD^\si(\modrRfr\A)\times\sD^\si(\A\mod\Rcof)\lrarrow
 H^0(\R\comod^\cofr)
$$
from Section~\textup{\ref{r-cofree-semi}} into the functor
$$
 \Cotor^\C\:\sD^\co(\comodrRfr\C)\times\sD^\co(\C\comod\Rcof)\lrarrow
 H^0(\R\comod^\cofr)
$$
from Section~\textup{\ref{r-cofree-co-derived}}. \par
\textup{(c)} The equivalences of categories
$$
 \sD^\si(\modrRco\A)\simeq\sD^\co(\comodrRco\C)
$$
and
$$
 \sD^\si(\A\mod\Rco)\simeq\sD^\co(\C\comod\Rco)
$$
transform the functor
$$
 \Tor^\A\:\sD^\si(\modrRco\A)\times\sD^\si(\A\mod\Rco)\lrarrow
 H^0(\R\comod^\cofr)
$$
from Section~\textup{\ref{wcdg-semiderived}} into the functor
$$
 \Cotor^\C\:\sD^\co(\comodrRco\C)\times\sD^\co(\C\comod\Rco)\lrarrow
 H^0(\R\comod^\cofr)
$$
from Section~\textup{\ref{non-adj-co-derived}}.
\end{thm}

 The following theorem shows how Koszul duality transforms
the functor $\Ext_\A$ into the functor $\Ext^\C$.

\begin{thm}  \label{contramodule-ext-and-ext}
\textup{(a)} The equivalence of categories
$$
 \sD^\si(\A\mod\Rctr)\simeq\sD^\ctr(\C\contra\Rctr)
$$
from Theorem~\textup{\ref{acyclic-cochain-duality}}
transforms the functor
$$
 \Ext_\A\:\sD^\si(\A\mod\Rctr)^\sop\times\sD^\si(\A\mod\Rctr)\lrarrow
 H^0(\R\contra^\free)
$$
from Section~\textup{\ref{wcdg-semiderived}} into the functor
$$
 \Ext^\C\:\sD^\ctr(\C\contra\Rctr)^\sop\times\sD^\ctr(\C\contra\Rctr)
 \lrarrow H^0(\R\contra^\free)
$$
from Section~\textup{\ref{non-adj-co-derived}}. \par
\textup{(b)} The equivalence of categories
$$
 \sD^\si(\A\mod\Rcof)\simeq\sD^\ctr(\C\contra\Rcof)
$$
transforms the functor
$$
 \Ext_\A\:\sD^\si(\A\mod\Rcof)^\sop\times\sD^\si(\A\mod\Rcof)\lrarrow
 H^0(\R\contra^\free)
$$
from Section~\textup{\ref{r-cofree-semi}} into the functor
$$
 \Ext^\C\:\sD^\ctr(\C\contra\Rcof)^\sop\times\sD^\ctr(\C\contra\Rcof)
 \lrarrow H^0(\R\contra^\free)
$$
from Section~\textup{\ref{r-cofree-co-derived}}. \par
\textup{(c)} The equivalences of categories
$$
 \sD^\si(\A\mod\Rfr)\simeq\sD^\ctr(\C\contra\Rfr)
$$
and
$$
 \sD^\si(\A\mod\Rcof)\simeq\sD^\ctr(\C\contra\Rcof)
$$
transform the functor
$$
 \Ext_\A\:\sD^\si(\A\mod\Rfr)^\sop\times\sD^\si(\A\mod\Rcof)
 \lrarrow H^0(\R\comod^\cofr)
$$
from Section~\textup{\ref{r-cofree-semi}} into the functor
$$
 \Ext^\C\:\sD^\ctr(\C\contra\Rfr)^\sop\times\sD^\ctr(\C\contra\Rcof)
 \lrarrow H^0(\R\comod^\cofr)
$$
from Section~\textup{\ref{r-cofree-co-derived}}.
\end{thm}

 The following theorem shows how Koszul duality transforms
the functor $\Ext_\A$ into the functor $\Ext_\C$.

\begin{thm}  \label{comodule-ext-and-ext}
\textup{(a)} The equivalence of categories
$$
 \sD^\co(\C\comod\Rfr)\simeq\sD^\si(\A\mod\Rfr)
$$
from Theorem~\textup{\ref{acyclic-cochain-duality}}
transforms the functor
$$
 \Ext_\A\:\sD^\si(\A\mod\Rfr)^\sop\times\sD^\si(\A\mod\Rfr)\lrarrow
 H^0(\R\contra^\free)
$$
from Section~\textup{\ref{r-free-semi}} into the functor
$$
 \Ext_\C\:\sD^\co(\C\comod\Rfr)^\sop\times\sD^\co(\C\comod\Rfr)\lrarrow
 H^0(\R\contra^\free)
$$
from Section~\textup{\ref{r-free-co-derived}}. \par
\textup{(b)} The equivalence of categories
$$
 \sD^\si(\A\mod\Rco)\simeq\sD^\co(\C\comod\Rco)
$$
transforms the functor
$$
 \Ext_\A\:\sD^\si(\A\mod\Rco)^\sop\times\sD^\si(\A\mod\Rco)\lrarrow
 H^0(\R\contra^\free)
$$
from Section~\textup{\ref{wcdg-semiderived}} into the functor
$$
 \Ext_\C\:\sD^\co(\C\comod\Rco)^\sop\times\sD^\co(\C\comod\Rco)\lrarrow
 H^0(\R\contra^\free)
$$
from Section~\textup{\ref{non-adj-co-derived}}. \par
\textup{(c)} The equivalences of categories
$$
 \sD^\si(\A\mod\Rfr)\simeq\sD^\co(\C\comod\Rfr)
$$
and
$$
 \sD^\si(\A\mod\Rcof)\simeq\sD^\co(\C\comod\Rcof)
$$
transform the functor
$$
 \Ext_\A\:\sD^\si(\A\mod\Rfr)^\sop\times\sD^\si(\A\mod\Rcof)\lrarrow
 H^0(\R\comod^\cofr)
$$
from Section~\textup{\ref{r-cofree-semi}} into the functor
$$
 \Ext_\C\:\sD^\co(\C\comod\Rfr)^\sop\times\sD^\co(\C\comod\Rcof)\lrarrow
 H^0(\R\comod^\cofr)
$$
from Section~\textup{\ref{r-cofree-co-derived}}.
\end{thm}

 The following theorem shows how Koszul duality transforms
the functor $\Ext_\A$ into the functor $\Coext_\C$.

\begin{thm}  \label{coext-and-ext}
\textup{(a)} The equivalences of categories
$$
 \sD^\si(\A\mod\Rfr)\simeq\sD^\co(\C\comod\Rfr)
$$
and
$$
 \sD^\si(\A\mod\Rfr)\simeq\sD^\ctr(\C\contra\Rfr)
$$
from Theorem~\textup{\ref{acyclic-cochain-duality}}
transform the functor
$$
 \Ext_\A\:\sD^\si(\A\mod\Rfr)^\sop\times\sD^\si(\A\mod\Rfr)\lrarrow
 H^0(\R\contra^\free)
$$
from Section~\textup{\ref{r-free-semi}} into the functor
$$
 \Coext_\C\:\sD^\co(\C\comod\Rfr)^\sop\times\sD^\ctr(\C\contra\Rfr)
 \lrarrow H^0(\R\contra^\free)
$$
from Section~\textup{\ref{r-free-co-derived}}. \par
\textup{(b)} The equivalences of categories
$$
 \sD^\si(\A\mod\Rcof)\simeq\sD^\co(\C\comod\Rcof)
$$
and
$$
 \sD^\si(\A\mod\Rcof)\simeq\sD^\ctr(\C\contra\Rcof)
$$
transform the functor
$$
 \Ext_\A\:\sD^\si(\A\mod\Rcof)^\sop\times\sD^\si(\A\mod\Rcof)\lrarrow
 H^0(\R\contra^\free)
$$
from Section~\textup{\ref{r-cofree-semi}} into the functor
$$
 \Coext_\C\:\sD^\co(\C\comod\Rcof)^\sop\times\sD^\ctr(\C\contra\Rcof)
 \lrarrow H^0(\R\contra^\free)
$$
from Section~\textup{\ref{r-cofree-co-derived}}. \par
\textup{(c)} The equivalences of categories
$$
 \sD^\si(\A\mod\Rfr)\simeq\sD^\co(\C\comod\Rfr)
$$
and
$$
 \sD^\si(\A\mod\Rcof)\simeq\sD^\ctr(\C\contra\Rcof)
$$
transform the functor
$$
 \Ext_\A\:\sD^\si(\A\mod\Rfr)^\sop\times\sD^\si(\A\mod\Rcof)\lrarrow
 H^0(\R\comod^\cofr)
$$
from Section~\textup{\ref{r-cofree-semi}} into the functor
$$
 \Coext_\C\:\sD^\co(\C\comod\Rfr)^\sop\times\sD^\ctr(\C\contra\Rcof)
 \lrarrow H^0(\R\comod^\cofr)
$$
from Section~\textup{\ref{r-cofree-co-derived}}. \par
\textup{(d)} The equivalences of categories
$$
 \sD^\si(\A\mod\Rco)\simeq\sD^\co(\C\comod\Rco)
$$
and
$$
 \sD^\si(\A\mod\Rctr)\simeq\sD^\ctr(\C\contra\Rctr)
$$
transform the functor
$$
 \Ext_\A\:\sD^\si(\A\mod\Rco)^\sop\times\sD^\si(\A\mod\Rctr)\lrarrow
 H^0(\R\contra^\free)
$$
from Section~\textup{\ref{wcdg-semiderived}} into the functor
$$
 \Coext_\C\:\sD^\co(\C\comod\Rco)^\sop\times\sD^\ctr(\C\contra\Rctr)
 \lrarrow H^0(\R\contra^\free)
$$
from Section~\textup{\ref{non-adj-co-derived}}.
\end{thm}

\begin{proof}[Proof of Theorems~\textup{\ref{ctrtor-and-tor}}\+-%
\textup{\ref{coext-and-ext}}]
 The proofs of all the sixteen assertions are very similar to
each other, so will only discuss two randomly chosen ones, namely,
Theorem~\ref{cotor-and-tor}(c) and
Theorem~\ref{contramodule-ext-and-ext}(a).

 Let $\cM$ be an $\R$\+comodule left wcDG\+module over $\A$ and
$\cN$ be an $\R$\+comodule right CDG\+comodule over~$\C$.
 Then there is a natural isomorphism of complexes of $\R$\+comodules
$\cN\oc_\C(\C\ocn^\tau\cM)\simeq(\cN\ocn^\tau\A)\ot_\A\cM$.
 Indeed, both complexes are isomorphic to the complex
$\cN\oc_\R\cM$ with the differential induced by the differentials
on $\cN$ and $\cM$ being twisted using the twisting cochain~$\tau$.
 Notice that $\C\ocn^\tau\cM\in\C\comod\Rcof_\inj$ whenever
$\cM\in\A\mod\Rcof$ and $\cN\ocn^\tau\A\in H^0(\modrRcofproj\A)_\proj$
whenever $\cN\in\comodrRcof\C$.
 The latter assertion follows from the adjunction of the functors
${-}\ocn^\tau\A$ and ${-}\ocn^\tau\C$ together with the fact that
the latter functor takes semiacyclic $\R$\+cofree wcDG\+modules
to contractible $\R$\+cofree CDG\+comodules (see the proof of
Theorem~\ref{acyclic-cochain-duality}(a)) and the semiorthogonal
decomposition of Theorem~\ref{r-cofree-semi-resolutions}(a).

 Let $\P$ be an $\R$\+contramodule left wcDG\+module over $\A$ and
$\Q$ be an $\R$\+contramod\-ule left CDG\+contramodule over~$\C$.
 Then there is a natural isomorphism of complexes of
$\R$\+contramodules $\Hom^\C(\Hom^\tau(\C,\P),\Q)\simeq
\Hom_\A(\P,\Hom^\tau(\A,\Q))$.
 Indeed, both complexes are isomorphic to the complex
$\Hom^\R(\P,\Q)$ with the differential induced by the differentials
on $\P$ and $\Q$ being twisted by~$\tau$.
 Notice that $\Hom^\tau(\C,\P)\in\C\contra\Rfr_\proj$ whenever
$\P\in\A\mod\Rfr$ and $\Hom^\tau(\A,\Q)\in H^0(\A\mod\Rfr_\inj)_\inj$
whenever $\Q\in\C\contra\Rfr$.
 The proof of the latter assertion is similar to the above.
 (See~\cite[proof of Theorem~6.9.1]{Pkoszul} for further details.)
\end{proof}

 Let $\A$ and $\B$ be wcDG\+algebras over $\R$, let $\C$ and $\D$ be
$\R$\+free CDG\+coalgebras with conilpotent reductions $\C/\m\C$ and
$\D/\m\D$, and let $\tau\:\C\rarrow\A$ and $\sigma\:\D\rarrow\B$ be
twisting cochains with acyclic reductions $\tau/\m\tau$ and
$\sigma/\m\sigma$.
 Let $f=(f,\xi)\:\A\rarrow\B$ be a morphism of wcDG\+algebras and
$g=(g,\eta)\:\C\rarrow\D$ be a morphism of CDG\+coalgebras.
 Assume that $(f,\xi)$ and $(g,\eta)$ make a commutative diagram
with $\tau$ and~$\sigma$, i.~e., $\sigma\circ(g,\eta) =
(f,\xi)\circ\tau$, or explicitly, $\sigma\circ g - f\circ\tau =
e_\B\circ\eta + \xi\circ \varepsilon_\C$.

\begin{prop} \label{koszul-restriction-extension}
 \textup{(a)} The equivalences of triangulated categories
$$
\sD^\si(\A\mod\Rfr)\simeq\sD^\co(\C\comod\Rfr)
$$
and
$$
 \sD^\si(\B\mod\Rfr)\simeq\sD^\co(\D\comod\Rfr)
$$
from Theorem~\textup{\ref{acyclic-cochain-duality}}
transform the functor
$$
 \boI R_f\:\sD^\si(\B\mod\Rfr)\lrarrow\sD^\si(\A\mod\Rfr)
$$
from Section~\textup{\ref{r-free-semi}} into the functor
$$
 \boR E_g\:\sD^\co(\D\comod\Rfr)\lrarrow\sD^\co(\C\comod\Rfr)
$$
from Section~\textup{\ref{r-free-co-derived}}, and the functor
$$
 \boL E_f\:\sD^\si(\A\mod\Rfr)\lrarrow\sD^\si(\B\mod\Rfr)
$$
into the functor
$$
 \boI R_g\:\sD^\co(\C\comod\Rfr)\rarrow\sD^\co(\D\comod\Rfr).
$$
 In other words, the following two square diagrams of categories,
functors, and equivalences are commutative:
$$
\begin{diagram}
 \node{\llap{$\boL E_f$}\:\sD^\si(\A\mod\Rfr)}
 \arrow{e}\arrow{s,=}
 \node{\sD^\si(\B\mod\Rfr)\,\,\:\!\rlap{$\boI R_f$}}
 \arrow{w}\arrow{s,=} \\
 \node{\llap{$\boI R_g$}\:\sD^\co(\C\comod\Rfr)} \arrow{e}
 \node{\sD^\co(\D\comod\Rfr)\,\,\:\!\rlap{$\boR E_g$}} \arrow{w}
\end{diagram}
$$ \par
 \textup{(b)} The equivalences of triangulated categories
$$
 \sD^\si(\A\mod\Rctr)\simeq\sD^\ctr(\C\contra\Rctr)
$$
and
$$
 \sD^\si(\B\mod\Rctr)\simeq\sD^\ctr(\D\contra\Rctr)
$$
transform the functor
$$
 \boI R_f\:\sD^\si(\B\mod\Rctr)\lrarrow\sD^\si(\A\mod\Rctr)
$$
from Section~\textup{\ref{wcdg-semiderived}} into the functor
$$
 \boL E^g\:\sD^\ctr(\D\contra\Rctr)\lrarrow\sD^\ctr(\C\contra\Rctr)
$$
from Section~\textup{\ref{non-adj-co-derived}}, and the functor
$$
 \boR E^f\:\sD^\si(\A\mod\Rctr)\lrarrow\sD^\si(\B\mod\Rctr)
$$
into the functor
$$
 \boI R^g\:\sD^\ctr(\C\contra\Rctr)\rarrow\sD^\ctr(\D\contra\Rctr).
$$
 In other words, the following two square diagrams of categories,
functors, and equivalences are commutative:
$$
\begin{diagram}
 \node{\llap{$\boR E^f$}\:\sD^\si(\A\mod\Rctr)}
 \arrow{e}\arrow{s,=}
 \node{\sD^\si(\B\mod\Rctr)\,\,\:\!\rlap{$\boI R_f$}}
 \arrow{w}\arrow{s,=} \\
 \node{\llap{$\boI R^g$}\:\sD^\ctr(\C\contra\Rctr)} \arrow{e}
 \node{\sD^\ctr(\D\contra\Rctr)\,\,\:\!\rlap{$\boL E^g$}} \arrow{w}
\end{diagram}
$$ \par
 \textup{(c)} The equivalences of triangulated categories
$$
 \sD^\si(\A\mod\Rco)\simeq\sD^\co(\C\comod\Rco)
$$
and
$$
 \sD^\si(\B\mod\Rco)\simeq\sD^\co(\D\comod\Rco)
$$
transform the functor
$$
 \boI R_f\:\sD^\si(\B\mod\Rco)\lrarrow\sD^\si(\A\mod\Rco)
$$
from Section~\textup{\ref{wcdg-semiderived}} into the functor
$$
 \boR E_g\:\sD^\co(\D\comod\Rco)\lrarrow\sD^\co(\C\comod\Rco),
$$
from Section~\textup{\ref{non-adj-co-derived}}, and the functor
$$
 \boL E_f\:\sD^\si(\A\mod\Rco)\lrarrow\sD^\si(\B\mod\Rco)
$$
into the functor
$$
 \boI R_g\:\sD^\co(\C\comod\Rco)\lrarrow\sD^\co(\D\comod\Rco).
$$ \par
 In other words, the following two square diagrams of categories,
functors, and equivalences are commutative:
$$
\begin{diagram}
 \node{\llap{$\boL E_f$}\:\sD^\si(\A\mod\Rco)}
 \arrow{e}\arrow{s,=}
 \node{\sD^\si(\B\mod\Rco)\,\,\:\!\rlap{$\boI R_f$}}
 \arrow{w}\arrow{s,=} \\
 \node{\llap{$\boI R_g$}\:\sD^\co(\C\comod\Rco)} \arrow{e}
 \node{\sD^\co(\D\comod\Rco)\,\,\:\!\rlap{$\boR E_g$}} \arrow{w}
\end{diagram}
$$ \par
 \textup{(d)} The equivalences of triangulated categories
$$
 \sD^\si(\A\mod\Rcof)\simeq\sD^\ctr(\C\contra\Rcof)
$$
and
$$
 \sD^\si(\B\mod\Rcof)\simeq\sD^\ctr(\D\contra\Rcof)
$$
transform the functor
$$
 \boI R_f\:\sD^\si(\B\mod\Rcof)\lrarrow\sD^\si(\A\mod\Rcof)
$$
from Section~\textup{\ref{r-cofree-semi}} into the functor
$$
 \boL E^g\:\sD^\ctr(\D\contra\Rcof)\lrarrow\sD^\ctr(\C\contra\Rcof),
$$
from Section~\textup{\ref{r-cofree-co-derived}} and the functor
$$
 \boR E^f\:\sD^\si(\A\mod\Rcof)\lrarrow\sD^\si(\B\mod\Rcof)
$$
into the functor
$$
 \boI R^g\:\sD^\ctr(\C\contra\Rcof)\lrarrow\sD^\ctr(\D\contra\Rcof).
$$
 In other words, the following two square diagrams of categories,
functors, and equivalences are commutative:
$$
\begin{diagram}
 \node{\llap{$\boR E^f$}\:\sD^\si(\A\mod\Rcof)}
 \arrow{e}\arrow{s,=}
 \node{\sD^\si(\B\mod\Rcof)\,\,\:\!\rlap{$\boI R_f$}}
 \arrow{w}\arrow{s,=} \\
 \node{\llap{$\boI R^g$}\:\sD^\ctr(\C\contra\Rcof)} \arrow{e}
 \node{\sD^\ctr(\D\contra\Rcof)\,\,\:\!\rlap{$\boL E^g$}} \arrow{w}
\end{diagram}
$$
\end{prop}

\begin{proof}
 See~\cite[proof of Proposition~6.9]{Pkoszul}.
\end{proof}

\begin{rem}
 The analogues of the results of this section also hold for
the nonconilpotent Koszul duality theory of
Section~\ref{nonconilpotent-sect}.
 In order to formulate them, though, one needs to define
the $\Tor$, $\Ext$, and derived (co)extension-of-scalars functors
acting in the co/contra/absolute derived categories of CDG\+modules.
 This can be done, at least, for a CDG\+algebra $\B$ such that
the categories of $\R$\+(co)free graded left/right $\B$\+modules
have finite homological dimension, using the results of
Theorem~\ref{r-free-absolute-derived} and
Corollary~\ref{non-adj-fin-dim-reduct-co-abs}.
 (Cf.\ \cite[Section~3.12 and Theorem~6.9.2]{Pkoszul}.)
\end{rem}

\subsection{Examples}
 Here we use the results of Sections~\ref{conilpotent-sect}
and~\ref{transformation-koszul} in order to compute the remaining
two of the three examples mentioned in the Introduction
(for the first one, see Example~\ref{kln-counterex}).
 We will only discuss the $\R$\+contramodules of morphisms
and the complexes $\Ext_\A$ of free $\R$\+contramodules in
the exotic derived categories of wcDG\+modules over certain
wcDG\+algebras~$\A$.

 According to the results of Section~\ref{wcainfty-semiderived} below,
the semiderived categories of strictly unital wc~\Ainfty\+modules
over the wcDG\+algebra $\A$ viewed as a strictly unital
wc~\Ainfty\+algebra are equivalent to the semiderived categories
of wcDG\+modules, and the complexes of free $\R$\+contramodules
$\Ext$ in them are isomorphic in the homotopy category.
 Moreover, in both examples the wcDG\+algebra $\A$ is actually
augmented, so one can interpret the same $\R$\+contramodules
of morphisms and complexes of free $\R$\+contramodules $\Ext_\A$
as being related to the semiderived categories of nonunital
wc~\Ainfty\+modules over the augmentation ideal of $\A$ viewed
as a nonunital wc~\Ainfty\+algebra (see
Sections~\ref{nonunital-wcainfty}\+-\ref{strictly-unital-wcainfty}).

\begin{ex} \label{kln-counterex2}
 Let $\R$ be a pro-Artinian topological local ring with the maximal
ideal~$\m$, and let $\epsilon\in\m$ be an element.
 Consider the $\R$\+free graded algebra $\A=\R[x]=
\bigoplus_{n=0}^\infty\R x^n$, where the direct sum is taken in
the category of free graded $\R$\+contramodules and $\deg x=2$.
 Define the wcDG\+algebra structure on $\A$ with $d=0$ and
$h=\epsilon x$ (cf.\ Example~\ref{kln-counterex}).

 Let $\C$ be the $\R$\+free DG\+coalgebra such that the $\R$\+free
DG\+algebra $\C^*$ is isomorphic to the graded algebra $\R[y]/(y^2)$
(i.~e., the direct sum of $\R$ and $\R y$) with $\deg y=-1$ and
$d(y)=\epsilon$.
 Then the wcDG\+algebra $\A$ is isomorphic to $\Cb_w(\C)$, where
$w\:\R\rarrow\C$ is defined by the rule that
$w^*\:\C^*\rarrow\R$ takes $y$ to~$0$.

 Let $N$ be any ($\R$\+contramodule or $\R$\+comodule)
DG\+contramodule or DG\+comodule over $\C$; equivalently, it
can be viewed as a DG\+module over~$\C^*$.
 Then the action of the element $y\in\C^*$ provides a contracting
homotopy for the endomorphism of multiplication with~$\epsilon$
on~$N$.
 By Corollary~\ref{conilpotent-cobar}, it follows that
all the $\R$\+modules of morphisms in the $\R$\+linear triangulated
category $\sD^\si(\A\mod\Rctr)\simeq\sD^\si(\A\mod\Rco)$ are
annihilated by the multiplication with~$\epsilon$.

 By Theorem~\ref{r-free-cofibrant} and
Corollaries~\ref{non-adj-fin-dim-reduct-co-abs}\+-%
\ref{non-adj-fin-dim-reduct-r-co-contra}, the same applies to
the $\R$\+linear triangulated categories
$\sD^\abs(\A\mod\Rfr)\simeq\sD^\abs(\A\mod\Rcof)$ and
$\sD^\ctr(\A\mod\Rctr)\simeq\sD^\co(\A\mod\Rco)$.

 Furthermore, by Theorem~\ref{contramodule-ext-and-ext}
or~\ref{comodule-ext-and-ext}, it follows that the endomorphisms
of multiplication with~$\epsilon$ are homotopic to zero on
the complexes of free $\R$\+contramodules $\Ext_\A$ between
$\R$\+comodule or $\R$\+contramodule wcDG\+modules over~$\A$.
\end{ex}

\begin{ex} \label{clifford-ex}
 Let $\R$ be a pro-Artinian topological local ring and $\C$ be
an ungraded $\R$\+free coalgebra considered as an $\R$\+free
CDG\+coalgebra concentrated entirely in grading zero with
the trivial CDG\+coalgebra structure $d=0=h$.
 Assume that the reduced coalgebra $\C/\m\C$ is conilpotent and
pick a section $w\:\R\rarrow\C$ of the counit map such that
$w/\m w=\bar w$ is the coaugmentation of $\C/\m\C$.

 Let $\N$ be an $\R$\+free $\C$\+comodule viewed as a CDG\+comodule
concentrated in degree zero and with zero differential.
 By~\cite[Remark~4.1]{Psemi}, one has 
$\Hom_{\sD^\co(\C\comod\Rfr)}(\N,\N)\allowbreak\simeq\Hom_\C(\N,\N)$.
 The same applies to $\R$\+contramodule $\C$\+contramodules,
$\R$\+comodule $\C$\+comodules, etc.
 In particular, for $\N=\C$ one has $\Ext_\C(\N,\N)\simeq\C^*$
(by the definition and since $\N\in\C\comod\Rfr_\inj$).

 Set $\A=\Cb_w(\C)$ and $\M=\A\ot^{\tau_{\C,w}}\N$.
 By Corollary~\ref{conilpotent-cobar}, it follows that
$\Hom_{\sD^\si(\A\mod\Rfr)}(\M,\M)\simeq\Hom_\C(\N,\N)$.
 In particular, for $\N=\C$ we have $\Ext_\A(\M,\M)\simeq\C^*$
by Theorem~\ref{comodule-ext-and-ext}(a).

 To point out a specific example, let $\C^*$ be the Clifford algebra
with one generator $\R[y]/(y^2-\epsilon)$, where $\epsilon\in
\m\subset\R$.
 Being by the definition a free $\R$\+module with two generators
endowed with an $\R$\+algebra structure, it is clearly an algebra in
the tensor category of (finitely generated) free $\R$\+contramodules
(since the tensor categories of finitely generated free
$\R$\+modules and finitely generated free $\R$\+contramodules
are equivalent).
 One easily recovers from it the dual $\R$\+free coalgebra~$\C$.
 Let $w\:\R\rarrow\C$ be the section of the counit map given by
the rule $w^*(y)=0$.
 
 Then $\A=\Cb_w(\C)$ is the $\R$\+free graded algebra
$\R[x]$ with $\deg x=1$ endowed with the wcDG\+algebra structure
with $d=0$ and $h=\epsilon x^2$.
 Thus there is wcDG\+module structure on the free graded
$\A$\+module $\M$ with two generators in degree~$0$ such that
the $\R$\+module $\Hom_{\sD^\si(\A\mod\Rfr)}(\M,\M[i])$ is
isomorphic to the two-dimensional $\R$\+free algebra $\C^*$ when
$i=0$ and vanishes otherwise.
 More precisely, the complex of free $\R$\+contramodules
$\Ext_\A(\M,\M)$ is homotopy equivalent to~$\C^*$.
\end{ex}

\subsection{Duality between algebras and coalgebras}
 Let $\R\coalg_\cdg^\cnlp$ denote the category of $\R$\+free
CDG\+coalgebras with conilpotent reductions modulo~$\m$.
 The objects of this category are $\R$\+free CDG\+coalgebras $\C$
such that the CDG\+coalgebra $\C/\m\C$ over~$k$ is conilpotent.
 Morphisms $\C\rarrow\D$ in $\R\coalg_\cdg^\cnlp$ are $\R$\+free
CDG\+co\-algebra morphisms $(f,a)\:\C\rarrow\D$ whose reductions
$(f/\m f\;a/\m a)\:\C/\m\C\rarrow\D/\m\D$ are conilpotent
CDG\+coalgebra morphisms; in other words, the composition
$k\rarrow\C/\m\C\rarrow k$ of the coaugmentation and
the reduced change-of-connection maps is requred to be zero
(a condition only nontrivial when $1=0$ in $\Gamma$ and
$2=0$ in~$\R$).
 The category of nonzero wcDG\+algebras over $\R$ will be denoted
by $\R\alg_\wcdg^+$.

 It follows from Lemma~\ref{free-algebra-morphisms} that the functor
$\A\mpsto\Br_v(\A)$ acting from the category $\R\alg_\wcdg^+$
to the category $\R\coalg_\cdg^\cnlp$ is right adjoint to
the functor $\C\mpsto\Cb_w(\C)$ acting in the opposite direction.

 Indeed, let $\A$ be a wcDG\+algebra over $\R$ and $\C$ be
an $\R$\+free CDG\+coalgebra with conilpotent reduction.
 Let $v\:\A\rarrow\R$ be a homogeneous retraction onto the unit map,
and let $w\:\R\rarrow\C$ be a homogeneous section of the counit map
such that $\bar w = w/\m w\:k\rarrow\C/\m\C$ is the coaugmentation.
 Then both wcDG\+algebra morphisms $\Cb_w(\C)\rarrow\A$ and
CDG\+coalgebra morphisms $\C\rarrow\Br_v(\A)$ correspond bijectively
to twisting cochains $\tau\:\C\rarrow\A$ for which the composition
$k\rarrow\C/\m\C\rarrow\A/\m\A$ vanishes.
 The former bijection assigns to a morphism $(f,a)\:\Cb_w(\C)\rarrow\A$
the twisting cochain $(f,a)\circ\tau_{\C,w}$, and the latter one
assigns to a morphism $(g,a)\:\C\rarrow\Br_v(\A)$ the twisting
cochain $\tau_{\A,v}\circ(g,a)$.

 More generally, let $\B$ be an $\R$\+free CDG\+algebra and $\C$
be an $\R$\+free CDG\+coalge\-bra.
 Let $v\:\B\rarrow\R$ be a homogeneous retraction and
$w\:\C\rarrow\R$ be a homogeneous section.
 Then CDG\+algebra morphisms $\Cb_w(\C)\rarrow\B$ correspond
bijectively to twisting cochains $\tau\:\C\rarrow\B$.

 Let us call a wcDG\+algebra morphism $(f,a)\:\B\rarrow\A$
a \emph{semi-isomorphism} if the DG\+algebra morphism
$f/\m f\:\B/\m\B\rarrow\A/\m\A$ is a quasi-isomorphism.
 The definition of the equivalence relation on CDG\+coalgebras
with conilpotent reductions requires a little more effort.

 Let $C$ be a coaugmented CDG\+coalgebra over~$k$ with
the coaugmentation $\bar w\:k\rarrow C$.
 An \emph{admissible filtration} $F$ on $C$ \cite{Hin}
(see also~\cite[Section~6.10]{Pkoszul}) is an increasing filtration
$F_0C\subset F_1C\subset F_2C\subset\dotsb$ that is preserved by
the differential, compatible with the comultiplication, and such that
$F_0C=\bar w(k)$ and $C=\bigcup_n F_nC$.
 The compatibility with the comultiplication means, as usually,
that $\mu(F_nC)\subset\sum_{p+q=n}F_pC\ot_k F_q C$, where
$\mu\:C\rarrow C\ot_k C$ is the comultiplication map.

 A coaugmented CDG\+coalgebra $C$ over $k$ admits an admissible
filtration if and only if it is conilpotent.
 A conilpotent CDG\+coalgebra $C$ has a canonical (maximal)
admissible filtration $F$ given by the rule
$F_nC=\ker(C\to (C/\bar w(k))^{\ot n+1})$;
for any other admissible filtration $G$ on $C$, one has
$G_nC\subset F_nC$.
 
 Let $F$ be an admissible filtration on a conilpotent
CDG\+coalgebra~$C$.
 Then the associated quotient coalgebra $\gr_FC=\bigoplus_n
F_nC/F_{n-1}C$ with the induced differential is a DG\+coalgebra
over~$k$ (since it is assumed that $h\circ \bar w=0$).
 A morphism of conilpotent CDG\+coalgebras $(g,a)\:C\rarrow D$
is said to be a \emph{filtered quasi-isomorphism} if there
exist admissible filtrations $F$ on $C$ and $D$ such that
$g(F_nC)\subset F_nD$ and the induced morphisms of complexes
$F_nC/F_{n-1}C\rarrow F_nD/F_{n-1}D$ are quasi-isomorphisms.

 A morphism of $\R$\+free CDG\+coalgebras with conilpotent
reductions $(f,a)\:\C\rarrow\D$ is called a \emph{filtered
semi-isomorphism} if its reduction $(f/\m f\;a/\m a)\:
\C/\m\C\rarrow\D/\m\D$ is a filtered quasi-isomorphism.
 In other words, there should exist admissible filtrations $F$
on the CDG\+coalgebras $\C/\m\C$ and $\D/\m\D$ over~$k$
such that the above conditions are satisfied.
 The classes of semi-isomorphisms of wcDG\+algebras over $\R$
and filtered semi-isomorphisms of $\R$\+free CDG\+coalgebras
with conilpotent reductions will be denoted by
$\Seis\subset\R\alg_\wcdg^+$ and
$\FSeis\subset\R\coalg_\cdg^\cnlp$.

 Let us make a warning that the class of filtered semi-isomorphisms
$\FSeis$ is \emph{not} closed under compositions, fractions, or
retractions of morphisms in $\R\coalg_\cdg^\cnlp$ (while the class
of semi-isomorphisms $\Seis\subset\R\alg_\wcdg^+$ is, of course,
closed under these operations).

\begin{thm}
 The functors\/ $\Br_v\:\R\alg_\wcdg^+\rarrow\R\coalg_\cdg^\cnlp$
and\/ $\Cb_w\:\allowbreak\R\coalg_\cdg^\cnlp\rarrow\R\alg_\wcdg^+$
induce functors between the localized categories
$$
 \R\alg_\wcdg^+[\Seis^{-1}]\lrarrow\R\coalg_\cdg^\cnlp[\FSeis^{-1}]
$$
and
$$
 \R\coalg_\cdg^\cnlp[\FSeis^{-1}]\lrarrow\R\alg_\wcdg^+[\Seis^{-1}],
$$
which are mutually inverse equivalences of categories.
\end{thm}

\begin{proof}
 It follows from the arguments in the proof
of~\cite[Theorem~6.10]{Pkoszul} that the functor $\Br_v$ takes
$\Seis$ to $\FSeis$ and the functor $\Cb_w$ takes $\FSeis$ to $\Seis$.
 Furthermore, the adjunction morphisms $\Cb_v(\Br_w(\A))\rarrow\A$
belong to $\Seis$ and the adjunction morphisms $\C\rarrow
\Br_w(\Cb_v(\C))$ belong to $\FSeis$ for all wcDG\+algebras
$\A\in\R\alg_\wcdg^+$ and all $\R$\+free CDG\+coalgebras with
conilpotent reductions $\C\in\R\coalg_\cdg^\cnlp$.
\end{proof}

\Section{Strictly Unital Weakly Curved \Ainfty-Algebras}

\subsection{Nonunital wc \Ainfty-algebras} \label{nonunital-wcainfty}
 Let $\A$ be a free graded $\R$\+contramodule.
 Consider the tensor coalgebra $\C=\bigoplus_{n=0}^\infty
\A[1]^{\ot n}$; it is an $\R$\+free graded coalgebra.
 Its reduction $\C/\m\C$ is the tensor coalgebra
$\bigoplus_{n=0}^\infty\A/\m\A[1]^{\ot n}$.
 Being a conilpotent graded $k$\+coalgebra, $\C/\m\C$ has a unique
coaugmentation $k\simeq\A/\m\A[1]^{\ot0}\rarrow
\bigoplus_n\A/\m\A[1]^{\ot n}$, which
we will denote by~$\bar w$.

 A \emph{nonunital wc~\Ainfty\+algebra} structure on $\A$ is,
by the definition, a structure of $\R$\+free DG\+coalgebra with
a conilpotent reduction on the $\R$\+free graded coalgebra
$\bigoplus_n\A[1]^{\ot n}$, i.~e., an odd coderivation
$d\:\bigoplus_n\A[1]^{\ot n}\rarrow\bigoplus_n\A[1]^{\ot n}$ of
degree~$1$ such that $d^2=0$ and $d/\m d\circ \bar w=0$.
 A nonunital wc~\Ainfty\+algebra $\A$ over $\R$ is a free graded
$\R$\+contramodule endowed with the mentioned structure.

 Since a coderivation of $\bigoplus_n\A[1]^{\ot n}$ is uniquely
determined by its composition with the projection
$\bigoplus_n\A[1]^{\ot n}\rarrow\A[1]^{\ot 1}\simeq\A[1]$,
a nonunital wc~\Ainfty\+algebra structure on $\A$ can be viewed
as a sequence of $\R$\+contramodule morphisms
$m_n\:\A^{\ot n}\rarrow\A$, \ $n=0$,~$1$, $2$,~\dots\ of
degree $2-n$ (see Lemma~\ref{free-algebra-derivations}(b)).
 For the sign rule (here and below), see~\cite[Section~7.1]{Pkoszul}.
 The sequence of maps~$m_n$ must satisfy the \emph{weak
curvature} condition $m_0/\m m_0=0$ (i.~e., $m_0\in\m\A^2$)
and a sequence of quadratic equations corresponding to
the equation $d^2=0$ on the coderivation~$d$.
 
 A morphism of nonunital wc~\Ainfty\+algebras $f\:\A\rarrow\B$ over
$\R$ is, by the definition, a morphism of $\R$\+free DG\+coalgebras
$\bigoplus_n\A[1]^{\ot n}\rarrow\bigoplus_n\B[1]^{\ot n}$
(such a morphism is always compatible with
the coaugmentations modulo~$\m$).

 Since an $\R$\+free graded coalgebra morphism into
$\bigoplus_n\B[1]^{\ot n}$ is uniquely determined by its
composition with the projection $\bigoplus_n\B[1]^{\ot n}
\rarrow\B[1]$, a morphism of nonunital wc~\Ainfty\+algebras
$f\:\A\rarrow\B$ can be viewed as a sequence of
$\R$\+contramodule morphisms $f_n\:\A^{\ot n}\rarrow\B$,
\ $n=0$,~$1$, $2$,~\dots\ of degree $1-n$ (see
Lemma~\ref{free-algebra-morphisms}(b)).
 The sequence of maps~$f_n$ must satisfy the \emph{weak
change of connection} condition $f_0/\m f_0=0$ (i.~e.,
$f_0\in\m\B^1$) and a sequence of polynomial equations corresponding
to the equation $d\circ f = f\circ d$ on the morphism~$f$.
 (With respect to~$f_0$, these are even formal power series rather
than polynomials.)

\begin{lem}  \label{free-co-module-derivations}
 Let\/ $\C$ be an\/ $\R$\+free graded coalgebra endowed with
an odd coderivation~$d$ of degree~$1$.  Then \par
\textup{(a)} for any free graded\/ $\R$\+contramodule\/ $\V$,
odd coderivations of degree~$1$ on the\/ $\R$\+free graded left\/
$\C$\+comodule\/ $\C\ot^\R\V$ compatible with the coderivation~$d$
on\/ $\C$ are determined by their compositions with the map\/
$\C\ot^\R\V\rarrow\V$ induced by the counit of\/~$\C$.
 Conversely, any homogeneous\/ $\R$\+contramodule morphism\/
$\C\ot^\R\V\rarrow\V$ of degree~$1$ gives rise to an odd
coderivation of degree~$1$ on\/ $\C\ot^\R\V$ compatible with
the coderivation~$d$ on\/~$\C$; \par
\textup{(b)} for any graded\/ $\R$\+contramodule\/ $\U$, odd
contraderivations of degree~$1$ on the\/ $\R$\+contramodule
graded left\/ $\C$\+contramodule\/ $\Hom^\R(\C,\U)$ compatible with
the coderivation~$d$ on\/ $\C$ are determined by their compositions
with the map\/ $\U\rarrow\Hom^\R(\C,\U)$ induced by the counit
of\/~$\C$.
 Conversely, any homogeneous\/ $\R$\+contramodule morphism\/
$\U\rarrow\Hom^\R(\C,\U)$ of degree~$1$ gives rise to an odd
contraderivation of degree~$1$ on\/ $\Hom^\R(\C,\U)$ compatible
with the coderivation~$d$ on\/~$\C$; \par
\textup{(c)} for any graded\/ $\R$\+comodule\/ $\cV$, odd
coderivations of degree~$1$ on the\/ $\R$\+comodule graded left\/
$\C$\+comodule\/ $\C\ocn_\R\cV$ compatible with the coderivation~$d$
on\/ $\C$ are determined by their compositions with the map\/
$\C\ocn_\R\cV\rarrow\cV$ induced by the counit of\/~$\C$.
 Conversely, any homogeneous\/ $\R$\+comodule morphism\/
$\C\ocn_\R\cV\rarrow\cV$ of degree~$1$ gives rise to an odd
coderivation of degree~$1$ on\/ $\C\ocn_\R\cV$ compatible with
the coderivation~$d$ on\/~$\C$; \par
\textup{(d)} for any cofree graded\/ $\R$\+comodule\/ $\cU$, odd
contraderivations of degree~$1$ on the\/ $\R$\+cofree graded left\/
$\C$\+contramodule\/ $\Ctrhom_\R(\C,\cV)$ compatible with
the coderivation~$d$ on\/ $\C$ are determined by their compositions
with the map\/ $\cV\rarrow\Ctrhom_\R(\C,\cV)$ induced by the counit
of\/~$\C$.
 Conversely, any homogeneous\/ $\R$\+comodule morphism\/
$\cV\rarrow\Ctrhom_\R(\C,\cV)$ of degree~$1$ gives rise to an odd
contraderivation of degree~$1$ on\/ $\Ctrhom_\R(\C,\cV)$ compatible
with the coderivation~$d$ on\/~$\C$.
\end{lem}

\begin{proof}
 Straightforward from the definitions.
\end{proof}

 Let $\A$ be a nonunital wc~\Ainfty\+algebra over $\R$ and $\M$ be
a free graded $\R$\+contramod\-ule.
 A structure of \emph{nonunital\/ $\R$\+free left wc~\Ainfty\+module}
over $\A$ on $\M$ is, by the definition, a structure of
DG\+comodule over the DG\+coalgebra $\bigoplus_n\A[1]^{\ot n}$ on
the injective $\R$\+free graded left comodule
$\bigoplus_n\A[1]^{\ot n}\ot^\R\M$ over the $\R$\+free graded
coalgebra $\bigoplus_n\A[1]^{\ot n}$.
 By Lemma~\ref{free-co-module-derivations}(a), a nonunital left
wc~\Ainfty\+module structure on $\M$ can be viewed as
a sequence of $\R$\+contramodule morphisms $l_n\:\A^{\ot n}\ot^\R
\M\rarrow\M$, \ $n=0$,~$1$,~\dots\ of degree $1-n$.
 The sequence of maps~$l_n$ must satisfy a system of nonhomogeneous
quadratic equations corresponding to the equation $d^2=0$ on
the coderivation~$d$ on $\bigoplus_n\A[1]^{\ot n}\ot^\R\M$.
 \emph{Nonunital\/ $\R$\+free right wc~\Ainfty\+modules} over $\A$
are defined similarly (see~\cite[Section~7.1]{Pkoszul}).

 The complex of morphisms between nonunital $\R$\+free left
wc~\Ainfty\+modules $\L$ and $\M$ over $\A$ is, by the definition,
the complex of morphisms between the $\R$\+free left
DG\+comodules $\bigoplus_n\A[1]^{\ot n}\ot^\R\L$ and
$\bigoplus_n\A[1]^{\ot n}\ot^\R\M$ over the $\R$\+free DG\+coalgebra
$\bigoplus_n\A[1]^{\ot n}$.
 Since an $\R$\+free graded $\C$\+comodule morphism into $\C\ot^\R\M$
is determined by its composition with the map $\C\ot^\R\M\rarrow\M$
induced by the counit of $\C$, a morphism of nonunital $\R$\+free
left wc~\Ainfty\+modules $f\:\L\rarrow\M$ of degree~$i$ can be viewed
as a sequence of $\R$\+contramodule morphisms
$f_n\:\A^{\ot n}\ot^\R\L\rarrow\M$, \ $n=0$,~$1$,~\dots\
of degree $i-n$ (see Section~\ref{r-free-graded-co}).
 Any sequence of homogeneous $\R$\+contramodule maps $f_n$ of
the required degrees corresponds to a (not necessarily closed)
morphism of nonunital $\R$\+free wc~\Ainfty\+modules~$f$.

 Let $\P$ be a graded $\R$\+contramodule.
 A structure of \emph{nonunital\/ $\R$\+contramodule left
wc~\Ainfty\+module} over $\A$ on $\P$ is, by the definition,
a structure of DG\+contramodule over the DG\+coalgebra
$\bigoplus_n\A[1]^{\ot n}$ on the induced $\R$\+contramodule graded
left contramodule $\Hom^\R(\bigoplus_n\A[1]^{\ot n}\;\P)$ over
the $\R$\+free graded coalgebra $\bigoplus_n\A[1]^{\ot n}$.
 By Lemma~\ref{free-co-module-derivations}(b), a nonunital
left wc~\Ainfty\+module structure on $\P$ can be viewed as
a sequence of $\R$\+contramodule morphisms $p_n\:\P\rarrow
\Hom^\R(\A^{\ot n}\;\P)$, \ $n=0$,~$1$,~\dots\ of degree $1-n$.
 The sequence of maps~$p_n$ must satisfy a sequence of
nonhomogeneous quadratic equations corresponding to the equation
$d^2=0$ on the contraderivation~$d$ on
$\Hom^\R(\bigoplus_n\A[1]^{\ot n}\;\P)$.

 The complex of morphisms between nonunital $\R$\+contramodule
left wc~\Ainfty\+modules $\P$ and $\Q$ over $\A$ is, by
the definition, the complex of morphisms between the $\R$\+free
left DG\+contramodules $\Hom^\R(\bigoplus_n\A[1]^{\ot n}\;\P)$
and $\Hom^\R(\bigoplus_n\A[1]^{\ot n}\;\Q)$ over
the $\R$\+free DG\+coalgebra $\bigoplus_n\A[1]^{\ot n}$.
 Since an $\R$\+contramodule graded $\C$\+contramodule morphism
from $\Hom^\R(\C,\P)$ is determined by the its composition with
the map $\P\rarrow\Hom^\R(\C,\P)$ induced by the counit of $\C$,
a morphism of nonunital $\R$\+contramodule left wc~\Ainfty\+modules
$f\:\P\rarrow\Q$ of degree~$i$ can be viewed as a sequence of
$\R$\+contra\-module morphisms $f^n\:\P\rarrow
\Hom^\R(\A^{\ot n}\;\Q)$ of degree~$i-n$ (see
Section~\ref{non-adj-graded-co}).
 Any sequence of homogeneous $\R$\+contramodule maps $f^n$ of
the required degrees corresponds to a (not necessarily closed)
morphism of nonunital $\R$\+contramodule wc~\Ainfty\+modules~$f$.

 For any $\R$\+free CDG\+coalgebra $\C$, the functors $\Phi_\C$
and $\Psi_\C$ of Section~\ref{r-free-co-derived} provide
an equivalence between the DG\+category of $\R$\+free left
CDG\+comodules over $\C$ which, considered as graded
$\C$\+comodules, are cofreely cogenerated by free graded
$\R$\+contramodules, and the DG\+category of $\R$\+free left
CDG\+contramodules over $\C$ which, considered as graded
$\C$\+contramodules, are freely generated by free graded
$\R$\+contramodules.
 So our terminology is consistent: a nonunital $\R$\+free
left wc~\Ainfty\+module $\M$ over $\A$ with the structure maps
$l_n\:\A^{\ot n}\ot^\R\M\rarrow\M$ can be equivalently viewed
as a nonunital $\R$\+contramodule wc~\Ainfty\+module with
a free underlying graded $\R$\+contramodule and
the structure maps $p_n\:\M\rarrow\Hom^\R(\A^{\ot n}\;\M)$
given by the rule $p_n(x)(a_1\ot\dotsb\ot a_n)=
(-1)^{|x|(|a_1|+\dotsb+|a_n|)}l_n(a_1\ot\dotsb\ot a_n\ot x)$.
 Similarly, a morphism of nonunital $\R$\+free left
wc~\Ainfty\+modules $\L\rarrow\M$ over $\A$ with the components
$f_n\:\A^{\ot n}\ot^\R\L\rarrow\M$ can be equivalently viewed
as a morphism of nonunital $\R$\+contramodule wc~\Ainfty\+modules
$\L\rarrow\M$ with the components $f^n\:\L\rarrow
\Hom^\R(\A^{\ot n}\;\M)$ given by the rule
$f^n(x)(a_1\ot\dotsb\ot a_n) =
(-1)^{|x|(|a_1|+\dotsb+|a_n|)}f_n(a_1\ot\dotsb\ot a_n\ot x)$.

 Let $\cM$ be a graded $\R$\+comodule.
 A structure of \emph{nonunital\/ $\R$\+comodule left
wc~\Ainfty\+module} over $\A$ on $\cM$ is, by the definition,
a structure of DG\+comodule over the DG\+coalgebra
$\bigoplus_n\A[1]^{\ot n}$ on the coinduced $\R$\+comodule
graded left comodule $\bigoplus_n\A[1]^{\ot n}\ocn_\R\cM$ over
the $\R$\+free graded coalgebra $\bigoplus_n\A[1]^{\ot n}$.
 By Lemma~\ref{free-co-module-derivations}(c), a nonunital
left wc~\Ainfty\+module structure on $\cM$ can be viewed as
a sequence of $\R$\+comodule morphisms $l_n\:\A^{\ot n}\ocn_\R
\cM\rarrow\cM$, \ $n=0$,~$1$,~\dots\ of degree $1-n$.
 The sequence of maps~$l_n$ must satisfy a sequence of
nonhomogeneous quadratic equations corresponding to
the equation $d^2=0$ on the coderivation~$d$ on
$\bigoplus_n\A[1]^{\ot n}\ocn_\R\cM$.
 \emph{Nonunital\/ $\R$\+comodule right wc~\Ainfty\+modules}
over $\A$ are defined similarly.

 The complex of morphisms between nonunital $\R$\+comodule left
wc~\Ainfty\+modules $\cL$ and $\cM$ over $\A$ is, by the definition,
the complex of morphisms between the $\R$\+comodule left
DG\+comodules $\bigoplus_n\A[1]^{\ot n}\ocn_\R\cL$ and
$\bigoplus_n\A[1]^{\ot n}\ocn_\R\cM$ over the $\R$\+free
DG\+coalgebra $\bigoplus_n\A[1]^{\ot n}$.
 Since an $\R$\+comodule graded $\C$\+comodule morphism into
$\C\ocn_\R\cM$ is determined by its composition with the map
$\C\ocn_\R\cM\rarrow\cM$ induced by the counit of $\C$,
a morphism of nonunital $\R$\+comodule left wc~\Ainfty\+modules
$f\:\cL\rarrow\cM$ of degree~$i$ can be viewed as a sequence
of $\R$\+comodule morphisms $f_n\:\A^{\ot n}\ocn_\R\cL
\rarrow\cM$, \ $n=0$,~$1$,~\dots\ of degree $i-n$
(see Section~\ref{non-adj-graded-co}).
 Any sequence of homogeneous $\R$\+comodule maps~$f_n$ of
the required degrees corresponds to a (not necessarily closed)
morphism of nonunital $\R$\+comodule wc~\Ainfty\+modules~$f$.

 For any $\R$\+free CDG\+coalgebra $\C$, the functors $\Phi_\C$
and $\Psi_\C$ of Section~\ref{r-cofree-co-derived} provide
an equivalence between the DG\+category of $\R$\+cofree left
CDG\+comodules over $\C$ which, considered as graded
$\C$\+comodules, are cofreely cogenerated by cofree graded
$\R$\+comodules, and the DG\+category of $\R$\+cofree left
CDG\+contramodules over $\C$ which, considered as graded
$\C$\+contramodules, are freely generated by cofree graded
$\R$\+comodules.
 So one can alternatively define a \emph{nonunital\/ $\R$\+cofree
left wc~\Ainfty\+module} $\cP$ over $\A$ as a cofree graded
$\R$\+comodule for which a structure of DG\+contramodule over
the DG\+coalgebra $\bigoplus_n\A[1]^{\ot n}$ is given on
the projective $\R$\+cofree graded contramodule
$\Ctrhom_\R(\bigoplus_n\A[1]^{\ot n}\;\cP)$.
 By Lemma~\ref{free-co-module-derivations}(d), such a structure
can be viewed as a sequence of $\R$\+comodule maps
$p_n\:\cP\rarrow\Ctrhom_\R(\A^{\ot n}\;\cP)$, \ $n=0$,~$1$,~\dots\
of degree $1-n$.
 The maps~$p_n$ are related to the above maps~$l_n$ by
the above rule.

 Furthermore, one can alternatively define the complex of
morphisms between nonunital $\R$\+cofree left wc~\Ainfty\+modules
$\cP$ and $\cQ$ over $\A$ as the complex of morphisms between
left DG\+contramodules $\Hom_\R(\bigoplus_n\A[1]^{\ot n}\;\cP)$
and $\Hom_\R(\bigoplus_n\A[1]^{\ot n}\;\cQ)$ over
the $\R$\+free DG\+coalgebra $\bigoplus_n\A[1]^{\ot n}$.
 Thus a (not necessarily closed) morphism of nonunital
$\R$\+cofree left wc~\Ainfty\+modules $f\:\cP\rarrow\cQ$
is the same thing as a sequence of $\R$\+comodule maps
$f^n\:\cP\rarrow\Ctrhom_\R(\A^{\ot n}\;\cQ)$ of degree~$i-n$
(see Section~\ref{r-cofree-graded-co}).
 The maps~$f^n$ are related to the above maps~$f_n$ by
the above rule.

\begin{prop}  \label{nonunital-wcainfty-r-c-co-contra}
 The equvalences of DG\+categories\/ $\Phi_\R=\Psi_\R^{-1}$
from Section~\textup{\ref{r-cofree-co-derived}} and
the DG\+functors $\Phi_{\R,\C}$, $\Psi_{\R,\C}$ from
Section~\textup{\ref{non-adj-co-derived}} restrict to an equivalence\/
$\Phi_\R=\Psi_\R^{-1}$ between the DG\+categories of nonunital\/
$\R$\+free left wc~\Ainfty\+modules and nonunital\/ $\R$\+cofree left
wc~\Ainfty\+modules over\/~$\A$ making a commutative diagram with
the forgetful functors and the equivalence\/ $\Phi_\R=\Psi_\R^{-1}$
between the categories of free graded\/ $\R$\+contramodules and
cofree graded\/ $\R$\+comodules from
Proposition~\textup{\ref{hom-operations}}. \qed
\end{prop}

 By the definition, the category of nonunital wc~\Ainfty\+algebras
over $\R$ is isomorphic to the full subcategory of the category
of $\R$\+free DG\+coalgebras consisting of those DG\+coalgebras $\C$
whose underlying graded coalgebras are the tensor coalgebras
$\bigoplus_n\U^{\ot n}$ and whose reductions $\U/\m\U$ are
conilpotent DG\+coalgebras.
 Analogously, the DG\+category of nonunital $\R$\+free
wc~\Ainfty\+modules over $\A$ is isomorphic to the full subcategory
in the category of $\R$\+free DG\+comodules $\N$ over $\C$ consisting
of those DG\+comodules whose underlying graded comodules are
the graded comodules $\C\ot^\R\M$ (and similarly for nonunital
$\R$\+cofree wc~\Ainfty\+modules).
 The following lemmas provide a characterization formulated entirely
in terms of the reductions modulo~$\m$.

\begin{lem}
 \textup{(a)} An\/ $\R$\+free graded comodule\/ $\N$ over
an\/ $\R$\+free graded coalgebra\/ $\C$ is cofreely cogenerated by
a free graded\/ $\R$\+contramodule\/ $\V$ (i.~e., is isomorphic to
$\C\ot^\R\V$) if and only if the graded comodule\/ $\N/\m\N$ over
the graded coalgebra\/ $\C/\m\C$ is isomorphic to\/
$\C/\m\C\ot_k\V/\m\V$. \par
 \textup{(b)} An\/ $\R$\+free graded contramodule\/ $\Q$ over
an\/ $\R$\+free graded coalgebra\/ $\C$ is freely generated by
a free graded\/ $\R$\+contramodule\/ $\U$ (i.~e., is isomorphic to
$\Hom^\R(\C,\U)$) if and only if the graded contramodule\/
$\Q/\m\Q$ over the graded coalgebra\/ $\C/\m\C$ is isomorphic to\/
$\Hom_k(\C/\m\C\;\U/\m\U)$. \par
 \textup{(c)} An\/ $\R$\+cofree graded comodule\/ $\cN$ over
an\/ $\R$\+free graded coalgebra\/ $\C$ is cofreely cogenerated by
a cofree graded\/ $\R$\+comodule\/ $\cV$ (i.~e., is isomorphic to
$\C\ocn_\R\cV$) if and only if the graded comodule\/ ${}_\m\cN$
over the graded coalgebra\/ $\C/\m\C$ is isomorphic to\/
$\C/\m\C\ot_k{}_\m\cV$. \par
 \textup{(d)} An\/ $\R$\+cofree graded contramodule\/ $\cQ$ over
an\/ $\R$\+free graded coalgebra\/ $\C$ is freely generated by
a cofree graded\/ $\R$\+comodule\/ $\cU$ (i.~e., is isomorphic to
$\Ctrhom_\R(\C,\cU)$) if and only if the graded contramodule\/
${}_\m\cQ$ over the graded coalgebra\/ $\C/\m\C$ is isomorphic to\/
$\Hom_k(\C/\m\C\;{}_\m\cU)$.
\end{lem}

\begin{proof}
 One lifts an isomorphism of the reductions to a morphism of
$\C$\+contra/comod\-ules using the projectivity/injectivity
properties of the (co)freely (co)generated contra/comodules in
the exact categories of $\R$\+(co)free graded
$\C$\+contra/comodules.
 Then it remains to use Lemma~\ref{nakayama-lemma}
or~\ref{comodule-nakayama}.
 See the proofs of Lemmas~\ref{r-free-co-inj-proj-reduction},
\ref{r-free-proj-inj-reduction}, and
Lemma~\ref{tensor-coalgebra-reduction} below for further
details.
\end{proof}

\begin{lem} \label{tensor-coalgebra-reduction}
 Let\/ $\C$ be an\/ $\R$\+free graded coalgebra and\/ $\U$ be
a free graded\/ $\R$\+contramodule.
 Then\/ $\C$ is isomorphic to the $\R$\+free graded coalgebra\/
$\bigoplus_n\U^{\ot n}$ if and only if the graded
$k$\+coalgebra\/ $\C/\m\C$ is isomorphic to\/
$\bigoplus_n(\U/\m\U)^{\ot n}$.
\end{lem}

\begin{proof}
 The ``only if'' part is obvious.
 To prove the ``if'', consider the composition $\C\rarrow\C/\m\C\simeq
\bigoplus_n(\U/\m\U)^{\ot n}\rarrow(\U/\m\U)^{\ot 1}\simeq\U/\m\U$
and lift it to a graded $\R$\+contramodule morphism $\C\rarrow\U$.
 Since the isomorphism of graded $k$\+coalgebras $\C/\m\C\simeq
\bigoplus_n(\U/\m\U)^{\ot n}$, as any morphism of conilpotent graded
coalgebras over~$k$, commutes with the coaugmentations, it follows
from Lemma~\ref{free-algebra-morphisms}(b) that the map
$\C\rarrow\U$ gives rise to a morphism of $\R$\+free graded
coalgebras $\C\rarrow\bigoplus_n\U^{\ot n}$.
 By the uniqueness assertion of the same Lemma applied to
the case of graded coalgebras over~$k$, the morphism
$\C\rarrow\bigoplus_n\U^{\ot n}$ reduces to the isomorphism
$\C/\m\C\simeq\bigoplus_n(\U/\m\U)^{\ot n}$ that we started from.
 It remains to apply Lemma~\ref{nakayama-lemma} in order to conclude
that the morphism $\C\rarrow\bigoplus_n\U^{\ot n}$ is an isomorphism.
\end{proof}

 Let $f\:\A\rarrow\B$ be a morphism of wc~\Ainfty\+algebras over $\R$
corresponding to a morphism of $\R$\+free DG\+coalgebras
$\C=\bigoplus_n\A[1]^{\ot n}\rarrow\D=\bigoplus_n\B[1]^{\ot n}$.
 Let $\M$ be an $\R$\+free wc~\Ainfty\+module over $\B$ corresponding
to an $\R$\+free DG\+comodule $\N=\bigoplus_n\B[1]^{\ot n}\ot^\R\M$
over~$\D$.
 Then the structure of wc~\Ainfty\+module over $\A$ on $\M$ obtained
by the restriction of scalars with respect to the morphism~$f$
from the given structure of wc~\Ainfty\+module over $\B$ corresponds,
by the definition, to the $\R$\+free DG\+comodule structure
$E_f(\N)=\C\oc_\D\N$ on the $\R$\+free graded $\C$\+comodule
$\bigoplus_n\A[1]^{\ot n}\ot^\R\M$ (see
Section~\ref{r-free-co-derived}).
 The structure maps~$l_n'$ of the new structure of wc~\Ainfty\+module
over $\A$ on $\M$ are expressed in terms of the structure maps~$l_n$
of the original structure of wc~\Ainfty\+module over $\B$ and
the structure maps~$f_n$ of the wc~\Ainfty\+morphism $f\:\A\rarrow\B$
as certain sums of compositions (with signs) depending linearly
on~$l_n$ and polynomially on~$f_n$ (and even as formal power series
on~$f_0$).
 Clearly, the restriction of scalars with respect to a morphism of
wc~\Ainfty\+algebras is naturally a DG\+functor (since the coextension
of scalars $E_f$ with respect to a morphism of $\R$\+free
DG\+coalgebras is).

 Analogously, let $\P$ be an $\R$\+contramodule wc~\Ainfty\+module
over $\B$ corresponding to an $\R$\+contramodule DG\+contramodule
$\Q=\Hom^\R(\bigoplus_n\B[1]^{\ot n}\;\P)$ over~$\D$.
 Then the structure of wc~\Ainfty\+module over $\A$ on $\P$
obtained by the restriction of scalars with respect to
the morphism~$f\:\A\rarrow\B$ from the given structure of
wc~\Ainfty\+module over $\B$ on $\P$ corresponds, by the definition,
to the $\R$\+contramodule DG\+contramodule structure
$E^f(\Q)=\Cohom_\D(\C,\Q)$ on the $\R$\+contramodule graded
$\C$\+contramodule $\Hom^\R(\bigoplus_n\A[1]^{\ot n}\;\P)$.
 One can easily see that this construction agrees with the previous
one in the case of $\R$\+free wc~\Ainfty\+modules
(cf.\ Proposition~\ref{r-free-co-extension}).
 The restriction-of-scalars functors for $\R$\+comodule and
$\R$\+cofree wc~\Ainfty\+modules are constructed in
the similar way (see Section~\ref{r-cofree-co-derived}).

\begin{rem}
 The assignment of the underlying graded $\R$\+contramodule or
$\R$\+co\-module to an $\R$\+contramodule or $\R$\+comodule
wc~\Ainfty\+module over a wc~\Ainfty\+algebra $\A$ over $\R$
does \emph{not} have good functorial properties.

 From the point of view of graded contramodules or comodules
over the tensor coalgebra $\C=\bigoplus_n\A[1]^{\ot n}$,
to recover, say, a free graded $\R$\+contramodule $\M$ from
the $\R$\+free graded $\C$\+comodule $\N=\C\ot^\R\M$ one has
to use the construction of cotensor product with $\R$ over $\C$,
that is $\M=\R\oc_\C\N$.
 This depends on the structure of right $\C$\+comodule on $\R$,
which is defined in terms of the coaugmentation map
$\R=\A[1]^{\ot 0}\rarrow\bigoplus_n\A[1]^{\ot n}=\C$.
 The problem is, such coaugmentations over $\R$ are \emph{not}
preserved by $\R$\+free DG\+coalgebra morphisms (or even
isomorphisms) $f\:\A\rarrow\B$, because of the presence of the
change-of-connection component $f_0\in\m\B^1$.
 
 Accordingly, the functor of forgetting the wc~\Ainfty\+module
structure is well-defined for wc~\Ainfty\+modules over
a \emph{fixed} wc~\Ainfty\+algebra $\A$ (as, e.~g., in
the setting of Proposition~\ref{nonunital-wcainfty-r-c-co-contra}),
but such functors do \emph{not} form commutative diagrams with
the restriction-of-scalars functors for (iso)morphisms of
wc~\Ainfty\+algebras.
 This is a phenomenon specific to the wc~\Ainfty\ situation; it does
not occur for wcDG\+modules over $\R$\+free wcDG\+algebras or
\Ainfty\+modules over \Ainfty\+algebras over fields.
 The described problem also does not manifest itself for
the forgetful functors acting from the DG\+categories of
wc~\Ainfty\+modules and \emph{strict} morphisms between them
(see Section~\ref{wcainfty-semiderived}).

 The situtation is even worse for the assignment of the underlying
free graded $\R$\+contramodule to a wc~\Ainfty\+algebra over $\R$,
which is \emph{not} a functor on the category of wc~\Ainfty\+algebras 
at all.
 It is a functor, however, on the category of wc~\Ainfty\+algebras
and morphisms~$f$ between them for which $f_0=0$, and it is also
a functor on the category of wc~\Ainfty\+algebra morphisms~$f$
such that $f_i=0$ for $i\ge2$.

 In general, it is only the reduction modulo~$\m$ of the underlying
free graded $\R$\+contramodule of a wc~\Ainfty\+algebra that is
a well-behaved operation (which even produces a complex of
vector spaces).
 Similarly, the functors of reduction modulo~$\m$ of the underlying
free graded $\R$\+contramodules or cofree graded $\R$\+comodules
of $\R$\+(co)free wc~\Ainfty\+modules are well-behaved (and
produce complexes).
\end{rem}

\subsection{Strictly unital wc \Ainfty-algebras}
\label{strictly-unital-wcainfty}
 Let $\A$ be a wc~\Ainfty\+algebra over~$\R$.
 An~element $1\in\A^0$ (or equivalently, the corresponding morphism of
graded $\R$\+contramod\-ules $\R\rarrow\A$) is called a \emph{strict
unit} if the compositions of the two maps $\A\rarrow\A\ot^\R\A$
induced by the unit map with the map $m_2\:\A\ot^\R\A\rarrow\A$ are
equal to the identity endomorphism of $\A$, while the compositions
of the $n$~maps $\A^{\ot n-1}\rarrow\A^{\ot n}$ induced by the unit
map with the map $m_n\:\A^{\ot n}\rarrow\A$ all vanish for $n\ge1$,
\ $n\ne 2$.
 It follows from the condition on~$m_2$ that a strict unit is unique
if it exists.  (See~\cite[Section~2.3.2]{Lef}
or~\cite[Section~7.2]{Pkoszul}.)

 A \emph{strictly unital wc~\Ainfty\+algebra} over $\R$ is
a nonunital wc~\Ainfty\+algebra that has a strict unit.
 A morphism of strictly unital wc~\Ainfty\+algebras $f\:\A\rarrow\B$
is a morphism of nonunital wc~\Ainfty\+algebras for which
$f_1(1_\A)=1_\B$ (i.~e., the unit maps $\R\rarrow\A$ and
$\R\rarrow\B$ form a commutative diagram with~$f_1$) and
the compositions of the $n$~maps $\A^{\ot n-1}\rarrow\A^{\ot n}$
induced by~$1_\A$ with the map $f_n\:\A^{\ot n}\rarrow\B$
vanish for all $n\ge2$.

 Let $\A$ be a strictly unital wc~\Ainfty\+algebra over~$\R$.
 A \emph{strictly unital\/ $\R$\+contramodule left wc~\Ainfty\+module}
$\P$ over $\A$ is a nonunital left wc~\Ainfty\+module such that
the composition of the structure map~$p_1$ with the map $\Hom^\R(\A,\P)
\rarrow\P$ induced by the unit map is the identity endomorphism
of $\P$, while the compositions of the map $p_n\:\P\rarrow
\Hom^\R(\A^{\ot n}\;\P)$ with the $n$~maps $\Hom^\R(\A^{\ot n}\;\P)
\rarrow\Hom^\R(\A^{\ot n-1}\;\P)$ induced by the counit map all
vanish for $n\ge2$.
 In the case of an $\R$\+free left wc~\Ainfty\+module $\M$, one can
rewrite these equations equivalently in terms of the maps~$l_n$,
requiring that the composition of the map $\M\rarrow\A\ot^\R\M$
induced by the unit map with the map $l_1\:\A\ot^\R\M\rarrow\M$
be the identity endomorphism of $\M$, while the compositions of
the $n$~maps $\A^{\ot n-1}\ot^\R\M\rarrow\A^{\ot n}\ot^\R\M$
induced by the unit map with the map $l_n\:\A^{\ot n}\ot^\R\M
\rarrow\M$ vanish for all $n\ge2$.
 \emph{Strictly unital\/ $\R$\+free right wc~\Ainfty\+modules}
over $\A$ are defined similarly.

 The complex of morphisms between strictly unital $\R$\+contramodule
left wc~\Ainfty\+mod\-ules $\P$ and $\Q$ over $\A$ is a complex of
$\R$\+contramodules defined as the subcomplex of the complex of
morphisms between $\P$ and $\Q$ as nonunital wc~\Ainfty\+modules
consisting of all morphisms $f\:\P\rarrow\Q$ such that
the compositions of the map $f^n\:\P\rarrow\Hom^\R(\A^{\ot n}\;\Q)$
with the $n$~maps $\Hom^\R(\A^{\ot n}\;\Q)\rarrow
\Hom^\R(\A^{\ot n-1}\;\Q)$ induced by the counit map all vanish
for $n\ge1$.
 For $\R$\+free left wc~\Ainfty\+modules $\L$ and $\M$, one can
rewrite these equations equivalently in terms of the maps~$f_n$,
requiring that the compositions of the $n$~maps $\A^{\ot n-1}\ot^\R\L
\rarrow\A^{\ot n}\ot^\R\L$ induced by the unit map with
the map $f_n\:\A^{\ot n}\ot^\R\L\rarrow\M$ vanish for all $n\ge1$.
{\hfuzz=3.2pt\par}

 A \emph{strictly unital\/ $\R$\+comodule left wc~\Ainfty\+module}
$\cM$ over $\A$ is a nonunital left wc~\Ainfty\+module such that
the composition of the map $\cM\rarrow\A\ocn_\R\cM$ induced by
the unit map with the structure map~$l_1$ is the identity endomorphism
of $\cM$, while the compositions of the $n$~maps $\A^{\ot n-1}\ocn_\R
\cM\rarrow\A^{\ot n}\ocn_\R\cM$ induced by the unit map with the map
$l_n\:\A^{\ot n}\ocn_\R\cM\rarrow\cM$ all vanish for $n\ge2$.
 In the case of an $\R$\+cofree left wc~\Ainfty\+module $\cP$, one
can rewrite these equations equivalently in terms of the maps~$p_n$,
just as it was done above.
 \emph{Strictly unital\/ $\R$\+comodule right wc~\Ainfty\+modules}
over $\A$ are defined similarly.

 The complex of morphisms between strictly unital $\R$\+comodule
left wc~\Ainfty\+modules $\cL$ and $\cM$ over $\A$ is a complex of
$\R$\+contramodules defined as the subcomplex of the complex
of morphisms between $\cL$ and $\cM$ as nonunital wc~\Ainfty\+modules
consisting of all morphisms $f\:\cL\rarrow\cM$ such that
the compositions of the $n$~maps $\A^{\ot n-1}\ocn_\R\cL\rarrow
\A^{\ot n}\ocn_\R\cL$ induced by the unit map with the map
$f_n\:\A^{\ot n}\ocn_\R\cL\rarrow\cM$ vanish for all $n\ge1$.
 For $\R$\+cofree left wc~\Ainfty\+modules $\cP$ and $\cQ$, one can
rewrite these equations equivalently in terms of the maps~$f^n$,
as it was done above.

 The following theorems show that strictly unital
wc~\Ainfty\+algebras and wc~\Ainfty\+mod\-ules are related to
CDG\+coalgebras and CDG\+contra/comodules in the same way as
nonunital wc~\Ainfty\+algebras and wc~\Ainfty\+modules are,
by the definition, related to DG\+coalgebras and DG\+contra/comodules. 
 The basic idea goes back to the paper~\cite{Pcurv}; our exposition
here follows that in~\cite[Section~7.2]{Pkoszul}, where similar
results were obtained for \Ainfty\+algebras and \Ainfty\+modules.

\begin{thm}  \label{strictly-unital-wcainfty-cdg-coalgebras}
 The category of nonzero strictly unital wc~\Ainfty\+algebras over\/
$\R$ is naturally equivalent to the subcategory of the category of\/
$\R$\+free CDG\+coalgebras formed by the CDG\+algebras\/ $\C$ whose
reductions\/ $\C/\m\C$ are coaugmented CDG\+coalge\-bras with
the underlying graded coalgebra isomorphic to\/ $\bigoplus_n U^{\ot n}$
for some graded $k$\+vector space $U$, and\/ $\R$\+free CDG\+coalgebra
morphisms\/ $\C\rarrow\D$ whose reductions\/ $\C/\m\C\rarrow\D/\m\D$
are morphisms of coaugmented CDG\+coalgebras over~$k$.
\end{thm}

 In other words, the category of strictly unital wc~\Ainfty\+algebras
over~$\R$ is equivalent to the full subcategory in the category
of $\R$\+free CDG\+coalgebras with conilpotent reductions
$\R\coalg_\cdg^\cnlp$ consisting of the CDG\+coalgebras $\C$
whose reductions $\C/\m\C$ have their underlying conilpotent graded
$k$\+coalgebras cofreely cogenerated by graded $k$\+vector spaces.

\begin{proof}[Proof of
Theorem~\textup{\ref{strictly-unital-wcainfty-cdg-coalgebras}}]
 The construction is based on the following \Ainfty\ generalization
of Lemma~\ref{unit-split}(a).

\begin{lem}  \label{wcainfty-unit-split}
 If\/ $\A$ is a nonzero strictly unital wc~\Ainfty\+algebra over\/
$\R$, then the unit map\/ $\R\rarrow\A$ is the embedding of a direct
summand in the category of free graded\/ $\R$\+contramodules.
\end{lem}

\begin{proof}
 Reducing modulo~$\m$, we obtain a strictly unital
\Ainfty\+algebra $\A/\m\A$ over the field~$k$.
 In particular, the reduction $k\rarrow\A/\m\A$ of the unit map
$\R\rarrow\A$ is a unit element of the (nonassociative) algebra
structure on $\A/\m\A$ given by the operation~$m_2$.
 Hence if the map $k\rarrow\A/\m\A$ is zero, then $\A/\m\A=0$
and $\A=0$.
 The rest of the argument is the same as in the proof of
Lemma~\ref{unit-split}(a).
\end{proof}

 Let $\A$ be a free graded $\R$\+contramodule and $e_\A\:\R\rarrow\A$
be a graded $\R$\+contramodule morphism with an $\R$\+free cokernel
(i.~e., an embedding of a direct summand in the category of free
graded $\R$\+contramodules); denote the cokernel by $\A_+=\A/\R$.
 Then there is a surjective morphism of $\R$\+free graded coalgebras
$\bigoplus_n\A[1]^{\ot n}\rarrow\bigoplus_n\A_+[1]^{\ot n}$.
 Denote by $\K_\A$ the kernel of this morphism; then $\K_\A$ is
the direct sum of the kernels of the morphisms $\A[1]^{\ot n}
\rarrow\A_+[1]^{\ot n}$ taken in the category of free graded
$\R$\+contramodules.
 Denote by $\varkappa_\A\:\K_\A\rarrow\R$ the $\R$\+contramodule
morphism of the total degree~$1$ vanishing on all the components of
tensor degree $n\ne1$ and defined on the component of tensor
degree~$1$ by the condition that the composition
$\R\rarrow\K_\A\rarrow\R$ of the emdedding $\R\rarrow\K_\A$
induced by the map $\R\rarrow\A$ with the map~$\varkappa_\A$ must be
the identity map.
 Let $\theta_\A\:\bigoplus_n\A[1]^{\ot n}\rarrow\R$ be any
$\R$\+contramodule morphism of the (total) degree~$1$ extending
the map~$\varkappa_\A$.

\begin{lem}  \label{d-prime-h-prime-coalgebra}
\textup{(a)} Let\/ $\A$ be a wc~\Ainfty\+algebra with the wc~\Ainfty
structure $d\:\bigoplus_n\A[1]^{\ot n}\allowbreak
\rarrow\bigoplus_n\A[1]^{\ot n}$.
 Then an embedding of a direct summand in the category of graded\/
$\R$\+contramodules $e_\A\:\R\rarrow\A$ is a strict unit of
the wc~\Ainfty\+structure on\/ $\A$ if and only if the following
two conditions hold:
\begin{enumerate}
\renewcommand{\theenumi}{\roman{enumi}}
\item the odd coderivation  of degree~$1$ on the tensor coalgebra\/
$\bigoplus_n\A[1]^{\ot n}$ given by the formula
$d'(c)=d(c)+\theta_\A*c-(-1)^{|c|}c*\theta_\A$ 
(the notation of Section~\textup{\ref{r-free-co-derived}})
preserves the\/ $\R$\+subcontramodule
$\K_\A\subset\bigoplus_n \A[1]^{\ot n}$, and
\item the\/ $\R$\+contramodule map $h'\:\bigoplus_n\A[1]^{\ot n}
\rarrow\R$ of degree~$2$ given by the formula
$h'(c)=\theta_\A(d(c))+\theta_\A^2(c)$ vanishes in the restriction
to~$\K_\A$.
\end{enumerate} \par
\textup{(b)} Let\/ $\A$ and\/ $\B$ be strictly unital wc~\Ainfty
algebras over\/ $\R$ with the units\/ $e_\A\:\R\rarrow\A$ and\/
$e_\B\:\R\rarrow\B$, and let $f\:\A\rarrow\B$ be a morphism of
nonunital wc~\Ainfty\+algebras.
 Then~$f$ is a morphism of strictly unital wc~\Ainfty\+algebras if
and only if \emergencystretch=0em
\begin{enumerate}
\renewcommand{\theenumi}{\roman{enumi}}
\item the morphism of\/ $\R$\+free DG\+coalgebras
$f\:\bigoplus_n\A[1]^{\ot n}\rarrow\bigoplus_n\B[1]^{\ot n}$
takes $\K_\A$ into~$\K_\B$, and
\item the composition of the map $f|_{\K_\A}\:\K_\A\rarrow\K_\B$
with the linear function~$\varkappa_\B$ on $\K_\B$ is equal to
the linear function~$\varkappa_\A$ on~$\K_\A$.
\end{enumerate}
\end{lem}

\begin{proof}
 Part~(a): Clearly, the condition~(i) does not depend on the choice
of a linear function~$\theta_\A$ (for a given~$e_\A$);
the condition~(ii) does not depend on the choice of~$\theta_\A$
assuming that the condition~(i) is fulfilled.
 The condition~(i) is equivalent to the vanishing of the composition
of the differential~$d'$ with the embedding $\K_\A\rarrow
\bigoplus_n\A[1]^{\ot n}$ and the composition of the projections
$\bigoplus_n\A[1]^{\ot n}\rarrow\A[1]^{\ot 1}\simeq\A[1]
\rarrow\A_+[1]$.

 The equations of the strict unit prescribe the values of certain
maps landing in~$\A$.
 The above vanishing condition is equivalent to these equations
being true modulo~$e_\A(\R)$.
 Furthermore, assume that the linear function~$\theta_\A$ is chosen
in such a way that it vanishes on the components of the tensor
degree different from~$1$.
 Then the choice of~$\theta_\A$ is equivalent to the choice of
a direct sum decomposition $\A\simeq\A_+\oplus\R$.
 The condition~(ii) is equivalent to the equations on the strict
unit being true in the projection to the component~$\R$.

 Part~(b): The condition~(i) is equivalent to the vanishing of
the composition $\K_\A\rarrow\bigoplus_n\A[1]^{\ot n}\rarrow
\bigoplus_n\B[1]^{\ot n}\rarrow\B[1]^{\ot 1}\simeq\B[1]$.
 The equations of a morphism of strictly unital wc~\Ainfty\+algebras
prescribe the values of certain maps landing in~$\B$.
 The above vanishing condition is equivalent to these equations
being true modulo~$e_\B(\R)$.
 Assuming that this condition holds, the maps we are interested in
actually land in $e_\B(\R)\subset\B$.
 The condition~(ii) means that they take the prescribed values there.
\end{proof}

 Let $\A$ be a strictly unital wc~\Ainfty\+algebra.
 Pick a homogeneous retraction $v\:\A\rarrow\R$ of the unit map
(i.~e., a graded $\R$\+contramodule morphism left inverse to~$e_\A$).
 Define the $\R$\+contramodule morphism $\theta_\A\:
\bigoplus_n\A[1]^{\ot n}\rarrow\R$ of degree~$1$ by the rule that
$\theta_\A=v$ on the components of the tensor degree~$1$ and 
$0$~on the components of all other tensor degrees.
 Then the $\R$\+contramodule morphism~$\theta_\A$ extends
the $\R$\+contramodule morphism $\varkappa_\A\:\K_\A\rarrow\R$.
 The odd coderivation $d'\:\bigoplus_n\A[1]^{\ot n}
\rarrow\bigoplus_n\A[1]^{\ot n}$ corresponding to~$\theta_\A$ takes
$\K_\A$ to $\K_\B$ and therefore induces an odd coderivation
$d\:\bigoplus_n\A_+[1]^{\ot n}\rarrow\A_+[1]^{\ot n}$ on the tensor
coalgebra $\bigoplus_n\A_+[1]^{\ot n}$.
 Similarly, the morphism $h'\:\bigoplus_n\A[1]^{\ot n}\rarrow\R$
corresponding to~$\theta_\A$ annihilates $\K_\A$ and therefore
induces a morphism $h\:\bigoplus_n\A_+[1]\rarrow\R$.

 The \emph{bar-construction} $\Br_v(\A) =
(\bigoplus_n\A_+[1]^{\ot n}\;d\;h)$ of a strictly unital
wc~\Ainfty\+al\-gebra $\A$ is the $\R$\+free CDG\+coalgebra
corresponding to $\A$ under the desired equivalences of categories.
 The CDG\+coalgebra $\Br_v(\A)/\m\Br_v(\A)$ over~$k$ is coaugmented
(and therefore, conilpotent) because $m_0\in\m\A$.

 Let $f\:\A\rarrow\B$ be a morphism of strictly unital
wc~\Ainfty\+algebras.
 Let $v_\A\:\A\rarrow\R$ and $v_\B\:\B\rarrow\R$ be homogeneous
retractions of the unit maps, and let $\theta_\A$ and $\theta_\B$
be the corresponding linear functions of degree~$1$ on
$\R$\+free graded tensor coalgebras.
 The morphism of tensor coalgebras $f\:\bigoplus_n\A[1]^{\ot n}
\rarrow\bigoplus_n\B[1]^{\ot n}$ takes $\K_\A$ into $\K_\B$, so
it induces a morphism of $\R$\+free graded coalgebras
$f\:\bigoplus_n\A_+[1]\rarrow\bigoplus_n\B_+[1]$.
 The $\R$\+contramodule morphism $\theta_\B\circ f-\theta_\A\:
\bigoplus_n\A[1]^{\ot n}\rarrow\R$ annihilates $\K_\A$, so
it induces an $\R$\+contramodule morphism $a_f\:\bigoplus_n
\A_+[1]\rarrow\R$.

 The pair $(f,a_f)$ is the morphism of $\R$\+free CDG\+coalgebras
$\Br_v(\A)\rarrow\Br_v(\B)$ corresponding to the morphism
of strictly unital wc~\Ainfty\+algebras $f\:\A\rarrow\B$.
 The morphism $(f/\m f,a_f/\m a_f)$ is a morphism of
coaugmented/conilpotent CDG\+coal\-gebras (i.~e.,
$a_f/\m a_f\circ \bar w=0$) because $f_0\in \m\B$ and
$\theta_\A/\m\theta_\A\circ \bar w=0=\theta_\B/\m\theta_\B
\circ\bar w$.
 Alternatively, one can use any linear functions $\theta_\A$
of degree~$1$ extending the linear functions~$\kappa_\A$ and
satisfying the condition $\theta_\A/\m\theta_\A\circ\bar w=0$
in this construction of the functor $\Br_v$.

 To obtain the inverse functor, assign to an $\R$\+free
CDG\+coalgebra $(\D,d_\D,h_\D)$ the $\R$\+free CDG\+algebra
$(\C,d_\C)$ constructed as follows.
 Adjoin to $\D$ a single ``cofree in the conilpotent sense''
cogenerator of degree~$-1$, obtaining an $\R$\+free graded
coalgebra $\C$ endowed with an $\R$\+free graded coalgebra morphism
$\C\rarrow\D$ and an $\R$\+contramodule morphism $\theta\:\C
\rarrow\R$ of degree~$1$ characterized by the following property.
 The map $\C\rarrow\prod_{n=1}^\infty\D^{\ot n}[n-1]$ whose
components are obtained by composing the iterated comultiplication
maps $\C\rarrow\C^{\ot 2n-1}$ with the tensor products of the maps
$\C\rarrow\D$ at the odd positions and the maps~$\theta$ at the even
ones should be equal to the composition of an isomorphism
$\C\simeq\bigoplus_{n=1}^\infty\D^{\ot n}[n-1]$ and the natural
embedding of the infinite direct sum of free graded
$\R$\+contramodules into their infinite product.

 Define the odd coderivation $d'_\C$ on the $\R$\+free graded
coalgebra $\C$ by the conditions that $d'_\C$ must preserve
the kernel of the morphism of $\R$\+free graded coalgebras
$\C\rarrow\D$ and induce the coderivation $d_\D$ on $\D$, and
that the equation $\theta\circ d'_\C=\theta^2+h_\D$ must
hold (where $h_\D$ is considered as a morphism $\C\rarrow\R$).
 Finally, set $d_\C(c)=d'_\C(c)-\theta*c+(-1)^{|c|}c*\theta$
for all $c\in\C$.

 Restricting this procedure to $\R$\+free CDG\+algebras with
conilpotent reductions whose underlying graded coalgebras are
cotensor coalgebras of free graded $\R$\+contra\-modules, we
obtain a functor recovering the $\R$\+free DG\+coalgebra
$\bigoplus_n\A[1]^{\ot n}$ from the $\R$\+free CDG\+coalgebra
$\Br_v(\A)=\bigoplus_n\A_+[1]^{\ot n}$.
\end{proof}

\begin{thm}   \label{strictly-unital-wcainfty-cdg-comodules}
 Let\/ $\A$ be a nonzero strictly unital wc~\Ainfty\+algebra over\/
$\R$ and\/ $\D$ be the corresponding\/ $\R$\+free CDG\+coalgebra.
 Then \par
\textup{(a)} the DG\+category of\/ $\R$\+free left wc~\Ainfty\+modules
over\/ $\A$ is naturally equivalent to the DG\+category of\/
$\R$\+free left CDG\+comodules over\/~$\D$; \par
\textup{(b)} the DG\+category of\/ $\R$\+contramodule left
wc~\Ainfty\+modules over\/ $\A$ is naturally equivalent to
the DG\+category of\/ $\R$\+contramodule left CDG\+contramodules
over\/~$\D$; \par
\textup{(c)} the DG\+category of\/ $\R$\+comodule left
wc~\Ainfty\+modules over\/ $\A$ is naturally equivalent to
the DG\+category of\/ $\R$\+comodule left CDG\+comodules
over\/~$\D$; \par
\textup{(d)} the DG\+category of\/ $\R$\+cofree left
wc~\Ainfty\+modules over\/ $\A$ is naturally equivalent to
the DG\+category of\/ $\R$\+cofree left CDG\+contramodules
over\/~$\D$; \par
\textup{(e)} the above equivalences of DG\+categories form
a commutative diagram with the natural embeddings of
the DG\+categories of strictly unital wc~\Ainfty\+modules into those
of nonunital wc~\Ainfty\+modules over~$\A$, the equivalences of
DG\+categories from Section~\textup{\ref{nonunital-wcainfty}} and
Proposition~\textup{\ref{nonunital-wcainfty-r-c-co-contra}},
and the similar equivalences between DG\+categories of
CDG\+comodules and CDG\+contramodules over\/~$\D$.
\end{thm}

\begin{proof}
 The proof is based on the following two lemmas.

\begin{lem}
 Let\/ $\A$ be a strictly unital wc~\Ainfty\+algebra over\/~$\R$;
set\/ $\C=\bigoplus_n\A[1]^{\ot n}$ and\/
$\D=\bigoplus_n\A_+[1]^{\ot n}$.
 Let\/ $\K_\A\subset\C$ be the kernel of the natural morphism
of\/ $\R$\+free graded coalgebras\/ $\C\rarrow\D$, let
$\theta_\A\:\C\rarrow\R$ be an\/ $\R$\+contramodule morphism
constructed in the proof of
Theorem~\textup{\ref{strictly-unital-wcainfty-cdg-coalgebras}},
and let $d'\:\C\rarrow\C$ be the corresponding odd coderivation from
Lemma~\textup{\ref{d-prime-h-prime-coalgebra}}.
 Then \par
\textup{(a)} a nonunital\/ $\R$\+free left wc~\Ainfty\+module\/ $\M$
over\/ $\A$ with the wc~\Ainfty\+module structure
$d\:\C\ot^\R\M\rarrow\C\ot^\R\M$ is a strictly unital\/ $\R$\+free
wc~\Ainfty\+module over\/ $\A$ if and only if the odd coderivation
$d'(z)=d(z)+\theta_\A*z$ of degree~$1$ on\/ $\C\ot^\R\M$ compatible
with the odd coderivation~$d'$ on\/ $\C$ preserves the graded\/
$\R$\+subcontramodule\/ $\K_\A\ot^\R\M\subset\C\ot^\R\M$; \par
\textup{(b)} a nonunital\/ $\R$\+contramodule left
wc~\Ainfty\+module\/ $\P$ over $\A$ with the wc~\Ainfty\+mod\-ule
structure $d\:\Hom^\R(\C,\P)\rarrow\Hom^\R(\C,\P)$ is a strictly
unital\/ $\R$\+contramodule wc~\Ainfty\+module over\/ $\A$ if and
only if the odd contraderivation $d'(q)=d(q)+\theta_\A*q$ of degree~$1$
on\/ $\Hom^\R(\C,\P)$ compatible with the odd coderivation~$d'$
on\/ $\C$ preserves the graded\/ $\R$\+subcontramodule
$\Hom^\R(\D,\P)\subset\Hom^\R(\C,\P)$; \par
\textup{(c)} a nonunital\/ $\R$\+comodule left wc~\Ainfty\+module\/
$\cM$ over $\A$ with the wc~\Ainfty\+module structure
$d\:\C\ocn_\R\cM\rarrow\C\ocn_\R\cM$ is a strictly unital\/
$\R$\+comodule wc~\Ainfty\+module over\/ $\A$ if and only if the odd
coderivation $d'(z)=d(z)+\theta_\A*z$ of degree~$1$ on\/
$\C\ocn_\R\cM$ compatible with the odd coderivation~$d'$
on\/ $\C$ preserves the graded\/ $\R$\+subcomodule\/
$\K_\A\ocn_\R\cM\subset\C\ocn_\R\cM$; \par
\textup{(d)} a nonunital\/ $\R$\+cofree left wc~\Ainfty\+module\/ $\cP$
over $\A$ with the wc~\Ainfty\+mod\-ule structure
$d\:\Ctrhom_\R(\C,\cP)\rarrow\Ctrhom_\R(\C,\cP)$ is a strictly
unital\/ $\R$\+cofree wc~\Ainfty\+module over\/ $\A$ if and only if
the odd contraderivation $d'(q)=d(q)+\theta_\A*q$ of degree~$1$ on\/
$\Ctrhom_\R(\C,\cP)$ compatible with the odd coderivation~$d'$
on\/ $\C$ preserves the graded\/ $\R$\+subcomodule
$\Ctrhom_\R(\D,\cP)\subset\Ctrhom_\R(\C,\cP)$. 
\end{lem}

\begin{proof}
 Part~(a): the coderivation~$d'$ takes $\K_\A\ot^\R\M$ into
$\K_\A\ot^\R\M$ if and only if its composition
$\K_\A\ot^\R\M\rarrow\C\ot^\R\M\rarrow\C\ot^\R\M\rarrow\M$ with
the natural embedding and the morphism induced by the counit of $\C$
vanishes.
 Part~(b): the contraderivation~$d'$ takes $\Hom^\R(\D,\P)$ into
$\Hom^\R(\D,\P)$ if and only if its composition
$\P\rarrow\Hom^\R(\C,\P)\rarrow\Hom^\R(\C,\P)\rarrow\Hom^\R(\K_\A,\P)$
with the morphism induced by the counit of $\C$ and the natural
surjection vanishes.
 In both cases, it is straightforward to check that the condition is
equivalent to the definition of a strictly unital wc~\Ainfty\+module.
 Parts~(c) and~(d) are similar.
\end{proof}

\begin{lem}
 Let\/ $\A$ be a strictly unital wc~\Ainfty\+algebra over\/~$\R$;
set\/ $\C=\bigoplus_n\A[1]^{\ot n}$ and\/
$\D=\bigoplus_n\A_+[1]^{\ot n}$, and let\/ $\K_\A\subset\C$ be
the kernel of the natural morphism\/ $\C\rarrow\D$.
 Then \par
\textup{(a)} for strictly unital\/ $\R$\+free left
wc~\Ainfty\+modules\/ $\L$ and\/ $\M$ over\/ $\A$, a (not necessarily
closed) morphism of nonunital\/ $\R$\+free wc~\Ainfty\+modules
$f\:\L\rarrow\M$ is a morphism of strictly unital\/ $\R$\+free
wc~\Ainfty\+modules if and only if  the morphism of\/ $\R$\+free
CDG\+comodules $f\:\C\ot^\R\L\rarrow\C\ot^\R\M$ over\/ $\C$ takes\/
$\K_\A\ot^\R\L$ into\/ $\K_\A\ot^\R\M$; \par
\textup{(b)} for strictly unital\/ $\R$\+contramodule left
wc~\Ainfty\+modules\/ $\P$ and\/ $\Q$ over\/ $\A$, a (not necessarily
closed) morphism of nonunital\/ $\R$\+contramodule wc~\Ainfty\+modules 
$f\:\L\rarrow\M$ is a morphism of strictly unital\/
$\R$\+contramodule wc~\Ainfty\+modules if and only if the morphism
of\/ $\R$\+contramodule CDG\+contramodules $f\:\Hom^\R(\C,\P)
\rarrow\Hom^\R(\C,\Q)$ over\/ $\C$ takes\/ $\Hom^\R(\D,\P)$ into
$\Hom^\R(\D,\Q)$; \par
\textup{(c)} for strictly unital\/ $\R$\+comodule left
wc~\Ainfty\+modules\/ $\cL$ and\/ $\cM$ over\/ $\A$, a (not necessarily
closed) morphism of nonunital\/ $\R$\+comodule wc~\Ainfty\+modules
$f\:\cL\rarrow\cM$ is a morphism of strictly unital\/ $\R$\+comodule
wc~\Ainfty\+modules if and only if the morphism of\/ $\R$\+comodule
CDG\+comodules $f:\C\ocn_\R\cL\rarrow\C\ocn_\R\cM$ over\/ $\C$ takes\/
$\K_\A\ocn_\R\cL$ into\/ $\K_\A\ocn_\R\cM$; \par
\textup{(d)} for strictly unital\/ $\R$\+cofree left
wc~\Ainfty\+modules\/ $\cP$ and\/ $\cQ$ over\/ $\A$, a (not necessarily
closed) morphism of nonunital\/ $\R$\+cofree wc~\Ainfty\+modules
$f\:\cP\rarrow\cQ$ is a morphism of strictly unital\/ $\R$\+cofree
wc~\Ainfty\+modules if and only if the morphism of\/ $\R$\+cofree
CDG\+contramodules $f\:\Ctrhom_\R(\C,\cP)\rarrow\Ctrhom_\R(\C,\cQ)$
over\/ $\C$ takes\/ $\Ctrhom_\R(\D,\cP)$ into\/ $\Ctrhom_\R(\D,\cQ)$.
\end{lem}

\begin{proof}
 Similar to the proofs of the three previous lemmas.
\end{proof}

 Let $\A$ be a strictly unital wc~\Ainfty\+algebra, $v\:\A\rarrow\R$
be a homogeneous retraction of the unit map, and
$\theta_\A\:\C\rarrow\R$ be the corresponding
linear function of degree~$1$.
 Part~(a): let $\M$ be a strictly unital $\R$\+free left
wc~\Ainfty\+module over~$\A$.
 Set $d\:\D\ot^\R\M\rarrow\D\ot^\R\M$ to be the map induced by
the coderivation~$d'$ on $\C\ot^\R\M$.
 Then $(\D\ot^\R\M\;d)$ is the $\R$\+free left CDG\+comodule over
the $\R$\+free CDG\+coalgebra $\D=\Br_v(\A)$ corresponding to
the strictly unital wc~\Ainfty\+module~$\M$.

 Let $f\:\L\rarrow\M$ be a (not necessarily closed) morphism of
strictly unital $\R$\+free left wc~\Ainfty\+modules over~$\A$.
 Then the morphism of $\R$\+free DG\+comodules
$f\:\C\ot^\R\L\rarrow\C\ot^\R\M$ induces a morphism of
$\R$\+free CDG\+comodules $f\:\D\ot^\R\L\rarrow\D\ot^\R\M$.
 This provides the desired DG\+functor in one direction; the inverse
DG\+functor can be constructed as the coextension of scalars with
respect to the natural morphism of $\R$\+free CDG\+coalgebras
$(p,\theta_\A)\:(\C,d)\rarrow(\D,d,h)$, where $p\:\C\rarrow\D$
is the natural surjection.
 The proofs of parts~(b\+d) are similar.
 Part~(e) is a particular case of (the underived versions of)
Propositions~\ref{r-free-co-extension}, \ref{r-cofree-co-extension},
and~\ref{non-adj-co-extension}.
\end{proof}

 A morphism of strictly unital wc~\Ainfty\+algebras $f:\A\rarrow\B$
is called \emph{strict} if $f_n=0$ for all $n\ne1$.
 An \emph{augmented} strictly unital wc~\Ainfty\+algebra $\A$ is
a strictly unital wc~\Ainfty\+algebra endowed with a morphism of
strictly unital wc~\Ainfty\+algebras $\A\rarrow\R$, where $\R$ is
endowed with its only strictly unital wc~\Ainfty\+algebra structure
in which $1\in\R$ is the strict unit.
 An augmented strictly unital wc~\Ainfty\+algebra is \emph{strictly
augmented} if the augmentation morphism is strict.
 A morphism of (strictly) augmented strictly unital
wc~\Ainfty\+algebras is a morphism of strictly unital
wc~\Ainfty\+algebras forming a commutative diagram with
the augmentation morphisms~\cite[Section~7.2]{Pkoszul}.

 The categories of augmented strictly unital wc~\Ainfty\+algebras,
strictly augmented strictly unital wc~\Ainfty\+algebras, and
nonunital wc~\Ainfty\+algebras are naturally equivalent.
 The equivalence of the latter two categories is provided by
the functor of formal adjoining of a strict unit, and
the equivalences of the former two ones can be obtained from
the equivalence between the categories of $\R$\+free DG\+coalgebras
$\C$ with coaugmented reductions and $\R$\+free CDG\+coalgebras
endowed with a CDG\+coalgebra morphisms $\C\rarrow\R$ compatible
with the coaugmentations of the reductions.
 The DG\+category of ($\R$\+contramodule or $\R$\+comodule)
strictly unital wc~\Ainfty\+modules over an augmented strictly
unital wc~\Ainfty\+algebra $\A$ is equivalent to the DG\+category
of nonunital wc~\Ainfty\+modules over the corresponding
nonunital wc~\Ainfty\+algebra.

 Now let $\A$ be a nonzero wcDG\+algebra over~$\R$ and
$v\:\A\rarrow\R$ be a homogeneous retraction of the unit map.
 The construction of the strictly unital wc~\Ainfty\+algebra
structure corresponding to the given wcDG\+algebra structure
on $\A$ can be recovered from the bar-construction of
Section~\ref{bar-cobar-sect}.
 Explicitly, set $m_0=h$, \ $m_1(a)=d(a)$, \ $m_2(a_1\ot a_2)
= a_1a_2$, and $m_n=0$ for $n\ge3$; then the CDG\+coalgebra
structure $\Br_v(\A)$ on the $\R$\+free graded tensor coalgebra
$\bigoplus_n\A_+[1]^{\ot n}$ defined in~\ref{bar-cobar-sect}
coincides with the CDG\+coalgebra structure $\Br_v(\A)$
constructed in Theorem~\ref{strictly-unital-wcainfty-cdg-coalgebras}.
 This defines a natural faithful functor from the category of
wcDG\+algebras to the category of wc~\Ainfty\+algebras over~$\R$.

 Similarly, the construction of the natural strictly unital
wc~\Ainfty\+module structure on any $\R$\+contramodule or
$\R$\+comodule wcDG\+module over $\A$ can be recovered from
the construction of the twisting cochain $\tau=\tau_{\A,v}$ and
the functors $\Hom^\tau(\Br_v(\A)\;{-})$ and
$\Br_v(\A)\ocn^\tau{-}$ from Section~\ref{twisting-cochains-sect}.
 Explicitly, given an $\R$\+contramodule left wcDG\+module $\P$
over $\A$, define the structure of $\R$\+contramodule left
wc~\Ainfty\+module over $\A$ on $\P$ by the rules
$p_0(q)=d(q)$, \ $p_1(q)(a) = (-1)^{|a||q|}aq$, and $p_n=0$
for $n\ge2$.
 Given an $\R$\+comodule left wcDG\+module $\cM$ over $\A$,
define the structure of $\R$\+comodule left wc~\Ainfty\+module
over $\A$ on $\cM$ by the rules $l_0(x)=d(x)$, \ 
$l_1(a\ot x)=ax$, and $l_n=0$ for $n\ge2$.
 Similar constructions apply to $\R$\+free and $\R$\+cofree
left wcDG\+modules, and to $\R$\+free and $\R$\+comodule right
wcDG\+modules.
 These constructions define natural faithful DG\+functors from
the DG\+categories of wcDG\+modules over $\A$ to the categories
of wc~\Ainfty\+modules over~$\A$.

\subsection{Semiderived category of wc \Ainfty-modules}
\label{wcainfty-semiderived}
 Let $\A$ be a strictly unital wc~\Ainfty\+algebra over~$\R$.
 A (not necessarily closed) morphism of strictly unital
$\R$\+contramodule left wc~\Ainfty\+modules $f\:\P\rarrow\Q$
is said to be \emph{strict} if one has $f^n=0$ and $(df)^n=0$
for all $n\ge1$.
 A (not necessarily closed) morphism of strictly unital
$\R$\+comodule left wc~\Ainfty\+modules $f\:\cL\rarrow\cM$
is said to be \emph{strict} if one has $f_n=0$ and $(df)_n=0$
for all $n\ge1$ \cite[Section~7.3]{Pkoszul}.

 The strictness property of morphisms of strictly unital $\R$\+free
or $\R$\+cofree wc~\Ainfty\+modules can be equivalently defined by
the similar conditions imposed on the components $f_n$ and $(df)_n$,
or $f^n$ and $(df)^n$, respectively.
 The categories of strictly unital wc~\Ainfty\+modules and strict
morphisms between them are DG\+subcategories of the DG\+categories
of strictly unital wc~\Ainfty\+modules and their (wc~\Ainfty)
morphisms.
 Both DG\+categories of strictly unital wc~\Ainfty\+modules
(from any of the four classes) with wc~\Ainfty\+morphisms or
strict morphisms between them admit shifts, twists, and
infinite direct sums and products.

 A closed strict morphism of strictly unital ($\R$\+contramodule,
$\R$\+free, $\R$\+comodule, or $\R$\+cofree) wc~\Ainfty\+modules
is called a \emph{strict homotopy equivalence} if it is
a homotopy equivalence in the DG\+category of strictly unital
wc~\Ainfty\+modules (of the respective class) and strict morphisms
between them.
 A short sequence $K\rarrow L\rarrow M$ of strictly unital
($\R$\+contramodule, $\R$\+free, $\R$\+comodule, or $\R$\+cofree)
wc~\Ainfty\+modules and closed strict morphisms between them is said
to be \emph{exact} if $K\rarrow L\rarrow M$ is a short exact sequence
of graded $\R$\+contramodules or $\R$\+comodules.
 The total strictly unital wc~\Ainfty\+module of such an exact
triple is constructed in the obvious way. 

 For any strictly unital wc~\Ainfty\+algebra $\A$, the graded
$k$\+vector space $\A/\m\A$ has a natural structure of
strictly unital \Ainfty\+algebra.
 For any $\R$\+contramodule strictly unital left wc~\Ainfty\+module
$\P$ over $\A$, the graded $k$\+vector space $\P/\m\P$ has a natural
structure of strictly unital left \Ainfty\+module over $\A/\m\A$.
 For any $\R$\+comodule strictly unital left wc~\Ainfty\+module $\cM$
over $\A$, the graded $k$\+vector space ${}_\m\cM$ has a natural
structure of strictly unital left \Ainfty\+module over $\A/\m\A$.
 In particular, there are natural differentials $d=m_0/\m m_0$ on
$\A/\m\A$, \ $d=p_0/\m p_0$ on $\P/\m\P$, and $d=l_0/\m l_0$ on
${}_\m\cM$, making $\A/\m\A$, \ $\P/\m\P$ and ${}_\m\cM$ complexes
of $k$\+vector spaces.

 A strictly unital $\R$\+free wc~\Ainfty\+module $\M$ over $\A$
is said to be \emph{semi-acyclic} if the strictly unital
\Ainfty\+module $\M/\m\M$ over $\A/\m\A$ (i.~e., the complex of
vector spaces $\M/\m\M$ with the differential $d = l_0/\m l_0$)
is acyclic.
 A closed morphism of strictly unital $\R$\+free wc~\Ainfty\+modules
$f\:\L\rarrow\M$ over $\A$ is called a \emph{semi-isomorphism} if
the morphism of strictly unital \Ainfty\+modules $f/\m f\:\L/\m\L
\rarrow\M/\m\M$ over $\A/\m\A$ is a quasi-isomorphism (i.~e.,
the morphism of complexes of $k$\+vector spaces
$f_0/\m f_0\:\L/\m\L\rarrow\M/\m\M$ is a quasi-isomorphism).

 Similarly, a strictly unital $\R$\+cofree wc~\Ainfty\+module
$\cP$ over $\A$ is said to be \emph{semi-acyclic} if the strictly
unital \Ainfty\+module ${}_\m\cP$ over $\A/\m\A$ (i.~e., the complex
of vector spaces ${}_\m\cP$ with the differential $d = p_0/\m p_0$)
is acyclic.
 A closed morphism of strictly unital $\R$\+cofree wc~\Ainfty\+modules
$f\:\cP\rarrow\cQ$ over $\A$ is called a \emph{semi-isomorphism} if
the morphism of strictly unital \Ainfty\+modules ${}_\m f\:{}_\m\cP
\rarrow{}_\m\cQ$ over $\A/\m\A$ is a quasi-isomorphism (i.~e.,
the morphism of complexes of $k$\+vector spaces
${}_\m f^0\:{}_\m\cP\rarrow{}_\m\cQ$ is a quasi-isomorphism).

 The following theorems provide a wealth of equivalent definitions
of the \emph{semiderived category of strictly unital
wc~\Ainfty\+modules} over a strictly unital wc~\Ainfty\+algebra
$\A$ over a pro-Artinian topological local ring~$\R$.

\begin{thm}  \label{r-free-wcainfty-semiderived}
\textup{(a)} The following six definitions of the \emph{semiderived
category $\sD^\si(\A\mod\Rfr)$ of strictly unital $\R$\+free left
wc~\Ainfty\+modules} over\/ $\A$ are equivalent, i.~e., lead to
naturally isomorphic (triangulated) categories:
\begin{enumerate}
\renewcommand{\theenumi}{\Roman{enumi}}
\item the homotopy category of the DG\+category of strictly unital\/
      $\R$\+free left wc~\Ainfty\+modules over\/ $\A$ and
      their (wc~\Ainfty) morphisms;
\item the localization of the category of strictly unital\/ $\R$\+free
      left wc~\Ainfty\+modules over\/ $\A$ and their closed morphisms
      by the class of semi-isomorphisms;
\item the localization of the category of strictly unital\/ $\R$\+free
      left wc~\Ainfty\+modules over\/ $\A$ and their closed morphisms
      by the class of strict homotopy equivalences;
\item the quotient category of the homotopy category of
      the DG\+category of strictly unital\/ $\R$\+free left
      wc~\Ainfty\+modules over\/ $\A$ and strict morphisms between
      them by the thick subcategory of semi-acyclic\/ $\R$\+free
      wc~\Ainfty\+modules; 
\item the localization of the category of strictly unital\/ $\R$\+free
      left wc~\Ainfty\+modules over\/ $\A$ and closed strict morphisms
      between them by the class of strict semi-isomorphisms;
\item the quotient category of the homotopy category of
      the DG\+category of strictly unital\/ $\R$\+free left
      wc~\Ainfty\+modules over $\A$ and strict morphisms between them
      by its minimal thick subcategory containing all the total
      strictly unital wc~\Ainfty\+modules of short exact sequences
      of strictly unital\/ $\R$\+free left wc~\Ainfty\+modules
      over\/ $\A$ with closed strict morphisms between them.
\end{enumerate} \par
\textup{(b)} The similar six definitions of the \emph{semiderived
category $\sD^\si(\A\mod\Rcof)$ of strictly unital $\R$\+cofree left
wc~\Ainfty\+modules} over\/ $\A$ (obtained from the above six
definitions by replacing\/ $\R$\+free wc~\Ainfty\+modules with\/
$\R$\+cofree ones) are equivalent.
\end{thm}

\begin{thm}  \label{non-adj-wcainfty-semiderived}
\textup{(a)} The following two definitions of the \emph{semiderived
category $\sD^\si(\A\mod\Rctr)$ of strictly unital $\R$\+contramodule
left wc~\Ainfty\+modules} over\/ $\A$ are equivalent, i.~e., lead to
naturally isomorphic triangulated categories:
\begin{enumerate}
\renewcommand{\theenumi}{\Roman{enumi}}
\item the quotient category of the homotopy category of
the DG\+category of strictly unital\/ $\R$\+contramodule left
wc~\Ainfty\+modules over\/ $\A$ and their (wc~\Ainfty) morphisms
by its minimal triangulated subcategory containing the total
strictly unital wc~\Ainfty\+modules of short exact sequences of
strictly unital\/ $\R$\+contramodule left wc~\Ainfty\+modules over\/
$\A$ with closed strict morphisms between them, and closed with respect
to infinite products;
\item the quotient category of the homotopy category of
the DG\+category of strictly unital\/ $\R$\+contramodule left
wc~\Ainfty\+modules over\/ $\A$ and strict morphisms between them
by its minimal thick subcategory containing all the total
strictly unital wc~\Ainfty\+modules of short exact sequences of
strictly unital\/ $\R$\+contramodule left wc~\Ainfty\+modules over\/
$\A$ with closed strict morphisms between them, and closed with respect
to infinite products.
\end{enumerate} \par
\textup{(b)} The following two definitions of the \emph{semiderived
category $\sD^\si(\A\mod\Rco)$ of strictly unital $\R$\+comodule
left wc~\Ainfty\+modules} over\/ $\A$ are equivalent, i.~e., lead to
naturally isomorphic triangulated categories:
\begin{enumerate}
\renewcommand{\theenumi}{\Roman{enumi}}
\item the quotient category of the homotopy category of
the DG\+category of strictly unital\/ $\R$\+comodule left
wc~\Ainfty\+modules over\/ $\A$ and their (wc~\Ainfty) morphisms
by its minimal triangulated subcategory containing the total
strictly unital wc~\Ainfty\+modules of short exact sequences of
strictly unital\/ $\R$\+comodule left wc~\Ainfty\+modules over\/ $\A$
with closed strict morphisms between them, and closed with respect
to infinite direct sums;
\item the quotient category of the homotopy category of
the DG\+category of strictly unital\/ $\R$\+comodule left
wc~\Ainfty\+modules over\/ $\A$ and strict morphisms between them
by its minimal thick subcategory containing all the total
strictly unital wc~\Ainfty\+modules of short exact sequences of
strictly unital\/ $\R$\+comodule left wc~\Ainfty\+modules over\/ $\A$
with closed strict morphisms between them, and closed with respect
to infinite direct sums.
\end{enumerate}
\end{thm}

 Let $\A$ be a nonzero strictly unital wc~\Ainfty\+algebra over
$\R$, let $v\:\A\rarrow\R$ be a homogeneous retraction of the unit map,
and let $\C=\Br_v(\A)$ be the corresponding CDG\+coalgebra
structure on the $\R$\+free graded tensor coalgebra
$\bigoplus_n\A_+[1]^{\ot n}$.
 Denote by $w\:\R\rarrow\C$ the natural section
$\R\simeq\A[1]^{\ot 0}\rarrow\bigoplus_n\A[1]^{\ot n}$ of the
counit map, and set $\U=\Cb_w(\C)$.
 The wcDG\+algebra $\U$ is called the \emph{enveloping wcDG\+algebra}
of a strictly unital wc~\Ainfty\+algebra~$\A$.
 The adjunction map $\C\rarrow\Br_v(\Cb_w(\C))$ provides a natural
morphism of strictly unital wc~\Ainfty\+algebras $\A\rarrow\U$.

\begin{thm}   \label{wcainfty-semi-co-derived}
 The following triangulated categories are naturally equivalent:
\par
\textup{(a)} the semiderived category\/ $\sD^\si(\A\mod\Rfr)$ of
strictly unital\/ $\R$\+free left wc~\Ainfty\+modules over\/~$\A$;
{\hbadness=2825\par}
\textup{(b)} the semiderived category\/ $\sD^\si(\A\mod\Rctr)$ of
strictly unital\/ $\R$\+contramodule left wc~\Ainfty\+modules
over\/~$\A$; \par
\textup{(c)} the semiderived category\/ $\sD^\si(\A\mod\Rco)$ of
strictly unital\/ $\R$\+comodule left wc~\Ainfty\+modules over\/~$\A$;
\par
\textup{(d)} the semiderived category\/ $\sD^\si(\A\mod\Rcof)$ of
strictly unital\/ $\R$\+cofree left wc~\Ainfty\+modules over\/~$\A$;
\par
\textup{(e)} the semiderived category\/ $\sD^\si(\U\mod\Rfr)$ of
strictly unital\/ $\R$\+free left wcDG\+mod\-ules over\/~$\U$;
{\hfuzz=2.4pt\par}
\textup{(f)} the semiderived category\/ $\sD^\si(\U\mod\Rctr)$ of
strictly unital\/ $\R$\+contramodule left wcDG\+modules over\/~$\U$;
\par
\textup{(g)} the semiderived category\/ $\sD^\si(\U\mod\Rco)$ of
strictly unital\/ $\R$\+comodule left wcDG\+modules over\/~$\U$; \par
\textup{(h)} the semiderived category\/ $\sD^\si(\U\mod\Rcof)$ of
strictly unital\/ $\R$\+cofree left wcDG\+modules over\/~$\U$; \par
\textup{(i)} the contraderived category\/ $\sD^\si(\U\mod\Rctr)$ of
strictly unital\/ $\R$\+contramodule left wcDG\+modules over\/~$\U$;
\par
\textup{(j)} the coderived category\/ $\sD^\si(\U\mod\Rco)$ of
strictly unital\/ $\R$\+comodule left wcDG\+modules over\/~$\U$; \par
\textup{(k)} the absolute derived category\/ $\sD^\abs(\U\mod\Rfr)$
of strictly unital\/ $\R$\+free left wcDG\+modules over\/~$\U$; \par
\textup{(l)} the absolute derived category\/ $\sD^\abs(\U\mod\Rcof)$
of strictly unital\/ $\R$\+cofree left wcDG\+modules over\/~$\U$; \par
\textup{(m)} the contraderived category\/ $\sD^\ctr(\C\contra\Rfr)$
of\/ $\R$\+free left CDG\+contramodules over\/~$\C$;
{\hfuzz=1.7pt\par}
\textup{(n)} the contraderived category\/ $\sD^\ctr(\C\contra\Rctr)$
of\/ $\R$\+contramodule left CDG\+con\-tramodules over\/~$\C$;
{\emergencystretch=0em\hfuzz=2.5pt\par}
\textup{(o)} the contraderived category\/ $\sD^\ctr(\C\contra\Rcof)$
of\/ $\R$\+cofree left CDG\+contra\-modules over\/~$\C$; \par
\textup{(p)} the coderived category\/ $\sD^\co(\C\comod\Rfr)$
of\/ $\R$\+free left CDG\+comodules over\/~$\C$; \par
\textup{(q)} the coderived category\/ $\sD^\co(\C\comod\Rco)$
of\/ $\R$\+comodule left CDG\+comodules over\/~$\C$; \par
\textup{(r)} the coderived category\/ $\sD^\co(\C\comod\Rcof)$
of\/ $\R$\+cofree left CDG\+comodules over\/~$\C$; \par
\textup{(s)} the absolute derived category\/ $\sD^\abs(\C\contra\Rfr)$
of\/ $\R$\+free left CDG\+contra\-modules over\/~$\C$; \par
\textup{(t)} the absolute derived category\/ $\sD^\abs(\C\contra\Rcof)$
of\/ $\R$\+cofree left CDG\+contra\-modules over\/~$\C$; \par
\textup{(u)} the absolute derived category\/ $\sD^\abs(\C\comod\Rfr)$
of\/ $\R$\+free left CDG\+comodules over\/~$\C$; \par
\textup{(v)} the absolute derived category\/ $\sD^\abs(\C\comod\Rcof)$
of\/ $\R$\+cofree left CDG\+comod\-ules over\/~$\C$.
\end{thm}

\begin{proof}[Proof of Theorems~\ref{r-free-wcainfty-semiderived}\+-%
\ref{wcainfty-semi-co-derived}]
 The equivalence of all the constructions in the three theorems with
the exception of (s\+-v) of Theorem~\ref{wcainfty-semi-co-derived}
holds in the much greater generality of CDG\+comod\-ules and
CDG\+contramodules over an arbitrary $\R$\+free CDG\+coalgebra $\C$
with conilpotent reduction $\C/\m\C$.

 Specifically, the DG\+categories of strictly unital $\R$\+free,
$\R$\+contramodule, $\R$\+comod\-ule, and $\R$\+cofree left
wc~\Ainfty\+modules over $\A$ with their wc~\Ainfty\+morphisms are
described as the DG\+categories of left CDG\+contra/comodules over $\C$
with the underlying graded $\C$\+contra/comodules, respectively,
cofreely cogenerated by free graded $\R$\+contramodules, induced from
graded $\R$\+contramodules, coinduced from graded $\R$\+comodules, and
freely generated by cofree graded $\R$\+comodules.
{\emergencystretch=1em\par}

 Furthermore, let $\C$ be an $\R$\+free CDG\+coalgebra.
 A morphism of $\R$\+free left CDG\+comodules over $\C$ with
the underlying graded $\C$\+comodules cofreely cogenerated by free
graded $\R$\+contramodules $f\:\C\ot^\R\L\rarrow\C\ot^\R\M$ can be
called \emph{strict} if both $f$ and $df$, viewed as morphisms of
graded $\C$\+comodules, can be obtained from morphisms of free
graded $\R$\+contramodules by applying the functor $\C\ot^\R{-}$.
 Using Lemma~\ref{free-co-module-derivations}(a), one can show that
the DG\+category of such $\R$\+free CDG\+comodules $\C\ot^\R\M$ and
strict morphisms between them is isomorphic to the DG\+category
of $\R$\+free left CDG\+modules $\M$ over the $\R$\+free CDG\+algebra
$\U=\Cb_w(\C)$; the equivalence is provided by the DG\+functor
$\M\mpsto\C\ot^{\tau_{\C,w}}\M$.

 Similarly, the DG\+categories of $\R$\+contramodule left
CDG\+contramodules over $\C$ with the underlying graded
$\C$\+contramodules $\Hom^\R(\C,\P)$ induced from graded
$\R$\+contramodules~$\P$, \ $\R$\+comodule left CDG\+comodules over
$\C$ with the underlying graded $\C$\+comodules $\C\ocn_\R\cM$
coinduced from graded $\R$\+comodules~$\cM$, \ $\R$\+cofree left
CDG\+contramodules over $\C$ with the underlying graded
$\C$\+contramodules $\Ctrhom_\R(\C,\cP)$ freely generated by cofree
graded $\R$\+comodules~$\cP$, and strict morphisms between them are
described as the DG\+categories of $\R$\+contramodule, $\R$\+comodule,
and $\R$\+cofree left CDG\+modules over $\U$, respectively.
{\hbadness=1600\par}

 When the CDG\+coalgebra $\C/\m\C$ is coaugmented (e.~g.,
conilpotent), we pick a homogeneous section $w\:\R\rarrow\C$ of
the counit map $\C\rarrow\R$ so that $w/\m w=\bar w$ be
the coaugmentation map; then $\U$ is a wcDG\+algebra over~$\R$.

 So the definition~(IV) of Theorem~\ref{r-free-wcainfty-semiderived}(a)
is that of the semiderived category $\sD^\si(\U\mod\Rfr)$, and
the definition~(VI) is that of the absolute derived category
$\sD^\abs(\U\mod\Rfr)$.
 As mentioned above, the definition~(I) is that of the homotopy
category of CDG\+comodules over $\C$ with the underlying graded
comodules cofreely cogenerated by free graded $\R$\+contramodules.

 Now the equivalence of (I) and~(IV) follows essentially from
Corollary~\ref{conilpotent-cobar}, and the equivalence
of (I) and~(VI) is Corollary~\ref{nonconilpotent-cobar}
(for the latter argument we do not even need the CDG\+coalgebra
$\C/\m\C$ to be conilpotent).
 Alternatively, the equivalence of (IV) and~(VI) is a particular
case of Theorem~\ref{r-free-cofibrant} (cf.\ the remarks after
the proof of Corollary~\ref{nonconilpotent-cobar}).

 In the rest of the proof of
Theorem~~\ref{r-free-wcainfty-semiderived}, we will need to use
the following result
(\cite[Proposition~III.4.2 and Lemma~III.4.3]{GM}).

\begin{lem}  \label{gm-lemma}
 Let\/ $\DG$ be a DG\+category with shifts and cones.
 Then the category obtained from the category of closed morphisms
$Z^0(\DG)$ by inverting formally all morphisms that are homotopy
equivalences in\/ $\DG$ (i.~e., represent isomorphisms in $H^0(\DG)$)
is naturally isomorphic to $H^0(\DG)$.
 More precisely, it suffices to invert in $Z^0(\DG)$ all the morphisms
of the form\/ $s_X=(\id_X,0)\:X\oplus\cone(\id_X)\rarrow X$, where
$X$ are objects of\/ $\DG$, in order to obtain the homotopy
category $H^0(\DG)$.
\end{lem}

\begin{proof}
 Clearly, the morphisms~$s_X$ are isomorphisms in $H^0(\DG)$, so one
only has to check that inverting such morphisms leads to any two
morphisms in $Z^0(\DG)$ that are equal to each other in $H^0(\DG)$
becoming equal in the localized category.
 A closed morphism $X\rarrow Y$ in $\DG$ is homotopic to zero if and
only if it can be factorized as the composition of the canonical
morphism $c_X\:X\rarrow\cone(\id_X)$ and a closed morphism 
$\cone(\id_X)\rarrow Y$.
 Let $i_X$, $j_X\:X\rarrow X\oplus\cone(\id_X)$ be the two closed
morphisms defined by the rules $i_X=(\id_X,0)$ and $j_X=(\id_X,c_X)$.
 Then two closed morphisms $f$, $g\:X\rarrow Y$ are homotopic to
each other if and only if there exists a closed morphism
$q\:X\oplus\cone(\id_X)\rarrow Y$ such that $f=q\circ i_X$ and
$g=q\circ j_X$.
 Now both compositions $s_X\circ i_X$ and $s_X\circ j_X$ are equal
to $\id_X$, so inverting the morphism~$s_X$ makes $i_X$ equal
to~$j_X$ and $f$ equal to~$g$.
\end{proof}

 The classes of semi-isomorphisms and homotopy equivalences in
the DG\+category of strictly unital $\R$\+(co)free
wc~\Ainfty\+modules and wc~\Ainfty\+morphisms coincide.
 The appropriate generalization of this holds for any $\R$\+free
CDG\+coalgebra~$\C$ with conilpotent reduction and follows from
the similar assertion for conilpotent CDG\+coalgebras over~$k$
\cite[proof of Theorem~7.3.1]{Pkoszul} and
Lemma~\ref{contractible-reduction-co}.

 So the equivalence of the definitions (I) and~(II) in
Theorem~\ref{r-free-wcainfty-semiderived}(a) follows from
the first assertion of Lemma~\ref{gm-lemma}, and
the equivalence of (I) and~(III) is clear from the second one.
 The equivalence of (IV) and~(V) also follows from the first
assertion of the lemma, as a strict morphism is a strict
semi-isomorphism if and only if its cone is semi-acyclic.
 The latter two deductions do not even require $\C/\m\C$ to
be conilpotent.
 The proof of Theorem~\ref{r-free-wcainfty-semiderived}(b)
is similar.

 The proof of Theorem~\ref{non-adj-wcainfty-semiderived} is
based on the next lemma.

\begin{lem}  \label{co-induced-co-derived}
 Let\/ $\C$ be an\/ $\R$\+free CDG\+algebra.  Then \par
\textup{(a)} the quotient category of the homotopy category of\/
$\R$\+contramodule left CDG\+contramodules over\/ $\C$ with
the underlying graded\/ $\C$\+contramodules induced from graded\/
$\R$\+contramodules by its minimal triangulated subcategory which
contains the total CDG\+contramodules of the short exact sequences
of CDG\+contramodules over $\C$ whose underlying short exact sequences
of graded\/ $\C$\+contramodules are induced from short exact sequences
of graded\/ $\R$\+contramodules, and is closed under infinite products,
is equivalent to the contraderived category\/
$\sD^\ctr(\C\contra\Rctr)$ of\/ $\R$\+contramodule left
CDG\+contramodules over\/~$\C$; {\hbadness=1150\par}
\textup{(b)} the quotient category of the homotopy category of\/
$\R$\+comodule left CDG\+comod\-ules over\/ $\C$ with the underlying
graded\/ $\C$\+comodules coinduced from graded\/ $\R$\+comod\-ules by
its minimal triangulated subcategory which contains the total
CDG\+comodules of the short exact sequences of CDG\+comodules over $\C$
whose underlying short exact sequences of graded\/ $\C$\+comodules are
induced from short exact sequences of graded\/ $\R$\+comodules, and is
closed under infinite direct sums, is equivalent to the coderived
category\/ $\sD^\co(\C\comod\Rco)$ of\/ $\R$\+comodule left
CDG\+comodules over\/~$\C$.
\end{lem}

\begin{proof}
 This is a much simpler version of~\cite[Theorem~5.5]{Psemi}.
 Part~(a): let us show that the natural functor from the homotopy
category of CDG\+contramodules over $\C$ whose underlying graded
$\C$\+contramodules are freely generated by free graded
$\R$\+contramodules to the quotient category of
the homotopy category of CDG\+contra\-modules with induced underlying
graded $\C$\+contramodules that we are interested in is
an equivalence of categories.
 Then it will remain to take into account the version
Theorem~\ref{non-adj-co-derived-res}(c) with projective graded
$\C$\+contramodules replaced by free ones.
 The semiorthogonality being provided by part~(a) of the same
Theorem, we only have to show that
the natural functor from the homotopy category of CDG\+contramodules
with free underlying graded $\C$\+contramodules to our quotient
category of the homotopy category of the homotopy category of
CDG\+contramodules with induced underlying graded $\C$\+contramodules
is essentially surjective.

 Applying the construction from the beginning of the proof of
Theorem~\ref{co-derived-mod} to $\R$\+contramodule left CDG\+modules
over the $\R$\+free CDG\+algebra $\U=\Cb_w(\C)$, we conclude that 
any $\R$\+contramodule left CDG\+contramodule over $\C$ with 
an induced underlying graded $\C$\+contramodule admits a left 
resolution by $\R$\+free CDG\+contra\-modules with free underlying
graded $\C$\+contramodules and closed morphisms between these
such that the underlying complex of graded $\C$\+contramodules
is induced from a complex of graded $\R$\+contramodules.
 Now it remains to use (the appropriate version of)
Lemma~\ref{bounded-above-lem}.
 The proof of part~(b) is similar.
\end{proof}

 The definition~(II) of Theorem~\ref{non-adj-wcainfty-semiderived}(a)
is that of the contraderived category $\sD^\ctr(\U\mod\Rctr)$,
while according to Lemma~\ref{co-induced-co-derived}(a)
the definition~(I) is that of the contraderived category
$\sD^\ctr(\C\contra\Rctr)$.
 The two categories are equivalent by
Corollary~\ref{nonconilpotent-cobar}.
 This argument does not even depend on the assumption about
$\C/\m\C$ being conilpotent.
 The proof of Theorem~\ref{non-adj-wcainfty-semiderived}(b)
is similar.

\medskip
 For the purposes of the proof of
Theorem~\ref{wcainfty-semi-co-derived}, we will use
the definitions~(I) in Theorems~\ref{r-free-wcainfty-semiderived}
and~\ref{non-adj-wcainfty-semiderived} as setting
the meaning of items~(a\+d).
 Items (e) and~(h) are the constructions~(IV) of
Theorem~\ref{r-free-wcainfty-semiderived}, items (k) and~(l)
are the constructions~(VI) of the same Theorem, and items (i) and~(j)
of Theorem~\ref{wcainfty-semi-co-derived} are the constructions~(II)
of Theorem~\ref{non-adj-wcainfty-semiderived}.

 So the equivalence of~(a), (e), and~(k) is provided by
Theorem~\ref{r-free-wcainfty-semiderived}(a), the equivalence
of~(d), (h), and~(l) is provided by
Theorem~\ref{r-free-wcainfty-semiderived}(b).
 The equivalence of (b) and~(i) is
Theorem~\ref{non-adj-wcainfty-semiderived}(a),
and the equivalence of (c) and~(j) is
Theorem~\ref{non-adj-wcainfty-semiderived}(b).

 The equivalence of (a) and~(p) is essentially
Theorem~\ref{r-free-co-derived-thm}(d),
and the equivalence of (d) and~(o) is
Theorem~\ref{r-cofree-co-derived-thm}(c).
 The equivalence of (b) and~(n) is
Lemma~\ref{co-induced-co-derived}(a), and
the equivalence of (c) and~(q) is
Lemma~\ref{co-induced-co-derived}(b).

 The equivalence of~(e\+-h) is the results of
Section~\ref{r-cofree-semi}, Proposition~\ref{non-adj-r-co-contra},
and the basic results of Section~\ref{wcdg-semiderived}.
 The equivalence of (e) and~(k) is Theorem~\ref{r-free-cofibrant},
and the equivalence of (h) and~(l) is similar.
 The equivalence of (f) and~(i) and the equivalence of
(g) and~(j) are Corollary~\ref{non-adj-cofibrant}.
 Alternatively, the equivalence of (i) and~(k) and
the equivalence of (j) and~(l) are
Corollary~\ref{non-adj-fin-dim-reduct-co-abs}.
 The equivalence of (i) and~(j) is
Corollary~\ref{non-adj-fin-dim-reduct-r-co-contra}, and
the equivalence of (k) and~(l) was established in
Section~\ref{r-cofree-absolute}.

 The equivalence of the respective items in~(e), (f), (g), (h)
and~(p), (n), (q), (o) is Corollary~\ref{conilpotent-cobar}.
 The equivalence of the respective items in~(i), (j), (k), (l)
and~(n), (q), (p), (o) is Corollary~\ref{nonconilpotent-cobar}.

 The equivalence of (m\+-r) is
Corollaries~\ref{r-free-derived-co-contra}
and~\ref{r-cofree-derived-co-contra}, the results of
Sections~\ref{r-free-co-derived} and~\ref{r-cofree-co-derived},
and Corollary~\ref{non-adj-derived-co-contra}.
 The equivalence of (m) and~(s) is
Theorem~\ref{r-free-finite-dim-co-derived} together with
Corollary~\ref{r-free-co-homol-dim}, and so is the equivalence
of (p) and~(u).
 The equivalence of (o) and~(t) and the equivalence of (r) and~(v)
are similar.
\end{proof}

 Let $\A$ be a wcDG\+algebra over $\R$; it can be considered also
as a strictly unital wc~\Ainfty\+algebra, and wcDG\+modules over it
can be viewed also as strictly unital wc~\Ainfty\+modules
(see Section~\ref{strictly-unital-wcainfty}).
 It follows from Corollary~\ref{bar-duality} that the semiderived
category of ($\R$\+free, $\R$\+contramodule, $\R$\+comodule, or
$\R$\+cofree) left wcDG\+modules over~$\A$ is equivalent to
the semiderived category of strictly unital left wc~\Ainfty\+modules
over $\A$ (from the same class), so our notation is consistent.

 When the pro-Artinian topological local ring $\R$ has finite
homological dimension, the semiderived categories of strictly
unital ($\R$\+free, $\R$\+contramodule, $\R$\+comodule, or
$\R$\+cofree) wc~\Ainfty\+modules over $\A$ can be called simply
their \emph{derived categories} (see
Sections~\ref{r-free-semi}, \ref{r-cofree-semi},
and~\ref{wcdg-semiderived} for the discussion).

\begin{cor}
 For any strictly unital wc~\Ainfty\+algebra\/ $\A$ over\/
a pro-Artinian topological local ring\/ $\R$, the semiderived
category\/ $\sD^\si(\A\mod\Rfr)\simeq\sD^\si(\A\mod\Rctr)\simeq
\sD^\si(\A\mod\Rco)\simeq\sD^\si(\A\mod\Rcof)$ has a single
compact generator.
\end{cor}

\begin{proof}
 Follows from Theorem~\ref{semiderived-compact-generator-thm}
applied to the enveloping wcDG\+algebra~$\U$
(see also Theorem~\ref{coderived-compact-generators-thm}).
\end{proof}

 Let $f\:\A\rarrow\B$ be a morphism of strictly unital
wc~\Ainfty\+algebras over~$\R$.
 Then the constructions of the restriction of scalars
from Section~\ref{nonunital-wcainfty} for nonunital ($\R$\+free,
$\R$\+contramodule, $\R$\+comodule, or $\R$\+cofree)
wc~\Ainfty\+modules with respect to the morphism~$f$
take strictly unital wc~\Ainfty\+modules over $\B$ to
strictly unital wc~\Ainfty\+modules over~$\A$.
 The induced restriction-of-scalars functors on the DG\+categories
of strictly unital wc~\Ainfty\+modules can be also obtained as
the contra/coextension-of-scalars functors related to
the morphism of $\R$\+free CDG\+coalgebras $g\:\C=\Br_v(\A)\rarrow
\Br_v(\B)=\D$ assigned to the morphism $f\:\A\rarrow\B$ by
the construction of
Theorem~\ref{strictly-unital-wcainfty-cdg-coalgebras}.

 Furthermore, the restrictions of these DG\+functors to
the DG\+categories of strictly unital wc~\Ainfty\+modules and
strict morphisms between them can be identified with
the restriction-of-scalars functors with respect to
the morphism of the enveloping wcDG\+algebras
$F\:\U=\Cb_w(\Br_v(\A))\rarrow\Cb_w(\Br_v(\B))=\V$
induced by~$g$.

 The functors $E_g\:\D\comod\Rfr_\inj\rarrow\C\comod\Rfr_\inj$
and $R_F\:\V\mod\Rfr\rarrow\U\mod\Rfr$ induce the same triangulated
functor on the semiderived categories of strictly unital $\R$\+free
left wc~\Ainfty\+modules, which we denote by {\hbadness=1300
$$
 \boI R_f\:\sD^\si(\B\mod\Rfr)\lrarrow\sD^\si(\A\mod\Rfr),
$$
and} similarly for strictly unital $\R$\+free right wc~\Ainfty\+modules.
 The functors $E^g\:\D\contra\Rctr\rarrow\C\contra\Rctr$ and
$R_F\:\V\mod\Rctr\rarrow\U\mod\Rctr$ induce the same triangulated
functor on the semiderived categories of strictly unital
$\R$\+contramodule left wc~\Ainfty\+modules, which we denote by
$$
 \boI R_f\:\sD^\si(\B\mod\Rctr)\lrarrow\sD^\si(\A\mod\Rctr).
$$
 The functors $E_g\:\D\comod\Rco\rarrow\C\comod\Rco$ and
$R_F\:\V\mod\Rco\rarrow\U\mod\Rco$ induce the same triangulated functor
on the semiderived categories of strictly unital $\R$\+comodule
left wc~\Ainfty\+modules, which we denote by {\hbadness=1800
$$
 \boI R_f\:\sD^\si(\B\mod\Rco)\lrarrow\sD^\si(\A\mod\Rco),
$$
and} similarly for strictly unital $\R$\+comodule right
wc~\Ainfty\+modules.
 The functors $E^g\:\D\contra\Rcof_\proj\rarrow\C\contra\Rcof_\proj$
and $R_F\:\V\mod\Rcof\rarrow\U\mod\Rcof$ induce the same triangulated
functor on the semiderived categories of strictly unital $\R$\+cofree
left wc~\Ainfty\+modules, which we denote by
$$
 \boI R_f\:\sD^\si(\B\mod\Rcof)\lrarrow\sD^\si(\A\mod\Rcof).
$$
 The functors $\boI R_f$ are identified by the equivalences of
categories from Theorem~\ref{wcainfty-semi-co-derived}.
 When $\A$ is a wcDG\+algebra over $\R$, the above functors
$\boI R_f$ are the same restriction-of-scalars functors that were
constructed in Sections~\ref{r-free-semi}, \ref{r-cofree-semi},
and~\ref{wcdg-semiderived}.

 The triangulated functors $\boI R_f$ have the left and right adjoint
functors $\boL E_f$ and $\boR E^f$ that can be constructed either
as the functors of co- and contrarestriction of scalars $\boI R_g$ and 
$\boI R^g$ on the level of the co- and contraderived categories of
CDG\+comodules and CDG\+contramodules over $\C$ and $\D$
(see Sections~\ref{r-free-co-derived}, \ref{r-cofree-co-derived},
and~\ref{non-adj-co-derived}), or as the functors of extension and
coextension of scalars $\boL E_F$ and $\boR E^F$ on the level of
the semiderived categories of wcDG\+modules over $\U$ and~$\V$
(cf.\ Proposition~\ref{koszul-restriction-extension}).

 A morphism of (strictly unital) wc~\Ainfty\+algebras $f\:\A\rarrow\B$
is called a \emph{semi-isomorphism} if the morphism of (strictly
unital) \Ainfty\+algebras $f/\m f\:\A/\m\A\rarrow\B/\m\B$ over
the field~$k$ is a quasi-isomorphism (i.~e., the morphism of complexes
of $k$\+vector spaces $f_1/\m f_1$ is a quasi-isomorphism).

\begin{thm}
 The triangulated functors\/ $\boI R_f$ are equivalences of categories
whenever a morphism of strictly unital wc~\Ainfty\+algebras
$f\:\A\rarrow\B$ is a semi-isomorphism.
\end{thm}

\begin{proof}
 Can be deduced easily either from
Theorem~\ref{r-free-co-restriction}(b), or from
Corollary~\ref{non-adj-quasi}
(see~\cite[proof of Theorem~7.3.2]{Pkoszul} for further details).
\end{proof}

 Let $\A$ be a strictly unital wc~\Ainfty\+algebra over $\R$, let
$\C=\Br_v(\A)$ be the corresponding $\R$\+free CDG\+coalgebra, and
let $\U=\Cb_w(\C)$ be the enveloping wcDG\+algebra.
 All the above results about strictly unital left
wc~\Ainfty\+modules over $\A$ apply to strictly unital right
wc~\Ainfty\+modules as well, e.~g., because one can pass to
the opposite $\R$\+free (C)DG\+coalgebras and wcDG\+algebras
(see~\cite[Section~4.7]{Pkoszul} and~\cite[Section~1.2]{PP2}
for the sign rules) and such a passage is compatible with
the functors $\Br_v$ and $\Cb_w$.

 In particular, the semiderived categories of strictly unital
$\R$\+free and $\R$\+comodule right wc~\Ainfty\+modules
$\sD^\si(\modrRfr\A)$ and $\sD^\si(\modrRco\A)$ over $\A$ are
defined and naturally equivalent to the coderived categories
$\sD^\co(\comodrRfr\C)$ and $\sD^\co(\modrRco\C)$ and to
the semiderived categories $\sD^\si(\modrRfr\U)$ and
$\sD^\si(\modrRco\U)$, respectively.

 The functor
$$
 \Tor^\A\:\sD^\si(\modrRfr\A)\times\sD^\si(\A\mod\Rfr)
 \lrarrow H^0(\R\contra^\free)
$$
can be defined either as the functor $\Cotor^\C\:\sD^\co(\comodrRfr\C)
\times\sD^\co(\C\comod\Rfr)\rarrow H^0(\R\contra^\free)$, or as
the functor $\Ctrtor^\C\:\sD^\co(\comodrRfr\C)\times
\sD^\ctr(\C\contra\Rfr)\allowbreak\rarrow H^0(\R\contra^\free)$;
the two points of view are equivalent by
Proposition~\ref{r-free-cotor-ctrtor}(b)
(see~\cite[Section~7.3]{Pkoszul} for further details).
 Alternatively, the functor $\Tor^\A$ can be defined as the functor
$\Tor^\U\:\sD^\si(\modrRfr\U)\times\sD^\si(\U\mod\Rfr)\rarrow
H^0(\R\contra^\free)$.
 The two definitions are equivalent and agree with the definition of
the functor $\Tor^\A$ for a wcDG\+algebra $\A$
by Theorem~\ref{ctrtor-and-tor}(a) and/or~\ref{cotor-and-tor}(a).

 The functor
$$
 \Tor^\A\:\sD^\si(\modrRfr\A)\times\sD^\si(\A\mod\Rco)
 \lrarrow H^0(\R\comod^\cofr)
$$
can be defined either as the functor $\Cotor^\C\:\sD^\co(\comodrRfr\C)
\times\sD^\co(\C\comod\Rco)\rarrow H^0(\R\comod^\cofr)$, or as
the functor $\Ctrtor^\C\:\sD^\co(\C\comod\Rfr)\times
\sD^\co(\C\contra\Rcof)\allowbreak\rarrow H^0(\R\comod^\cofr)$;
the two points of view are equivalent by
Proposition~\ref{r-cofree-cotor-ctrtor}(c).
 Alternatively, the functor $\Tor^\A$ can be defined as the functor
$\Tor^\U\:\sD^\si(\modrRctr\U)\times\sD^\si(\U\mod\Rco)\rarrow
H^0(\R\comod^\cofr)$.
 The two definitions are equivalent and agree with the definition of
the functor $\Tor^\A$ for a wcDG\+algebra $\A$
by Theorem~\ref{ctrtor-and-tor}(b) and/or~\ref{cotor-and-tor}(b).
{\hfuzz=4.1pt\par}

 The functor
$$
 \Tor^\A\:\sD^\si(\modrRco\A)\times\sD^\si(\A\mod\Rctr)
 \lrarrow H^0(\R\comod^\cofr)
$$
can be defined as the functor $\Ctrtor^\C\:\sD^\co(\comodrRco\C)
\times\sD^\ctr(\C\contra\Rctr)\rarrow H^0(\R\comod^\cofr)$, or
alternatively as the functor $\Tor^\U\:\sD^\si(\modrRco\U)\times
\sD^\si(\U\mod\Rctr)\allowbreak\rarrow H^0(\R\comod^\cofr)$.
 The two definitions are equivalent and agree with the definition
of the functor $\Tor^\A$ for a wcDG\+algebra $\A$
by Theorem~\ref{ctrtor-and-tor}(c).
{\hfuzz=14.7pt\par}

 The functor
$$
 \Tor^\A\:\sD^\si(\modrRco\A)\times\sD^\si(\A\mod\Rco)
 \lrarrow H^0(\R\comod^\cofr)
$$
can be defined either as the functor $\Cotor^\C\:\sD^\co(\comodrRco\C)
\times\sD^\co(\C\comod\Rco)\allowbreak\rarrow H^0(\R\comod^\cofr)$, or 
alternatively as the functor $\Tor^\U\:\sD^\si(\modrRco\U)\times
\sD^\si(\U\mod\Rco)\allowbreak\rarrow H^0(\R\comod^\cofr)$
 The two definitions are equivalent and agree with the definition
of the functor $\Tor^\A$ for a wcDG\+algebra $\A$
by Theorem~\ref{cotor-and-tor}(c).

 The above four functors $\Tor^\A$ are transformed into each other
by the equivalences of categories from
Theorem~\ref{wcainfty-semi-co-derived}(a\+d) and the equivalence
of categories $\R\contra^\free\simeq\R\comod^\cofr$ from
Proposition~\ref{r-co-contra}.

 The functor
$$
 \Ext_\A\:\sD^\si(\A\mod\Rctr)^\sop\times\sD^\si(\A\mod\Rctr)
 \lrarrow H^0(\R\contra^\free)
$$
can be defined either as the functor
$\Coext_\C\:\sD^\co(\C\comod\Rfr)^\sop\times\sD^\ctr(\C\contra\Rctr)
\allowbreak\rarrow H^0(\R\contra^\free)$, or as the functor
$\Ext^\C\:\sD^\ctr(\C\contra\Rctr)^\sop\times\sD^\ctr(\C\contra\Rctr)
\allowbreak\rarrow H^0(\R\contra^\free)$, or as the functor
$\Ext_\C\:\sD^\co(\C\comod\Rfr)^\sop\times\sD^\co(\C\comod\Rfr)
\allowbreak\rarrow H^0(\R\contra^\free)$; the three points of view are
equivalent by Proposition~\ref{r-free-cotor-ctrtor}(a)
(see~\cite[Section~7.3]{Pkoszul} for further details).
 Alternatively, the functor $\Ext_\A$ can be defined as the functor
$\Ext_\U\:\sD^\si(\U\mod\Rctr)^\sop\times\sD^\si(\U\mod\Rctr)
\rarrow H^0(\R\contra^\free)$.
 The two definitions are equivalent and agree with the definition
of the functor $\Ext_\A$ for a wcDG\+algebra $\A$ by
Theorem~\ref{contramodule-ext-and-ext}(a), 
\ref{comodule-ext-and-ext}(a), and/or~\ref{coext-and-ext}(a).
{\hfuzz=8.2pt\par}

 The functor
$$
 \Ext_\A\:\sD^\si(\A\mod\Rco)^\sop\times\sD^\si(\A\mod\Rco)
 \lrarrow H^0(\R\contra^\free)
$$
can be defined either as the functor
$\Coext_\C\:\sD^\co(\C\comod\Rco)^\sop\times\sD^\ctr(\C\contra\Rcof)
\allowbreak\rarrow H^0(\R\contra^\free)$, or as the functor
$\Ext^\C\:\sD^\ctr(\C\contra\Rcof)^\sop\times\sD^\ctr(\C\contra\Rcof)
\allowbreak\rarrow H^0(\R\contra^\free)$, or as the functor
$\Ext_\C\:\sD^\co(\C\comod\Rco)^\sop\times\sD^\co(\C\comod\Rco)
\allowbreak\rarrow H^0(\R\contra^\free)$; the three points of view are
equivalent by Proposition~\ref{r-cofree-cotor-ctrtor}(b).
 Alternatively, the functor $\Ext_\A$ can be defined as the functor
$\Ext_\U\:\sD^\si(\U\mod\Rco)^\sop\times\sD^\si(\U\mod\Rco)
\rarrow H^0(\R\contra^\free)$.
 The two definitions are equivalent and agree with the definition
of the functor $\Ext_\A$ for a wcDG\+algebra $\A$ by
Theorem~\ref{contramodule-ext-and-ext}(b), 
\ref{comodule-ext-and-ext}(b), and/or~\ref{coext-and-ext}(b).
{\hfuzz=10.6pt\hbadness=1700\par}

 The functor
$$
 \Ext_\A\:\sD^\si(\A\mod\Rctr)^\sop\times\sD^\si(\A\mod\Rco)
 \lrarrow H^0(\R\comod^\cofr)
$$
can be defined either as the functor
$\Coext_\C\:\sD^\co(\C\comod\Rfr)^\sop\times\sD^\ctr(\C\contra\Rcof)
\allowbreak\rarrow H^0(\R\comod^\cofr)$, or as the functor
$\Ext^\C\:\sD^\ctr(\C\contra\Rfr)^\sop\times\sD^\ctr(\C\contra\Rcof)
\allowbreak\rarrow H^0(\R\comod^\cofr)$, or as the functor
$\Ext_\C\:\sD^\co(\C\comod\Rfr)^\sop\times\sD^\co(\C\comod\Rcof)
\allowbreak\rarrow H^0(\R\comod^\cofr)$; the three points of view are
equivalent by Proposition~\ref{r-cofree-cotor-ctrtor}(a).
 Alternatively, the functor $\Ext_\A$ can be defined as the functor
$\Ext_\U\:\sD^\si(\U\mod\Rctr)^\sop\times\sD^\si(\U\mod\Rco)
\rarrow H^0(\R\comod^\cofr)$.
 The two definitions are equivalent and agree with the definition
of the functor $\Ext_\A$ for a wcDG\+algebra $\A$ by
Theorem~\ref{contramodule-ext-and-ext}(c), 
\ref{comodule-ext-and-ext}(c), and/or~\ref{coext-and-ext}(c).
{\hfuzz=7.6pt\hbadness=1700\par}

 The functor
$$
 \Ext_\A\:\sD^\si(\A\mod\Rco)^\sop\times\sD^\si(\A\mod\Rctr)
 \lrarrow H^0(\R\contra^\free)
$$
can be defined as the functor $\Coext_\C\:\sD^\co(\C\comod\Rco)^\sop
\times\sD^\ctr(\C\contra\Rctr)\rarrow H^0(\R\contra^\free)$, or
alternatively as the functor $\Ext_\U\:\sD^\si(\U\mod\Rco)^\sop
\times\sD^\si(\U\mod\Rctr)\rarrow H^0(\R\contra^\free)$.
 The two definitions are equivalent and agree with the definition
of the functor $\Ext_\A$ for a wcDG\+algebra $\A$ by
Theorem~\ref{coext-and-ext}(d).
{\emergencystretch=2.5em\hbadness=4100\par}

 The above four functors $\Ext_\A$ are transformed into each other
by the equivalences of categories from
Theorem~\ref{wcainfty-semi-co-derived}(a\+d) and the equivalence
of categories $\R\contra^\free\simeq\R\comod^\cofr$ from
Proposition~\ref{r-co-contra}.

%
%
%
%
%

%
%
%
%

\appendix


\Section{Projective Limits of Artinian Modules}  \label{proj-lim-appx}

 The following results are well-known and easy to prove in the case
of projective systems indexed by countable sets.
 They are certainly known to the specialists in the general case
as well.
 We include their proofs in this appendix, because we have not
succeeded in finding a suitable reference.

\subsection{Main lemma} \label{appx-main-lemma-subsect}
 Here is our main technical assertion.

\begin{lem} \label{appx-main-lemma}
 Let $(R_\alpha)$ be a filtered projective system of (noncommutative)
rings and surjective morphisms between them and $(M_\alpha)$ be
a projective system of Artinian $R_\alpha$\+modules.
 Then the first derived functor of projective limit vanishes on
$(M_\alpha)$, that is\/ $\limpr_\alpha^1 M_\alpha=0$.
\end{lem}

\begin{proof}
 First of all, it suffices to consider the case when the indices
$\alpha$ form a filtered (directed) poset, rather than a filtered
category~\cite[Proposition~I.8.1.6]{Groth}.
 Furthermore, the derived functors $\limpr^n_\alpha$ computed in
the abelian category of projective systems of $R_\alpha$\+modules
agree with the ones computed in the category of projective
systems of abelian groups (so the assertion of Lemma is unambigous).
 Indeed, given an index $\alpha_0$ and an $R_{\alpha_0}$\+module
$N_0$, consider the projective system $(N_\alpha)$ with 
$N_\alpha=N_0$ for $\alpha\ge\alpha_0$ and $N_\alpha=0$ otherwise.
 Then infinite products of projective systems of this form are
adjusted both to the projective limits of $R_\alpha$\+modules and
abelian groups.

 In order to compute $\limpr_\alpha^1 M_\alpha$, it suffices to
embed $(M_\alpha)$ into an injective projective system of
$R_\alpha$\+modules $(L_\alpha)$; denoting by $(K_\alpha)$
the cokernel of this embedding, we have $\limpr_\alpha^1 M_\alpha
= \coker(\limpr_\alpha L_\alpha\to\limpr_\alpha K_\alpha)$.
 Replacing one of the modules $M_\alpha$ with the image of
the map $M_\beta\rarrow M_\alpha$ into it from some module
$M_\beta$ (and all the modules $M_\gamma$ with $\gamma>\alpha$
with the related full preimage under the map $M_\gamma\to M_\alpha$)
does not affect the projective limits we are interested in.

 Iterating this procedure over a well-ordered set of steps and
passing to the projective limits of projective systems as needed
still does not change $\limpr_\alpha M_\alpha$ 
and $\limpr_\alpha K_\alpha$, because projective limits commute
with projective limits.
 At every component with a given index~$\alpha$ such a projective
limit stabilizes due to the Artinian condition.
 When the possibilities to iterate the procedure nontrivially are
exhausted, we will obtain a projective system $M'_\alpha$ of
Artinian $R_\alpha$\+modules and surjective maps between them
such that $\limpr_\alpha M'_\alpha = \limpr_\alpha M_\alpha$ and
$\limpr_\alpha^1 M'_\alpha = \limpr_\alpha^1 M_\alpha$.
 So we can assume $M_\alpha$ to be a filtered projective system
of Artinian modules and surjective maps between them.

 Let $(k_\alpha)$ be an element of the projective limit
$\limpr_\alpha K_\alpha$; we want to lift it to an element of
$\limpr_\alpha L_\alpha$.
 For this purpose, we will apply Zorn's lemma to the following
poset~$X$.
 Its elements are subsets $D$ in the set of all indices
$\{\alpha\}$ endowed with chosen preimages $l_\alpha\in L_\alpha$
of the elements $k_\alpha\in K_\alpha$ for all $\alpha\in D$.
 The elements $l_\alpha$ must satisfy the following condition.
 For any finite subset $S\subset D$ there should exist
an index~$\beta$ (not necessarily belonging to~$D$) and a preimage
$l'_\beta\in L_\beta$ of the element $k_\beta\in K_\beta$
such that $\beta>\alpha$ for all $\alpha\in S$ and the map
$L_\beta\rarrow L_\alpha$ takes $l'_\beta$ to~$l_\alpha$.

 Clearly, $X$ contains the unions of all its linearly ordered
subsets.
 It remains to check that for any $(D;l_\alpha)\in X$ and
$\gamma\notin D$ one can find a preimage $l_\gamma\in L_\gamma$
of the element~$k_\gamma$ so as to make $(D\cup\{\gamma\};\>
l_\alpha,l_\gamma)$ a new element of~$X$.

 Notice that, given a finite set of indices $S$, preimages
$l_\alpha$ of the elements $k_\alpha$ for all $\alpha\in S$,
and an index $\beta>\alpha$ for all $\alpha\in S$,
the possibility to choose an appropriate element $l'_\beta$
as above does not depent on the choice of~$\beta$.
 Indeed, suppose $\beta'>\beta>\alpha$ for all $\alpha\in S$.
 Then, having an appropriate element $l'_{\beta'}\in L_{\beta'}$,
one can take its image under the map $L_{\beta'}\rarrow L_\beta$,
obtaining an appropriate element in~$L_\beta$.
 Conversely, given an appropriate element $l'_\beta\in L_\beta$,
one can pick any preimage $l''_{\beta'}$ of $k_{\beta'}$ in
$L_{\beta'}$, take its image $l''_\beta$ under the map
$L_{\beta'}\rarrow L_\beta$, consider the difference
$l'_\beta-l''_\beta$ as an element of $M_\beta$, lift it to
$M_{\beta'}$, and add the result to $l''_{\beta'}$, obtaining
the desired element~$l'_{\beta'}$.

 Now let $S\subset D$ be a finite subset, let $\beta$ be an index
greater than all the elements of $S$, and $l'_\beta\in L_\beta$
be an appropriate preimage of the element~$k_\beta$, as above.
 Consider the set $P_S\subset L_\gamma$ of all
the preimages~$l_\gamma$ of the element~$k_\gamma$ for which
the pair $(S\cup\{\gamma\};\>l_\alpha,l_\gamma)$ belongs
to~$X$.
 First of all, let us show that the set $P_S$ is nonempty.
 Indeed, let $\delta$ be any index greater than both $\beta$
and~$\gamma$.
 As we have seen, the element $l'_\beta$ can be lifted to
a preimage $l'_\delta\in L_\delta$ of the element~$k_\delta$
that would be appropriate for $(S;l_\alpha)$.
 Set $l_\gamma$ to be the image of~$l_\delta$ under the map
$L_\delta\rarrow L_\gamma$.

 Furthermore, $P_S$ is an affine $R_\gamma$\+submodule (i.~e.,
an additive coset by a conventional $R_\gamma$\+submodule)
of~$L_\gamma$.
 Indeed, one can easily see that $P_S$ is closed under
those linear combinations with coefficients in~$R_\gamma$
in which the sum of all coefficients is equal to~$1$.
 Finally, if $S'$, $S''\subset D$ are two finite subsets and
$S=S'\cup S''$, then the affine submodule $P_S\subset L_\gamma$
is contained in the intersection $P_{S'}\cap P_{S''}$.
 In addition, all the affine submodules $P_S$ are contained in
the coset of $L_\gamma$ modulo $M_\gamma$ corresponding to
the element $k_\gamma\in K_\gamma$.

 Since we assume the $R_\gamma$\+module $M_\gamma$ to be Artinian,
it follows that the intersection of all $P_S$ is nonempty.

 Therefore, Zorn's lemma provides us with a system of preimages
$l_\alpha\in L_\alpha$ of the elements $k_\alpha$ defined for all
indices~$\alpha$ and forming an element of the set~$X$, i.~e.,
satisfying the compatibility condition formulated above.
 Clearly, it follows that the map $L_\beta\to L_\alpha$ takes
$l_\beta$ to~$l_\alpha$ for all $\beta>\alpha$; so Lemma is
proven.
\end{proof}

\subsection{Pro-Artinian rings}  \label{appx-pro-rings}
 The following two corollaries of the main lemma demonstrate
that a pro-Artinian topological ring is a well-behaved notion.

\begin{cor} \label{appx-pro-rings-cor1}
 Let $(R_\alpha$) be a filtered projective system of (right) Artinian
(noncommutative) rings and surjective morphisms between them, and
let\/ $\R$ be its projective limit.
 Then the projection map\/ $\R\rarrow R_\alpha$ is surjective for
each~$\alpha$.
\end{cor}

\begin{proof}
 Consider the short exact sequence of projective systems
$\ker(R_\beta\to R_\alpha)\rarrow R_\beta\rarrow R_\alpha$
of right $R_\beta$\+modules, indexed by all $\beta>\alpha$
with $\alpha$~fixed, pass to the projective limits, and
apply Lemma~\ref{appx-main-lemma}.
\end{proof}

\begin{cor} \label{appx-pro-rings-cor2}
 Let\/ $\R$ be a complete separated (noncommutative) topological
ring with a base of neighborhoods of zero formed by open ideals
with right Artinian quotient rings.
 Let\/ $\J\subset\R$ be a closed ideal.
 Then the quotient ring\/ $\R/\J$ is complete in the quotient
topology (and has all the other properties assumed above for\/~$\R$).
\end{cor}

\begin{proof}
 Consider the short exact sequence of projective systems
$(\J+\I)/\I\rarrow\R/\I\rarrow\R/(\J+\I)$ of right modules
over the Artinian rings $\R/\I$, indexed by all the open
ideals $\I\subset\R$.
 Passing to the projective limits and applying
Lemma~\ref{appx-main-lemma}, we obtain the desired
isomorphism $\R/\J\simeq\limpr_\I\R/(\J+\I)$.
\end{proof}

\subsection{Pro-Artinian modules} \label{appx-pro-modules}
 The following result will be useful for us when dealing with
comodules over pro-Artinian topological rings.

\begin{cor} \label{appx-pro-modules-cor}
 Let $\R$ be (noncommutative) topological ring. 
 Then the functor of projective limit acting from the category
of pro-objects in the abelian category of discrete\/
$\R$\+modules of finite length (or discrete Artinian $\R$\+modules)
to the abelian category of (nontopological)\/ $\R$\+modules
is exact and conservative.
\end{cor}

\begin{proof}
 Given a morphism between the pro-objects $M$ and $N$ represented
by filtered projective systems $(M_\alpha)$ and $(N_\beta)$, one
can consider the set of all triples $(\alpha,\beta,f)$,
where $f\:M_\alpha\rarrow N_\beta$ is a morphism of modules
representing the morphism of pro-objects $M\rarrow N_\beta$.
 This is a filtered poset in the natural order.
 Replacing the index sets $\{\alpha\}$ and $\{\beta\}$ with
the poset $\{\.(\alpha,\beta,f)\.\}$, one can assume the pro-objects
$M$ and $N$ to be represented by projective systems $(M_\gamma)$
and $(N_\gamma)$ indexed by the same filtered poset $\{\gamma\}$
and the morphism $M\rarrow N$ to be represented by a morphism of
projective systems $M_\gamma\rarrow N_\gamma$.
 The kernel and cokernel of the morphism of pro-objects
$M\rarrow N$ are then represented by the (terms-wise) kernel
and cokernel of this morphism of projective systems.

 Now the assertion that the projective limit functor is exact
follows straightforwardly from Lemma~\ref{appx-main-lemma}
(it suffices to consider the constant projective system of
rings $R_\gamma=R$).
 To prove conservativity, notice that any pro-object represented
by a projective system $(M_\alpha)$ is also represented by
its maximal projective subsystem $M'_\alpha\subset M_\alpha$
with surjective morphisms $M'_\beta\rarrow M'_\alpha$ (because
$M_\alpha$ are Artinian, so $M'_\alpha$ is the image of
the morphism into $M_\alpha$ from some $M_\delta$).
 Now if $M'_\beta\ne0$ for some~$\beta$, then
$\limpr_{\gamma>\beta} M'_\gamma\rarrow M'_\beta$ is
a surjective morphism of $R$\+modules by the same Lemma,
hence $\limpr_\alpha M_\alpha\ne0$ (cf.\ the proof of
Corollary~\ref{appx-pro-rings-cor1}).
\end{proof}

\subsection{Closed ideals}  \label{appx-ideals}
 The following result demonstrates that intersections of closed
ideals in a pro-Artinian topological ring are well-behaved.
 It will be used in the proof of Lemma~\ref{idempotent-lifting}.

\begin{cor} \label{appx-ideals-cor}
 Let\/ $\R$ be a (right) pro-Artinian topological ring and\/
$\J_\alpha$ be a filtered poset of its closed ideals ordered
by inclusion.
 Then for any open ideal\/ $\I\subset\R$ one has\/
$\I+\bigcap_\alpha\J_\alpha=\bigcap_\alpha(\I+\J_\alpha)$.
 In particular, if the intersection of all\/ $\J_\alpha$
is contained in\/ $\I$, then there exists an index\/~$\alpha$
such that\/ $\J_\alpha$ is contained in\/~$\I$.
\end{cor}

\begin{proof}
 Since the family of ideals $\I+\J_\alpha$ is filtered and
the ring $\R/\I$ is (right) Artinian, this family stabilizes,
that is $\I+\J_\alpha=\I'$ is the same open ideal in~$\R$
for all $\alpha$ large enough.
 Therefore, $\limpr_\alpha(\J_\alpha+\I)/\I = 
(\bigcap_\alpha(\J_\alpha+\I))/\I$.
 Clearly, $\bigcap_\alpha\J_\alpha=\limpr_\alpha\J_\alpha$
and $\J_\alpha=\limpr_\K(\J_\alpha+\K)/\K$, where $\K$ runs
over all the open ideals in~$\R$; hence
$\bigcap_\alpha\J_\alpha=\limpr_{\alpha,\K}(\J_\alpha+\K)/\K$.
 The map $(\J_\alpha+\K)/\K\rarrow(\J_\alpha+\I)/\I$
is well-defined and surjective for all $\K\subset\I$.
 Passing to the projective limit in $\K$ and~$\alpha$, and
using Lemma~\ref{appx-main-lemma}, we conclude that
the map $\bigcap_\alpha\J_\alpha\rarrow
(\bigcap_\alpha(\J_\alpha+\I))/\I$ is surjective, as desired.
\end{proof}

\Section{Contramodules over a Complete Noetherian Ring}
\label{noetherian-local-appx}

 The aim of this appendix is to describe the abelian category of
contramodules over a complete Noetherian local ring $\R$ as
a full subcategory of the category of $\R$\+modules formed by
the $\R$\+modules with some special properties.
 In fact, we obtain a somewhat more general result applicable
to any adic completion of a Noetherian ring.

 This description of $\R$\+contramodules does not appear to be
particularly useful for our purposes (and indeed, it is never used
in the main body of the paper).
 However, it may present an independent interest, and may also
help some readers to acquaint themselves with the notion of
an $\R$\+contramodule.

 We restrict ourselves to commutative rings here.
 A generalization to noncommutative Noetherian rings can be found
in~\cite[Appendix~C]{Pcosh}.

\subsection{Formulation of theorem}
\label{appx-b-thm-formulation}
 Recall that for any topological ring $\R$ (in the sense of
Section~\ref{topological-rings}) there is a natural exact,
conservative forgetful functor $\R\contra\rarrow\R\mod$ from
the abelian category of $\R$\+contramodules to the abelian
category of modules over the ring $\R$ viewed as an abstract
(nontopological) ring.

 The forgetful functor $\R\contra\rarrow\R\mod$ preserves
infinite products, but may not preserve infinite direct sums.
 It is \emph{not} in general a tensor functor with respect to
the natural tensor category structures on $\R\contra$ and
$\R\mod$, though it commutes with the internal $\Hom$
(see Section~\ref{contramodules-sect}).

 Let $R$ be a Noetherian ring and $m\subset R$ be an ideal.
 Consider the $m$\+adic completion $\R=\limpr_n R/m^n$ of $R$
and let $\m=\limpr_n m/m^n\subset\R$ be the extension of~$m$
in~$\R$.
 Endow the ring $\R$ with the topology of the projective limit
(i.~e., the $\m$\+adic topology).

 In particular, given a complete Noetherian local ring $\R$ with
the maximal ideal $\m$, one can always take $R=\R$ and $m=\m$.
 We will be interested in the composition of the forgetful functor
$\R\contra\rarrow\R\mod$ with the functor of restriction of scalars
$\R\mod\rarrow R\mod$.
 By $\Ext_R^*({-},{-})$ we denote the $\Ext$ functor computed in
the abelian category of $R$\+modules.

\begin{thm}  \label{noetherian-r-contramodules-thm}
\textup{(1)} The forgetful functor\/ $\R\contra\rarrow R\mod$ is
fully faithful. \par
\textup{(2)} The image of the functor\/ $\R\contra\rarrow R\mod$
is the full abelian subcategory in $R\mod$ consisting of all
the $R$\+modules $P$ satisfying any of the following equivalent
conditions:
\begin{enumerate}
\renewcommand{\theenumi}{\alph{enumi}}
\item for any multiplicatively closed subset $S\subset R$ having
a nonempty intersection with~$m$, any $R[S^{-1}]$\+module $L$,
and any integer $i\ge0$ one has\/ $\Ext^i_R(L,P)=0$;
\item assuming $m$ is a maximal ideal in $R$, for any
multiplicatively closed subset $S\subset R$ and
any integer $i\ge0$ one has\/ $\Ext^i_R(R[S^{-1}],P)=0$,
with the only exception of $S\cap m=\varnothing$ and $i=0$;
\item for any element $s\in m$ and any $i\ge0$ one has\/
$\Ext^i_R(R[s^{-1}],P)=0$;
\item for any $j=1$,~\dots,~$n$ and any $i=0$ or~$1$ one has\/
$\Ext^i_R(R[x_j^{-1}],P)=0$, where $x_1$,~\dots,~$x_n$ is some
set of generators of the ideal $m\subset R$.
\end{enumerate}
\end{thm}

 Notice that one always has $\Ext_R^i(R[s^{-1}],P)=0$ for $i\ge 2$,
as the $R$\+module $R[s^{-1}]$ has projective dimension at most~$1$
(see Lemma~\ref{ext-0-1-condition} below).
 The equivalent conditions (a\+d) in part~(2) are modelled after
Jannsen's definition of \emph{weakly $l$\+complete abelian groups}
in~\cite[Definition~4.6]{Jan} or, which is the same,
the definition of \emph{Ext-$p$-complete abelian groups} in
Bousfield--Kan~\cite[Sections~VI.3\+-4]{BK}.
 (These are also closely related to Harrison's \emph{co-torsion
groups}~\cite{Har}.)

 The theorem can be reformulated by saying that contramodules over
the adic completion of a Noetherian ring are the same thing as
the \emph{cohomologically complete modules} of
Porta--Shaul--Yekutieli~\cite{PSY,PSY2,Yek}.
 The case of $\R=k[[\epsilon]]$ was considered
in~\cite[Remark~A.1.1]{Psemi} and the comparison with Jannsen's
definition in the $l$\+adic case ($R=\boZ$ and $\R=\boZ_l$) was
mentioned in~\cite[Remark~A.3]{Psemi}.

 The proof of the above theorem occupies
Sections~\ref{hom-ext-contramodule-subsect}\+-%
\ref{one-variable-subsect}.

\subsection{Hom and Ext into a contramodule}
\label{hom-ext-contramodule-subsect}
 Let us show that the underlying $R$\+module of any $\R$\+contramodule
$\P$ satisfies the conditions~(a\+b).

 Quite generally, let $R$ be a ring, $\R$ be a topological ring,
and $R\rarrow\R$ be a ring homomorphism.
 Let $M$ be an $R$\+module and $\P$ be an $\R$\+contramodule.
 Then the set $\Hom_R(M,\P)$ of all $R$\+module homomorphisms
$M\rarrow\P$ has a natural $\R$\+contramodule structure with
the ``pointwise'' infinite summation operation
$(\sum_\alpha r_\alpha f_\alpha)(x) =
\sum_\alpha r_\alpha f_\alpha(x)$, where $f_\alpha\in\Hom_R(M,\P)$
are $R$\+module homomorphisms and $r_\alpha\in\R$ is any family of
elements converging to zero (cf.\ Section~\ref{hom-operations}).
 The $R$\+module structure obtained by applying the forgetful
functor $\R\contra\rarrow R\mod$ to this $\R$\+contramodule structure
coincides with the natural $R$\+module structure on $\Hom_R(M,\P)$.

 The $\R$\+contramodule structure on $\Hom_R(M,\P)$ is functorial
with respect to $R$\+module morphisms $M'\rarrow M''$ and
$\R$\+contramodule morphisms $\P'\rarrow\P''$.
 It follows that for every $i\ge0$ the functor $M\mpsto\Ext^i_R(M,\P)$
factorizes through the forgetful functor $\R\contra\rarrow R\mod$,
i.~e., for any $R$\+module $M$ the $R$\+module $\Ext^i_R(M,\P)$
has a natural $\R$\+contramodule structure.
 Indeed, the $\Ext$ module in question can be computed in terms of
a left projective resolution of the $R$\+module $M$, which makes it
the cohomology (contra)module of a complex of $\R$\+contramodules.

 On the other hand, if an element $s\in R$ acts invertibly in
an $R$\+module $L$, then its action in the $R$\+module
$\Ext^i_R(L,N)$ is also invertible for any $R$\+module $N$ and
$i\ge0$.
 Now if $\m$~is a topologically nilpotent ideal in $\R$ and
the action of an element $s\in\m$ in the $\R$\+contramodule
$\Q=\Ext^i_R(L,\P)$ is invertible, then $\m\Q=\Q$ and
$\Q=0$ by Lemma~\ref{nakayama-lemma}.
 This proves the property~(a).

 To check~(b), it remains to consider the case when $S\cap m =
\varnothing$.
 Then the elements of $S$ become invertible in $\R$, hence
$\Ext^i_R(M,P)\simeq\Ext^i_{R[S^{-1}]}(M[S^{-1}],P)$ for
any $R$\+mod\-ule $M$ and $\R$\+module~$P$.
 In particular, $\Ext^i_R(R[S^{-1}],P)\simeq
\Ext^i_{R[S^{-1}]}(R[S^{-1}],P)\simeq P$ for $i=0$ and $0$ for
$i>0$.

\subsection{Ring of power series} \label{power-series-and-completion}
 Abusing notation, we will denote the images of the generators $x_j\in R$
of the ideal $m\subset R$ in the complete ring $\R$ also by~$x_j$.
 Consider the ring $\T=R[[t_j]]$ of formal power series in
the variables~$t_j$ with coefficients in the ring $R$, and endow
$\T$ with the standard formal power series topology (i.~e.,
the adic topology for the ideal generated by~$t_j$).

 There exists a unique continuous ring homomorphism
$\tau\:\T\rarrow\R$ equal to the natural map $R\rarrow\R$ in
the restriction to $R$ and taking $t_j$ to~$x_j$.
 Since $x_j$ generate~$\m$, the map $\tau$ is surjective and open,
so the topology on $\R$ is the quotient topology of the topology
on~$\T$.
 In addition, the image under~$\tau$ of the ideal $(t_j)\subset\T$
generated by all the elements $t_j\in\T$ is equal to $\m\subset\R$.

 By~\cite[Theorem~8.12]{Mats}, the ideal $\J=\ker\tau\subset
\T=R[[t_j]]$ is generated (as an abstract ideal in a nontopological
ring) by the elements $x_j-t_j$.
 For completeness, let us give an independent proof of this
claim.
 Clearly, $x_j-t_j\in\ker\tau$.

\begin{lem}
 For any Noetherian ring $R$ and finite set of variables~$t_j$,
any ideal in the ring of formal power series $R[[t_j]]$ is closed
in the $(t_j)$\+adic topology.
\end{lem}

\begin{proof}
 The assertion of Lemma can be equivalently rephrased by saying that
an ideal in $R[[t_j]]$ is determined by its images in the quotient
rings $R[[t_j]]/(t_j)^N$, where $N$ runs over positive integers.
 In the case of a single variable~$t$, it is clear from
the standard proof of the Hilbert basis theorem for formal power
series (see, e.~g., \cite[Theorem~3.3]{Mats}) that an ideal in
$R[[t]]$ is determined by its images in the rings $R[[t]]/(t^N)$.
 The general case is handled by induction: an ideal in
$R[[t_1,\dotsc,t_n]]$ is determined by its images in
$R[[t_2,\dotsc,t_n]][[t_1]]/(t_1^{N_1})\simeq
(R[t_1]/(t_1^{N_1}))[[t_2,\dotsc,t_n]]$, which in turn are
determined by their images in $(R[[t_1,t_2]]/(t_1^{N_1},t_2^{N_2}))
[[t_3,\dotsc,t_n]]$, etc.
 So finally an ideal in
$R[[t_1,\dotsc,t_n]]$ is determined by its images in
$R[[t_1,\dotsc,t_n]]/(t_j^{N_j})$, where $N_1$,~\dots, $N_n$
are positive integers.
\end{proof}

 Hence it suffices to show that the images of the elements $x_j-t_j$
generate the image of $\J$ in the quotient ring $\T/(t_j)^N$
for each~$N$.
 Since $\tau((t_j)^N)=\m^N$, the ideal $\J+(t_j)^N\subset\T$
is the kernel of the ring homomorphism $\T\rarrow\R/\m^N$
taking $t_j$ to the images of~$x_j$.
 As $\R/\m^N\simeq R/m^N$, it remains to show that the kernel of
the homomorphism $R[[t_j]]/(t_j)^N\rarrow R/m^N$ is generated
by $x_j-t_j$.

 Since the monomial $t_1^{l_1}\dotsm t_n^{l_n}$ is equal to
$x_1^{l_1}\dotsm x_n^{l_n}$ modulo the ideal generated by
$x_j-t_j$, the ideal generated by $x_j-t_j$ in $R[[t_j]]/(t_j)^N$
contains~$m^N$.
 We have reduced to the obvious assertion that
the kernel of the map $(R/m^N)[t_j]/(t_j)^N\rarrow R/m^N$ is
generated by $x_j-t_j$.

\subsection{Contramodules over a quotient ring}
 The following lemma is an almost tautological restatement of
the definitions.

\begin{lem}  \label{restriction-faithful}
\textup{(a)} Let\/ $\U\rarrow\V$ be a continuous homomorphism of
topological rings such that any family $X\rarrow\V$ of elements
of\/ $\V$ indexed by a set $X$ and converging to~$0$ in
the topology of\/ $\V$ can be lifted to a family $X\rarrow\U$
converging to~$0$ in the topology of\/~$\U$.
 Then the functor of restriction of scalars $\V\contra\rarrow
\U\contra$ is fully faithful. \par
\textup{(b)} In the situation of part~\textup{(a)}, let\/ $\J$
be the kernel of the morphism\/ $\U\rarrow\V$.
 Assume further that $f\:Z\rarrow\J$ is a family of elements
converging to~$0$ in the topology of\/ $\U$ such that for any
family of elements $g\: X\rarrow\J$ converging to~$0$ there
exists a family of families of elements $h\:Z\times X\rarrow\J$
converging to~$0$ for every fixed $z\in Z$ such that
$g(x)=\sum_{z\in Z}f(z)h(z,x)$ for any $x\in X$.
 Then the image of the functor of restriction of scalars
$\V\contra\rarrow\U\contra$ consists precisely of those\/
$\U$\+contramodules\/ $\P$ for which $\sum_{z\in Z} f(z)q(z)=0$
in\/ $\P$ for every family of elements $q\:Z\rarrow\P$.
\end{lem}

\begin{proof}
 Given two $\V$\+contramodules $\P$ and $\Q$, any morphism of
$\U$\+contramodules $\P\rarrow\Q$ is also a morphism of
$\V$\+contramodules provided that the map $\U[[\P]]\rarrow
\V[[\P]]$ is surjective.
 The condition of part~(a) means that the map $\U[[X]]\rarrow
\V[[X]]$ is surjective for any set~$X$, so the assertion of
part~(a) follows.

 To prove part~(b), notice that the sequence $0\rarrow\J[[X]]
\rarrow\U[[X]]\rarrow\V[[X]]\rarrow0$ is exact for any set $X$
in the assumptions of~(a).
 One has to show that, for a $\U$\+contramodule $\P$, the map
$\J[[\P]]\rarrow\P$ is zero provided that the family of elements
$f(z)$ acts by zero in~$\P$.
 Let $g\:\P\rarrow\J$ be an element of $\J[[\P]]$, i.~e.,
a $\P$\+indexed family of elements of $\J$ converging to~$0$.
 By the assumption of part~(b), we have $g(p)=\sum_{z\in Z}
f(z)h(z,p)$, hence $\sum_{p\in\P}g(p)p =
\sum_{z\in Z} f(z)\sum_{p\in\P}h(z,p)p = 0$.
\end{proof}

 Now we return to our ring homomorphism $\T=R[[t_j]]\rarrow\R$.

\begin{cor}  \label{r-t-restriction}
 The functor of restriction of scalars\/ $\R\contra\rarrow\T\contra$
is fully faithful, and its image consists precisely of those\/
$\T$\+contramodules in (the underlying\/ $\T$\+module structure of)
which the elements $x_j-t_j$ act by zero.
\end{cor}

\begin{proof}
 The topologies of $\T$ and $\R$ having countable bases of
neighborhoods of zero, the condition~(a) of
Lemma~\ref{restriction-faithful} clearly holds.
 It remains to check the condition~(b) for the finite family of
elements $f(j)=x_j-t_j$.

 By the Artin--Rees lemma~\cite[Theorem~8.5]{Mats} applied to
the $\T$\+submodule $\J\subset\T$ and the ideal $(t_j)\subset\T$,
there exists an integer~$l$ such that
$(t_j)^N\cap\J\subset(t_j)^{N-l}\J$ for all large enough~$N$.
 It follows easily from this observation together with the fact
that $\J$ is the ideal generated by the elements $x_j-t_j$ in $\T$
that any family of elements of $\J$ converging to~$0$ in
the topology of $\T$ can be presented as a linear combination of
$n$~families of elements converging to~$0$ in $\T$ with
the coefficients $x_j-t_j$.
\end{proof}

\subsection{Contramodules over $R[[t]]$}
 The results of this section and the next one do not depend on
the Noetherianness assumption on a ring~$R$.
 Let $R[[t]]$ be the ring of formal power series in one variable~$t$
with coefficients in $R$, endowed with the $t$\+adic topology;
and let $R[t]$ be the ring of polynomials.

\begin{lem} \label{one-variable-power-series}
 The forgetful functor $R[[t]]\contra\rarrow R[t]\mod$ identifies
the category of $R[[t]]$\+contramodules with the full subcategory
of the category of $R[t]$\+modules consisting of all the modules $P$
with the following property.
 For any sequence of elements $p_i\in P$, \ $i\ge0$, there exists
a unique sequence of elements $q_i\in P$ such that
$q_i=p_i+tq_{i+1}$ for all $i\ge0$.
\end{lem}

\begin{proof}
 First of all, the category of $R[[t]]$\+contramodules is equivalent
to the category of $R$\+modules $P$ endowed with the operation
assigning to every sequence of elements $p_i\in P$, \ $i\ge0$,
an element $\sum_{i=0}^\infty t^ip_i\in P$ and satisfying
the following equations of linearity, unitality, and associativity:
$$\textstyle
 \sum_{i=0}^\infty t^i(r'p'_i+r''p''_i) =
 r'\sum_{i=0}^\infty t^ip'_i + r''\sum_{i=0}^\infty t^ip''_i
$$
for any $p'_i$, $p''_i\in P$ and $r'$, $r''\in R$;
$$\textstyle
 \sum_{i=0}^\infty t^ip_i = p_0
$$
when $p_1=p_2=\dotsb=0$ in $P$, and
$$\textstyle
 \sum_{i=0}^\infty t^i\sum_{j=0}^\infty t^j p_{ij}=
 \sum_{n=0}^\infty t^n\sum_{i+j=n} p_{ij}
$$
for any $p_{ij}\in P$, \ $i\ge0$, \ $j\ge0$.
 Morphisms of $R[[t]]$\+contramodules correspond to $R$\+module
morphisms $f\:P\rarrow Q$ such that $\sum_{i=0}^\infty
t^i f(p_i) = f(\sum_{i=0}^\infty t^i p_i)$ in $Q$ for any sequence
of elements $p_i\in P$.

 Indeed, the $R[[t]]$\+contramodule infinite summation operations
restricted to the case of the sequence of coefficients~$t^i$
clearly have the above properties.
 Conversely, given an $R$\+module $P$ endowed with the operation
of infinite summation with the coefficients $t^i$ as above,
one defines
$$\textstyle
 \sum_\alpha(\sum_{i=0}^\infty g_{\alpha,i}t^i)p_\alpha=
 \sum_{i=0}^\infty t^i\sum_\alpha g_{\alpha,i}p_\alpha
$$
for any family of elements $g_\alpha=
\sum_{i=0}^\infty g_{\alpha,i}t^i$, \ $g_{\alpha,i}\in R$,
converging to~$0$ in $R[[t]]$ (so $g_{\alpha,i}=0$ for all but
a finite number of indices~$\alpha$, for every fixed~$i$),
and any $p_\alpha\in P$.

 Let us check that such infinite summations with
the coefficients~$g_\alpha$ are associative provided that
the infinite summations with the coefficients~$t^i$ were
associative and linear.
 We have $\sum_\alpha g_\alpha(t)\sum_\beta h_{\alpha\beta}(t)
p_{\alpha\beta} = \sum_{i=0}^\infty t^i\sum_\alpha
g_{\alpha,i}\sum_{j=0}^\infty t^j\sum_\beta h_{\alpha\beta,j}
p_{\alpha\beta} = \sum_{i=0}^\infty t^i\sum_{j=0}^\infty t^j
\sum_\alpha g_{\alpha,i}\sum_\beta h_{\alpha\beta,j}p_{\alpha\beta}
= \sum_{n=0}^\infty t^n \sum_{i+j=n}\sum_{\alpha,\beta}
g_{\alpha,i}h_{\alpha\beta,j}p_{\alpha\beta}=
\sum_{n=0}^\infty t^n\allowbreak \sum_{\alpha,\beta} k_{\alpha\beta,n}
p_{\alpha\beta}$, where $g_\alpha(t)=\sum_{i=0}^\infty 
g_{\alpha,i}t^i$, \ $h_{\alpha\beta}(t)=\sum_{j=0}^\infty
h_{\alpha\beta,j}t^j$, and $k_{\alpha\beta}(t) =
g_{\alpha}(t)h_{\alpha\beta}(t) = \sum_{n=0}^\infty
k_{\alpha\beta,n}(t)t^n$.
 It is clear that any morphism preserving the operations of
infinite summation with the coefficients~$t^i$ preserves also
the operations of infinite summation with
the coeffcients~$g_\alpha$.

 Now assume that the infinite summation operations with
the coefficients~$t^i$ are defined on an $R$\+module~$P$.
 The action of the operator~$t$ on $P$ is then provided by
the obvious rule $tp=\sum_{i=0}^\infty t^ip_i$, where
$p_1=p$ and $p_i=0$ for $i\ne 1$.

 Notice that any sequence of elements $q_i\in P$ such that $q_i =
tq_{i+1}$ for $i\ge0$ is zero.
 Indeed, one has $\sum_{i=0}^\infty t^iq_{i+n} =
\sum_{i=0}^\infty t^i t q_{i+n+1} = \sum_{i=0}^\infty t^{i+1}
q_{i+n+1} = -q_n + \sum_{i=0}^\infty t^iq_{i+n}$, hence $q_n=0$
for all $n\ge0$ (cf.\ the proof of Lemma~\ref{nakayama-lemma}).

 It follows that for any given sequence $p_i\in P$, a sequence
$q_i\in P$ such that $q_i=p_i+tq_{i+1}$ is unique if it exists.
 To prove the existence, set $q_n=\sum_{i=0}^\infty t^ip_{n+i}$.

 Conversely, assuming the existence and uniquence property for
the sequence~$q_i$ such that $q_i=p_i+tq_{i+1}$, set
$\sum_{i=0}^\infty t^ip_i=q_0$.
 Let us check the equations of linearity, unitality, and
associativity.

 Given two sequences $p'_i$ and $p''_i\in P$, and two elements
$r'$, $r''\in R$, let $q'_i$ and $q''_i\in P$ be the sequences
such that $q^{(s)}_i=p^{(s)}_i+tq^{(s)}_{i+1}$.
 Then the sequence $q_i=r'q'_i+r''q''_i$ satisfies the equation
$q_i = (r'p'_i+r''p''_i) + tq_{i+1}$, hence
$\sum_{i=0}^\infty t^i(r'p'_i+r''p''_i)=q_0=r'q'_0+r''q''_0=
r'\sum_{i=0}^\infty t^ip'_i + r''\sum_{i=0}^\infty t^ip''_i$.

 Given a sequence $p_i\in P$ such that $p_i=0$ for $i>0$, set
$q_0=p_0$ and $q_i=0$ for $i>0$.
 Then the equation $q_i=p_i+tq_{i+1}$ holds for all~$i$, so
$\sum_{i=0}^\infty t^ip_i=q_0=p_0$.

 Finally, given a biindexed sequence $p_{ij}\in P$, find
sequences $q_{ij}\in P$ such that $q_{ij}=p_{ij}+tq_{i,j+1}$
for $i$, $j\ge0$.
 Let $u_n\in P$ be a sequence satisfying $u_n=q_{n,0}+tu_{n+1}$;
set $v_n = q_{0,n}+q_{1,n-1}+\dotsb+q_{n-1,1}+u_n =
q_{0,n}+q_{1,n-1}+\dotsb+q_{n-1,1}+q_{n,0}+tu_{n+1}$.
 Then we have $v_n-tv_{n+1} = (q_{0,n}-tq_{0,n+1}) + \dotsb +
(q_{n,0}-tq_{n+1,0}) = p_{0,n}+p_{1,n-1}+\dotsb+p_{n-1,1}+p_{n,0}$.
 Hence $\sum_{i=0}^\infty t^i\sum_{j=0}^\infty t^jp_{ij} =
\sum_{i=0}^\infty t^iq_{i,0}=u_0=v_0=\sum_{n=0}^\infty
t^n\sum_{i+j=n}p_{ij}$.

 We have described the image of the functor $R[[t]]\contra\rarrow
R\mod$ on the level of objects.
 It remains to show that it is surjective on morphisms.
 Indeed, let $\P$ and $\Q$ be $R[[t]]$\+contramodules, and
$f\:\P\rarrow\Q$ be an $R[t]$\+module morphism.
 Given a sequence $p_i\in P$, \ $i\ge0$, find a sequence $u_i\in P$
such that $u_i=p_i+tu_{i+1}$.
 Then $f(u_i)=f(p_i)+t f(u_{i+1})$, hence
$f(\sum_{i=0}^\infty t^ip_i) = f(u_0)=\sum_{i=0}^\infty t^if(p_i)$. 
\end{proof}

\subsection{Contramodules over $R[[t_1,\dotsc,t_n]]$}
 Now we describe the category of contramodules over the ring
of formal power series $R[[t_j]]$ in several variables~$t_j$,
where $j=1$,~\dots,~$n$, with the $(t_j)$\+adic topology
on $R[[t_j]]$.

\begin{lem}  \label{many-variables-power-series}
 The forgetful functor $R[[t_j]]\contra\rarrow R[t_j]\mod$
identifies the category of $R[[t_j]]$\+contramodules with
the full subcategory of the category of $R[t_j]$\+modules
with the following property.
 For any sequence of elements $p_i\in P$, \ $i\ge0$, and
any variable~$t_j$, \ $1\le j\le n$, there exists a unique
sequence of elements $q_i\in P$ such that $q_i =
p_i+t_jq_{i+1}$ for all $i\ge0$. \hbadness=1150
\end{lem}

\begin{proof}
 For a multiindex $a=(a_1,\dotsc,a_n)$, \ $a_i\ge0$, denote
by $t^a$ the monomial $t_1^{a_1}\dotsm t_n^{a_n}$.
 The category of $R[[t_j]]$\+contramodules is equivalent to
the category of $R$\+modules $P$ endowed with the operation
assigning to every multiindexed sequence of elements $p_a\in P$
an element $\sum_a t^ap_a\in P$ satisfying the equations of
linearity, unitality and associativity similar to those
introduced in the proof of Lemma~\ref{one-variable-power-series}.
 The proof is the same as in the case of one variable~$t$.

 Furthermore, the data of infinite summation operations with
the coefficients~$t^a$ is equivalent to that of infinite summation
operations with the coefficients~$t_j^i$, for every fixed~$j$,
with the commutativity equation $\sum_{i'=0}^\infty t_{j'}^{i'}
\sum_{i''=0}^\infty t_{j''}^{i''}p_{i'i''} = \sum_{i''=0}^\infty
t_{j''}^{i''}\sum_{i'=0}^\infty t_{j'}^{i'}p_{i'i''}$ for every
biindexed sequence $p_{i'i''}\in P$, \ $i'$, $i''\ge0$, and
every two variable numbers $1\le j',\,j''\le n$.
 Indeed, given the infinite summation operations with
the coefficients~$t_j^i$, for every fixed~$j$, on an $R$\+module
$P$, one defines the infinite summation operations with
the coefficients~$t^a$ on $P$ by the rule
$\sum_a t^ap_a=\sum_{a_1=0}^\infty t_1^{a_1}\sum_{a_2=0}^\infty
t_2^{a_2}\dotsb\sum_{a_n=0}^\infty t_n^{a_n}p_{a_1\dotso a_n}$.
 Clearly, any morphism preserving the infinite summation operations
with the coefficients~$t_j^i$ preserves also the infinite summation
operations with the coefficients~$t^a$.
{\hbadness=1150\par}

 Finally, it remains to show that the infinite summation operations
with the coefficients $t_{j'}^{i'}$ and~$t_{j''}^{i''}$ recovered
from an $R[t_j]$\+module structure satisfying our condition commute
with each other.
 For simplicity of notation, denote our variables $t_{j'}$
and~$t_{j''}$ by $s$ and~$t$.
 Given a biindexed sequence $p_{ij}\in P$, \ $i$, $j\ge0$, find
a sequence $q_{ij}$ such that $q_{ij}=p_{ij}+tq_{i,j+1}$ and
a sequence $u_{ij}$ such that $u_{ij}=q_{ij}+su_{i+1,j}$.
 Similarly, find a sequence $v_{ij}$ such that
$v_{ij}=p_{ij}+sv_{i+1,j}$ and a sequence $w_{ij}$ such that
$w_{ij}=v_{ij}+tw_{i,j+1}$ for all $i$, $j\ge0$.
 Set $w_{ij}-sw_{i+1,j}=z_{ij}$.
 Then we have $z_{ij}-tz_{i,j+1}=
(w_{ij}-sw_{i+1,j}) - t(w_{i,j+1}-sw_{i+1,j+1}) =
(w_{ij}-tw_{i,j+1}) - s(w_{i+1,j}-tw_{i+1,j+1}) = 
v_{ij}-sv_{i+1,j} = p_{ij}$.
 It follows that $z_{ij}=q_{ij}$ and therefore $w_{ij}=u_{ij}$
for all $i$, $j\ge0$.
 Now $\sum_{i=0}^\infty s^i\sum_{j=0}^\infty t^j p_{ij} =
\sum_{i=0}^\infty s^i q_{i,0} = u_{0,0} = w_{0,0} =
\sum_{j=0}^\infty t^j v_{0,j} = \sum_{j=0}^\infty t^j
\sum_{i=0}^\infty s^i p_{ij}$.
\end{proof}

\subsection{The one-variable Ext condition}
\label{one-variable-subsect}
 The following lemma compares the conditions~(c\+d) of
Theorem~\ref{noetherian-r-contramodules-thm} with the conditions of
Lemmas~\ref{one-variable-power-series}\+-%
\ref{many-variables-power-series}.

\begin{lem}  \label{ext-0-1-condition}
 Let $R$ be a ring, $s\in R$ be an element, and $P$ be an $R$\+module.
 Then one has\/ $\Ext^1_R(R[s^{-1}],P)=0$ if and only if for any
sequence of elements $p_i\in P$, \ $i\ge0$, there exists a sequence
of elements $q_i\in P$ such that $q_i = p_i+sq_{i+1}$ for all
$i\ge0$.
 One has\/ $\Hom_R(R[s^{-1}],P)=0$ if and only if the sequence of
elements~$q_i$ as above is unique for every (or some particular)
sequence of elements~$p_i$.
\end{lem}

\begin{proof}
 Consider the free resolution $0\rarrow\bigoplus_{i=0}^\infty
Rf_i\rarrow\bigoplus_{i=0}^\infty Re_i\rarrow R[s^{-1}]\rarrow0$
of the $R$\+module $R[s^{-1}]$ with the maps taking
$f_i$ to $e_i-se_{i+1}$ and $e_i$ to $s^{-i}$.

 Then the map $\Hom_R(\bigoplus_i Re_i\;P)\rarrow\Hom_R(\bigoplus_i
Rf_i\;P)$ computing the desired modules $\Hom_R$ and $\Ext^1_R$
is identified with the map taking a sequence $(q_i)\in P$ to
the sequence $p_i=q_i-sq_{i+1}$.
\end{proof}

 Now we are in the position to finish the proof of
Theorem~\ref{noetherian-r-contramodules-thm}.
 The forgetful functor $\R\contra\rarrow R\mod$ factorizes into
the composition of three functors $\R\contra\rarrow\T\contra
\rarrow R[t_j]\mod\rarrow R\mod$.

 Denote by $\T\contra_0$ and $R[t_j]\mod_0$ the full subcategories
of $\T\contra$ and $R[t_j]\mod$, respectively, consisting of
those (contra)modules where the elements $x_j-t_j$ act by zero.
 Then the functor $\R\contra\rarrow R\mod$ can be also decomposed
as $\R\contra\rarrow\T\contra_0\rarrow R[t_j]\mod_0\rarrow R\mod$.

 The functor $\R\contra\rarrow\T\contra_0$ is an equivalence of
categories by Corollary~\ref{r-t-restriction} and the functor
$R[t_j]\mod_0\rarrow R\mod$ is obviously an equivalence of
categories.
 The functor $\T\contra\rarrow R[t_j]\mod$ is fully faithful by
Lemma~\ref{many-variables-power-series}.
 
 The restriction of the latter functor to the category
$\T\contra_0$ identifies it with the full subcategory of
$R[t_j]\mod_0$ consisting of those modules where the action of
the elements~$t_j$ satisfies the condition of
Lemma~\ref{many-variables-power-series}.
 Finally, this action coincides with the action of the elements~$x_j$,
and the condition on the action of these elements is equivalent to
the condition~(d) by Lemma~\ref{ext-0-1-condition}. \qed

\subsection{The Ext comparison}
 The following theorem complements the assertion of
Theorem~\ref{noetherian-r-contramodules-thm}(1).
 We keep the notation of Section~\ref{appx-b-thm-formulation};
so $m$ is an ideal in a Noetherian commutative ring $R$, and
$\m=\varprojlim_n m/m^n$ is the extension of $m$ in the $m$\+adic
completion $\R=\varprojlim_n R/m^n$ of the ring~$R$.
 Let $\Ext_\R^*({-},{-})$ denote the Ext functor computed in
the abelian category of $\R$\+contramodules.
 
\begin{thm} \label{ext-comparison-thm}
 The map\/ $\Ext_\R^*(\P,\Q)\rarrow\Ext_R^*(\P,\Q)$ induced by
the forgetful functor $\R\contra\rarrow R\mod$ is an isomorphism
for any\/ $\R$\+contramodules\/ $\P$ and\/~$\Q$.
\end{thm}

\begin{proof}
 The map $\Hom_\R(\P,\Q)\rarrow\Hom_R(\P,\Q)$ is an isomorphism by
Theorem~\ref{noetherian-r-contramodules-thm}(1).
 Clearly, it suffices to prove that $\Ext_R^i(\P,\Q)=0$ for any
projective $\R$\+contramodule $\P$, any $\R$\+contramodule $\Q$,
and all $i>0$.
 Propositions~\ref{proj-contra-flat}
and~\ref{contra-cotorsion} below imply as much.
\end{proof}

 The next corollary, which is to be compared to
Lemma~\ref{nakayama-reduct-proj}, will follow straightforwardly from
Theorem~\ref{ext-comparison-thm}, Lemma~\ref{contraflat-flat},
and Proposition~\ref{contra-cotorsion}.
 In the case of a complete Noetherian local ring $\R$ with a maximal
ideal~$\m$, such a result was obtained in~\cite[Corollary~4.5]{Yek0},
and in the general case of the adic completion of a Noetherian ring,
in~\cite[Corollary~1.8(1) and Theorem~1.10]{PSY2}.

\begin{cor}
 An\/ $\R$\+contramodule $\P$ is a projective object in\/ $\R\contra$
if and only if either of the following equivalent conditions holds:
\begin{enumerate}
\renewcommand{\theenumi}{\alph{enumi}}
\item $\P/\m^n\P$ is a projective $R/m^n$\+module for every integer
$n\ge1$;
\item $\P/\m^n\P$ is a flat $R/m^n$\+module for every integer $n\ge1$
and\/ $\P/\m\P$ is a projective $R/m$\+module;
\item $\P$ is a flat $R$\+module and\/ $\P/\m\P$ is a projective
$R/m$\+module. \qed 
\end{enumerate}
\end{cor}

\subsection{Free/contraflat contramodules are flat modules}
 The following result can be found
in~\cite[Theorem~3.4]{Yek0} and~\cite[Theorem~1.5]{PSY2}.
 It goes back at least to~\cite[Lemma~1.4 and Theorem]{En}.

\begin{prop} \label{proj-contra-flat}
 The free\/ $\R$\+contramodule\/ $\F=\R[[X]]$ generated by a set $X$
is a flat $R$\+module, and the tensor product $R/m\ot_R\R[[X]]$
is naturally isomorphic to the free $R/m$\+module $R/m[X]$.
\end{prop}

\begin{proof}
 For any $\R$\+contramodule $\P$, one has $R/m\ot_R\P\simeq
\R/\m\ot_\R\P$, since $\m$ is the extension of~$m$ in $\R$, and
$\R/\m\ot_\R\P\simeq\P/\m\P$, since any family of elements of $\m$
converging to~$0$ in the topology of $\R$ can be represented as
a finite linear combination of families of elements of $\R$
converging to~$0$ in the topology of $\R$ with any given set of
generators of the ideal~$\m$ as the coefficients.
 Clearly, $\R[[X]]/\m\R[[X]]\simeq\R/\m[X]\simeq R/m[X]$, so
the second assertion of Proposition follows.

 The first assertion is a particular case of the next lemma.
 In the case of a complete Noetherian local ring $\R$, one can also
notice that $\R[[X]]\simeq\Hom_R(C_\R\;\bigoplus_X C_\R)$ and both
$C_\R$ and $\bigoplus_X C_\R$ are injective $R$\+modules~\cite{En}.
 Both assertions of Proposition can be obtained from this isomorphism.
\end{proof}

 The following lemma is to be compared to
Lemma~\ref{free-r-contramodules}(a)
and~\cite[Lemma~A.3]{Psemi}.

\begin{lem}  \label{contraflat-flat}
 Let\/ $\P$ be an\/ $\R$\+contramodule.
 Then\/ $\P$ is a flat $R$\+module if and only if\/ $\P/\m^n\P$ is
a flat $R/m^n$\+module for every integer $n\ge1$.
\end{lem}

\begin{proof}
 As explained in the above proof, one has $\P/\m^n\P\simeq
R/m^n\ot_R\P$, so the ``only if'' assertion is clear.
 To prove the ``if'', notice that the natural map
$\P\rarrow\varprojlim_n \P/\m^n\P$ is surjective for any
$\R$\+contramodule $\P$, as one can easily show using the infinite
summation operations in $\P$ (see~\cite[Lemma~A.2.3 and
Remark~A.3]{Psemi}).
 Assume first that this map is an isomorphism (we will see below in
this proof that for an $\R$\+contramodule $\P$ satisfying
the equivalent conditions of Lemma it always is).
 
 Consider the functor $M\mpsto \varprojlim_n M\ot_R\P/\m^n\P$ acting
from the category of finitely generated $R$\+modules $M$ to, e.~g.,
the category of abelian groups.
 Let us show that this functor is exact.
 Indeed, for any short exact sequence of finitely generated
$R$\+modules $K\rarrow L\rarrow M$ there are short exact sequences
of $R/m^n$\+modules $K/(K\cap m^nL)\rarrow L/m^nL\rarrow M/m^nM$.
 The tensor multiplication with $\P/\m^n\P$ over $R/m^n$ preserves
exactness of these triples, since $\P/\m^n\P$ is a flat $R/m^n$\+module.
 The passage to the projective limits over~$n$ preserves exactness
of the resulting sequences of tensor products, because these are
countable filtered projective systems of surjective maps.

 On the other hand, by the Artin--Rees lemma the projective
system $K/(K\cap m^nL)\ot_{R/m^n}\P/\m^n\P\simeq
K/(K\cap m^nL)\ot_R\P$ is mutually cofinal with the projective
system $K/m^nK\ot_{R/m^n}\P/\m^n\P\simeq K/m^nK\ot_R\P$.
 Hence the related projective limits coincide, and we have proven
that our functor is exact.
 For finitely generated free $R$\+modules $M$, the natural morphism
$M\ot_R\P\rarrow\varprojlim_n M\ot_R\P/\m^n\P$ is an isomorphism
due to our assumption on~$\P$.
 Both functors being right exact on the category of finitely generated
$R$\+modules $M$, it follows that the morphism is an isomorphism for
all such $M$ and the functor $M\mpsto M\ot_R\P$ is exact.

 Now let $\F$ be an arbitrary $\R$\+contramodule such that
$\F/\m^n\F$ is a flat $R/m^n$\+module for all~$n$.
 Set $\P=\varprojlim_n\F/\m^n\F=\F/\bigcap_n\m^n\F$.
 Clearly, $\P/\m^n\P\simeq\F/\m^n\F$, so we already know that
$\P$ is a flat $R$\+module.
 Consider the short exact sequence of $\R$\+contramodules
$\K\rarrow\F\rarrow\P$, where $\K=\bigcap_n\m^n\F$, and take its
tensor product with $R/m$ over~$R$.
 Since $\P$ is a flat $R$\+module, the sequence will remain exact.
 Hence $\K/\m\K=R/m\ot_R\K=0$ and $\K=0$ by Lemma~\ref{nakayama-lemma}.
\end{proof}

\subsection{All contramodules are relatively cotorsion modules}
 The following proposition demonstrates the connection between our
$\R$\+contramodules and Harrison's \emph{co-torsion abelian
groups}~\cite{Har} or Enochs' \emph{cotorsion $R$\+modules}~\cite{En}
(see~\cite[Section~1.3]{Pcosh} for further details).

\begin{prop} \label{contra-cotorsion}
 For any flat $R$\+module $F$ such that $R/m\ot_R F$ is a projective
$R/m$\+module, and any $\R$\+contramodule\/ $\Q$, the vanishing\/
$\Ext_R^i(F,\Q)=0$ holds for all\/ $i>0$.
\end{prop}

\begin{proof}
 Present an $\R$\+contramodule $\Q$ as the quotient contramodule of
a free $\R$\+contramodule $\Q'$ by its $\R$\+subcontramodule~$\Q''$.
 Then one has $\Q'\simeq\varprojlim_n \Q'/\m^n\Q'$ and
$\Q''\simeq\varprojlim_n\Q''/\m^n\Q''$ \cite[Lemma~A.2.3 and
Remark~A.3]{Psemi}.
 This allows one to assume that $\Q\simeq\varprojlim_n\Q/\m^n\Q$
in the assertion of Proposition.

 Since $F$ is a flat $R$\+module, there is a natural isomorphism
$\Ext_R^*(F\;\Q/\m^n\Q)\simeq\Ext_{R/m^n}^*(F/m^n F\;\Q/\m^n\Q)$.
 The following lemma is standard.

\begin{lem}
 Let $A$ be an associative ring, $J\subset A$ be an ideal, and $n\ge1$
be an integer such that $J^n=0$.
 Suppose $P$ is a flat left $A$\+module such that $P/JP$ is
a projective left $A/J$\+module.
 Then the $A$\+module $P$ is projective.
\end{lem}

\begin{proof}
 The proof is similar to that of Lemma~\ref{r-free-proj-inj-reduction}.
 One picks a free $A$\+module $G$ such that $P/JP$ is the image of
an idempotent endomorphism~$e$ of the $A/J$\+module $G/JG$; notices
that the associative ring homomorphism $\Hom_A(G,G)\rarrow
\Hom_{A/J}(G/JG\;\allowbreak G/JG)$ is surjective with the kernel
$I=\Hom_A(G,JG)\subset\Hom_A(G,G)$ for which one has $I^n=0$; lifts
the idempotent element $e\in\Hom_{A/J}(G/JG\;G/JG)$ to an idempotent
element $f\in\Hom_A(G,G)$; considers the projective $A$\+module $fG$;
and lifts its natural morphism onto $fG/J(fG)\simeq e(G/JG)\simeq P/JP$
to an $A$\+module morphism $fG\rarrow P$.
 For the cokernel $C$ of the latter morphism, one has $C/JC=0$, hence
$C=0$.
 Finally, one considers the exact sequence of $A$\+modules
$0\rarrow K\rarrow fG\rarrow P\rarrow 0$, which remains exact after
tensoring with $A/J$, since the $A$\+module $P$ is flat; hence
$K/JK=0$ and $K=0$, so $P\simeq fG$.
\end{proof}

 We have proven that $F/m^nF$ is a projective $R/m^n$\+module.
 It follows that $\Ext_R^i(F\;\Q/\m^n\Q)=0$ for all $i>0$
and all~$n$.
 Besides, the natural map
\begin{multline*}
\Hom_R(F\;\Q/\m^n\Q)\simeq\Hom_R(F/m^nF\;\Q/\m^n\Q) \\
\lrarrow\Hom_R(F/m^nF\;\Q/\m^{n-1}\Q)\simeq\Hom_R(F\;\Q/\m^{n-1}\Q)
\end{multline*}
is surjective.
 Now it remains to apply the following general result.
\end{proof}

\begin{lem}
 Let $A$ be an associative ring, $L$ be a left $A$\+module, and
$M_\alpha$ be a projective system of left $A$\+modules.
 Assume that\/ $\varprojlim^i_\alpha M_\alpha=0$ and\/
$\Ext_A^i(L,M_\alpha)=0$ for all\/ $i>0$ and all\/~$\alpha$.
 Then there is a natural isomorphism\/ $\Ext_A^*(L\;
\varprojlim_\alpha M_\alpha)\simeq
\varprojlim_\alpha^*\Hom_A(L,M_\alpha)$.
 In particular, if $M_\alpha$ form a countable filtered projective
system and the natural maps\/ $\Hom_A(L,M_\alpha)\rarrow
\Hom_A(L,M_\beta)$ are surjective for all\/ $\alpha>\beta$, then
$\Ext_A^i(L\;\varprojlim_\alpha M_\alpha)=0$ for all\/ $i>0$.
\end{lem}

\begin{proof}
 Let $P_\bu\rarrow L$ be a left projective resolution of
the $A$\+module~$L$.
 Consider the complex of projective systems $\Hom_A(P_\bu, M_\alpha)$.
 The condition $\Ext_A^i(L,M_\alpha)=0$ for all $i>0$ and all~$\alpha$
implies that this complex is a right resolution of the projective
system $\Hom_A(L,M_\alpha)$.
 The condition $\varprojlim^i_\alpha M_\alpha=0$ for all $i>0$
implies $\varprojlim_\alpha^i\Hom_A(P_j,M_\alpha)=0$ for all $i>0$
and all~$j$.
 Now the complex $\varprojlim_\alpha\Hom_A(P_\bu,M_\alpha)\simeq
\Hom_A(P_\bu\;\varprojlim_\alpha M_\alpha)$ computes both
$\Ext_A^*(L\;\varprojlim_\alpha M_\alpha)$ and
$\varprojlim_\alpha^*\Hom_A(L,M_\alpha)$.
\end{proof}


\medskip
\begin{thebibliography}{99}
\smallskip

\bibitem{Groth}
 M.~Artin, A.~Grothendieck, J.-L.~Verdier.
   Th\'eorie des topos et cohomologie \'etale des sch\'emas.
S\'eminaire de G\'eom\'etrie Alg\'ebrique du Bois Marie
1963--1964 (SGA~4), tome~1.
\textit{Lecture Notes in Math.} \textbf{269}, Springer, 1972.

\bibitem{Bas}
 H.~Bass.
   Finitistic dimension and a homological generalization of
semi-primary rings.
\textit{Transactions of the American Math.\ Society}
\textbf{95}, p.~466--488, 1960.

\bibitem{BL}
 J.~Bernstein, V.~Lunts.
   Equivariant sheaves and functors.
\textit{Lecture Notes in Math.}\ \textbf{1578}, Springer,
Berlin, 1994.

\bibitem{BK}
 A.~K.~Bousfield, D.~M.~Kan.
   Homotopy limits, completions and localizations.
\textit{Lecture Notes in Math.}\ \textbf{304}, Springer,
1972--1987.

\bibitem{Ch}
 S.~U.~Chase.
   Direct products of modules.
\textit{Transactions of the American Math.\ Society}
\textbf{97}, p.~457--473, 1960.

\bibitem{Cho}
 C.-H.~Cho.
   On the obstructed Lagrangian Floer theory.
\textit{Advances in Math.}\ \textbf{229}, \#2, p.~804--853, 2012.
\texttt{arXiv:0909.1251 [math.SG]}

\bibitem{DL}
 O.~De Deken, W.~Lowen.
   On deformations of triangulated models.
\textit{Advances in Math.}\ \textbf{243}, p.~330--374, 2013.
\texttt{arXiv:1202.1647 [math.KT]}

\bibitem{Dur}
 N.~Durov.
   New approach to Arakelov geometry.
Doctoral Dissertation, University of Bonn, 2007.
\texttt{arXiv:0704.2030 [math.AG]}

\bibitem{DG}
 W.~G.~Dwyer, J.~P.~C.~Greenlees.
   Complete modules and torsion modules.
\textit{American Journ.\ of Math.}\ \textbf{124}, \#1,
p.~199--220, 2002.

\bibitem{EM0}
 S.~Eilenberg, J.~C.~Moore.
   Limits and spectral sequences.
\textit{Topology} \textbf{1}, p.~1--23, 1962.

\bibitem{EM}
 S.~Eilenberg, J.~C.~Moore.
   Foundations of relative homological algebra.
\textit{Memoirs of the American Math.\ Society} \textbf{55}, 1965.

\bibitem{En}
 E.~Enochs.
   Flat covers and flat cotorsion modules.
\textit{Proceedings of the American Math. Society} \textbf{92}, \#2,
p.~179--184, 1984.

\bibitem{Fe}
 W.~Feit.
   The representation theory of finite groups.
North-Holland Mathematical Library, 1982.

\bibitem{Fu}
 K.~Fukaya.
   Deformation theory, homological algebra and mirror symmetry.
Geometry and physics of branes (Como, 2001), Ser.\ High Energy Phys.\
Cosmol.\ Gravit., IOP, Bristol, 2003, p.~121--209.

\bibitem{FOOO}
 K.~Fukaya, Y.-G.~Oh, H.~Ohta, K.~Ono.
   Lagrangian intersection Floer theory: anomaly and obstruction.
Parts~I and~II.
American Math. Society/Intern.\ Press Studies in Advanced
Math., vol.~46.1--2, 2009.

\bibitem{Gab}
 P.~Gabriel.
   Des cat\'egories ab\'eliennes.
\textit{Bulletin de la Soc.\ Math.\ de France} \textbf{90},
p.~323--448, 1962.

\bibitem{Gai}
 D.~Gaitsgory.
   Ind-coherent sheaves.
\textit{Moscow Math.\ Journ.}\ \textbf{13}, \#3, p.~399-528, 2013.
\texttt{arXiv:1105.4857 [math.AG]} {\hbadness=1200\par}

\bibitem{GR}
 D.~Gaitsgory, N.~Rozenblyum.
   A study in derived algebraic geometry, vol~II: Deformations,
Lie theory, and formal geometry.
 Mathematical Surveys and Monographs~221.2,
American Math.\ Society, Providence, RI, 2017.

\bibitem{GM}
 S.~I.~Gelfand, Yu.~I.~Manin.
     Methods of homological algebra.  Translation from the 1988 Russian
original.  Springer, Berlin, 1996.  Or: Second edition.
Springer, 2003.

\bibitem{Har}
 D.~K.~Harrison.
   Infinite abelian groups and homological methods.
\textit{Annals of Math.}\ \textbf{69}, \#2, p.~366--391, 1959.

\bibitem{Hin}
 V.~Hinich.
   DG coalgebras as formal stacks.
\textit{Journ.\ of Pure and Appl.\ Algebra} \textbf{162}, \#2--3,
p.~209--250, 2001.  \texttt{arXiv:math.AG/9812034}

\bibitem{HMS}
 D.~Husemoller, J.~C.~Moore, J.~Stasheff.
   Differential homological algebra and homogeneous spaces.
\textit{Journ.\ of Pure and Appl.\ Algebra} \textbf{5}, p.~113--185,
1974.

\bibitem{Jan}
 U.~Jannsen.
   Continuous \'etale cohomology.
\textit{Mathematische Annalen} \textbf{280}, \#2, p.~207--245, 1988.

\bibitem{Jor}
 P.~J\o rgensen.
   The homotopy category of complexes of projective modules.
\textit{Advances in Math.} \textbf{193}, \#1, p.~223--232, 2005.
\texttt{arXiv:math.RA/0312088}

\bibitem{KS}
 M.~Kashiwara, P.~Schapira.
   Categories and Sheaves.
Grundlehren der mathematischen Wissenschaften~332,
Springer-Verlag, Berlin--Heidelberg, 2006.

\bibitem{Kel}
 B.~Keller.
   Deriving DG-categories.
\textit{Ann.\ Sci.\ \'Ecole Norm.\ Sup.\ (4)} \textbf{27}, \#1,
p.~63--102, 1994.

\bibitem{Kel2}
 B.~Keller.
   Koszul duality and coderived categories (after K.~Lef\`evre).
October 2003.  Available from
\texttt{http://www.math.jussieu.fr/\~{ }keller/publ/index.html}\,.

\bibitem{KLN}
 B.~Keller, W.~Lowen, P.~Nicol\'as.
   On the (non)vanishing of some ``derived'' categories of curved
dg algebras.
\textit{Journ.\ of Pure and Appl.\ Algebra} \textbf{214}, \#7,
p.~1271--1284, 2010.  \texttt{arXiv:0905.3845 [math.KT]}

\bibitem{Kra}
 H.~Krause.
   The stable derived category of a noetherian scheme.
\textit{Compositio Math.}\ \textbf{141}, \#5, p.~1128--1162, 2005.
\texttt{arXiv:math.AG/0403526}

\bibitem{Lef}
 K.~Lef\`evre-Hasegawa.
   Sur les \Ainfty-cat\'egories.
Th\`ese de doctorat, Universit\'e Denis Diderot -- Paris 7,
November 2003.  \texttt{arXiv:math.CT/0310337}\,. 
Corrections, by B.~Keller.  Available from
\texttt{http://people.math.jussieu.fr/\~{ }keller/lefevre/publ.html}\,.

\bibitem{Lur}
 J.~Lurie.
   Derived algebraic geometry II: noncommutative algebra.
Electronic preprint \texttt{arXiv:math.CT/0702299}\,.
See also: Higher algebra.  Available from the author's homepage
at \texttt{http://www.math.harvard.edu/\~{ }lurie/}\,.
{\hbadness=1700\par}

\bibitem{Mat}
 E.~Matlis.
   The higher properties of $R$\+sequences.
\textit{Journ.\ of Algebra} \textbf{50}, \#1, p.~77--112, 1978.

\bibitem{Mats}
 H.~Matsumura.
   Commutative ring theory.
Translated by M.~Reid.  Cambridge University Press, 1986--2006.

\bibitem{Nic}
 P.~Nicol\'as.
   The bar derived category of a curved dg algebra.
\textit{Journ.\ of Pure and Appl.\ Algebra} \textbf{212}, \#12,
p.~2633--2659, 2008.  \texttt{arXiv:math.RT/0702449}

\bibitem{Pcurv}
 L.~Positselski.
   Nonhomogeneous quadratic duality and curvature.
\textit{Functional Analysis and its Applications} \textbf{27},
\#3, p.~197--204, 1993.

\bibitem{Psemi}
 L.~Positselski.
   Homological algebra of semimodules and semicontramodules:
Semi-infinite homological algebra of associative algebraic structures.
 Appendix~C in collaboration with D.~Rumynin; Appendix~D in
collaboration with S.~Arkhipov.
 Monografie Matematyczne vol.~70, Birkh\"auser/Springer Basel, 2010. 
xxiv+349~pp. \texttt{arXiv:0708.3398 [math.CT]}

\bibitem{Pkoszul}
 L.~Positselski.
   Two kinds of derived categories, Koszul duality, and
comodule-contramodule correspondence.
\textit{Memoirs of the American Math.\ Society} \textbf{212},
\#996, 2011.  vi+133~pp.  \texttt{arXiv:0905.2621 [math.CT]}

\bibitem{PP2}
 A.~Polishchuk, L.~Positselski.
   Hochschild (co)homology of the second kind~I\hbox{}.
\textit{Transactions of the American Math.\ Society} \textbf{364},
\#10, p.~5311--5368, 2012.  \texttt{arXiv:1010.0982 [math.CT]}

\bibitem{Psing}
 A.~I.~Efimov, L.~Positselski.
   Coherent analogues of matrix factorizations and relative
singularity categories.
\textit{Algebra and Number Theory} \textbf{9}, \#5, p.~1159--1292,
2015.  \texttt{arXiv:1102.0261 [math.CT]}

\bibitem{Pcosh}
 L.~Positselski.
   Contraherent cosheaves.
Electronic preprint \texttt{arXiv:1209.2995 [math.CT]}\,.

\bibitem{Pcon}
 L.~Positselski.
   Contramodules.
Electronic preprint \texttt{arXiv:1503.00991 [math.CT]}\,.

\bibitem{Pmgm}
 L.~Positselski.
   Dedualizing complexes and MGM duality.
\textit{Journ.\ of Pure and Appl.\ Algebra} \textbf{220}, \#12,
p.~3866--3909, 2016.  \texttt{arXiv:1503.05523 [math.CT]}

\bibitem{PR}
 L.~Positselski, J.~Rosick\'y.
   Covers, envelopes, and cotorsion theories in locally presentable
abelian categories and contramodule categories.
\textit{Journ.\ of Algebra} \textbf{483}, p.~83--128, 2017.
\texttt{arXiv:1512.08119 [math.CT]}

\bibitem{Pcta}
 L.~Positselski.
   Contraadjusted modules, contramodules, and reduced cotorsion modules.
\textit{Moscow Math.\ Journ.}\ \textbf{17}, \#3, p.~385--455, 2017.
\texttt{arXiv:1605.03934 [math.CT]}

\bibitem{Sim}
 A.-M.~Simon.
   Approximations of complete modules by complete big
Cohen--Macaulay modules over a Cohen--Macaulay local ring.
\textit{Algebras and Representation Theory} \textbf{12}, \#2--5,
p.~385--400, 2009.

\bibitem{Spal}
 N.~Spaltenstein.
   Resolutions of unbounded complexes.
\textit{Compositio Math.}\ \textbf{65}, \#2, p.121--154, 1988.

\bibitem{Yek0}
 A.~Yekutieli.
   On flatness and completion for infinitely generated modules over
noetherian rings.
\textit{Communications in Algebra} \textbf{39}, \#11,
p.~4221--4245, 2011.  \texttt{arXiv:0902.4378 [math.AC]}

\bibitem{PSY}
 M.~Porta, L.~Shaul, A.~Yekutieli.
   On the homology of completion and torsion.
\textit{Algebras and Representation Theory} \textbf{17}, \#1,
p.~31--67, 2014.  \texttt{arXiv:1010.4386 [math.AC]}
Erratum in \textit{Algebras and Representation Theory} \textbf{18},
\#5, p.~1401--1405, 2015.  \texttt{arXiv:1506.07765 [math.AC]}

\bibitem{PSY2}
 M.~Porta, L.~Shaul, A.~Yekutieli.
   Cohomologically cofinite complexes.
\textit{Communications in Algebra} \textbf{43}, \#2,
p.~597--615, 2015.  \texttt{arXiv:1208.4064 [math.AC]}

\bibitem{Yek}
 A.~Yekutieli.
   A separated cohomologically complete module is complete.
\textit{Communications in Algebra} \textbf{43}, \#2,
p.~616--622, 2015.  \texttt{arXiv:1312.2714 [math.AC]}

\end{thebibliography}
\end{document}